\documentclass{article}
\usepackage{amsfonts,amsmath,amssymb,amsthm}
\usepackage{bm}
\usepackage{hyperref}
\usepackage{url}
\usepackage{tikz}
\usetikzlibrary{tikzmark, arrows.meta, calc, positioning}
\usepackage{comment}
\usepackage{appendix}
\usetikzlibrary{arrows.meta}
\usepackage{authblk}
\newcommand{\alert}[2][magenta]%
{{\color{#1}\mbox{[\hspace{-0.4ex}[}#2\mbox{]\hspace{-0.4ex}]}}}

\title{Generalized ovals, 2.5-dimensional additive codes, and multispreads}

\author[$\dagger\S$]{Denis~S.~Krotov\thanks{
D.S.~Krotov is supported
  by the Natural Science Foundation of Hebei Province (A2023205045)
and within the framework of the state contract of the Sobolev Institute of Mathematics (FWNF-2026-0011).}}

\author[$\ddag$]{Sascha Kurz}

\affil[$\dagger$]{School of Mathematical Sciences, Hebei Key Laboratory of Computational Mathematics and Applications, Hebei
Normal University, Shijiazhuang, China}
\affil[$\S$]{Sobolev Institute of Mathematics, Novosibirsk, Russia}
\affil[$\ddag$]{Mathematisches Institut, Universit\"at Bayreuth, Bayreuth, Germany}
\date{December 2025}

\def\FF{\mathbb{F}}

\def\PG{\mathrm{PG}}

\newtheorem{corollary}{Corollary}
\newtheorem{theorem}{Theorem}
\newtheorem{definition}{Definition}
\newtheorem{proposition}{Proposition}
\newtheorem{lemma}{Lemma}
\newtheorem{example}{Example}
\theoremstyle{remark}
\newtheorem{remark}{Remark}

\newcommand\PGF[2]{\PG\big(\FF_{#2}^{#1}\big)}

\begin{document}

\maketitle

\begin{abstract}
    We present constructions and bounds for additive codes over a finite field in terms of their geometric counterpart, i.e., projective systems. It is known that the maximum number of $(h-1)$-spaces in PG$(2,q)$, such that no hyperplane contains three, is given by $q^h+1$ if $q$ is odd. Those geometric objects are called generalized ovals. We show that cardinality $q^h+2$ is possible if we decrease the dimension a bit. We completely determine the minimum possible lengths of additive codes over GF$(9)$ of dimension $2.5$ and give improved constructions for other small parameters, including codes outperforming the best linear codes. As an application, we consider multispreads in PG$(4,q)$, in particular, completing the characterization of parameters of GF$(4)$-linear $64$-ary one-weight codes.

    \smallskip

\end{abstract}

\section{Introduction}
An additive code $C$ is a subset of $\mathcal{R}^n$ that is closed under addition, where $\mathcal{R}$ is some mathematical structure that allows addition like e.g.\ a ring, a field, or an additive group. Here we will always assume that $\mathcal{R}$ is a field and that the distances between the codewords, i.e.\ the elements of $C$, are measured by the Hamming metric. In computer science additive codes are studied as a subclass of general block codes, containing linear codes, that has better algorithmic properties. In \cite{guruswami2008explicit} folded Reed--Solomon codes were introduced with the property that they  admit much better list-decoding algorithms than the original Guruswami--Sudan algorithm \cite{guruswami1999improved}. Another subclass of additive codes
allowing improved decoding algorithms, see e.g.~\cite{tamo2024tighter}, and also obtained by some folding (or grouping) construction is given by so-called multiplicity codes
\cite{kopparty2014high,rosenbloom1997codes}. There is some renewed interest in additive codes due to applications in the construction of quantum codes, see e.g.\ \cite{dastbasteh2024new,polynomialRepresentation,grassl2021algebraic,li2024ternary}. Also, for classical codes additive codes are interesting since they can have better parameters than linear codes, see e.g.\ \cite{guan2023some}. It is well known that additive codes over $\FF_4$ can geometrically be described by multisets of lines in a projective space over $\FF_2$, see e.g.\ \cite{blokhuis2004small}. In general, additive codes are in one-to-one correspondence to multisets of subspaces in projective spaces, see \cite{ball2025griesmer}, which generalizes the well-known relation between linear codes and multisets of points. Here we will mainly use the geometric language.

A classical geometric result considers the maximum number of points in a projective plane over $\FF_q$ such that every line contains at most two of them. For odd field sizes $q$ the maximum number is given by $q+1$ and the geometric objects are called ovals. For even $q$ size $q+2$ is possible, and the geometric objects are called hyperovals. Via the subfield construction we also obtain sets of subspaces, no more than two contained in a hyperplane, with those cardinalities. For odd~$q$ we can slightly improve the construction, see Proposition~\ref{prop_generalized_construction_oval} and Theorem~\ref{prop_generalized_construction_oval_s}.

For linear codes, it is a classical and important problem to determine the minimum possible length $n$ for each minimum Hamming distance $d$, given the field size $q$ and the dimension~$k$. Due to the Griesmer bound and the Solomon--Stiffler construction, this is a finite problem for each pair of parameters $q$ and~$k$. So far it has been completely solved for $k\le 8$ when $q=2$, for $k\le 5$ when $q=3$, for $k\le 4$ when $q=4$, for $k\le 3$ when $q\le 9$, and for all $k\le 2$. For additive codes there is also a Griesmer type bound, see \cite[Theorem 12]{ball2025griesmer}, which can always be attained with equality if the minimum distance $d$ is sufficiently large \cite{kurz2024additive}. Here we determine the minimum lengths of $5/2$-dimensional additive codes over $\FF_9$ for all minimum distances, see Theorem~\ref{thm_n_3_5_2}. 
For $5/2$-dimensional additive codes over $\FF_{25}$ and $\FF_4$-linear codes over $\FF_{16}$ we give partial results. Additionally, we find two $3$-dimensional additive codes over~$\FF_{4}$ outperforming the best linear codes of the same length and dimension. 

For each prime power $q$ there exists a line spread in $\PG(3,q)$, i.e.\ a set of lines such that each point is contained in exactly one element. The notion of $t$-spreads has been generalized to multispreads \cite{KroMog:multispread}. Here we are e.g.\ interested in sets of $q^3+1$ planes and $q^2+1$ lines in $\PG(4,q)$ such that each point is either contained in $q+1$ planes and no line or, alternatively, in one plane and one line. For $q=3$ we classify all such systems and for $q=4,5$ we give partial results. In coding theoretic terms multispreads correspond to additive one-weight codes.

The remaining part of this paper is structured as follows. In Section~\ref{sec_preliminaries} we introduce the necessary preliminaries. Our results on oval-like structures
are presented in Section~\ref{sec_oval}.
In Section~\ref{sec_ILP} we use prescribed automorphisms and integer linear programming to obtain improved constructions for additive codes with small parameters. Multispreads are treated in Section~\ref{sec_multispreads}. For the mentioned oval-like structure we give an alternative construction in Appendix~\ref{appendixC}, which is a rare example of a non-computational construction of a code with prescribed automorphism group. Multispreads with new parameters are listed in Appendix~\ref{appendixB}. Explicit representations for the $5/2$-dimensional additive codes over $\FF_{16}$ and $\FF_{25}$ mentioned in Theorem~\ref{thm_n_4_5_2} and Theorem~\ref{thm_n_5_5_2} are listed in Appendix~\ref{sec_explicit_4} and Appendix~\ref{sec_explicit_5}, respectively.


\section{Preliminaries}
\label{sec_preliminaries}
Let $\FF_q$ denote the finite field with $q$ elements, where $q=p^l$ is a prime power and the prime is called the characteristic of $\FF_q$. An \emph{additive code} $C$ of length $n$
over the alphabet $\mathcal{A}=\FF_{q'}$ is a subset of $\FF_{q'}^n$ such that $u+v\in C$
for all $u,v\in C$. It turns out that each code $C$ that is additive over $\FF_{q'}$ is
linear over some subfield $\FF_{q}\le\FF_{q'}$, i.e., $\alpha u+\beta v\in C$ for all
$u,v\in C$ and all $\alpha,\beta\in\FF_{q}$ \cite{ball2023additive,ball2025griesmer}.
So, we use the notation $[n,r/h,d]_q^h$ for an additive code $C$ that is linear
over $\FF_q$ and has length $n$, minimum distance $d$,  alphabet
$\mathcal{A}=\FF_{q^h}$, and size $q^r$, where $r\in \mathbb{N}$. We also call
$k=r/h\in\mathbb{Q}$ the \emph{dimension} of~$C$, so that $|C|=\left|\mathcal{A}^k\right|$
and an $[n,k,d]_q^1$ additive code is an $[n,k,d]_q$ linear code. Note that $k$ can be fractional. An $[n,k,d]_q$ linear code $C$ can be defined as the rowspace of a $k\times n$ matrix
with entries in $\FF_q$, called a generator matrix for~$C$. Similarly,  an
$[n,r/h,d]_q^h$ additive code $C$ can be defined as the $\FF_q$-space spanned by the
rows of an $r\times n$ matrix $G$ with entries in $\FF_{q^h}$,  again called a
\emph{generator matrix} for~$C$. Let $\mathcal{B}$ be a basis for $\FF_{q^h}$ over
$\FF_q$ and write out the elements of $G$ over the basis $\mathcal{B}$ to obtain
an $r\times nh$ matrix $\widetilde{G}$ with entries from $\FF_q$. Here we assume that the
columns of $\widetilde{G}$ are grouped together into $n$ groups of $h$ columns and say
that $\widetilde{G}$ is a subfield generator matrix for $C$.

Let $\PGF{v}{q}$ denote the projective geometry over $\FF_q^v$. While we will use the geometric language, we will use the algebraic dimension for subspaces, i.e., points are $1$-dimensional, lines are $2$-dimensional, planes are $3$-dimensional, 
and hyperplanes are $(v-1)$-dimensional subspaces of~$\FF_q^v$. In general, a $t$-space is a $t$-dimensional subspace. By $[t]_q:=\tfrac{q^t-1}{q-1}$ we denote the number of points of a $t$-space. Given a generator matrix $G$ of a linear $[n,k,d]_q$ code, the columns of $G$ span $1$-dimensional subspaces, i.e.\ points in $\PGF{k}{q}$ (allowing the $0$-dimensional space as a degenerated case or forbidding zero columns in $G$). This induces the well-known correspondence between linear codes and multisets of points in projective spaces, see e.g.\ \cite{dodunekov1998codes}. More precisely, a minimum distance $d$ of the code relates to the geometric property that at most $n-d$ elements are contained in a hyperplane. Similarly, the groups of $h$ subsequent columns in a subfield generator matrix $\widetilde{G}$ span subspaces of dimension at most $h$ such that at most $n-d$ elements are contained in each hyperplane, if the additive code has length $n$ and minimum distance $d$. By $\mathcal{X}_G(C)$ we define the multiset of the $n$ subspaces spanned by the
$n$ blocks of $h$ columns of $\widetilde{G}$ in this way.
We say that $C$ is faithful if all elements of $\mathcal{X}_G(C)$ have dimension
$h$ and unfaithful otherwise.
This is indeed a property of the code $C$ and does not
depend on the choice of a generator matrix $G$ or the choice of a basis
$\mathcal{B}$, see e.g.\ \cite{ball2025griesmer}, so that we also write $\mathcal{X}(C)$. We remark that a linear code $C$, i.e.\ an additive code with
$h=1$ is unfaithful iff an arbitrary generator matrix for $C$ contains a column
consisting entirely of zeroes.

\begin{example}
\label{ex_additive_code}
Writing $\FF_4\simeq \FF_2[\omega]/\left(\omega^2+\omega+1\right)$,
we can start with the generator matrix of a linear $[5,2,4]_4$ code, interpret it as the generator matrix of an additive code, and use the basis $\mathcal{B}$ to obtain the example
$$
  \begin{pmatrix}
      0  & 1 & 1 & 1      & 1        \\
      1  & 0 & 1 & \omega & \omega^2
  \end{pmatrix}
  \,\,\rightarrow\,\,
  \begin{pmatrix}
      0  & 1 & 1 & 1      & 1        \\
      0  & \omega & \omega & \omega      & \omega        \\
      1  & 0 & 1 & \omega & \omega^2 \\
      \omega  & 0 & \omega & \omega^2 & 1
  \end{pmatrix}
  \,\,\rightarrow\,\,
  \begin{pmatrix}
      00  & 10 & 10 & 10 & 10 \\
      00  & 01 & 01 & 01 & 01 \\
      10  & 00 & 10 & 01 & 11 \\
      01  & 00 & 01 & 11 & 10
  \end{pmatrix}\!.
$$
The first matrix is a generator matrix of a $[5,2,4]_4$ code, which geometrically is the set of all five points in the projective line $\PGF{2}{4}$. The second matrix is a generator matrix of a $[5,4/2,4]_2^2$ (additive) code, which is linear over $\FF_4$ by construction (called subfield construction). The third matrix is the corresponding subfield generator matrix, where the five pairs of columns span lines, i.e.\ we geometrically have a line spread of $\PGF{4}{2}$.
\end{example}

\begin{definition}
  A projective $h-(n, r, s)_q$ system is a multiset $\mathcal{S}$ of $n$ subspaces of $\PGF{r}{q}$ of dimension at most $h$ such that each hyperplane contains at most~$s$ elements of $\mathcal{S}$ and some hyperplane contains exactly~$s$ elements of $\mathcal{S}$. We say that $\mathcal{S}$ is faithful if all elements have dimension~$h$. A projective $h-(n, r, s)_q$ system $\mathcal{S}$ is a projective $h-(n, r, s, \mu)_q$ system if each point is contained in at most~$\mu$ elements from $\mathcal{S}$ and there is some point that is contained in exactly $\mu$ elements from $\mathcal{S}$.
\end{definition}

Note that the elements of $\mathcal{S}$ span the entire ambient space $\PGF{r}{q}$ iff $s<n$. If $\mathcal{S}$ is a projective $h-(n, r, s)_q$ system that is not faithful, then we can easily construct a projective $h-(n, r, {\le} s)_q$ system $\mathcal{S}'$ by replacing each element $S\in\mathcal{S}$ with dimension smaller than $h$ by an arbitrary $h$-space containing $S$.

\begin{definition}
  By $n_q(r,h;s)$ we denote the maximum number $n$ such that a projective $h-(n, r, s)_q$ system exists.
\end{definition}

From the subfield construction we conclude $n_q(rh,h;s)\ge n_{q^h}(r,1;s)$.
In general, projective $h-(n,r,s)_q$ systems (with $s<n$) are in one-to-one correspondence
to additive codes:
\begin{theorem}(\cite[Theorem 5]{ball2025griesmer})
   \label{thm_connection}
   If $C$ is an additive $[n,r/h,d]_q^h$ code with generator matrix $G$, then
   $\mathcal{X}_G(C)$ is a projective $h-(n, r, n-d)_q$ system $\mathcal{S}$, and conversely,  each projective $h-(n,r, s)_q$ system $\mathcal{S}$ defines an additive $[n,r/h,n-s]_q^h$ code~$C$.
\end{theorem}

We have that $C$ is faithful iff $\mathcal{S}$ is faithful. Whenever the specific choice of a generator matrix $G$ or a subfield generator matrix $\widetilde{G}$ is irrelevant, we write $\mathcal{S}=\mathcal{X}(C)$ or $C=\mathcal{X}^{-1}(\mathcal{S})$. Unfaithful projective systems can arise by projection. If we e.g.\ start with the faithful projective $2-(5,4;1)_2$ system from Example~\ref{ex_additive_code}, then projection through an arbitrary point yields an unfaithful projective $2-(5,3;1)_2$ system consisting of four lines and one point (obtained from the line containing the projection point). So clearly, we have $n_q(r,h;s)\ge n_q(r',h;s)$ for all $r\le r'$.

\begin{definition}
  \begin{equation*}
    \overline{n}_q(r,h;s):=n_{q^h}(\left\lceil r/h\right\rceil,1;s)
  \end{equation*}
\end{definition}
In words, $\overline{n}_q(r,h;s)$ is the size of the largest projective $h-(n,r,s)_{q}$ system that we can obtain starting from a linear code over $\FF_{q^h}$ interpreted as a multiset of points via the subfield construction by an iterative application of projection. Whenever $\overline{n}_q(r,h;s)<n_q(r,h;s)$ we say that additive codes \emph{outperform} linear codes for the corresponding parameters, which is especially interesting if $r/h$ is integral.

We can also utilize linear codes to obtain upper bounds for $n_q(r,h;s)$. 

\begin{definition}
  For a faithful projective  $h-(n,r,s)_q$ system $\mathcal{S}$ let $\mathcal{P}(\mathcal{S})$ denote the multiset of points that we obtain
  by replacing each element of $\mathcal{S}$ by its contained $[h]_q$ points.
\end{definition}
\begin{lemma}(\cite[Lemma 4]{kurz2024additive})
  \label{lemma_linear_code}
  Let $\mathcal{S}$ be a faithful projective  $h-(n,r,s,\mu)_q$ system.
  Then $\mathcal{P}(\mathcal{S})$ is a faithful projective $1-(n',r,s',\mu)_q$ system,
  where $n'=n[h]_q$ and $s'=n\cdot[h-1]_q+s\cdot  q^{h-1}$.
  Moreover,
  $C:=\mathcal{X}^{-1}(\mathcal{P}(\mathcal{S}))$ is a  
  $q^{h-1}$-divisible linear $[n',r,d']_q$ code $C$ with maximum weight at most $n\cdot q^{h-1}$, where $d'=q^{h-1}\cdot (n-s)$.
\end{lemma}
Here a linear code is called $\Delta$-divisible if the weights of all codewords are divisible by~$\Delta$.

\begin{definition}
  The strong coding upper bound for $n_q(r,h;s)$ is the largest integer~$n$ such that a linear code with properties as specified in Lemma~\ref{lemma_linear_code} exists. Ignoring the conditions on divisibility and the maximum weight, the maximum possible length $n$ is called the weak coding upper bound.
\end{definition}

\begin{example}
  Since the existence of a $[84,8,40]_2$ code and the non-existence of a $[87,8,42]_2$
  code are known, we obtain $n_2(8,2;8)\le 28$, i.e.\ the weak coding upper bound for $n_2(8,2;8)$ is $28$.
\end{example}
One specific lower bound for the length $n$ of an $[n,k,d]_q$ code is the so-called
Griesmer bound \cite{griesmer1960bound,solomon1965algebraically}
\begin{equation}
  \label{eq_griesmer_bound}
  n\ge \sum_{i=0}^{k-1} \left\lceil\frac{d}{q^i}\right\rceil=:g_q(k,d).
\end{equation}
Interestingly enough, this bound can always be attained with equality if the minimum distance $d$ is sufficiently large and a nice geometric construction was given by Solomon and Stiffler \cite{solomon1965algebraically}. For additive codes we also have a Griesmer bound via the reformulation as a projective system $\mathcal{S}$ and Lemma~\ref{lemma_linear_code} by only using the Griesmer bound for linear codes
\begin{eqnarray*}
  n&\ge& \left\lceil
  \frac{g_q\!\left(r,d\cdot q^{h-1}\right)}{[h]_q}
  \right\rceil
  =
  \left \lceil
  \raisebox{-1.5ex}{$\displaystyle
  \frac{ \sum\limits_{i=0}^{r-1} \left\lceil d\cdot q^{h-1-i}
  \right\rceil}{[h]_q}
  $}
  \right \rceil \\
  &=&d+
  \left \lceil
  \raisebox{-1.8ex}{$\displaystyle
  \frac{\sum\limits_{i=1}^{r-h} \bigl \lceil\frac{d}{q^i} \bigr \rceil}{[h]_q}
  $}
  \right \rceil
  =
  d+ \left\lceil \frac{g_q(r-h+1,d)-d   }{[h]_q}\right\rceil,
  \end{eqnarray*}
see e.g.\ \cite[Theorem 12]{ball2025griesmer} or \cite[Lemma 15]{kurz2024additive}. Here we also speak of the Griesmer upper bound, which is $30$ in our example since the Griesmer bound implies that no $[93,8,46]_2$ code exists but cannot rule out the existence of a $[90,8,44]_2$ code. One of the main results of \cite{kurz2024additive} is that the mentioned Griesmer bound for additive codes can always be attained if the minimum distance $d$ is sufficiently large. Thus, also the determination of $n_q(r,h;s)$ is a finite problem for each triple of parameters $(q,r,h)$.

\section{At most two subspaces in a hyperplane -- generalized ovals}
\label{sec_oval}
An oval in the projective plane $\PGF{3}{q}$ is a set $\mathcal{O}$ of $q+1$ points such that no three points are on a line. If the field size $q$ is odd, then each oval is a conic and each set of $q+2$ points does contain a line incident with at least three points \cite{segre1955ovals}. This is different if $q$ is even, where other constructions exist and cardinality $q+2$ can indeed be reached, where the geometric object then is called hyperoval. Given an arbitrary set of points $\mathcal{S}$ in $\PGF{3}{q}$, a line that contains two points from $\mathcal{S}$ is called a secant and a line that contains exactly one point from $\mathcal{S}$ is called a tangent. Given an oval $\mathcal{O}$ in $\PGF{3}{q}$, where $q$ is odd, each point in $\mathcal{O}$ is contained in $q$ secants and a unique tangent. Applying the subfield construction to a hyperoval in $\PGF{3}{q^2}$, where $q$ is even, gives a faithful projective $2-\left(q^2+2,6,2,1\right)_q$ system $\mathcal{S}$. Projection through a point~$P$ then gives a projective $2-\left(q^2+2,5,2\right)_q$ system that is faithful iff $P$ is not contained in one of the elements from $\mathcal{S}$.
Arcs, ovals, and hyperovals were generalized by replacing their points with $h$-spaces and their maximum possible cardinality in $\PGF{3h}{q}$ can indeed be obtained by applying the subfield construction in $\PGF{3}{q^h}$ \cite{thas1971m}. For the relation to (translation) generalised quadrangles we refer to \cite{thas19734,thas1974translation}.

Starting from an oval in $\PGF{3}{q^2}$, applying the subfield construction and projection through a point, yields a projective $2-\left(q^2+1,5,2\right)_q$ system.
The following result shows that we can improve upon that.
\begin{proposition}
  \label{prop_construction_oval}
  $$n_q(5,2;2)\ge q^2+2$$
\end{proposition}
In Section~\ref{ss:constr} we will prove this in a more general form (and in Appendix~\ref{appendixC}, in a less general form, but with a completely different approach), but firstly we derive some upper bounds and discuss what we know about the values for small~$q$.
\subsection{An upper bound}\label{ss:ub}

In order to obtain an upper bound for $m_q(5,2;2)$ we can mimic the argument for the maximum cardinality of arcs in the projective plane $\PGF{3}{q}$.
\begin{lemma}
  \label{lemma_ub_oval_projection_parameters}
  We have $n_q(5,2;s)\le (s-1)(q^2+q+1)+1$ for each $s\ge 2$. If a projective system misses this number by less than $q$, then it has to be faithful and no two lines intersect in a point.
\end{lemma}
\begin{proof}
  Let $\mathcal{S}$ be a projective $2-(n,5,s)_q$ system. Assume that $L\in \mathcal{S}$ is an arbitrary $2$-dimensional element. Projection through $L$ yields a projective $2-(n-c,3,s-c)_q$ system $\mathcal{S}'$, where $c$ denotes the multiplicity of $L$ in $\mathcal{S}$. Since $n_q(3,2;s')=s'\cdot(q^2+q+1)$ we conclude the statement. If no element of $\mathcal{S}'$ has dimension $2$ then we assume that $L$ is spanned by two different  arbitrary $1$-dimensional elements in $\mathcal{S}$. Projection through $L$ yields a projective $2-(n-c,3,s-c)_q$ system $\mathcal{S}''$, where $c$ denotes the sum of the multiplicities of the points in $L$ w.r.t.\ $\mathcal{S}$. Here we can proceed as before. Otherwise, we clearly have $|\mathcal{S}|\le s$.

  For the second part we remark that each projective $2-(n',3,s')_q$ system satisfies $n'\le (q^2+q+1)s'-(q-1)$ if it is unfaithful or contains two lines that intersect in a point.
\end{proof}
So we e.g.\ have $n_q(5,2;2)\le q^2+q+2$ and each projective $2-(n,5,2,\mu)_q$ system with $n>q^2+2$ satisfies $\mu=1$ and is faithful.
These are generalized in p.(i) and p.(ii),
respectively, of the following theorem.

\begin{theorem}\label{th:generalized_bound}
  If integers $s$, $t$, $h$, $j$ satisfy
  $s\ge t\ge 2$, $h,j\ge 1$,
  then
  \begin{itemize}
      \item[\rm (i)] it holds
      \begin{equation}\label{eq:generalized_ubound}
          n_q(th{+}j,h;s)\le (s{-}t{+}1)\cdot\frac{[h+j]_q}{[j]_q} +t-1;
      \end{equation}
      \item[\rm (ii)] if there is a projective
      $h-(n,th+j,s)$ system $\mathcal{S}$
      such that
      $$\displaystyle n> \frac{(s{-}t{+}1)[h+j]_q - [j+1]_q}{[j]_q}+t,$$ then every $t$ elements of $\mathcal S$ span a $th$-space; in particular, $\mathcal S$ is a faithful $h-(n,th{+}j,s,1)_q$ system.
  \end{itemize}
\end{theorem}
Assertions~(i) and (ii) of Theorem~\ref{th:generalized_bound} are straightforward from subcases $i=0$ and $i=1$ of the following lemma.
\begin{lemma}\label{l:generalized_bound}
  If $\mathcal{S}$ is a projective
  $h-(n,th+j,s)_q$ system,
  $s\ge t\ge 2$, $h,j\ge 1$, such that
  some $(th-i)$-space, $i\in \{0,...,h\}$, includes at least~$t$ elements of~$\mathcal{S}$,
  then
  \begin{equation}\label{eq:l1i1+2}
      n\le \frac{(s{-}t{+}1)[h+j]_q - [i+j]_q}{[j]_q}+t
      .
  \end{equation}
\end{lemma}
\begin{proof}
Let $L''$ be a $(th-i)$-space containing as subspaces $t$ elements
$L_1$, \ldots, $L_t$ of~$\mathcal{S}$.
Consider a $(t-1)h$-space~$L'$ such that
$L_1,\ldots,L_{t-1}\subseteq L'\subseteq L''$.
Any element of
$\mathcal{S}\setminus \{L_1,\ldots,L_{t}\}$
lies in at least $[j]_q$ hyperplanes containing~$L'$.
There are $[h+j]_q$ such hyperplanes,
but $[i+j]_q$ of them are superspaces of~$L''$
and hence include all $L_1$, \ldots, $L_t$.
By the definition of a projective system,
each of the last $[i+j]_q$ hyperplanes
includes at most $s-t$ elements of $ \mathcal{S}\setminus\left\{L_1,\dots , L_t\right\}$, while each of the remaining $[h+j]_q-[i+j]_q$ includes at most $s-t+1$ elements of $ \mathcal{S}\setminus\left\{L_1,\dots, L_t\right\}$.
It follows that $\left|\mathcal{S}\setminus \left\{L_1, \dots, L_t\right\}\right|\cdot [j]_q\le (s-t+1)[h+j]_q - [i+j]_q$,
which proves~\eqref{eq:l1i1+2}.
\end{proof}

We remark that the proof of Lemma~\ref{lemma_ub_oval_projection_parameters} also shows that a hypothetical faithful projective $2-\left(q^2+q+2,5,2\right)_q$ system $\mathcal{S}$ has the property that the number of lines per hyperplane is always even, i.e.\ all $q^2+q+1$ lines through an element of $\mathcal{S}$ contain a second (disjoint) line. Considering the $q^2+q+1$ hyperplanes through a line that is not contained in $\mathcal{S}$ gives the condition that $q^2+q+2=q(q+1)+2$ is even, which is always satisfied -- different as in the situation for arcs in $\PGF{3}{q}$.

\begin{theorem}
  \label{thm_no_gen_hyperoval}
  If $s\ge 2$, $h,j\ge 1$, and $s$ contains a factor that is coprime to $q$, then
      \begin{equation}\label{eq:strict_ubound}
          n_q(2h{+}j,h;s) < (s-1)\cdot\frac{[h+j]_q}{[j]_q} +1;
      \end{equation}
  in particular,
$n_q(5,2;s)<(s-1)\cdot\left(q^2+q+1\right)+1$.
\end{theorem}
We remark that the case $h=j=1$, $s=2$ corresponds to the well-known fact that for odd $q$ there are no hyperovals.
\begin{proof}
  Assume that $\mathcal{S}$ is a faithful projective
  $h-\left(n,r,s,\mu'\right)_q$ system, where $n=(s-1)\cdot\frac{[h+j]_q}{[j]_q} +1$ and $r=2h+j$.
  Due to Theorem~\ref{th:generalized_bound}(ii),
  we have $\mu'=1$. Replacing each element of $\mathcal{S}$ by its contained $[h]_q$ points, we obtain a faithful projective $1-(n[h]_q,r,s',1)_q$ system, $s'=n[h{-}1]_q{+}sq^{h-1}$, i.e., a set $\mathcal{M}$ of $n[l]_q$ points in $\PGF{r}{q}$ such that every hyperplane contains at most $s'$ elements from~$\mathcal{M}$. Since, by the proof of Lemma~\ref{l:generalized_bound} ($t=2$, $i=0$), each hyperplane~$H$ contains either~$s$ or no elements from $\mathcal{S}$, we see that $H$ intersects~$\mathcal{M}$ in $n[h{-}1]_q+sq^{h-1}$ or $n[h{-}1]_q$ points. In other words, the linear code corresponding to~$\mathcal{M}$ is a projective two-weight $\left[n[h]_q,r,\{w_1,w_2\}\right]_q$ code, where $w_1=w_2+sq^{h-1}$.
  However, the difference of the occurring non-zero weights equals $sq^{h-1}$. Since it has to be a power of the characteristic of~$\FF_q$, see \cite[Corollary 2]{delsarte1972weights}, we obtain a contradiction if $s$ contains a factor that is coprime to~$q$.
\end{proof}

\subsection{Discussing small \texorpdfstring{$q$}{q}}\label{ss:smallQ}

So we e.g.\ have $n_q(5,2;2)< q^2+q+2$ for odd~$q$. For $s,h=2$, $j=1$,
the strongly regular graph corresponding to the utilized linear code has parameters $\left(q^5,k,\lambda,\mu\right)$, where $k=(q^2-1)\cdot\left(q^2+q+2\right)$, $\lambda=q^3+2q^2+q-2$, and $\mu=q^3+2q^2+3q+2$ (for the definition of strongly regular graphs and their connection with projective two-weight codes, we refer, e.g., to~\cite{BvM:SRG}). Interestingly enough, the case for $q=3$ is marked as \textit{open} at \url{https://aeb.win.tue.nl/graphs/srg/srgtab201-250.html}, \cite{BvM:SRG}.

 For $q=2$, the upper bound from Lemma~\ref{lemma_ub_oval_projection_parameters}
 can be attained. To this end, we remark that a vector space partition of type $3^1 2^{q^3}$ of $\PGF{5}{q}$, i.e.\ a faithful projective $2-\left(q^3,5,s,1\right)_q$ system whose elements all are disjoint to some plane $\pi$, has the property that each hyperplane contains either no or $q$ lines (depending on whether $\pi$ is contained in the hyperplane or not), i.e.\ we have $s=q$. Such a partition can e.g.\ be constructed by lifting MRD codes
or by retracting a spread, a partition of $\PGF{6}{q}$ into $q^3+1$ planes.

The case of field size $q=3$ can be solved exactly; we have $n_3(5,2;2)=12=q^2+q$. The nonequivalent faithful projective $2-(n,5,2)_3$ systems can be easily enumerated. There are exactly six faithful projective $2-(12,5,2)_3$ systems and no faithful projective $2-(13,5,2)_3$ system. More precisely, we have verified the numbers stated in \cite[Table 3, line $r=5$]{ball2025griesmer} and \cite[Remark 27]{ball2025griesmer}. The number of unfaithful projective $2-(n,5,2)_3$ systems is given by $1,5,21,26,9,1,1$ for $n=4,\dots,10$ and there is no unfaithful $2-(11,5,2)_3$ system. In principle, the multiset of points covered by a faithful projective system can admit several partitions. Interestingly enough, for our parameters all partitions are equivalent except for two cases for $n=10$, where trading switches eight lines. For more details on trades we refer e.g.\ to \cite{krotov2017minimum}.

For $q=4$ a strongly regular graph with parameters $(1024,330,98,110)$ indeed exists \cite{polhill2008generalizations}, even as linear codes, i.e.\ as a projective $[110,5,\{80,88\}]_4$ code \cite{dissett2000combinatorial} and as a projective $[330,10,\{160,176\}]_2$ code \cite{momihara2013strongly,momihara2014certain}. We have computationally checked that none of the ten $[110,5,\{80,88\}]_4$ codes listed in \cite[Appendix B]{dissett2000combinatorial} can be partitioned into $22$ lines, when considered as a set of points in $\PGF{5}{4}$. An example showing $n_4(5,2;2)\ge 20=q^2+q$ is given
in Appendix~\ref{sec_explicit_4}.
We have computationally checked that a hypothetical $2-(21,5;2)_4$ or a $2-(22,5;2)_4$ system cannot have a non-trivial automorphism group.

For $q=5$, we could not find a projective system better
than the $2-(27,5,2)_5$ system from Proposition~\ref{prop_generalized_construction_oval}. A (non-exhaustive) search with prescribed automorphism group of order at least~$4$ gives thousands of projective $2-(26,5,2)_5$ systems from only four equivalence classes,
none of them can be completed to a $2-(28,5,2)_5$ system, see also the discussion in Section~\ref{sec_multispreads}.

\subsection{Constructions}
\label{ss:constr}
A projective system from Proposition~\ref{prop_construction_oval}
can be constructed in more general form, starting from ovals in the projective plane $\PGF{3}{q^h}$ for $h\ge 2$.

\begin{proposition}
  \label{prop_generalized_construction_oval}
  For $h\ge 2$ we have $n_q(\lfloor2.5h \rfloor,h;2)\ge q^h+2$.
\end{proposition}
\begin{proof}
  W.l.o.g.\ we assume that $q$ is odd.
  We construct a required projective system in the following four steps.

  Step~1. Consider an oval $\mathcal{O}$ from $q^h+1$ points of the projective plane $\PGF{3}{q^h}$ and fix some point $P\in\mathcal{O}$.
  Let $L$ be the unique tangent line through~$P$.
  Denote $\mathcal{S}=\mathcal{O}\backslash\{P\}$.

  Step~2. Applying the subfield construction, i.e., treating the objects above as
  objects in $\PGF{3h}{q}$, we have a collection
  $\mathcal{O}'=\mathcal{S}'\cup\{P'\}$ of $q^h+1$ disjoint $h$-spaces
  and a $2h$-space~$L'$ including $P'$ as a subspace and disjoint
  with each $h$-space in~$\mathcal{S}'$.

  Step~3. Applying projection through an arbitrary $\lceil h/2 \rceil$-subspace of~$P'$, we get the following in $\PGF{\lfloor 2.5h \rfloor}{q}$:
  a collection $\mathcal{S}''$ of $q^h$ disjoint $h$-spaces,
 a $\lfloor  0.5h \rfloor$-space~$P''$,
  a $\lfloor  1.5h \rfloor$-space~$L''$.

  Step~4. Finally, we consider
  two $h$-spaces $\alpha_1$, $\alpha_2$ spanning $L''$ and each including $P''$
  (this is possible because $\dim(\alpha_1)+\dim(\alpha_2)=2h\ge \lfloor  0.5h \rfloor+\lfloor  1.5h \rfloor=\dim(P'')+\dim(L'')$).
  We claim that $\mathcal{S}'' \cup\{\alpha_1,\alpha_2\}$ is a required
  projective
  $h-(q^h,\lfloor 2.5h\rfloor,2)_q $ system. Indeed,
    \begin{itemize}
  \item[(i)]
  Any three $h$-spaces of $\mathcal{S}''$ span the whole space
  (they are images of tree elements of an oval);
  \item[(ii)]
  For $i\in\{1,2\}$, $a_i$ and any two $h$-spaces $\beta_1$, $\beta_2$ of $\mathcal{S}''$ span the whole space
  (indeed, $P''\subset a_i$ and the subspaces $P''$, $\beta_1$, $\beta_2$ are images of tree elements of an oval).
  \item[(iii)] Finally, $\alpha_1$, $\alpha_2$, and any $h$-space $\beta$ from $\mathcal{S}''$
  span the whole space (indeed, $\alpha_1$ and $\alpha_2$ span $L''$, while $L''$ and $\beta$ are images of the line $L$ and a point not in~$L$ in a projective plane).
  \end{itemize}
  So, any three elements of $\mathcal{S}'' \cup\{\alpha_1,\alpha_2\}$ span the whole space;
  therefore, each hyperplane includes
  at most two of them.
\end{proof}

We remark that it is known that for odd $q$ an arc of size $q$ in $\PGF{3}{q}$ can always be extended to an oval i.e.\ a conic \cite{segre1955curve,segre1955ovals}.

Underlying our geometric constructions seems the following recipe: Start with a linear code of a large field and consider it as a multiset of points. After removing a point (or possibly several) we apply the subfield construction and project through some subspaces contained in the image of a removed point. Afterwards it is sometimes possible to add more subspaces than removed before. Of course this approach is not always successful. Starting from an ovoid over $\FF_4$ we obtain $n_2(8, 2; 5)\ge 17$, which is tight. However, we also have $n_2(7, 2; 5)=17$, see e.g.~\cite[Theorem 14]{kurz2024additive}, i.e.\ no improvement possible for these parameters. In the proof of the following construction the {\lq\lq}removed{\rq\rq} point is denoted as~$S_1$.
If we add it as the last column to a generator matrix of the corresponding linear code we get  the generator matrix of the doubly-extended Reed--Solomon code -- an $s$-arc of size $q^h+1$ in geometrical terms, where $s$ and $h$ are parameters of our construction. The subspace $S_2$ is the analog of the tangent line through~$S_1$. But now we need that the following generalization of the {\lq\lq}tangent{\rq\rq} property: any $s-1$ points of the arc are disjoint with~$S_2$. For our situation we will present an algebraic proof.

\begin{theorem}
\label{prop_generalized_construction_oval_s}
  For $h,s\ge 2$ we have $n_q(hs+\lfloor0.5h \rfloor,h;s)\ge q^h+2$.
\end{theorem}
\begin{proof}
We start with the extended Reed--Solomon code $R(s)$ with the generator matrix
\begin{equation*}G_s =
\left(
    \begin{array}{cccccc}
      1 & 1& 1& \ldots   & 1   \\
      0 & \alpha^1& \alpha^2& \ldots   & \alpha^{q^h-1}   \\
      0 & \alpha^2& \alpha^4& \ldots   & \alpha^{2(q^h-1)}  \\
       \vdots&\vdots&\vdots& \ddots   &\vdots\\
      0 & \alpha^{s-1}& \alpha^{2s-2}& \ldots   & \alpha^{(s-1)(q^h-1)}  \\
      0 & \alpha^s& \alpha^{2s}& \ldots   & \alpha^{s(q^h-1)}
    \end{array}
    \right)\!,
\end{equation*}
where $\alpha$ is a primitive element of $\FF_{q^h}$. It is a linear $\left[q^h,s+1,q^h-s\right]_{q^h}$ code.
We consider the $(s+1)$-dimensional column space $V_{s+1}$ of the matrix above; the columns of the matrix~$G_s$ are treated as points ($1$-dimensional subspaces) $a_0$, \ldots,  $a_{q^h-1}$ in this space.

We denote by $S_i$ the $i$-dimensional subspace of the column space consisting of all columns with first $s+1-i$ zeros. We claim the following:

\begin{itemize}
    \item[(*)] \emph{The span of any $s+1-i$
    columns of $G_s$ is an $(s+1-i)$-dimensional space disjoint with $S_i$, $i=1,\ldots,s$.}
The claim follows from the non-zero determinant of the Vandermonde matrix formed by the upper $s+1-i$
positions of the chosen $s+1-i$ columns.
\end{itemize}

Now, we consider the $\FF_{q^h}$-space $V_{s+1}$
as an $(sh+h)$-dimensional $\FF_{q}$-space, points $a_0$, \ldots, $a_{q^h-1}$ as $h$-subspaces.
Let $\pi$ be projecting of $V_{s+1}$ onto an $(sh+\lfloor h/2 \rfloor)$-dimensional $\FF_{q}$-space such that
$\dim_{\FF_q}\pi(S_1) =\dim_{\FF_q} S_1-\lceil h/2 \rceil = \lfloor h/2 \rfloor$.
We choose two arbitrary $h$-subspaces $b$, $b'$ of $\pi(V_{s+1})$ such that $\pi(S_1)\subseteq b\cap b'$ and $ b+b' = \pi(S_2)$ (we can do this because $\dim b + \dim b' = h+h \ge \lfloor h/2 \rfloor + \lfloor h/2 \rfloor + h = \dim \pi(S_1) + \dim \pi(S_2)$).

To complete the proof, we have to show that
the $h$-subspaces
$\pi(a_0)$, \ldots, $\pi(a_{q^h-1})$, $b$, $b'$ form a projective $h-\left(q^h+2,hs+\lfloor0.5h \rfloor,s\right)_q$ system;
that is, any $s+1$ of them span the space $\pi(V_{s+1})$.

a) Any $s+1$ of $a_0$, \ldots, $a_{q^h-1}$ span $V_{s+1}$;
therefore, any $s+1$ of $\pi(a_0)$, \ldots, $\pi(a_{q^h-1})$ span $\pi(V_{s+1})$.

b) From (*), $i=1$, any $s$ of $\pi(a_0)$, \ldots, $\pi(a_{q^h-1})$ span an $sh$-space disjoint with $\pi(S_1)$.
Since $\pi(S_1)\subset b$, we have that $b$ (similarly, $b'$) and any $s$ of $\pi(a_0)$, \ldots, $\pi(a_{q^h-1})$ span
an $(sh+\lfloor h/2 \rfloor)$-space, i.e.,
$\pi(V_{s+1})$.

c) Finally, from (*), $i=2$, any $s-1$ of $\pi(a_0)$, \ldots, $\pi(a_{q^h-1})$ span an $(s-1)h$-space disjoint with $\pi(S_2)$.
Since $\pi(S_2) = b+b'$, we have that $b$, $b'$, and any $s-1$ of $\pi(a_0)$, \ldots, $\pi(a_{q^h-1})$ span an $((s-1)h+h+\lfloor h/2 \rfloor)$-space, i.e., again $\pi(V_{s+1})$.
\end{proof}

To compare with the upper bound, we mention here the following special case $t=s$ of Theorem~\ref{th:generalized_bound}. 

\begin{corollary}
\label{cor_ub_oval_projection_parameters}
  $ n_q(sh+j,h;s)\le
  \frac{q^{h+j}-1}{q^j-1}+s-1$.
  Moreover, if an $h-(n,sh+j,s;\mu)_q$ system is unfaithful or $\mu>1$, then $ n\le \frac{q^{h+j}-q^{j+1}}{q^j-1}+s$. 
  The last bound is tight for $h=s=2$, $j=1$.
\end{corollary}

The tightness of the last bound, which turns to $ n\le q^2+2$ in that special case, follows from constructions in the proofs of Proposition~\ref{prop_generalized_construction_oval} and Theorem~\ref{prop_generalized_construction_oval_s}.


\section{Prescribing automorphisms and integer linear programming}
\label{sec_ILP}

The determination of $n_q(r,h;s)$ is a hard problem in general, since it is already for the special case $h=1$, i.e.\ for linear codes. So, it makes sense to consider small parameters. For $h=2$ and $r\le 2$ we technically have $n_q(r,2;s)=\infty$ since we can take as many copies of the ambient space as elements, which are not contained in any hyperplane. In $PG(2,q)$ hyperplanes and lines coincide, so that $n_q(3,2;s)=[3]_q\cdot s$, i.e.\ we can take each of the $[3]_q$ lines $s$ times. Also the situation in $\PGF{4}{q}$ still allows a complete analytical solution, i.e.\ we have $n_q(4,2;s)=\tfrac{[4]_q}{[2]_q}\cdot s=\left(q^2+1\right)\cdot s$, see \cite[Theorem 8]{kurz2024additive}. The determination of $n_2(5,2;s)$, $n_2(6,2;s)$ was completed in \cite{bierbrauer2021optimal} and the determination of $n_2(7,2;s)$ was completed in \cite{kurz2024optimal}. Here we complete the determination of $n_3(5,2;s)$.

Projective systems can be easily formulated as feasible solutions of integer linear programs. As variables we use $x_S\in\mathbb{N}$ for each $h$-space in $\PGF{r}{q}$ (assuming a faithful projective system). The conditions are given by $\sum_{S\le H} x_S\le s$ for each hyperplane $H$. Since those ILPs are rather large even for small parameters $q$, $r$, and $h$, a common technique to reduce the problem size is to assume a subgroup of the automorphism group. Via this approach we can only obtain lower bounds for $n_q(r,h;s)$. In our subsequent lists of generator matrices for the subspaces of a projective system exponents mark orbit representatives and the value in the exponent is the orbit size.

\begin{theorem}
  \label{thm_n_3_5_2}
  We have
  \begin{itemize}
\item $n_3(5,2;13t)=121t$ for $t\ge 1$;
\item $n_3(5,2;13t-1)=121t-10$ for $t\ge 1$;
\item $n_3(5,2;13t-2)=121t-20$ for $t\ge 1$;
\item $n_3(5,2;13t-3)=121t-30$ for $t\ge 1$;
\item $n_3(5,2;13t-4)=121t-40$ for $t\ge 1$;
\item $n_3(5,2;13t-5)=121t-50$ for $t\ge 1$;
\item $n_3(5,2;13t-6)=121t-60$ for $t\ge 1$;
\item $n_3(5,2;13t-7)=121t-67$ for $t\ge 1$;
\item $n_3(5,2;13t-8)=121t-77$ for $t\ge 1$;
\item $n_3(5,2;13t-9)=121t-87$ for $t\ge 1$;
\item $n_3(5,2;13t-10)=121t-94$ for $t\ge 1$;
\item $n_3(5,2;13t-11)=121t-104$ for $t\ge 2$ and $n_3(5,2;2)=12$;
\item $n_3(5,2;13t-12)=121t-114$ for $t\ge 2$ and  $n_3(5,2;1)=1$.
\end{itemize}
\end{theorem}
\begin{proof}
  Due to \cite[Theorem 19]{kurz2024additive} it remains to show the lower bounds $n_3(5,2;4)\ge 34$ and $n_3(5,2;5)\ge 44$. To this end, we have prescribed the cyclic group of order three  
  generated by
  $$
  \begin{pmatrix}
    1 & 1 & 0 & 0 & 0 \\
    0 & 1 & 0 & 0 & 0 \\
    0 & 0 & 1 & 1 & 0 \\
    0 & 0 & 0 & 1 & 1 \\
    0 & 0 & 0 & 0 & 1
  \end{pmatrix}
  $$
  as an automorphism for the desired faithful $2-(n,5,s)_3$ systems and utilized an integer linear programming formulation. Explicit lists of lines are given as follows.

\noindent
$n_3(5,2;4)\ge 34$:
\def\subspace#1#2{\left(\begin{smallmatrix}#1\end{smallmatrix}\right)#2}
$\subspace{00010\\00001}{^1}$,
$\subspace{10002\\00110}{^3}$,
$\subspace{11002\\00121}{}$,
$\subspace{12002\\00100}{}$,
$\subspace{10011\\00120}{^3}$,
$\subspace{11012\\00102}{}$,
$\subspace{12010\\00112}{}$,
$\subspace{10021\\00122}{^3}$,
$\subspace{11020\\00101}{}$,
$\subspace{12022\\00111}{}$,
$\subspace{10000\\01021}{^3}$,
$\subspace{10010\\01020}{}$,
$\subspace{10022\\01022}{}$,
$\subspace{10010\\01101}{^3}$,
$\subspace{10101\\01122}{}$,
$\subspace{10200\\01111}{}$,
$\subspace{10020\\01100}{^3}$,
$\subspace{10112\\01121}{}$,
$\subspace{10212\\01110}{}$,
$\subspace{10020\\01102}{^3}$,
$\subspace{10111\\01120}{}$,
$\subspace{10210\\01112}{}$,
$\subspace{10012\\01200}{^3}$,
$\subspace{10120\\01220}{}$,
$\subspace{10220\\01212}{}$,
$\subspace{10012\\01222}{^3}$,
$\subspace{10102\\01211}{}$,
$\subspace{10210\\01202}{}$,
$\subspace{10021\\01200}{^3}$,
$\subspace{10100\\01220}{}$,
$\subspace{10201\\01212}{}$,
$\subspace{10121\\01010}{^3}$,
$\subspace{10122\\01011}{}$,
$\subspace{10122\\01012}{}$.

\medskip

\noindent
$n_3(5,2;5)\ge 44$:
$\subspace{00100\\00010}{^3}$,
$\subspace{00100\\00012}{}$,
$\subspace{00102\\00011}{}$,
$\subspace{01010\\00001}{^1}$,
$\subspace{01020\\00001}{^1}$,
$\subspace{10010\\00101}{^3}$,
$\subspace{11011\\00111}{}$,
$\subspace{12012\\00122}{}$,
$\subspace{10020\\00101}{^3}$,
$\subspace{11022\\00111}{}$,
$\subspace{12021\\00122}{}$,
$\subspace{10000\\01000}{^1}$,
$\subspace{10012\\01001}{^1}$,
$\subspace{10021\\01002}{^1}$,
$\subspace{10001\\01110}{^3}$,
$\subspace{10101\\01100}{}$,
$\subspace{10210\\01121}{}$,
$\subspace{10002\\01120}{^3}$,
$\subspace{10111\\01112}{}$,
$\subspace{10200\\01102}{}$,
$\subspace{10011\\01101}{^3}$,
$\subspace{10102\\01122}{}$,
$\subspace{10201\\01111}{}$,
$\subspace{10020\\01101}{^3}$,
$\subspace{10110\\01122}{}$,
$\subspace{10211\\01111}{}$,
$\subspace{10021\\01102}{^3}$,
$\subspace{10112\\01120}{}$,
$\subspace{10211\\01112}{}$,
$\subspace{10001\\01210}{^3}$,
$\subspace{10100\\01201}{}$,
$\subspace{10222\\01221}{}$,
$\subspace{10002\\01220}{^3}$,
$\subspace{10120\\01212}{}$,
$\subspace{10202\\01200}{}$,
$\subspace{10010\\01202}{^3}$,
$\subspace{10122\\01222}{}$,
$\subspace{10220\\01211}{}$,
$\subspace{10012\\01201}{^3}$,
$\subspace{10122\\01221}{}$,
$\subspace{10221\\01210}{}$,
$\subspace{10022\\01202}{^3}$,
$\subspace{10102\\01222}{}$,
$\subspace{10201\\01211}{}$.
\end{proof}


Our next aim is to obtain bounds for $n_4(5,2;s)$ and $n_5(5,2;s)$. As mentioned before, the determination of $n_q(r,h;s)$ as a function of $s$ is a finite problem since the Griesmer bound is always attained if $s$ is sufficiently large, see \cite{kurz2024additive}. In order to indicate the underlying constructions we have to introduce some notation. For each subspace $S$ in $\PGF{r}{q}$ we denote its characteristic function by $\chi_S$, i.e.\ we have $\chi_S(P)=1$ if $P\le S$ and $\chi_S(P)=0$ otherwise.
\begin{definition}
  We say that a multiset of points $\mathcal{M}$ in $\PGF{r}{q}$ is $h$-partitionable if there exist $h$-spaces $S_1,\dots,S_l$, for some integer $l$, such that $\mathcal{M}=\sum_{i=1}^l \chi_{S_i}$, i.e.\ $\mathcal{M}$ can be partitioned into $h$-spaces.
\end{definition}
In order to describe Solomon--Stiffler type constructions we choose a chain of subspaces $S_1\le S_2\le\dots\le S_r$ in $\PGF{r}{q}$, where $S_i$ has dimension $i$.
\begin{definition}
  We say that $\sum_{i=1}^r a_i[i]$ is $h$-partitionable over $\FF_q$ if the multiset of
  points $\sum_{i=1}^r a_i\chi_{S_i}$ in $\PGF{r}{q}$ is $h$-partitionable, where $a_i\in\mathbb{Z}$ for all $1\le i\le r$.
\end{definition}
So, we trivially have that $[2]$ is $2$-partitionable over $\FF_q$ and the existence of line spreads in $\PGF{4}{q}$ implies that $[4]$ is $2$-partitionable over $\FF_q$. Due to the number of points $[5]$ is not $2$-partitionable over $\FF_q$ while $(q+1)\cdot [5]$ is. For the details of this and the other constructions mentioned in the two subsequent theorems we refer the reader to \cite{kurz2024additive}.

\begin{theorem}
  \label{thm_n_4_5_2}
  For $n_4(5,2;s)$ we have the bounds and exact values stated in Table~\ref{table_n_4_5_2}.
\end{theorem}

\begin{table}[htp!]
\def\lp#1{#1^{\mathrm{ilp}}}																			
$$\begin{array}{||r|c||r|c||l@{\,}c@{\,}l|c||l||}																			\hline
s	&	n_4(5,2;s)	&	s	&	n_4(\ldots)	&	s	&&		&	n_4(5,2;s)	&	\mbox{construction}	\\\hline
1	&	-	&	22	&	354	&		21t-20	& \ge &	22	&	341t-328	&	t\ge 4:\ (5t{-}4){\cdot}[5]{-}3[4]{-}[3]	\\
2	&	\lp{20}-22^{\text{L.\ref{lemma_ub_oval_projection_parameters}}}	&	23	&	{371}	&	21t-19	& \ge &	23	&	341t-311	&	t\ge 3:\ (5t{-}4){\cdot}[5]{-}2[4]{-}[3]	\\
3	&	\lp{39}-42^{\text{T.\ref{thm_no_gen_hyperoval}}}	&	24	&	388	&	21t-18	& \ge &	24	&	341t-294	&	t\ge 2:\ (5t{-}4){\cdot}[5]{-}[4]{-}[3]	\\
4	&	64	&	25	&	405	&	21t-17	& \ge &	25	&	341t-277	&	t\ge 1:\ (5t{-}4){\cdot}[5]{-}[3]	\\
5	&	\lp{75}-77	&	26	&	418	&	21t-16	& \ge &	26	&	341t-264	&	t\ge 4:\ (5t{-}3){\cdot}[5]{-}3[4]{-}2[3]	\\
6	&	\lp{90}-94	&	27	&	435	&	21t-15	& \ge &	27	&	341t-247	&	t\ge 3:\ (5t{-}3){\cdot}[5]{-}2[4]{-}2[3]	\\
7	&	\lp{107}-111	&	28	&	452	&	21t-14	& \ge &	28	&	341t-230	&	t\ge 2:\ (5t{-}3){\cdot}[5]{-}[4]{-}2[3]	\\
8	&	128	&	29	&	469	&		21t-13	& \ge &	29	&	341t-213	&	t\ge 1:\ (5t{-}3){\cdot}[5]{-}2[3]	\\
9	&	\lp{141}	&	30	&	482	&	21t-12	& \ge &	30	&	341t-200	&	t\ge 4:\ (5t{-}2){\cdot}[5]{-}3[4]{-}3[3]	\\
10	&	\lp{156}-158	&	31	&	499	&	21t-11	& \ge &	31	&	341t-183	&	t\ge 3:\ (5t{-}2){\cdot}[5]{-}2[4]{-}3[3]	\\
11	&	\lp{175}	&	32	&	516	&	21t-10	& \ge &	32	&	341t-166	&	t\ge 2:\ (5t{-}2){\cdot}[5]{-}[4]{-}3[3]	\\
12	&	192	&	33	&	533	&		21t-9	& \ge &	33	&	341t-149	&	t\ge 1:\ (5t{-}2){\cdot}[5]{-}3[3]	\\
13	&	205	&	34	&	546	&	21t-8	& \ge &	34	&	341t-136	&	t\ge 4:\ 5t{\cdot}[5]{-}3[4]{-}4[3]	\\
14	&	\lp{222}	&	35	&	563	&	21t-7	& \ge &	35	&	341t-119	&	t\ge 3:\ 5t{\cdot}[5]{-}2[4]{-}4[3]	\\
15	&	239	&	36	&	580	&	21t-6	& \ge &	36	&	341t-102	&	t\ge 2:\ 5t{\cdot}[5]{-}[4]{-}4[3]	\\
16	&	256	&	37	&	597	&	21t-5	& \ge &	37	&	341t-85	&	t\ge 1:\ (5t{-}1){\cdot}[5]{-}4[3]	\\
17	&	273	&	38	&	614		&	21t-4	& \ge &	38	&	341t-68	&	t\ge 1:\ (5t{-}1){\cdot}[5]+[2]{-}4[1]	\\
18	&	\lp{290}	&	39	&	631	&	21t-3	& \ge &	39	&	341t-51	&	t\ge 3:\ 5t{\cdot}[5]{-}3[4]	\\
19	&	\lp{307}	&	40	&	648	&		21t-2	& \ge &	40	&	341t-34	&	t\ge 2:\ 5t{\cdot}[5]{-}2[4]	\\
20	&	324	&	41	&	665	&	21t-1	& \ge &	41	&	341t-17	&	t\ge 1:\ 5t{\cdot}[5]{-}[4]	\\
21	&	341	&	42	&	682	&	21t	& \ge &	42	&	341t	&	t\ge 1:\ 5t{\cdot}[5]	\\
\hline\end{array}$$							
\caption{Bounds for $n_4(5,2;s).$}
\label{table_n_4_5_2}
\end{table}
\begin{proof}
  The upper bound $n_4(5,2;2)\le 22$ follows from Lemma~\ref{lemma_ub_oval_projection_parameters}, $n_4(5,2;3)\le 42$ follows from Theorem~\ref{thm_no_gen_hyperoval}, and all other upper bounds follow from the Griesmer bound.
  We have that $5[5]$, $[4]$, $[2]$, $[5]-[3]$, $5[5]-[4]-4[2]$, $4[5]+[2]-4[1]$,
  and $[5]+4[1]$ are $2$-partitionable over $\FF_4$. These constructions imply that the expressions mentioned in the rightmost column of Table~\ref{table_n_4_5_2} are $2$-partitionable over $\FF_4$. For more details we refer to \cite{kurz2024additive}. Here we only use those constructions for $s\in\{4,17,20,21\}$.
  In Section~\ref{sec_oval} we have stated an example showing $n_4(5,2;2)\ge 20$.
  Via integer linear programming computations we have found the following examples, referring to Section~\ref{sec_explicit_4} in the appendix for explicit lists of generator matrices and used automorphism groups: $n_4(5,2;3)\ge 39$, $n_4(5,2;5)\ge 75$, $n_4(5,2;6)\ge 90$, $n_4(5,2;7)\ge 107$, $n_4(5,2;9)\ge 141$, $n_4(5,2;11)\ge 175$, $n_4(5,2;14)\ge 222$, $n_4(5,2;18)\ge 290$, and $n_4(5,2;19)\ge 307$.
  All other lower bounds follow from 
  $n_q(5,2;s_1+s_2)\ge n_q(5,2;s_1)+n_q(5,2;s_2)$.
\end{proof}

\begin{theorem}
  \label{thm_n_5_5_2}
  For $n_5(5,2;s)$ we have the bounds and exact values stated in Table~\ref{table_n_5_5_2}.
\end{theorem}
\begin{proof}

  The upper bounds $n_5(5,2;2)\le 31$, $n_5(5,2;3)\le 62$, $n_5(5,2;3)\le 93$ follow from Theorem~\ref{thm_no_gen_hyperoval} and all other upper bounds follow from the Griesmer bound.
  We have that $6[5]$, $[4]$, $[2]$, $[5]-[3]$, $6[5]-[4]-5[2]$, $6[5]+[2]-5[1]$,
  and $[5]+5[1]$ are $2$-partitionable over $\FF_5$. These constructions imply that the expressions mentioned in the rightmost column of Table~\ref{table_n_5_5_2} are $2$-partitionable over $\FF_5$. For more details we refer to \cite{kurz2024additive}. Here we only use those constructions for $s\in\{5,25,26,30,31,35,60\}$.
  In Proposition~\ref{prop_construction_oval} we have described a general construction showing $n_5(5,2;2)\ge 27$.
  Via integer linear programming computations we have found the following examples, referring to Section~\ref{sec_explicit_5} in the appendix for explicit lists of generator matrices. Prescribing the cyclic group of order
  five generated by
  $$
    \begin{pmatrix}
    1 & 0 & 0 & 0 & 0 \\
    0 & 1 & 1 & 0 & 0 \\
    0 & 0 & 1 & 1 & 0 \\
    0 & 0 & 0 & 1 & 1 \\
    0 & 0 & 0 & 0 & 1
    \end{pmatrix}
  $$
yields $n_5(5,2;3)\ge 50$.
  The cyclic group of order eleven generated by
  $$
    \begin{pmatrix}
    0 & 0 & 0 & 0 & 1\\
    1 & 0 & 0 & 0 & 4\\
    0 & 1 & 0 & 0 & 4\\
    0 & 0 & 1 & 0 & 1\\
    0 & 0 & 0 & 1 & 3
    \end{pmatrix}
  $$
  gives $n_5(5,2;4)\ge 77$
  and $n_5(5,2;6)\ge 132$.
  The cyclic group of order fifteen generated by
  $$
    \begin{pmatrix}
    1 & 0 & 0 & 0 & 0 \\
    0 & 0 & 1 & 0 & 0 \\
    0 & 4 & 4 & 1 & 0 \\
    0 & 0 & 0 & 0 & 1 \\
    0 & 0 & 0 & 4 & 4
    \end{pmatrix}
  $$
  gives $n_5(5,2;7)\ge 157$, $n_5(5,2;11)\ge 260$, $n_5(5,2;16)\ge 391$, $n_5(5,2;17)\ge 412$,  $n_5(5,2;21)\ge 521$, $n_5(5,2;22)\ge 542$, $n_5(5,2;37)\ge 927$,
  $n_5(5,2;39)\ge 969$,
  $n_5(5,2;44)\ge 1096$,
  $n_5(5,2;48)\ge 1203$,
  $n_5(5,2;49)\ge 1224$,
  $n_5(5,2;54)\ge 1354$,
  $n_5(5,2;58)\ge 1458$, and
  $n_5(5,2;59)\ge 1484$.
  The non-commutative group of order fifty-five generated by
  $$
    \begin{pmatrix}
    0 & 0 & 0 & 0 & 1\\
    1 & 0 & 0 & 0 & 4\\
    0 & 1 & 0 & 0 & 4\\
    0 & 0 & 1 & 0 & 1\\
    0 & 0 & 0 & 1 & 3
    \end{pmatrix}
    \quad\text{and}\quad
    \begin{pmatrix}
    0 & 0 & 2 & 1 & 0\\
    2 & 1 & 3 & 3 & 4\\
    3 & 3 & 2 & 2 & 3\\
    2 & 2 & 0 & 0 & 4\\
    0 & 0 & 0 & 0 & 2
    \end{pmatrix}
  $$
  gives 
  $n_5(5,2;12)\ge 286$,
  $n_5(5,2;8)\ge 176$,
  $n_5(5,2;13)\ge 308$,  
  $n_5(5,2;14)\ge 330$,
  and $n_5(5,2;24)\ge 594$.
  The cyclic group of order $71$ generated by
  $$
    \begin{pmatrix}
    0 & 0 & 0 & 0 & 1\\
    1 & 0 & 0 & 0 & 3\\
    0 & 1 & 0 & 0 & 4\\
    0 & 0 & 1 & 0 & 0\\
    0 & 0 & 0 & 1 & 0
    \end{pmatrix}
  $$
  gives $n_5(5,2;23)\ge 568$.
  All other lower bounds follow from $n_q(5,2;s_1+s_2)\ge n_q(5,2;s_1)+n_q(5,2;s_2)$.
\end{proof}
Bounds on the sizes of arcs in $\PGF{3}{25}$ can e.g.\ be found in \cite{braun2019new}.

\begin{table}[htp!]
\def\lp#1{#1^{\mathrm{ilp}}}																			
\footnotesize$$\begin{array}{||r|c||r|c||c|c||l||}																			
\hline 
s	&	n_5(5,2;s)	&	s	&	n_5(5,2;s)	&			s	&	n_5(5,2;s)	&	\mbox{construction}	\\\hline
1	&	-	&	32	&	792-802	&	31t-30	 \ge 	63	&	781t-760	&	t\ge 5:\ (6t{-}5) {\cdot}[5] {-}4[4]{-} [3]	\\
2	&	27^{\text{P.\ref{prop_construction_oval}}}-31^{\text{T,\ref{thm_no_gen_hyperoval}}}	&	33	&	818-828	&	31t-29	\ge	64	&	781t-734	&	t\ge 4:\  (6t{-}5) {\cdot}[5]{-}3[4]{-}[3]	\\
3	&	\lp{50}-62^{\text{T,\ref{thm_no_gen_hyperoval}}}	&	34	&	844-854	&	31t-28	\ge	65	&	781t-708	&	t\ge 3:\ (6t{-}5) {\cdot}[5]{-}2[4]{-}[3]	\\
4	&	\lp{77}-93^{\text{T,\ref{thm_no_gen_hyperoval}}}	&	35	&	880	&		31t-27	\ge	35	&	781t-682	&	t\ge 2:\ (6t{-}5){\cdot} [5]{-}[4]{-}[3]	\\
5	&	125	&	36	&	906	&		31t-26	\ge	5	&	781t-656	&	t\ge 1:\ (6t{-}5){\cdot}[5]{-}[3]	\\
6	&	\lp{132}-146	&	37	&	\lp{927}	&		31t-25	\ge	37	&	781t-635	&	t\ge 5:\ (6t{-}4){\cdot}[5]{-}4[4]{-}2[3]	\\
7	&	\lp{157}-172	&	38	&	941-953	&	31t-24	\ge	69	&	781t-609	&	t\ge 4:\ (6t{-}4){\cdot}[5]{-}3[4]{-}2[3]	\\
8	&	\lp{176}-198	&	39	&	\lp{969}-979	&	31t-23	\ge	70	&	781t-583	&	t\ge 3:\ (6t{-}4){\cdot}[5]{-}2[4]{-}2[3]	\\
9	&	202-224	&	40	&	1005	&		31t-22	\ge	40	&	781t-557	&	t\ge 2:\ (6t{-}4){\cdot}[5]{-}[4]{-}2[3]	\\
10	&	250	&	41	&	1031	&	31t-21	\ge	10	&	781t-531	&	t\ge 1:\ (6t{-}4){\cdot}[5]{-}2[3]	\\
11	&	\lp{260}-271	&	42	&	1052	&	31t-20	\ge	42	&	781t-510	&	t\ge 5:\ (6t{-}3){\cdot}[5]{-}4[4]{-}3[3]	\\
12	&	\lp{286}-297	&	43	&	1067-1078	&	31t-19	\ge	74	&	781t-484	&	t\ge 4:\ (6t{-}3){\cdot}[5]{-}3[4]{-}3[3]	\\
13	&	\lp{308}-323	&	44	&	\lp{1096}-1104	&	31t-18	\ge	75	&	781t-458	&	t\ge 3:\ (6t{-}3) {\cdot}[5]{-}2[4]{-}3[3]	\\
14	&	\lp{330}-349	&	45	&	1130	&	31t-17	\ge	45	&	781t-432	&	t\ge 2:\ (6t{-}3){\cdot}[5]{-}[4]{-}3[3]	\\
15	&	375	&	46	&	1156	&	31t-16	\ge	15	&	781t-406	&	t\ge 1:\ (6t{-}3){\cdot}[5]{-}3[3]	\\
16	&	\lp{391}-396	&	47	&	1177	&	31t-15	\ge	47	&	781t-385	&	t\ge 5:\ (6t{-}2){\cdot}[5]{-}4[4]{-}4[3]	\\
17	&	\lp{412}-422	&	48	&	\lp{1203}	&	31t-14	\ge	48	&	781t-359	&	t\ge 4:\ (6t{-}2){\cdot}[5]{-}3[4]{-}4[3]	\\
18	&	433-448	&	49	&	\lp{1224}-1229	&	31t-13	\ge	80	&	781t-333	&	t\ge 3:\ (6t{-}2){\cdot}[5]{-}2[4]{-}4[3]	\\
19	&	455-474	&	50	&	1255	&	31t-12	\ge	50	&	781t-307	&	t\ge 2:\ (6t{-}2){\cdot}[5]{-}[4]{-}4[3]	\\
20	&	500	&	51	&	1281	&	31t-11	\ge	20	&	781t-281	&	t\ge 1:\ (6t{-}2){\cdot}[5]{-}4[3]	\\
21	&	\lp{521}	&	52	&	1302	&	31t-10	\ge	21	&	781t-260	&	t\ge 5:\ (6t{-}1){\cdot}[5]{-}4[4]{-}5[3]	\\
22	&	\lp{542}-547	&	53	&	1328	&	31t-9	\ge	53	&	781t-234	&	t\ge 4:\ (6t{-}1){\cdot} [5]{-}3[4]{-}5[3]	\\
23	&	\lp{568}-573	&	54	&	\lp{1354}	&	31t-8	\ge	54	&	781t-208	&	t\ge 3:\ (6t{-}1){\cdot}[5]{-}2[4]{-}5[3]	\\
24	&	\lp{594}-599	&	55	&	1380	&	31t-7	\ge	55	&	781t-182	&	t\ge 2:\ (6t{-}1){\cdot}[5]{-}[4]{-}5[3]	\\
25	&	625	&	56	&	1406	&	31t-6	\ge	25	&	781t-156	&	t\ge 1:\ (6t{-}1){\cdot}[5]{-}5[3]	\\
26	&	651	&	57	&	1432	&	31t-5	\ge	26	&	781t-130	&	t\ge 1:\ (6t{-}1) {\cdot}[5]{+}[2]{-}5[1]	\\
27	&	667-677	&	58	&	\lp{1458}	&	31t-4	\ge	58	&	781t-104	&	t\ge 4:\ 6t{\cdot}[5]{-}4[4]	\\
28	&	693-703	&	59	&	\lp{1484}	&	31t-3	\ge	59	&	781t-78	&	t\ge3:\ 6t{\cdot}[5]{-}3[4]	\\
29	&	719-729	&	60	&	1510	&	31t-2	\ge	60	&	781t-52 	&	t\ge 2:\ 6t{\cdot}[5]{-}2[4];	\\
30	&	755	&	61	&	1536	&	31t-1	\ge	30	&	781t-26	&	t\ge 1:\ 6t{\cdot}[5]{-} [4]	\\
31	&	781	&	62	&	1562	&	31t	\ge	31	&	781t	&	t\ge 1:\ 5t{\cdot}[5]	\\
\hline\end{array}$$																			

\caption{Bounds for $n_5(5,2;s).$}
\label{table_n_5_5_2}
\end{table}

\bigskip

We would like to remark that instead of searching projective systems directly using integer linear programming we can also search for linear codes or multisets of points associated to a projective system, see Lemma~\ref{lemma_linear_code}.
The nonexistence of such a multiset, with a prescribed automorphism group, guarantees
the nonexistence of a projective system with the same automorphism group 
(any automorphism of a projective system $\mathcal{S}$ is necessarily an automorphism
of the multiset of points $\mathcal{P}(\mathcal{S})$, but not vice versa); 
sometimes this gives a fast way to reject putative automorphism groups.
However, the existence of such a multiset
does not imply the existence of a projective system (to get a projective system, we additionally need to partition the multiset of points into $h$-spaces, which is not always possible); it only means that the corresponding value cannot be cannot be rejected by the strong coding upper bound.

In the end of this section, we establish the existence of $8$-ary additive 
$[51,3,44]_2^3$ and $[51,3,44]_2^3$
codes, which outperform the best linear codes
with the corresponding parameters.
\begin{theorem}
  $n_2(9,3;7)\ge 51>49=\overline n_2(9,3;7)$,  $n_2(9,3;10)\ge 76>74=\overline n_2(9,3;10)$.
\end{theorem}
\begin{proof}
Projective $3-(51,9,7)_2$ and $3-(76,9,10)_2$
were found with the prescribed noncommutative
automorphism group of order~$21$ generated by
$$\left(\begin{matrix}0&0&1&0&0&0&0&0&0\\1&0&0&0&0&0&0&0&0\\0&1&1&0&0&0&0&0&0\\0&0&0&0&0&1&0&0&0\\0&0&0&1&0&1&0&0&0\\0&0&0&0&1&0&0&0&0\\0&0&0&0&0&0&0&0&1\\0&0&0&0&0&0&1&0&1\\0&0&0&0&0&0&0&1&0\end{matrix}\right), \quad \left(\begin{matrix}1&1&0&0&0&0&0&0&0\\0&1&1&0&0&0&0&0&0\\0&1&0&0&0&0&0&0&0\\0&0&0&1&0&0&0&0&0\\0&0&0&0&1&1&0&0&0\\0&0&0&0&1&0&0&0&0\\0&0&0&0&0&0&0&1&1\\0&0&0&0&0&0&0&0&1\\0&0&0&0&0&0&1&1&0\end{matrix}\right);$$
the orbit representatives are
\def\subspace#1#2{\left(\!\begin{smallmatrix}#1\end{smallmatrix}\!\right)^{\!\makebox[0mm][l]{$\scriptscriptstyle #2$}}}
$\subspace{000000100\\000000010\\000000001}{1}$,
$\subspace{010011111\\001001011\\000100111}{21}$,
$\subspace{010010011\\001011001\\000101111}{21}$,
$\subspace{100000000\\010000000\\001000000}{1}$,
$\subspace{100011101\\010010000\\001001000}{7}$ and
$\subspace{000100000\\000010000\\000001000}{1}$,
$\subspace{000100100\\000010010\\000001001}{3}$,
$\subspace{000100101\\000010111\\000001011}{1}$,
$\subspace{010001000\\001011000\\000000101}{7}$,
$\subspace{010001010\\001000101\\000101111}{21}$,
$\subspace{100000000\\010000000\\001000000}{1}$,
$\subspace{100011010\\010001110\\001001111}{21}$,
$\subspace{100001001\\010011111\\001010011}{7}$,
$\subspace{100111100\\010010111\\001001011}{7}$,
$\subspace{100110001\\010001111\\001011011}{7}$, respectively.

The values for $\overline n_2(\ldots)$
are retrieved from \url{https://www.codetables.de/}. 
\end{proof}

\section{Multispreads in PG\texorpdfstring{$(4,q)$}{(4,q)}}
\label{sec_multispreads}

Multispreads are a special type
of projective $h-(n,k,s)_q$ systems $\mathcal S$ (not necessarily faithful)
that correspond to one-weight $q^h$-ary codes (to be exact, $\FF_q$-linear one-weight $[n,k/h,n-s]_q^h$-codes), that is, having the property that every hyperplane includes exactly~$s$ spaces from~$\mathcal S$.
We will indicate this ``exactly'' by framing $s$ with braces as follows:
$h-(n,k,\{s\})_q$,
Multispreads can be alternatively defined as follows.
\begin{definition}\label{def:multispread}
A collection $\mathcal S$ of subspaces of~$\FF^k$ of dimension~$h$ is called a $(\lambda,\mu)_q^{h,k}$-multispread if for every point (non-zero vector) $v$ it holds
\begin{equation*}
  \sum_{S\in\mathcal S,\, v\in S}q^{h-\dim S} = \mu,
  \qquad
  \sum_{S\in\mathcal S}(q^{h-\dim S}-1) = \lambda,
\end{equation*}
i.e., every nonzero vector is covered $\mu$ times and  the all-zero vector is covered $\lambda+|\mathcal{S}|$ times, where a subspace~$S$ contributes $q^{h-\dim S}$ to the multiplicity of covering.
\end{definition}
The parameters of the projective $h-(n,k,\{s\})_q$ system corresponding to  a $(\lambda,\mu)_q^{h,k}$-multispread can be found from the following equations \cite[Theorem~3]{KroMog:multispread}:
\begin{equation*}
    s = n - q^{k-h}\mu, \qquad
    (q^h - 1)s = (q^{k-h}- 1)\mu+\lambda.
\end{equation*}
As shown in \cite{KroMog:multispread}, the natural necessary condition
\begin{equation*}
    \lambda \equiv -\mu(q^k-1) \bmod q^h-1
\end{equation*}
is sufficient for the existence of
$(\lambda,\mu)_q^{h,k}$-multispreads for $h=2$,
$\mu\ge q$,
which completely characterizes the parameters of $\FF_q$-linear $q^2$-ary one-weight codes.
To show the same for $h=3$ (which corresponds to $\FF_q$-linear $q^3$-ary one-weight codes), it is sufficient, for a given prime power~$q$, to find $q-1$ of $(\lambda,\mu)_q^{3,5}$ multispreads
($\mu = q+1, q+2, \ldots,$)
that correspond to
projective $3-(n,5,\{2\})_q$-systems,
starting with $n=(q^2+1)+(q^3+1)$
($q^2+1$ $2$-spaces and $q^3+1$ $3$-spaces) and with the step $(-q-1)+(q^2+q+1)$.
The next four paragraphs show what we know for $q=2,3,4,5$.

{\boldmath $q=2$.} Projective $3-(5{+}9,5,\{2\})_2$ systems ($(5,3)^{3,5}_2$-multispreads) were found and classified up to equivalence in~\cite[Appendix A.1]{KroMog:multispread}. The most symmetric are four multispreads with automorphism group of order~$12$ generated by $$\footnotesize
\left(\begin{array}{@{\ }c@{\ \ }c@{\ }c@{\ \ }c@{\ }c@{\ }}
1&0&0&0&0\\
0&0&1&0&0\\[-0.2em]
0&1&1&0&0\\
0&0&0&0&1\\[-0.2em]
0&0&0&1&1
\end{array}\right)
,\quad
\left(\begin{array}{@{\ }c@{\ \ }c@{\ }c@{\ \ }c@{\ }c@{\ }}
1&0&0&0&0\\
0&0&1&0&0\\[-0.2em]
0&1&0&0&0\\
0&0&0&0&1\\[-0.2em]
0&0&0&1&0
\end{array} \right)
,\quad
\left(\begin{array}{@{\ }c@{\ \ }c@{\ }c@{\ \ }c@{\ }c@{\ }}
1&0&0&0&0\\
0&0&0&1&0\\[-0.2em]
0&0&0&0&1\\
0&1&0&0&0\\[-0.2em]
0&0&1&0&0
\end{array}\right).$$

{\boldmath $q=3$.} Projective $3-(10{+}28,5,\{2\})_3$ and $3-(6{+}41,5,\{2\})_3$ systems ($(20,4)^{3,5}_3$- and $(12,5)^{3,5}_3$-multispreads, respectively) were found in~\cite[Appendix A.1]{KroMog:multispread}. In the current work, we classify all projective $3-(38,5,\{2\})_3$ and systems up to equivalence. The search is rather straightforward. Firstly, by adding $2$-spaces one-by-one and keeping only nonequivalent collections at each step, we find all nonequivalent faithful $2-(n,5,2)_3$ systems for all~$n$ (repeating the results in \cite{ball2025griesmer}, Table~3 ($r=5$) and Remark~27).
Then, we complete each of the $91720$ $2-(10,5,2,1)_3$ systems to a $3-(38,5,\{2\})$-system
by adding $28$ $3$-subspaces (some of them can coincide); at this step, we solve the corresponding exact covering problem (it happens that every $2-(10,5,2,1)_3$ system can be completed in such a way). Although the automorphism group orders of $2-(10,5,2,1)_3$ systems possess values
$1$ ($91488$ classes), $2$ ($169$), $3$ ($40$), $4$ ($2$), $5$ ($15$), $6$ ($4$), $10$ ($1$), $36$ ($1$ class),
a multispread can have only $1$ ($39857065$ equivalence classes), $2$ ($1889$), $3$ ($243$), or $6$ ($15$ classes) automorphisms. $14$ of the last $15$ multispreads have an automorphism of order $6$ conjugate to $ A_2A_3 $;
the remaining one has an automorphism group
conjugate to $\langle A_3,A_2' \rangle\sim S_3$, where
$$\footnotesize
A_2=\left(\begin{array}{@{\,}c@{\ \,}c@{\ \,}c@{\ \,}c@{\ \,}c@{\,}}
2&0&0&0&0\\[-0.5ex]
0&2&0&0&0\\[-0.5ex]
0&0&1&0&0\\[-0.5ex]
0&0&0&1&0\\[-0.5ex]
0&0&0&0&1
\end{array}\right),\quad
A_3=\left(\begin{array}{@{\,}c@{\ \,}c@{\ \,}c@{\ \,}c@{\ \,}c@{\,}}
0&1&0&0&0\\[-0.5ex]
2&2&0&0&0\\[-0.5ex]
0&0&0&0&1\\[-0.5ex]
0&0&1&0&0\\[-0.5ex]
0&0&0&1&0
\end{array}\right),\quad
A_2'=
\left(\begin{array}{@{\,}c@{\ \,}c@{\ \,}c@{\ \,}c@{\ \,}c@{\,}}
0&2&0&0&0\\[-0.5ex]
2&0&0&0&0\\[-0.5ex]
0&0&0&1&0\\[-0.5ex]
0&0&1&0&0\\[-0.5ex]
0&0&0&0&1
\end{array}\right).
$$

{\boldmath $q=4$.} By prescribing different automorphism groups,
we have found
projective
$3-(17{+}65,5,\{2\})_4$,
$3-(12{+}86,5,\{2\})_4$,
$3-(7{+}107,5,\{2\})_4$ systems
(multispreads with parameters
$(51,5)_4^{3,5}$,
$(36,6)_4^{3,5}$,
$(21,7)_4^{3,5}$, respectively), with the full automorphism group of order~$120$, $40$, and~$7$, respectively (the basis matrices are listed in Appendix~\ref{appendixB}).
By \cite[Section~7.2]{KroMog:multispread}, this completes the characterization
of the parameters of $(\lambda,\mu)_4^{3,k}$-multispreads
and $\FF_4$-linear $64$-ary one-weight codes.

{\boldmath $q=5$.}
To complete the characterization of parameters of additive $5^3$-ary one-weight codes, similarly to the case $q=4$ considered above, it is sufficient to find
projective
$3-(26{+}126,5,\{2\})_5$,
$3-(20{+}157,5,\{2\})_5$,
$3-(14{+}188,5,\{2\})_5$,
$3-(8{+}219,5,\{2\})_5$ systems
($(104,6)_5^{(3,5)}$-,
$(80,7)_5^{(3,5)}$-,
$(56,8 )_5^{(3,5)}$-,
$(32,9)_5^{(3,5)}$-mul\-ti\-spreads). The last $3$ types of systems are relatively easy to find.
We first search for faithful
$2-(n,5,2,1)_5$ systems, $n=8$, $14$, $20$, respectively, with the prescribed automorphism group of order six generated by
$$
\left(
\begin{array}{c@{\ }c@{\ \ }c@{\ }c@{\ }c}
0&1&0&0&0\\[-0.5ex]
1&4&0&0&0\\[-0.0ex]
0&0&0&1&0\\[-0.5ex]
0&0&1&4&0\\[-0.0ex]
0&0&0&0&1
\end{array}\right)
,\quad\left(
\begin{array}{c@{\ }c@{\ \ }c@{\ }c@{\ }c}
0&1&0&0&0\\[-0.5ex]
1&0&0&0&0\\[-0.0ex]
0&0&0&1&0\\[-0.5ex]
0&0&1&0&0\\[-0.0ex]
0&0&0&0&1
\end{array}\right).
$$
In each case,
one of the first several solutions is extendable to a required multispread (see examples in Appendix~\ref{appendixB}).
However, we have not succeeded in finding a projective
$3-(26{+}126,5,\{2\})_5$ system. By prescribing different automorphism groups of order at least~$4$, we have found thousands of faithful $2-(26,5,2,1)_5$ system from only~$4$ equivalence classes. No one of them can be completed to a multispread.
We conjecture that if a projective
$3-(26{+}126,5,\{2\})_5$ system exists, then it has at most~$3$
automorphisms.
If it does not exist, then
the following multispreads cannot be constructed recursively and should be considered separately:
\begin{itemize}
    \item Next $\lambda$:
  a $(228,6)_5^{3,5}$-multispread (projective $3-(57{+}96,5,3)_5$ system).
    \item Next $k$:
  a $(104,6)_5^{3,8}$-multispread
  (projective $3-(26{+}18876,5,152)_5$ system).
\end{itemize}

\section*{Acknowledgments}
D.~Krotov thanks Aleksandr Buturlakin for consulting in the field of representation theory.
S.~Kurz thanks the organizers of the Seventh Irsee Conference
{\lq\lq}Finite Geometries 2025{\rq\rq} for the invitation and the organization of the workshop.

\section*{Statements and Declarations}
There are no competing interests.



\appendix

\section{An alternative construction for sets of subspaces with at most two in each hyperplane}
\label{appendixC}
 Here we give an alternative construction
of projective $2-(q^2,5,2)_q$ systems, for the special case when the field size $q$ is an odd prime. Although the construction gives projective systems (and codes) equivalent to those constructed in Section~\ref{sec_oval} in more general settings, the approach is completely different, and we believe it is of independent interests, as a rare example of a non-computational construction of a code with prescribed automorphism group. Based on this we give an alternative proof of  Theorem~\ref{prop_generalized_construction_oval_s} for the case $h=s=2$.

For an odd prime~$q$, we will construct a faithful $2-(q^2+2,5,2)_q$ system with a prescribed automorphism group generated
by the following two matrices (only the non-zero elements are shown), where $\beta$ is a non-square in $\FF_q$:

\begin{equation}
    \label{eq:AB}
A = 
\raisebox{-8ex}{
\begin{tikzpicture}[ scale=0.6,
    cell/.style={minimum size=4mm, inner sep=0pt, font=\Large\bfseries}, 
    arrow/.style={-{Stealth[length=2mm, width=1.2mm]}, line width=0.4pt, black!60}
]
\foreach \i in {0,...,4} {
    \node[cell] at (\i, -\i) {1};
}
\node[cell] at (2, 0) {1}; 
\node[cell] at (3, -1) {1}; 
\node[cell] at (4, -2) {1}; 
\draw[arrow, blue] (0.25, 0) to[out=0, in=180] (1.75, 0);    
\draw[arrow, green!70!black] (1.25, -1) to[out=0, in=180] (2.75, -1);  
\draw[arrow, blue] (2.25, -2) to[out=0, in=180] (3.75, -2);  
\draw[arrow, blue] (2, -0.25) to[out=-90, in=90] (2, -1.75);  
\draw[arrow, green!70!black] (3, -1.25) to[out=-90, in=90] (3, -2.75);  
\draw[arrow, blue] (4, -2.25) to[out=-90, in=90] (4, -3.75);  
\draw[black!40!white] (-0.5, 0.5) rectangle (4.5, -4.5);
\end{tikzpicture}
}
\qquad B=
\raisebox{-8ex}{
\begin{tikzpicture}[ scale=0.6,
    cell/.style={minimum size=4mm, inner sep=0pt, font=\Large\bfseries}, 
    arrow/.style={-{Stealth[length=2mm, width=1.2mm]}, line width=0.4pt, black!60}
]
\foreach \i in {0,...,4} {
    \node[cell] at (\i, -\i) {1};
}
\node[cell] at (3, 0) {$\bm\beta$}; 
\node[cell] at (2, -1) {1}; 
\node[cell] at (4, -3) {1}; 

\draw[arrow, blue] (0.25, 0) to[out=0, in=180] (2.75, 0);    
\draw[arrow, green!70!black] (1.25, -1) to[out=0, in=180] (1.75, -1);  
\draw[arrow, blue] (3.25, -3) to[out=0, in=180] (3.75, -3);  

\draw[arrow, blue] (3, -0.25) to[out=-90, in=90] (3, -2.75);  
\draw[arrow, green!70!black] (2, -1.25) to[out=-90, in=90] (2, -1.75);  
\draw[arrow, blue] (4, -3.25) to[out=-90, in=90] (4, -3.75);  

\draw[black!40!white] (-0.5, 0.5) rectangle (4.5, -4.5);

\end{tikzpicture}
}
\end{equation}
These matrices commute and each of them has order~$q$ as a group element, which can be seen from the proof of the following lemma.
\begin{lemma}
\label{lemma_workhorse}
In $\PGF{5}{q}$,
where $q$ is an odd prime,
consider the $q^4$ lines with generator 
matrices of the form
\begin{equation}\label{eq:xyz}
 \left[\begin{array}{lllll}
    1 & 0 & x & y & z \\
    0 & 1 & x'& y'& z'
\end{array}\right]
\qquad\mbox{where} \quad
x-y'=\frac12 , \quad
 y - \beta x' =\frac{\beta}2.
\end{equation}
\begin{itemize}
    \item[\rm (i)]
All these $q^4$ lines 
are disjoint to the $3$-space $$S=\langle (0,0,0,0,1),(0,0,0,1,0),(0,0,1,0,0) \rangle.$$
    \item[\rm (ii)]
The set of these $q^4$ lines is partitioned into $q^2$ orbits under 
$\langle A,B\rangle$, where
each orbit is a faithful projective
$2-(q^2,5,2,1)_q$ system.
    \item[\rm (iii)] Moreover, the affine span of the set of generator matrices, in the form~\eqref{eq:xyz}, 
of each of the $q^2$ orbits coincides with the $4$-dimensional affine space 
of all matrices defined by equations~\eqref{eq:xyz}.
\end{itemize}
\end{lemma}
\begin{proof}
The action of $A$ and $B$ on subspaces with generator 
matrix of the form in Equation~\eqref{eq:xyz} is the following:
\begin{equation*}\label{eq:i}
A^i: 
\left[\begin{array}{lllll}
    1 & 0 & x & y & z \\
    0 & 1 & x'& y'& z'
\end{array}\right]
\to 
\left[\begin{array}{lllll}
    1\tikzmark{a} & 0\tikzmark{b} & x\tikzmark{aa}+\tikzmark{a_a}i & y\phantom{+1}\tikzmark{bb} & z+x\tikzmark{aaa}i +\tikzmark{aa_a}\binom{i}{2} \\
    0 & 1 & x'& y'+i& z'+x'i
\end{array}\right]\!,
    \begin{tikzpicture}[overlay, remember picture]
        \draw[->, line width=0.8pt, -Stealth, blue, out=30, in=150, looseness=1.0]
            ([yshift=3mm] pic cs:a) to ([yshift=3mm] pic cs:a_a);
        \draw[->, line width=0.8pt, -Stealth, blue, out=20, in=160, looseness=1.0]
            ([yshift=3mm] pic cs:aa) to ([yshift=3mm] pic cs:aaa);
        \draw[->, line width=0.8pt, -Stealth, blue, out=20, in=160, looseness=1.0]
            ([yshift=3mm] pic cs:a_a) to ([yshift=3mm] pic cs:aa_a);
        \draw[->, line width=0.8pt, -Stealth, green!70!black, out=25, in=155, looseness=1.0]
            ([yshift=3mm] pic cs:b) to ([yshift=3mm] pic cs:bb);
    \end{tikzpicture}
\end{equation*}
\begin{equation*}\label{eq:j}
    B^j: 
\left[\begin{array}{lllll}
    1 & 0 & x & y & z \\
    0 & 1 & x'& y'& z'
\end{array}\right]
\to 
\left[\begin{array}{lllll}
    1\tikzmark{a_} & 0\tikzmark{b_} & x\phantom{+j}\tikzmark{bb_} & y\tikzmark{a_a_} +\beta\tikzmark{aa_} j & z+y\tikzmark{aa_a_}j+\beta\tikzmark{aaa_}\binom{j}{2} \\
    0 & 1 & x'+j& y'& z'+y'j
\end{array}\right]\!,
    \begin{tikzpicture}[overlay, remember picture]
    \coordinate (ab) at ($(pic cs:a_)!0.7!(pic cs:aa_)$);
        \draw[->, line width=0.8pt, -Stealth, blue, out=30, in=160, looseness=1.0]
            ([yshift=3mm] pic cs:a_) to ([yshift=3mm] pic cs:aa_);
        \draw[->, line width=0.8pt, -Stealth, blue, out=30, in=150, looseness=1.0]
            ([yshift=3mm] pic cs:aa_) to ([yshift=3mm] pic cs:aaa_);
        \draw[->, line width=0.8pt, -Stealth, blue, out=30, in=150, looseness=1.0]
            ([yshift=3mm] pic cs:a_a_) to ([yshift=3mm] pic cs:aa_a_);
        \draw[->, line width=0.8pt, -Stealth, green!70!black, out=40, in=140, looseness=1.0]
            ([yshift=3mm] pic cs:b_) to ([yshift=3mm] pic cs:bb_);
        \node[above=3mm of ab, blue, fill=white] {$\beta$};
    \end{tikzpicture}
\end{equation*}
\begin{multline} \label{eq:ij}
    A^iB^j: 
\left[\begin{array}{lllll}
    1 & 0 & x & y & z \\
    0 & 1 & x'& y'& z'
\end{array}\right]
\\ \to 
\left[\begin{array}{lllll}
    1 & 0 & x+i & y+\beta j & z+ix+jy+\underline{\binom{i}{2}+\beta\binom{j}{2}} \\
    0 & 1 & x'+j& y'+i& z'+ix'+jy'+\underline{ij}
\end{array}\right]\!.
\end{multline}

We start with two auxiliary statements.

(a) \emph{If a line (to be exact, its generator 
matrix in form~\eqref{eq:xyz}) belongs to the $4$-di\-men\-si\-o\-nal affine 
subspace defined by equations~\eqref{eq:xyz},
then the same is true for all lines from its orbit.} Indeed, as we see from~\eqref{eq:ij}, the transform $A^iB^j$ keeps the differences 
$x-y'$ and $y-\beta x'$.

(b) \emph{The span of the generator matrices in form~\eqref{eq:xyz} of the subspaces of an orbit (of one of subspaces in such form) is a $4$-dimensional affine space.}
From the argument in (a) we see that 
the dimension is at most~$4$.
For a matrix $M$ defined as in~\eqref{eq:xyz},
the three matrices $M$, $MA$, and $MB$
generate a $2$-dimensional affine space
of matrices of form~\eqref{eq:ij} without the underlined parts. The matrices $MA^2$ and $MAB$
have form~\eqref{eq:ij} with nonzero underlined part and add two more dimensions. So,
the affine span of $M$, $MA$, $MB$, $MA^2$, $MAB$ has dimension~$4$.

Next, we want to find $x$, $y$, $z$, $x'$, $y'$, $z'$ such that all $q^2$
subspaces $T_{i,j}$ spanned by the matrix~\eqref{eq:ij} for different~$i$, $j$ are mutually disjoint (actually, this condition holds automatically if the next one holds\footnote{If a line $T$ from the system intersects 
with $T' = A^{i}B^{j} T$, then it intersects 
with $T'' = A^{-i}B^{-j} T$; among $q^2+q+1$ hyperplanes including $T$, $2q+2$ include $T'$ or $T''$; the remaining $q^2-q-1$ hyperplanes include the remaining $q^2-3$ system lines, which means that some hyperplanes includes more that two system lines, a contradiction.}) and, moreover, every hyperplane includes at most two of them.

The last condition is equivalent to the fact that the dual $3$-subspaces $T_{i,j}^\perp$
generated by
\begin{equation}\label{eq:dual}
    M_{i,j}=
\left[\begin{array}{rrrrr}
     x+i & x'+j & -1 & 0 & 0 \\
    y+\beta j & y'+i & 0 & -1 & 0 \\
    z+ix+jy+\binom{i}{2}+\beta\binom{j}{2} & z'+ix'+jy'+ij & 0 & 0 & -1 
\end{array}\right]
\end{equation}
covers each point with multiplicity  $0$, $1$, or $2$. It is sufficient to check that every point in~$T_{0,0}^\perp$ belongs to at most one other~$T_{i,j}^\perp$. Note that if $T_{0,0}^\perp$ and $T_{i,j}^\perp$ have a common vector $v=(v,v',-k,-l,-m)$, then it is the linear combination of the rows of each of~$M_{0,0}$, $M_{i,j}$ with the same coefficients $k$, $l$, $m$:
$v=(k,l,m)M_{0,0}=(k,l,m)M_{i,j}$.

 Let us first check the case where a nontrivial linear combination of the first two rows, with coefficients $k$ and $l$, coincides, i.e., $m=0$.
We have 
$$
\left\{ \begin{array}{rcl}
    k(x+i) + l(y+\beta j)& = & kx+ly \\[0.5ex]
    k(x'+j) + l(y'+i) & = &  kx' + ly'
\end{array} \right. \!,
$$
$$
\left\{ \begin{array}{r@{\ }c@{\ }r@{\ }c@{\ }l}
    ki &+& l\beta j& = & 0 \\[0.5ex]
    li &+& kj & = &  0
\end{array} \right. \!.
$$
The last system has a non-zero solution $(i,j)$ only if $l^2\beta = k^{2}$, which is impossible because $\beta$ is not a square. Hence, the corresponding 
points are covered only once, for any
$x$, $y$, $x'$, $y'$, $z$, $z'$.
Linear combinations including the last row
are considered differently because the system becomes non-linear in $i$, $j$.
We consider only the $q^2$ orbits
with the representatives 
$x=-x'=\frac12$, $y=y'=0$,
$z$ and $z'$ arbitrary. In particular,
$$ M_{i,j}=
\left[\begin{array}{r@{\ \ \ \ \ }r@{\ \ \ \ }r@{\ \ \ \ }r@{\ \ \ \ }r}
     \frac12+i & -\frac12+j & -1 & 0 & 0 \\
    \beta j & i & 0 & -1 & 0  \\
z+\frac{i}2+\binom{i}{2}+\beta\binom{j}{2} & z'-\frac{i}2+ij & 0 & 0 & -1 
\end{array}\right]\!.
$$
For these values, the linear combinations of the rows with 
coefficients $k$, $l$, $1$ coincide for $M_{0,0}$ and~$M_{i,j}$ if and only if
$$
\left\{ \begin{array}{r@{\ }c@{\ }l}
    ki+l\beta j+\frac{i}{2} +\binom{i}{2}+\beta \binom{j}{2}& = & 0 \\[0.5ex]
    kj+li-\frac{i}{2}+ij& = &  0
\end{array} \right.\!,
$$
$$
\left\{ \begin{array}{r@{\ }c@{\ }l}
    i^2+2ki +\beta j^2+2L\beta j& = & 0 \\[0.5ex]
    kj+Li+ij& = &  0
\end{array} \right.\!, \quad\mbox{where }
L=l-\frac12,
$$
$$
\left\{ \begin{array}{r@{\ }c@{\ }l}
    (i+k)^2 +\beta (j+L)^2& = & k^2+\beta L^2 \\[0.5ex]
    (i+k)(j+L)& = &kL
\end{array} \right.\!. \quad\mbox{ }
$$

If $k=0$ and $L=0$, we have no non-zero solutions. 
If $k=0$ and $L\ne 0 $, then either $i=0$ and we have two solutions for $j$ ($0$ and $-2L$)
or $j=-L$ and $i^2-L^2\beta=0$, which has no solutions because $\beta$ is not a square.

Similarly, if $k\ne 0$ and $L= 0 $, then either $j=0$ and we have two solutions for $i$ ($0$ and $-2k$)
or $i=-k$ and $\beta j^2=k^2$, which has no solutions because $\beta$ is not a square.

If $k\ne 0$ and $L\ne 0 $, we have
$$
\left\{ \begin{array}{r@{\ }c@{\ }l}
    I^2 +\beta J^2& = & K^2+\beta \\[0.5ex]
    IJ& = &K
\end{array} \right.\!, \quad\mbox{where }
K=\frac kL,\ I=\frac{i+k}L,\ J=\frac{j+L}L,
$$
$$ X^2 - (K^2+\beta) X + \beta K^2, \quad\mbox{where }X=I^2\!.$$
If the last quadratic equation in $X$ has two solutions, then one of them is square and one non-square because their product is the non-square $\beta K^2$. Hence, only one of them is resolvable in~$I$, and we have only two solutions for $(i,j)$. We know that one of them is $(0,0)$, and we can easily find the other.

We have proved that

(c) \emph{for any $z$, $z'$, the orbit of the line with generator 
matrix
 $\left[\begin{array}{rrrrl}
    1 & 0 & ^1\!/_2 & 0 & z \\
    0 & 1 & -^1\!/_2 & 0 & z'
\end{array}\right]$  is a faithful $2-(q^2,5,2,1)_q$ system.}

From \eqref{eq:ij}, it if easy to see that all these $q^2$ orbits are different. Combining claims (a), (b), and (c) completes the proof.
\end{proof}

We note that experiments with small~$q$ show that other orbits under $\langle A,B\rangle$ do not give $2-(q^2,5,2)_q$ systems.

\begin{theorem}
    For any odd prime $q$,
    there is a faithful $2-(q^2+2,5,2)_q$ system (and hence, an additive $[q^2+2,2.5,q^2]_q^2$ code) invariant under the actions of $A$, $B$ defined in~\eqref{eq:AB}.
\end{theorem}
\begin{proof}
We will add two more lines to the faithful projective $2-(q^2,5,2)_q$ system $C$
constructed in Lemma~\ref{lemma_workhorse}. 
One of them is the $2$-space $D$ generated by 
 $\left(\begin{smallmatrix}
    0 & 0 & 0 & 1 & 0 \\
    0 & 0 & 0 & 0 & 1
\end{smallmatrix}\right)$.
To ensure that we can add it,
we need to check that every
point from the dual space~$D^\perp$, with basis 
 $\left(\begin{smallmatrix}
    1 & 0 & 0 & 0 & 0 \\
    0 & 1 & 0 & 0 & 0 \\
    0 & 0 & 1 & 0 & 0
\end{smallmatrix}\right)$,
belongs to only one $T_{i,j}^\perp$. This is obvious
from the basis matrix~\eqref{eq:dual} of~$T_{i,j}^\perp$. Moreover,
this point can only be the first line of~\eqref{eq:dual}.

The same arguments work for the line $D'$ generated by
$\left(\begin{smallmatrix}
    0 & 0 & 1 & 0 & 0 \\
    0 & 0 & 0 & 0 & 1
\end{smallmatrix}\right)$,
with intersecting points of dual $3$-subspaces coinciding with the second line of~\eqref{eq:dual}.

Since the common points of~$D^\perp$ and~$D'^\perp$ do not belong to any~$T_{i,j}^\perp$, we can add both~$D$ and~$D'$ to get a 
faithful projective $2-(q^2+2,5,2)_q$ system. Finally, it is easy to see that $DA = DB=D$ and $D'A = D'B=D'$, i.e., adding $D$ and $D'$ keeps the system $\langle A,B \rangle$-invariant.
\end{proof}

\begin{remark}
(i) Actually, there are $q+1$ $\langle A,B \rangle$-invariant lines, $D$ and $q$ lines
with basis matrix of the form $\left(\begin{smallmatrix}
    0 & 0 & 1 & y & 0 \\
    0 & 0 & 0 & 0 & 1
\end{smallmatrix}\right)$.
So, any two of them can be chosen as the two added extra lines. Different choices can result in non-equivalent systems and additive codes.
(ii) Since $D$ and~$D'$ intersect, 
the resulting projective system does not have parameters $2-(q^2+2,5,2,1)_q$.
However, by adding only one of them, we get a faithful projective  $2-(q^2+1,5,2,1)_q$ system. 
\end{remark}

\section{Examples of multispreads}
\label{appendixB}
 Here, we list examples of projective systems corresponding to one-weight codes with new parameters. The elements of $\FF_4$ are denoted $0$, $1$, $\omega$, $\upsilon=\omega^2$.

A projective $3-(82,5,\{2\})_4$ system, i.e., a $(51,5)^{3,5}_4$-multispread:
\def\subspace#1#2{\left(\!\begin{smallmatrix}#1\end{smallmatrix}\!\right)^{\!\makebox[0mm][l]{$\scriptscriptstyle #2$}}}
\def\2{\omega} \def\3{\upsilon}
$\subspace{1101{\2}\\00110}{}$,
$\subspace{1001{\2}\\00100}{}$,
$\subspace{0110{\2}\\00010}{}$,
$\subspace{1{\2}000\\001{\2}0}{}$,
$\subspace{1{\3}000\\001{\3}0}{}$,
$\subspace{100{\2}1\\01000}{}$,
$\subspace{10{\2}{\2}1\\01{\2}{\2}1}{}$,
$\subspace{10000\\01{\2}01}{}$,
$\subspace{10{\3}{\2}1\\01{\2}{\3}1}{}$,
$\subspace{101{\2}1\\01010}{}$,
$\subspace{10100\\01{\2}11}{}$,
$\subspace{10{\2}{\3}{\3}\\010{\2}0}{}$,
$\subspace{10{\2}00\\01{\3}{\2}{\3}}{}$,
$\subspace{101{\3}{\3}\\01{\3}1{\3}}{}$,
$\subspace{10{\3}00\\01{\3}{\3}{\3}}{}$,
$\subspace{100{\3}{\3}\\01{\3}0{\3}}{}$,
$\subspace{10{\3}{\3}{\3}\\010{\3}0}{}$,
$\subspace{01{\2}00\\00010\\00001}{}$,
$\subspace{10010\\010{\2}1\\001{\3}{\2}}{}$,
$\subspace{1000{\3}\\0101{\3}\\001{\2}{\3}}{}$,
$\subspace{1000{\2}\\01000\\0001{\3}}{}$,
$\subspace{1{\2}00{\3}\\0010{\3}\\0001{\2}}{}$,
$\subspace{1000{\2}\\010{\3}{\3}\\0011{\3}}{}$,
$\subspace{1000{\3}\\0100{\2}\\001{\2}1}{}$,
$\subspace{1010{\2}\\0100{\3}\\0001{\3}}{}$,
$\subspace{100{\3}1\\010{\3}{\3}\\0010{\3}}{}$,
$\subspace{10011\\010{\2}{\3}\\0011{\3}}{}$,
$\subspace{100{\2}0\\0100{\3}\\001{\2}{\3}}{}$,
$\subspace{10101\\01{\3}0{\2}\\00011}{}$,
$\subspace{10{\3}0{\3}\\0100{\2}\\0001{\3}}{}$,
$\subspace{1000{\3}\\010{\2}1\\001{\2}1}{}$,
$\subspace{100{\2}{\2}\\010{\3}1\\001{\3}{\2}}{}$,
$\subspace{100{\3}0\\010{\2}{\3}\\001{\3}{\2}}{}$,
$\subspace{1000{\2}\\010{\3}{\3}\\001{\3}{\2}}{}$,
$\subspace{1001{\3}\\0100{\2}\\001{\3}{\3}}{}$,
$\subspace{10{\3}01\\01{\3}00\\00011}{}$,
$\subspace{100{\3}1\\010{\3}{\3}\\001{\2}{\2}}{}$,
$\subspace{100{\2}{\3}\\010{\3}{\2}\\00111}{}$,
$\subspace{10011\\010{\3}{\2}\\00101}{}$,
$\subspace{100{\2}{\2}\\010{\3}1\\001{\2}1}{}$,
$\subspace{10011\\0100{\2}\\001{\2}{\3}}{}$,
$\subspace{1000{\2}\\010{\3}{\2}\\0010{\3}}{}$,
$\subspace{1001{\3}\\01001\\0011{\2}}{}$,
$\subspace{100{\3}0\\010{\3}{\2}\\0011{\3}}{}$,
$\subspace{100{\2}1\\01001\\00101}{}$,
$\subspace{1001{\3}\\010{\3}1\\001{\2}{\2}}{}$,
$\subspace{10001\\01{\3}0{\3}\\00011}{}$,
$\subspace{100{\3}1\\01011\\00111}{}$,
$\subspace{100{\2}0\\01011\\00101}{}$,
$\subspace{1001{\2}\\01001\\0010{\2}}{}$,
$\subspace{10{\2}01\\01{\3}01\\00011}{}$,
$\subspace{100{\2}0\\01001\\00111}{}$,
$\subspace{100{\2}{\3}\\010{\2}{\2}\\0011{\3}}{}$,
$\subspace{100{\2}{\2}\\0101{\2}\\001{\2}1}{}$,
$\subspace{1{\3}00{\3}\\00101\\0001{\2}}{}$,
$\subspace{100{\3}0\\0101{\2}\\001{\3}1}{}$,
$\subspace{100{\2}{\2}\\0101{\3}\\0011{\2}}{}$,
$\subspace{100{\2}{\3}\\0100{\3}\\001{\3}{\3}}{}$,
$\subspace{10010\\01011\\0010{\2}}{}$,
$\subspace{10010\\0101{\3}\\001{\3}{\3}}{}$,
$\subspace{1001{\3}\\010{\3}1\\001{\3}1}{}$,
$\subspace{100{\3}1\\010{\2}1\\001{\2}{\2}}{}$,
$\subspace{100{\3}{\3}\\0100{\3}\\0010{\2}}{}$,
$\subspace{1000{\3}\\00100\\0001{\2}}{}$,
$\subspace{10001\\010{\2}{\2}\\0010{\3}}{}$,
$\subspace{1100{\3}\\0010{\2}\\0001{\2}}{}$,
$\subspace{1001{\3}\\010{\2}{\3}\\001{\2}{\2}}{}$,
$\subspace{100{\3}{\2}\\010{\2}{\3}\\0010{\3}}{}$,
$\subspace{10010\\010{\2}{\2}\\00101}{}$,
$\subspace{100{\3}{\2}\\010{\2}{\2}\\00111}{}$,
$\subspace{100{\3}0\\0100{\3}\\0011{\2}}{}$,
$\subspace{100{\2}0\\0101{\2}\\001{\3}{\3}}{}$,
$\subspace{100{\2}{\2}\\010{\2}1\\001{\3}1}{}$,
$\subspace{100{\3}1\\0101{\3}\\0010{\2}}{}$,
$\subspace{10001\\01011\\0011{\2}}{}$,
$\subspace{100{\3}{\2}\\0101{\2}\\001{\2}{\3}}{}$,
$\subspace{10{\2}0{\3}\\01001\\0001{\3}}{}$,
$\subspace{10001\\0100{\2}\\001{\3}1}{}$,
$\subspace{10000\\01110\\00001}{}$,
$\subspace{10100\\01100\\00001}{}$,
$\subspace{10{\2}{\2}0\\011{\3}0\\00001}{}$,
$\subspace{10{\3}{\2}0\\011{\2}0\\00001}{}$.

A projective $3-(98,5,\{2\})_4$ system, i.e., a $(36,6)^{3,5}_4$-multispread:
\def\subspace#1#2{\left(\!\begin{smallmatrix}#1\end{smallmatrix}\!\right)^{\!\makebox[0mm][l]{$\scriptscriptstyle #2$}}}
\def\2{\omega} \def\3{\upsilon}
$\subspace{01{\3}0{\3}\\00010}{}$,
$\subspace{110{\3}{\3}\\00110}{}$,
$\subspace{100{\3}{\3}\\00100}{}$,
$\subspace{1011{\2}\\0111{\2}}{}$,
$\subspace{10000\\0110{\2}}{}$,
$\subspace{1001{\2}\\01000}{}$,
$\subspace{10{\3}{\2}1\\01{\2}{\3}1}{}$,
$\subspace{101{\2}1\\01010}{}$,
$\subspace{10100\\01{\2}11}{}$,
$\subspace{10{\2}{\2}1\\010{\2}0}{}$,
$\subspace{10{\2}00\\01{\2}{\2}1}{}$,
$\subspace{100{\2}1\\01{\2}01}{}$,
$\subspace{100{\2}{\2}\\010{\2}1\\0011{\2}}{}$,
$\subspace{10{\3}01\\01{\2}0{\2}\\00011}{}$,
$\subspace{1000{\3}\\01000\\0001{\3}}{}$,
$\subspace{100{\2}{\2}\\010{\2}1\\0010{\2}}{}$,
$\subspace{10001\\0110{\2}\\00011}{}$,
$\subspace{1010{\2}\\01001\\00011}{}$,
$\subspace{10{\2}0{\2}\\01{\3}01\\00011}{}$,
$\subspace{100{\2}{\2}\\010{\2}1\\001{\2}1}{}$,
$\subspace{100{\3}1\\0100{\3}\\001{\3}{\2}}{}$,
$\subspace{0100{\2}\\00101\\00010}{}$,
$\subspace{1001{\3}\\0101{\2}\\00111}{}$,
$\subspace{10001\\0101{\2}\\00111}{}$,
$\subspace{1001{\3}\\0100{\3}\\001{\2}{\3}}{}$,
$\subspace{1100{\2}\\00101\\00011}{}$,
$\subspace{100{\3}{\2}\\0101{\2}\\00111}{}$,
$\subspace{1000{\3}\\0100{\3}\\0011{\3}}{}$,
$\subspace{10{\2}01\\0100{\3}\\0001{\2}}{}$,
$\subspace{100{\2}{\2}\\010{\2}1\\001{\3}1}{}$,
$\subspace{100{\2}0\\0101{\2}\\00111}{}$,
$\subspace{100{\2}1\\0100{\3}\\0010{\2}}{}$,
$\subspace{100{\3}0\\010{\3}{\3}\\00101}{}$,
$\subspace{1000{\2}\\00100\\00010}{}$,
$\subspace{100{\3}1\\0101{\2}\\001{\2}0}{}$,
$\subspace{100{\2}1\\01000\\00111}{}$,
$\subspace{100{\3}{\3}\\0100{\3}\\00101}{}$,
$\subspace{10{\2}0{\3}\\01{\2}0{\3}\\00011}{}$,
$\subspace{100{\3}0\\01010\\0010{\3}}{}$,
$\subspace{1000{\3}\\010{\2}0\\0010{\2}}{}$,
$\subspace{10011\\010{\3}{\3}\\001{\3}0}{}$,
$\subspace{100{\3}{\2}\\010{\2}1\\0010{\3}}{}$,
$\subspace{1{\3}0{\2}0\\001{\2}0\\00001}{}$,
$\subspace{1{\2}0{\2}0\\001{\3}0\\00001}{}$,
$\subspace{1000{\2}\\010{\3}{\3}\\001{\2}{\2}}{}$,
$\subspace{100{\3}{\2}\\010{\2}{\3}\\001{\2}{\2}}{}$,
$\subspace{1{\3}00{\3}\\00101\\0001{\2}}{}$,
$\subspace{100{\2}{\3}\\01001\\001{\2}{\2}}{}$,
$\subspace{1001{\3}\\01011\\001{\2}{\2}}{}$,
$\subspace{10010\\010{\2}{\2}\\001{\3}{\3}}{}$,
$\subspace{100{\3}{\2}\\010{\3}{\3}\\001{\3}{\3}}{}$,
$\subspace{100{\2}0\\0100{\2}\\001{\3}{\3}}{}$,
$\subspace{1000{\2}\\0101{\3}\\001{\3}{\3}}{}$,
$\subspace{1{\2}00{\2}\\00101\\0001{\3}}{}$,
$\subspace{10001\\010{\3}0\\0010{\2}}{}$,
$\subspace{100{\3}1\\010{\3}0\\0011{\2}}{}$,
$\subspace{10{\3}00\\01001\\0001{\2}}{}$,
$\subspace{100{\2}0\\010{\3}0\\001{\3}0}{}$,
$\subspace{10010\\010{\3}0\\001{\2}0}{}$,
$\subspace{100{\2}{\3}\\010{\3}1\\001{\3}0}{}$,
$\subspace{10{\3}01\\0100{\2}\\0001{\3}}{}$,
$\subspace{10000\\010{\3}{\2}\\00100}{}$,
$\subspace{1001{\3}\\010{\3}{\2}\\001{\2}0}{}$,
$\subspace{100{\3}{\3}\\010{\3}1\\0011{\3}}{}$,
$\subspace{10100\\01110\\00001}{}$,
$\subspace{10000\\01{\3}00\\00001}{}$,
$\subspace{10{\2}10\\011{\3}0\\00001}{}$,
$\subspace{10{\3}10\\011{\2}0\\00001}{}$,
$\subspace{100{\3}0\\010{\2}{\3}\\0011{\3}}{}$,
$\subspace{10{\2}0{\3}\\01101\\0001{\2}}{}$,
$\subspace{10010\\010{\3}{\3}\\0010{\3}}{}$,
$\subspace{1000{\3}\\0100{\2}\\001{\2}{\3}}{}$,
$\subspace{100{\2}{\3}\\01011\\001{\2}1}{}$,
$\subspace{100{\3}0\\0101{\3}\\001{\3}{\2}}{}$,
$\subspace{100{\2}0\\010{\3}{\3}\\0011{\2}}{}$,
$\subspace{1000{\3}\\01001\\001{\3}1}{}$,
$\subspace{10{\3}0{\2}\\01100\\0001{\3}}{}$,
$\subspace{1001{\3}\\010{\2}{\2}\\0010{\3}}{}$,
$\subspace{1001{\2}\\0100{\2}\\0010{\3}}{}$,
$\subspace{10001\\010{\2}{\2}\\001{\2}{\3}}{}$,
$\subspace{1001{\3}\\010{\2}{\3}\\0011{\2}}{}$,
$\subspace{100{\2}{\3}\\0101{\3}\\0010{\2}}{}$,
$\subspace{100{\2}0\\01011\\001{\3}1}{}$,
$\subspace{10{\3}0{\3}\\01{\3}01\\0001{\3}}{}$,
$\subspace{1000{\2}\\01{\2}01\\0001{\3}}{}$,
$\subspace{10011\\01001\\001{\3}{\2}}{}$,
$\subspace{100{\2}0\\010{\3}1\\0010{\2}}{}$,
$\subspace{1010{\3}\\01{\2}0{\3}\\0001{\2}}{}$,
$\subspace{10011\\010{\3}1\\0011{\2}}{}$,
$\subspace{10011\\010{\2}{\3}\\001{\2}{\3}}{}$,
$\subspace{10101\\01{\3}01\\0001{\2}}{}$,
$\subspace{1000{\2}\\010{\2}{\2}\\0011{\3}}{}$,
$\subspace{10010\\01011\\001{\3}{\2}}{}$,
$\subspace{10011\\0101{\3}\\0010{\3}}{}$,
$\subspace{100{\3}1\\01001\\001{\2}1}{}$,
$\subspace{10001\\010{\3}1\\001{\3}1}{}$,
$\subspace{100{\3}0\\010{\3}1\\001{\2}1}{}$,
$\subspace{100{\3}1\\0100{\2}\\0011{\3}}{}$.

A projective $3-(114,5,\{2\})_4$ system, i.e., a $(21,7)^{3,7}_4$-multispread, with prescribed automorphism 
$
\left[\!
\begin{smallmatrix}00100\\10000\\01100\\00010\\00001\end{smallmatrix}
\!\right]
$:
\def\subspace#1#2{\left(\!\begin{smallmatrix}#1\end{smallmatrix}\!\right)^{\!\makebox[0mm][l]{$\scriptscriptstyle #2$}}}
\def\2{\omega} \def\3{\upsilon}
$\subspace{010{\2}{\2}\\001{\2}0}{7}$,
$\subspace{100{\2}0\\0010{\2}}{}$,
$\subspace{1100{\2}\\001{\2}{\2}}{}$,
$\subspace{1000{\2}\\010{\2}0}{}$,
$\subspace{101{\2}0\\0100{\2}}{}$,
$\subspace{100{\2}{\2}\\011{\2}0}{}$,
$\subspace{1010{\2}\\011{\2}{\2}}{}$,
$\subspace{01010\\00110\\00001}{7}$,
$\subspace{10010\\00100\\00001}{}$,
$\subspace{11000\\00110\\00001}{}$,
$\subspace{10000\\01010\\00001}{}$,
$\subspace{10110\\01000\\00001}{}$,
$\subspace{10010\\01110\\00001}{}$,
$\subspace{10100\\01110\\00001}{}$,
$\subspace{1{\2}00{\2}\\0010{\3}\\00010}{7}$,
$\subspace{10{\2}01\\0100{\3}\\00010}{}$,
$\subspace{10{\3}01\\0110{\3}\\00010}{}$,
$\subspace{1000{\3}\\01{\2}0{\3}\\00010}{}$,
$\subspace{1010{\3}\\01{\2}00\\00010}{}$,
$\subspace{10{\2}00\\01{\2}0{\3}\\00010}{}$,
$\subspace{10{\2}01\\01{\3}0{\2}\\00010}{}$,
$\subspace{1{\2}00{\2}\\00101\\00011}{7}$,
$\subspace{10{\2}00\\01001\\00011}{}$,
$\subspace{10{\3}01\\01101\\00011}{}$,
$\subspace{10001\\01{\2}00\\00011}{}$,
$\subspace{10101\\01{\2}0{\2}\\00011}{}$,
$\subspace{10{\2}0{\3}\\01{\2}0{\2}\\00011}{}$,
$\subspace{10{\2}0{\3}\\01{\3}0{\2}\\00011}{}$,
$\subspace{1{\3}00{\3}\\0010{\3}\\0001{\2}}{7}$,
$\subspace{10{\3}0{\3}\\0100{\3}\\0001{\2}}{}$,
$\subspace{10{\2}01\\0110{\3}\\0001{\2}}{}$,
$\subspace{10{\3}00\\01{\2}0{\3}\\0001{\2}}{}$,
$\subspace{1000{\3}\\01{\3}01\\0001{\2}}{}$,
$\subspace{1010{\3}\\01{\3}01\\0001{\2}}{}$,
$\subspace{10{\3}0{\3}\\01{\3}00\\0001{\2}}{}$,
$\subspace{1{\3}000\\0010{\3}\\0001{\3}}{7}$,
$\subspace{10{\3}01\\0100{\3}\\0001{\3}}{}$,
$\subspace{10{\2}00\\0110{\3}\\0001{\3}}{}$,
$\subspace{10{\3}0{\3}\\01{\2}00\\0001{\3}}{}$,
$\subspace{1000{\3}\\01{\3}0{\2}\\0001{\3}}{}$,
$\subspace{1010{\3}\\01{\3}0{\3}\\0001{\3}}{}$,
$\subspace{10{\3}0{\2}\\01{\3}01\\0001{\3}}{}$,
$\subspace{10000\\01000\\00100}{1}$,
$\subspace{10000\\01000\\00100}{1}$,
$\subspace{100{\3}0\\01000\\00101}{7}$,
$\subspace{10001\\010{\3}0\\00101}{}$,
$\subspace{100{\3}1\\010{\3}0\\00100}{}$,
$\subspace{10001\\01001\\001{\3}1}{}$,
$\subspace{100{\3}1\\01001\\001{\3}0}{}$,
$\subspace{10000\\010{\3}1\\001{\3}0}{}$,
$\subspace{100{\3}0\\010{\3}1\\001{\3}1}{}$,
$\subspace{100{\3}{\3}\\0100{\3}\\0010{\3}}{7}$,
$\subspace{1000{\3}\\010{\3}{\3}\\00100}{}$,
$\subspace{100{\3}0\\010{\3}0\\0010{\3}}{}$,
$\subspace{10000\\0100{\3}\\001{\3}{\3}}{}$,
$\subspace{100{\3}{\3}\\01000\\001{\3}0}{}$,
$\subspace{1000{\3}\\010{\3}0\\001{\3}{\3}}{}$,
$\subspace{100{\3}0\\010{\3}{\3}\\001{\3}0}{}$,
$\subspace{100{\2}1\\0101{\2}\\0010{\2}}{7}$,
$\subspace{1000{\2}\\010{\2}1\\00110}{}$,
$\subspace{100{\2}{\3}\\010{\3}{\3}\\0011{\2}}{}$,
$\subspace{1001{\2}\\010{\2}{\3}\\001{\2}1}{}$,
$\subspace{100{\3}{\3}\\010{\3}1\\001{\2}{\3}}{}$,
$\subspace{10010\\0100{\2}\\001{\3}1}{}$,
$\subspace{100{\3}1\\01010\\001{\3}{\3}}{}$,
$\subspace{100{\3}1\\0101{\2}\\00101}{7}$,
$\subspace{100{\3}0\\010{\2}{\3}\\0011{\2}}{}$,
$\subspace{10001\\010{\3}1\\0011{\3}}{}$,
$\subspace{1001{\3}\\01001\\001{\2}{\2}}{}$,
$\subspace{100{\2}{\2}\\0101{\3}\\001{\2}{\3}}{}$,
$\subspace{100{\2}{\3}\\010{\2}{\2}\\001{\3}0}{}$,
$\subspace{1001{\2}\\010{\3}0\\001{\3}1}{}$,
$\subspace{100{\3}{\2}\\01011\\0010{\2}}{7}$,
$\subspace{100{\3}0\\010{\2}{\3}\\00111}{}$,
$\subspace{1000{\2}\\010{\3}{\2}\\0011{\3}}{}$,
$\subspace{1001{\3}\\0100{\2}\\001{\2}1}{}$,
$\subspace{100{\2}1\\0101{\3}\\001{\2}{\3}}{}$,
$\subspace{100{\2}{\3}\\010{\2}1\\001{\3}0}{}$,
$\subspace{10011\\010{\3}0\\001{\3}{\2}}{}$,
$\subspace{1001{\3}\\010{\2}0\\00101}{7}$,
$\subspace{100{\2}0\\0101{\2}\\0011{\3}}{}$,
$\subspace{100{\3}{\3}\\010{\3}{\2}\\0011{\2}}{}$,
$\subspace{10001\\0101{\3}\\001{\2}1}{}$,
$\subspace{1001{\2}\\010{\3}{\3}\\001{\2}0}{}$,
$\subspace{100{\2}1\\01001\\001{\3}{\2}}{}$,
$\subspace{100{\3}{\2}\\010{\2}1\\001{\3}{\3}}{}$,
$\subspace{100{\3}1\\010{\2}1\\00100}{7}$,
$\subspace{100{\2}1\\01000\\00110}{}$,
$\subspace{10010\\010{\2}1\\00110}{}$,
$\subspace{100{\3}1\\01010\\001{\2}1}{}$,
$\subspace{10000\\010{\3}1\\001{\2}1}{}$,
$\subspace{10010\\01010\\001{\3}1}{}$,
$\subspace{100{\2}1\\010{\3}1\\001{\3}1}{}$,
$\subspace{100{\3}{\2}\\010{\2}{\3}\\00100}{7}$,
$\subspace{100{\2}{\3}\\01000\\00111}{}$,
$\subspace{10011\\010{\2}{\3}\\00111}{}$,
$\subspace{100{\3}{\2}\\01011\\001{\2}{\3}}{}$,
$\subspace{10000\\010{\3}{\2}\\001{\2}{\3}}{}$,
$\subspace{10011\\01011\\001{\3}{\2}}{}$,
$\subspace{100{\2}{\3}\\010{\3}{\2}\\001{\3}{\2}}{}$,
$\subspace{10011\\010{\3}{\3}\\0010{\3}}{7}$,
$\subspace{100{\3}{\3}\\0101{\2}\\00111}{}$,
$\subspace{100{\2}{\2}\\010{\2}1\\0011{\2}}{}$,
$\subspace{100{\3}0\\0100{\3}\\001{\2}1}{}$,
$\subspace{100{\2}1\\010{\3}0\\001{\2}{\2}}{}$,
$\subspace{1000{\3}\\01011\\001{\3}0}{}$,
$\subspace{1001{\2}\\010{\2}{\2}\\001{\3}{\3}}{}$,
$\subspace{100{\2}0\\010{\3}{\2}\\00101}{7}$,
$\subspace{100{\3}{\3}\\01001\\0011{\3}}{}$,
$\subspace{1001{\3}\\010{\3}{\3}\\0011{\2}}{}$,
$\subspace{1001{\2}\\0101{\3}\\001{\2}1}{}$,
$\subspace{100{\3}{\2}\\010{\2}1\\001{\2}0}{}$,
$\subspace{100{\2}1\\0101{\2}\\001{\3}{\2}}{}$,
$\subspace{10001\\010{\2}0\\001{\3}{\3}}{}$.

A projective $3-(177,5,\{2\})_5$ system, i.e., a $(80,7)^{3,5}_5$-multispread, with prescribed automorphisms 
$
\left[\!
\begin{smallmatrix}10000\\00100\\04400\\00001\\00044\end{smallmatrix}
\!\right]$, $\left[\!
\begin{smallmatrix}10000\\00100\\01000\\00001\\00010\end{smallmatrix}
\!\right]
$:
\def\subspace#1#2{\left(\!\begin{smallmatrix}#1\end{smallmatrix}\!\right)^{\!\makebox[0mm][l]{$\scriptscriptstyle #2$}}}
$\subspace{01003\\00122}{2}$,
$\subspace{01022\\00130}{}$,
$\subspace{10231\\01011}{6}$,
$\subspace{10223\\01110}{}$,
$\subspace{10322\\01004}{}$,
$\subspace{10330\\01101}{}$,
$\subspace{12013\\00111}{}$,
$\subspace{13022\\00140}{}$,
$\subspace{10243\\01024}{6}$,
$\subspace{10214\\01123}{}$,
$\subspace{10324\\01031}{}$,
$\subspace{10332\\01132}{}$,
$\subspace{12034\\00142}{}$,
$\subspace{13042\\00113}{}$,
$\subspace{10311\\01310}{6}$,
$\subspace{10340\\01311}{}$,
$\subspace{10410\\01202}{}$,
$\subspace{10443\\01222}{}$,
$\subspace{10414\\01420}{}$,
$\subspace{10444\\01403}{}$,
$\subspace{10000\\00010\\00001}{1}$,
$\subspace{10001\\00102\\00013}{3}$,
$\subspace{10003\\01001\\00012}{}$,
$\subspace{10003\\01104\\00014}{}$,
$\subspace{10003\\01004\\00102}{3}$,
$\subspace{10022\\01011\\00111}{}$,
$\subspace{10030\\01020\\00140}{}$,
$\subspace{10004\\01001\\00103}{3}$,
$\subspace{10011\\01044\\00144}{}$,
$\subspace{10040\\01030\\00110}{}$,
$\subspace{10002\\01001\\00130}{6}$,
$\subspace{10002\\01032\\00122}{}$,
$\subspace{10020\\01003\\00110}{}$,
$\subspace{10020\\01022\\00123}{}$,
$\subspace{10033\\01013\\00144}{}$,
$\subspace{10033\\01044\\00131}{}$,
$\subspace{10004\\01000\\00133}{6}$,
$\subspace{10004\\01030\\00120}{}$,
$\subspace{10011\\01000\\00103}{}$,
$\subspace{10011\\01030\\00100}{}$,
$\subspace{10040\\01002\\00103}{}$,
$\subspace{10040\\01033\\00100}{}$,
$\subspace{10001\\01014\\00103}{3}$,
$\subspace{10010\\01030\\00141}{}$,
$\subspace{10044\\01021\\00112}{}$,
$\subspace{10004\\01014\\00103}{3}$,
$\subspace{10011\\01021\\00112}{}$,
$\subspace{10040\\01030\\00141}{}$,
$\subspace{10000\\01014\\00112}{2}$,
$\subspace{10000\\01021\\00141}{}$,
$\subspace{10002\\01010\\00124}{6}$,
$\subspace{10002\\01034\\00132}{}$,
$\subspace{10020\\01023\\00143}{}$,
$\subspace{10020\\01042\\00101}{}$,
$\subspace{10033\\01010\\00114}{}$,
$\subspace{10033\\01041\\00101}{}$,
$\subspace{10000\\01012\\00134}{2}$,
$\subspace{10000\\01043\\00121}{}$,
$\subspace{10000\\01012\\00134}{2}$,
$\subspace{10000\\01043\\00121}{}$,
$\subspace{10001\\01023\\00101}{3}$,
$\subspace{10010\\01010\\00132}{}$,
$\subspace{10044\\01042\\00124}{}$,
$\subspace{10002\\01020\\00110}{6}$,
$\subspace{10002\\01033\\00144}{}$,
$\subspace{10020\\01001\\00102}{}$,
$\subspace{10020\\01044\\00133}{}$,
$\subspace{10033\\01002\\00102}{}$,
$\subspace{10033\\01020\\00120}{}$,
$\subspace{10004\\01022\\00120}{6}$,
$\subspace{10004\\01042\\00133}{}$,
$\subspace{10011\\01003\\00132}{}$,
$\subspace{10011\\01023\\00130}{}$,
$\subspace{10040\\01002\\00122}{}$,
$\subspace{10040\\01033\\00124}{}$,
$\subspace{10001\\01034\\00100}{3}$,
$\subspace{10010\\01000\\00143}{}$,
$\subspace{10044\\01041\\00114}{}$,
$\subspace{10001\\01040\\00104}{3}$,
$\subspace{10010\\01040\\00104}{}$,
$\subspace{10044\\01040\\00104}{}$,
$\subspace{10003\\01041\\00102}{3}$,
$\subspace{10022\\01034\\00143}{}$,
$\subspace{10030\\01020\\00114}{}$,
$\subspace{10013\\01000\\00112}{6}$,
$\subspace{10024\\01014\\00141}{}$,
$\subspace{10023\\01021\\00100}{}$,
$\subspace{10032\\01000\\00112}{}$,
$\subspace{10031\\01021\\00100}{}$,
$\subspace{10042\\01014\\00141}{}$,
$\subspace{10013\\01002\\00133}{6}$,
$\subspace{10023\\01002\\00133}{}$,
$\subspace{10024\\01002\\00133}{}$,
$\subspace{10031\\01033\\00120}{}$,
$\subspace{10032\\01033\\00120}{}$,
$\subspace{10042\\01033\\00120}{}$,
$\subspace{10013\\01004\\00141}{6}$,
$\subspace{10023\\01001\\00111}{}$,
$\subspace{10024\\01021\\00144}{}$,
$\subspace{10032\\01011\\00110}{}$,
$\subspace{10031\\01014\\00140}{}$,
$\subspace{10042\\01044\\00112}{}$,
$\subspace{10014\\01012\\00101}{6}$,
$\subspace{10012\\01032\\00134}{}$,
$\subspace{10021\\01043\\00123}{}$,
$\subspace{10034\\01010\\00131}{}$,
$\subspace{10043\\01013\\00101}{}$,
$\subspace{10041\\01010\\00121}{}$,
$\subspace{10014\\01024\\00114}{6}$,
$\subspace{10012\\01030\\00113}{}$,
$\subspace{10021\\01031\\00103}{}$,
$\subspace{10034\\01034\\00103}{}$,
$\subspace{10043\\01030\\00143}{}$,
$\subspace{10041\\01041\\00142}{}$,
$\subspace{10014\\01032\\00102}{6}$,
$\subspace{10012\\01041\\00131}{}$,
$\subspace{10021\\01013\\00114}{}$,
$\subspace{10034\\01020\\00143}{}$,
$\subspace{10041\\01020\\00123}{}$,
$\subspace{10043\\01034\\00102}{}$,
$\subspace{10013\\01032\\00142}{6}$,
$\subspace{10024\\01031\\00132}{}$,
$\subspace{10023\\01042\\00131}{}$,
$\subspace{10032\\01013\\00124}{}$,
$\subspace{10031\\01024\\00123}{}$,
$\subspace{10042\\01023\\00113}{}$,
$\subspace{10013\\01043\\00110}{6}$,
$\subspace{10023\\01033\\00121}{}$,
$\subspace{10024\\01044\\00120}{}$,
$\subspace{10031\\01001\\00134}{}$,
$\subspace{10032\\01012\\00133}{}$,
$\subspace{10042\\01002\\00144}{}$,
$\subspace{10012\\01042\\00121}{6}$,
$\subspace{10014\\01043\\00131}{}$,
$\subspace{10021\\01012\\00124}{}$,
$\subspace{10034\\01032\\00132}{}$,
$\subspace{10041\\01013\\00134}{}$,
$\subspace{10043\\01023\\00123}{}$,
$\subspace{10013\\01040\\00121}{6}$,
$\subspace{10023\\01010\\00104}{}$,
$\subspace{10024\\01043\\00101}{}$,
$\subspace{10031\\01012\\00104}{}$,
$\subspace{10032\\01040\\00101}{}$,
$\subspace{10042\\01010\\00134}{}$,
$\subspace{10003\\01200\\00012}{3}$,
$\subspace{10001\\01300\\00013}{}$,
$\subspace{10003\\01400\\00014}{}$,
$\subspace{10104\\01000\\00014}{6}$,
$\subspace{10104\\01100\\00012}{}$,
$\subspace{10403\\01000\\00013}{}$,
$\subspace{10403\\01100\\00013}{}$,
$\subspace{11004\\00100\\00014}{}$,
$\subspace{14004\\00100\\00012}{}$,
$\subspace{10110\\01040\\00001}{6}$,
$\subspace{10140\\01140\\00001}{}$,
$\subspace{10400\\01001\\00011}{}$,
$\subspace{10400\\01104\\00010}{}$,
$\subspace{11001\\00104\\00010}{}$,
$\subspace{14000\\00104\\00011}{}$,
$\subspace{10102\\01204\\00010}{3}$,
$\subspace{10102\\01401\\00011}{}$,
$\subspace{10200\\01320\\00001}{}$,
$\subspace{10103\\01201\\00010}{3}$,
$\subspace{10103\\01404\\00011}{}$,
$\subspace{10200\\01330\\00001}{}$,
$\subspace{10120\\01320\\00001}{6}$,
$\subspace{10130\\01320\\00001}{}$,
$\subspace{10303\\01204\\00010}{}$,
$\subspace{10304\\01204\\00010}{}$,
$\subspace{10303\\01401\\00011}{}$,
$\subspace{10304\\01401\\00011}{}$,
$\subspace{10203\\01203\\00014}{6}$,
$\subspace{10204\\01204\\00013}{}$,
$\subspace{10203\\01402\\00012}{}$,
$\subspace{10204\\01401\\00013}{}$,
$\subspace{10400\\01301\\00012}{}$,
$\subspace{10400\\01304\\00014}{}$.

A projective $3-(202,5,\{2\})_5$ system, i.e., a $(56,8)^{3,5}_5$-multispread:
\def\subspace#1#2{\left(\!\begin{smallmatrix}#1\end{smallmatrix}\!\right)^{\!\makebox[0mm][l]{$\scriptscriptstyle #2$}}}
$\subspace{01002\\00133}{2}$,
$\subspace{01033\\00120}{}$,
$\subspace{10110\\01004}{6}$,
$\subspace{10144\\01101}{}$,
$\subspace{10421\\01011}{}$,
$\subspace{10434\\01110}{}$,
$\subspace{11001\\00140}{}$,
$\subspace{14012\\00111}{}$,
$\subspace{10114\\01023}{6}$,
$\subspace{10143\\01124}{}$,
$\subspace{10413\\01042}{}$,
$\subspace{10442\\01142}{}$,
$\subspace{11041\\00132}{}$,
$\subspace{14031\\00124}{}$,
$\subspace{10001\\00102\\00013}{3}$,
$\subspace{10003\\01001\\00012}{}$,
$\subspace{10003\\01104\\00014}{}$,
$\subspace{10001\\00102\\00013}{3}$,
$\subspace{10003\\01001\\00012}{}$,
$\subspace{10003\\01104\\00014}{}$,
$\subspace{10002\\01001\\00102}{6}$,
$\subspace{10002\\01002\\00102}{}$,
$\subspace{10020\\01020\\00110}{}$,
$\subspace{10020\\01020\\00120}{}$,
$\subspace{10033\\01033\\00144}{}$,
$\subspace{10033\\01044\\00133}{}$,
$\subspace{10002\\01000\\00112}{6}$,
$\subspace{10002\\01014\\00141}{}$,
$\subspace{10020\\01014\\00141}{}$,
$\subspace{10020\\01021\\00100}{}$,
$\subspace{10033\\01000\\00112}{}$,
$\subspace{10033\\01021\\00100}{}$,
$\subspace{10000\\01003\\00122}{2}$,
$\subspace{10000\\01022\\00130}{}$,
$\subspace{10004\\01001\\00130}{6}$,
$\subspace{10004\\01032\\00122}{}$,
$\subspace{10011\\01013\\00144}{}$,
$\subspace{10011\\01044\\00131}{}$,
$\subspace{10040\\01003\\00110}{}$,
$\subspace{10040\\01022\\00123}{}$,
$\subspace{10002\\01001\\00142}{6}$,
$\subspace{10002\\01041\\00113}{}$,
$\subspace{10020\\01024\\00110}{}$,
$\subspace{10020\\01031\\00114}{}$,
$\subspace{10033\\01034\\00144}{}$,
$\subspace{10033\\01044\\00143}{}$,
$\subspace{10001\\01014\\00103}{3}$,
$\subspace{10010\\01030\\00141}{}$,
$\subspace{10044\\01021\\00112}{}$,
$\subspace{10000\\01014\\00112}{2}$,
$\subspace{10000\\01021\\00141}{}$,
$\subspace{10003\\01012\\00114}{6}$,
$\subspace{10003\\01021\\00143}{}$,
$\subspace{10022\\01014\\00134}{}$,
$\subspace{10022\\01043\\00141}{}$,
$\subspace{10030\\01034\\00112}{}$,
$\subspace{10030\\01041\\00121}{}$,
$\subspace{10000\\01013\\00123}{2}$,
$\subspace{10000\\01032\\00131}{}$,
$\subspace{10002\\01010\\00122}{6}$,
$\subspace{10002\\01031\\00130}{}$,
$\subspace{10020\\01003\\00113}{}$,
$\subspace{10020\\01022\\00101}{}$,
$\subspace{10033\\01010\\00142}{}$,
$\subspace{10033\\01024\\00101}{}$,
$\subspace{10003\\01011\\00133}{6}$,
$\subspace{10003\\01040\\00120}{}$,
$\subspace{10022\\01004\\00104}{}$,
$\subspace{10022\\01040\\00140}{}$,
$\subspace{10030\\01002\\00104}{}$,
$\subspace{10030\\01033\\00111}{}$,
$\subspace{10004\\01022\\00103}{3}$,
$\subspace{10011\\01003\\00130}{}$,
$\subspace{10040\\01030\\00122}{}$,
$\subspace{10004\\01020\\00111}{6}$,
$\subspace{10004\\01032\\00140}{}$,
$\subspace{10011\\01013\\00102}{}$,
$\subspace{10011\\01020\\00131}{}$,
$\subspace{10040\\01004\\00123}{}$,
$\subspace{10040\\01011\\00102}{}$,
$\subspace{10001\\01022\\00121}{6}$,
$\subspace{10001\\01041\\00134}{}$,
$\subspace{10010\\01012\\00122}{}$,
$\subspace{10010\\01043\\00114}{}$,
$\subspace{10044\\01003\\00143}{}$,
$\subspace{10044\\01034\\00130}{}$,
$\subspace{10001\\01030\\00103}{3}$,
$\subspace{10010\\01030\\00103}{}$,
$\subspace{10044\\01030\\00103}{}$,
$\subspace{10000\\01034\\00114}{2}$,
$\subspace{10000\\01041\\00143}{}$,
$\subspace{10004\\01043\\00103}{3}$,
$\subspace{10011\\01012\\00121}{}$,
$\subspace{10040\\01030\\00134}{}$,
$\subspace{10013\\01000\\00123}{6}$,
$\subspace{10024\\01013\\00142}{}$,
$\subspace{10023\\01031\\00100}{}$,
$\subspace{10032\\01000\\00113}{}$,
$\subspace{10031\\01032\\00100}{}$,
$\subspace{10042\\01024\\00131}{}$,
$\subspace{10014\\01000\\00121}{6}$,
$\subspace{10012\\01014\\00100}{}$,
$\subspace{10021\\01000\\00141}{}$,
$\subspace{10034\\01043\\00112}{}$,
$\subspace{10041\\01012\\00100}{}$,
$\subspace{10043\\01021\\00134}{}$,
$\subspace{10014\\01001\\00141}{6}$,
$\subspace{10012\\01023\\00144}{}$,
$\subspace{10021\\01044\\00132}{}$,
$\subspace{10034\\01021\\00124}{}$,
$\subspace{10041\\01014\\00110}{}$,
$\subspace{10043\\01042\\00112}{}$,
$\subspace{10013\\01010\\00110}{6}$,
$\subspace{10023\\01003\\00101}{}$,
$\subspace{10024\\01044\\00122}{}$,
$\subspace{10031\\01001\\00101}{}$,
$\subspace{10032\\01010\\00130}{}$,
$\subspace{10042\\01022\\00144}{}$,
$\subspace{10014\\01010\\00114}{6}$,
$\subspace{10012\\01042\\00101}{}$,
$\subspace{10021\\01010\\00124}{}$,
$\subspace{10034\\01034\\00132}{}$,
$\subspace{10043\\01023\\00143}{}$,
$\subspace{10041\\01041\\00101}{}$,
$\subspace{10013\\01012\\00132}{6}$,
$\subspace{10024\\01042\\00104}{}$,
$\subspace{10023\\01040\\00134}{}$,
$\subspace{10031\\01023\\00121}{}$,
$\subspace{10032\\01043\\00104}{}$,
$\subspace{10042\\01040\\00124}{}$,
$\subspace{10014\\01024\\00104}{6}$,
$\subspace{10012\\01031\\00113}{}$,
$\subspace{10021\\01031\\00113}{}$,
$\subspace{10034\\01040\\00142}{}$,
$\subspace{10043\\01024\\00104}{}$,
$\subspace{10041\\01040\\00142}{}$,
$\subspace{10013\\01020\\00114}{6}$,
$\subspace{10024\\01034\\00143}{}$,
$\subspace{10023\\01041\\00102}{}$,
$\subspace{10032\\01020\\00114}{}$,
$\subspace{10031\\01041\\00102}{}$,
$\subspace{10042\\01034\\00143}{}$,
$\subspace{10013\\01022\\00121}{6}$,
$\subspace{10023\\01034\\00130}{}$,
$\subspace{10024\\01043\\00114}{}$,
$\subspace{10032\\01003\\00143}{}$,
$\subspace{10031\\01012\\00122}{}$,
$\subspace{10042\\01041\\00134}{}$,
$\subspace{10013\\01030\\00101}{6}$,
$\subspace{10023\\01013\\00103}{}$,
$\subspace{10024\\01010\\00123}{}$,
$\subspace{10031\\01010\\00103}{}$,
$\subspace{10032\\01030\\00131}{}$,
$\subspace{10042\\01032\\00101}{}$,
$\subspace{10013\\01044\\00141}{6}$,
$\subspace{10023\\01002\\00110}{}$,
$\subspace{10024\\01021\\00133}{}$,
$\subspace{10032\\01001\\00120}{}$,
$\subspace{10031\\01014\\00144}{}$,
$\subspace{10042\\01033\\00112}{}$,
$\subspace{10001\\01200\\00010}{6}$,
$\subspace{10004\\01200\\00010}{}$,
$\subspace{10010\\01300\\00001}{}$,
$\subspace{10040\\01300\\00001}{}$,
$\subspace{10001\\01400\\00011}{}$,
$\subspace{10004\\01400\\00011}{}$,
$\subspace{10100\\01203\\00011}{6}$,
$\subspace{10130\\01230\\00001}{}$,
$\subspace{10100\\01402\\00010}{}$,
$\subspace{10120\\01430\\00001}{}$,
$\subspace{10204\\01301\\00011}{}$,
$\subspace{10204\\01304\\00010}{}$,
$\subspace{10100\\01320\\00001}{3}$,
$\subspace{10301\\01204\\00010}{}$,
$\subspace{10301\\01401\\00011}{}$,
$\subspace{10100\\01340\\00001}{3}$,
$\subspace{10302\\01203\\00010}{}$,
$\subspace{10302\\01402\\00011}{}$,
$\subspace{10201\\01000\\00012}{6}$,
$\subspace{10201\\01100\\00014}{}$,
$\subspace{10301\\01000\\00012}{}$,
$\subspace{10301\\01100\\00014}{}$,
$\subspace{12002\\00100\\00013}{}$,
$\subspace{13002\\00100\\00013}{}$,
$\subspace{10203\\01003\\00013}{6}$,
$\subspace{10203\\01102\\00013}{}$,
$\subspace{10302\\01004\\00014}{}$,
$\subspace{10302\\01101\\00012}{}$,
$\subspace{12004\\00104\\00012}{}$,
$\subspace{13002\\00104\\00014}{}$,
$\subspace{10201\\01201\\00010}{3}$,
$\subspace{10201\\01404\\00011}{}$,
$\subspace{10400\\01330\\00001}{}$,
$\subspace{10202\\01202\\00010}{3}$,
$\subspace{10202\\01403\\00011}{}$,
$\subspace{10400\\01310\\00001}{}$,
$\subspace{10303\\01301\\00014}{6}$,
$\subspace{10303\\01304\\00012}{}$,
$\subspace{10402\\01202\\00014}{}$,
$\subspace{10403\\01201\\00013}{}$,
$\subspace{10402\\01403\\00012}{}$,
$\subspace{10403\\01404\\00013}{}$.

A projective $3-(227,5,\{2\})_5$ system, i.e., a $(32,9)^{3,5}_5$-multispread:
\def\subspace#1#2{\left(\!\begin{smallmatrix}#1\end{smallmatrix}\!\right)^{\!\makebox[0mm][l]{$\scriptscriptstyle #2$}}}
$\subspace{01013\\00123}{2}$,
$\subspace{01032\\00131}{}$,
$\subspace{10141\\01212}{3}$,
$\subspace{10112\\01414}{}$,
$\subspace{10201\\01313}{}$,
$\subspace{10302\\01334}{3}$,
$\subspace{10432\\01231}{}$,
$\subspace{10424\\01432}{}$,
$\subspace{10003\\00100\\00013}{3}$,
$\subspace{10004\\01000\\00012}{}$,
$\subspace{10004\\01100\\00014}{}$,
$\subspace{10004\\00100\\00013}{3}$,
$\subspace{10002\\01000\\00012}{}$,
$\subspace{10002\\01100\\00014}{}$,
$\subspace{10002\\01001\\00103}{3}$,
$\subspace{10020\\01030\\00110}{}$,
$\subspace{10033\\01044\\00144}{}$,
$\subspace{10002\\01004\\00102}{3}$,
$\subspace{10020\\01020\\00140}{}$,
$\subspace{10033\\01011\\00111}{}$,
$\subspace{10001\\01002\\00110}{6}$,
$\subspace{10001\\01014\\00144}{}$,
$\subspace{10010\\01001\\00120}{}$,
$\subspace{10010\\01044\\00141}{}$,
$\subspace{10044\\01021\\00133}{}$,
$\subspace{10044\\01033\\00112}{}$,
$\subspace{10000\\01003\\00122}{2}$,
$\subspace{10000\\01022\\00130}{}$,
$\subspace{10003\\01001\\00130}{6}$,
$\subspace{10003\\01032\\00122}{}$,
$\subspace{10022\\01013\\00144}{}$,
$\subspace{10022\\01044\\00131}{}$,
$\subspace{10030\\01003\\00110}{}$,
$\subspace{10030\\01022\\00123}{}$,
$\subspace{10004\\01004\\00132}{6}$,
$\subspace{10004\\01032\\00124}{}$,
$\subspace{10011\\01013\\00111}{}$,
$\subspace{10011\\01011\\00131}{}$,
$\subspace{10040\\01023\\00140}{}$,
$\subspace{10040\\01042\\00123}{}$,
$\subspace{10001\\01001\\00141}{6}$,
$\subspace{10001\\01042\\00112}{}$,
$\subspace{10010\\01014\\00110}{}$,
$\subspace{10010\\01021\\00124}{}$,
$\subspace{10044\\01023\\00144}{}$,
$\subspace{10044\\01044\\00132}{}$,
$\subspace{10001\\01004\\00143}{6}$,
$\subspace{10001\\01042\\00114}{}$,
$\subspace{10010\\01034\\00140}{}$,
$\subspace{10010\\01041\\00124}{}$,
$\subspace{10044\\01011\\00132}{}$,
$\subspace{10044\\01023\\00111}{}$,
$\subspace{10000\\01010\\00101}{1}$,
$\subspace{10000\\01014\\00114}{6}$,
$\subspace{10000\\01021\\00113}{}$,
$\subspace{10000\\01024\\00143}{}$,
$\subspace{10000\\01031\\00112}{}$,
$\subspace{10000\\01034\\00142}{}$,
$\subspace{10000\\01041\\00141}{}$,
$\subspace{10002\\01014\\00113}{6}$,
$\subspace{10002\\01020\\00142}{}$,
$\subspace{10020\\01024\\00102}{}$,
$\subspace{10020\\01031\\00141}{}$,
$\subspace{10033\\01021\\00102}{}$,
$\subspace{10033\\01020\\00112}{}$,
$\subspace{10004\\01011\\00133}{6}$,
$\subspace{10004\\01040\\00120}{}$,
$\subspace{10011\\01004\\00104}{}$,
$\subspace{10011\\01040\\00140}{}$,
$\subspace{10040\\01002\\00104}{}$,
$\subspace{10040\\01033\\00111}{}$,
$\subspace{10002\\01021\\00100}{3}$,
$\subspace{10020\\01000\\00112}{}$,
$\subspace{10033\\01014\\00141}{}$,
$\subspace{10001\\01022\\00121}{6}$,
$\subspace{10001\\01041\\00134}{}$,
$\subspace{10010\\01012\\00122}{}$,
$\subspace{10010\\01043\\00114}{}$,
$\subspace{10044\\01003\\00143}{}$,
$\subspace{10044\\01034\\00130}{}$,
$\subspace{10004\\01020\\00122}{6}$,
$\subspace{10004\\01042\\00130}{}$,
$\subspace{10011\\01023\\00102}{}$,
$\subspace{10011\\01020\\00132}{}$,
$\subspace{10040\\01003\\00124}{}$,
$\subspace{10040\\01022\\00102}{}$,
$\subspace{10003\\01033\\00102}{3}$,
$\subspace{10022\\01002\\00120}{}$,
$\subspace{10030\\01020\\00133}{}$,
$\subspace{10003\\01031\\00111}{6}$,
$\subspace{10003\\01042\\00140}{}$,
$\subspace{10022\\01024\\00132}{}$,
$\subspace{10022\\01023\\00142}{}$,
$\subspace{10030\\01004\\00124}{}$,
$\subspace{10030\\01011\\00113}{}$,
$\subspace{10001\\01041\\00102}{3}$,
$\subspace{10010\\01020\\00114}{}$,
$\subspace{10044\\01034\\00143}{}$,
$\subspace{10003\\01040\\00104}{3}$,
$\subspace{10022\\01040\\00104}{}$,
$\subspace{10030\\01040\\00104}{}$,
$\subspace{10012\\01002\\00103}{6}$,
$\subspace{10014\\01030\\00100}{}$,
$\subspace{10021\\01030\\00120}{}$,
$\subspace{10034\\01000\\00133}{}$,
$\subspace{10041\\01000\\00103}{}$,
$\subspace{10043\\01033\\00100}{}$,
$\subspace{10012\\01001\\00114}{6}$,
$\subspace{10014\\01034\\00111}{}$,
$\subspace{10021\\01041\\00110}{}$,
$\subspace{10034\\01004\\00144}{}$,
$\subspace{10041\\01011\\00143}{}$,
$\subspace{10043\\01044\\00140}{}$,
$\subspace{10013\\01000\\00111}{6}$,
$\subspace{10024\\01004\\00101}{}$,
$\subspace{10023\\01010\\00100}{}$,
$\subspace{10032\\01000\\00101}{}$,
$\subspace{10031\\01011\\00100}{}$,
$\subspace{10042\\01010\\00140}{}$,
$\subspace{10012\\01001\\00132}{6}$,
$\subspace{10014\\01042\\00103}{}$,
$\subspace{10021\\01023\\00110}{}$,
$\subspace{10034\\01030\\00144}{}$,
$\subspace{10041\\01030\\00124}{}$,
$\subspace{10043\\01044\\00103}{}$,
$\subspace{10012\\01011\\00101}{6}$,
$\subspace{10014\\01010\\00141}{}$,
$\subspace{10021\\01010\\00111}{}$,
$\subspace{10034\\01021\\00140}{}$,
$\subspace{10043\\01004\\00112}{}$,
$\subspace{10041\\01014\\00101}{}$,
$\subspace{10013\\01012\\00134}{6}$,
$\subspace{10023\\01012\\00134}{}$,
$\subspace{10024\\01012\\00134}{}$,
$\subspace{10031\\01043\\00121}{}$,
$\subspace{10032\\01043\\00121}{}$,
$\subspace{10042\\01043\\00121}{}$,
$\subspace{10014\\01014\\00143}{6}$,
$\subspace{10012\\01022\\00112}{}$,
$\subspace{10021\\01021\\00122}{}$,
$\subspace{10034\\01041\\00130}{}$,
$\subspace{10043\\01003\\00114}{}$,
$\subspace{10041\\01034\\00141}{}$,
$\subspace{10014\\01021\\00101}{6}$,
$\subspace{10012\\01020\\00141}{}$,
$\subspace{10021\\01014\\00102}{}$,
$\subspace{10034\\01010\\00102}{}$,
$\subspace{10041\\01010\\00112}{}$,
$\subspace{10043\\01020\\00101}{}$,
$\subspace{10013\\01023\\00121}{6}$,
$\subspace{10024\\01043\\00104}{}$,
$\subspace{10023\\01040\\00124}{}$,
$\subspace{10031\\01012\\00132}{}$,
$\subspace{10032\\01042\\00104}{}$,
$\subspace{10042\\01040\\00134}{}$,
$\subspace{10012\\01031\\00121}{6}$,
$\subspace{10014\\01043\\00130}{}$,
$\subspace{10021\\01012\\00113}{}$,
$\subspace{10034\\01022\\00142}{}$,
$\subspace{10041\\01003\\00134}{}$,
$\subspace{10043\\01024\\00122}{}$,
$\subspace{10013\\01033\\00143}{6}$,
$\subspace{10023\\01014\\00120}{}$,
$\subspace{10024\\01041\\00112}{}$,
$\subspace{10032\\01002\\00141}{}$,
$\subspace{10031\\01034\\00133}{}$,
$\subspace{10042\\01021\\00114}{}$,
$\subspace{10013\\01041\\00103}{6}$,
$\subspace{10024\\01030\\00104}{}$,
$\subspace{10023\\01040\\00143}{}$,
$\subspace{10032\\01034\\00104}{}$,
$\subspace{10031\\01030\\00114}{}$,
$\subspace{10042\\01040\\00103}{}$,
$\subspace{10013\\01044\\00113}{6}$,
$\subspace{10023\\01022\\00110}{}$,
$\subspace{10024\\01024\\00130}{}$,
$\subspace{10032\\01001\\00122}{}$,
$\subspace{10031\\01031\\00144}{}$,
$\subspace{10042\\01003\\00142}{}$,
$\subspace{10002\\01202\\00011}{6}$,
$\subspace{10030\\01220\\00001}{}$,
$\subspace{10003\\01301\\00010}{}$,
$\subspace{10003\\01304\\00011}{}$,
$\subspace{10002\\01403\\00010}{}$,
$\subspace{10020\\01420\\00001}{}$,
$\subspace{10100\\01002\\00010}{6}$,
$\subspace{10100\\01103\\00011}{}$,
$\subspace{10403\\01003\\00010}{}$,
$\subspace{10403\\01102\\00011}{}$,
$\subspace{11000\\00120\\00001}{}$,
$\subspace{14030\\00130\\00001}{}$,
$\subspace{10102\\01001\\00014}{6}$,
$\subspace{10102\\01104\\00012}{}$,
$\subspace{10401\\01002\\00013}{}$,
$\subspace{10401\\01103\\00013}{}$,
$\subspace{11002\\00101\\00014}{}$,
$\subspace{14003\\00101\\00012}{}$,
$\subspace{10104\\01200\\00012}{3}$,
$\subspace{10104\\01400\\00014}{}$,
$\subspace{10203\\01300\\00013}{}$,
$\subspace{10101\\01300\\00013}{3}$,
$\subspace{10303\\01200\\00012}{}$,
$\subspace{10303\\01400\\00014}{}$,
$\subspace{10100\\01320\\00001}{3}$,
$\subspace{10301\\01204\\00010}{}$,
$\subspace{10301\\01401\\00011}{}$,
$\subspace{10202\\01000\\00011}{6}$,
$\subspace{10202\\01100\\00010}{}$,
$\subspace{10330\\01000\\00001}{}$,
$\subspace{10320\\01100\\00001}{}$,
$\subspace{12003\\00100\\00011}{}$,
$\subspace{13003\\00100\\00010}{}$,
$\subspace{10202\\01002\\00011}{6}$,
$\subspace{10202\\01103\\00010}{}$,
$\subspace{10340\\01030\\00001}{}$,
$\subspace{10310\\01130\\00001}{}$,
$\subspace{12003\\00103\\00011}{}$,
$\subspace{13004\\00103\\00010}{}$,
$\subspace{10203\\01202\\00012}{6}$,
$\subspace{10204\\01203\\00012}{}$,
$\subspace{10203\\01403\\00014}{}$,
$\subspace{10204\\01402\\00014}{}$,
$\subspace{10402\\01302\\00013}{}$,
$\subspace{10402\\01303\\00013}{}$,
$\subspace{10304\\01300\\00013}{3}$,
$\subspace{10402\\01200\\00012}{}$,
$\subspace{10402\\01400\\00014}{}$.

\section{Explicit lists of generator matrices for lower bounds for \texorpdfstring{$n_4(5,2;s)$}{n4(5,2;s)}}
\label{sec_explicit_4}
 In this section we explicitly list the generator matrices of the projective $2-(n,5,s)_4$ systems found by ILP computations. For each system, we also list generators of the prescribed automorphism group of the system (which is not necessarily the full automorphism group), where $\omega\leftrightarrow\upsilon$ is the Frobenius automorphism of the field $\FF_4=\{0,1,\omega,\upsilon\}$; representatives of orbits under this group are marked by an index whose value indicates the orbit size.

\def\subspace#1#2{\left(\!\begin{smallmatrix}#1\end{smallmatrix}\!\right)^{\!\raisebox{0mm}{\makebox[0mm][l]{$\scriptscriptstyle #2$}}}}
\def\2{\scalebox{0.75}[1]{$\scriptstyle\omega$}}
\def\3{\upsilon}
$n_4(5,2;2)\ge 20$, $\left[\!\begin{smallmatrix}01000\\1{\2}000\\00010\\001{\3}0\\00001\end{smallmatrix}\!\right]$, $(\omega\leftrightarrow\upsilon)\cdot\left[\!\begin{smallmatrix}001{\3}0\\00{\3}{\3}0\\{\2}{\2}000\\{\2}1000\\00001\end{smallmatrix}\!\right]$:
$\subspace{01010\\00101}{10}$,
$\subspace{100{\2}{\2}\\001{\2}{\2}}{}$,
$\subspace{110{\3}{\2}\\001{\3}1}{}$,
$\subspace{1{\2}10{\3}\\00011}{}$,
$\subspace{1{\3}01{\3}\\0011{\2}}{}$,
$\subspace{1001{\2}\\01001}{}$,
$\subspace{10001\\01100}{}$,
$\subspace{10111\\0111{\2}}{}$,
$\subspace{10{\2}1{\2}\\011{\3}{\2}}{}$,
$\subspace{10{\3}10\\011{\2}1}{}$,
$\subspace{010{\3}{\3}\\00100}{10}$,
$\subspace{1001{\3}\\001{\2}0}{}$,
$\subspace{110{\2}{\2}\\001{\3}0}{}$,
$\subspace{1{\2}{\3}0{\3}\\00010}{}$,
$\subspace{1{\3}0{\3}{\2}\\00110}{}$,
$\subspace{100{\3}1\\01000}{}$,
$\subspace{10000\\01{\3}01}{}$,
$\subspace{101{\3}{\2}\\01{\3}{\2}1}{}$,
$\subspace{10{\2}{\3}1\\01{\3}1{\2}}{}$,
$\subspace{10{\3}{\3}{\2}\\01{\3}{\3}{\2}}{}$.

$n_4(5,2;3)\ge 39$,

$\left[\!
\begin{smallmatrix}10000\\0{\2}000\\00{\2}00\\000{\3}0\\0000{\3}\end{smallmatrix}
\!\right]$, $x\to x^2$:
$\subspace{00100\\00010}{1}$,
$\subspace{01000\\00001}{1}$,
$\subspace{01100\\00011}{1}$,
$\subspace{110{\2}{\3}\\001{\2}1}{6}$,
$\subspace{110{\3}{\2}\\001{\3}1}{}$,
$\subspace{1{\2}0{\2}1\\0011{\2}}{}$,
$\subspace{1{\2}01{\2}\\001{\3}{\2}}{}$,
$\subspace{1{\3}0{\3}1\\0011{\3}}{}$,
$\subspace{1{\3}01{\3}\\001{\2}{\3}}{}$,
$\subspace{101{\2}{\2}\\0101{\3}}{6}$,
$\subspace{101{\3}{\3}\\0101{\2}}{}$,
$\subspace{10{\2}11\\010{\2}1}{}$,
$\subspace{10{\2}{\2}{\2}\\010{\2}{\3}}{}$,
$\subspace{10{\3}11\\010{\3}1}{}$,
$\subspace{10{\3}{\3}{\3}\\010{\3}{\2}}{}$,
$\subspace{1011{\3}\\011{\2}{\3}}{6}$,
$\subspace{1011{\2}\\011{\3}{\2}}{}$,
$\subspace{10{\2}{\3}1\\0111{\3}}{}$,
$\subspace{10{\2}{\3}{\2}\\011{\3}1}{}$,
$\subspace{10{\3}{\2}1\\0111{\2}}{}$,
$\subspace{10{\3}{\2}{\3}\\011{\2}1}{}$,
$\subspace{101{\3}0\\01{\2}01}{6}$,
$\subspace{10{\2}{\2}0\\01{\2}0{\2}}{}$,
$\subspace{10{\3}10\\01{\2}0{\3}}{}$,
$\subspace{101{\2}0\\01{\3}01}{}$,
$\subspace{10{\2}10\\01{\3}0{\2}}{}$,
$\subspace{10{\3}{\3}0\\01{\3}0{\3}}{}$,
$\subspace{1010{\2}\\01{\2}10}{6}$,
$\subspace{10{\2}01\\01{\2}{\2}0}{}$,
$\subspace{10{\3}0{\3}\\01{\2}{\3}0}{}$,
$\subspace{1010{\3}\\01{\3}10}{}$,
$\subspace{10{\2}0{\2}\\01{\3}{\2}0}{}$,
$\subspace{10{\3}01\\01{\3}{\3}0}{}$,
$\subspace{101{\3}{\2}\\01{\2}11}{6}$,
$\subspace{10{\2}{\2}1\\01{\2}{\2}{\2}}{}$,
$\subspace{10{\3}1{\3}\\01{\2}{\3}{\3}}{}$,
$\subspace{101{\2}{\3}\\01{\3}11}{}$,
$\subspace{10{\2}1{\2}\\01{\3}{\2}{\2}}{}$,
$\subspace{10{\3}{\3}1\\01{\3}{\3}{\3}}{}$.

$n_4(5,2;5)\ge 75$, $\left[\!\begin{smallmatrix}01000\\1{\2}000\\00010\\001{\3}0\\00001\end{smallmatrix}\!\right]$:
$\subspace{01{\3}0{\3}\\0001{\3}}{5}$,
$\subspace{100{\3}{\2}\\0010{\3}}{}$,
$\subspace{110{\3}0\\00111}{}$,
$\subspace{1{\2}0{\3}{\3}\\001{\3}{\3}}{}$,
$\subspace{1{\3}0{\3}{\3}\\001{\2}1}{}$,
$\subspace{010{\2}1\\001{\2}{\2}}{5}$,
$\subspace{100{\2}{\2}\\0011{\2}}{}$,
$\subspace{11{\3}00\\00011}{}$,
$\subspace{1{\2}01{\3}\\00101}{}$,
$\subspace{1{\3}0{\3}0\\001{\3}1}{}$,
$\subspace{010{\3}{\2}\\001{\2}{\2}}{5}$,
$\subspace{100{\3}{\3}\\0011{\2}}{}$,
$\subspace{11100\\00011}{}$,
$\subspace{1{\2}0{\2}1\\00101}{}$,
$\subspace{1{\3}010\\001{\3}1}{}$,
$\subspace{01000\\001{\3}0}{5}$,
$\subspace{10000\\00010}{}$,
$\subspace{11000\\001{\2}0}{}$,
$\subspace{1{\2}000\\00110}{}$,
$\subspace{1{\3}000\\00100}{}$,
$\subspace{010{\3}{\2}\\001{\3}{\2}}{5}$,
$\subspace{10{\3}0{\2}\\0001{\2}}{}$,
$\subspace{110{\3}{\2}\\001{\2}{\3}}{}$,
$\subspace{1{\2}010\\0011{\3}}{}$,
$\subspace{1{\3}0{\2}{\3}\\0010{\2}}{}$,
$\subspace{1001{\3}\\0100{\3}}{5}$,
$\subspace{1000{\3}\\0110{\2}}{}$,
$\subspace{10110\\0111{\2}}{}$,
$\subspace{10{\2}1{\2}\\011{\3}{\3}}{}$,
$\subspace{10{\3}1{\2}\\011{\2}0}{}$,
$\subspace{10001\\010{\3}0}{5}$,
$\subspace{10{\3}{\2}{\2}\\01001}{}$,
$\subspace{10{\3}00\\01101}{}$,
$\subspace{10{\3}{\3}{\2}\\01{\2}{\2}{\2}}{}$,
$\subspace{10{\3}11\\01{\3}1{\2}}{}$,
$\subspace{1010{\2}\\010{\2}1}{5}$,
$\subspace{10{\2}01\\01010}{}$,
$\subspace{10{\2}{\3}{\2}\\0101{\2}}{}$,
$\subspace{100{\3}1\\01{\3}0{\2}}{}$,
$\subspace{10100\\01{\3}{\2}1}{}$,
$\subspace{101{\2}0\\010{\2}{\3}}{5}$,
$\subspace{100{\2}1\\01111}{}$,
$\subspace{10{\2}11\\011{\2}{\3}}{}$,
$\subspace{1011{\3}\\01{\2}00}{}$,
$\subspace{10{\2}0{\3}\\01{\2}11}{}$,
$\subspace{10101\\010{\3}{\3}}{5}$,
$\subspace{101{\2}{\3}\\010{\3}1}{}$,
$\subspace{10{\3}01\\01011}{}$,
$\subspace{100{\2}0\\01{\2}0{\3}}{}$,
$\subspace{10{\3}0{\3}\\01{\2}10}{}$,
$\subspace{10{\2}{\2}0\\01011}{5}$,
$\subspace{10101\\011{\2}{\2}}{}$,
$\subspace{1011{\2}\\011{\3}{\2}}{}$,
$\subspace{1001{\2}\\01{\2}{\3}1}{}$,
$\subspace{10{\2}{\2}1\\01{\2}{\3}0}{}$,
$\subspace{10{\2}00\\010{\2}{\2}}{5}$,
$\subspace{10{\2}{\2}{\2}\\010{\2}0}{}$,
$\subspace{10{\3}{\2}{\3}\\0101{\3}}{}$,
$\subspace{1010{\3}\\01{\2}{\3}{\2}}{}$,
$\subspace{10{\2}0{\2}\\01{\2}{\2}{\3}}{}$,
$\subspace{101{\2}0\\0111{\3}}{5}$,
$\subspace{10{\2}{\2}{\3}\\01{\2}00}{}$,
$\subspace{10{\2}1{\3}\\01{\2}{\3}1}{}$,
$\subspace{10{\3}{\3}1\\01{\2}{\3}{\3}}{}$,
$\subspace{10{\2}{\2}1\\01{\3}11}{}$,
$\subspace{10{\2}{\3}0\\011{\3}1}{5}$,
$\subspace{10011\\01{\2}0{\2}}{}$,
$\subspace{101{\3}{\2}\\01{\3}01}{}$,
$\subspace{101{\2}{\2}\\01{\3}{\3}{\2}}{}$,
$\subspace{10{\2}{\3}1\\01{\3}00}{}$,
$\subspace{10{\3}{\2}1\\011{\2}1}{5}$,
$\subspace{100{\2}{\3}\\01{\2}{\2}0}{}$,
$\subspace{10111\\01{\2}1{\3}}{}$,
$\subspace{101{\3}{\3}\\01{\2}{\2}1}{}$,
$\subspace{10{\3}{\2}0\\01{\3}{\3}{\3}}{}$.

$n_4(5,2;6)\ge 90$, $\left[\!\begin{smallmatrix}0{\2}000\\{\2}0000\\000{\3}0\\00{\3}00\\00001\end{smallmatrix}\!\right]$:
$\subspace{0100{\2}\\00011}{6}$,
$\subspace{01001\\0001{\2}}{}$,
$\subspace{0100{\3}\\0001{\3}}{}$,
$\subspace{1000{\2}\\00101}{}$,
$\subspace{10001\\0010{\2}}{}$,
$\subspace{1000{\3}\\0010{\3}}{}$,
$\subspace{01011\\001{\2}1}{6}$,
$\subspace{010{\2}{\3}\\001{\2}{\2}}{}$,
$\subspace{010{\3}{\2}\\001{\2}{\3}}{}$,
$\subspace{10011\\001{\3}1}{}$,
$\subspace{100{\2}{\3}\\001{\3}{\2}}{}$,
$\subspace{100{\3}{\2}\\001{\3}{\3}}{}$,
$\subspace{1100{\3}\\001{\2}1}{6}$,
$\subspace{1100{\2}\\001{\2}{\2}}{}$,
$\subspace{11001\\001{\2}{\3}}{}$,
$\subspace{1100{\2}\\001{\3}1}{}$,
$\subspace{11001\\001{\3}{\2}}{}$,
$\subspace{1100{\3}\\001{\3}{\3}}{}$,
$\subspace{1{\2}011\\00100}{6}$,
$\subspace{1{\2}0{\2}{\3}\\00100}{}$,
$\subspace{1{\2}0{\3}{\2}\\00100}{}$,
$\subspace{1{\3}10{\2}\\00010}{}$,
$\subspace{1{\3}{\2}01\\00010}{}$,
$\subspace{1{\3}{\3}0{\3}\\00010}{}$,
$\subspace{1{\2}01{\2}\\0010{\3}}{6}$,
$\subspace{1{\2}0{\2}1\\00101}{}$,
$\subspace{1{\2}0{\3}{\3}\\0010{\2}}{}$,
$\subspace{1{\3}10{\3}\\00011}{}$,
$\subspace{1{\3}{\2}0{\2}\\0001{\2}}{}$,
$\subspace{1{\3}{\3}01\\0001{\3}}{}$,
$\subspace{100{\2}1\\01010}{6}$,
$\subspace{100{\3}{\3}\\010{\2}0}{}$,
$\subspace{1001{\2}\\010{\3}0}{}$,
$\subspace{10{\3}00\\0110{\2}}{}$,
$\subspace{10100\\01{\2}01}{}$,
$\subspace{10{\2}00\\01{\3}0{\3}}{}$,
$\subspace{101{\3}1\\010{\3}{\3}}{6}$,
$\subspace{10{\2}1{\3}\\0101{\2}}{}$,
$\subspace{10{\3}{\2}{\2}\\010{\2}1}{}$,
$\subspace{1010{\2}\\011{\2}{\3}}{}$,
$\subspace{10{\2}01\\01{\2}{\3}{\2}}{}$,
$\subspace{10{\3}0{\3}\\01{\3}11}{}$,
$\subspace{100{\3}1\\0111{\2}}{6}$,
$\subspace{10{\2}{\2}1\\0110{\3}}{}$,
$\subspace{1001{\3}\\01{\2}{\2}1}{}$,
$\subspace{10{\3}{\3}{\3}\\01{\2}0{\2}}{}$,
$\subspace{100{\2}{\2}\\01{\3}{\3}{\3}}{}$,
$\subspace{1011{\2}\\01{\3}01}{}$,
$\subspace{100{\2}{\2}\\011{\2}1}{6}$,
$\subspace{101{\3}{\2}\\0110{\3}}{}$,
$\subspace{100{\3}1\\01{\2}{\3}{\3}}{}$,
$\subspace{10{\2}11\\01{\2}0{\2}}{}$,
$\subspace{1001{\3}\\01{\3}1{\2}}{}$,
$\subspace{10{\3}{\2}{\3}\\01{\3}01}{}$,
$\subspace{100{\3}0\\011{\2}{\2}}{6}$,
$\subspace{10{\3}{\2}1\\01100}{}$,
$\subspace{10010\\01{\2}{\3}1}{}$,
$\subspace{101{\3}{\3}\\01{\2}00}{}$,
$\subspace{100{\2}0\\01{\3}1{\3}}{}$,
$\subspace{10{\2}1{\2}\\01{\3}00}{}$,
$\subspace{10010\\011{\3}{\3}}{6}$,
$\subspace{10{\3}1{\3}\\01100}{}$,
$\subspace{100{\2}0\\01{\2}1{\2}}{}$,
$\subspace{101{\2}{\2}\\01{\2}00}{}$,
$\subspace{100{\3}0\\01{\3}{\2}1}{}$,
$\subspace{10{\2}{\3}1\\01{\3}00}{}$,
$\subspace{10111\\011{\2}{\3}}{6}$,
$\subspace{10{\2}1{\3}\\01111}{}$,
$\subspace{10{\2}{\2}{\3}\\01{\2}{\3}{\2}}{}$,
$\subspace{10{\3}{\2}{\2}\\01{\2}{\2}{\3}}{}$,
$\subspace{101{\3}1\\01{\3}{\3}{\2}}{}$,
$\subspace{10{\3}{\3}{\2}\\01{\3}11}{}$,
$\subspace{101{\3}0\\011{\2}0}{6}$,
$\subspace{10{\3}{\2}0\\011{\2}0}{}$,
$\subspace{101{\3}0\\01{\2}{\3}0}{}$,
$\subspace{10{\2}10\\01{\2}{\3}0}{}$,
$\subspace{10{\2}10\\01{\3}10}{}$,
$\subspace{10{\3}{\2}0\\01{\3}10}{}$,
$\subspace{10{\2}{\2}0\\0111{\2}}{6}$,
$\subspace{10{\3}{\3}{\3}\\01110}{}$,
$\subspace{1011{\2}\\01{\2}{\2}0}{}$,
$\subspace{10{\3}{\3}0\\01{\2}{\2}1}{}$,
$\subspace{10110\\01{\3}{\3}{\3}}{}$,
$\subspace{10{\2}{\2}1\\01{\3}{\3}0}{}$,
$\subspace{10{\2}1{\2}\\011{\3}1}{6}$,
$\subspace{10{\3}11\\011{\2}{\2}}{}$,
$\subspace{101{\2}{\3}\\01{\2}{\3}1}{}$,
$\subspace{10{\3}{\2}1\\01{\2}1{\3}}{}$,
$\subspace{101{\3}{\3}\\01{\3}{\2}{\2}}{}$,
$\subspace{10{\2}{\3}{\2}\\01{\3}1{\3}}{}$.

$n_4(5,2;7)\ge 107$, $\left[\!\begin{smallmatrix}00100\\10000\\01100\\00010\\00001\end{smallmatrix}\!\right]$:
$\subspace{00010\\00001}{1}$,
$\subspace{00010\\00001}{1}$,
$\subspace{0100{\2}\\00110}{7}$,
$\subspace{10010\\0011{\2}}{}$,
$\subspace{1101{\2}\\0010{\2}}{}$,
$\subspace{1001{\2}\\01010}{}$,
$\subspace{10110\\0101{\2}}{}$,
$\subspace{1000{\2}\\01110}{}$,
$\subspace{1011{\2}\\0110{\2}}{}$,
$\subspace{010{\2}0\\00110}{7}$,
$\subspace{10010\\001{\3}0}{}$,
$\subspace{110{\3}0\\001{\2}0}{}$,
$\subspace{100{\3}0\\01010}{}$,
$\subspace{10110\\010{\3}0}{}$,
$\subspace{100{\2}0\\01110}{}$,
$\subspace{101{\3}0\\011{\2}0}{}$,
$\subspace{0101{\3}\\001{\2}{\3}}{7}$,
$\subspace{100{\2}{\3}\\001{\3}0}{}$,
$\subspace{110{\3}0\\0011{\3}}{}$,
$\subspace{100{\3}0\\010{\2}{\3}}{}$,
$\subspace{101{\2}{\3}\\010{\3}0}{}$,
$\subspace{1001{\3}\\011{\2}{\3}}{}$,
$\subspace{101{\3}0\\0111{\3}}{}$,
$\subspace{01000\\001{\3}{\3}}{7}$,
$\subspace{100{\3}{\3}\\001{\3}{\3}}{}$,
$\subspace{110{\3}{\3}\\00100}{}$,
$\subspace{100{\3}{\3}\\010{\3}{\3}}{}$,
$\subspace{101{\3}{\3}\\010{\3}{\3}}{}$,
$\subspace{10000\\011{\3}{\3}}{}$,
$\subspace{101{\3}{\3}\\01100}{}$,
$\subspace{010{\3}1\\001{\3}1}{7}$,
$\subspace{100{\3}1\\00100}{}$,
$\subspace{11000\\001{\3}1}{}$,
$\subspace{10000\\010{\3}1}{}$,
$\subspace{101{\3}1\\01000}{}$,
$\subspace{100{\3}1\\011{\3}1}{}$,
$\subspace{10100\\011{\3}1}{}$,
$\subspace{1{\2}001\\0011{\3}}{7}$,
$\subspace{10{\2}{\3}0\\0101{\3}}{}$,
$\subspace{10{\3}0{\3}\\0111{\3}}{}$,
$\subspace{1001{\3}\\01{\2}{\2}0}{}$,
$\subspace{1011{\3}\\01{\2}11}{}$,
$\subspace{10{\2}{\2}{\2}\\01{\2}{\3}1}{}$,
$\subspace{10{\2}1{\2}\\01{\3}01}{}$,
$\subspace{1{\2}0{\2}0\\00111}{7}$,
$\subspace{10{\2}0{\3}\\01011}{}$,
$\subspace{10{\3}10\\01111}{}$,
$\subspace{10011\\01{\2}0{\2}}{}$,
$\subspace{10111\\01{\2}{\2}1}{}$,
$\subspace{10{\2}{\3}{\2}\\01{\2}{\2}{\3}}{}$,
$\subspace{10{\2}{\3}1\\01{\3}{\2}0}{}$,
$\subspace{1{\2}0{\3}{\2}\\00111}{7}$,
$\subspace{10{\2}{\2}0\\01011}{}$,
$\subspace{10{\3}{\2}1\\01111}{}$,
$\subspace{10011\\01{\2}10}{}$,
$\subspace{10111\\01{\2}0{\2}}{}$,
$\subspace{10{\2}0{\3}\\01{\2}1{\2}}{}$,
$\subspace{10{\2}{\2}{\3}\\01{\3}{\3}{\2}}{}$,
$\subspace{1{\2}0{\2}{\3}\\001{\3}{\2}}{7}$,
$\subspace{10{\2}10\\010{\3}{\2}}{}$,
$\subspace{10{\3}1{\2}\\011{\3}{\2}}{}$,
$\subspace{100{\3}{\2}\\01{\2}{\3}0}{}$,
$\subspace{101{\3}{\2}\\01{\2}0{\3}}{}$,
$\subspace{10{\2}01\\01{\2}{\3}{\3}}{}$,
$\subspace{10{\2}11\\01{\3}{\2}{\3}}{}$,
$\subspace{1{\2}0{\2}{\3}\\001{\3}{\2}}{7}$,
$\subspace{10{\2}10\\010{\3}{\2}}{}$,
$\subspace{10{\3}1{\2}\\011{\3}{\2}}{}$,
$\subspace{100{\3}{\2}\\01{\2}{\3}0}{}$,
$\subspace{101{\3}{\2}\\01{\2}0{\3}}{}$,
$\subspace{10{\2}01\\01{\2}{\3}{\3}}{}$,
$\subspace{10{\2}11\\01{\3}{\2}{\3}}{}$,
$\subspace{1{\3}001\\00101}{7}$,
$\subspace{10{\3}01\\01001}{}$,
$\subspace{10{\2}0{\2}\\01101}{}$,
$\subspace{10{\3}00\\01{\2}01}{}$,
$\subspace{10001\\01{\3}0{\2}}{}$,
$\subspace{10101\\01{\3}0{\2}}{}$,
$\subspace{10{\3}01\\01{\3}00}{}$,
$\subspace{1{\3}01{\3}\\0010{\3}}{7}$,
$\subspace{10{\3}{\3}{\3}\\0100{\3}}{}$,
$\subspace{10{\2}{\2}1\\0110{\3}}{}$,
$\subspace{10{\3}10\\01{\2}1{\3}}{}$,
$\subspace{1000{\3}\\01{\3}11}{}$,
$\subspace{1010{\3}\\01{\3}{\3}1}{}$,
$\subspace{10{\3}{\2}{\3}\\01{\3}{\2}0}{}$,
$\subspace{1{\3}0{\3}{\2}\\0011{\2}}{7}$,
$\subspace{10{\3}0{\2}\\0101{\2}}{}$,
$\subspace{10{\2}1{\3}\\0111{\2}}{}$,
$\subspace{10{\3}{\2}0\\01{\2}{\3}{\2}}{}$,
$\subspace{1001{\2}\\01{\3}0{\3}}{}$,
$\subspace{1011{\2}\\01{\3}{\3}{\3}}{}$,
$\subspace{10{\3}{\2}{\2}\\01{\3}{\3}0}{}$,
$\subspace{1{\3}001\\001{\2}{\2}}{7}$,
$\subspace{10{\3}{\3}0\\010{\2}{\2}}{}$,
$\subspace{10{\2}0{\2}\\011{\2}{\2}}{}$,
$\subspace{10{\3}{\2}{\3}\\01{\2}01}{}$,
$\subspace{100{\2}{\2}\\01{\3}10}{}$,
$\subspace{101{\2}{\2}\\01{\3}{\2}1}{}$,
$\subspace{10{\3}1{\3}\\01{\3}{\3}1}{}$,
$\subspace{1{\3}01{\3}\\001{\2}1}{7}$,
$\subspace{10{\3}00\\010{\2}1}{}$,
$\subspace{10{\2}{\2}1\\011{\2}1}{}$,
$\subspace{10{\3}{\3}{\2}\\01{\2}1{\3}}{}$,
$\subspace{100{\2}1\\01{\3}00}{}$,
$\subspace{101{\2}1\\01{\3}1{\3}}{}$,
$\subspace{10{\3}{\3}{\2}\\01{\3}1{\3}}{}$.

$n_4(5,2;9)\ge 141$, $\left[\!\begin{smallmatrix}00100\\10000\\01100\\00010\\00001\end{smallmatrix}\!\right]$:
$\subspace{00010\\00001}{1}$,
$\subspace{010{\2}{\3}\\0010{\2}}{7}$,
$\subspace{1000{\2}\\001{\2}1}{}$,
$\subspace{110{\2}1\\001{\2}{\3}}{}$,
$\subspace{100{\2}1\\0100{\2}}{}$,
$\subspace{1010{\2}\\010{\2}1}{}$,
$\subspace{100{\2}{\3}\\0110{\2}}{}$,
$\subspace{101{\2}1\\011{\2}{\3}}{}$,
$\subspace{010{\3}1\\00101}{7}$,
$\subspace{10001\\001{\3}0}{}$,
$\subspace{110{\3}0\\001{\3}1}{}$,
$\subspace{100{\3}0\\01001}{}$,
$\subspace{10101\\010{\3}0}{}$,
$\subspace{100{\3}1\\01101}{}$,
$\subspace{101{\3}0\\011{\3}1}{}$,
$\subspace{010{\3}1\\0010{\3}}{7}$,
$\subspace{1000{\3}\\001{\3}{\2}}{}$,
$\subspace{110{\3}{\2}\\001{\3}1}{}$,
$\subspace{100{\3}{\2}\\0100{\3}}{}$,
$\subspace{1010{\3}\\010{\3}{\2}}{}$,
$\subspace{100{\3}1\\0110{\3}}{}$,
$\subspace{101{\3}{\2}\\011{\3}1}{}$,
$\subspace{01000\\0011{\3}}{7}$,
$\subspace{1001{\3}\\0011{\3}}{}$,
$\subspace{1101{\3}\\00100}{}$,
$\subspace{1001{\3}\\0101{\3}}{}$,
$\subspace{1011{\3}\\0101{\3}}{}$,
$\subspace{10000\\0111{\3}}{}$,
$\subspace{1011{\3}\\01100}{}$,
$\subspace{010{\2}{\3}\\00111}{7}$,
$\subspace{10011\\001{\3}{\2}}{}$,
$\subspace{110{\3}{\2}\\001{\2}{\3}}{}$,
$\subspace{100{\3}{\2}\\01011}{}$,
$\subspace{10111\\010{\3}{\2}}{}$,
$\subspace{100{\2}{\3}\\01111}{}$,
$\subspace{101{\3}{\2}\\011{\2}{\3}}{}$,
$\subspace{010{\2}{\2}\\001{\2}1}{7}$,
$\subspace{100{\2}1\\0010{\3}}{}$,
$\subspace{1100{\3}\\001{\2}{\2}}{}$,
$\subspace{1000{\3}\\010{\2}1}{}$,
$\subspace{101{\2}1\\0100{\3}}{}$,
$\subspace{100{\2}{\2}\\011{\2}1}{}$,
$\subspace{1010{\3}\\011{\2}{\2}}{}$,
$\subspace{1{\2}00{\3}\\0010{\2}}{7}$,
$\subspace{10{\2}00\\0100{\2}}{}$,
$\subspace{10{\3}0{\2}\\0110{\2}}{}$,
$\subspace{1000{\2}\\01{\2}00}{}$,
$\subspace{1010{\2}\\01{\2}0{\3}}{}$,
$\subspace{10{\2}01\\01{\2}0{\3}}{}$,
$\subspace{10{\2}01\\01{\3}0{\3}}{}$,
$\subspace{1{\2}0{\3}0\\0011{\2}}{7}$,
$\subspace{10{\2}{\2}1\\0101{\2}}{}$,
$\subspace{10{\3}{\2}0\\0111{\2}}{}$,
$\subspace{1001{\2}\\01{\2}1{\3}}{}$,
$\subspace{1011{\2}\\01{\2}0{\2}}{}$,
$\subspace{10{\2}0{\3}\\01{\2}11}{}$,
$\subspace{10{\2}{\2}{\2}\\01{\3}{\3}0}{}$,
$\subspace{1{\2}00{\2}\\001{\2}0}{7}$,
$\subspace{10{\2}1{\3}\\010{\2}0}{}$,
$\subspace{10{\3}01\\011{\2}0}{}$,
$\subspace{100{\2}0\\01{\2}{\3}{\2}}{}$,
$\subspace{101{\2}0\\01{\2}{\2}{\3}}{}$,
$\subspace{10{\2}{\3}1\\01{\2}11}{}$,
$\subspace{10{\2}{\2}{\2}\\01{\3}0{\2}}{}$,
$\subspace{1{\2}0{\2}0\\001{\2}{\2}}{7}$,
$\subspace{10{\2}{\2}1\\010{\2}{\2}}{}$,
$\subspace{10{\3}10\\011{\2}{\2}}{}$,
$\subspace{100{\2}{\2}\\01{\2}1{\3}}{}$,
$\subspace{101{\2}{\2}\\01{\2}1{\2}}{}$,
$\subspace{10{\2}{\2}{\3}\\01{\2}01}{}$,
$\subspace{10{\2}0{\2}\\01{\3}{\2}0}{}$,
$\subspace{1{\2}00{\2}\\001{\3}0}{7}$,
$\subspace{10{\2}{\2}{\3}\\010{\3}0}{}$,
$\subspace{10{\3}01\\011{\3}0}{}$,
$\subspace{100{\3}0\\01{\2}1{\2}}{}$,
$\subspace{101{\3}0\\01{\2}{\3}{\3}}{}$,
$\subspace{10{\2}11\\01{\2}{\2}1}{}$,
$\subspace{10{\2}{\3}{\2}\\01{\3}0{\2}}{}$,
$\subspace{1{\2}0{\2}{\3}\\001{\3}{\3}}{7}$,
$\subspace{10{\2}1{\3}\\010{\3}{\3}}{}$,
$\subspace{10{\3}1{\2}\\011{\3}{\3}}{}$,
$\subspace{100{\3}{\3}\\01{\2}{\3}{\2}}{}$,
$\subspace{101{\3}{\3}\\01{\2}0{\2}}{}$,
$\subspace{10{\2}0{\3}\\01{\2}{\3}0}{}$,
$\subspace{10{\2}10\\01{\3}{\2}{\3}}{}$,
$\subspace{1{\3}010\\00100}{7}$,
$\subspace{10{\3}{\3}0\\01000}{}$,
$\subspace{10{\2}{\2}0\\01100}{}$,
$\subspace{10{\3}10\\01{\2}10}{}$,
$\subspace{10000\\01{\3}10}{}$,
$\subspace{10100\\01{\3}{\3}0}{}$,
$\subspace{10{\3}{\2}0\\01{\3}{\2}0}{}$,
$\subspace{1{\3}0{\2}0\\00101}{7}$,
$\subspace{10{\3}1{\2}\\01001}{}$,
$\subspace{10{\2}{\3}0\\01101}{}$,
$\subspace{10{\3}{\2}1\\01{\2}{\2}0}{}$,
$\subspace{10001\\01{\3}{\2}{\3}}{}$,
$\subspace{10101\\01{\3}11}{}$,
$\subspace{10{\3}{\3}{\3}\\01{\3}{\3}{\2}}{}$,
$\subspace{1{\3}000\\00111}{7}$,
$\subspace{10{\3}{\2}{\2}\\01011}{}$,
$\subspace{10{\2}00\\01111}{}$,
$\subspace{10{\3}11\\01{\2}00}{}$,
$\subspace{10011\\01{\3}{\3}{\3}}{}$,
$\subspace{10111\\01{\3}11}{}$,
$\subspace{10{\3}{\3}{\3}\\01{\3}{\2}{\2}}{}$,
$\subspace{1{\3}0{\2}{\2}\\00110}{7}$,
$\subspace{10{\3}{\3}1\\01010}{}$,
$\subspace{10{\2}{\3}{\3}\\01110}{}$,
$\subspace{10{\3}{\3}{\2}\\01{\2}{\2}{\2}}{}$,
$\subspace{10010\\01{\3}1{\2}}{}$,
$\subspace{10110\\01{\3}01}{}$,
$\subspace{10{\3}0{\3}\\01{\3}1{\3}}{}$,
$\subspace{1{\3}0{\2}{\2}\\00110}{7}$,
$\subspace{10{\3}{\3}1\\01010}{}$,
$\subspace{10{\2}{\3}{\3}\\01110}{}$,
$\subspace{10{\3}{\3}{\2}\\01{\2}{\2}{\2}}{}$,
$\subspace{10010\\01{\3}1{\2}}{}$,
$\subspace{10110\\01{\3}01}{}$,
$\subspace{10{\3}0{\3}\\01{\3}1{\3}}{}$,
$\subspace{1{\3}0{\3}1\\0011{\2}}{7}$,
$\subspace{10{\3}00\\0101{\2}}{}$,
$\subspace{10{\2}1{\2}\\0111{\2}}{}$,
$\subspace{10{\3}{\2}{\3}\\01{\2}{\3}1}{}$,
$\subspace{1001{\2}\\01{\3}00}{}$,
$\subspace{1011{\2}\\01{\3}{\3}1}{}$,
$\subspace{10{\3}{\2}{\3}\\01{\3}{\3}1}{}$,
$\subspace{1{\3}0{\3}1\\001{\2}0}{7}$,
$\subspace{10{\3}1{\3}\\010{\2}0}{}$,
$\subspace{10{\2}1{\2}\\011{\2}0}{}$,
$\subspace{10{\3}11\\01{\2}{\3}1}{}$,
$\subspace{100{\2}0\\01{\3}{\2}1}{}$,
$\subspace{101{\2}0\\01{\3}0{\3}}{}$,
$\subspace{10{\3}0{\2}\\01{\3}{\2}{\2}}{}$,
$\subspace{1{\3}0{\2}{\3}\\001{\3}{\3}}{7}$,
$\subspace{10{\3}0{\3}\\010{\3}{\3}}{}$,
$\subspace{10{\2}{\3}1\\011{\3}{\3}}{}$,
$\subspace{10{\3}10\\01{\2}{\2}{\3}}{}$,
$\subspace{100{\3}{\3}\\01{\3}01}{}$,
$\subspace{101{\3}{\3}\\01{\3}{\2}1}{}$,
$\subspace{10{\3}1{\3}\\01{\3}{\2}0}{}$.

$n_4(5,2;10)\ge 156$, $\left[\!\begin{smallmatrix}0{\2}000\\{\2}0000\\000{\3}0\\00{\3}00\\00001\end{smallmatrix}\!\right]$, $\omega\leftrightarrow\upsilon$:
$\subspace{001{\2}0\\00001}{2}$,
$\subspace{001{\3}0\\00001}{}$,
$\subspace{01000\\00010}{2}$,
$\subspace{10000\\00100}{}$,
$\subspace{01000\\00010}{2}$,
$\subspace{10000\\00100}{}$,
$\subspace{0101{\2}\\00101}{12}$,
$\subspace{0101{\3}\\00101}{}$,
$\subspace{010{\2}1\\0010{\2}}{}$,
$\subspace{010{\2}{\2}\\0010{\2}}{}$,
$\subspace{010{\3}1\\0010{\3}}{}$,
$\subspace{010{\3}{\3}\\0010{\3}}{}$,
$\subspace{1010{\2}\\00011}{}$,
$\subspace{1010{\3}\\00011}{}$,
$\subspace{10{\2}01\\0001{\2}}{}$,
$\subspace{10{\2}0{\2}\\0001{\2}}{}$,
$\subspace{10{\3}01\\0001{\3}}{}$,
$\subspace{10{\3}0{\3}\\0001{\3}}{}$,
$\subspace{1{\2}01{\3}\\0011{\2}}{6}$,
$\subspace{1{\2}0{\2}{\2}\\0011{\3}}{}$,
$\subspace{1{\2}0{\3}1\\00111}{}$,
$\subspace{1{\3}01{\2}\\0011{\3}}{}$,
$\subspace{1{\3}0{\2}1\\00111}{}$,
$\subspace{1{\3}0{\3}{\3}\\0011{\2}}{}$,
$\subspace{1{\2}01{\2}\\001{\2}{\2}}{6}$,
$\subspace{1{\2}0{\2}1\\001{\2}{\3}}{}$,
$\subspace{1{\2}0{\3}{\3}\\001{\2}1}{}$,
$\subspace{1{\3}01{\3}\\001{\3}{\3}}{}$,
$\subspace{1{\3}0{\2}{\2}\\001{\3}1}{}$,
$\subspace{1{\3}0{\3}1\\001{\3}{\2}}{}$,
$\subspace{100{\2}0\\0101{\2}}{12}$,
$\subspace{100{\3}0\\0101{\3}}{}$,
$\subspace{10010\\010{\2}{\2}}{}$,
$\subspace{100{\3}0\\010{\2}1}{}$,
$\subspace{10010\\010{\3}{\3}}{}$,
$\subspace{100{\2}0\\010{\3}1}{}$,
$\subspace{10{\2}0{\2}\\01100}{}$,
$\subspace{10{\3}0{\3}\\01100}{}$,
$\subspace{1010{\2}\\01{\2}00}{}$,
$\subspace{10{\3}01\\01{\2}00}{}$,
$\subspace{1010{\3}\\01{\3}00}{}$,
$\subspace{10{\2}01\\01{\3}00}{}$,
$\subspace{101{\2}0\\01001}{12}$,
$\subspace{101{\3}0\\01001}{}$,
$\subspace{10{\2}10\\0100{\3}}{}$,
$\subspace{10{\2}{\3}0\\0100{\3}}{}$,
$\subspace{10{\3}10\\0100{\2}}{}$,
$\subspace{10{\3}{\2}0\\0100{\2}}{}$,
$\subspace{1000{\3}\\011{\2}0}{}$,
$\subspace{1000{\2}\\011{\3}0}{}$,
$\subspace{10001\\01{\2}10}{}$,
$\subspace{1000{\2}\\01{\2}{\3}0}{}$,
$\subspace{10001\\01{\3}10}{}$,
$\subspace{1000{\3}\\01{\3}{\2}0}{}$,
$\subspace{10111\\01011}{6}$,
$\subspace{10{\2}{\2}{\3}\\010{\2}{\3}}{}$,
$\subspace{10{\3}{\3}{\2}\\010{\3}{\2}}{}$,
$\subspace{10101\\01111}{}$,
$\subspace{10{\2}0{\3}\\01{\2}{\2}{\3}}{}$,
$\subspace{10{\3}0{\2}\\01{\3}{\3}{\2}}{}$,
$\subspace{10111\\01011}{6}$,
$\subspace{10{\2}{\2}{\3}\\010{\2}{\3}}{}$,
$\subspace{10{\3}{\3}{\2}\\010{\3}{\2}}{}$,
$\subspace{10101\\01111}{}$,
$\subspace{10{\2}0{\3}\\01{\2}{\2}{\3}}{}$,
$\subspace{10{\3}0{\2}\\01{\3}{\3}{\2}}{}$,
$\subspace{101{\2}0\\010{\2}{\2}}{12}$,
$\subspace{101{\3}0\\010{\3}{\3}}{}$,
$\subspace{10{\2}10\\0101{\2}}{}$,
$\subspace{10{\2}{\3}0\\010{\3}1}{}$,
$\subspace{10{\3}10\\0101{\3}}{}$,
$\subspace{10{\3}{\2}0\\010{\2}1}{}$,
$\subspace{1010{\2}\\011{\2}0}{}$,
$\subspace{1010{\3}\\011{\3}0}{}$,
$\subspace{10{\2}0{\2}\\01{\2}10}{}$,
$\subspace{10{\2}01\\01{\2}{\3}0}{}$,
$\subspace{10{\3}0{\3}\\01{\3}10}{}$,
$\subspace{10{\3}01\\01{\3}{\2}0}{}$,
$\subspace{1001{\3}\\0110{\2}}{6}$,
$\subspace{1001{\2}\\0110{\3}}{}$,
$\subspace{100{\2}{\2}\\01{\2}01}{}$,
$\subspace{100{\2}1\\01{\2}0{\2}}{}$,
$\subspace{100{\3}{\3}\\01{\3}01}{}$,
$\subspace{100{\3}1\\01{\3}0{\3}}{}$,
$\subspace{100{\2}{\2}\\0110{\2}}{6}$,
$\subspace{100{\3}{\3}\\0110{\3}}{}$,
$\subspace{1001{\2}\\01{\2}0{\2}}{}$,
$\subspace{100{\3}1\\01{\2}01}{}$,
$\subspace{1001{\3}\\01{\3}0{\3}}{}$,
$\subspace{100{\2}1\\01{\3}01}{}$,
$\subspace{1001{\3}\\011{\2}{\2}}{12}$,
$\subspace{1001{\2}\\011{\3}{\3}}{}$,
$\subspace{10{\2}1{\2}\\0110{\3}}{}$,
$\subspace{10{\3}1{\3}\\0110{\2}}{}$,
$\subspace{100{\2}1\\01{\2}1{\2}}{}$,
$\subspace{100{\2}{\2}\\01{\2}{\3}1}{}$,
$\subspace{101{\2}{\2}\\01{\2}01}{}$,
$\subspace{10{\3}{\2}1\\01{\2}0{\2}}{}$,
$\subspace{100{\3}1\\01{\3}1{\3}}{}$,
$\subspace{100{\3}{\3}\\01{\3}{\2}1}{}$,
$\subspace{101{\3}{\3}\\01{\3}01}{}$,
$\subspace{10{\2}{\3}1\\01{\3}0{\3}}{}$,
$\subspace{100{\3}{\2}\\011{\2}1}{12}$,
$\subspace{100{\2}{\3}\\011{\3}1}{}$,
$\subspace{10{\2}{\3}{\2}\\01101}{}$,
$\subspace{10{\3}{\2}{\3}\\01101}{}$,
$\subspace{100{\3}{\2}\\01{\2}1{\3}}{}$,
$\subspace{10011\\01{\2}{\3}{\3}}{}$,
$\subspace{101{\3}{\2}\\01{\2}0{\3}}{}$,
$\subspace{10{\3}11\\01{\2}0{\3}}{}$,
$\subspace{100{\2}{\3}\\01{\3}1{\2}}{}$,
$\subspace{10011\\01{\3}{\2}{\2}}{}$,
$\subspace{101{\2}{\3}\\01{\3}0{\2}}{}$,
$\subspace{10{\2}11\\01{\3}0{\2}}{}$,
$\subspace{101{\2}{\3}\\011{\2}1}{6}$,
$\subspace{101{\3}{\2}\\011{\3}1}{}$,
$\subspace{10{\2}11\\01{\2}1{\3}}{}$,
$\subspace{10{\2}{\3}{\2}\\01{\2}{\3}{\3}}{}$,
$\subspace{10{\3}11\\01{\3}1{\2}}{}$,
$\subspace{10{\3}{\2}{\3}\\01{\3}{\2}{\2}}{}$,
$\subspace{101{\3}1\\011{\2}{\3}}{12}$,
$\subspace{101{\2}1\\011{\3}{\2}}{}$,
$\subspace{10{\2}{\3}{\3}\\011{\3}{\2}}{}$,
$\subspace{10{\3}{\2}{\2}\\011{\2}{\3}}{}$,
$\subspace{101{\3}1\\01{\2}{\3}{\2}}{}$,
$\subspace{10{\2}{\3}{\3}\\01{\2}11}{}$,
$\subspace{10{\2}1{\3}\\01{\2}{\3}{\2}}{}$,
$\subspace{10{\3}1{\2}\\01{\2}11}{}$,
$\subspace{101{\2}1\\01{\3}{\2}{\3}}{}$,
$\subspace{10{\2}1{\3}\\01{\3}11}{}$,
$\subspace{10{\3}{\2}{\2}\\01{\3}11}{}$,
$\subspace{10{\3}1{\2}\\01{\3}{\2}{\3}}{}$,
$\subspace{10{\2}{\3}1\\01110}{12}$,
$\subspace{10{\2}{\2}0\\011{\3}{\3}}{}$,
$\subspace{10{\3}{\2}1\\01110}{}$,
$\subspace{10{\3}{\3}0\\011{\2}{\2}}{}$,
$\subspace{101{\3}{\3}\\01{\2}{\2}0}{}$,
$\subspace{10110\\01{\2}{\3}1}{}$,
$\subspace{10{\3}{\3}0\\01{\2}1{\2}}{}$,
$\subspace{10{\3}1{\3}\\01{\2}{\2}0}{}$,
$\subspace{10110\\01{\3}{\2}1}{}$,
$\subspace{101{\2}{\2}\\01{\3}{\3}0}{}$,
$\subspace{10{\2}{\2}0\\01{\3}1{\3}}{}$,
$\subspace{10{\2}1{\2}\\01{\3}{\3}0}{}$,
$\subspace{10{\2}{\3}{\2}\\0111{\2}}{12}$,
$\subspace{10{\2}{\2}1\\011{\3}1}{}$,
$\subspace{10{\3}{\2}{\3}\\0111{\3}}{}$,
$\subspace{10{\3}{\3}1\\011{\2}1}{}$,
$\subspace{101{\3}{\2}\\01{\2}{\2}{\2}}{}$,
$\subspace{1011{\3}\\01{\2}{\3}{\3}}{}$,
$\subspace{10{\3}{\3}{\3}\\01{\2}1{\3}}{}$,
$\subspace{10{\3}11\\01{\2}{\2}1}{}$,
$\subspace{1011{\2}\\01{\3}{\2}{\2}}{}$,
$\subspace{101{\2}{\3}\\01{\3}{\3}{\3}}{}$,
$\subspace{10{\2}{\2}{\2}\\01{\3}1{\2}}{}$,
$\subspace{10{\2}11\\01{\3}{\3}1}{}$.

$n_4(5,2;11)\ge 175$, $\left[\!\begin{smallmatrix}00100\\10000\\01100\\00010\\00001\end{smallmatrix}\!\right]$:
$\subspace{0101{\2}\\00101}{7}$,
$\subspace{10001\\0011{\3}}{}$,
$\subspace{1101{\3}\\0011{\2}}{}$,
$\subspace{1001{\3}\\01001}{}$,
$\subspace{10101\\0101{\3}}{}$,
$\subspace{1001{\2}\\01101}{}$,
$\subspace{1011{\3}\\0111{\2}}{}$,
$\subspace{010{\2}0\\0010{\2}}{7}$,
$\subspace{1000{\2}\\001{\2}{\2}}{}$,
$\subspace{110{\2}{\2}\\001{\2}0}{}$,
$\subspace{100{\2}{\2}\\0100{\2}}{}$,
$\subspace{1010{\2}\\010{\2}{\2}}{}$,
$\subspace{100{\2}0\\0110{\2}}{}$,
$\subspace{101{\2}{\2}\\011{\2}0}{}$,
$\subspace{010{\2}{\3}\\0010{\3}}{7}$,
$\subspace{1000{\3}\\001{\2}0}{}$,
$\subspace{110{\2}0\\001{\2}{\3}}{}$,
$\subspace{100{\2}0\\0100{\3}}{}$,
$\subspace{1010{\3}\\010{\2}0}{}$,
$\subspace{100{\2}{\3}\\0110{\3}}{}$,
$\subspace{101{\2}0\\011{\2}{\3}}{}$,
$\subspace{010{\3}0\\0010{\2}}{7}$,
$\subspace{1000{\2}\\001{\3}{\2}}{}$,
$\subspace{110{\3}{\2}\\001{\3}0}{}$,
$\subspace{100{\3}{\2}\\0100{\2}}{}$,
$\subspace{1010{\2}\\010{\3}{\2}}{}$,
$\subspace{100{\3}0\\0110{\2}}{}$,
$\subspace{101{\3}{\2}\\011{\3}0}{}$,
$\subspace{01001\\00110}{7}$,
$\subspace{10010\\00111}{}$,
$\subspace{11011\\00101}{}$,
$\subspace{10011\\01010}{}$,
$\subspace{10110\\01011}{}$,
$\subspace{10001\\01110}{}$,
$\subspace{10111\\01101}{}$,
$\subspace{01010\\00111}{7}$,
$\subspace{10011\\00101}{}$,
$\subspace{11001\\00110}{}$,
$\subspace{10001\\01011}{}$,
$\subspace{10111\\01001}{}$,
$\subspace{10010\\01111}{}$,
$\subspace{10101\\01110}{}$,
$\subspace{010{\2}{\2}\\0011{\2}}{7}$,
$\subspace{1001{\2}\\001{\3}0}{}$,
$\subspace{110{\3}0\\001{\2}{\2}}{}$,
$\subspace{100{\3}0\\0101{\2}}{}$,
$\subspace{1011{\2}\\010{\3}0}{}$,
$\subspace{100{\2}{\2}\\0111{\2}}{}$,
$\subspace{101{\3}0\\011{\2}{\2}}{}$,
$\subspace{010{\3}1\\0011{\2}}{7}$,
$\subspace{1001{\2}\\001{\2}{\3}}{}$,
$\subspace{110{\2}{\3}\\001{\3}1}{}$,
$\subspace{100{\2}{\3}\\0101{\2}}{}$,
$\subspace{1011{\2}\\010{\2}{\3}}{}$,
$\subspace{100{\3}1\\0111{\2}}{}$,
$\subspace{101{\2}{\3}\\011{\3}1}{}$,
$\subspace{0100{\3}\\001{\2}{\3}}{7}$,
$\subspace{100{\2}{\3}\\001{\2}0}{}$,
$\subspace{110{\2}0\\0010{\3}}{}$,
$\subspace{100{\2}0\\010{\2}{\3}}{}$,
$\subspace{101{\2}{\3}\\010{\2}0}{}$,
$\subspace{1000{\3}\\011{\2}{\3}}{}$,
$\subspace{101{\2}0\\0110{\3}}{}$,
$\subspace{0100{\3}\\001{\3}1}{7}$,
$\subspace{100{\3}1\\001{\3}{\2}}{}$,
$\subspace{110{\3}{\2}\\0010{\3}}{}$,
$\subspace{100{\3}{\2}\\010{\3}1}{}$,
$\subspace{101{\3}1\\010{\3}{\2}}{}$,
$\subspace{1000{\3}\\011{\3}1}{}$,
$\subspace{101{\3}{\2}\\0110{\3}}{}$,
$\subspace{0101{\3}\\001{\3}{\2}}{7}$,
$\subspace{100{\3}{\2}\\001{\2}1}{}$,
$\subspace{110{\2}1\\0011{\3}}{}$,
$\subspace{100{\2}1\\010{\3}{\2}}{}$,
$\subspace{101{\3}{\2}\\010{\2}1}{}$,
$\subspace{1001{\3}\\011{\3}{\2}}{}$,
$\subspace{101{\2}1\\0111{\3}}{}$,
$\subspace{1{\2}00{\2}\\00100}{7}$,
$\subspace{10{\2}0{\3}\\01000}{}$,
$\subspace{10{\3}01\\01100}{}$,
$\subspace{10000\\01{\2}0{\2}}{}$,
$\subspace{10100\\01{\2}0{\3}}{}$,
$\subspace{10{\2}01\\01{\2}01}{}$,
$\subspace{10{\2}0{\2}\\01{\3}0{\2}}{}$,
$\subspace{1{\2}0{\3}{\3}\\00100}{7}$,
$\subspace{10{\2}11\\01000}{}$,
$\subspace{10{\3}{\2}{\2}\\01100}{}$,
$\subspace{10000\\01{\2}{\3}{\3}}{}$,
$\subspace{10100\\01{\2}11}{}$,
$\subspace{10{\2}{\2}{\2}\\01{\2}{\2}{\2}}{}$,
$\subspace{10{\2}{\3}{\3}\\01{\3}{\3}{\3}}{}$,
$\subspace{1{\2}01{\3}\\00110}{7}$,
$\subspace{10{\2}11\\01010}{}$,
$\subspace{10{\3}{\3}{\2}\\01110}{}$,
$\subspace{10010\\01{\2}{\3}{\3}}{}$,
$\subspace{10110\\01{\2}{\3}1}{}$,
$\subspace{10{\2}1{\2}\\01{\2}0{\2}}{}$,
$\subspace{10{\2}0{\3}\\01{\3}1{\3}}{}$,
$\subspace{1{\2}0{\2}{\2}\\0011{\3}}{7}$,
$\subspace{10{\2}01\\0101{\3}}{}$,
$\subspace{10{\3}11\\0111{\3}}{}$,
$\subspace{1001{\3}\\01{\2}0{\3}}{}$,
$\subspace{1011{\3}\\01{\2}{\2}0}{}$,
$\subspace{10{\2}{\3}0\\01{\2}{\2}{\3}}{}$,
$\subspace{10{\2}{\3}1\\01{\3}{\2}{\2}}{}$,
$\subspace{1{\2}0{\3}0\\00111}{7}$,
$\subspace{10{\2}{\2}{\3}\\01011}{}$,
$\subspace{10{\3}{\2}0\\01111}{}$,
$\subspace{10011\\01{\2}1{\2}}{}$,
$\subspace{10111\\01{\2}01}{}$,
$\subspace{10{\2}0{\2}\\01{\2}1{\3}}{}$,
$\subspace{10{\2}{\2}1\\01{\3}{\3}0}{}$,
$\subspace{1{\2}00{\3}\\001{\2}{\2}}{7}$,
$\subspace{10{\2}10\\010{\2}{\2}}{}$,
$\subspace{10{\3}0{\2}\\011{\2}{\2}}{}$,
$\subspace{100{\2}{\2}\\01{\2}{\3}0}{}$,
$\subspace{101{\2}{\2}\\01{\2}{\2}{\3}}{}$,
$\subspace{10{\2}{\3}1\\01{\2}1{\3}}{}$,
$\subspace{10{\2}{\2}1\\01{\3}0{\3}}{}$,
$\subspace{1{\2}001\\001{\3}{\3}}{7}$,
$\subspace{10{\2}{\2}0\\010{\3}{\3}}{}$,
$\subspace{10{\3}0{\3}\\011{\3}{\3}}{}$,
$\subspace{100{\3}{\3}\\01{\2}10}{}$,
$\subspace{101{\3}{\3}\\01{\2}{\3}1}{}$,
$\subspace{10{\2}1{\2}\\01{\2}{\2}1}{}$,
$\subspace{10{\2}{\3}{\2}\\01{\3}01}{}$,
$\subspace{1{\2}001\\001{\3}{\3}}{7}$,
$\subspace{10{\2}{\2}0\\010{\3}{\3}}{}$,
$\subspace{10{\3}0{\3}\\011{\3}{\3}}{}$,
$\subspace{100{\3}{\3}\\01{\2}10}{}$,
$\subspace{101{\3}{\3}\\01{\2}{\3}1}{}$,
$\subspace{10{\2}1{\2}\\01{\2}{\2}1}{}$,
$\subspace{10{\2}{\3}{\2}\\01{\3}01}{}$,
$\subspace{1{\2}00{\3}\\001{\3}{\3}}{7}$,
$\subspace{10{\2}{\2}{\3}\\010{\3}{\3}}{}$,
$\subspace{10{\3}0{\2}\\011{\3}{\3}}{}$,
$\subspace{100{\3}{\3}\\01{\2}1{\2}}{}$,
$\subspace{101{\3}{\3}\\01{\2}{\3}{\2}}{}$,
$\subspace{10{\2}1{\3}\\01{\2}{\2}0}{}$,
$\subspace{10{\2}{\3}0\\01{\3}0{\3}}{}$,
$\subspace{1{\2}0{\2}0\\001{\3}0}{7}$,
$\subspace{10{\2}10\\010{\3}0}{}$,
$\subspace{10{\3}10\\011{\3}0}{}$,
$\subspace{100{\3}0\\01{\2}{\3}0}{}$,
$\subspace{101{\3}0\\01{\2}00}{}$,
$\subspace{10{\2}00\\01{\2}{\3}0}{}$,
$\subspace{10{\2}10\\01{\3}{\2}0}{}$,
$\subspace{1{\2}0{\3}{\2}\\001{\3}1}{7}$,
$\subspace{10{\2}{\3}0\\010{\3}1}{}$,
$\subspace{10{\3}{\2}1\\011{\3}1}{}$,
$\subspace{100{\3}1\\01{\2}{\2}0}{}$,
$\subspace{101{\3}1\\01{\2}{\2}{\2}}{}$,
$\subspace{10{\2}{\3}{\3}\\01{\2}0{\2}}{}$,
$\subspace{10{\2}0{\3}\\01{\3}{\3}{\2}}{}$,
$\subspace{1{\3}01{\3}\\0010{\2}}{7}$,
$\subspace{10{\3}{\3}1\\0100{\2}}{}$,
$\subspace{10{\2}{\2}1\\0110{\2}}{}$,
$\subspace{10{\3}11\\01{\2}1{\3}}{}$,
$\subspace{1000{\2}\\01{\3}1{\2}}{}$,
$\subspace{1010{\2}\\01{\3}{\3}0}{}$,
$\subspace{10{\3}{\2}0\\01{\3}{\2}{\2}}{}$,
$\subspace{1{\3}00{\3}\\001{\2}1}{7}$,
$\subspace{10{\3}{\3}0\\010{\2}1}{}$,
$\subspace{10{\2}01\\011{\2}1}{}$,
$\subspace{10{\3}{\2}{\2}\\01{\2}0{\3}}{}$,
$\subspace{100{\2}1\\01{\3}10}{}$,
$\subspace{101{\2}1\\01{\3}{\2}{\3}}{}$,
$\subspace{10{\3}1{\2}\\01{\3}{\3}{\3}}{}$,
$\subspace{1{\3}011\\001{\2}1}{7}$,
$\subspace{10{\3}01\\010{\2}1}{}$,
$\subspace{10{\2}{\2}{\2}\\011{\2}1}{}$,
$\subspace{10{\3}{\3}0\\01{\2}11}{}$,
$\subspace{100{\2}1\\01{\3}0{\2}}{}$,
$\subspace{101{\2}1\\01{\3}1{\2}}{}$,
$\subspace{10{\3}{\3}1\\01{\3}10}{}$.

$n_4(5,2;14)\ge 222$, $\left[\!\begin{smallmatrix}0{\2}000\\{\2}0000\\000{\3}0\\00{\3}00\\00001\end{smallmatrix}\!\right]$, $\left[\!\begin{smallmatrix}00100\\00010\\10000\\01000\\00001\end{smallmatrix}\!\right]$, $\omega\leftrightarrow\upsilon$:
$\subspace{00010\\00001}{4}$,
$\subspace{00100\\00001}{}$,
$\subspace{01000\\00001}{}$,
$\subspace{10000\\00001}{}$,
$\subspace{01000\\00010}{2}$,
$\subspace{10000\\00100}{}$,
$\subspace{01001\\00011}{6}$,
$\subspace{0100{\3}\\0001{\2}}{}$,
$\subspace{0100{\2}\\0001{\3}}{}$,
$\subspace{10001\\00101}{}$,
$\subspace{1000{\3}\\0010{\2}}{}$,
$\subspace{1000{\2}\\0010{\3}}{}$,
$\subspace{01001\\00111}{12}$,
$\subspace{0100{\3}\\0011{\2}}{}$,
$\subspace{0100{\2}\\0011{\3}}{}$,
$\subspace{10001\\00111}{}$,
$\subspace{1000{\3}\\0011{\2}}{}$,
$\subspace{1000{\2}\\0011{\3}}{}$,
$\subspace{11001\\00011}{}$,
$\subspace{1100{\3}\\0001{\2}}{}$,
$\subspace{1100{\2}\\0001{\3}}{}$,
$\subspace{11001\\00101}{}$,
$\subspace{1100{\3}\\0010{\2}}{}$,
$\subspace{1100{\2}\\0010{\3}}{}$,
$\subspace{11001\\00110}{6}$,
$\subspace{1100{\2}\\00110}{}$,
$\subspace{1100{\3}\\00110}{}$,
$\subspace{11000\\00111}{}$,
$\subspace{11000\\0011{\2}}{}$,
$\subspace{11000\\0011{\3}}{}$,
$\subspace{1{\2}000\\001{\2}0}{2}$,
$\subspace{1{\3}000\\001{\3}0}{}$,
$\subspace{1{\2}000\\001{\2}0}{2}$,
$\subspace{1{\3}000\\001{\3}0}{}$,
$\subspace{1{\2}000\\001{\2}0}{2}$,
$\subspace{1{\3}000\\001{\3}0}{}$,
$\subspace{1{\2}00{\2}\\001{\2}1}{6}$,
$\subspace{1{\2}001\\001{\2}{\2}}{}$,
$\subspace{1{\2}00{\3}\\001{\2}{\3}}{}$,
$\subspace{1{\3}00{\3}\\001{\3}1}{}$,
$\subspace{1{\3}00{\2}\\001{\3}{\2}}{}$,
$\subspace{1{\3}001\\001{\3}{\3}}{}$,
$\subspace{1{\2}010\\001{\2}{\2}}{24}$,
$\subspace{1{\2}011\\001{\2}{\2}}{}$,
$\subspace{1{\2}0{\2}0\\001{\2}{\3}}{}$,
$\subspace{1{\2}0{\2}{\3}\\001{\2}{\3}}{}$,
$\subspace{1{\2}0{\3}0\\001{\2}1}{}$,
$\subspace{1{\2}0{\3}{\2}\\001{\2}1}{}$,
$\subspace{1{\3}010\\001{\3}{\3}}{}$,
$\subspace{1{\3}011\\001{\3}{\3}}{}$,
$\subspace{1{\3}0{\2}0\\001{\3}1}{}$,
$\subspace{1{\3}0{\2}{\3}\\001{\3}1}{}$,
$\subspace{1{\3}0{\3}0\\001{\3}{\2}}{}$,
$\subspace{1{\3}0{\3}{\2}\\001{\3}{\2}}{}$,
$\subspace{10{\2}{\3}{\2}\\011{\2}0}{}$,
$\subspace{10{\2}{\3}0\\011{\2}1}{}$,
$\subspace{10{\3}{\2}{\3}\\011{\3}0}{}$,
$\subspace{10{\3}{\2}0\\011{\3}1}{}$,
$\subspace{101{\3}{\2}\\01{\2}10}{}$,
$\subspace{101{\3}0\\01{\2}1{\3}}{}$,
$\subspace{10{\3}11\\01{\2}{\3}0}{}$,
$\subspace{10{\3}10\\01{\2}{\3}{\3}}{}$,
$\subspace{101{\2}{\3}\\01{\3}10}{}$,
$\subspace{101{\2}0\\01{\3}1{\2}}{}$,
$\subspace{10{\2}11\\01{\3}{\2}0}{}$,
$\subspace{10{\2}10\\01{\3}{\2}{\2}}{}$,
$\subspace{1010{\2}\\01010}{12}$,
$\subspace{1010{\3}\\01010}{}$,
$\subspace{10100\\0101{\2}}{}$,
$\subspace{10100\\0101{\3}}{}$,
$\subspace{10{\2}01\\010{\2}0}{}$,
$\subspace{10{\2}0{\2}\\010{\2}0}{}$,
$\subspace{10{\2}00\\010{\2}1}{}$,
$\subspace{10{\2}00\\010{\2}{\2}}{}$,
$\subspace{10{\3}01\\010{\3}0}{}$,
$\subspace{10{\3}0{\3}\\010{\3}0}{}$,
$\subspace{10{\3}00\\010{\3}1}{}$,
$\subspace{10{\3}00\\010{\3}{\3}}{}$,
$\subspace{1010{\3}\\010{\2}1}{12}$,
$\subspace{1010{\2}\\010{\2}{\2}}{}$,
$\subspace{1010{\2}\\010{\3}1}{}$,
$\subspace{1010{\3}\\010{\3}{\3}}{}$,
$\subspace{10{\2}0{\2}\\0101{\2}}{}$,
$\subspace{10{\2}01\\0101{\3}}{}$,
$\subspace{10{\2}01\\010{\3}1}{}$,
$\subspace{10{\2}0{\2}\\010{\3}{\3}}{}$,
$\subspace{10{\3}01\\0101{\2}}{}$,
$\subspace{10{\3}0{\3}\\0101{\3}}{}$,
$\subspace{10{\3}01\\010{\2}1}{}$,
$\subspace{10{\3}0{\3}\\010{\2}{\2}}{}$,
$\subspace{10010\\01100}{3}$,
$\subspace{100{\2}0\\01{\2}00}{}$,
$\subspace{100{\3}0\\01{\3}00}{}$,
$\subspace{10010\\01100}{3}$,
$\subspace{100{\2}0\\01{\2}00}{}$,
$\subspace{100{\3}0\\01{\3}00}{}$,
$\subspace{1001{\3}\\0110{\2}}{6}$,
$\subspace{1001{\2}\\0110{\3}}{}$,
$\subspace{100{\2}{\2}\\01{\2}01}{}$,
$\subspace{100{\2}1\\01{\2}0{\2}}{}$,
$\subspace{100{\3}{\3}\\01{\3}01}{}$,
$\subspace{100{\3}1\\01{\3}0{\3}}{}$,
$\subspace{10011\\011{\2}0}{24}$,
$\subspace{10011\\011{\2}{\2}}{}$,
$\subspace{10011\\011{\3}0}{}$,
$\subspace{10011\\011{\3}{\3}}{}$,
$\subspace{10{\2}10\\01101}{}$,
$\subspace{10{\2}1{\2}\\01101}{}$,
$\subspace{10{\3}10\\01101}{}$,
$\subspace{10{\3}1{\3}\\01101}{}$,
$\subspace{100{\2}{\3}\\01{\2}10}{}$,
$\subspace{100{\2}{\3}\\01{\2}1{\2}}{}$,
$\subspace{100{\2}{\3}\\01{\2}{\3}0}{}$,
$\subspace{100{\2}{\3}\\01{\2}{\3}1}{}$,
$\subspace{101{\2}0\\01{\2}0{\3}}{}$,
$\subspace{101{\2}{\2}\\01{\2}0{\3}}{}$,
$\subspace{10{\3}{\2}0\\01{\2}0{\3}}{}$,
$\subspace{10{\3}{\2}1\\01{\2}0{\3}}{}$,
$\subspace{100{\3}{\2}\\01{\3}10}{}$,
$\subspace{100{\3}{\2}\\01{\3}1{\3}}{}$,
$\subspace{100{\3}{\2}\\01{\3}{\2}0}{}$,
$\subspace{100{\3}{\2}\\01{\3}{\2}1}{}$,
$\subspace{101{\3}0\\01{\3}0{\2}}{}$,
$\subspace{101{\3}{\3}\\01{\3}0{\2}}{}$,
$\subspace{10{\2}{\3}0\\01{\3}0{\2}}{}$,
$\subspace{10{\2}{\3}1\\01{\3}0{\2}}{}$,
$\subspace{1001{\3}\\011{\2}{\2}}{24}$,
$\subspace{1001{\3}\\011{\2}{\3}}{}$,
$\subspace{1001{\2}\\011{\3}{\2}}{}$,
$\subspace{1001{\2}\\011{\3}{\3}}{}$,
$\subspace{10{\2}1{\2}\\0110{\3}}{}$,
$\subspace{10{\2}1{\3}\\0110{\3}}{}$,
$\subspace{10{\3}1{\2}\\0110{\2}}{}$,
$\subspace{10{\3}1{\3}\\0110{\2}}{}$,
$\subspace{100{\2}1\\01{\2}11}{}$,
$\subspace{100{\2}1\\01{\2}1{\2}}{}$,
$\subspace{100{\2}{\2}\\01{\2}{\3}1}{}$,
$\subspace{100{\2}{\2}\\01{\2}{\3}{\2}}{}$,
$\subspace{101{\2}1\\01{\2}01}{}$,
$\subspace{101{\2}{\2}\\01{\2}01}{}$,
$\subspace{10{\3}{\2}1\\01{\2}0{\2}}{}$,
$\subspace{10{\3}{\2}{\2}\\01{\2}0{\2}}{}$,
$\subspace{100{\3}1\\01{\3}11}{}$,
$\subspace{100{\3}1\\01{\3}1{\3}}{}$,
$\subspace{100{\3}{\3}\\01{\3}{\2}1}{}$,
$\subspace{100{\3}{\3}\\01{\3}{\2}{\3}}{}$,
$\subspace{101{\3}1\\01{\3}01}{}$,
$\subspace{101{\3}{\3}\\01{\3}01}{}$,
$\subspace{10{\2}{\3}1\\01{\3}0{\3}}{}$,
$\subspace{10{\2}{\3}{\3}\\01{\3}0{\3}}{}$,
$\subspace{101{\2}0\\01111}{24}$,
$\subspace{101{\2}1\\0111{\3}}{}$,
$\subspace{101{\3}0\\01111}{}$,
$\subspace{101{\3}1\\0111{\2}}{}$,
$\subspace{10{\2}{\2}{\3}\\011{\2}0}{}$,
$\subspace{10{\2}{\2}1\\011{\2}{\3}}{}$,
$\subspace{10{\3}{\3}{\2}\\011{\3}0}{}$,
$\subspace{10{\3}{\3}1\\011{\3}{\2}}{}$,
$\subspace{10111\\01{\2}10}{}$,
$\subspace{1011{\3}\\01{\2}11}{}$,
$\subspace{10{\2}1{\3}\\01{\2}{\2}1}{}$,
$\subspace{10{\2}10\\01{\2}{\2}{\3}}{}$,
$\subspace{10{\2}{\3}{\3}\\01{\2}{\2}{\2}}{}$,
$\subspace{10{\2}{\3}0\\01{\2}{\2}{\3}}{}$,
$\subspace{10{\3}{\3}{\2}\\01{\2}{\3}0}{}$,
$\subspace{10{\3}{\3}{\3}\\01{\2}{\3}{\2}}{}$,
$\subspace{10111\\01{\3}10}{}$,
$\subspace{1011{\2}\\01{\3}11}{}$,
$\subspace{10{\2}{\2}{\3}\\01{\3}{\2}0}{}$,
$\subspace{10{\2}{\2}{\2}\\01{\3}{\2}{\3}}{}$,
$\subspace{10{\3}1{\2}\\01{\3}{\3}1}{}$,
$\subspace{10{\3}10\\01{\3}{\3}{\2}}{}$,
$\subspace{10{\3}{\2}0\\01{\3}{\3}{\2}}{}$,
$\subspace{10{\3}{\2}{\2}\\01{\3}{\3}{\3}}{}$,
$\subspace{1011{\2}\\011{\2}{\2}}{24}$,
$\subspace{10110\\011{\2}{\3}}{}$,
$\subspace{10110\\011{\3}{\2}}{}$,
$\subspace{1011{\3}\\011{\3}{\3}}{}$,
$\subspace{10{\2}1{\3}\\01110}{}$,
$\subspace{10{\2}1{\2}\\0111{\2}}{}$,
$\subspace{10{\3}1{\2}\\01110}{}$,
$\subspace{10{\3}1{\3}\\0111{\3}}{}$,
$\subspace{101{\2}1\\01{\2}{\2}0}{}$,
$\subspace{101{\2}{\2}\\01{\2}{\2}{\2}}{}$,
$\subspace{10{\2}{\2}0\\01{\2}11}{}$,
$\subspace{10{\2}{\2}{\2}\\01{\2}1{\2}}{}$,
$\subspace{10{\2}{\2}1\\01{\2}{\3}1}{}$,
$\subspace{10{\2}{\2}0\\01{\2}{\3}{\2}}{}$,
$\subspace{10{\3}{\2}{\2}\\01{\2}{\2}0}{}$,
$\subspace{10{\3}{\2}1\\01{\2}{\2}1}{}$,
$\subspace{101{\3}1\\01{\3}{\3}0}{}$,
$\subspace{101{\3}{\3}\\01{\3}{\3}{\3}}{}$,
$\subspace{10{\2}{\3}{\3}\\01{\3}{\3}0}{}$,
$\subspace{10{\2}{\3}1\\01{\3}{\3}1}{}$,
$\subspace{10{\3}{\3}0\\01{\3}11}{}$,
$\subspace{10{\3}{\3}{\3}\\01{\3}1{\3}}{}$,
$\subspace{10{\3}{\3}1\\01{\3}{\2}1}{}$,
$\subspace{10{\3}{\3}0\\01{\3}{\2}{\3}}{}$,
$\subspace{101{\3}{\2}\\011{\2}1}{24}$,
$\subspace{101{\2}{\3}\\011{\3}1}{}$,
$\subspace{10{\2}{\3}0\\0111{\3}}{}$,
$\subspace{10{\2}{\2}{\2}\\011{\3}0}{}$,
$\subspace{10{\2}{\3}{\2}\\011{\3}1}{}$,
$\subspace{10{\3}{\2}0\\0111{\2}}{}$,
$\subspace{10{\3}{\2}{\3}\\011{\2}1}{}$,
$\subspace{10{\3}{\3}{\3}\\011{\2}0}{}$,
$\subspace{101{\3}0\\01{\2}{\2}1}{}$,
$\subspace{1011{\2}\\01{\2}{\3}0}{}$,
$\subspace{101{\3}{\2}\\01{\2}{\3}{\3}}{}$,
$\subspace{10{\2}{\3}{\2}\\01{\2}1{\3}}{}$,
$\subspace{10{\2}11\\01{\2}{\3}{\3}}{}$,
$\subspace{10{\3}11\\01{\2}1{\3}}{}$,
$\subspace{10{\3}{\3}1\\01{\2}10}{}$,
$\subspace{10{\3}10\\01{\2}{\2}{\2}}{}$,
$\subspace{1011{\3}\\01{\3}{\2}0}{}$,
$\subspace{101{\2}{\3}\\01{\3}{\2}{\2}}{}$,
$\subspace{101{\2}0\\01{\3}{\3}1}{}$,
$\subspace{10{\2}11\\01{\3}1{\2}}{}$,
$\subspace{10{\2}{\2}1\\01{\3}10}{}$,
$\subspace{10{\2}10\\01{\3}{\3}{\3}}{}$,
$\subspace{10{\3}{\2}{\3}\\01{\3}1{\2}}{}$,
$\subspace{10{\3}11\\01{\3}{\2}{\2}}{}$.

$n_4(5,2;18)\ge 290$, $\left[\!\begin{smallmatrix}0{\2}000\\{\2}0000\\000{\3}0\\00{\3}00\\00001\end{smallmatrix}\!\right]$, $\left[\!\begin{smallmatrix}00100\\00010\\10000\\01000\\00001\end{smallmatrix}\!\right]$:
$\subspace{001{\2}0\\00001}{4}$,
$\subspace{001{\3}0\\00001}{}$,
$\subspace{1{\2}000\\00001}{}$,
$\subspace{1{\3}000\\00001}{}$,
$\subspace{00101\\00010}{12}$,
$\subspace{0010{\2}\\00010}{}$,
$\subspace{0010{\3}\\00010}{}$,
$\subspace{00100\\00011}{}$,
$\subspace{00100\\0001{\2}}{}$,
$\subspace{00100\\0001{\3}}{}$,
$\subspace{10001\\01000}{}$,
$\subspace{1000{\2}\\01000}{}$,
$\subspace{1000{\3}\\01000}{}$,
$\subspace{10000\\01001}{}$,
$\subspace{10000\\0100{\2}}{}$,
$\subspace{10000\\0100{\3}}{}$,
$\subspace{0100{\3}\\00101}{6}$,
$\subspace{0100{\2}\\0010{\2}}{}$,
$\subspace{01001\\0010{\3}}{}$,
$\subspace{1000{\3}\\00011}{}$,
$\subspace{1000{\2}\\0001{\2}}{}$,
$\subspace{10001\\0001{\3}}{}$,
$\subspace{0101{\2}\\0010{\2}}{12}$,
$\subspace{010{\2}1\\0010{\3}}{}$,
$\subspace{010{\3}{\3}\\00101}{}$,
$\subspace{1010{\2}\\0001{\2}}{}$,
$\subspace{10{\2}01\\0001{\3}}{}$,
$\subspace{10{\3}0{\3}\\00011}{}$,
$\subspace{1000{\2}\\0101{\2}}{}$,
$\subspace{10001\\010{\2}1}{}$,
$\subspace{1000{\3}\\010{\3}{\3}}{}$,
$\subspace{1010{\2}\\0100{\2}}{}$,
$\subspace{10{\2}01\\01001}{}$,
$\subspace{10{\3}0{\3}\\0100{\3}}{}$,
$\subspace{0101{\2}\\0010{\2}}{12}$,
$\subspace{010{\2}1\\0010{\3}}{}$,
$\subspace{010{\3}{\3}\\00101}{}$,
$\subspace{1010{\2}\\0001{\2}}{}$,
$\subspace{10{\2}01\\0001{\3}}{}$,
$\subspace{10{\3}0{\3}\\00011}{}$,
$\subspace{1000{\2}\\0101{\2}}{}$,
$\subspace{10001\\010{\2}1}{}$,
$\subspace{1000{\3}\\010{\3}{\3}}{}$,
$\subspace{1010{\2}\\0100{\2}}{}$,
$\subspace{10{\2}01\\01001}{}$,
$\subspace{10{\3}0{\3}\\0100{\3}}{}$,
$\subspace{01000\\00110}{4}$,
$\subspace{10000\\00110}{}$,
$\subspace{11000\\00010}{}$,
$\subspace{11000\\00100}{}$,
$\subspace{1{\2}011\\0011{\3}}{12}$,
$\subspace{1{\2}0{\2}{\3}\\00111}{}$,
$\subspace{1{\2}0{\3}{\2}\\0011{\2}}{}$,
$\subspace{1{\3}01{\3}\\00111}{}$,
$\subspace{1{\3}0{\2}{\2}\\0011{\2}}{}$,
$\subspace{1{\3}0{\3}1\\0011{\3}}{}$,
$\subspace{101{\2}{\2}\\011{\2}1}{}$,
$\subspace{101{\3}{\2}\\011{\3}{\3}}{}$,
$\subspace{10{\2}11\\01{\2}1{\2}}{}$,
$\subspace{10{\2}{\3}1\\01{\2}{\3}{\3}}{}$,
$\subspace{10{\3}1{\3}\\01{\3}1{\2}}{}$,
$\subspace{10{\3}{\2}{\3}\\01{\3}{\2}1}{}$,
$\subspace{1{\2}01{\2}\\001{\2}{\3}}{12}$,
$\subspace{1{\2}0{\2}1\\001{\2}1}{}$,
$\subspace{1{\2}0{\3}{\3}\\001{\2}{\2}}{}$,
$\subspace{1{\3}010\\001{\3}1}{}$,
$\subspace{1{\3}0{\2}0\\001{\3}{\2}}{}$,
$\subspace{1{\3}0{\3}0\\001{\3}{\3}}{}$,
$\subspace{10{\2}{\3}0\\011{\2}{\2}}{}$,
$\subspace{10{\3}{\2}1\\011{\3}0}{}$,
$\subspace{101{\3}{\3}\\01{\2}10}{}$,
$\subspace{10{\3}10\\01{\2}{\3}1}{}$,
$\subspace{101{\2}0\\01{\3}1{\3}}{}$,
$\subspace{10{\2}1{\2}\\01{\3}{\2}0}{}$,
$\subspace{1{\2}000\\001{\3}0}{2}$,
$\subspace{1{\3}000\\001{\2}0}{}$,
$\subspace{1{\2}000\\001{\3}0}{2}$,
$\subspace{1{\3}000\\001{\2}0}{}$,
$\subspace{1{\2}000\\001{\3}0}{2}$,
$\subspace{1{\3}000\\001{\2}0}{}$,
$\subspace{1010{\3}\\01011}{6}$,
$\subspace{10101\\0101{\3}}{}$,
$\subspace{10{\2}0{\3}\\010{\2}{\2}}{}$,
$\subspace{10{\2}0{\2}\\010{\2}{\3}}{}$,
$\subspace{10{\3}0{\2}\\010{\3}1}{}$,
$\subspace{10{\3}01\\010{\3}{\2}}{}$,
$\subspace{1010{\3}\\0101{\2}}{6}$,
$\subspace{1010{\2}\\0101{\3}}{}$,
$\subspace{10{\2}0{\2}\\010{\2}1}{}$,
$\subspace{10{\2}01\\010{\2}{\2}}{}$,
$\subspace{10{\3}0{\3}\\010{\3}1}{}$,
$\subspace{10{\3}01\\010{\3}{\3}}{}$,
$\subspace{101{\2}{\2}\\0101{\3}}{12}$,
$\subspace{101{\2}{\3}\\0101{\3}}{}$,
$\subspace{10{\2}{\3}1\\010{\2}{\2}}{}$,
$\subspace{10{\2}{\3}{\2}\\010{\2}{\2}}{}$,
$\subspace{10{\3}11\\010{\3}1}{}$,
$\subspace{10{\3}1{\3}\\010{\3}1}{}$,
$\subspace{10{\3}01\\011{\3}1}{}$,
$\subspace{10{\3}01\\011{\3}{\3}}{}$,
$\subspace{1010{\3}\\01{\2}1{\2}}{}$,
$\subspace{1010{\3}\\01{\2}1{\3}}{}$,
$\subspace{10{\2}0{\2}\\01{\3}{\2}1}{}$,
$\subspace{10{\2}0{\2}\\01{\3}{\2}{\2}}{}$,
$\subspace{101{\3}1\\01011}{12}$,
$\subspace{101{\3}{\2}\\01011}{}$,
$\subspace{10{\2}11\\010{\2}{\3}}{}$,
$\subspace{10{\2}1{\3}\\010{\2}{\3}}{}$,
$\subspace{10{\3}{\2}{\2}\\010{\3}{\2}}{}$,
$\subspace{10{\3}{\2}{\3}\\010{\3}{\2}}{}$,
$\subspace{10{\2}0{\3}\\011{\2}1}{}$,
$\subspace{10{\2}0{\3}\\011{\2}{\3}}{}$,
$\subspace{10{\3}0{\2}\\01{\2}{\3}{\2}}{}$,
$\subspace{10{\3}0{\2}\\01{\2}{\3}{\3}}{}$,
$\subspace{10101\\01{\3}11}{}$,
$\subspace{10101\\01{\3}1{\2}}{}$,
$\subspace{1011{\3}\\010{\2}0}{12}$,
$\subspace{101{\3}{\3}\\010{\3}0}{}$,
$\subspace{10{\2}1{\2}\\01010}{}$,
$\subspace{10{\2}{\2}{\2}\\010{\3}0}{}$,
$\subspace{10{\3}{\3}1\\01010}{}$,
$\subspace{10{\3}{\2}1\\010{\2}0}{}$,
$\subspace{10100\\011{\2}{\2}}{}$,
$\subspace{10{\2}00\\0111{\3}}{}$,
$\subspace{10{\2}00\\01{\2}{\3}1}{}$,
$\subspace{10{\3}00\\01{\2}{\2}{\2}}{}$,
$\subspace{10100\\01{\3}{\3}1}{}$,
$\subspace{10{\3}00\\01{\3}1{\3}}{}$,
$\subspace{1011{\3}\\010{\2}{\3}}{12}$,
$\subspace{101{\3}1\\010{\3}{\2}}{}$,
$\subspace{10{\2}1{\3}\\01011}{}$,
$\subspace{10{\2}{\2}{\2}\\010{\3}{\2}}{}$,
$\subspace{10{\3}{\3}1\\01011}{}$,
$\subspace{10{\3}{\2}{\2}\\010{\2}{\3}}{}$,
$\subspace{10101\\011{\2}{\3}}{}$,
$\subspace{10{\2}0{\3}\\0111{\3}}{}$,
$\subspace{10{\2}0{\3}\\01{\2}{\3}{\2}}{}$,
$\subspace{10{\3}0{\2}\\01{\2}{\2}{\2}}{}$,
$\subspace{10101\\01{\3}{\3}1}{}$,
$\subspace{10{\3}0{\2}\\01{\3}11}{}$,
$\subspace{101{\2}1\\010{\2}0}{12}$,
$\subspace{10111\\010{\3}0}{}$,
$\subspace{10{\2}{\2}{\3}\\01010}{}$,
$\subspace{10{\2}{\3}{\3}\\010{\3}0}{}$,
$\subspace{10{\3}1{\2}\\01010}{}$,
$\subspace{10{\3}{\3}{\2}\\010{\2}0}{}$,
$\subspace{10100\\011{\3}{\2}}{}$,
$\subspace{10{\3}00\\01111}{}$,
$\subspace{10100\\01{\2}{\2}{\3}}{}$,
$\subspace{10{\2}00\\01{\2}11}{}$,
$\subspace{10{\2}00\\01{\3}{\3}{\2}}{}$,
$\subspace{10{\3}00\\01{\3}{\2}{\3}}{}$,
$\subspace{10011\\01100}{6}$,
$\subspace{10010\\01101}{}$,
$\subspace{100{\2}{\3}\\01{\2}00}{}$,
$\subspace{100{\2}0\\01{\2}0{\3}}{}$,
$\subspace{100{\3}{\2}\\01{\3}00}{}$,
$\subspace{100{\3}0\\01{\3}0{\2}}{}$,
$\subspace{1001{\2}\\01100}{6}$,
$\subspace{10010\\0110{\2}}{}$,
$\subspace{100{\2}1\\01{\2}00}{}$,
$\subspace{100{\2}0\\01{\2}01}{}$,
$\subspace{100{\3}{\3}\\01{\3}00}{}$,
$\subspace{100{\3}0\\01{\3}0{\3}}{}$,
$\subspace{1001{\3}\\0110{\3}}{3}$,
$\subspace{100{\2}{\2}\\01{\2}0{\2}}{}$,
$\subspace{100{\3}1\\01{\3}01}{}$,
$\subspace{1001{\3}\\0110{\3}}{3}$,
$\subspace{100{\2}{\2}\\01{\2}0{\2}}{}$,
$\subspace{100{\3}1\\01{\3}01}{}$,
$\subspace{100{\3}1\\01111}{12}$,
$\subspace{100{\2}{\2}\\011{\2}{\3}}{}$,
$\subspace{101{\3}1\\0110{\3}}{}$,
$\subspace{10{\2}{\2}{\3}\\0110{\3}}{}$,
$\subspace{1001{\3}\\01{\2}{\2}{\3}}{}$,
$\subspace{100{\3}1\\01{\2}{\3}{\2}}{}$,
$\subspace{10{\2}1{\3}\\01{\2}0{\2}}{}$,
$\subspace{10{\3}{\3}{\2}\\01{\2}0{\2}}{}$,
$\subspace{1001{\3}\\01{\3}11}{}$,
$\subspace{100{\2}{\2}\\01{\3}{\3}{\2}}{}$,
$\subspace{10111\\01{\3}01}{}$,
$\subspace{10{\3}{\2}{\2}\\01{\3}01}{}$,
$\subspace{10010\\011{\2}1}{6}$,
$\subspace{10{\2}11\\01100}{}$,
$\subspace{100{\2}0\\01{\2}{\3}{\3}}{}$,
$\subspace{10{\3}{\2}{\3}\\01{\2}00}{}$,
$\subspace{100{\3}0\\01{\3}1{\2}}{}$,
$\subspace{101{\3}{\2}\\01{\3}00}{}$,
$\subspace{100{\3}{\3}\\011{\2}0}{12}$,
$\subspace{100{\2}1\\011{\3}{\2}}{}$,
$\subspace{10{\2}{\3}{\3}\\0110{\2}}{}$,
$\subspace{10{\3}{\2}0\\0110{\2}}{}$,
$\subspace{100{\3}{\3}\\01{\2}11}{}$,
$\subspace{1001{\2}\\01{\2}{\3}0}{}$,
$\subspace{101{\3}0\\01{\2}01}{}$,
$\subspace{10{\3}1{\2}\\01{\2}01}{}$,
$\subspace{100{\2}1\\01{\3}10}{}$,
$\subspace{1001{\2}\\01{\3}{\2}{\3}}{}$,
$\subspace{101{\2}1\\01{\3}0{\3}}{}$,
$\subspace{10{\2}10\\01{\3}0{\3}}{}$,
$\subspace{10011\\011{\3}0}{12}$,
$\subspace{10011\\011{\3}{\3}}{}$,
$\subspace{10{\3}10\\01101}{}$,
$\subspace{10{\3}1{\3}\\01101}{}$,
$\subspace{100{\2}{\3}\\01{\2}10}{}$,
$\subspace{100{\2}{\3}\\01{\2}1{\2}}{}$,
$\subspace{101{\2}0\\01{\2}0{\3}}{}$,
$\subspace{101{\2}{\2}\\01{\2}0{\3}}{}$,
$\subspace{100{\3}{\2}\\01{\3}{\2}0}{}$,
$\subspace{100{\3}{\2}\\01{\3}{\2}1}{}$,
$\subspace{10{\2}{\3}0\\01{\3}0{\2}}{}$,
$\subspace{10{\2}{\3}1\\01{\3}0{\2}}{}$,
$\subspace{10011\\011{\3}0}{12}$,
$\subspace{10011\\011{\3}{\3}}{}$,
$\subspace{10{\3}10\\01101}{}$,
$\subspace{10{\3}1{\3}\\01101}{}$,
$\subspace{100{\2}{\3}\\01{\2}10}{}$,
$\subspace{100{\2}{\3}\\01{\2}1{\2}}{}$,
$\subspace{101{\2}0\\01{\2}0{\3}}{}$,
$\subspace{101{\2}{\2}\\01{\2}0{\3}}{}$,
$\subspace{100{\3}{\2}\\01{\3}{\2}0}{}$,
$\subspace{100{\3}{\2}\\01{\3}{\2}1}{}$,
$\subspace{10{\2}{\3}0\\01{\3}0{\2}}{}$,
$\subspace{10{\2}{\3}1\\01{\3}0{\2}}{}$,
$\subspace{101{\2}{\3}\\0111{\2}}{12}$,
$\subspace{101{\2}0\\0111{\3}}{}$,
$\subspace{10{\3}{\3}1\\011{\3}0}{}$,
$\subspace{10{\3}{\3}{\3}\\011{\3}1}{}$,
$\subspace{1011{\3}\\01{\2}10}{}$,
$\subspace{1011{\2}\\01{\2}1{\3}}{}$,
$\subspace{10{\2}{\3}{\2}\\01{\2}{\2}1}{}$,
$\subspace{10{\2}{\3}0\\01{\2}{\2}{\2}}{}$,
$\subspace{10{\2}{\2}{\2}\\01{\3}{\2}0}{}$,
$\subspace{10{\2}{\2}1\\01{\3}{\2}{\2}}{}$,
$\subspace{10{\3}10\\01{\3}{\3}1}{}$,
$\subspace{10{\3}11\\01{\3}{\3}{\3}}{}$,
$\subspace{101{\3}0\\01111}{12}$,
$\subspace{101{\3}1\\0111{\2}}{}$,
$\subspace{10{\2}{\2}{\3}\\011{\2}0}{}$,
$\subspace{10{\2}{\2}1\\011{\2}{\3}}{}$,
$\subspace{10{\2}1{\3}\\01{\2}{\2}1}{}$,
$\subspace{10{\2}10\\01{\2}{\2}{\3}}{}$,
$\subspace{10{\3}{\3}{\2}\\01{\2}{\3}0}{}$,
$\subspace{10{\3}{\3}{\3}\\01{\2}{\3}{\2}}{}$,
$\subspace{10111\\01{\3}10}{}$,
$\subspace{1011{\2}\\01{\3}11}{}$,
$\subspace{10{\3}{\2}0\\01{\3}{\3}{\2}}{}$,
$\subspace{10{\3}{\2}{\2}\\01{\3}{\3}{\3}}{}$,
$\subspace{10111\\011{\2}1}{12}$,
$\subspace{10110\\011{\2}{\2}}{}$,
$\subspace{10{\2}1{\2}\\01110}{}$,
$\subspace{10{\2}11\\01111}{}$,
$\subspace{10{\2}{\2}0\\01{\2}{\3}1}{}$,
$\subspace{10{\2}{\2}{\3}\\01{\2}{\3}{\3}}{}$,
$\subspace{10{\3}{\2}1\\01{\2}{\2}0}{}$,
$\subspace{10{\3}{\2}{\3}\\01{\2}{\2}{\3}}{}$,
$\subspace{101{\3}{\3}\\01{\3}{\3}0}{}$,
$\subspace{101{\3}{\2}\\01{\3}{\3}{\2}}{}$,
$\subspace{10{\3}{\3}{\2}\\01{\3}1{\2}}{}$,
$\subspace{10{\3}{\3}0\\01{\3}1{\3}}{}$,
$\subspace{101{\3}0\\011{\2}{\3}}{12}$,
$\subspace{10{\2}{\3}{\2}\\0111{\3}}{}$,
$\subspace{10{\2}{\2}{\2}\\011{\3}1}{}$,
$\subspace{10{\3}{\2}{\2}\\011{\2}0}{}$,
$\subspace{101{\3}1\\01{\2}{\3}0}{}$,
$\subspace{10{\2}10\\01{\2}{\3}{\2}}{}$,
$\subspace{10{\3}{\3}1\\01{\2}1{\3}}{}$,
$\subspace{10{\3}11\\01{\2}{\2}{\2}}{}$,
$\subspace{1011{\3}\\01{\3}{\2}{\2}}{}$,
$\subspace{101{\2}{\3}\\01{\3}{\3}1}{}$,
$\subspace{10{\2}1{\3}\\01{\3}10}{}$,
$\subspace{10{\3}{\2}0\\01{\3}11}{}$,
$\subspace{1011{\2}\\011{\3}{\3}}{6}$,
$\subspace{10{\3}1{\3}\\0111{\2}}{}$,
$\subspace{101{\2}{\2}\\01{\2}{\2}1}{}$,
$\subspace{10{\2}{\2}1\\01{\2}1{\2}}{}$,
$\subspace{10{\2}{\3}1\\01{\3}{\3}{\3}}{}$,
$\subspace{10{\3}{\3}{\3}\\01{\3}{\2}1}{}$,
$\subspace{101{\2}{\3}\\011{\3}1}{12}$,
$\subspace{10{\2}{\3}{\2}\\011{\3}1}{}$,
$\subspace{10{\3}{\2}0\\0111{\2}}{}$,
$\subspace{10{\3}{\3}{\3}\\011{\2}0}{}$,
$\subspace{101{\3}0\\01{\2}{\2}1}{}$,
$\subspace{1011{\2}\\01{\2}{\3}0}{}$,
$\subspace{10{\2}{\3}{\2}\\01{\2}1{\3}}{}$,
$\subspace{10{\3}11\\01{\2}1{\3}}{}$,
$\subspace{101{\2}{\3}\\01{\3}{\2}{\2}}{}$,
$\subspace{10{\2}{\2}1\\01{\3}10}{}$,
$\subspace{10{\2}10\\01{\3}{\3}{\3}}{}$,
$\subspace{10{\3}11\\01{\3}{\2}{\2}}{}$.

\def\subspace#1#2{\left(\!\begin{smallmatrix}#1\end{smallmatrix}\!\right)^{\!\raisebox{0mm}{\makebox[0mm][l]{$\scriptscriptstyle {#2}$}}}}
$n_4(5,2;19)\ge 307$, $\left[\!\begin{smallmatrix}0{\2}000\\{\2}0000\\000{\3}0\\00{\3}00\\00001\end{smallmatrix}\!\right]$, $\omega\leftrightarrow\upsilon$:
$\subspace{00100\\00010}{1}$,
$\subspace{00100\\00010}{1}$,
$\subspace{00100\\00010}{1}$,
$\subspace{01000\\00001}{2}$,
$\subspace{10000\\00001}{}$,
$\subspace{01000\\00011}{6}$,
$\subspace{01000\\0001{\2}}{}$,
$\subspace{01000\\0001{\3}}{}$,
$\subspace{10000\\00101}{}$,
$\subspace{10000\\0010{\2}}{}$,
$\subspace{10000\\0010{\3}}{}$,
$\subspace{0100{\2}\\00101}{12}$,
$\subspace{0100{\3}\\00101}{}$,
$\subspace{01001\\0010{\2}}{}$,
$\subspace{0100{\2}\\0010{\2}}{}$,
$\subspace{01001\\0010{\3}}{}$,
$\subspace{0100{\3}\\0010{\3}}{}$,
$\subspace{1000{\2}\\00011}{}$,
$\subspace{1000{\3}\\00011}{}$,
$\subspace{10001\\0001{\2}}{}$,
$\subspace{1000{\2}\\0001{\2}}{}$,
$\subspace{10001\\0001{\3}}{}$,
$\subspace{1000{\3}\\0001{\3}}{}$,
$\subspace{0101{\3}\\001{\2}1}{12}$,
$\subspace{010{\2}{\2}\\001{\2}{\2}}{}$,
$\subspace{010{\3}1\\001{\2}{\3}}{}$,
$\subspace{0101{\2}\\001{\3}1}{}$,
$\subspace{010{\2}1\\001{\3}{\2}}{}$,
$\subspace{010{\3}{\3}\\001{\3}{\3}}{}$,
$\subspace{10010\\001{\2}1}{}$,
$\subspace{100{\2}0\\001{\2}{\2}}{}$,
$\subspace{100{\3}0\\001{\2}{\3}}{}$,
$\subspace{10010\\001{\3}1}{}$,
$\subspace{100{\2}0\\001{\3}{\2}}{}$,
$\subspace{100{\3}0\\001{\3}{\3}}{}$,
$\subspace{11000\\00001}{1}$,
$\subspace{11000\\00111}{3}$,
$\subspace{11000\\0011{\2}}{}$,
$\subspace{11000\\0011{\3}}{}$,
$\subspace{11010\\00111}{6}$,
$\subspace{11011\\00111}{}$,
$\subspace{110{\2}0\\0011{\2}}{}$,
$\subspace{110{\2}{\3}\\0011{\2}}{}$,
$\subspace{110{\3}0\\0011{\3}}{}$,
$\subspace{110{\3}{\2}\\0011{\3}}{}$,
$\subspace{1100{\3}\\001{\2}1}{6}$,
$\subspace{1100{\2}\\001{\2}{\2}}{}$,
$\subspace{11001\\001{\2}{\3}}{}$,
$\subspace{1100{\2}\\001{\3}1}{}$,
$\subspace{11001\\001{\3}{\2}}{}$,
$\subspace{1100{\3}\\001{\3}{\3}}{}$,
$\subspace{10000\\01000}{1}$,
$\subspace{10011\\01001}{6}$,
$\subspace{100{\2}{\3}\\0100{\3}}{}$,
$\subspace{100{\3}{\2}\\0100{\2}}{}$,
$\subspace{10001\\01101}{}$,
$\subspace{1000{\3}\\01{\2}0{\3}}{}$,
$\subspace{1000{\2}\\01{\3}0{\2}}{}$,
$\subspace{10011\\01001}{6}$,
$\subspace{100{\2}{\3}\\0100{\3}}{}$,
$\subspace{100{\3}{\2}\\0100{\2}}{}$,
$\subspace{10001\\01101}{}$,
$\subspace{1000{\3}\\01{\2}0{\3}}{}$,
$\subspace{1000{\2}\\01{\3}0{\2}}{}$,
$\subspace{10100\\01010}{3}$,
$\subspace{10{\2}00\\010{\2}0}{}$,
$\subspace{10{\3}00\\010{\3}0}{}$,
$\subspace{10100\\01010}{3}$,
$\subspace{10{\2}00\\010{\2}0}{}$,
$\subspace{10{\3}00\\010{\3}0}{}$,
$\subspace{10100\\01010}{3}$,
$\subspace{10{\2}00\\010{\2}0}{}$,
$\subspace{10{\3}00\\010{\3}0}{}$,
$\subspace{101{\2}1\\01011}{12}$,
$\subspace{101{\3}1\\01011}{}$,
$\subspace{10{\2}1{\3}\\010{\2}{\3}}{}$,
$\subspace{10{\2}{\3}{\3}\\010{\2}{\3}}{}$,
$\subspace{10{\3}1{\2}\\010{\3}{\2}}{}$,
$\subspace{10{\3}{\2}{\2}\\010{\3}{\2}}{}$,
$\subspace{10{\2}0{\3}\\011{\2}{\3}}{}$,
$\subspace{10{\3}0{\2}\\011{\3}{\2}}{}$,
$\subspace{10101\\01{\2}11}{}$,
$\subspace{10{\3}0{\2}\\01{\2}{\3}{\2}}{}$,
$\subspace{10101\\01{\3}11}{}$,
$\subspace{10{\2}0{\3}\\01{\3}{\2}{\3}}{}$,
$\subspace{101{\2}{\3}\\0101{\2}}{12}$,
$\subspace{101{\3}{\2}\\0101{\3}}{}$,
$\subspace{10{\2}11\\010{\2}{\2}}{}$,
$\subspace{10{\2}{\3}{\2}\\010{\2}1}{}$,
$\subspace{10{\3}11\\010{\3}{\3}}{}$,
$\subspace{10{\3}{\2}{\3}\\010{\3}1}{}$,
$\subspace{10{\2}0{\2}\\011{\2}1}{}$,
$\subspace{10{\3}0{\3}\\011{\3}1}{}$,
$\subspace{1010{\2}\\01{\2}1{\3}}{}$,
$\subspace{10{\3}01\\01{\2}{\3}{\3}}{}$,
$\subspace{1010{\3}\\01{\3}1{\2}}{}$,
$\subspace{10{\2}01\\01{\3}{\2}{\2}}{}$,
$\subspace{1010{\2}\\010{\2}{\2}}{6}$,
$\subspace{1010{\3}\\010{\3}{\3}}{}$,
$\subspace{10{\2}0{\2}\\0101{\2}}{}$,
$\subspace{10{\2}01\\010{\3}1}{}$,
$\subspace{10{\3}0{\3}\\0101{\3}}{}$,
$\subspace{10{\3}01\\010{\2}1}{}$,
$\subspace{1011{\3}\\010{\2}1}{12}$,
$\subspace{1011{\2}\\010{\3}1}{}$,
$\subspace{10{\2}{\2}1\\0101{\3}}{}$,
$\subspace{10{\2}{\2}{\2}\\010{\3}{\3}}{}$,
$\subspace{10{\3}{\3}1\\0101{\2}}{}$,
$\subspace{10{\3}{\3}{\3}\\010{\2}{\2}}{}$,
$\subspace{10{\2}01\\0111{\3}}{}$,
$\subspace{10{\3}01\\0111{\2}}{}$,
$\subspace{1010{\3}\\01{\2}{\2}1}{}$,
$\subspace{10{\3}0{\3}\\01{\2}{\2}{\2}}{}$,
$\subspace{1010{\2}\\01{\3}{\3}1}{}$,
$\subspace{10{\2}0{\2}\\01{\3}{\3}{\3}}{}$,
$\subspace{101{\2}{\3}\\010{\2}0}{12}$,
$\subspace{101{\3}{\2}\\010{\3}0}{}$,
$\subspace{10{\2}11\\01010}{}$,
$\subspace{10{\2}{\3}{\2}\\010{\3}0}{}$,
$\subspace{10{\3}11\\01010}{}$,
$\subspace{10{\3}{\2}{\3}\\010{\2}0}{}$,
$\subspace{10100\\011{\2}1}{}$,
$\subspace{10100\\011{\3}1}{}$,
$\subspace{10{\2}00\\01{\2}1{\3}}{}$,
$\subspace{10{\2}00\\01{\2}{\3}{\3}}{}$,
$\subspace{10{\3}00\\01{\3}1{\2}}{}$,
$\subspace{10{\3}00\\01{\3}{\2}{\2}}{}$,
$\subspace{101{\3}1\\010{\2}{\3}}{12}$,
$\subspace{101{\2}1\\010{\3}{\2}}{}$,
$\subspace{10{\2}{\3}{\3}\\01011}{}$,
$\subspace{10{\2}1{\3}\\010{\3}{\2}}{}$,
$\subspace{10{\3}{\2}{\2}\\01011}{}$,
$\subspace{10{\3}1{\2}\\010{\2}{\3}}{}$,
$\subspace{10{\2}0{\3}\\011{\3}{\2}}{}$,
$\subspace{10{\3}0{\2}\\011{\2}{\3}}{}$,
$\subspace{10101\\01{\2}{\3}{\2}}{}$,
$\subspace{10{\3}0{\2}\\01{\2}11}{}$,
$\subspace{10101\\01{\3}{\2}{\3}}{}$,
$\subspace{10{\2}0{\3}\\01{\3}11}{}$,
$\subspace{10011\\01100}{6}$,
$\subspace{10010\\01101}{}$,
$\subspace{100{\2}{\3}\\01{\2}00}{}$,
$\subspace{100{\2}0\\01{\2}0{\3}}{}$,
$\subspace{100{\3}{\2}\\01{\3}00}{}$,
$\subspace{100{\3}0\\01{\3}0{\2}}{}$,
$\subspace{1001{\3}\\0110{\2}}{6}$,
$\subspace{1001{\2}\\0110{\3}}{}$,
$\subspace{100{\2}{\2}\\01{\2}01}{}$,
$\subspace{100{\2}1\\01{\2}0{\2}}{}$,
$\subspace{100{\3}{\3}\\01{\3}01}{}$,
$\subspace{100{\3}1\\01{\3}0{\3}}{}$,
$\subspace{100{\2}{\2}\\01101}{12}$,
$\subspace{100{\2}{\3}\\0110{\2}}{}$,
$\subspace{100{\3}{\3}\\01101}{}$,
$\subspace{100{\3}{\2}\\0110{\3}}{}$,
$\subspace{10011\\01{\2}0{\2}}{}$,
$\subspace{1001{\2}\\01{\2}0{\3}}{}$,
$\subspace{100{\3}{\2}\\01{\2}01}{}$,
$\subspace{100{\3}1\\01{\2}0{\3}}{}$,
$\subspace{1001{\3}\\01{\3}0{\2}}{}$,
$\subspace{10011\\01{\3}0{\3}}{}$,
$\subspace{100{\2}{\3}\\01{\3}01}{}$,
$\subspace{100{\2}1\\01{\3}0{\2}}{}$,
$\subspace{100{\2}1\\011{\2}0}{12}$,
$\subspace{100{\3}1\\011{\3}0}{}$,
$\subspace{101{\2}0\\0110{\3}}{}$,
$\subspace{101{\3}0\\0110{\2}}{}$,
$\subspace{1001{\3}\\01{\2}10}{}$,
$\subspace{100{\3}{\3}\\01{\2}{\3}0}{}$,
$\subspace{10{\2}10\\01{\2}01}{}$,
$\subspace{10{\2}{\3}0\\01{\2}0{\2}}{}$,
$\subspace{1001{\2}\\01{\3}10}{}$,
$\subspace{100{\2}{\2}\\01{\3}{\2}0}{}$,
$\subspace{10{\3}10\\01{\3}01}{}$,
$\subspace{10{\3}{\2}0\\01{\3}0{\3}}{}$,
$\subspace{10111\\01110}{6}$,
$\subspace{10110\\01111}{}$,
$\subspace{10{\2}{\2}{\3}\\01{\2}{\2}0}{}$,
$\subspace{10{\2}{\2}0\\01{\2}{\2}{\3}}{}$,
$\subspace{10{\3}{\3}{\2}\\01{\3}{\3}0}{}$,
$\subspace{10{\3}{\3}0\\01{\3}{\3}{\2}}{}$,
$\subspace{101{\2}{\3}\\01110}{12}$,
$\subspace{101{\3}{\2}\\01110}{}$,
$\subspace{10{\2}{\2}0\\011{\2}1}{}$,
$\subspace{10{\3}{\3}0\\011{\3}1}{}$,
$\subspace{10110\\01{\2}1{\3}}{}$,
$\subspace{10{\2}11\\01{\2}{\2}0}{}$,
$\subspace{10{\2}{\3}{\2}\\01{\2}{\2}0}{}$,
$\subspace{10{\3}{\3}0\\01{\2}{\3}{\3}}{}$,
$\subspace{10110\\01{\3}1{\2}}{}$,
$\subspace{10{\2}{\2}0\\01{\3}{\2}{\2}}{}$,
$\subspace{10{\3}11\\01{\3}{\3}0}{}$,
$\subspace{10{\3}{\2}{\3}\\01{\3}{\3}0}{}$,
$\subspace{101{\2}{\3}\\0111{\3}}{12}$,
$\subspace{101{\3}{\2}\\0111{\2}}{}$,
$\subspace{10{\2}{\2}1\\011{\2}1}{}$,
$\subspace{10{\3}{\3}1\\011{\3}1}{}$,
$\subspace{1011{\3}\\01{\2}1{\3}}{}$,
$\subspace{10{\2}11\\01{\2}{\2}1}{}$,
$\subspace{10{\2}{\3}{\2}\\01{\2}{\2}{\2}}{}$,
$\subspace{10{\3}{\3}{\3}\\01{\2}{\3}{\3}}{}$,
$\subspace{1011{\2}\\01{\3}1{\2}}{}$,
$\subspace{10{\2}{\2}{\2}\\01{\3}{\2}{\2}}{}$,
$\subspace{10{\3}11\\01{\3}{\3}1}{}$,
$\subspace{10{\3}{\2}{\3}\\01{\3}{\3}{\3}}{}$,
$\subspace{1011{\2}\\011{\2}1}{12}$,
$\subspace{1011{\3}\\011{\3}1}{}$,
$\subspace{10{\2}11\\0111{\2}}{}$,
$\subspace{10{\3}11\\0111{\3}}{}$,
$\subspace{101{\2}{\3}\\01{\2}{\2}{\2}}{}$,
$\subspace{10{\2}{\2}{\2}\\01{\2}1{\3}}{}$,
$\subspace{10{\2}{\2}1\\01{\2}{\3}{\3}}{}$,
$\subspace{10{\3}{\2}{\3}\\01{\2}{\2}1}{}$,
$\subspace{101{\3}{\2}\\01{\3}{\3}{\3}}{}$,
$\subspace{10{\2}{\3}{\2}\\01{\3}{\3}1}{}$,
$\subspace{10{\3}{\3}{\3}\\01{\3}1{\2}}{}$,
$\subspace{10{\3}{\3}1\\01{\3}{\2}{\2}}{}$,
$\subspace{10110\\011{\2}{\2}}{12}$,
$\subspace{10110\\011{\3}{\3}}{}$,
$\subspace{10{\2}1{\2}\\01110}{}$,
$\subspace{10{\3}1{\3}\\01110}{}$,
$\subspace{101{\2}{\2}\\01{\2}{\2}0}{}$,
$\subspace{10{\2}{\2}0\\01{\2}1{\2}}{}$,
$\subspace{10{\2}{\2}0\\01{\2}{\3}1}{}$,
$\subspace{10{\3}{\2}1\\01{\2}{\2}0}{}$,
$\subspace{101{\3}{\3}\\01{\3}{\3}0}{}$,
$\subspace{10{\2}{\3}1\\01{\3}{\3}0}{}$,
$\subspace{10{\3}{\3}0\\01{\3}1{\3}}{}$,
$\subspace{10{\3}{\3}0\\01{\3}{\2}1}{}$,
$\subspace{10111\\011{\2}{\2}}{12}$,
$\subspace{10111\\011{\3}{\3}}{}$,
$\subspace{10{\2}1{\2}\\01111}{}$,
$\subspace{10{\3}1{\3}\\01111}{}$,
$\subspace{101{\2}{\2}\\01{\2}{\2}{\3}}{}$,
$\subspace{10{\2}{\2}{\3}\\01{\2}1{\2}}{}$,
$\subspace{10{\2}{\2}{\3}\\01{\2}{\3}1}{}$,
$\subspace{10{\3}{\2}1\\01{\2}{\2}{\3}}{}$,
$\subspace{101{\3}{\3}\\01{\3}{\3}{\2}}{}$,
$\subspace{10{\2}{\3}1\\01{\3}{\3}{\2}}{}$,
$\subspace{10{\3}{\3}{\2}\\01{\3}1{\3}}{}$,
$\subspace{10{\3}{\3}{\2}\\01{\3}{\2}1}{}$,
$\subspace{101{\3}1\\011{\2}{\2}}{12}$,
$\subspace{101{\2}1\\011{\3}{\3}}{}$,
$\subspace{10{\2}{\3}1\\011{\3}{\2}}{}$,
$\subspace{10{\3}{\2}1\\011{\2}{\3}}{}$,
$\subspace{101{\3}{\3}\\01{\2}{\3}{\2}}{}$,
$\subspace{10{\2}{\3}{\3}\\01{\2}1{\2}}{}$,
$\subspace{10{\2}1{\3}\\01{\2}{\3}1}{}$,
$\subspace{10{\3}1{\3}\\01{\2}11}{}$,
$\subspace{101{\2}{\2}\\01{\3}{\2}{\3}}{}$,
$\subspace{10{\2}1{\2}\\01{\3}11}{}$,
$\subspace{10{\3}{\2}{\2}\\01{\3}1{\3}}{}$,
$\subspace{10{\3}1{\2}\\01{\3}{\2}1}{}$,
$\subspace{10{\2}{\2}{\2}\\0111{\2}}{6}$,
$\subspace{10{\3}{\3}{\3}\\0111{\3}}{}$,
$\subspace{1011{\2}\\01{\2}{\2}{\2}}{}$,
$\subspace{10{\3}{\3}1\\01{\2}{\2}1}{}$,
$\subspace{1011{\3}\\01{\3}{\3}{\3}}{}$,
$\subspace{10{\2}{\2}1\\01{\3}{\3}1}{}$,
$\subspace{10{\2}{\3}0\\01111}{12}$,
$\subspace{10{\2}{\2}{\3}\\011{\3}0}{}$,
$\subspace{10{\3}{\2}0\\01111}{}$,
$\subspace{10{\3}{\3}{\2}\\011{\2}0}{}$,
$\subspace{101{\3}0\\01{\2}{\2}{\3}}{}$,
$\subspace{10111\\01{\2}{\3}0}{}$,
$\subspace{10{\3}{\3}{\2}\\01{\2}10}{}$,
$\subspace{10{\3}10\\01{\2}{\2}{\3}}{}$,
$\subspace{10111\\01{\3}{\2}0}{}$,
$\subspace{101{\2}0\\01{\3}{\3}{\2}}{}$,
$\subspace{10{\2}{\2}{\3}\\01{\3}10}{}$,
$\subspace{10{\2}10\\01{\3}{\3}{\2}}{}$,
$\subspace{10{\2}1{\3}\\011{\2}0}{12}$,
$\subspace{10{\2}10\\011{\2}{\3}}{}$,
$\subspace{10{\3}1{\2}\\011{\3}0}{}$,
$\subspace{10{\3}10\\011{\3}{\2}}{}$,
$\subspace{101{\2}1\\01{\2}10}{}$,
$\subspace{101{\2}0\\01{\2}11}{}$,
$\subspace{10{\3}{\2}{\2}\\01{\2}{\3}0}{}$,
$\subspace{10{\3}{\2}0\\01{\2}{\3}{\2}}{}$,
$\subspace{101{\3}1\\01{\3}10}{}$,
$\subspace{101{\3}0\\01{\3}11}{}$,
$\subspace{10{\2}{\3}{\3}\\01{\3}{\2}0}{}$,
$\subspace{10{\2}{\3}0\\01{\3}{\2}{\3}}{}$,
$\subspace{10{\2}{\3}0\\011{\2}0}{6}$,
$\subspace{10{\3}{\2}0\\011{\3}0}{}$,
$\subspace{101{\3}0\\01{\2}10}{}$,
$\subspace{10{\3}10\\01{\2}{\3}0}{}$,
$\subspace{101{\2}0\\01{\3}10}{}$,
$\subspace{10{\2}10\\01{\3}{\2}0}{}$,
$\subspace{10{\2}1{\2}\\011{\3}0}{12}$,
$\subspace{10{\2}10\\011{\3}{\3}}{}$,
$\subspace{10{\3}1{\3}\\011{\2}0}{}$,
$\subspace{10{\3}10\\011{\2}{\2}}{}$,
$\subspace{101{\2}{\2}\\01{\2}{\3}0}{}$,
$\subspace{101{\2}0\\01{\2}{\3}1}{}$,
$\subspace{10{\3}{\2}1\\01{\2}10}{}$,
$\subspace{10{\3}{\2}0\\01{\2}1{\2}}{}$,
$\subspace{101{\3}{\3}\\01{\3}{\2}0}{}$,
$\subspace{101{\3}0\\01{\3}{\2}1}{}$,
$\subspace{10{\2}{\3}1\\01{\3}10}{}$,
$\subspace{10{\2}{\3}0\\01{\3}1{\3}}{}$,
$\subspace{10{\2}1{\2}\\011{\3}{\3}}{6}$,
$\subspace{10{\3}1{\3}\\011{\2}{\2}}{}$,
$\subspace{101{\2}{\2}\\01{\2}{\3}1}{}$,
$\subspace{10{\3}{\2}1\\01{\2}1{\2}}{}$,
$\subspace{101{\3}{\3}\\01{\3}{\2}1}{}$,
$\subspace{10{\2}{\3}1\\01{\3}1{\3}}{}$.

\section{Explicit lists of generator matrices for lower bounds for \texorpdfstring{$n_5(5,2;s)$}{n5(5,2;s)}}
\label{sec_explicit_5}
In this section we explicitly list the generator matrices of the projective $2-(n,5,s)_5$ systems found by ILP computations. For each system, we also list generators of the prescribed automorphism group of the system (which is not necessarily the full automorphism group); representatives of orbits under this group are marked by an index whose value indicates the orbit size.

\smallskip

\def\subspace#1#2{\left(\!\begin{smallmatrix}#1\end{smallmatrix}\!\right)^{\!\makebox[0mm][l]{$\scriptscriptstyle #2$}}}

$n_5(5,2;3)\ge 50$,$\left[\!\begin{smallmatrix}10000\\01100\\00110\\00011\\00001\end{smallmatrix}\!\right]$:
$\subspace{10012\\01044}{5}$,
$\subspace{10013\\01143}{}$,
$\subspace{10014\\01202}{}$,
$\subspace{10010\\01322}{}$,
$\subspace{10011\\01404}{}$,
$\subspace{10020\\01022}{5}$,
$\subspace{10022\\01124}{}$,
$\subspace{10024\\01231}{}$,
$\subspace{10021\\01304}{}$,
$\subspace{10023\\01434}{}$,
$\subspace{10040\\01020}{5}$,
$\subspace{10044\\01122}{}$,
$\subspace{10043\\01234}{}$,
$\subspace{10042\\01302}{}$,
$\subspace{10041\\01432}{}$,
$\subspace{10102\\01000}{5}$,
$\subspace{10112\\01100}{}$,
$\subspace{10123\\01210}{}$,
$\subspace{10130\\01331}{}$,
$\subspace{10143\\01414}{}$,
$\subspace{10104\\01033}{5}$,
$\subspace{10114\\01131}{}$,
$\subspace{10120\\01244}{}$,
$\subspace{10132\\01313}{}$,
$\subspace{10140\\01444}{}$,
$\subspace{10144\\01024}{5}$,
$\subspace{10103\\01121}{}$,
$\subspace{10113\\01233}{}$,
$\subspace{10124\\01301}{}$,
$\subspace{10131\\01431}{}$,
$\subspace{10201\\01003}{5}$,
$\subspace{10221\\01103}{}$,
$\subspace{10243\\01213}{}$,
$\subspace{10212\\01334}{}$,
$\subspace{10233\\01412}{}$,
$\subspace{10303\\01002}{5}$,
$\subspace{10333\\01102}{}$,
$\subspace{10311\\01212}{}$,
$\subspace{10342\\01333}{}$,
$\subspace{10321\\01411}{}$,
$\subspace{10302\\01032}{5}$,
$\subspace{10332\\01130}{}$,
$\subspace{10310\\01243}{}$,
$\subspace{10341\\01312}{}$,
$\subspace{10320\\01443}{}$,
$\subspace{10323\\01031}{5}$,
$\subspace{10300\\01134}{}$,
$\subspace{10330\\01242}{}$,
$\subspace{10313\\01311}{}$,
$\subspace{10344\\01442}{}$.

\smallskip

$n_5(5,2;4)\ge 77$,$\left[\!\begin{smallmatrix}00001\\10004\\01004\\00101\\00013\end{smallmatrix}\!\right]$:
$\subspace{01041\\00133}{11}$,
$\subspace{10021\\00104}{}$,
$\subspace{10020\\01032}{}$,
$\subspace{10133\\01320}{}$,
$\subspace{10140\\01444}{}$,
$\subspace{10200\\01240}{}$,
$\subspace{10312\\01332}{}$,
$\subspace{10341\\01304}{}$,
$\subspace{10340\\01404}{}$,
$\subspace{10402\\01002}{}$,
$\subspace{10411\\01331}{}$,
$\subspace{10000\\01134}{11}$,
$\subspace{10044\\01122}{}$,
$\subspace{10122\\01133}{}$,
$\subspace{10112\\01223}{}$,
$\subspace{10442\\01120}{}$,
$\subspace{10443\\01213}{}$,
$\subspace{10434\\01302}{}$,
$\subspace{11000\\00113}{}$,
$\subspace{11340\\00001}{}$,
$\subspace{13024\\00130}{}$,
$\subspace{13401\\00013}{}$,
$\subspace{10030\\01111}{11}$,
$\subspace{10033\\01430}{}$,
$\subspace{10111\\01022}{}$,
$\subspace{10143\\01033}{}$,
$\subspace{10231\\01003}{}$,
$\subspace{10222\\01114}{}$,
$\subspace{10301\\01113}{}$,
$\subspace{10331\\01433}{}$,
$\subspace{10400\\01123}{}$,
$\subspace{11014\\00103}{}$,
$\subspace{14020\\00111}{}$,
$\subspace{10040\\01103}{11}$,
$\subspace{10041\\01434}{}$,
$\subspace{10220\\01224}{}$,
$\subspace{10240\\01343}{}$,
$\subspace{10324\\01004}{}$,
$\subspace{10311\\01144}{}$,
$\subspace{10332\\01204}{}$,
$\subspace{10314\\01424}{}$,
$\subspace{10421\\01232}{}$,
$\subspace{11032\\00100}{}$,
$\subspace{14024\\00142}{}$,
$\subspace{10032\\01220}{11}$,
$\subspace{10114\\01102}{}$,
$\subspace{10123\\01101}{}$,
$\subspace{10110\\01342}{}$,
$\subspace{10124\\01403}{}$,
$\subspace{10213\\01140}{}$,
$\subspace{10223\\01244}{}$,
$\subspace{10244\\01242}{}$,
$\subspace{10330\\01231}{}$,
$\subspace{10431\\01143}{}$,
$\subspace{12024\\00140}{}$,
$\subspace{10014\\01303}{11}$,
$\subspace{10101\\01321}{}$,
$\subspace{10130\\01341}{}$,
$\subspace{10102\\01440}{}$,
$\subspace{10134\\01414}{}$,
$\subspace{10232\\01013}{}$,
$\subspace{10214\\01301}{}$,
$\subspace{10342\\01221}{}$,
$\subspace{10420\\01021}{}$,
$\subspace{10433\\01340}{}$,
$\subspace{13030\\00144}{}$,
$\subspace{10043\\01314}{11}$,
$\subspace{10104\\01234}{}$,
$\subspace{10211\\01030}{}$,
$\subspace{10313\\01042}{}$,
$\subspace{10322\\01043}{}$,
$\subspace{10302\\01334}{}$,
$\subspace{10304\\01423}{}$,
$\subspace{10413\\01214}{}$,
$\subspace{10424\\01211}{}$,
$\subspace{10432\\01442}{}$,
$\subspace{13023\\00121}{}$.

\smallskip

$n_5(5,2;6)\ge 132$,$\left[\!\begin{smallmatrix}00001\\10004\\01004\\00101\\00013\end{smallmatrix}\!\right]$:
$\subspace{01002\\00130}{11}$,
$\subspace{10200\\00013}{}$,
$\subspace{10032\\01141}{}$,
$\subspace{10020\\01302}{}$,
$\subspace{10000\\01403}{}$,
$\subspace{10442\\01041}{}$,
$\subspace{10411\\01133}{}$,
$\subspace{10441\\01444}{}$,
$\subspace{11000\\00140}{}$,
$\subspace{13022\\00104}{}$,
$\subspace{14030\\00001}{}$,
$\subspace{01012\\00113}{11}$,
$\subspace{10044\\00101}{}$,
$\subspace{10003\\01202}{}$,
$\subspace{10044\\01332}{}$,
$\subspace{10120\\01120}{}$,
$\subspace{10121\\01134}{}$,
$\subspace{10133\\01210}{}$,
$\subspace{10320\\01134}{}$,
$\subspace{10443\\01400}{}$,
$\subspace{12003\\00010}{}$,
$\subspace{13041\\00113}{}$,
$\subspace{01011\\00131}{11}$,
$\subspace{10013\\00114}{}$,
$\subspace{10002\\01104}{}$,
$\subspace{10111\\01240}{}$,
$\subspace{10113\\01310}{}$,
$\subspace{10233\\01114}{}$,
$\subspace{10203\\01443}{}$,
$\subspace{10312\\01130}{}$,
$\subspace{10422\\01044}{}$,
$\subspace{10440\\01330}{}$,
$\subspace{11001\\00011}{}$,
$\subspace{01042\\00132}{11}$,
$\subspace{10034\\00120}{}$,
$\subspace{10024\\01034}{}$,
$\subspace{10011\\01230}{}$,
$\subspace{10012\\01412}{}$,
$\subspace{10141\\01001}{}$,
$\subspace{10103\\01100}{}$,
$\subspace{10300\\01314}{}$,
$\subspace{10424\\01323}{}$,
$\subspace{12020\\00110}{}$,
$\subspace{14004\\00123}{}$,
$\subspace{10022\\01110}{11}$,
$\subspace{10144\\01300}{}$,
$\subspace{10124\\01414}{}$,
$\subspace{10210\\01144}{}$,
$\subspace{10232\\01421}{}$,
$\subspace{10332\\01244}{}$,
$\subspace{10321\\01431}{}$,
$\subspace{10333\\01441}{}$,
$\subspace{10420\\01121}{}$,
$\subspace{10412\\01301}{}$,
$\subspace{11042\\00112}{}$,
$\subspace{10030\\01113}{11}$,
$\subspace{10033\\01311}{}$,
$\subspace{10131\\01003}{}$,
$\subspace{10142\\01123}{}$,
$\subspace{10303\\01033}{}$,
$\subspace{10331\\01014}{}$,
$\subspace{10314\\01420}{}$,
$\subspace{10400\\01213}{}$,
$\subspace{10434\\01232}{}$,
$\subspace{11030\\00103}{}$,
$\subspace{13004\\00111}{}$,
$\subspace{10041\\01131}{11}$,
$\subspace{10040\\01211}{}$,
$\subspace{10334\\01004}{}$,
$\subspace{10343\\01112}{}$,
$\subspace{10313\\01312}{}$,
$\subspace{10410\\01221}{}$,
$\subspace{10413\\01402}{}$,
$\subspace{10433\\01424}{}$,
$\subspace{10444\\01423}{}$,
$\subspace{11043\\00142}{}$,
$\subspace{12011\\00100}{}$,
$\subspace{10023\\01214}{11}$,
$\subspace{10043\\01224}{}$,
$\subspace{10001\\01442}{}$,
$\subspace{10023\\01411}{}$,
$\subspace{10311\\01432}{}$,
$\subspace{10322\\01411}{}$,
$\subspace{11031\\00143}{}$,
$\subspace{12004\\00121}{}$,
$\subspace{12004\\00141}{}$,
$\subspace{14023\\00141}{}$,
$\subspace{14402\\00014}{}$,
$\subspace{10031\\01204}{11}$,
$\subspace{10114\\01324}{}$,
$\subspace{10143\\01343}{}$,
$\subspace{10110\\01422}{}$,
$\subspace{10132\\01434}{}$,
$\subspace{10234\\01103}{}$,
$\subspace{10220\\01243}{}$,
$\subspace{10240\\01433}{}$,
$\subspace{10324\\01101}{}$,
$\subspace{10421\\01143}{}$,
$\subspace{12044\\00124}{}$,
$\subspace{10004\\01320}{11}$,
$\subspace{10100\\01334}{}$,
$\subspace{10112\\01413}{}$,
$\subspace{10211\\01010}{}$,
$\subspace{10211\\01032}{}$,
$\subspace{10202\\01122}{}$,
$\subspace{10341\\01000}{}$,
$\subspace{10323\\01334}{}$,
$\subspace{10402\\01222}{}$,
$\subspace{13203\\00012}{}$,
$\subspace{14010\\00122}{}$,
$\subspace{10014\\01344}{11}$,
$\subspace{10042\\01321}{}$,
$\subspace{10102\\01013}{}$,
$\subspace{10134\\01021}{}$,
$\subspace{10130\\01231}{}$,
$\subspace{10101\\01313}{}$,
$\subspace{10123\\01440}{}$,
$\subspace{10212\\01303}{}$,
$\subspace{10214\\01340}{}$,
$\subspace{13003\\00134}{}$,
$\subspace{13031\\00144}{}$,
$\subspace{10021\\01401}{11}$,
$\subspace{10204\\01040}{}$,
$\subspace{10230\\01024}{}$,
$\subspace{10204\\01241}{}$,
$\subspace{10230\\01322}{}$,
$\subspace{10243\\01401}{}$,
$\subspace{10403\\01023}{}$,
$\subspace{10423\\01023}{}$,
$\subspace{10401\\01322}{}$,
$\subspace{10423\\01331}{}$,
$\subspace{14013\\00133}{}$.

\smallskip

$n_5(5,2;7)\ge 157$,$\left[\!\begin{smallmatrix}10000\\00100\\04410\\00001\\00044\end{smallmatrix}\!\right]$:
$\subspace{00010\\00001}{1}$,
$\subspace{00010\\00001}{1}$,
$\subspace{01000\\00124}{5}$,
$\subspace{01014\\00131}{}$,
$\subspace{01023\\00143}{}$,
$\subspace{01032\\00100}{}$,
$\subspace{01041\\00112}{}$,
$\subspace{10030\\00103}{15}$,
$\subspace{10030\\00110}{}$,
$\subspace{10030\\00122}{}$,
$\subspace{10030\\00134}{}$,
$\subspace{10030\\00141}{}$,
$\subspace{10022\\01003}{}$,
$\subspace{10022\\01012}{}$,
$\subspace{10022\\01021}{}$,
$\subspace{10022\\01030}{}$,
$\subspace{10022\\01044}{}$,
$\subspace{10003\\01102}{}$,
$\subspace{10003\\01110}{}$,
$\subspace{10003\\01123}{}$,
$\subspace{10003\\01131}{}$,
$\subspace{10003\\01144}{}$,
$\subspace{10002\\01202}{15}$,
$\subspace{10002\\01213}{}$,
$\subspace{10002\\01224}{}$,
$\subspace{10002\\01230}{}$,
$\subspace{10002\\01241}{}$,
$\subspace{10033\\01301}{}$,
$\subspace{10033\\01311}{}$,
$\subspace{10033\\01321}{}$,
$\subspace{10033\\01331}{}$,
$\subspace{10033\\01341}{}$,
$\subspace{10020\\01410}{}$,
$\subspace{10020\\01411}{}$,
$\subspace{10020\\01412}{}$,
$\subspace{10020\\01413}{}$,
$\subspace{10020\\01414}{}$,
$\subspace{10100\\01004}{15}$,
$\subspace{10112\\01013}{}$,
$\subspace{10124\\01022}{}$,
$\subspace{10131\\01031}{}$,
$\subspace{10143\\01040}{}$,
$\subspace{10403\\01134}{}$,
$\subspace{10410\\01113}{}$,
$\subspace{10422\\01142}{}$,
$\subspace{10434\\01121}{}$,
$\subspace{10441\\01100}{}$,
$\subspace{14003\\00130}{}$,
$\subspace{14012\\00123}{}$,
$\subspace{14021\\00111}{}$,
$\subspace{14030\\00104}{}$,
$\subspace{14044\\00142}{}$,
$\subspace{10104\\01140}{15}$,
$\subspace{10111\\01111}{}$,
$\subspace{10123\\01132}{}$,
$\subspace{10130\\01103}{}$,
$\subspace{10142\\01124}{}$,
$\subspace{10400\\01000}{}$,
$\subspace{10412\\01041}{}$,
$\subspace{10424\\01032}{}$,
$\subspace{10431\\01023}{}$,
$\subspace{10443\\01014}{}$,
$\subspace{11004\\00112}{}$,
$\subspace{11013\\00124}{}$,
$\subspace{11022\\00131}{}$,
$\subspace{11031\\00143}{}$,
$\subspace{11040\\00100}{}$,
$\subspace{10102\\01210}{15}$,
$\subspace{10114\\01243}{}$,
$\subspace{10121\\01221}{}$,
$\subspace{10133\\01204}{}$,
$\subspace{10140\\01232}{}$,
$\subspace{10104\\01430}{}$,
$\subspace{10111\\01432}{}$,
$\subspace{10123\\01434}{}$,
$\subspace{10130\\01431}{}$,
$\subspace{10142\\01433}{}$,
$\subspace{10204\\01332}{}$,
$\subspace{10211\\01302}{}$,
$\subspace{10223\\01322}{}$,
$\subspace{10230\\01342}{}$,
$\subspace{10242\\01312}{}$,
$\subspace{10100\\01320}{15}$,
$\subspace{10112\\01310}{}$,
$\subspace{10124\\01300}{}$,
$\subspace{10131\\01340}{}$,
$\subspace{10143\\01330}{}$,
$\subspace{10301\\01222}{}$,
$\subspace{10313\\01233}{}$,
$\subspace{10320\\01244}{}$,
$\subspace{10332\\01200}{}$,
$\subspace{10344\\01211}{}$,
$\subspace{10301\\01441}{}$,
$\subspace{10313\\01440}{}$,
$\subspace{10320\\01444}{}$,
$\subspace{10332\\01443}{}$,
$\subspace{10344\\01442}{}$,
$\subspace{10204\\01013}{15}$,
$\subspace{10211\\01040}{}$,
$\subspace{10223\\01022}{}$,
$\subspace{10230\\01004}{}$,
$\subspace{10242\\01031}{}$,
$\subspace{10304\\01100}{}$,
$\subspace{10311\\01142}{}$,
$\subspace{10323\\01134}{}$,
$\subspace{10330\\01121}{}$,
$\subspace{10342\\01113}{}$,
$\subspace{13004\\00130}{}$,
$\subspace{13013\\00104}{}$,
$\subspace{13022\\00123}{}$,
$\subspace{13031\\00142}{}$,
$\subspace{13040\\00111}{}$,
$\subspace{10202\\01101}{15}$,
$\subspace{10214\\01114}{}$,
$\subspace{10221\\01122}{}$,
$\subspace{10233\\01130}{}$,
$\subspace{10240\\01143}{}$,
$\subspace{10300\\01001}{}$,
$\subspace{10312\\01024}{}$,
$\subspace{10324\\01042}{}$,
$\subspace{10331\\01010}{}$,
$\subspace{10343\\01033}{}$,
$\subspace{12002\\00144}{}$,
$\subspace{12011\\00120}{}$,
$\subspace{12020\\00101}{}$,
$\subspace{12034\\00132}{}$,
$\subspace{12043\\00113}{}$,
$\subspace{10201\\01224}{15}$,
$\subspace{10213\\01213}{}$,
$\subspace{10220\\01202}{}$,
$\subspace{10232\\01241}{}$,
$\subspace{10244\\01230}{}$,
$\subspace{10203\\01413}{}$,
$\subspace{10210\\01414}{}$,
$\subspace{10222\\01410}{}$,
$\subspace{10234\\01411}{}$,
$\subspace{10241\\01412}{}$,
$\subspace{10400\\01331}{}$,
$\subspace{10412\\01341}{}$,
$\subspace{10424\\01301}{}$,
$\subspace{10431\\01311}{}$,
$\subspace{10443\\01321}{}$,
$\subspace{10302\\01340}{15}$,
$\subspace{10314\\01320}{}$,
$\subspace{10321\\01300}{}$,
$\subspace{10333\\01330}{}$,
$\subspace{10340\\01310}{}$,
$\subspace{10401\\01211}{}$,
$\subspace{10413\\01233}{}$,
$\subspace{10420\\01200}{}$,
$\subspace{10432\\01222}{}$,
$\subspace{10444\\01244}{}$,
$\subspace{10403\\01443}{}$,
$\subspace{10410\\01441}{}$,
$\subspace{10422\\01444}{}$,
$\subspace{10434\\01442}{}$,
$\subspace{10441\\01440}{}$.

\smallskip

$n_5(5,2;8)\ge 176$,$\left[\!\begin{smallmatrix}00001\\10004\\01004\\00101\\00013\end{smallmatrix}\!\right]$,$\left[\!\begin{smallmatrix}00130\\13442\\44114\\11002\\00001\end{smallmatrix}\!\right]$:
$\subspace{00110\\00001}{11}$,
$\subspace{01100\\00013}{}$,
$\subspace{10000\\00011}{}$,
$\subspace{10002\\01003}{}$,
$\subspace{10034\\01302}{}$,
$\subspace{10210\\01120}{}$,
$\subspace{10443\\01133}{}$,
$\subspace{10400\\01300}{}$,
$\subspace{10442\\01323}{}$,
$\subspace{11003\\00130}{}$,
$\subspace{13022\\00132}{}$,
$\subspace{01014\\00142}{55}$,
$\subspace{10041\\00124}{}$,
$\subspace{10031\\01004}{}$,
$\subspace{10012\\01111}{}$,
$\subspace{10003\\01231}{}$,
$\subspace{10023\\01221}{}$,
$\subspace{10021\\01241}{}$,
$\subspace{10040\\01240}{}$,
$\subspace{10030\\01314}{}$,
$\subspace{10024\\01422}{}$,
$\subspace{10104\\01012}{}$,
$\subspace{10104\\01040}{}$,
$\subspace{10121\\01042}{}$,
$\subspace{10123\\01123}{}$,
$\subspace{10123\\01141}{}$,
$\subspace{10131\\01140}{}$,
$\subspace{10142\\01220}{}$,
$\subspace{10132\\01304}{}$,
$\subspace{10131\\01324}{}$,
$\subspace{10141\\01311}{}$,
$\subspace{10112\\01410}{}$,
$\subspace{10114\\01420}{}$,
$\subspace{10142\\01424}{}$,
$\subspace{10222\\01001}{}$,
$\subspace{10244\\01014}{}$,
$\subspace{10223\\01122}{}$,
$\subspace{10221\\01201}{}$,
$\subspace{10234\\01231}{}$,
$\subspace{10212\\01331}{}$,
$\subspace{10243\\01314}{}$,
$\subspace{10300\\01022}{}$,
$\subspace{10320\\01033}{}$,
$\subspace{10301\\01132}{}$,
$\subspace{10340\\01143}{}$,
$\subspace{10303\\01221}{}$,
$\subspace{10330\\01213}{}$,
$\subspace{10312\\01311}{}$,
$\subspace{10331\\01400}{}$,
$\subspace{10403\\01042}{}$,
$\subspace{10424\\01024}{}$,
$\subspace{10433\\01102}{}$,
$\subspace{10424\\01210}{}$,
$\subspace{10441\\01230}{}$,
$\subspace{10433\\01313}{}$,
$\subspace{10401\\01411}{}$,
$\subspace{10414\\01420}{}$,
$\subspace{10434\\01412}{}$,
$\subspace{11044\\00123}{}$,
$\subspace{12001\\00100}{}$,
$\subspace{12030\\00141}{}$,
$\subspace{12042\\00133}{}$,
$\subspace{12303\\00010}{}$,
$\subspace{13041\\00103}{}$,
$\subspace{14033\\00120}{}$,
$\subspace{14040\\00101}{}$,
$\subspace{01032\\00104}{55}$,
$\subspace{10020\\00123}{}$,
$\subspace{10012\\01002}{}$,
$\subspace{10043\\01232}{}$,
$\subspace{10013\\01341}{}$,
$\subspace{10001\\01441}{}$,
$\subspace{10014\\01410}{}$,
$\subspace{10024\\01444}{}$,
$\subspace{10101\\01034}{}$,
$\subspace{10110\\01021}{}$,
$\subspace{10122\\01114}{}$,
$\subspace{10114\\01204}{}$,
$\subspace{10141\\01242}{}$,
$\subspace{10102\\01312}{}$,
$\subspace{10134\\01333}{}$,
$\subspace{10144\\01322}{}$,
$\subspace{10103\\01412}{}$,
$\subspace{10111\\01401}{}$,
$\subspace{10130\\01402}{}$,
$\subspace{10200\\01001}{}$,
$\subspace{10214\\01101}{}$,
$\subspace{10220\\01211}{}$,
$\subspace{10233\\01313}{}$,
$\subspace{10240\\01321}{}$,
$\subspace{10204\\01430}{}$,
$\subspace{10212\\01442}{}$,
$\subspace{10230\\01421}{}$,
$\subspace{10231\\01432}{}$,
$\subspace{10241\\01423}{}$,
$\subspace{10321\\01023}{}$,
$\subspace{10323\\01041}{}$,
$\subspace{10344\\01013}{}$,
$\subspace{10313\\01143}{}$,
$\subspace{10322\\01110}{}$,
$\subspace{10341\\01112}{}$,
$\subspace{10342\\01214}{}$,
$\subspace{10300\\01320}{}$,
$\subspace{10324\\01303}{}$,
$\subspace{10333\\01310}{}$,
$\subspace{10314\\01431}{}$,
$\subspace{10343\\01434}{}$,
$\subspace{10410\\01103}{}$,
$\subspace{10414\\01233}{}$,
$\subspace{10423\\01223}{}$,
$\subspace{10444\\01201}{}$,
$\subspace{10413\\01340}{}$,
$\subspace{10411\\01413}{}$,
$\subspace{10412\\01443}{}$,
$\subspace{10431\\01440}{}$,
$\subspace{11023\\00143}{}$,
$\subspace{12013\\00121}{}$,
$\subspace{13044\\00131}{}$,
$\subspace{14011\\00120}{}$,
$\subspace{14012\\00144}{}$,
$\subspace{14403\\00014}{}$,
$\subspace{01043\\00102}{55}$,
$\subspace{10010\\00134}{}$,
$\subspace{10042\\01030}{}$,
$\subspace{10040\\01103}{}$,
$\subspace{10004\\01241}{}$,
$\subspace{10020\\01301}{}$,
$\subspace{10032\\01444}{}$,
$\subspace{10041\\01434}{}$,
$\subspace{10100\\01024}{}$,
$\subspace{10140\\01011}{}$,
$\subspace{10120\\01104}{}$,
$\subspace{10122\\01130}{}$,
$\subspace{10133\\01223}{}$,
$\subspace{10124\\01332}{}$,
$\subspace{10143\\01341}{}$,
$\subspace{10113\\01404}{}$,
$\subspace{10202\\01040}{}$,
$\subspace{10232\\01041}{}$,
$\subspace{10243\\01000}{}$,
$\subspace{10201\\01234}{}$,
$\subspace{10203\\01244}{}$,
$\subspace{10213\\01233}{}$,
$\subspace{10220\\01224}{}$,
$\subspace{10241\\01212}{}$,
$\subspace{10200\\01342}{}$,
$\subspace{10240\\01343}{}$,
$\subspace{10224\\01433}{}$,
$\subspace{10242\\01403}{}$,
$\subspace{10244\\01421}{}$,
$\subspace{10324\\01004}{}$,
$\subspace{10330\\01031}{}$,
$\subspace{10310\\01102}{}$,
$\subspace{10311\\01144}{}$,
$\subspace{10304\\01200}{}$,
$\subspace{10321\\01203}{}$,
$\subspace{10332\\01204}{}$,
$\subspace{10343\\01202}{}$,
$\subspace{10342\\01222}{}$,
$\subspace{10314\\01424}{}$,
$\subspace{10403\\01010}{}$,
$\subspace{10410\\01044}{}$,
$\subspace{10420\\01002}{}$,
$\subspace{10432\\01020}{}$,
$\subspace{10401\\01124}{}$,
$\subspace{10421\\01232}{}$,
$\subspace{10430\\01220}{}$,
$\subspace{10404\\01312}{}$,
$\subspace{10411\\01414}{}$,
$\subspace{10440\\01402}{}$,
$\subspace{11032\\00100}{}$,
$\subspace{12013\\00122}{}$,
$\subspace{12402\\00012}{}$,
$\subspace{13014\\00104}{}$,
$\subspace{14024\\00142}{}$,
$\subspace{14031\\00140}{}$.

\smallskip

$n_5(5,2;11)\ge 260$,$\left[\!\begin{smallmatrix}10000\\00100\\04410\\00001\\00044\end{smallmatrix}\!\right]$:
$\subspace{00102\\00012}{3}$,
$\subspace{01002\\00014}{}$,
$\subspace{01104\\00013}{}$,
$\subspace{01003\\00141}{5}$,
$\subspace{01012\\00103}{}$,
$\subspace{01021\\00110}{}$,
$\subspace{01030\\00122}{}$,
$\subspace{01044\\00134}{}$,
$\subspace{01200\\00011}{3}$,
$\subspace{01300\\00010}{}$,
$\subspace{01440\\00001}{}$,
$\subspace{10001\\00012}{3}$,
$\subspace{10001\\00014}{}$,
$\subspace{10002\\00013}{}$,
$\subspace{10002\\00012}{3}$,
$\subspace{10002\\00014}{}$,
$\subspace{10004\\00013}{}$,
$\subspace{10003\\00013}{3}$,
$\subspace{10004\\00012}{}$,
$\subspace{10004\\00014}{}$,
$\subspace{10103\\01004}{15}$,
$\subspace{10110\\01013}{}$,
$\subspace{10122\\01022}{}$,
$\subspace{10134\\01031}{}$,
$\subspace{10141\\01040}{}$,
$\subspace{10402\\01142}{}$,
$\subspace{10414\\01121}{}$,
$\subspace{10421\\01100}{}$,
$\subspace{10433\\01134}{}$,
$\subspace{10440\\01113}{}$,
$\subspace{14002\\00104}{}$,
$\subspace{14011\\00142}{}$,
$\subspace{14020\\00130}{}$,
$\subspace{14034\\00123}{}$,
$\subspace{14043\\00111}{}$,
$\subspace{10101\\01024}{15}$,
$\subspace{10113\\01033}{}$,
$\subspace{10120\\01042}{}$,
$\subspace{10132\\01001}{}$,
$\subspace{10144\\01010}{}$,
$\subspace{10403\\01130}{}$,
$\subspace{10410\\01114}{}$,
$\subspace{10422\\01143}{}$,
$\subspace{10434\\01122}{}$,
$\subspace{10441\\01101}{}$,
$\subspace{14003\\00120}{}$,
$\subspace{14012\\00113}{}$,
$\subspace{14021\\00101}{}$,
$\subspace{14030\\00144}{}$,
$\subspace{14044\\00132}{}$,
$\subspace{10104\\01103}{15}$,
$\subspace{10111\\01124}{}$,
$\subspace{10123\\01140}{}$,
$\subspace{10130\\01111}{}$,
$\subspace{10142\\01132}{}$,
$\subspace{10401\\01014}{}$,
$\subspace{10413\\01000}{}$,
$\subspace{10420\\01041}{}$,
$\subspace{10432\\01032}{}$,
$\subspace{10444\\01023}{}$,
$\subspace{11004\\00143}{}$,
$\subspace{11013\\00100}{}$,
$\subspace{11022\\00112}{}$,
$\subspace{11031\\00124}{}$,
$\subspace{11040\\00131}{}$,
$\subspace{10102\\01113}{15}$,
$\subspace{10114\\01134}{}$,
$\subspace{10121\\01100}{}$,
$\subspace{10133\\01121}{}$,
$\subspace{10140\\01142}{}$,
$\subspace{10403\\01031}{}$,
$\subspace{10410\\01022}{}$,
$\subspace{10422\\01013}{}$,
$\subspace{10434\\01004}{}$,
$\subspace{10441\\01040}{}$,
$\subspace{11002\\00123}{}$,
$\subspace{11011\\00130}{}$,
$\subspace{11020\\00142}{}$,
$\subspace{11034\\00104}{}$,
$\subspace{11043\\00111}{}$,
$\subspace{10102\\01203}{15}$,
$\subspace{10114\\01231}{}$,
$\subspace{10121\\01214}{}$,
$\subspace{10133\\01242}{}$,
$\subspace{10140\\01220}{}$,
$\subspace{10101\\01424}{}$,
$\subspace{10113\\01421}{}$,
$\subspace{10120\\01423}{}$,
$\subspace{10132\\01420}{}$,
$\subspace{10144\\01422}{}$,
$\subspace{10203\\01344}{}$,
$\subspace{10210\\01314}{}$,
$\subspace{10222\\01334}{}$,
$\subspace{10234\\01304}{}$,
$\subspace{10241\\01324}{}$,
$\subspace{10104\\01210}{15}$,
$\subspace{10111\\01243}{}$,
$\subspace{10123\\01221}{}$,
$\subspace{10130\\01204}{}$,
$\subspace{10142\\01232}{}$,
$\subspace{10100\\01431}{}$,
$\subspace{10112\\01433}{}$,
$\subspace{10124\\01430}{}$,
$\subspace{10131\\01432}{}$,
$\subspace{10143\\01434}{}$,
$\subspace{10201\\01322}{}$,
$\subspace{10213\\01342}{}$,
$\subspace{10220\\01312}{}$,
$\subspace{10232\\01332}{}$,
$\subspace{10244\\01302}{}$,
$\subspace{10100\\01300}{15}$,
$\subspace{10112\\01340}{}$,
$\subspace{10124\\01330}{}$,
$\subspace{10131\\01320}{}$,
$\subspace{10143\\01310}{}$,
$\subspace{10302\\01244}{}$,
$\subspace{10314\\01200}{}$,
$\subspace{10321\\01211}{}$,
$\subspace{10333\\01222}{}$,
$\subspace{10340\\01233}{}$,
$\subspace{10300\\01440}{}$,
$\subspace{10312\\01444}{}$,
$\subspace{10324\\01443}{}$,
$\subspace{10331\\01442}{}$,
$\subspace{10343\\01441}{}$,
$\subspace{10103\\01300}{15}$,
$\subspace{10110\\01340}{}$,
$\subspace{10122\\01330}{}$,
$\subspace{10134\\01320}{}$,
$\subspace{10141\\01310}{}$,
$\subspace{10301\\01211}{}$,
$\subspace{10313\\01222}{}$,
$\subspace{10320\\01233}{}$,
$\subspace{10332\\01244}{}$,
$\subspace{10344\\01200}{}$,
$\subspace{10303\\01442}{}$,
$\subspace{10310\\01441}{}$,
$\subspace{10322\\01440}{}$,
$\subspace{10334\\01444}{}$,
$\subspace{10341\\01443}{}$,
$\subspace{10204\\01001}{15}$,
$\subspace{10211\\01033}{}$,
$\subspace{10223\\01010}{}$,
$\subspace{10230\\01042}{}$,
$\subspace{10242\\01024}{}$,
$\subspace{10300\\01143}{}$,
$\subspace{10312\\01130}{}$,
$\subspace{10324\\01122}{}$,
$\subspace{10331\\01114}{}$,
$\subspace{10343\\01101}{}$,
$\subspace{13000\\00120}{}$,
$\subspace{13014\\00144}{}$,
$\subspace{13023\\00113}{}$,
$\subspace{13032\\00132}{}$,
$\subspace{13041\\00101}{}$,
$\subspace{10204\\01023}{15}$,
$\subspace{10211\\01000}{}$,
$\subspace{10223\\01032}{}$,
$\subspace{10230\\01014}{}$,
$\subspace{10242\\01041}{}$,
$\subspace{10301\\01140}{}$,
$\subspace{10313\\01132}{}$,
$\subspace{10320\\01124}{}$,
$\subspace{10332\\01111}{}$,
$\subspace{10344\\01103}{}$,
$\subspace{13001\\00131}{}$,
$\subspace{13010\\00100}{}$,
$\subspace{13024\\00124}{}$,
$\subspace{13033\\00143}{}$,
$\subspace{13042\\00112}{}$,
$\subspace{10202\\01120}{15}$,
$\subspace{10214\\01133}{}$,
$\subspace{10221\\01141}{}$,
$\subspace{10233\\01104}{}$,
$\subspace{10240\\01112}{}$,
$\subspace{10303\\01034}{}$,
$\subspace{10310\\01002}{}$,
$\subspace{10322\\01020}{}$,
$\subspace{10334\\01043}{}$,
$\subspace{10341\\01011}{}$,
$\subspace{12002\\00121}{}$,
$\subspace{12011\\00102}{}$,
$\subspace{12020\\00133}{}$,
$\subspace{12034\\00114}{}$,
$\subspace{12043\\00140}{}$,
$\subspace{10200\\01144}{15}$,
$\subspace{10212\\01102}{}$,
$\subspace{10224\\01110}{}$,
$\subspace{10231\\01123}{}$,
$\subspace{10243\\01131}{}$,
$\subspace{10304\\01021}{}$,
$\subspace{10311\\01044}{}$,
$\subspace{10323\\01012}{}$,
$\subspace{10330\\01030}{}$,
$\subspace{10342\\01003}{}$,
$\subspace{12000\\00110}{}$,
$\subspace{12014\\00141}{}$,
$\subspace{12023\\00122}{}$,
$\subspace{12032\\00103}{}$,
$\subspace{12041\\00134}{}$,
$\subspace{10201\\01224}{15}$,
$\subspace{10213\\01213}{}$,
$\subspace{10220\\01202}{}$,
$\subspace{10232\\01241}{}$,
$\subspace{10244\\01230}{}$,
$\subspace{10203\\01413}{}$,
$\subspace{10210\\01414}{}$,
$\subspace{10222\\01410}{}$,
$\subspace{10234\\01411}{}$,
$\subspace{10241\\01412}{}$,
$\subspace{10400\\01331}{}$,
$\subspace{10412\\01341}{}$,
$\subspace{10424\\01301}{}$,
$\subspace{10431\\01311}{}$,
$\subspace{10443\\01321}{}$,
$\subspace{10202\\01220}{15}$,
$\subspace{10214\\01214}{}$,
$\subspace{10221\\01203}{}$,
$\subspace{10233\\01242}{}$,
$\subspace{10240\\01231}{}$,
$\subspace{10200\\01424}{}$,
$\subspace{10212\\01420}{}$,
$\subspace{10224\\01421}{}$,
$\subspace{10231\\01422}{}$,
$\subspace{10243\\01423}{}$,
$\subspace{10400\\01314}{}$,
$\subspace{10412\\01324}{}$,
$\subspace{10424\\01334}{}$,
$\subspace{10431\\01344}{}$,
$\subspace{10443\\01304}{}$,
$\subspace{10302\\01312}{15}$,
$\subspace{10314\\01342}{}$,
$\subspace{10321\\01322}{}$,
$\subspace{10333\\01302}{}$,
$\subspace{10340\\01332}{}$,
$\subspace{10404\\01210}{}$,
$\subspace{10411\\01232}{}$,
$\subspace{10423\\01204}{}$,
$\subspace{10430\\01221}{}$,
$\subspace{10442\\01243}{}$,
$\subspace{10404\\01432}{}$,
$\subspace{10411\\01430}{}$,
$\subspace{10423\\01433}{}$,
$\subspace{10430\\01431}{}$,
$\subspace{10442\\01434}{}$,
$\subspace{10304\\01343}{15}$,
$\subspace{10311\\01323}{}$,
$\subspace{10323\\01303}{}$,
$\subspace{10330\\01333}{}$,
$\subspace{10342\\01313}{}$,
$\subspace{10401\\01201}{}$,
$\subspace{10413\\01223}{}$,
$\subspace{10420\\01240}{}$,
$\subspace{10432\\01212}{}$,
$\subspace{10444\\01234}{}$,
$\subspace{10402\\01401}{}$,
$\subspace{10414\\01404}{}$,
$\subspace{10421\\01402}{}$,
$\subspace{10433\\01400}{}$,
$\subspace{10440\\01403}{}$.

\smallskip

$n_5(5,2;12)\ge 286$,$\left[\!\begin{smallmatrix}00001\\10004\\01004\\00101\\00013\end{smallmatrix}\!\right]$,$\left[\!\begin{smallmatrix}00130\\13442\\44114\\11002\\00001\end{smallmatrix}\!\right]$:
$\subspace{00110\\00001}{11}$,
$\subspace{01100\\00013}{}$,
$\subspace{10000\\00011}{}$,
$\subspace{10002\\01003}{}$,
$\subspace{10034\\01302}{}$,
$\subspace{10210\\01120}{}$,
$\subspace{10443\\01133}{}$,
$\subspace{10400\\01300}{}$,
$\subspace{10442\\01323}{}$,
$\subspace{11003\\00130}{}$,
$\subspace{13022\\00132}{}$,
$\subspace{01004\\00120}{55}$,
$\subspace{10024\\00100}{}$,
$\subspace{10040\\01042}{}$,
$\subspace{10010\\01124}{}$,
$\subspace{10012\\01204}{}$,
$\subspace{10020\\01203}{}$,
$\subspace{10032\\01220}{}$,
$\subspace{10041\\01201}{}$,
$\subspace{10010\\01313}{}$,
$\subspace{10032\\01412}{}$,
$\subspace{10122\\01031}{}$,
$\subspace{10131\\01020}{}$,
$\subspace{10104\\01103}{}$,
$\subspace{10114\\01102}{}$,
$\subspace{10123\\01101}{}$,
$\subspace{10141\\01144}{}$,
$\subspace{10142\\01200}{}$,
$\subspace{10110\\01342}{}$,
$\subspace{10124\\01403}{}$,
$\subspace{10200\\01031}{}$,
$\subspace{10213\\01040}{}$,
$\subspace{10240\\01001}{}$,
$\subspace{10213\\01140}{}$,
$\subspace{10201\\01221}{}$,
$\subspace{10220\\01230}{}$,
$\subspace{10223\\01244}{}$,
$\subspace{10244\\01242}{}$,
$\subspace{10201\\01311}{}$,
$\subspace{10212\\01420}{}$,
$\subspace{10224\\01444}{}$,
$\subspace{10234\\01403}{}$,
$\subspace{10243\\01410}{}$,
$\subspace{10310\\01002}{}$,
$\subspace{10303\\01144}{}$,
$\subspace{10310\\01223}{}$,
$\subspace{10332\\01203}{}$,
$\subspace{10330\\01231}{}$,
$\subspace{10342\\01241}{}$,
$\subspace{10332\\01341}{}$,
$\subspace{10300\\01424}{}$,
$\subspace{10414\\01024}{}$,
$\subspace{10430\\01014}{}$,
$\subspace{10433\\01020}{}$,
$\subspace{10430\\01041}{}$,
$\subspace{10431\\01143}{}$,
$\subspace{10411\\01200}{}$,
$\subspace{10403\\01342}{}$,
$\subspace{10401\\01422}{}$,
$\subspace{11043\\00102}{}$,
$\subspace{12010\\00142}{}$,
$\subspace{12024\\00140}{}$,
$\subspace{12031\\00104}{}$,
$\subspace{12044\\00123}{}$,
$\subspace{13034\\00102}{}$,
$\subspace{14032\\00140}{}$,
$\subspace{01021\\00121}{55}$,
$\subspace{01034\\00133}{}$,
$\subspace{01032\\00144}{}$,
$\subspace{01300\\00010}{}$,
$\subspace{10003\\00131}{}$,
$\subspace{10014\\00112}{}$,
$\subspace{10021\\00132}{}$,
$\subspace{10043\\00114}{}$,
$\subspace{10013\\01013}{}$,
$\subspace{10022\\01012}{}$,
$\subspace{10034\\01023}{}$,
$\subspace{10001\\01132}{}$,
$\subspace{10011\\01131}{}$,
$\subspace{10021\\01120}{}$,
$\subspace{10044\\01141}{}$,
$\subspace{10002\\01303}{}$,
$\subspace{10003\\01321}{}$,
$\subspace{10033\\01431}{}$,
$\subspace{10111\\01211}{}$,
$\subspace{10130\\01210}{}$,
$\subspace{10101\\01400}{}$,
$\subspace{10121\\01442}{}$,
$\subspace{10221\\01113}{}$,
$\subspace{10231\\01134}{}$,
$\subspace{10214\\01214}{}$,
$\subspace{10204\\01304}{}$,
$\subspace{10221\\01423}{}$,
$\subspace{10322\\01012}{}$,
$\subspace{10313\\01132}{}$,
$\subspace{10340\\01100}{}$,
$\subspace{10320\\01240}{}$,
$\subspace{10334\\01330}{}$,
$\subspace{10344\\01331}{}$,
$\subspace{10312\\01413}{}$,
$\subspace{10323\\01440}{}$,
$\subspace{10333\\01432}{}$,
$\subspace{10412\\01003}{}$,
$\subspace{10400\\01141}{}$,
$\subspace{10434\\01110}{}$,
$\subspace{10443\\01114}{}$,
$\subspace{10441\\01213}{}$,
$\subspace{10422\\01322}{}$,
$\subspace{10423\\01333}{}$,
$\subspace{10441\\01331}{}$,
$\subspace{10413\\01430}{}$,
$\subspace{11023\\00113}{}$,
$\subspace{11034\\00110}{}$,
$\subspace{11044\\00133}{}$,
$\subspace{11304\\00014}{}$,
$\subspace{13002\\00011}{}$,
$\subspace{13001\\00101}{}$,
$\subspace{13044\\00143}{}$,
$\subspace{13201\\00010}{}$,
$\subspace{14023\\00101}{}$,
$\subspace{14033\\00111}{}$,
$\subspace{01034\\00101}{55}$,
$\subspace{10103\\00010}{}$,
$\subspace{10003\\01011}{}$,
$\subspace{10021\\01222}{}$,
$\subspace{10042\\01320}{}$,
$\subspace{10004\\01441}{}$,
$\subspace{10043\\01422}{}$,
$\subspace{10120\\01044}{}$,
$\subspace{10144\\01000}{}$,
$\subspace{10142\\01240}{}$,
$\subspace{10100\\01314}{}$,
$\subspace{10121\\01333}{}$,
$\subspace{10131\\01330}{}$,
$\subspace{10104\\01421}{}$,
$\subspace{10111\\01420}{}$,
$\subspace{10140\\01400}{}$,
$\subspace{10202\\01042}{}$,
$\subspace{10234\\01141}{}$,
$\subspace{10203\\01210}{}$,
$\subspace{10221\\01212}{}$,
$\subspace{10242\\01214}{}$,
$\subspace{10231\\01304}{}$,
$\subspace{10232\\01312}{}$,
$\subspace{10241\\01343}{}$,
$\subspace{10302\\01032}{}$,
$\subspace{10312\\01014}{}$,
$\subspace{10323\\01043}{}$,
$\subspace{10344\\01012}{}$,
$\subspace{10311\\01114}{}$,
$\subspace{10320\\01104}{}$,
$\subspace{10314\\01234}{}$,
$\subspace{10341\\01213}{}$,
$\subspace{10340\\01233}{}$,
$\subspace{10304\\01344}{}$,
$\subspace{10322\\01301}{}$,
$\subspace{10321\\01331}{}$,
$\subspace{10303\\01430}{}$,
$\subspace{10343\\01442}{}$,
$\subspace{10402\\01030}{}$,
$\subspace{10424\\01010}{}$,
$\subspace{10410\\01132}{}$,
$\subspace{10422\\01224}{}$,
$\subspace{10434\\01221}{}$,
$\subspace{10433\\01232}{}$,
$\subspace{10440\\01202}{}$,
$\subspace{10421\\01311}{}$,
$\subspace{10404\\01404}{}$,
$\subspace{10420\\01402}{}$,
$\subspace{10432\\01413}{}$,
$\subspace{10441\\01414}{}$,
$\subspace{11002\\00122}{}$,
$\subspace{12014\\00133}{}$,
$\subspace{13040\\00134}{}$,
$\subspace{14031\\00121}{}$,
$\subspace{14400\\00012}{}$,
$\subspace{01041\\00120}{55}$,
$\subspace{10024\\00104}{}$,
$\subspace{10020\\01022}{}$,
$\subspace{10030\\01104}{}$,
$\subspace{10023\\01222}{}$,
$\subspace{10031\\01204}{}$,
$\subspace{10004\\01313}{}$,
$\subspace{10012\\01441}{}$,
$\subspace{10042\\01420}{}$,
$\subspace{10122\\01044}{}$,
$\subspace{10144\\01001}{}$,
$\subspace{10141\\01014}{}$,
$\subspace{10104\\01130}{}$,
$\subspace{10113\\01123}{}$,
$\subspace{10131\\01122}{}$,
$\subspace{10100\\01221}{}$,
$\subspace{10114\\01324}{}$,
$\subspace{10112\\01344}{}$,
$\subspace{10143\\01343}{}$,
$\subspace{10110\\01422}{}$,
$\subspace{10132\\01434}{}$,
$\subspace{10142\\01412}{}$,
$\subspace{10212\\01000}{}$,
$\subspace{10234\\01103}{}$,
$\subspace{10200\\01230}{}$,
$\subspace{10220\\01243}{}$,
$\subspace{10203\\01314}{}$,
$\subspace{10202\\01341}{}$,
$\subspace{10232\\01444}{}$,
$\subspace{10241\\01411}{}$,
$\subspace{10240\\01433}{}$,
$\subspace{10301\\01002}{}$,
$\subspace{10314\\01042}{}$,
$\subspace{10331\\01011}{}$,
$\subspace{10324\\01101}{}$,
$\subspace{10342\\01233}{}$,
$\subspace{10343\\01242}{}$,
$\subspace{10303\\01312}{}$,
$\subspace{10300\\01414}{}$,
$\subspace{10404\\01033}{}$,
$\subspace{10433\\01010}{}$,
$\subspace{10421\\01143}{}$,
$\subspace{10411\\01201}{}$,
$\subspace{10424\\01232}{}$,
$\subspace{10440\\01223}{}$,
$\subspace{10410\\01311}{}$,
$\subspace{10414\\01402}{}$,
$\subspace{10431\\01410}{}$,
$\subspace{11040\\00103}{}$,
$\subspace{12004\\00122}{}$,
$\subspace{12044\\00124}{}$,
$\subspace{12043\\00141}{}$,
$\subspace{13100\\00012}{}$,
$\subspace{14030\\00123}{}$,
$\subspace{14044\\00134}{}$,
$\subspace{01202\\00014}{55}$,
$\subspace{10001\\00124}{}$,
$\subspace{10031\\01043}{}$,
$\subspace{10013\\01111}{}$,
$\subspace{10014\\01212}{}$,
$\subspace{10040\\01224}{}$,
$\subspace{10030\\01322}{}$,
$\subspace{10023\\01431}{}$,
$\subspace{10041\\01443}{}$,
$\subspace{10101\\01024}{}$,
$\subspace{10133\\01021}{}$,
$\subspace{10102\\01110}{}$,
$\subspace{10112\\01131}{}$,
$\subspace{10132\\01140}{}$,
$\subspace{10130\\01224}{}$,
$\subspace{10134\\01241}{}$,
$\subspace{10133\\01324}{}$,
$\subspace{10140\\01303}{}$,
$\subspace{10120\\01411}{}$,
$\subspace{10140\\01432}{}$,
$\subspace{10204\\01030}{}$,
$\subspace{10211\\01022}{}$,
$\subspace{10233\\01004}{}$,
$\subspace{10230\\01043}{}$,
$\subspace{10243\\01013}{}$,
$\subspace{10242\\01122}{}$,
$\subspace{10242\\01234}{}$,
$\subspace{10214\\01332}{}$,
$\subspace{10222\\01334}{}$,
$\subspace{10223\\01404}{}$,
$\subspace{10302\\01023}{}$,
$\subspace{10330\\01030}{}$,
$\subspace{10302\\01102}{}$,
$\subspace{10304\\01123}{}$,
$\subspace{10311\\01202}{}$,
$\subspace{10304\\01332}{}$,
$\subspace{10311\\01310}{}$,
$\subspace{10334\\01340}{}$,
$\subspace{10301\\01424}{}$,
$\subspace{10333\\01404}{}$,
$\subspace{10331\\01423}{}$,
$\subspace{10413\\01033}{}$,
$\subspace{10444\\01040}{}$,
$\subspace{10403\\01112}{}$,
$\subspace{10423\\01212}{}$,
$\subspace{10412\\01321}{}$,
$\subspace{10401\\01440}{}$,
$\subspace{10432\\01401}{}$,
$\subspace{11031\\00131}{}$,
$\subspace{12023\\00143}{}$,
$\subspace{12030\\00144}{}$,
$\subspace{12041\\00100}{}$,
$\subspace{13023\\00103}{}$,
$\subspace{14020\\00142}{}$,
$\subspace{14043\\00141}{}$.

\smallskip

$n_5(5,2;13)\ge 308$,$\left[\!\begin{smallmatrix}00001\\10004\\01004\\00101\\00013\end{smallmatrix}\!\right]$,$\left[\!\begin{smallmatrix}00130\\13442\\44114\\11002\\00001\end{smallmatrix}\!\right]$:
$\subspace{00110\\00001}{11}$,
$\subspace{01100\\00013}{}$,
$\subspace{10000\\00011}{}$,
$\subspace{10002\\01003}{}$,
$\subspace{10034\\01302}{}$,
$\subspace{10210\\01120}{}$,
$\subspace{10443\\01133}{}$,
$\subspace{10400\\01300}{}$,
$\subspace{10442\\01323}{}$,
$\subspace{11003\\00130}{}$,
$\subspace{13022\\00132}{}$,
$\subspace{00110\\00001}{11}$,
$\subspace{01100\\00013}{}$,
$\subspace{10000\\00011}{}$,
$\subspace{10002\\01003}{}$,
$\subspace{10034\\01302}{}$,
$\subspace{10210\\01120}{}$,
$\subspace{10443\\01133}{}$,
$\subspace{10400\\01300}{}$,
$\subspace{10442\\01323}{}$,
$\subspace{11003\\00130}{}$,
$\subspace{13022\\00132}{}$,
$\subspace{00110\\00001}{11}$,
$\subspace{01100\\00013}{}$,
$\subspace{10000\\00011}{}$,
$\subspace{10002\\01003}{}$,
$\subspace{10034\\01302}{}$,
$\subspace{10210\\01120}{}$,
$\subspace{10443\\01133}{}$,
$\subspace{10400\\01300}{}$,
$\subspace{10442\\01323}{}$,
$\subspace{11003\\00130}{}$,
$\subspace{13022\\00132}{}$,
$\subspace{01014\\00101}{55}$,
$\subspace{10303\\00010}{}$,
$\subspace{10003\\01032}{}$,
$\subspace{10014\\01114}{}$,
$\subspace{10043\\01322}{}$,
$\subspace{10014\\01401}{}$,
$\subspace{10021\\01440}{}$,
$\subspace{10142\\01012}{}$,
$\subspace{10101\\01221}{}$,
$\subspace{10123\\01211}{}$,
$\subspace{10134\\01233}{}$,
$\subspace{10144\\01214}{}$,
$\subspace{10104\\01333}{}$,
$\subspace{10111\\01334}{}$,
$\subspace{10131\\01320}{}$,
$\subspace{10143\\01340}{}$,
$\subspace{10101\\01422}{}$,
$\subspace{10103\\01431}{}$,
$\subspace{10121\\01420}{}$,
$\subspace{10134\\01430}{}$,
$\subspace{10231\\01013}{}$,
$\subspace{10241\\01013}{}$,
$\subspace{10230\\01141}{}$,
$\subspace{10204\\01231}{}$,
$\subspace{10204\\01340}{}$,
$\subspace{10211\\01330}{}$,
$\subspace{10221\\01322}{}$,
$\subspace{10234\\01331}{}$,
$\subspace{10211\\01433}{}$,
$\subspace{10230\\01442}{}$,
$\subspace{10312\\01042}{}$,
$\subspace{10322\\01023}{}$,
$\subspace{10340\\01034}{}$,
$\subspace{10314\\01110}{}$,
$\subspace{10320\\01311}{}$,
$\subspace{10333\\01314}{}$,
$\subspace{10344\\01304}{}$,
$\subspace{10313\\01441}{}$,
$\subspace{10321\\01401}{}$,
$\subspace{10341\\01431}{}$,
$\subspace{10441\\01023}{}$,
$\subspace{10433\\01132}{}$,
$\subspace{10412\\01210}{}$,
$\subspace{10423\\01213}{}$,
$\subspace{10422\\01232}{}$,
$\subspace{10424\\01240}{}$,
$\subspace{10412\\01334}{}$,
$\subspace{10402\\01402}{}$,
$\subspace{10410\\01400}{}$,
$\subspace{10423\\01440}{}$,
$\subspace{10434\\01421}{}$,
$\subspace{11002\\00144}{}$,
$\subspace{13032\\00121}{}$,
$\subspace{14013\\00144}{}$,
$\subspace{14032\\00133}{}$,
$\subspace{01010\\00142}{55}$,
$\subspace{01044\\00101}{}$,
$\subspace{01040\\00140}{}$,
$\subspace{10032\\00131}{}$,
$\subspace{10041\\00133}{}$,
$\subspace{10203\\00010}{}$,
$\subspace{10003\\01024}{}$,
$\subspace{10013\\01011}{}$,
$\subspace{10021\\01004}{}$,
$\subspace{10041\\01124}{}$,
$\subspace{10040\\01140}{}$,
$\subspace{10004\\01220}{}$,
$\subspace{10010\\01304}{}$,
$\subspace{10042\\01303}{}$,
$\subspace{10040\\01321}{}$,
$\subspace{10001\\01414}{}$,
$\subspace{10121\\01104}{}$,
$\subspace{10102\\01344}{}$,
$\subspace{10100\\01424}{}$,
$\subspace{10130\\01424}{}$,
$\subspace{10224\\01004}{}$,
$\subspace{10244\\01000}{}$,
$\subspace{10232\\01141}{}$,
$\subspace{10201\\01240}{}$,
$\subspace{10214\\01213}{}$,
$\subspace{10223\\01222}{}$,
$\subspace{10243\\01301}{}$,
$\subspace{10221\\01432}{}$,
$\subspace{10312\\01020}{}$,
$\subspace{10332\\01040}{}$,
$\subspace{10310\\01112}{}$,
$\subspace{10340\\01131}{}$,
$\subspace{10334\\01200}{}$,
$\subspace{10330\\01210}{}$,
$\subspace{10420\\01024}{}$,
$\subspace{10434\\01021}{}$,
$\subspace{10440\\01012}{}$,
$\subspace{10444\\01031}{}$,
$\subspace{10404\\01102}{}$,
$\subspace{10401\\01132}{}$,
$\subspace{10403\\01144}{}$,
$\subspace{10422\\01203}{}$,
$\subspace{10441\\01241}{}$,
$\subspace{10421\\01330}{}$,
$\subspace{10430\\01343}{}$,
$\subspace{10403\\01403}{}$,
$\subspace{11002\\00100}{}$,
$\subspace{11013\\00142}{}$,
$\subspace{11022\\00122}{}$,
$\subspace{12204\\00012}{}$,
$\subspace{13011\\00100}{}$,
$\subspace{13030\\00134}{}$,
$\subspace{13043\\00102}{}$,
$\subspace{14033\\00143}{}$,
$\subspace{14102\\00014}{}$,
$\subspace{01033\\00103}{55}$,
$\subspace{10030\\00103}{}$,
$\subspace{10030\\01043}{}$,
$\subspace{10023\\01214}{}$,
$\subspace{10031\\01202}{}$,
$\subspace{10043\\01224}{}$,
$\subspace{10013\\01333}{}$,
$\subspace{10001\\01442}{}$,
$\subspace{10023\\01411}{}$,
$\subspace{10031\\01423}{}$,
$\subspace{10102\\01022}{}$,
$\subspace{10112\\01030}{}$,
$\subspace{10130\\01032}{}$,
$\subspace{10140\\01030}{}$,
$\subspace{10103\\01122}{}$,
$\subspace{10111\\01123}{}$,
$\subspace{10120\\01114}{}$,
$\subspace{10112\\01234}{}$,
$\subspace{10120\\01243}{}$,
$\subspace{10133\\01224}{}$,
$\subspace{10140\\01243}{}$,
$\subspace{10132\\01332}{}$,
$\subspace{10132\\01404}{}$,
$\subspace{10222\\01021}{}$,
$\subspace{10233\\01043}{}$,
$\subspace{10242\\01022}{}$,
$\subspace{10242\\01034}{}$,
$\subspace{10214\\01111}{}$,
$\subspace{10222\\01212}{}$,
$\subspace{10331\\01033}{}$,
$\subspace{10304\\01112}{}$,
$\subspace{10302\\01122}{}$,
$\subspace{10331\\01123}{}$,
$\subspace{10344\\01212}{}$,
$\subspace{10304\\01310}{}$,
$\subspace{10301\\01320}{}$,
$\subspace{10311\\01324}{}$,
$\subspace{10323\\01303}{}$,
$\subspace{10341\\01321}{}$,
$\subspace{10302\\01404}{}$,
$\subspace{10301\\01413}{}$,
$\subspace{10311\\01432}{}$,
$\subspace{10322\\01411}{}$,
$\subspace{10402\\01111}{}$,
$\subspace{10444\\01202}{}$,
$\subspace{10413\\01324}{}$,
$\subspace{10432\\01443}{}$,
$\subspace{11031\\00143}{}$,
$\subspace{12004\\00121}{}$,
$\subspace{12004\\00141}{}$,
$\subspace{12023\\00124}{}$,
$\subspace{13040\\00131}{}$,
$\subspace{14023\\00141}{}$,
$\subspace{14043\\00124}{}$,
$\subspace{14402\\00014}{}$,
$\subspace{01030\\00123}{55}$,
$\subspace{10012\\00134}{}$,
$\subspace{10042\\01001}{}$,
$\subspace{10020\\01143}{}$,
$\subspace{10004\\01242}{}$,
$\subspace{10012\\01223}{}$,
$\subspace{10024\\01234}{}$,
$\subspace{10024\\01301}{}$,
$\subspace{10114\\01002}{}$,
$\subspace{10114\\01010}{}$,
$\subspace{10110\\01042}{}$,
$\subspace{10122\\01001}{}$,
$\subspace{10140\\01044}{}$,
$\subspace{10100\\01143}{}$,
$\subspace{10123\\01130}{}$,
$\subspace{10131\\01101}{}$,
$\subspace{10141\\01104}{}$,
$\subspace{10120\\01231}{}$,
$\subspace{10104\\01313}{}$,
$\subspace{10110\\01343}{}$,
$\subspace{10113\\01341}{}$,
$\subspace{10133\\01412}{}$,
$\subspace{10142\\01444}{}$,
$\subspace{10212\\01222}{}$,
$\subspace{10212\\01230}{}$,
$\subspace{10232\\01201}{}$,
$\subspace{10200\\01311}{}$,
$\subspace{10202\\01314}{}$,
$\subspace{10234\\01313}{}$,
$\subspace{10203\\01410}{}$,
$\subspace{10240\\01410}{}$,
$\subspace{10242\\01414}{}$,
$\subspace{10303\\01041}{}$,
$\subspace{10324\\01043}{}$,
$\subspace{10342\\01011}{}$,
$\subspace{10300\\01204}{}$,
$\subspace{10300\\01212}{}$,
$\subspace{10302\\01230}{}$,
$\subspace{10311\\01420}{}$,
$\subspace{10411\\01014}{}$,
$\subspace{10431\\01000}{}$,
$\subspace{10424\\01101}{}$,
$\subspace{10414\\01201}{}$,
$\subspace{10414\\01202}{}$,
$\subspace{10421\\01224}{}$,
$\subspace{10404\\01332}{}$,
$\subspace{10420\\01422}{}$,
$\subspace{10432\\01434}{}$,
$\subspace{10440\\01404}{}$,
$\subspace{11030\\00104}{}$,
$\subspace{12032\\00120}{}$,
$\subspace{12042\\00122}{}$,
$\subspace{12042\\00123}{}$,
$\subspace{12403\\00012}{}$,
$\subspace{13014\\00120}{}$,
$\subspace{01403\\00014}{55}$,
$\subspace{10001\\00140}{}$,
$\subspace{10032\\01041}{}$,
$\subspace{10013\\01143}{}$,
$\subspace{10012\\01233}{}$,
$\subspace{10043\\01223}{}$,
$\subspace{10024\\01333}{}$,
$\subspace{10010\\01402}{}$,
$\subspace{10020\\01410}{}$,
$\subspace{10130\\01031}{}$,
$\subspace{10143\\01032}{}$,
$\subspace{10124\\01101}{}$,
$\subspace{10144\\01114}{}$,
$\subspace{10103\\01230}{}$,
$\subspace{10102\\01242}{}$,
$\subspace{10110\\01232}{}$,
$\subspace{10111\\01342}{}$,
$\subspace{10122\\01310}{}$,
$\subspace{10114\\01442}{}$,
$\subspace{10241\\01001}{}$,
$\subspace{10220\\01144}{}$,
$\subspace{10214\\01200}{}$,
$\subspace{10233\\01214}{}$,
$\subspace{10201\\01312}{}$,
$\subspace{10212\\01320}{}$,
$\subspace{10200\\01441}{}$,
$\subspace{10213\\01432}{}$,
$\subspace{10224\\01423}{}$,
$\subspace{10240\\01413}{}$,
$\subspace{10343\\01020}{}$,
$\subspace{10342\\01034}{}$,
$\subspace{10300\\01124}{}$,
$\subspace{10332\\01131}{}$,
$\subspace{10341\\01103}{}$,
$\subspace{10322\\01244}{}$,
$\subspace{10334\\01204}{}$,
$\subspace{10310\\01303}{}$,
$\subspace{10324\\01313}{}$,
$\subspace{10344\\01341}{}$,
$\subspace{10314\\01443}{}$,
$\subspace{10323\\01433}{}$,
$\subspace{10414\\01002}{}$,
$\subspace{10430\\01021}{}$,
$\subspace{10402\\01112}{}$,
$\subspace{10413\\01201}{}$,
$\subspace{10431\\01203}{}$,
$\subspace{10410\\01321}{}$,
$\subspace{10444\\01434}{}$,
$\subspace{11012\\00131}{}$,
$\subspace{12020\\00123}{}$,
$\subspace{12041\\00121}{}$,
$\subspace{13023\\00120}{}$,
$\subspace{14002\\00104}{}$,
$\subspace{14021\\00102}{}$,
$\subspace{14034\\00143}{}$.

\smallskip

$n_5(5,2;14)\ge 330$,$\left[\!\begin{smallmatrix}00001\\10004\\01004\\00101\\00013\end{smallmatrix}\!\right]$,$\left[\!\begin{smallmatrix}00130\\13442\\44114\\11002\\00001\end{smallmatrix}\!\right]$:
$\subspace{00102\\00014}{55}$,
$\subspace{01021\\00114}{}$,
$\subspace{01020\\00143}{}$,
$\subspace{01100\\00012}{}$,
$\subspace{10001\\00014}{}$,
$\subspace{10004\\00110}{}$,
$\subspace{10102\\00011}{}$,
$\subspace{10001\\01003}{}$,
$\subspace{10002\\01011}{}$,
$\subspace{10011\\01000}{}$,
$\subspace{10013\\01121}{}$,
$\subspace{10022\\01144}{}$,
$\subspace{10012\\01203}{}$,
$\subspace{10034\\01222}{}$,
$\subspace{10013\\01310}{}$,
$\subspace{10032\\01301}{}$,
$\subspace{10044\\01313}{}$,
$\subspace{10110\\01010}{}$,
$\subspace{10100\\01101}{}$,
$\subspace{10102\\01321}{}$,
$\subspace{10130\\01410}{}$,
$\subspace{10203\\01021}{}$,
$\subspace{10214\\01044}{}$,
$\subspace{10202\\01112}{}$,
$\subspace{10213\\01120}{}$,
$\subspace{10214\\01142}{}$,
$\subspace{10233\\01300}{}$,
$\subspace{10233\\01343}{}$,
$\subspace{10201\\01432}{}$,
$\subspace{10212\\01423}{}$,
$\subspace{10210\\01443}{}$,
$\subspace{10232\\01403}{}$,
$\subspace{10300\\01031}{}$,
$\subspace{10310\\01001}{}$,
$\subspace{10334\\01323}{}$,
$\subspace{10332\\01443}{}$,
$\subspace{10414\\01104}{}$,
$\subspace{10413\\01134}{}$,
$\subspace{10444\\01131}{}$,
$\subspace{10400\\01200}{}$,
$\subspace{10430\\01230}{}$,
$\subspace{10404\\01303}{}$,
$\subspace{10413\\01342}{}$,
$\subspace{10421\\01310}{}$,
$\subspace{10440\\01321}{}$,
$\subspace{10420\\01423}{}$,
$\subspace{10443\\01414}{}$,
$\subspace{11001\\00112}{}$,
$\subspace{11003\\00122}{}$,
$\subspace{11001\\00131}{}$,
$\subspace{12011\\00132}{}$,
$\subspace{12031\\00123}{}$,
$\subspace{13014\\00140}{}$,
$\subspace{13023\\00113}{}$,
$\subspace{13021\\00131}{}$,
$\subspace{01001\\00014}{55}$,
$\subspace{01000\\00123}{}$,
$\subspace{10001\\00103}{}$,
$\subspace{10012\\00143}{}$,
$\subspace{10030\\01040}{}$,
$\subspace{10013\\01102}{}$,
$\subspace{10023\\01140}{}$,
$\subspace{10004\\01230}{}$,
$\subspace{10024\\01222}{}$,
$\subspace{10042\\01220}{}$,
$\subspace{10041\\01311}{}$,
$\subspace{10040\\01343}{}$,
$\subspace{10031\\01410}{}$,
$\subspace{10114\\01011}{}$,
$\subspace{10123\\01044}{}$,
$\subspace{10131\\01004}{}$,
$\subspace{10110\\01104}{}$,
$\subspace{10113\\01143}{}$,
$\subspace{10102\\01324}{}$,
$\subspace{10132\\01301}{}$,
$\subspace{10142\\01313}{}$,
$\subspace{10112\\01423}{}$,
$\subspace{10130\\01412}{}$,
$\subspace{10141\\01422}{}$,
$\subspace{10212\\01014}{}$,
$\subspace{10232\\01021}{}$,
$\subspace{10203\\01111}{}$,
$\subspace{10222\\01131}{}$,
$\subspace{10234\\01243}{}$,
$\subspace{10233\\01344}{}$,
$\subspace{10244\\01310}{}$,
$\subspace{10214\\01414}{}$,
$\subspace{10223\\01443}{}$,
$\subspace{10334\\01022}{}$,
$\subspace{10303\\01122}{}$,
$\subspace{10301\\01130}{}$,
$\subspace{10330\\01221}{}$,
$\subspace{10404\\01112}{}$,
$\subspace{10444\\01101}{}$,
$\subspace{10401\\01201}{}$,
$\subspace{10440\\01231}{}$,
$\subspace{10414\\01321}{}$,
$\subspace{10420\\01303}{}$,
$\subspace{10413\\01424}{}$,
$\subspace{10421\\01420}{}$,
$\subspace{10433\\01411}{}$,
$\subspace{11024\\00131}{}$,
$\subspace{11043\\00141}{}$,
$\subspace{12030\\00120}{}$,
$\subspace{12041\\00134}{}$,
$\subspace{12300\\00012}{}$,
$\subspace{13023\\00142}{}$,
$\subspace{13034\\00100}{}$,
$\subspace{14003\\00122}{}$,
$\subspace{14032\\00124}{}$,
$\subspace{01031\\00133}{55}$,
$\subspace{01034\\00140}{}$,
$\subspace{10021\\00140}{}$,
$\subspace{10032\\00121}{}$,
$\subspace{10032\\01020}{}$,
$\subspace{10043\\01041}{}$,
$\subspace{10010\\01141}{}$,
$\subspace{10043\\01103}{}$,
$\subspace{10003\\01320}{}$,
$\subspace{10020\\01330}{}$,
$\subspace{10010\\01430}{}$,
$\subspace{10122\\01020}{}$,
$\subspace{10121\\01032}{}$,
$\subspace{10103\\01244}{}$,
$\subspace{10111\\01244}{}$,
$\subspace{10124\\01203}{}$,
$\subspace{10124\\01204}{}$,
$\subspace{10221\\01031}{}$,
$\subspace{10220\\01114}{}$,
$\subspace{10224\\01144}{}$,
$\subspace{10201\\01223}{}$,
$\subspace{10213\\01203}{}$,
$\subspace{10213\\01240}{}$,
$\subspace{10231\\01213}{}$,
$\subspace{10200\\01442}{}$,
$\subspace{10231\\01413}{}$,
$\subspace{10240\\01442}{}$,
$\subspace{10310\\01132}{}$,
$\subspace{10320\\01124}{}$,
$\subspace{10323\\01200}{}$,
$\subspace{10323\\01210}{}$,
$\subspace{10324\\01214}{}$,
$\subspace{10344\\01200}{}$,
$\subspace{10310\\01331}{}$,
$\subspace{10312\\01333}{}$,
$\subspace{10322\\01342}{}$,
$\subspace{10341\\01341}{}$,
$\subspace{10322\\01430}{}$,
$\subspace{10332\\01413}{}$,
$\subspace{10332\\01434}{}$,
$\subspace{10342\\01400}{}$,
$\subspace{10344\\01403}{}$,
$\subspace{10402\\01012}{}$,
$\subspace{10422\\01002}{}$,
$\subspace{10430\\01032}{}$,
$\subspace{10430\\01034}{}$,
$\subspace{10434\\01144}{}$,
$\subspace{10402\\01342}{}$,
$\subspace{10441\\01333}{}$,
$\subspace{11012\\00102}{}$,
$\subspace{11032\\00121}{}$,
$\subspace{13002\\00104}{}$,
$\subspace{13203\\00010}{}$,
$\subspace{14002\\00102}{}$,
$\subspace{14010\\00101}{}$,
$\subspace{01040\\00123}{55}$,
$\subspace{01042\\00142}{}$,
$\subspace{10012\\00111}{}$,
$\subspace{10041\\00113}{}$,
$\subspace{10033\\01001}{}$,
$\subspace{10044\\01004}{}$,
$\subspace{10022\\01141}{}$,
$\subspace{10002\\01221}{}$,
$\subspace{10024\\01241}{}$,
$\subspace{10040\\01210}{}$,
$\subspace{10003\\01314}{}$,
$\subspace{10034\\01311}{}$,
$\subspace{10011\\01420}{}$,
$\subspace{10021\\01420}{}$,
$\subspace{10114\\01014}{}$,
$\subspace{10104\\01132}{}$,
$\subspace{10121\\01120}{}$,
$\subspace{10123\\01140}{}$,
$\subspace{10142\\01143}{}$,
$\subspace{10110\\01323}{}$,
$\subspace{10141\\01331}{}$,
$\subspace{10131\\01412}{}$,
$\subspace{10131\\01422}{}$,
$\subspace{10142\\01410}{}$,
$\subspace{10223\\01121}{}$,
$\subspace{10234\\01100}{}$,
$\subspace{10221\\01231}{}$,
$\subspace{10243\\01240}{}$,
$\subspace{10210\\01314}{}$,
$\subspace{10244\\01300}{}$,
$\subspace{10300\\01003}{}$,
$\subspace{10312\\01024}{}$,
$\subspace{10330\\01042}{}$,
$\subspace{10303\\01142}{}$,
$\subspace{10303\\01213}{}$,
$\subspace{10414\\01014}{}$,
$\subspace{10443\\01012}{}$,
$\subspace{10424\\01102}{}$,
$\subspace{10434\\01101}{}$,
$\subspace{10403\\01230}{}$,
$\subspace{10422\\01201}{}$,
$\subspace{10441\\01220}{}$,
$\subspace{10400\\01330}{}$,
$\subspace{10433\\01311}{}$,
$\subspace{10424\\01424}{}$,
$\subspace{11022\\00112}{}$,
$\subspace{11030\\00114}{}$,
$\subspace{12003\\00100}{}$,
$\subspace{12033\\00101}{}$,
$\subspace{12034\\00120}{}$,
$\subspace{12201\\00011}{}$,
$\subspace{13033\\00132}{}$,
$\subspace{13104\\00010}{}$,
$\subspace{14032\\00110}{}$,
$\subspace{14041\\00133}{}$,
$\subspace{01401\\00010}{55}$,
$\subspace{10003\\00141}{}$,
$\subspace{10023\\01023}{}$,
$\subspace{10021\\01122}{}$,
$\subspace{10014\\01213}{}$,
$\subspace{10031\\01340}{}$,
$\subspace{10014\\01431}{}$,
$\subspace{10030\\01433}{}$,
$\subspace{10134\\01110}{}$,
$\subspace{10101\\01233}{}$,
$\subspace{10134\\01211}{}$,
$\subspace{10132\\01233}{}$,
$\subspace{10144\\01210}{}$,
$\subspace{10101\\01331}{}$,
$\subspace{10121\\01322}{}$,
$\subspace{10143\\01322}{}$,
$\subspace{10112\\01440}{}$,
$\subspace{10143\\01401}{}$,
$\subspace{10211\\01023}{}$,
$\subspace{10221\\01033}{}$,
$\subspace{10204\\01123}{}$,
$\subspace{10211\\01240}{}$,
$\subspace{10222\\01312}{}$,
$\subspace{10230\\01334}{}$,
$\subspace{10241\\01324}{}$,
$\subspace{10204\\01400}{}$,
$\subspace{10241\\01440}{}$,
$\subspace{10313\\01013}{}$,
$\subspace{10333\\01013}{}$,
$\subspace{10343\\01022}{}$,
$\subspace{10320\\01110}{}$,
$\subspace{10331\\01132}{}$,
$\subspace{10340\\01111}{}$,
$\subspace{10314\\01232}{}$,
$\subspace{10301\\01330}{}$,
$\subspace{10313\\01304}{}$,
$\subspace{10314\\01340}{}$,
$\subspace{10343\\01334}{}$,
$\subspace{10312\\01402}{}$,
$\subspace{10321\\01441}{}$,
$\subspace{10333\\01421}{}$,
$\subspace{10423\\01012}{}$,
$\subspace{10423\\01141}{}$,
$\subspace{10410\\01211}{}$,
$\subspace{10412\\01232}{}$,
$\subspace{10434\\01243}{}$,
$\subspace{10410\\01312}{}$,
$\subspace{10422\\01402}{}$,
$\subspace{10441\\01433}{}$,
$\subspace{11010\\00133}{}$,
$\subspace{12043\\00144}{}$,
$\subspace{13024\\00124}{}$,
$\subspace{14013\\00101}{}$,
$\subspace{14033\\00103}{}$,
$\subspace{14044\\00144}{}$,
$\subspace{10014\\01224}{55}$,
$\subspace{10020\\01202}{}$,
$\subspace{10020\\01234}{}$,
$\subspace{10120\\01002}{}$,
$\subspace{10140\\01041}{}$,
$\subspace{10122\\01110}{}$,
$\subspace{10134\\01103}{}$,
$\subspace{10140\\01103}{}$,
$\subspace{10101\\01224}{}$,
$\subspace{10120\\01204}{}$,
$\subspace{10133\\01242}{}$,
$\subspace{10140\\01202}{}$,
$\subspace{10120\\01334}{}$,
$\subspace{10122\\01332}{}$,
$\subspace{10133\\01404}{}$,
$\subspace{10133\\01434}{}$,
$\subspace{10200\\01043}{}$,
$\subspace{10204\\01041}{}$,
$\subspace{10220\\01013}{}$,
$\subspace{10200\\01212}{}$,
$\subspace{10211\\01223}{}$,
$\subspace{10240\\01202}{}$,
$\subspace{10242\\01211}{}$,
$\subspace{10242\\01212}{}$,
$\subspace{10240\\01341}{}$,
$\subspace{10242\\01340}{}$,
$\subspace{10220\\01404}{}$,
$\subspace{10230\\01444}{}$,
$\subspace{10304\\01043}{}$,
$\subspace{10311\\01043}{}$,
$\subspace{10302\\01234}{}$,
$\subspace{10304\\01242}{}$,
$\subspace{10311\\01204}{}$,
$\subspace{10324\\01212}{}$,
$\subspace{10333\\01234}{}$,
$\subspace{10302\\01322}{}$,
$\subspace{10304\\01341}{}$,
$\subspace{10313\\01332}{}$,
$\subspace{10324\\01332}{}$,
$\subspace{10302\\01431}{}$,
$\subspace{10311\\01434}{}$,
$\subspace{10342\\01401}{}$,
$\subspace{10342\\01440}{}$,
$\subspace{10411\\01023}{}$,
$\subspace{10412\\01030}{}$,
$\subspace{10423\\01030}{}$,
$\subspace{10432\\01002}{}$,
$\subspace{10432\\01030}{}$,
$\subspace{10431\\01223}{}$,
$\subspace{10432\\01224}{}$,
$\subspace{10411\\01404}{}$,
$\subspace{10431\\01444}{}$,
$\subspace{12013\\00144}{}$,
$\subspace{12023\\00104}{}$,
$\subspace{12040\\00104}{}$.

\smallskip

$n_5(5,2;16)\ge 391$,$\left[\!\begin{smallmatrix}10000\\00100\\04410\\00001\\00044\end{smallmatrix}\!\right]$:
$\subspace{00010\\00001}{1}$,
$\subspace{01001\\00120}{15}$,
$\subspace{01001\\00142}{}$,
$\subspace{01004\\00144}{}$,
$\subspace{01010\\00104}{}$,
$\subspace{01013\\00101}{}$,
$\subspace{01010\\00132}{}$,
$\subspace{01022\\00113}{}$,
$\subspace{01024\\00111}{}$,
$\subspace{01024\\00144}{}$,
$\subspace{01033\\00101}{}$,
$\subspace{01031\\00120}{}$,
$\subspace{01033\\00123}{}$,
$\subspace{01042\\00113}{}$,
$\subspace{01040\\00132}{}$,
$\subspace{01042\\00130}{}$,
$\subspace{10000\\00012}{3}$,
$\subspace{10000\\00013}{}$,
$\subspace{10000\\00014}{}$,
$\subspace{10001\\00012}{3}$,
$\subspace{10001\\00014}{}$,
$\subspace{10002\\00013}{}$,
$\subspace{10001\\00013}{3}$,
$\subspace{10003\\00012}{}$,
$\subspace{10003\\00014}{}$,
$\subspace{10002\\00012}{3}$,
$\subspace{10002\\00014}{}$,
$\subspace{10004\\00013}{}$,
$\subspace{10003\\00013}{3}$,
$\subspace{10004\\00012}{}$,
$\subspace{10004\\00014}{}$,
$\subspace{10101\\01000}{15}$,
$\subspace{10113\\01014}{}$,
$\subspace{10120\\01023}{}$,
$\subspace{10132\\01032}{}$,
$\subspace{10144\\01041}{}$,
$\subspace{10404\\01111}{}$,
$\subspace{10411\\01140}{}$,
$\subspace{10423\\01124}{}$,
$\subspace{10430\\01103}{}$,
$\subspace{10442\\01132}{}$,
$\subspace{14004\\00100}{}$,
$\subspace{14013\\00143}{}$,
$\subspace{14022\\00131}{}$,
$\subspace{14031\\00124}{}$,
$\subspace{14040\\00112}{}$,
$\subspace{10100\\01011}{15}$,
$\subspace{10112\\01020}{}$,
$\subspace{10124\\01034}{}$,
$\subspace{10131\\01043}{}$,
$\subspace{10143\\01002}{}$,
$\subspace{10400\\01141}{}$,
$\subspace{10412\\01120}{}$,
$\subspace{10424\\01104}{}$,
$\subspace{10431\\01133}{}$,
$\subspace{10443\\01112}{}$,
$\subspace{14000\\00140}{}$,
$\subspace{14014\\00133}{}$,
$\subspace{14023\\00121}{}$,
$\subspace{14032\\00114}{}$,
$\subspace{14041\\00102}{}$,
$\subspace{10103\\01020}{15}$,
$\subspace{10110\\01034}{}$,
$\subspace{10122\\01043}{}$,
$\subspace{10134\\01002}{}$,
$\subspace{10141\\01011}{}$,
$\subspace{10402\\01133}{}$,
$\subspace{10414\\01112}{}$,
$\subspace{10421\\01141}{}$,
$\subspace{10433\\01120}{}$,
$\subspace{10440\\01104}{}$,
$\subspace{14002\\00133}{}$,
$\subspace{14011\\00121}{}$,
$\subspace{14020\\00114}{}$,
$\subspace{14034\\00102}{}$,
$\subspace{14043\\00140}{}$,
$\subspace{10102\\01124}{15}$,
$\subspace{10114\\01140}{}$,
$\subspace{10121\\01111}{}$,
$\subspace{10133\\01132}{}$,
$\subspace{10140\\01103}{}$,
$\subspace{10401\\01000}{}$,
$\subspace{10413\\01041}{}$,
$\subspace{10420\\01032}{}$,
$\subspace{10432\\01023}{}$,
$\subspace{10444\\01014}{}$,
$\subspace{11002\\00124}{}$,
$\subspace{11011\\00131}{}$,
$\subspace{11020\\00143}{}$,
$\subspace{11034\\00100}{}$,
$\subspace{11043\\00112}{}$,
$\subspace{10103\\01143}{15}$,
$\subspace{10110\\01114}{}$,
$\subspace{10122\\01130}{}$,
$\subspace{10134\\01101}{}$,
$\subspace{10141\\01122}{}$,
$\subspace{10401\\01024}{}$,
$\subspace{10413\\01010}{}$,
$\subspace{10420\\01001}{}$,
$\subspace{10432\\01042}{}$,
$\subspace{10444\\01033}{}$,
$\subspace{11003\\00144}{}$,
$\subspace{11012\\00101}{}$,
$\subspace{11021\\00113}{}$,
$\subspace{11030\\00120}{}$,
$\subspace{11044\\00132}{}$,
$\subspace{10104\\01144}{15}$,
$\subspace{10111\\01110}{}$,
$\subspace{10123\\01131}{}$,
$\subspace{10130\\01102}{}$,
$\subspace{10142\\01123}{}$,
$\subspace{10401\\01030}{}$,
$\subspace{10413\\01021}{}$,
$\subspace{10420\\01012}{}$,
$\subspace{10432\\01003}{}$,
$\subspace{10444\\01044}{}$,
$\subspace{11004\\00122}{}$,
$\subspace{11013\\00134}{}$,
$\subspace{11022\\00141}{}$,
$\subspace{11031\\00103}{}$,
$\subspace{11040\\00110}{}$,
$\subspace{10104\\01203}{15}$,
$\subspace{10111\\01231}{}$,
$\subspace{10123\\01214}{}$,
$\subspace{10130\\01242}{}$,
$\subspace{10142\\01220}{}$,
$\subspace{10102\\01420}{}$,
$\subspace{10114\\01422}{}$,
$\subspace{10121\\01424}{}$,
$\subspace{10133\\01421}{}$,
$\subspace{10140\\01423}{}$,
$\subspace{10200\\01334}{}$,
$\subspace{10212\\01304}{}$,
$\subspace{10224\\01324}{}$,
$\subspace{10231\\01344}{}$,
$\subspace{10243\\01314}{}$,
$\subspace{10100\\01213}{15}$,
$\subspace{10112\\01241}{}$,
$\subspace{10124\\01224}{}$,
$\subspace{10131\\01202}{}$,
$\subspace{10143\\01230}{}$,
$\subspace{10104\\01410}{}$,
$\subspace{10111\\01412}{}$,
$\subspace{10123\\01414}{}$,
$\subspace{10130\\01411}{}$,
$\subspace{10142\\01413}{}$,
$\subspace{10203\\01301}{}$,
$\subspace{10210\\01321}{}$,
$\subspace{10222\\01341}{}$,
$\subspace{10234\\01311}{}$,
$\subspace{10241\\01331}{}$,
$\subspace{10101\\01223}{15}$,
$\subspace{10113\\01201}{}$,
$\subspace{10120\\01234}{}$,
$\subspace{10132\\01212}{}$,
$\subspace{10144\\01240}{}$,
$\subspace{10101\\01400}{}$,
$\subspace{10113\\01402}{}$,
$\subspace{10120\\01404}{}$,
$\subspace{10132\\01401}{}$,
$\subspace{10144\\01403}{}$,
$\subspace{10201\\01323}{}$,
$\subspace{10213\\01343}{}$,
$\subspace{10220\\01313}{}$,
$\subspace{10232\\01333}{}$,
$\subspace{10244\\01303}{}$,
$\subspace{10100\\01300}{15}$,
$\subspace{10112\\01340}{}$,
$\subspace{10124\\01330}{}$,
$\subspace{10131\\01320}{}$,
$\subspace{10143\\01310}{}$,
$\subspace{10302\\01244}{}$,
$\subspace{10314\\01200}{}$,
$\subspace{10321\\01211}{}$,
$\subspace{10333\\01222}{}$,
$\subspace{10340\\01233}{}$,
$\subspace{10300\\01440}{}$,
$\subspace{10312\\01444}{}$,
$\subspace{10324\\01443}{}$,
$\subspace{10331\\01442}{}$,
$\subspace{10343\\01441}{}$,
$\subspace{10103\\01302}{15}$,
$\subspace{10110\\01342}{}$,
$\subspace{10122\\01332}{}$,
$\subspace{10134\\01322}{}$,
$\subspace{10141\\01312}{}$,
$\subspace{10303\\01232}{}$,
$\subspace{10310\\01243}{}$,
$\subspace{10322\\01204}{}$,
$\subspace{10334\\01210}{}$,
$\subspace{10341\\01221}{}$,
$\subspace{10304\\01430}{}$,
$\subspace{10311\\01434}{}$,
$\subspace{10323\\01433}{}$,
$\subspace{10330\\01432}{}$,
$\subspace{10342\\01431}{}$,
$\subspace{10102\\01310}{15}$,
$\subspace{10114\\01300}{}$,
$\subspace{10121\\01340}{}$,
$\subspace{10133\\01330}{}$,
$\subspace{10140\\01320}{}$,
$\subspace{10300\\01211}{}$,
$\subspace{10312\\01222}{}$,
$\subspace{10324\\01233}{}$,
$\subspace{10331\\01244}{}$,
$\subspace{10343\\01200}{}$,
$\subspace{10300\\01441}{}$,
$\subspace{10312\\01440}{}$,
$\subspace{10324\\01444}{}$,
$\subspace{10331\\01443}{}$,
$\subspace{10343\\01442}{}$,
$\subspace{10204\\01012}{15}$,
$\subspace{10211\\01044}{}$,
$\subspace{10223\\01021}{}$,
$\subspace{10230\\01003}{}$,
$\subspace{10242\\01030}{}$,
$\subspace{10303\\01144}{}$,
$\subspace{10310\\01131}{}$,
$\subspace{10322\\01123}{}$,
$\subspace{10334\\01110}{}$,
$\subspace{10341\\01102}{}$,
$\subspace{13003\\00103}{}$,
$\subspace{13012\\00122}{}$,
$\subspace{13021\\00141}{}$,
$\subspace{13030\\00110}{}$,
$\subspace{13044\\00134}{}$,
$\subspace{10204\\01023}{15}$,
$\subspace{10211\\01000}{}$,
$\subspace{10223\\01032}{}$,
$\subspace{10230\\01014}{}$,
$\subspace{10242\\01041}{}$,
$\subspace{10301\\01140}{}$,
$\subspace{10313\\01132}{}$,
$\subspace{10320\\01124}{}$,
$\subspace{10332\\01111}{}$,
$\subspace{10344\\01103}{}$,
$\subspace{13001\\00131}{}$,
$\subspace{13010\\00100}{}$,
$\subspace{13024\\00124}{}$,
$\subspace{13033\\00143}{}$,
$\subspace{13042\\00112}{}$,
$\subspace{10204\\01044}{15}$,
$\subspace{10211\\01021}{}$,
$\subspace{10223\\01003}{}$,
$\subspace{10230\\01030}{}$,
$\subspace{10242\\01012}{}$,
$\subspace{10301\\01131}{}$,
$\subspace{10313\\01123}{}$,
$\subspace{10320\\01110}{}$,
$\subspace{10332\\01102}{}$,
$\subspace{10344\\01144}{}$,
$\subspace{13001\\00110}{}$,
$\subspace{13010\\00134}{}$,
$\subspace{13024\\00103}{}$,
$\subspace{13033\\00122}{}$,
$\subspace{13042\\00141}{}$,
$\subspace{10200\\01104}{15}$,
$\subspace{10212\\01112}{}$,
$\subspace{10224\\01120}{}$,
$\subspace{10231\\01133}{}$,
$\subspace{10243\\01141}{}$,
$\subspace{10302\\01020}{}$,
$\subspace{10314\\01043}{}$,
$\subspace{10321\\01011}{}$,
$\subspace{10333\\01034}{}$,
$\subspace{10340\\01002}{}$,
$\subspace{12000\\00121}{}$,
$\subspace{12014\\00102}{}$,
$\subspace{12023\\00133}{}$,
$\subspace{12032\\00114}{}$,
$\subspace{12041\\00140}{}$,
$\subspace{10203\\01134}{15}$,
$\subspace{10210\\01142}{}$,
$\subspace{10222\\01100}{}$,
$\subspace{10234\\01113}{}$,
$\subspace{10241\\01121}{}$,
$\subspace{10304\\01031}{}$,
$\subspace{10311\\01004}{}$,
$\subspace{10323\\01022}{}$,
$\subspace{10330\\01040}{}$,
$\subspace{10342\\01013}{}$,
$\subspace{12003\\00111}{}$,
$\subspace{12012\\00142}{}$,
$\subspace{12021\\00123}{}$,
$\subspace{12030\\00104}{}$,
$\subspace{12044\\00130}{}$,
$\subspace{10203\\01134}{15}$,
$\subspace{10210\\01142}{}$,
$\subspace{10222\\01100}{}$,
$\subspace{10234\\01113}{}$,
$\subspace{10241\\01121}{}$,
$\subspace{10304\\01031}{}$,
$\subspace{10311\\01004}{}$,
$\subspace{10323\\01022}{}$,
$\subspace{10330\\01040}{}$,
$\subspace{10342\\01013}{}$,
$\subspace{12003\\00111}{}$,
$\subspace{12012\\00142}{}$,
$\subspace{12021\\00123}{}$,
$\subspace{12030\\00104}{}$,
$\subspace{12044\\00130}{}$,
$\subspace{10201\\01213}{15}$,
$\subspace{10213\\01202}{}$,
$\subspace{10220\\01241}{}$,
$\subspace{10232\\01230}{}$,
$\subspace{10244\\01224}{}$,
$\subspace{10201\\01414}{}$,
$\subspace{10213\\01410}{}$,
$\subspace{10220\\01411}{}$,
$\subspace{10232\\01412}{}$,
$\subspace{10244\\01413}{}$,
$\subspace{10402\\01311}{}$,
$\subspace{10414\\01321}{}$,
$\subspace{10421\\01331}{}$,
$\subspace{10433\\01341}{}$,
$\subspace{10440\\01301}{}$,
$\subspace{10202\\01212}{15}$,
$\subspace{10214\\01201}{}$,
$\subspace{10221\\01240}{}$,
$\subspace{10233\\01234}{}$,
$\subspace{10240\\01223}{}$,
$\subspace{10200\\01401}{}$,
$\subspace{10212\\01402}{}$,
$\subspace{10224\\01403}{}$,
$\subspace{10231\\01404}{}$,
$\subspace{10243\\01400}{}$,
$\subspace{10404\\01303}{}$,
$\subspace{10411\\01313}{}$,
$\subspace{10423\\01323}{}$,
$\subspace{10430\\01333}{}$,
$\subspace{10442\\01343}{}$,
$\subspace{10202\\01244}{15}$,
$\subspace{10214\\01233}{}$,
$\subspace{10221\\01222}{}$,
$\subspace{10233\\01211}{}$,
$\subspace{10240\\01200}{}$,
$\subspace{10202\\01441}{}$,
$\subspace{10214\\01442}{}$,
$\subspace{10221\\01443}{}$,
$\subspace{10233\\01444}{}$,
$\subspace{10240\\01440}{}$,
$\subspace{10404\\01340}{}$,
$\subspace{10411\\01300}{}$,
$\subspace{10423\\01310}{}$,
$\subspace{10430\\01320}{}$,
$\subspace{10442\\01330}{}$,
$\subspace{10301\\01303}{15}$,
$\subspace{10313\\01333}{}$,
$\subspace{10320\\01313}{}$,
$\subspace{10332\\01343}{}$,
$\subspace{10344\\01323}{}$,
$\subspace{10403\\01201}{}$,
$\subspace{10410\\01223}{}$,
$\subspace{10422\\01240}{}$,
$\subspace{10434\\01212}{}$,
$\subspace{10441\\01234}{}$,
$\subspace{10403\\01400}{}$,
$\subspace{10410\\01403}{}$,
$\subspace{10422\\01401}{}$,
$\subspace{10434\\01404}{}$,
$\subspace{10441\\01402}{}$,
$\subspace{10302\\01334}{15}$,
$\subspace{10314\\01314}{}$,
$\subspace{10321\\01344}{}$,
$\subspace{10333\\01324}{}$,
$\subspace{10340\\01304}{}$,
$\subspace{10402\\01214}{}$,
$\subspace{10414\\01231}{}$,
$\subspace{10421\\01203}{}$,
$\subspace{10433\\01220}{}$,
$\subspace{10440\\01242}{}$,
$\subspace{10400\\01421}{}$,
$\subspace{10412\\01424}{}$,
$\subspace{10424\\01422}{}$,
$\subspace{10431\\01420}{}$,
$\subspace{10443\\01423}{}$,
$\subspace{10303\\01330}{15}$,
$\subspace{10310\\01310}{}$,
$\subspace{10322\\01340}{}$,
$\subspace{10334\\01320}{}$,
$\subspace{10341\\01300}{}$,
$\subspace{10403\\01200}{}$,
$\subspace{10410\\01222}{}$,
$\subspace{10422\\01244}{}$,
$\subspace{10434\\01211}{}$,
$\subspace{10441\\01233}{}$,
$\subspace{10400\\01443}{}$,
$\subspace{10412\\01441}{}$,
$\subspace{10424\\01444}{}$,
$\subspace{10431\\01442}{}$,
$\subspace{10443\\01440}{}$.

\smallskip

$n_5(5,2;17)\ge 412$,$\left[\!\begin{smallmatrix}10000\\00100\\04410\\00001\\00044\end{smallmatrix}\!\right]$:
$\subspace{00104\\00012}{3}$,
$\subspace{01004\\00014}{}$,
$\subspace{01100\\00013}{}$,
$\subspace{01001\\00113}{5}$,
$\subspace{01010\\00120}{}$,
$\subspace{01024\\00132}{}$,
$\subspace{01033\\00144}{}$,
$\subspace{01042\\00101}{}$,
$\subspace{01003\\00141}{5}$,
$\subspace{01012\\00103}{}$,
$\subspace{01021\\00110}{}$,
$\subspace{01030\\00122}{}$,
$\subspace{01044\\00134}{}$,
$\subspace{01203\\00011}{3}$,
$\subspace{01304\\00010}{}$,
$\subspace{01420\\00001}{}$,
$\subspace{10001\\00011}{3}$,
$\subspace{10004\\00010}{}$,
$\subspace{10010\\00001}{}$,
$\subspace{10003\\00013}{3}$,
$\subspace{10004\\00012}{}$,
$\subspace{10004\\00014}{}$,
$\subspace{10003\\01204}{15}$,
$\subspace{10003\\01210}{}$,
$\subspace{10003\\01221}{}$,
$\subspace{10003\\01232}{}$,
$\subspace{10003\\01243}{}$,
$\subspace{10022\\01302}{}$,
$\subspace{10022\\01312}{}$,
$\subspace{10022\\01322}{}$,
$\subspace{10022\\01332}{}$,
$\subspace{10022\\01342}{}$,
$\subspace{10030\\01430}{}$,
$\subspace{10030\\01431}{}$,
$\subspace{10030\\01432}{}$,
$\subspace{10030\\01433}{}$,
$\subspace{10030\\01434}{}$,
$\subspace{10033\\01204}{15}$,
$\subspace{10033\\01210}{}$,
$\subspace{10033\\01221}{}$,
$\subspace{10033\\01232}{}$,
$\subspace{10033\\01243}{}$,
$\subspace{10020\\01302}{}$,
$\subspace{10020\\01312}{}$,
$\subspace{10020\\01322}{}$,
$\subspace{10020\\01332}{}$,
$\subspace{10020\\01342}{}$,
$\subspace{10002\\01430}{}$,
$\subspace{10002\\01431}{}$,
$\subspace{10002\\01432}{}$,
$\subspace{10002\\01433}{}$,
$\subspace{10002\\01434}{}$,
$\subspace{10102\\01002}{15}$,
$\subspace{10114\\01011}{}$,
$\subspace{10121\\01020}{}$,
$\subspace{10133\\01034}{}$,
$\subspace{10140\\01043}{}$,
$\subspace{10403\\01104}{}$,
$\subspace{10410\\01133}{}$,
$\subspace{10422\\01112}{}$,
$\subspace{10434\\01141}{}$,
$\subspace{10441\\01120}{}$,
$\subspace{14003\\00102}{}$,
$\subspace{14012\\00140}{}$,
$\subspace{14021\\00133}{}$,
$\subspace{14030\\00121}{}$,
$\subspace{14044\\00114}{}$,
$\subspace{10100\\01014}{15}$,
$\subspace{10112\\01023}{}$,
$\subspace{10124\\01032}{}$,
$\subspace{10131\\01041}{}$,
$\subspace{10143\\01000}{}$,
$\subspace{10404\\01124}{}$,
$\subspace{10411\\01103}{}$,
$\subspace{10423\\01132}{}$,
$\subspace{10430\\01111}{}$,
$\subspace{10442\\01140}{}$,
$\subspace{14004\\00131}{}$,
$\subspace{14013\\00124}{}$,
$\subspace{14022\\00112}{}$,
$\subspace{14031\\00100}{}$,
$\subspace{14040\\00143}{}$,
$\subspace{10103\\01041}{15}$,
$\subspace{10110\\01000}{}$,
$\subspace{10122\\01014}{}$,
$\subspace{10134\\01023}{}$,
$\subspace{10141\\01032}{}$,
$\subspace{10402\\01124}{}$,
$\subspace{10414\\01103}{}$,
$\subspace{10421\\01132}{}$,
$\subspace{10433\\01111}{}$,
$\subspace{10440\\01140}{}$,
$\subspace{14002\\00112}{}$,
$\subspace{14011\\00100}{}$,
$\subspace{14020\\00143}{}$,
$\subspace{14034\\00131}{}$,
$\subspace{14043\\00124}{}$,
$\subspace{10101\\01113}{15}$,
$\subspace{10113\\01134}{}$,
$\subspace{10120\\01100}{}$,
$\subspace{10132\\01121}{}$,
$\subspace{10144\\01142}{}$,
$\subspace{10402\\01040}{}$,
$\subspace{10414\\01031}{}$,
$\subspace{10421\\01022}{}$,
$\subspace{10433\\01013}{}$,
$\subspace{10440\\01004}{}$,
$\subspace{11001\\00130}{}$,
$\subspace{11010\\00142}{}$,
$\subspace{11024\\00104}{}$,
$\subspace{11033\\00111}{}$,
$\subspace{11042\\00123}{}$,
$\subspace{10101\\01113}{15}$,
$\subspace{10113\\01134}{}$,
$\subspace{10120\\01100}{}$,
$\subspace{10132\\01121}{}$,
$\subspace{10144\\01142}{}$,
$\subspace{10402\\01040}{}$,
$\subspace{10414\\01031}{}$,
$\subspace{10421\\01022}{}$,
$\subspace{10433\\01013}{}$,
$\subspace{10440\\01004}{}$,
$\subspace{11001\\00130}{}$,
$\subspace{11010\\00142}{}$,
$\subspace{11024\\00104}{}$,
$\subspace{11033\\00111}{}$,
$\subspace{11042\\00123}{}$,
$\subspace{10104\\01144}{15}$,
$\subspace{10111\\01110}{}$,
$\subspace{10123\\01131}{}$,
$\subspace{10130\\01102}{}$,
$\subspace{10142\\01123}{}$,
$\subspace{10401\\01030}{}$,
$\subspace{10413\\01021}{}$,
$\subspace{10420\\01012}{}$,
$\subspace{10432\\01003}{}$,
$\subspace{10444\\01044}{}$,
$\subspace{11004\\00122}{}$,
$\subspace{11013\\00134}{}$,
$\subspace{11022\\00141}{}$,
$\subspace{11031\\00103}{}$,
$\subspace{11040\\00110}{}$,
$\subspace{10104\\01212}{15}$,
$\subspace{10111\\01240}{}$,
$\subspace{10123\\01223}{}$,
$\subspace{10130\\01201}{}$,
$\subspace{10142\\01234}{}$,
$\subspace{10104\\01400}{}$,
$\subspace{10111\\01402}{}$,
$\subspace{10123\\01404}{}$,
$\subspace{10130\\01401}{}$,
$\subspace{10142\\01403}{}$,
$\subspace{10200\\01313}{}$,
$\subspace{10212\\01333}{}$,
$\subspace{10224\\01303}{}$,
$\subspace{10231\\01323}{}$,
$\subspace{10243\\01343}{}$,
$\subspace{10103\\01221}{15}$,
$\subspace{10110\\01204}{}$,
$\subspace{10122\\01232}{}$,
$\subspace{10134\\01210}{}$,
$\subspace{10141\\01243}{}$,
$\subspace{10103\\01432}{}$,
$\subspace{10110\\01434}{}$,
$\subspace{10122\\01431}{}$,
$\subspace{10134\\01433}{}$,
$\subspace{10141\\01430}{}$,
$\subspace{10204\\01322}{}$,
$\subspace{10211\\01342}{}$,
$\subspace{10223\\01312}{}$,
$\subspace{10230\\01332}{}$,
$\subspace{10242\\01302}{}$,
$\subspace{10100\\01233}{15}$,
$\subspace{10112\\01211}{}$,
$\subspace{10124\\01244}{}$,
$\subspace{10131\\01222}{}$,
$\subspace{10143\\01200}{}$,
$\subspace{10102\\01444}{}$,
$\subspace{10114\\01441}{}$,
$\subspace{10121\\01443}{}$,
$\subspace{10133\\01440}{}$,
$\subspace{10140\\01442}{}$,
$\subspace{10202\\01300}{}$,
$\subspace{10214\\01320}{}$,
$\subspace{10221\\01340}{}$,
$\subspace{10233\\01310}{}$,
$\subspace{10240\\01330}{}$,
$\subspace{10100\\01320}{15}$,
$\subspace{10112\\01310}{}$,
$\subspace{10124\\01300}{}$,
$\subspace{10131\\01340}{}$,
$\subspace{10143\\01330}{}$,
$\subspace{10301\\01222}{}$,
$\subspace{10313\\01233}{}$,
$\subspace{10320\\01244}{}$,
$\subspace{10332\\01200}{}$,
$\subspace{10344\\01211}{}$,
$\subspace{10301\\01441}{}$,
$\subspace{10313\\01440}{}$,
$\subspace{10320\\01444}{}$,
$\subspace{10332\\01443}{}$,
$\subspace{10344\\01442}{}$,
$\subspace{10102\\01331}{15}$,
$\subspace{10114\\01321}{}$,
$\subspace{10121\\01311}{}$,
$\subspace{10133\\01301}{}$,
$\subspace{10140\\01341}{}$,
$\subspace{10300\\01202}{}$,
$\subspace{10312\\01213}{}$,
$\subspace{10324\\01224}{}$,
$\subspace{10331\\01230}{}$,
$\subspace{10343\\01241}{}$,
$\subspace{10304\\01411}{}$,
$\subspace{10311\\01410}{}$,
$\subspace{10323\\01414}{}$,
$\subspace{10330\\01413}{}$,
$\subspace{10342\\01412}{}$,
$\subspace{10101\\01340}{15}$,
$\subspace{10113\\01330}{}$,
$\subspace{10120\\01320}{}$,
$\subspace{10132\\01310}{}$,
$\subspace{10144\\01300}{}$,
$\subspace{10303\\01244}{}$,
$\subspace{10310\\01200}{}$,
$\subspace{10322\\01211}{}$,
$\subspace{10334\\01222}{}$,
$\subspace{10341\\01233}{}$,
$\subspace{10303\\01441}{}$,
$\subspace{10310\\01440}{}$,
$\subspace{10322\\01444}{}$,
$\subspace{10334\\01443}{}$,
$\subspace{10341\\01442}{}$,
$\subspace{10204\\01010}{15}$,
$\subspace{10211\\01042}{}$,
$\subspace{10223\\01024}{}$,
$\subspace{10230\\01001}{}$,
$\subspace{10242\\01033}{}$,
$\subspace{10301\\01122}{}$,
$\subspace{10313\\01114}{}$,
$\subspace{10320\\01101}{}$,
$\subspace{10332\\01143}{}$,
$\subspace{10344\\01130}{}$,
$\subspace{13001\\00144}{}$,
$\subspace{13010\\00113}{}$,
$\subspace{13024\\00132}{}$,
$\subspace{13033\\00101}{}$,
$\subspace{13042\\00120}{}$,
$\subspace{10204\\01012}{15}$,
$\subspace{10211\\01044}{}$,
$\subspace{10223\\01021}{}$,
$\subspace{10230\\01003}{}$,
$\subspace{10242\\01030}{}$,
$\subspace{10303\\01144}{}$,
$\subspace{10310\\01131}{}$,
$\subspace{10322\\01123}{}$,
$\subspace{10334\\01110}{}$,
$\subspace{10341\\01102}{}$,
$\subspace{13003\\00103}{}$,
$\subspace{13012\\00122}{}$,
$\subspace{13021\\00141}{}$,
$\subspace{13030\\00110}{}$,
$\subspace{13044\\00134}{}$,
$\subspace{10203\\01043}{15}$,
$\subspace{10210\\01020}{}$,
$\subspace{10222\\01002}{}$,
$\subspace{10234\\01034}{}$,
$\subspace{10241\\01011}{}$,
$\subspace{10302\\01112}{}$,
$\subspace{10314\\01104}{}$,
$\subspace{10321\\01141}{}$,
$\subspace{10333\\01133}{}$,
$\subspace{10340\\01120}{}$,
$\subspace{13002\\00114}{}$,
$\subspace{13011\\00133}{}$,
$\subspace{13020\\00102}{}$,
$\subspace{13034\\00121}{}$,
$\subspace{13043\\00140}{}$,
$\subspace{10201\\01124}{15}$,
$\subspace{10213\\01132}{}$,
$\subspace{10220\\01140}{}$,
$\subspace{10232\\01103}{}$,
$\subspace{10244\\01111}{}$,
$\subspace{10304\\01041}{}$,
$\subspace{10311\\01014}{}$,
$\subspace{10323\\01032}{}$,
$\subspace{10330\\01000}{}$,
$\subspace{10342\\01023}{}$,
$\subspace{12001\\00112}{}$,
$\subspace{12010\\00143}{}$,
$\subspace{12024\\00124}{}$,
$\subspace{12033\\00100}{}$,
$\subspace{12042\\00131}{}$,
$\subspace{10201\\01133}{15}$,
$\subspace{10213\\01141}{}$,
$\subspace{10220\\01104}{}$,
$\subspace{10232\\01112}{}$,
$\subspace{10244\\01120}{}$,
$\subspace{10304\\01020}{}$,
$\subspace{10311\\01043}{}$,
$\subspace{10323\\01011}{}$,
$\subspace{10330\\01034}{}$,
$\subspace{10342\\01002}{}$,
$\subspace{12001\\00133}{}$,
$\subspace{12010\\00114}{}$,
$\subspace{12024\\00140}{}$,
$\subspace{12033\\00121}{}$,
$\subspace{12042\\00102}{}$,
$\subspace{10203\\01130}{15}$,
$\subspace{10210\\01143}{}$,
$\subspace{10222\\01101}{}$,
$\subspace{10234\\01114}{}$,
$\subspace{10241\\01122}{}$,
$\subspace{10302\\01001}{}$,
$\subspace{10314\\01024}{}$,
$\subspace{10321\\01042}{}$,
$\subspace{10333\\01010}{}$,
$\subspace{10340\\01033}{}$,
$\subspace{12003\\00101}{}$,
$\subspace{12012\\00132}{}$,
$\subspace{12021\\00113}{}$,
$\subspace{12030\\00144}{}$,
$\subspace{12044\\00120}{}$,
$\subspace{10200\\01201}{15}$,
$\subspace{10212\\01240}{}$,
$\subspace{10224\\01234}{}$,
$\subspace{10231\\01223}{}$,
$\subspace{10243\\01212}{}$,
$\subspace{10202\\01404}{}$,
$\subspace{10214\\01400}{}$,
$\subspace{10221\\01401}{}$,
$\subspace{10233\\01402}{}$,
$\subspace{10240\\01403}{}$,
$\subspace{10404\\01313}{}$,
$\subspace{10411\\01323}{}$,
$\subspace{10423\\01333}{}$,
$\subspace{10430\\01343}{}$,
$\subspace{10442\\01303}{}$,
$\subspace{10200\\01201}{15}$,
$\subspace{10212\\01240}{}$,
$\subspace{10224\\01234}{}$,
$\subspace{10231\\01223}{}$,
$\subspace{10243\\01212}{}$,
$\subspace{10202\\01404}{}$,
$\subspace{10214\\01400}{}$,
$\subspace{10221\\01401}{}$,
$\subspace{10233\\01402}{}$,
$\subspace{10240\\01403}{}$,
$\subspace{10404\\01313}{}$,
$\subspace{10411\\01323}{}$,
$\subspace{10423\\01333}{}$,
$\subspace{10430\\01343}{}$,
$\subspace{10442\\01303}{}$,
$\subspace{10201\\01203}{15}$,
$\subspace{10213\\01242}{}$,
$\subspace{10220\\01231}{}$,
$\subspace{10232\\01220}{}$,
$\subspace{10244\\01214}{}$,
$\subspace{10203\\01422}{}$,
$\subspace{10210\\01423}{}$,
$\subspace{10222\\01424}{}$,
$\subspace{10234\\01420}{}$,
$\subspace{10241\\01421}{}$,
$\subspace{10403\\01314}{}$,
$\subspace{10410\\01324}{}$,
$\subspace{10422\\01334}{}$,
$\subspace{10434\\01344}{}$,
$\subspace{10441\\01304}{}$,
$\subspace{10302\\01301}{15}$,
$\subspace{10314\\01331}{}$,
$\subspace{10321\\01311}{}$,
$\subspace{10333\\01341}{}$,
$\subspace{10340\\01321}{}$,
$\subspace{10400\\01213}{}$,
$\subspace{10412\\01230}{}$,
$\subspace{10424\\01202}{}$,
$\subspace{10431\\01224}{}$,
$\subspace{10443\\01241}{}$,
$\subspace{10401\\01410}{}$,
$\subspace{10413\\01413}{}$,
$\subspace{10420\\01411}{}$,
$\subspace{10432\\01414}{}$,
$\subspace{10444\\01412}{}$,
$\subspace{10300\\01332}{15}$,
$\subspace{10312\\01312}{}$,
$\subspace{10324\\01342}{}$,
$\subspace{10331\\01322}{}$,
$\subspace{10343\\01302}{}$,
$\subspace{10400\\01232}{}$,
$\subspace{10412\\01204}{}$,
$\subspace{10424\\01221}{}$,
$\subspace{10431\\01243}{}$,
$\subspace{10443\\01210}{}$,
$\subspace{10400\\01432}{}$,
$\subspace{10412\\01430}{}$,
$\subspace{10424\\01433}{}$,
$\subspace{10431\\01431}{}$,
$\subspace{10443\\01434}{}$,
$\subspace{10300\\01334}{15}$,
$\subspace{10312\\01314}{}$,
$\subspace{10324\\01344}{}$,
$\subspace{10331\\01324}{}$,
$\subspace{10343\\01304}{}$,
$\subspace{10401\\01203}{}$,
$\subspace{10413\\01220}{}$,
$\subspace{10420\\01242}{}$,
$\subspace{10432\\01214}{}$,
$\subspace{10444\\01231}{}$,
$\subspace{10403\\01420}{}$,
$\subspace{10410\\01423}{}$,
$\subspace{10422\\01421}{}$,
$\subspace{10434\\01424}{}$,
$\subspace{10441\\01422}{}$.

\smallskip

$n_5(5,2;21)\ge 521$,$\left[\!\begin{smallmatrix}10000\\00100\\04410\\00001\\00044\end{smallmatrix}\!\right]$:
$\subspace{01002\\00102}{5}$,
$\subspace{01011\\00114}{}$,
$\subspace{01020\\00121}{}$,
$\subspace{01034\\00133}{}$,
$\subspace{01043\\00140}{}$,
$\subspace{01002\\00102}{5}$,
$\subspace{01011\\00114}{}$,
$\subspace{01020\\00121}{}$,
$\subspace{01034\\00133}{}$,
$\subspace{01043\\00140}{}$,
$\subspace{01000\\00124}{5}$,
$\subspace{01014\\00131}{}$,
$\subspace{01023\\00143}{}$,
$\subspace{01032\\00100}{}$,
$\subspace{01041\\00112}{}$,
$\subspace{01004\\00130}{5}$,
$\subspace{01013\\00142}{}$,
$\subspace{01022\\00104}{}$,
$\subspace{01031\\00111}{}$,
$\subspace{01040\\00123}{}$,
$\subspace{10000\\00001}{3}$,
$\subspace{10000\\00010}{}$,
$\subspace{10000\\00011}{}$,
$\subspace{10000\\00012}{3}$,
$\subspace{10000\\00013}{}$,
$\subspace{10000\\00014}{}$,
$\subspace{10001\\00010}{3}$,
$\subspace{10004\\00011}{}$,
$\subspace{10040\\00001}{}$,
$\subspace{10001\\00011}{3}$,
$\subspace{10004\\00010}{}$,
$\subspace{10010\\00001}{}$,
$\subspace{10001\\00012}{3}$,
$\subspace{10001\\00014}{}$,
$\subspace{10002\\00013}{}$,
$\subspace{10001\\00013}{3}$,
$\subspace{10003\\00012}{}$,
$\subspace{10003\\00014}{}$,
$\subspace{10003\\00013}{3}$,
$\subspace{10004\\00012}{}$,
$\subspace{10004\\00014}{}$,
$\subspace{10100\\01003}{15}$,
$\subspace{10112\\01012}{}$,
$\subspace{10124\\01021}{}$,
$\subspace{10131\\01030}{}$,
$\subspace{10143\\01044}{}$,
$\subspace{10400\\01123}{}$,
$\subspace{10412\\01102}{}$,
$\subspace{10424\\01131}{}$,
$\subspace{10431\\01110}{}$,
$\subspace{10443\\01144}{}$,
$\subspace{14000\\00103}{}$,
$\subspace{14014\\00141}{}$,
$\subspace{14023\\00134}{}$,
$\subspace{14032\\00122}{}$,
$\subspace{14041\\00110}{}$,
$\subspace{10101\\01021}{15}$,
$\subspace{10113\\01030}{}$,
$\subspace{10120\\01044}{}$,
$\subspace{10132\\01003}{}$,
$\subspace{10144\\01012}{}$,
$\subspace{10404\\01102}{}$,
$\subspace{10411\\01131}{}$,
$\subspace{10423\\01110}{}$,
$\subspace{10430\\01144}{}$,
$\subspace{10442\\01123}{}$,
$\subspace{14004\\00134}{}$,
$\subspace{14013\\00122}{}$,
$\subspace{14022\\00110}{}$,
$\subspace{14031\\00103}{}$,
$\subspace{14040\\00141}{}$,
$\subspace{10100\\01040}{15}$,
$\subspace{10112\\01004}{}$,
$\subspace{10124\\01013}{}$,
$\subspace{10131\\01022}{}$,
$\subspace{10143\\01031}{}$,
$\subspace{10400\\01100}{}$,
$\subspace{10412\\01134}{}$,
$\subspace{10424\\01113}{}$,
$\subspace{10431\\01142}{}$,
$\subspace{10443\\01121}{}$,
$\subspace{14000\\00111}{}$,
$\subspace{14014\\00104}{}$,
$\subspace{14023\\00142}{}$,
$\subspace{14032\\00130}{}$,
$\subspace{14041\\00123}{}$,
$\subspace{10100\\01040}{15}$,
$\subspace{10112\\01004}{}$,
$\subspace{10124\\01013}{}$,
$\subspace{10131\\01022}{}$,
$\subspace{10143\\01031}{}$,
$\subspace{10400\\01100}{}$,
$\subspace{10412\\01134}{}$,
$\subspace{10424\\01113}{}$,
$\subspace{10431\\01142}{}$,
$\subspace{10443\\01121}{}$,
$\subspace{14000\\00111}{}$,
$\subspace{14014\\00104}{}$,
$\subspace{14023\\00142}{}$,
$\subspace{14032\\00130}{}$,
$\subspace{14041\\00123}{}$,
$\subspace{10102\\01103}{15}$,
$\subspace{10114\\01124}{}$,
$\subspace{10121\\01140}{}$,
$\subspace{10133\\01111}{}$,
$\subspace{10140\\01132}{}$,
$\subspace{10404\\01032}{}$,
$\subspace{10411\\01023}{}$,
$\subspace{10423\\01014}{}$,
$\subspace{10430\\01000}{}$,
$\subspace{10442\\01041}{}$,
$\subspace{11002\\00112}{}$,
$\subspace{11011\\00124}{}$,
$\subspace{11020\\00131}{}$,
$\subspace{11034\\00143}{}$,
$\subspace{11043\\00100}{}$,
$\subspace{10103\\01110}{15}$,
$\subspace{10110\\01131}{}$,
$\subspace{10122\\01102}{}$,
$\subspace{10134\\01123}{}$,
$\subspace{10141\\01144}{}$,
$\subspace{10402\\01012}{}$,
$\subspace{10414\\01003}{}$,
$\subspace{10421\\01044}{}$,
$\subspace{10433\\01030}{}$,
$\subspace{10440\\01021}{}$,
$\subspace{11003\\00141}{}$,
$\subspace{11012\\00103}{}$,
$\subspace{11021\\00110}{}$,
$\subspace{11030\\00122}{}$,
$\subspace{11044\\00134}{}$,
$\subspace{10102\\01122}{15}$,
$\subspace{10114\\01143}{}$,
$\subspace{10121\\01114}{}$,
$\subspace{10133\\01130}{}$,
$\subspace{10140\\01101}{}$,
$\subspace{10403\\01010}{}$,
$\subspace{10410\\01001}{}$,
$\subspace{10422\\01042}{}$,
$\subspace{10434\\01033}{}$,
$\subspace{10441\\01024}{}$,
$\subspace{11002\\00144}{}$,
$\subspace{11011\\00101}{}$,
$\subspace{11020\\00113}{}$,
$\subspace{11034\\00120}{}$,
$\subspace{11043\\00132}{}$,
$\subspace{10101\\01143}{15}$,
$\subspace{10113\\01114}{}$,
$\subspace{10120\\01130}{}$,
$\subspace{10132\\01101}{}$,
$\subspace{10144\\01122}{}$,
$\subspace{10404\\01042}{}$,
$\subspace{10411\\01033}{}$,
$\subspace{10423\\01024}{}$,
$\subspace{10430\\01010}{}$,
$\subspace{10442\\01001}{}$,
$\subspace{11001\\00113}{}$,
$\subspace{11010\\00120}{}$,
$\subspace{11024\\00132}{}$,
$\subspace{11033\\00144}{}$,
$\subspace{11042\\00101}{}$,
$\subspace{10103\\01200}{15}$,
$\subspace{10110\\01233}{}$,
$\subspace{10122\\01211}{}$,
$\subspace{10134\\01244}{}$,
$\subspace{10141\\01222}{}$,
$\subspace{10103\\01441}{}$,
$\subspace{10110\\01443}{}$,
$\subspace{10122\\01440}{}$,
$\subspace{10134\\01442}{}$,
$\subspace{10141\\01444}{}$,
$\subspace{10203\\01300}{}$,
$\subspace{10210\\01320}{}$,
$\subspace{10222\\01340}{}$,
$\subspace{10234\\01310}{}$,
$\subspace{10241\\01330}{}$,
$\subspace{10104\\01220}{15}$,
$\subspace{10111\\01203}{}$,
$\subspace{10123\\01231}{}$,
$\subspace{10130\\01214}{}$,
$\subspace{10142\\01242}{}$,
$\subspace{10104\\01423}{}$,
$\subspace{10111\\01420}{}$,
$\subspace{10123\\01422}{}$,
$\subspace{10130\\01424}{}$,
$\subspace{10142\\01421}{}$,
$\subspace{10203\\01324}{}$,
$\subspace{10210\\01344}{}$,
$\subspace{10222\\01314}{}$,
$\subspace{10234\\01334}{}$,
$\subspace{10241\\01304}{}$,
$\subspace{10101\\01231}{15}$,
$\subspace{10113\\01214}{}$,
$\subspace{10120\\01242}{}$,
$\subspace{10132\\01220}{}$,
$\subspace{10144\\01203}{}$,
$\subspace{10101\\01423}{}$,
$\subspace{10113\\01420}{}$,
$\subspace{10120\\01422}{}$,
$\subspace{10132\\01424}{}$,
$\subspace{10144\\01421}{}$,
$\subspace{10204\\01334}{}$,
$\subspace{10211\\01304}{}$,
$\subspace{10223\\01324}{}$,
$\subspace{10230\\01344}{}$,
$\subspace{10242\\01314}{}$,
$\subspace{10104\\01240}{15}$,
$\subspace{10111\\01223}{}$,
$\subspace{10123\\01201}{}$,
$\subspace{10130\\01234}{}$,
$\subspace{10142\\01212}{}$,
$\subspace{10102\\01402}{}$,
$\subspace{10114\\01404}{}$,
$\subspace{10121\\01401}{}$,
$\subspace{10133\\01403}{}$,
$\subspace{10140\\01400}{}$,
$\subspace{10202\\01323}{}$,
$\subspace{10214\\01343}{}$,
$\subspace{10221\\01313}{}$,
$\subspace{10233\\01333}{}$,
$\subspace{10240\\01303}{}$,
$\subspace{10100\\01301}{15}$,
$\subspace{10112\\01341}{}$,
$\subspace{10124\\01331}{}$,
$\subspace{10131\\01321}{}$,
$\subspace{10143\\01311}{}$,
$\subspace{10303\\01202}{}$,
$\subspace{10310\\01213}{}$,
$\subspace{10322\\01224}{}$,
$\subspace{10334\\01230}{}$,
$\subspace{10341\\01241}{}$,
$\subspace{10303\\01414}{}$,
$\subspace{10310\\01413}{}$,
$\subspace{10322\\01412}{}$,
$\subspace{10334\\01411}{}$,
$\subspace{10341\\01410}{}$,
$\subspace{10102\\01304}{15}$,
$\subspace{10114\\01344}{}$,
$\subspace{10121\\01334}{}$,
$\subspace{10133\\01324}{}$,
$\subspace{10140\\01314}{}$,
$\subspace{10302\\01214}{}$,
$\subspace{10314\\01220}{}$,
$\subspace{10321\\01231}{}$,
$\subspace{10333\\01242}{}$,
$\subspace{10340\\01203}{}$,
$\subspace{10304\\01424}{}$,
$\subspace{10311\\01423}{}$,
$\subspace{10323\\01422}{}$,
$\subspace{10330\\01421}{}$,
$\subspace{10342\\01420}{}$,
$\subspace{10103\\01330}{15}$,
$\subspace{10110\\01320}{}$,
$\subspace{10122\\01310}{}$,
$\subspace{10134\\01300}{}$,
$\subspace{10141\\01340}{}$,
$\subspace{10302\\01233}{}$,
$\subspace{10314\\01244}{}$,
$\subspace{10321\\01200}{}$,
$\subspace{10333\\01211}{}$,
$\subspace{10340\\01222}{}$,
$\subspace{10302\\01441}{}$,
$\subspace{10314\\01440}{}$,
$\subspace{10321\\01444}{}$,
$\subspace{10333\\01443}{}$,
$\subspace{10340\\01442}{}$,
$\subspace{10104\\01331}{15}$,
$\subspace{10111\\01321}{}$,
$\subspace{10123\\01311}{}$,
$\subspace{10130\\01301}{}$,
$\subspace{10142\\01341}{}$,
$\subspace{10301\\01230}{}$,
$\subspace{10313\\01241}{}$,
$\subspace{10320\\01202}{}$,
$\subspace{10332\\01213}{}$,
$\subspace{10344\\01224}{}$,
$\subspace{10301\\01414}{}$,
$\subspace{10313\\01413}{}$,
$\subspace{10320\\01412}{}$,
$\subspace{10332\\01411}{}$,
$\subspace{10344\\01410}{}$,
$\subspace{10202\\01023}{15}$,
$\subspace{10214\\01000}{}$,
$\subspace{10221\\01032}{}$,
$\subspace{10233\\01014}{}$,
$\subspace{10240\\01041}{}$,
$\subspace{10300\\01124}{}$,
$\subspace{10312\\01111}{}$,
$\subspace{10324\\01103}{}$,
$\subspace{10331\\01140}{}$,
$\subspace{10343\\01132}{}$,
$\subspace{13000\\00143}{}$,
$\subspace{13014\\00112}{}$,
$\subspace{13023\\00131}{}$,
$\subspace{13032\\00100}{}$,
$\subspace{13041\\00124}{}$,
$\subspace{10202\\01023}{15}$,
$\subspace{10214\\01000}{}$,
$\subspace{10221\\01032}{}$,
$\subspace{10233\\01014}{}$,
$\subspace{10240\\01041}{}$,
$\subspace{10300\\01124}{}$,
$\subspace{10312\\01111}{}$,
$\subspace{10324\\01103}{}$,
$\subspace{10331\\01140}{}$,
$\subspace{10343\\01132}{}$,
$\subspace{13000\\00143}{}$,
$\subspace{13014\\00112}{}$,
$\subspace{13023\\00131}{}$,
$\subspace{13032\\00100}{}$,
$\subspace{13041\\00124}{}$,
$\subspace{10201\\01033}{15}$,
$\subspace{10213\\01010}{}$,
$\subspace{10220\\01042}{}$,
$\subspace{10232\\01024}{}$,
$\subspace{10244\\01001}{}$,
$\subspace{10304\\01101}{}$,
$\subspace{10311\\01143}{}$,
$\subspace{10323\\01130}{}$,
$\subspace{10330\\01122}{}$,
$\subspace{10342\\01114}{}$,
$\subspace{13004\\00120}{}$,
$\subspace{13013\\00144}{}$,
$\subspace{13022\\00113}{}$,
$\subspace{13031\\00132}{}$,
$\subspace{13040\\00101}{}$,
$\subspace{10200\\01040}{15}$,
$\subspace{10212\\01022}{}$,
$\subspace{10224\\01004}{}$,
$\subspace{10231\\01031}{}$,
$\subspace{10243\\01013}{}$,
$\subspace{10300\\01100}{}$,
$\subspace{10312\\01142}{}$,
$\subspace{10324\\01134}{}$,
$\subspace{10331\\01121}{}$,
$\subspace{10343\\01113}{}$,
$\subspace{13000\\00111}{}$,
$\subspace{13014\\00130}{}$,
$\subspace{13023\\00104}{}$,
$\subspace{13032\\00123}{}$,
$\subspace{13041\\00142}{}$,
$\subspace{10200\\01102}{15}$,
$\subspace{10212\\01110}{}$,
$\subspace{10224\\01123}{}$,
$\subspace{10231\\01131}{}$,
$\subspace{10243\\01144}{}$,
$\subspace{10301\\01030}{}$,
$\subspace{10313\\01003}{}$,
$\subspace{10320\\01021}{}$,
$\subspace{10332\\01044}{}$,
$\subspace{10344\\01012}{}$,
$\subspace{12000\\00141}{}$,
$\subspace{12014\\00122}{}$,
$\subspace{12023\\00103}{}$,
$\subspace{12032\\00134}{}$,
$\subspace{12041\\00110}{}$,
$\subspace{10201\\01112}{15}$,
$\subspace{10213\\01120}{}$,
$\subspace{10220\\01133}{}$,
$\subspace{10232\\01141}{}$,
$\subspace{10244\\01104}{}$,
$\subspace{10300\\01002}{}$,
$\subspace{10312\\01020}{}$,
$\subspace{10324\\01043}{}$,
$\subspace{10331\\01011}{}$,
$\subspace{10343\\01034}{}$,
$\subspace{12001\\00121}{}$,
$\subspace{12010\\00102}{}$,
$\subspace{12024\\00133}{}$,
$\subspace{12033\\00114}{}$,
$\subspace{12042\\00140}{}$,
$\subspace{10203\\01112}{15}$,
$\subspace{10210\\01120}{}$,
$\subspace{10222\\01133}{}$,
$\subspace{10234\\01141}{}$,
$\subspace{10241\\01104}{}$,
$\subspace{10302\\01043}{}$,
$\subspace{10314\\01011}{}$,
$\subspace{10321\\01034}{}$,
$\subspace{10333\\01002}{}$,
$\subspace{10340\\01020}{}$,
$\subspace{12003\\00114}{}$,
$\subspace{12012\\00140}{}$,
$\subspace{12021\\00121}{}$,
$\subspace{12030\\00102}{}$,
$\subspace{12044\\00133}{}$,
$\subspace{10204\\01122}{15}$,
$\subspace{10211\\01130}{}$,
$\subspace{10223\\01143}{}$,
$\subspace{10230\\01101}{}$,
$\subspace{10242\\01114}{}$,
$\subspace{10301\\01010}{}$,
$\subspace{10313\\01033}{}$,
$\subspace{10320\\01001}{}$,
$\subspace{10332\\01024}{}$,
$\subspace{10344\\01042}{}$,
$\subspace{12004\\00144}{}$,
$\subspace{12013\\00120}{}$,
$\subspace{12022\\00101}{}$,
$\subspace{12031\\00132}{}$,
$\subspace{12040\\00113}{}$,
$\subspace{10200\\01221}{15}$,
$\subspace{10212\\01210}{}$,
$\subspace{10224\\01204}{}$,
$\subspace{10231\\01243}{}$,
$\subspace{10243\\01232}{}$,
$\subspace{10203\\01433}{}$,
$\subspace{10210\\01434}{}$,
$\subspace{10222\\01430}{}$,
$\subspace{10234\\01431}{}$,
$\subspace{10241\\01432}{}$,
$\subspace{10402\\01312}{}$,
$\subspace{10414\\01322}{}$,
$\subspace{10421\\01332}{}$,
$\subspace{10433\\01342}{}$,
$\subspace{10440\\01302}{}$,
$\subspace{10201\\01221}{15}$,
$\subspace{10213\\01210}{}$,
$\subspace{10220\\01204}{}$,
$\subspace{10232\\01243}{}$,
$\subspace{10244\\01232}{}$,
$\subspace{10201\\01432}{}$,
$\subspace{10213\\01433}{}$,
$\subspace{10220\\01434}{}$,
$\subspace{10232\\01430}{}$,
$\subspace{10244\\01431}{}$,
$\subspace{10403\\01322}{}$,
$\subspace{10410\\01332}{}$,
$\subspace{10422\\01342}{}$,
$\subspace{10434\\01302}{}$,
$\subspace{10441\\01312}{}$,
$\subspace{10204\\01230}{15}$,
$\subspace{10211\\01224}{}$,
$\subspace{10223\\01213}{}$,
$\subspace{10230\\01202}{}$,
$\subspace{10242\\01241}{}$,
$\subspace{10204\\01414}{}$,
$\subspace{10211\\01410}{}$,
$\subspace{10223\\01411}{}$,
$\subspace{10230\\01412}{}$,
$\subspace{10242\\01413}{}$,
$\subspace{10401\\01331}{}$,
$\subspace{10413\\01341}{}$,
$\subspace{10420\\01301}{}$,
$\subspace{10432\\01311}{}$,
$\subspace{10444\\01321}{}$,
$\subspace{10202\\01241}{15}$,
$\subspace{10214\\01230}{}$,
$\subspace{10221\\01224}{}$,
$\subspace{10233\\01213}{}$,
$\subspace{10240\\01202}{}$,
$\subspace{10200\\01410}{}$,
$\subspace{10212\\01411}{}$,
$\subspace{10224\\01412}{}$,
$\subspace{10231\\01413}{}$,
$\subspace{10243\\01414}{}$,
$\subspace{10402\\01331}{}$,
$\subspace{10414\\01341}{}$,
$\subspace{10421\\01301}{}$,
$\subspace{10433\\01311}{}$,
$\subspace{10440\\01321}{}$,
$\subspace{10303\\01310}{15}$,
$\subspace{10310\\01340}{}$,
$\subspace{10322\\01320}{}$,
$\subspace{10334\\01300}{}$,
$\subspace{10341\\01330}{}$,
$\subspace{10401\\01222}{}$,
$\subspace{10413\\01244}{}$,
$\subspace{10420\\01211}{}$,
$\subspace{10432\\01233}{}$,
$\subspace{10444\\01200}{}$,
$\subspace{10402\\01442}{}$,
$\subspace{10414\\01440}{}$,
$\subspace{10421\\01443}{}$,
$\subspace{10433\\01441}{}$,
$\subspace{10440\\01444}{}$,
$\subspace{10303\\01313}{15}$,
$\subspace{10310\\01343}{}$,
$\subspace{10322\\01323}{}$,
$\subspace{10334\\01303}{}$,
$\subspace{10341\\01333}{}$,
$\subspace{10400\\01201}{}$,
$\subspace{10412\\01223}{}$,
$\subspace{10424\\01240}{}$,
$\subspace{10431\\01212}{}$,
$\subspace{10443\\01234}{}$,
$\subspace{10404\\01404}{}$,
$\subspace{10411\\01402}{}$,
$\subspace{10423\\01400}{}$,
$\subspace{10430\\01403}{}$,
$\subspace{10442\\01401}{}$,
$\subspace{10304\\01330}{15}$,
$\subspace{10311\\01310}{}$,
$\subspace{10323\\01340}{}$,
$\subspace{10330\\01320}{}$,
$\subspace{10342\\01300}{}$,
$\subspace{10401\\01233}{}$,
$\subspace{10413\\01200}{}$,
$\subspace{10420\\01222}{}$,
$\subspace{10432\\01244}{}$,
$\subspace{10444\\01211}{}$,
$\subspace{10401\\01441}{}$,
$\subspace{10413\\01444}{}$,
$\subspace{10420\\01442}{}$,
$\subspace{10432\\01440}{}$,
$\subspace{10444\\01443}{}$,
$\subspace{10304\\01342}{15}$,
$\subspace{10311\\01322}{}$,
$\subspace{10323\\01302}{}$,
$\subspace{10330\\01332}{}$,
$\subspace{10342\\01312}{}$,
$\subspace{10403\\01243}{}$,
$\subspace{10410\\01210}{}$,
$\subspace{10422\\01232}{}$,
$\subspace{10434\\01204}{}$,
$\subspace{10441\\01221}{}$,
$\subspace{10403\\01432}{}$,
$\subspace{10410\\01430}{}$,
$\subspace{10422\\01433}{}$,
$\subspace{10434\\01431}{}$,
$\subspace{10441\\01434}{}$.

\smallskip

$n_5(5,2;22)\ge 542$,$\left[\!\begin{smallmatrix}10000\\00100\\04410\\00001\\00044\end{smallmatrix}\!\right]$:
$\subspace{01002\\00102}{5}$,
$\subspace{01011\\00114}{}$,
$\subspace{01020\\00121}{}$,
$\subspace{01034\\00133}{}$,
$\subspace{01043\\00140}{}$,
$\subspace{01001\\00113}{5}$,
$\subspace{01010\\00120}{}$,
$\subspace{01024\\00132}{}$,
$\subspace{01033\\00144}{}$,
$\subspace{01042\\00101}{}$,
$\subspace{01001\\00113}{5}$,
$\subspace{01010\\00120}{}$,
$\subspace{01024\\00132}{}$,
$\subspace{01033\\00144}{}$,
$\subspace{01042\\00101}{}$,
$\subspace{01003\\00141}{5}$,
$\subspace{01012\\00103}{}$,
$\subspace{01021\\00110}{}$,
$\subspace{01030\\00122}{}$,
$\subspace{01044\\00134}{}$,
$\subspace{10001\\00012}{3}$,
$\subspace{10001\\00014}{}$,
$\subspace{10002\\00013}{}$,
$\subspace{10001\\00012}{3}$,
$\subspace{10001\\00014}{}$,
$\subspace{10002\\00013}{}$,
$\subspace{10003\\00013}{3}$,
$\subspace{10004\\00012}{}$,
$\subspace{10004\\00014}{}$,
$\subspace{10003\\00013}{3}$,
$\subspace{10004\\00012}{}$,
$\subspace{10004\\00014}{}$,
$\subspace{10014\\01203}{15}$,
$\subspace{10014\\01214}{}$,
$\subspace{10014\\01220}{}$,
$\subspace{10014\\01231}{}$,
$\subspace{10014\\01242}{}$,
$\subspace{10012\\01304}{}$,
$\subspace{10012\\01314}{}$,
$\subspace{10012\\01324}{}$,
$\subspace{10012\\01334}{}$,
$\subspace{10012\\01344}{}$,
$\subspace{10034\\01420}{}$,
$\subspace{10034\\01421}{}$,
$\subspace{10034\\01422}{}$,
$\subspace{10034\\01423}{}$,
$\subspace{10034\\01424}{}$,
$\subspace{10043\\01200}{15}$,
$\subspace{10043\\01211}{}$,
$\subspace{10043\\01222}{}$,
$\subspace{10043\\01233}{}$,
$\subspace{10043\\01244}{}$,
$\subspace{10021\\01300}{}$,
$\subspace{10021\\01310}{}$,
$\subspace{10021\\01320}{}$,
$\subspace{10021\\01330}{}$,
$\subspace{10021\\01340}{}$,
$\subspace{10041\\01440}{}$,
$\subspace{10041\\01441}{}$,
$\subspace{10041\\01442}{}$,
$\subspace{10041\\01443}{}$,
$\subspace{10041\\01444}{}$,
$\subspace{10101\\01003}{15}$,
$\subspace{10113\\01012}{}$,
$\subspace{10120\\01021}{}$,
$\subspace{10132\\01030}{}$,
$\subspace{10144\\01044}{}$,
$\subspace{10403\\01144}{}$,
$\subspace{10410\\01123}{}$,
$\subspace{10422\\01102}{}$,
$\subspace{10434\\01131}{}$,
$\subspace{10441\\01110}{}$,
$\subspace{14003\\00141}{}$,
$\subspace{14012\\00134}{}$,
$\subspace{14021\\00122}{}$,
$\subspace{14030\\00110}{}$,
$\subspace{14044\\00103}{}$,
$\subspace{10103\\01032}{15}$,
$\subspace{10110\\01041}{}$,
$\subspace{10122\\01000}{}$,
$\subspace{10134\\01014}{}$,
$\subspace{10141\\01023}{}$,
$\subspace{10404\\01140}{}$,
$\subspace{10411\\01124}{}$,
$\subspace{10423\\01103}{}$,
$\subspace{10430\\01132}{}$,
$\subspace{10442\\01111}{}$,
$\subspace{14004\\00143}{}$,
$\subspace{14013\\00131}{}$,
$\subspace{14022\\00124}{}$,
$\subspace{14031\\00112}{}$,
$\subspace{14040\\00100}{}$,
$\subspace{10104\\01043}{15}$,
$\subspace{10111\\01002}{}$,
$\subspace{10123\\01011}{}$,
$\subspace{10130\\01020}{}$,
$\subspace{10142\\01034}{}$,
$\subspace{10401\\01112}{}$,
$\subspace{10413\\01141}{}$,
$\subspace{10420\\01120}{}$,
$\subspace{10432\\01104}{}$,
$\subspace{10444\\01133}{}$,
$\subspace{14001\\00114}{}$,
$\subspace{14010\\00102}{}$,
$\subspace{14024\\00140}{}$,
$\subspace{14033\\00133}{}$,
$\subspace{14042\\00121}{}$,
$\subspace{10104\\01043}{15}$,
$\subspace{10111\\01002}{}$,
$\subspace{10123\\01011}{}$,
$\subspace{10130\\01020}{}$,
$\subspace{10142\\01034}{}$,
$\subspace{10401\\01112}{}$,
$\subspace{10413\\01141}{}$,
$\subspace{10420\\01120}{}$,
$\subspace{10432\\01104}{}$,
$\subspace{10444\\01133}{}$,
$\subspace{14001\\00114}{}$,
$\subspace{14010\\00102}{}$,
$\subspace{14024\\00140}{}$,
$\subspace{14033\\00133}{}$,
$\subspace{14042\\00121}{}$,
$\subspace{10101\\01111}{15}$,
$\subspace{10113\\01132}{}$,
$\subspace{10120\\01103}{}$,
$\subspace{10132\\01124}{}$,
$\subspace{10144\\01140}{}$,
$\subspace{10404\\01000}{}$,
$\subspace{10411\\01041}{}$,
$\subspace{10423\\01032}{}$,
$\subspace{10430\\01023}{}$,
$\subspace{10442\\01014}{}$,
$\subspace{11001\\00100}{}$,
$\subspace{11010\\00112}{}$,
$\subspace{11024\\00124}{}$,
$\subspace{11033\\00131}{}$,
$\subspace{11042\\00143}{}$,
$\subspace{10103\\01113}{15}$,
$\subspace{10110\\01134}{}$,
$\subspace{10122\\01100}{}$,
$\subspace{10134\\01121}{}$,
$\subspace{10141\\01142}{}$,
$\subspace{10404\\01022}{}$,
$\subspace{10411\\01013}{}$,
$\subspace{10423\\01004}{}$,
$\subspace{10430\\01040}{}$,
$\subspace{10442\\01031}{}$,
$\subspace{11003\\00111}{}$,
$\subspace{11012\\00123}{}$,
$\subspace{11021\\00130}{}$,
$\subspace{11030\\00142}{}$,
$\subspace{11044\\00104}{}$,
$\subspace{10101\\01124}{15}$,
$\subspace{10113\\01140}{}$,
$\subspace{10120\\01111}{}$,
$\subspace{10132\\01132}{}$,
$\subspace{10144\\01103}{}$,
$\subspace{10400\\01014}{}$,
$\subspace{10412\\01000}{}$,
$\subspace{10424\\01041}{}$,
$\subspace{10431\\01032}{}$,
$\subspace{10443\\01023}{}$,
$\subspace{11001\\00131}{}$,
$\subspace{11010\\00143}{}$,
$\subspace{11024\\00100}{}$,
$\subspace{11033\\00112}{}$,
$\subspace{11042\\00124}{}$,
$\subspace{10101\\01142}{15}$,
$\subspace{10113\\01113}{}$,
$\subspace{10120\\01134}{}$,
$\subspace{10132\\01100}{}$,
$\subspace{10144\\01121}{}$,
$\subspace{10400\\01022}{}$,
$\subspace{10412\\01013}{}$,
$\subspace{10424\\01004}{}$,
$\subspace{10431\\01040}{}$,
$\subspace{10443\\01031}{}$,
$\subspace{11001\\00123}{}$,
$\subspace{11010\\00130}{}$,
$\subspace{11024\\00142}{}$,
$\subspace{11033\\00104}{}$,
$\subspace{11042\\00111}{}$,
$\subspace{10102\\01201}{15}$,
$\subspace{10114\\01234}{}$,
$\subspace{10121\\01212}{}$,
$\subspace{10133\\01240}{}$,
$\subspace{10140\\01223}{}$,
$\subspace{10102\\01400}{}$,
$\subspace{10114\\01402}{}$,
$\subspace{10121\\01404}{}$,
$\subspace{10133\\01401}{}$,
$\subspace{10140\\01403}{}$,
$\subspace{10204\\01303}{}$,
$\subspace{10211\\01323}{}$,
$\subspace{10223\\01343}{}$,
$\subspace{10230\\01313}{}$,
$\subspace{10242\\01333}{}$,
$\subspace{10102\\01222}{15}$,
$\subspace{10114\\01200}{}$,
$\subspace{10121\\01233}{}$,
$\subspace{10133\\01211}{}$,
$\subspace{10140\\01244}{}$,
$\subspace{10102\\01441}{}$,
$\subspace{10114\\01443}{}$,
$\subspace{10121\\01440}{}$,
$\subspace{10133\\01442}{}$,
$\subspace{10140\\01444}{}$,
$\subspace{10200\\01320}{}$,
$\subspace{10212\\01340}{}$,
$\subspace{10224\\01310}{}$,
$\subspace{10231\\01330}{}$,
$\subspace{10243\\01300}{}$,
$\subspace{10100\\01232}{15}$,
$\subspace{10112\\01210}{}$,
$\subspace{10124\\01243}{}$,
$\subspace{10131\\01221}{}$,
$\subspace{10143\\01204}{}$,
$\subspace{10100\\01432}{}$,
$\subspace{10112\\01434}{}$,
$\subspace{10124\\01431}{}$,
$\subspace{10131\\01433}{}$,
$\subspace{10143\\01430}{}$,
$\subspace{10200\\01332}{}$,
$\subspace{10212\\01302}{}$,
$\subspace{10224\\01322}{}$,
$\subspace{10231\\01342}{}$,
$\subspace{10243\\01312}{}$,
$\subspace{10100\\01232}{15}$,
$\subspace{10112\\01210}{}$,
$\subspace{10124\\01243}{}$,
$\subspace{10131\\01221}{}$,
$\subspace{10143\\01204}{}$,
$\subspace{10100\\01432}{}$,
$\subspace{10112\\01434}{}$,
$\subspace{10124\\01431}{}$,
$\subspace{10131\\01433}{}$,
$\subspace{10143\\01430}{}$,
$\subspace{10200\\01332}{}$,
$\subspace{10212\\01302}{}$,
$\subspace{10224\\01322}{}$,
$\subspace{10231\\01342}{}$,
$\subspace{10243\\01312}{}$,
$\subspace{10103\\01313}{15}$,
$\subspace{10110\\01303}{}$,
$\subspace{10122\\01343}{}$,
$\subspace{10134\\01333}{}$,
$\subspace{10141\\01323}{}$,
$\subspace{10301\\01234}{}$,
$\subspace{10313\\01240}{}$,
$\subspace{10320\\01201}{}$,
$\subspace{10332\\01212}{}$,
$\subspace{10344\\01223}{}$,
$\subspace{10300\\01402}{}$,
$\subspace{10312\\01401}{}$,
$\subspace{10324\\01400}{}$,
$\subspace{10331\\01404}{}$,
$\subspace{10343\\01403}{}$,
$\subspace{10103\\01343}{15}$,
$\subspace{10110\\01333}{}$,
$\subspace{10122\\01323}{}$,
$\subspace{10134\\01313}{}$,
$\subspace{10141\\01303}{}$,
$\subspace{10302\\01201}{}$,
$\subspace{10314\\01212}{}$,
$\subspace{10321\\01223}{}$,
$\subspace{10333\\01234}{}$,
$\subspace{10340\\01240}{}$,
$\subspace{10304\\01401}{}$,
$\subspace{10311\\01400}{}$,
$\subspace{10323\\01404}{}$,
$\subspace{10330\\01403}{}$,
$\subspace{10342\\01402}{}$,
$\subspace{10104\\01344}{15}$,
$\subspace{10111\\01334}{}$,
$\subspace{10123\\01324}{}$,
$\subspace{10130\\01314}{}$,
$\subspace{10142\\01304}{}$,
$\subspace{10301\\01203}{}$,
$\subspace{10313\\01214}{}$,
$\subspace{10320\\01220}{}$,
$\subspace{10332\\01231}{}$,
$\subspace{10344\\01242}{}$,
$\subspace{10303\\01424}{}$,
$\subspace{10310\\01423}{}$,
$\subspace{10322\\01422}{}$,
$\subspace{10334\\01421}{}$,
$\subspace{10341\\01420}{}$,
$\subspace{10104\\01344}{15}$,
$\subspace{10111\\01334}{}$,
$\subspace{10123\\01324}{}$,
$\subspace{10130\\01314}{}$,
$\subspace{10142\\01304}{}$,
$\subspace{10301\\01203}{}$,
$\subspace{10313\\01214}{}$,
$\subspace{10320\\01220}{}$,
$\subspace{10332\\01231}{}$,
$\subspace{10344\\01242}{}$,
$\subspace{10303\\01424}{}$,
$\subspace{10310\\01423}{}$,
$\subspace{10322\\01422}{}$,
$\subspace{10334\\01421}{}$,
$\subspace{10341\\01420}{}$,
$\subspace{10201\\01002}{15}$,
$\subspace{10213\\01034}{}$,
$\subspace{10220\\01011}{}$,
$\subspace{10232\\01043}{}$,
$\subspace{10244\\01020}{}$,
$\subspace{10302\\01120}{}$,
$\subspace{10314\\01112}{}$,
$\subspace{10321\\01104}{}$,
$\subspace{10333\\01141}{}$,
$\subspace{10340\\01133}{}$,
$\subspace{13002\\00140}{}$,
$\subspace{13011\\00114}{}$,
$\subspace{13020\\00133}{}$,
$\subspace{13034\\00102}{}$,
$\subspace{13043\\00121}{}$,
$\subspace{10201\\01003}{15}$,
$\subspace{10213\\01030}{}$,
$\subspace{10220\\01012}{}$,
$\subspace{10232\\01044}{}$,
$\subspace{10244\\01021}{}$,
$\subspace{10303\\01131}{}$,
$\subspace{10310\\01123}{}$,
$\subspace{10322\\01110}{}$,
$\subspace{10334\\01102}{}$,
$\subspace{10341\\01144}{}$,
$\subspace{13003\\00122}{}$,
$\subspace{13012\\00141}{}$,
$\subspace{13021\\00110}{}$,
$\subspace{13030\\00134}{}$,
$\subspace{13044\\00103}{}$,
$\subspace{10201\\01004}{15}$,
$\subspace{10213\\01031}{}$,
$\subspace{10220\\01013}{}$,
$\subspace{10232\\01040}{}$,
$\subspace{10244\\01022}{}$,
$\subspace{10304\\01142}{}$,
$\subspace{10311\\01134}{}$,
$\subspace{10323\\01121}{}$,
$\subspace{10330\\01113}{}$,
$\subspace{10342\\01100}{}$,
$\subspace{13004\\00104}{}$,
$\subspace{13013\\00123}{}$,
$\subspace{13022\\00142}{}$,
$\subspace{13031\\00111}{}$,
$\subspace{13040\\00130}{}$,
$\subspace{10202\\01022}{15}$,
$\subspace{10214\\01004}{}$,
$\subspace{10221\\01031}{}$,
$\subspace{10233\\01013}{}$,
$\subspace{10240\\01040}{}$,
$\subspace{10304\\01113}{}$,
$\subspace{10311\\01100}{}$,
$\subspace{10323\\01142}{}$,
$\subspace{10330\\01134}{}$,
$\subspace{10342\\01121}{}$,
$\subspace{13004\\00111}{}$,
$\subspace{13013\\00130}{}$,
$\subspace{13022\\00104}{}$,
$\subspace{13031\\00123}{}$,
$\subspace{13040\\00142}{}$,
$\subspace{10202\\01111}{15}$,
$\subspace{10214\\01124}{}$,
$\subspace{10221\\01132}{}$,
$\subspace{10233\\01140}{}$,
$\subspace{10240\\01103}{}$,
$\subspace{10303\\01000}{}$,
$\subspace{10310\\01023}{}$,
$\subspace{10322\\01041}{}$,
$\subspace{10334\\01014}{}$,
$\subspace{10341\\01032}{}$,
$\subspace{12002\\00100}{}$,
$\subspace{12011\\00131}{}$,
$\subspace{12020\\00112}{}$,
$\subspace{12034\\00143}{}$,
$\subspace{12043\\00124}{}$,
$\subspace{10200\\01122}{15}$,
$\subspace{10212\\01130}{}$,
$\subspace{10224\\01143}{}$,
$\subspace{10231\\01101}{}$,
$\subspace{10243\\01114}{}$,
$\subspace{10302\\01033}{}$,
$\subspace{10314\\01001}{}$,
$\subspace{10321\\01024}{}$,
$\subspace{10333\\01042}{}$,
$\subspace{10340\\01010}{}$,
$\subspace{12000\\00113}{}$,
$\subspace{12014\\00144}{}$,
$\subspace{12023\\00120}{}$,
$\subspace{12032\\00101}{}$,
$\subspace{12041\\00132}{}$,
$\subspace{10204\\01123}{15}$,
$\subspace{10211\\01131}{}$,
$\subspace{10223\\01144}{}$,
$\subspace{10230\\01102}{}$,
$\subspace{10242\\01110}{}$,
$\subspace{10304\\01030}{}$,
$\subspace{10311\\01003}{}$,
$\subspace{10323\\01021}{}$,
$\subspace{10330\\01044}{}$,
$\subspace{10342\\01012}{}$,
$\subspace{12004\\00134}{}$,
$\subspace{12013\\00110}{}$,
$\subspace{12022\\00141}{}$,
$\subspace{12031\\00122}{}$,
$\subspace{12040\\00103}{}$,
$\subspace{10203\\01130}{15}$,
$\subspace{10210\\01143}{}$,
$\subspace{10222\\01101}{}$,
$\subspace{10234\\01114}{}$,
$\subspace{10241\\01122}{}$,
$\subspace{10302\\01001}{}$,
$\subspace{10314\\01024}{}$,
$\subspace{10321\\01042}{}$,
$\subspace{10333\\01010}{}$,
$\subspace{10340\\01033}{}$,
$\subspace{12003\\00101}{}$,
$\subspace{12012\\00132}{}$,
$\subspace{12021\\00113}{}$,
$\subspace{12030\\00144}{}$,
$\subspace{12044\\00120}{}$,
$\subspace{10204\\01202}{15}$,
$\subspace{10211\\01241}{}$,
$\subspace{10223\\01230}{}$,
$\subspace{10230\\01224}{}$,
$\subspace{10242\\01213}{}$,
$\subspace{10203\\01412}{}$,
$\subspace{10210\\01413}{}$,
$\subspace{10222\\01414}{}$,
$\subspace{10234\\01410}{}$,
$\subspace{10241\\01411}{}$,
$\subspace{10402\\01321}{}$,
$\subspace{10414\\01331}{}$,
$\subspace{10421\\01341}{}$,
$\subspace{10433\\01301}{}$,
$\subspace{10440\\01311}{}$,
$\subspace{10204\\01202}{15}$,
$\subspace{10211\\01241}{}$,
$\subspace{10223\\01230}{}$,
$\subspace{10230\\01224}{}$,
$\subspace{10242\\01213}{}$,
$\subspace{10203\\01412}{}$,
$\subspace{10210\\01413}{}$,
$\subspace{10222\\01414}{}$,
$\subspace{10234\\01410}{}$,
$\subspace{10241\\01411}{}$,
$\subspace{10402\\01321}{}$,
$\subspace{10414\\01331}{}$,
$\subspace{10421\\01341}{}$,
$\subspace{10433\\01301}{}$,
$\subspace{10440\\01311}{}$,
$\subspace{10202\\01210}{15}$,
$\subspace{10214\\01204}{}$,
$\subspace{10221\\01243}{}$,
$\subspace{10233\\01232}{}$,
$\subspace{10240\\01221}{}$,
$\subspace{10202\\01432}{}$,
$\subspace{10214\\01433}{}$,
$\subspace{10221\\01434}{}$,
$\subspace{10233\\01430}{}$,
$\subspace{10240\\01431}{}$,
$\subspace{10401\\01312}{}$,
$\subspace{10413\\01322}{}$,
$\subspace{10420\\01332}{}$,
$\subspace{10432\\01342}{}$,
$\subspace{10444\\01302}{}$,
$\subspace{10201\\01224}{15}$,
$\subspace{10213\\01213}{}$,
$\subspace{10220\\01202}{}$,
$\subspace{10232\\01241}{}$,
$\subspace{10244\\01230}{}$,
$\subspace{10203\\01413}{}$,
$\subspace{10210\\01414}{}$,
$\subspace{10222\\01410}{}$,
$\subspace{10234\\01411}{}$,
$\subspace{10241\\01412}{}$,
$\subspace{10400\\01331}{}$,
$\subspace{10412\\01341}{}$,
$\subspace{10424\\01301}{}$,
$\subspace{10431\\01311}{}$,
$\subspace{10443\\01321}{}$,
$\subspace{10300\\01320}{15}$,
$\subspace{10312\\01300}{}$,
$\subspace{10324\\01330}{}$,
$\subspace{10331\\01310}{}$,
$\subspace{10343\\01340}{}$,
$\subspace{10403\\01222}{}$,
$\subspace{10410\\01244}{}$,
$\subspace{10422\\01211}{}$,
$\subspace{10434\\01233}{}$,
$\subspace{10441\\01200}{}$,
$\subspace{10403\\01441}{}$,
$\subspace{10410\\01444}{}$,
$\subspace{10422\\01442}{}$,
$\subspace{10434\\01440}{}$,
$\subspace{10441\\01443}{}$,
$\subspace{10300\\01322}{15}$,
$\subspace{10312\\01302}{}$,
$\subspace{10324\\01332}{}$,
$\subspace{10331\\01312}{}$,
$\subspace{10343\\01342}{}$,
$\subspace{10404\\01243}{}$,
$\subspace{10411\\01210}{}$,
$\subspace{10423\\01232}{}$,
$\subspace{10430\\01204}{}$,
$\subspace{10442\\01221}{}$,
$\subspace{10401\\01434}{}$,
$\subspace{10413\\01432}{}$,
$\subspace{10420\\01430}{}$,
$\subspace{10432\\01433}{}$,
$\subspace{10444\\01431}{}$,
$\subspace{10300\\01344}{15}$,
$\subspace{10312\\01324}{}$,
$\subspace{10324\\01304}{}$,
$\subspace{10331\\01334}{}$,
$\subspace{10343\\01314}{}$,
$\subspace{10402\\01242}{}$,
$\subspace{10414\\01214}{}$,
$\subspace{10421\\01231}{}$,
$\subspace{10433\\01203}{}$,
$\subspace{10440\\01220}{}$,
$\subspace{10402\\01423}{}$,
$\subspace{10414\\01421}{}$,
$\subspace{10421\\01424}{}$,
$\subspace{10433\\01422}{}$,
$\subspace{10440\\01420}{}$,
$\subspace{10301\\01344}{15}$,
$\subspace{10313\\01324}{}$,
$\subspace{10320\\01304}{}$,
$\subspace{10332\\01334}{}$,
$\subspace{10344\\01314}{}$,
$\subspace{10400\\01220}{}$,
$\subspace{10412\\01242}{}$,
$\subspace{10424\\01214}{}$,
$\subspace{10431\\01231}{}$,
$\subspace{10443\\01203}{}$,
$\subspace{10403\\01421}{}$,
$\subspace{10410\\01424}{}$,
$\subspace{10422\\01422}{}$,
$\subspace{10434\\01420}{}$,
$\subspace{10441\\01423}{}$.

\smallskip

$n_5(5,2;23)\ge 568$,$\left[\!\begin{smallmatrix}00001\\10003\\01004\\00100\\00010\end{smallmatrix}\!\right]$:
$\subspace{00101\\00012}{71}$,
$\subspace{00104\\00011}{}$,
$\subspace{01014\\00120}{}$,
$\subspace{01044\\00110}{}$,
$\subspace{10001\\00011}{}$,
$\subspace{10003\\00012}{}$,
$\subspace{10001\\01002}{}$,
$\subspace{10003\\01000}{}$,
$\subspace{10011\\01202}{}$,
$\subspace{10012\\01333}{}$,
$\subspace{10041\\01421}{}$,
$\subspace{10100\\01000}{}$,
$\subspace{10101\\01010}{}$,
$\subspace{10103\\01010}{}$,
$\subspace{10103\\01030}{}$,
$\subspace{10113\\01011}{}$,
$\subspace{10144\\01032}{}$,
$\subspace{10140\\01041}{}$,
$\subspace{10113\\01130}{}$,
$\subspace{10131\\01114}{}$,
$\subspace{10141\\01121}{}$,
$\subspace{10101\\01212}{}$,
$\subspace{10114\\01201}{}$,
$\subspace{10143\\01204}{}$,
$\subspace{10100\\01310}{}$,
$\subspace{10131\\01303}{}$,
$\subspace{10134\\01301}{}$,
$\subspace{10140\\01340}{}$,
$\subspace{10222\\01130}{}$,
$\subspace{10240\\01112}{}$,
$\subspace{10201\\01313}{}$,
$\subspace{10200\\01400}{}$,
$\subspace{10201\\01432}{}$,
$\subspace{10213\\01411}{}$,
$\subspace{10231\\01430}{}$,
$\subspace{10231\\01432}{}$,
$\subspace{10240\\01431}{}$,
$\subspace{10242\\01444}{}$,
$\subspace{10303\\01011}{}$,
$\subspace{10301\\01032}{}$,
$\subspace{10303\\01030}{}$,
$\subspace{10323\\01021}{}$,
$\subspace{10321\\01141}{}$,
$\subspace{10331\\01121}{}$,
$\subspace{10331\\01140}{}$,
$\subspace{10332\\01141}{}$,
$\subspace{10310\\01201}{}$,
$\subspace{10324\\01213}{}$,
$\subspace{10340\\01223}{}$,
$\subspace{10332\\01310}{}$,
$\subspace{10302\\01410}{}$,
$\subspace{10310\\01411}{}$,
$\subspace{10324\\01431}{}$,
$\subspace{10323\\01440}{}$,
$\subspace{10401\\01043}{}$,
$\subspace{10433\\01043}{}$,
$\subspace{10424\\01123}{}$,
$\subspace{10434\\01103}{}$,
$\subspace{10440\\01103}{}$,
$\subspace{10443\\01104}{}$,
$\subspace{10420\\01202}{}$,
$\subspace{10433\\01223}{}$,
$\subspace{10440\\01214}{}$,
$\subspace{10401\\01342}{}$,
$\subspace{10420\\01303}{}$,
$\subspace{10424\\01313}{}$,
$\subspace{10413\\01400}{}$,
$\subspace{10422\\01423}{}$,
$\subspace{12021\\00111}{}$,
$\subspace{13022\\00121}{}$,
$\subspace{14031\\00144}{}$,
$\subspace{00102\\00012}{71}$,
$\subspace{01030\\00001}{}$,
$\subspace{01011\\00100}{}$,
$\subspace{01024\\00120}{}$,
$\subspace{01033\\00114}{}$,
$\subspace{10003\\00014}{}$,
$\subspace{10000\\00103}{}$,
$\subspace{10044\\00142}{}$,
$\subspace{10303\\00010}{}$,
$\subspace{10004\\01000}{}$,
$\subspace{10020\\01010}{}$,
$\subspace{10032\\01022}{}$,
$\subspace{10014\\01102}{}$,
$\subspace{10021\\01103}{}$,
$\subspace{10020\\01133}{}$,
$\subspace{10041\\01121}{}$,
$\subspace{10043\\01130}{}$,
$\subspace{10001\\01311}{}$,
$\subspace{10034\\01433}{}$,
$\subspace{10113\\01004}{}$,
$\subspace{10111\\01102}{}$,
$\subspace{10112\\01242}{}$,
$\subspace{10130\\01220}{}$,
$\subspace{10101\\01300}{}$,
$\subspace{10122\\01324}{}$,
$\subspace{10133\\01303}{}$,
$\subspace{10133\\01313}{}$,
$\subspace{10140\\01330}{}$,
$\subspace{10100\\01400}{}$,
$\subspace{10102\\01431}{}$,
$\subspace{10234\\01043}{}$,
$\subspace{10203\\01131}{}$,
$\subspace{10211\\01122}{}$,
$\subspace{10231\\01120}{}$,
$\subspace{10204\\01202}{}$,
$\subspace{10243\\01204}{}$,
$\subspace{10220\\01331}{}$,
$\subspace{10240\\01330}{}$,
$\subspace{10220\\01401}{}$,
$\subspace{10224\\01414}{}$,
$\subspace{10224\\01433}{}$,
$\subspace{10320\\01042}{}$,
$\subspace{10334\\01020}{}$,
$\subspace{10334\\01034}{}$,
$\subspace{10331\\01111}{}$,
$\subspace{10333\\01144}{}$,
$\subspace{10343\\01131}{}$,
$\subspace{10324\\01334}{}$,
$\subspace{10332\\01323}{}$,
$\subspace{10320\\01432}{}$,
$\subspace{10344\\01444}{}$,
$\subspace{10420\\01013}{}$,
$\subspace{10424\\01013}{}$,
$\subspace{10401\\01122}{}$,
$\subspace{10404\\01143}{}$,
$\subspace{10412\\01110}{}$,
$\subspace{10412\\01141}{}$,
$\subspace{10441\\01104}{}$,
$\subspace{10404\\01244}{}$,
$\subspace{10442\\01223}{}$,
$\subspace{10440\\01332}{}$,
$\subspace{10423\\01401}{}$,
$\subspace{10433\\01422}{}$,
$\subspace{11022\\00141}{}$,
$\subspace{11033\\00103}{}$,
$\subspace{11032\\00133}{}$,
$\subspace{11034\\00144}{}$,
$\subspace{11040\\00124}{}$,
$\subspace{12033\\00110}{}$,
$\subspace{13104\\00011}{}$,
$\subspace{14043\\00132}{}$,
$\subspace{01012\\00110}{71}$,
$\subspace{01030\\00134}{}$,
$\subspace{01043\\00101}{}$,
$\subspace{01103\\00013}{}$,
$\subspace{10002\\00112}{}$,
$\subspace{10010\\00114}{}$,
$\subspace{10042\\00111}{}$,
$\subspace{10202\\00011}{}$,
$\subspace{10001\\01021}{}$,
$\subspace{10011\\01010}{}$,
$\subspace{10033\\01033}{}$,
$\subspace{10044\\01023}{}$,
$\subspace{10023\\01201}{}$,
$\subspace{10003\\01311}{}$,
$\subspace{10020\\01303}{}$,
$\subspace{10013\\01423}{}$,
$\subspace{10021\\01441}{}$,
$\subspace{10100\\01034}{}$,
$\subspace{10123\\01104}{}$,
$\subspace{10101\\01203}{}$,
$\subspace{10143\\01223}{}$,
$\subspace{10104\\01330}{}$,
$\subspace{10132\\01334}{}$,
$\subspace{10144\\01312}{}$,
$\subspace{10142\\01343}{}$,
$\subspace{10113\\01433}{}$,
$\subspace{10133\\01432}{}$,
$\subspace{10231\\01013}{}$,
$\subspace{10200\\01131}{}$,
$\subspace{10242\\01113}{}$,
$\subspace{10210\\01231}{}$,
$\subspace{10212\\01302}{}$,
$\subspace{10232\\01301}{}$,
$\subspace{10220\\01420}{}$,
$\subspace{10300\\01041}{}$,
$\subspace{10310\\01002}{}$,
$\subspace{10330\\01040}{}$,
$\subspace{10343\\01000}{}$,
$\subspace{10301\\01102}{}$,
$\subspace{10322\\01141}{}$,
$\subspace{10312\\01230}{}$,
$\subspace{10341\\01204}{}$,
$\subspace{10303\\01344}{}$,
$\subspace{10324\\01304}{}$,
$\subspace{10332\\01314}{}$,
$\subspace{10340\\01313}{}$,
$\subspace{10302\\01431}{}$,
$\subspace{10314\\01403}{}$,
$\subspace{10334\\01412}{}$,
$\subspace{10331\\01434}{}$,
$\subspace{10402\\01003}{}$,
$\subspace{10413\\01001}{}$,
$\subspace{10412\\01011}{}$,
$\subspace{10433\\01014}{}$,
$\subspace{10403\\01101}{}$,
$\subspace{10410\\01121}{}$,
$\subspace{10414\\01130}{}$,
$\subspace{10420\\01212}{}$,
$\subspace{10424\\01241}{}$,
$\subspace{10404\\01321}{}$,
$\subspace{10430\\01320}{}$,
$\subspace{10400\\01440}{}$,
$\subspace{10421\\01430}{}$,
$\subspace{10444\\01413}{}$,
$\subspace{11032\\00130}{}$,
$\subspace{12013\\00120}{}$,
$\subspace{12011\\00143}{}$,
$\subspace{13030\\00103}{}$,
$\subspace{13103\\00012}{}$,
$\subspace{14022\\00131}{}$,
$\subspace{14042\\00133}{}$,
$\subspace{01013\\00112}{71}$,
$\subspace{01022\\00114}{}$,
$\subspace{01031\\00130}{}$,
$\subspace{01302\\00014}{}$,
$\subspace{10004\\00133}{}$,
$\subspace{10033\\00134}{}$,
$\subspace{10044\\00130}{}$,
$\subspace{10304\\00013}{}$,
$\subspace{10002\\01033}{}$,
$\subspace{10021\\01034}{}$,
$\subspace{10042\\01003}{}$,
$\subspace{10002\\01101}{}$,
$\subspace{10013\\01101}{}$,
$\subspace{10042\\01111}{}$,
$\subspace{10044\\01113}{}$,
$\subspace{10020\\01240}{}$,
$\subspace{10013\\01343}{}$,
$\subspace{10020\\01422}{}$,
$\subspace{10021\\01420}{}$,
$\subspace{10102\\01131}{}$,
$\subspace{10132\\01122}{}$,
$\subspace{10133\\01124}{}$,
$\subspace{10142\\01231}{}$,
$\subspace{10124\\01314}{}$,
$\subspace{10122\\01334}{}$,
$\subspace{10133\\01322}{}$,
$\subspace{10111\\01433}{}$,
$\subspace{10214\\01013}{}$,
$\subspace{10211\\01033}{}$,
$\subspace{10223\\01144}{}$,
$\subspace{10234\\01144}{}$,
$\subspace{10224\\01232}{}$,
$\subspace{10230\\01220}{}$,
$\subspace{10203\\01341}{}$,
$\subspace{10220\\01311}{}$,
$\subspace{10220\\01314}{}$,
$\subspace{10224\\01321}{}$,
$\subspace{10241\\01331}{}$,
$\subspace{10322\\01034}{}$,
$\subspace{10330\\01001}{}$,
$\subspace{10342\\01020}{}$,
$\subspace{10300\\01142}{}$,
$\subspace{10314\\01102}{}$,
$\subspace{10344\\01102}{}$,
$\subspace{10341\\01231}{}$,
$\subspace{10313\\01304}{}$,
$\subspace{10314\\01323}{}$,
$\subspace{10330\\01330}{}$,
$\subspace{10333\\01332}{}$,
$\subspace{10334\\01344}{}$,
$\subspace{10300\\01433}{}$,
$\subspace{10343\\01441}{}$,
$\subspace{10421\\01001}{}$,
$\subspace{10444\\01042}{}$,
$\subspace{10404\\01122}{}$,
$\subspace{10423\\01113}{}$,
$\subspace{10444\\01110}{}$,
$\subspace{10412\\01211}{}$,
$\subspace{10402\\01320}{}$,
$\subspace{10412\\01330}{}$,
$\subspace{10400\\01404}{}$,
$\subspace{10402\\01402}{}$,
$\subspace{11004\\00013}{}$,
$\subspace{11012\\00131}{}$,
$\subspace{11023\\00114}{}$,
$\subspace{11033\\00134}{}$,
$\subspace{12004\\00103}{}$,
$\subspace{13011\\00131}{}$,
$\subspace{13020\\00140}{}$,
$\subspace{14023\\00103}{}$,
$\subspace{14043\\00133}{}$,
$\subspace{01012\\00120}{71}$,
$\subspace{10202\\00012}{}$,
$\subspace{10003\\01023}{}$,
$\subspace{10001\\01243}{}$,
$\subspace{10012\\01203}{}$,
$\subspace{10033\\01230}{}$,
$\subspace{10034\\01312}{}$,
$\subspace{10041\\01412}{}$,
$\subspace{10104\\01030}{}$,
$\subspace{10100\\01040}{}$,
$\subspace{10113\\01023}{}$,
$\subspace{10123\\01003}{}$,
$\subspace{10123\\01043}{}$,
$\subspace{10101\\01124}{}$,
$\subspace{10104\\01241}{}$,
$\subspace{10123\\01204}{}$,
$\subspace{10131\\01230}{}$,
$\subspace{10142\\01201}{}$,
$\subspace{10104\\01312}{}$,
$\subspace{10103\\01412}{}$,
$\subspace{10140\\01442}{}$,
$\subspace{10232\\01010}{}$,
$\subspace{10230\\01032}{}$,
$\subspace{10201\\01121}{}$,
$\subspace{10202\\01141}{}$,
$\subspace{10221\\01104}{}$,
$\subspace{10233\\01103}{}$,
$\subspace{10233\\01130}{}$,
$\subspace{10210\\01243}{}$,
$\subspace{10232\\01203}{}$,
$\subspace{10212\\01303}{}$,
$\subspace{10212\\01341}{}$,
$\subspace{10231\\01310}{}$,
$\subspace{10202\\01444}{}$,
$\subspace{10212\\01412}{}$,
$\subspace{10221\\01420}{}$,
$\subspace{10232\\01413}{}$,
$\subspace{10303\\01040}{}$,
$\subspace{10331\\01132}{}$,
$\subspace{10312\\01203}{}$,
$\subspace{10310\\01324}{}$,
$\subspace{10332\\01302}{}$,
$\subspace{10342\\01313}{}$,
$\subspace{10322\\01403}{}$,
$\subspace{10323\\01402}{}$,
$\subspace{10324\\01400}{}$,
$\subspace{10403\\01000}{}$,
$\subspace{10400\\01012}{}$,
$\subspace{10401\\01012}{}$,
$\subspace{10403\\01011}{}$,
$\subspace{10431\\01040}{}$,
$\subspace{10440\\01023}{}$,
$\subspace{10410\\01132}{}$,
$\subspace{10403\\01243}{}$,
$\subspace{10410\\01223}{}$,
$\subspace{10431\\01214}{}$,
$\subspace{10433\\01230}{}$,
$\subspace{10431\\01312}{}$,
$\subspace{10414\\01432}{}$,
$\subspace{10410\\01442}{}$,
$\subspace{10424\\01403}{}$,
$\subspace{10420\\01441}{}$,
$\subspace{10430\\01403}{}$,
$\subspace{10430\\01411}{}$,
$\subspace{10430\\01431}{}$,
$\subspace{12020\\00112}{}$,
$\subspace{12031\\00121}{}$,
$\subspace{12044\\00110}{}$,
$\subspace{12403\\00011}{}$,
$\subspace{13043\\00132}{}$,
$\subspace{14030\\00144}{}$,
$\subspace{01021\\00104}{71}$,
$\subspace{01022\\00123}{}$,
$\subspace{01020\\00140}{}$,
$\subspace{10024\\00140}{}$,
$\subspace{10040\\00122}{}$,
$\subspace{10102\\00014}{}$,
$\subspace{10004\\01014}{}$,
$\subspace{10031\\01042}{}$,
$\subspace{10004\\01100}{}$,
$\subspace{10010\\01120}{}$,
$\subspace{10022\\01142}{}$,
$\subspace{10023\\01211}{}$,
$\subspace{10004\\01331}{}$,
$\subspace{10011\\01443}{}$,
$\subspace{10121\\01022}{}$,
$\subspace{10122\\01020}{}$,
$\subspace{10122\\01022}{}$,
$\subspace{10102\\01114}{}$,
$\subspace{10111\\01123}{}$,
$\subspace{10112\\01142}{}$,
$\subspace{10122\\01110}{}$,
$\subspace{10112\\01332}{}$,
$\subspace{10144\\01333}{}$,
$\subspace{10102\\01421}{}$,
$\subspace{10111\\01423}{}$,
$\subspace{10110\\01440}{}$,
$\subspace{10141\\01434}{}$,
$\subspace{10200\\01031}{}$,
$\subspace{10211\\01041}{}$,
$\subspace{10211\\01042}{}$,
$\subspace{10213\\01044}{}$,
$\subspace{10222\\01042}{}$,
$\subspace{10234\\01020}{}$,
$\subspace{10242\\01031}{}$,
$\subspace{10204\\01211}{}$,
$\subspace{10203\\01220}{}$,
$\subspace{10211\\01222}{}$,
$\subspace{10223\\01210}{}$,
$\subspace{10223\\01220}{}$,
$\subspace{10223\\01234}{}$,
$\subspace{10234\\01301}{}$,
$\subspace{10203\\01404}{}$,
$\subspace{10204\\01410}{}$,
$\subspace{10203\\01422}{}$,
$\subspace{10244\\01404}{}$,
$\subspace{10313\\01002}{}$,
$\subspace{10311\\01031}{}$,
$\subspace{10302\\01120}{}$,
$\subspace{10304\\01134}{}$,
$\subspace{10344\\01110}{}$,
$\subspace{10313\\01212}{}$,
$\subspace{10313\\01221}{}$,
$\subspace{10344\\01211}{}$,
$\subspace{10340\\01233}{}$,
$\subspace{10304\\01323}{}$,
$\subspace{10321\\01331}{}$,
$\subspace{10304\\01422}{}$,
$\subspace{10423\\01112}{}$,
$\subspace{10422\\01142}{}$,
$\subspace{10432\\01220}{}$,
$\subspace{10411\\01323}{}$,
$\subspace{10413\\01332}{}$,
$\subspace{10423\\01332}{}$,
$\subspace{10443\\01342}{}$,
$\subspace{11002\\00014}{}$,
$\subspace{11000\\00101}{}$,
$\subspace{11024\\00140}{}$,
$\subspace{11030\\00113}{}$,
$\subspace{12023\\00143}{}$,
$\subspace{13301\\00014}{}$,
$\subspace{14040\\00111}{}$,
$\subspace{01040\\00124}{71}$,
$\subspace{10043\\00111}{}$,
$\subspace{10011\\01004}{}$,
$\subspace{10023\\01123}{}$,
$\subspace{10000\\01224}{}$,
$\subspace{10030\\01212}{}$,
$\subspace{10010\\01322}{}$,
$\subspace{10031\\01300}{}$,
$\subspace{10032\\01340}{}$,
$\subspace{10040\\01301}{}$,
$\subspace{10014\\01434}{}$,
$\subspace{10022\\01412}{}$,
$\subspace{10024\\01441}{}$,
$\subspace{10033\\01442}{}$,
$\subspace{10110\\01012}{}$,
$\subspace{10121\\01024}{}$,
$\subspace{10123\\01100}{}$,
$\subspace{10134\\01132}{}$,
$\subspace{10130\\01213}{}$,
$\subspace{10144\\01221}{}$,
$\subspace{10142\\01243}{}$,
$\subspace{10114\\01410}{}$,
$\subspace{10120\\01423}{}$,
$\subspace{10143\\01403}{}$,
$\subspace{10141\\01420}{}$,
$\subspace{10200\\01044}{}$,
$\subspace{10233\\01003}{}$,
$\subspace{10202\\01112}{}$,
$\subspace{10214\\01133}{}$,
$\subspace{10242\\01124}{}$,
$\subspace{10244\\01140}{}$,
$\subspace{10222\\01242}{}$,
$\subspace{10230\\01222}{}$,
$\subspace{10232\\01234}{}$,
$\subspace{10243\\01210}{}$,
$\subspace{10221\\01342}{}$,
$\subspace{10212\\01413}{}$,
$\subspace{10213\\01443}{}$,
$\subspace{10312\\01041}{}$,
$\subspace{10301\\01143}{}$,
$\subspace{10311\\01203}{}$,
$\subspace{10322\\01230}{}$,
$\subspace{10302\\01333}{}$,
$\subspace{10342\\01414}{}$,
$\subspace{10340\\01424}{}$,
$\subspace{10400\\01023}{}$,
$\subspace{10410\\01014}{}$,
$\subspace{10411\\01021}{}$,
$\subspace{10443\\01002}{}$,
$\subspace{10442\\01134}{}$,
$\subspace{10403\\01244}{}$,
$\subspace{10413\\01241}{}$,
$\subspace{10434\\01200}{}$,
$\subspace{10430\\01341}{}$,
$\subspace{10441\\01312}{}$,
$\subspace{10414\\01430}{}$,
$\subspace{10422\\01402}{}$,
$\subspace{10431\\01440}{}$,
$\subspace{11001\\00143}{}$,
$\subspace{11200\\00010}{}$,
$\subspace{12002\\00100}{}$,
$\subspace{12024\\00102}{}$,
$\subspace{12240\\00001}{}$,
$\subspace{13000\\00122}{}$,
$\subspace{13010\\00104}{}$,
$\subspace{13023\\00101}{}$,
$\subspace{13042\\00142}{}$,
$\subspace{14011\\00113}{}$,
$\subspace{14021\\00141}{}$,
$\subspace{14031\\00112}{}$,
$\subspace{14032\\00123}{}$,
$\subspace{01200\\00013}{71}$,
$\subspace{10002\\00123}{}$,
$\subspace{10024\\01024}{}$,
$\subspace{10024\\01101}{}$,
$\subspace{10030\\01100}{}$,
$\subspace{10032\\01100}{}$,
$\subspace{10013\\01222}{}$,
$\subspace{10022\\01234}{}$,
$\subspace{10040\\01224}{}$,
$\subspace{10042\\01224}{}$,
$\subspace{10043\\01221}{}$,
$\subspace{10014\\01321}{}$,
$\subspace{10031\\01304}{}$,
$\subspace{10031\\01343}{}$,
$\subspace{10000\\01443}{}$,
$\subspace{10024\\01424}{}$,
$\subspace{10031\\01443}{}$,
$\subspace{10110\\01113}{}$,
$\subspace{10120\\01133}{}$,
$\subspace{10121\\01143}{}$,
$\subspace{10130\\01123}{}$,
$\subspace{10121\\01221}{}$,
$\subspace{10124\\01234}{}$,
$\subspace{10132\\01200}{}$,
$\subspace{10134\\01200}{}$,
$\subspace{10141\\01244}{}$,
$\subspace{10110\\01333}{}$,
$\subspace{10110\\01414}{}$,
$\subspace{10121\\01410}{}$,
$\subspace{10120\\01443}{}$,
$\subspace{10244\\01001}{}$,
$\subspace{10241\\01044}{}$,
$\subspace{10244\\01100}{}$,
$\subspace{10222\\01210}{}$,
$\subspace{10244\\01231}{}$,
$\subspace{10311\\01004}{}$,
$\subspace{10311\\01134}{}$,
$\subspace{10300\\01224}{}$,
$\subspace{10314\\01210}{}$,
$\subspace{10330\\01222}{}$,
$\subspace{10341\\01210}{}$,
$\subspace{10321\\01314}{}$,
$\subspace{10432\\01112}{}$,
$\subspace{10411\\01233}{}$,
$\subspace{10411\\01240}{}$,
$\subspace{10422\\01222}{}$,
$\subspace{10442\\01233}{}$,
$\subspace{10443\\01232}{}$,
$\subspace{10444\\01233}{}$,
$\subspace{10411\\01320}{}$,
$\subspace{10421\\01340}{}$,
$\subspace{10432\\01300}{}$,
$\subspace{10441\\01344}{}$,
$\subspace{10432\\01421}{}$,
$\subspace{11002\\00102}{}$,
$\subspace{11002\\00142}{}$,
$\subspace{11012\\00123}{}$,
$\subspace{11201\\00010}{}$,
$\subspace{12001\\00130}{}$,
$\subspace{12012\\00100}{}$,
$\subspace{12012\\00113}{}$,
$\subspace{12014\\00131}{}$,
$\subspace{12023\\00124}{}$,
$\subspace{12033\\00134}{}$,
$\subspace{12043\\00104}{}$,
$\subspace{13002\\00122}{}$,
$\subspace{13033\\00141}{}$,
$\subspace{13040\\00122}{}$,
$\subspace{14001\\00122}{}$,
$\subspace{14003\\00123}{}$,
$\subspace{14430\\00001}{}$.

\smallskip

$n_5(5,2;24)\ge 594$,$\left[\!\begin{smallmatrix}00001\\10004\\01004\\00101\\00013\end{smallmatrix}\!\right]$,$\left[\!\begin{smallmatrix}00130\\13442\\44114\\11002\\00001\end{smallmatrix}\!\right]$:
$\subspace{00110\\00001}{11}$,
$\subspace{01100\\00013}{}$,
$\subspace{10000\\00011}{}$,
$\subspace{10002\\01003}{}$,
$\subspace{10034\\01302}{}$,
$\subspace{10210\\01120}{}$,
$\subspace{10443\\01133}{}$,
$\subspace{10400\\01300}{}$,
$\subspace{10442\\01323}{}$,
$\subspace{11003\\00130}{}$,
$\subspace{13022\\00132}{}$,
$\subspace{00110\\00001}{11}$,
$\subspace{01100\\00013}{}$,
$\subspace{10000\\00011}{}$,
$\subspace{10002\\01003}{}$,
$\subspace{10034\\01302}{}$,
$\subspace{10210\\01120}{}$,
$\subspace{10443\\01133}{}$,
$\subspace{10400\\01300}{}$,
$\subspace{10442\\01323}{}$,
$\subspace{11003\\00130}{}$,
$\subspace{13022\\00132}{}$,
$\subspace{00110\\00001}{11}$,
$\subspace{01100\\00013}{}$,
$\subspace{10000\\00011}{}$,
$\subspace{10002\\01003}{}$,
$\subspace{10034\\01302}{}$,
$\subspace{10210\\01120}{}$,
$\subspace{10443\\01133}{}$,
$\subspace{10400\\01300}{}$,
$\subspace{10442\\01323}{}$,
$\subspace{11003\\00130}{}$,
$\subspace{13022\\00132}{}$,
$\subspace{00110\\00001}{11}$,
$\subspace{01100\\00013}{}$,
$\subspace{10000\\00011}{}$,
$\subspace{10002\\01003}{}$,
$\subspace{10034\\01302}{}$,
$\subspace{10210\\01120}{}$,
$\subspace{10443\\01133}{}$,
$\subspace{10400\\01300}{}$,
$\subspace{10442\\01323}{}$,
$\subspace{11003\\00130}{}$,
$\subspace{13022\\00132}{}$,
$\subspace{01004\\00113}{55}$,
$\subspace{01040\\00112}{}$,
$\subspace{10022\\00121}{}$,
$\subspace{10044\\00100}{}$,
$\subspace{10043\\01021}{}$,
$\subspace{10040\\01034}{}$,
$\subspace{10001\\01142}{}$,
$\subspace{10041\\01134}{}$,
$\subspace{10041\\01232}{}$,
$\subspace{10011\\01310}{}$,
$\subspace{10002\\01424}{}$,
$\subspace{10013\\01413}{}$,
$\subspace{10033\\01423}{}$,
$\subspace{10034\\01443}{}$,
$\subspace{10040\\01432}{}$,
$\subspace{10111\\01024}{}$,
$\subspace{10103\\01102}{}$,
$\subspace{10102\\01300}{}$,
$\subspace{10130\\01442}{}$,
$\subspace{10244\\01024}{}$,
$\subspace{10241\\01032}{}$,
$\subspace{10210\\01140}{}$,
$\subspace{10243\\01114}{}$,
$\subspace{10243\\01220}{}$,
$\subspace{10223\\01303}{}$,
$\subspace{10233\\01320}{}$,
$\subspace{10223\\01424}{}$,
$\subspace{10231\\01402}{}$,
$\subspace{10314\\01004}{}$,
$\subspace{10343\\01040}{}$,
$\subspace{10321\\01131}{}$,
$\subspace{10330\\01120}{}$,
$\subspace{10334\\01140}{}$,
$\subspace{10341\\01100}{}$,
$\subspace{10323\\01233}{}$,
$\subspace{10322\\01241}{}$,
$\subspace{10413\\01003}{}$,
$\subspace{10403\\01121}{}$,
$\subspace{10410\\01113}{}$,
$\subspace{10400\\01241}{}$,
$\subspace{10403\\01312}{}$,
$\subspace{10402\\01323}{}$,
$\subspace{10401\\01430}{}$,
$\subspace{10443\\01421}{}$,
$\subspace{11022\\00142}{}$,
$\subspace{11400\\00014}{}$,
$\subspace{12000\\00142}{}$,
$\subspace{13000\\00114}{}$,
$\subspace{13042\\00110}{}$,
$\subspace{14004\\00131}{}$,
$\subspace{14001\\00143}{}$,
$\subspace{14014\\00111}{}$,
$\subspace{14010\\00132}{}$,
$\subspace{14024\\00100}{}$,
$\subspace{14200\\00011}{}$,
$\subspace{01002\\00134}{55}$,
$\subspace{01010\\00122}{}$,
$\subspace{01010\\00131}{}$,
$\subspace{01011\\00143}{}$,
$\subspace{10013\\00122}{}$,
$\subspace{10042\\00134}{}$,
$\subspace{10202\\00012}{}$,
$\subspace{10404\\00014}{}$,
$\subspace{10004\\01021}{}$,
$\subspace{10001\\01044}{}$,
$\subspace{10042\\01012}{}$,
$\subspace{10004\\01103}{}$,
$\subspace{10003\\01112}{}$,
$\subspace{10001\\01331}{}$,
$\subspace{10013\\01301}{}$,
$\subspace{10021\\01321}{}$,
$\subspace{10020\\01344}{}$,
$\subspace{10102\\01000}{}$,
$\subspace{10100\\01222}{}$,
$\subspace{10100\\01310}{}$,
$\subspace{10102\\01343}{}$,
$\subspace{10113\\01304}{}$,
$\subspace{10122\\01303}{}$,
$\subspace{10113\\01432}{}$,
$\subspace{10130\\01432}{}$,
$\subspace{10220\\01000}{}$,
$\subspace{10203\\01104}{}$,
$\subspace{10203\\01130}{}$,
$\subspace{10233\\01112}{}$,
$\subspace{10214\\01343}{}$,
$\subspace{10232\\01310}{}$,
$\subspace{10232\\01443}{}$,
$\subspace{10312\\01044}{}$,
$\subspace{10320\\01041}{}$,
$\subspace{10340\\01011}{}$,
$\subspace{10334\\01213}{}$,
$\subspace{10342\\01414}{}$,
$\subspace{10421\\01021}{}$,
$\subspace{10404\\01130}{}$,
$\subspace{10420\\01131}{}$,
$\subspace{10421\\01223}{}$,
$\subspace{10440\\01240}{}$,
$\subspace{10434\\01341}{}$,
$\subspace{10444\\01301}{}$,
$\subspace{10411\\01400}{}$,
$\subspace{10420\\01414}{}$,
$\subspace{10444\\01444}{}$,
$\subspace{11001\\00012}{}$,
$\subspace{11004\\00101}{}$,
$\subspace{11013\\00143}{}$,
$\subspace{11100\\00010}{}$,
$\subspace{13000\\00133}{}$,
$\subspace{13014\\00131}{}$,
$\subspace{13041\\00104}{}$,
$\subspace{13303\\00014}{}$,
$\subspace{01012\\00102}{55}$,
$\subspace{01023\\00120}{}$,
$\subspace{10010\\00101}{}$,
$\subspace{10024\\00140}{}$,
$\subspace{10032\\01034}{}$,
$\subspace{10024\\01110}{}$,
$\subspace{10043\\01101}{}$,
$\subspace{10003\\01200}{}$,
$\subspace{10012\\01240}{}$,
$\subspace{10014\\01313}{}$,
$\subspace{10012\\01322}{}$,
$\subspace{10021\\01410}{}$,
$\subspace{10121\\01020}{}$,
$\subspace{10134\\01032}{}$,
$\subspace{10141\\01031}{}$,
$\subspace{10103\\01101}{}$,
$\subspace{10101\\01141}{}$,
$\subspace{10114\\01214}{}$,
$\subspace{10110\\01230}{}$,
$\subspace{10124\\01230}{}$,
$\subspace{10110\\01334}{}$,
$\subspace{10141\\01442}{}$,
$\subspace{10212\\01132}{}$,
$\subspace{10211\\01144}{}$,
$\subspace{10201\\01210}{}$,
$\subspace{10230\\01201}{}$,
$\subspace{10224\\01320}{}$,
$\subspace{10213\\01401}{}$,
$\subspace{10221\\01413}{}$,
$\subspace{10231\\01400}{}$,
$\subspace{10312\\01001}{}$,
$\subspace{10320\\01203}{}$,
$\subspace{10340\\01201}{}$,
$\subspace{10300\\01342}{}$,
$\subspace{10322\\01331}{}$,
$\subspace{10323\\01340}{}$,
$\subspace{10333\\01330}{}$,
$\subspace{10300\\01431}{}$,
$\subspace{10310\\01412}{}$,
$\subspace{10332\\01412}{}$,
$\subspace{10341\\01440}{}$,
$\subspace{10402\\01013}{}$,
$\subspace{10423\\01001}{}$,
$\subspace{10414\\01143}{}$,
$\subspace{10441\\01124}{}$,
$\subspace{10422\\01244}{}$,
$\subspace{10412\\01304}{}$,
$\subspace{10430\\01430}{}$,
$\subspace{11011\\00121}{}$,
$\subspace{11034\\00120}{}$,
$\subspace{12002\\00010}{}$,
$\subspace{12024\\00123}{}$,
$\subspace{13033\\00123}{}$,
$\subspace{13042\\00144}{}$,
$\subspace{14023\\00133}{}$,
$\subspace{01012\\00102}{55}$,
$\subspace{01023\\00120}{}$,
$\subspace{10010\\00101}{}$,
$\subspace{10024\\00140}{}$,
$\subspace{10032\\01034}{}$,
$\subspace{10024\\01110}{}$,
$\subspace{10043\\01101}{}$,
$\subspace{10003\\01200}{}$,
$\subspace{10012\\01240}{}$,
$\subspace{10014\\01313}{}$,
$\subspace{10012\\01322}{}$,
$\subspace{10021\\01410}{}$,
$\subspace{10121\\01020}{}$,
$\subspace{10134\\01032}{}$,
$\subspace{10141\\01031}{}$,
$\subspace{10103\\01101}{}$,
$\subspace{10101\\01141}{}$,
$\subspace{10114\\01214}{}$,
$\subspace{10110\\01230}{}$,
$\subspace{10124\\01230}{}$,
$\subspace{10110\\01334}{}$,
$\subspace{10141\\01442}{}$,
$\subspace{10212\\01132}{}$,
$\subspace{10211\\01144}{}$,
$\subspace{10201\\01210}{}$,
$\subspace{10230\\01201}{}$,
$\subspace{10224\\01320}{}$,
$\subspace{10213\\01401}{}$,
$\subspace{10221\\01413}{}$,
$\subspace{10231\\01400}{}$,
$\subspace{10312\\01001}{}$,
$\subspace{10320\\01203}{}$,
$\subspace{10340\\01201}{}$,
$\subspace{10300\\01342}{}$,
$\subspace{10322\\01331}{}$,
$\subspace{10323\\01340}{}$,
$\subspace{10333\\01330}{}$,
$\subspace{10300\\01431}{}$,
$\subspace{10310\\01412}{}$,
$\subspace{10332\\01412}{}$,
$\subspace{10341\\01440}{}$,
$\subspace{10402\\01013}{}$,
$\subspace{10423\\01001}{}$,
$\subspace{10414\\01143}{}$,
$\subspace{10441\\01124}{}$,
$\subspace{10422\\01244}{}$,
$\subspace{10412\\01304}{}$,
$\subspace{10430\\01430}{}$,
$\subspace{11011\\00121}{}$,
$\subspace{11034\\00120}{}$,
$\subspace{12002\\00010}{}$,
$\subspace{12024\\00123}{}$,
$\subspace{13033\\00123}{}$,
$\subspace{13042\\00144}{}$,
$\subspace{14023\\00133}{}$,
$\subspace{01014\\00103}{55}$,
$\subspace{01042\\00124}{}$,
$\subspace{10031\\00123}{}$,
$\subspace{10030\\00141}{}$,
$\subspace{10012\\01022}{}$,
$\subspace{10023\\01012}{}$,
$\subspace{10001\\01122}{}$,
$\subspace{10030\\01141}{}$,
$\subspace{10013\\01213}{}$,
$\subspace{10024\\01240}{}$,
$\subspace{10003\\01311}{}$,
$\subspace{10031\\01310}{}$,
$\subspace{10021\\01412}{}$,
$\subspace{10023\\01422}{}$,
$\subspace{10130\\01033}{}$,
$\subspace{10142\\01033}{}$,
$\subspace{10102\\01132}{}$,
$\subspace{10131\\01131}{}$,
$\subspace{10132\\01143}{}$,
$\subspace{10114\\01231}{}$,
$\subspace{10110\\01243}{}$,
$\subspace{10132\\01201}{}$,
$\subspace{10112\\01324}{}$,
$\subspace{10112\\01331}{}$,
$\subspace{10123\\01443}{}$,
$\subspace{10141\\01411}{}$,
$\subspace{10214\\01022}{}$,
$\subspace{10221\\01123}{}$,
$\subspace{10233\\01101}{}$,
$\subspace{10222\\01321}{}$,
$\subspace{10234\\01410}{}$,
$\subspace{10301\\01001}{}$,
$\subspace{10301\\01021}{}$,
$\subspace{10320\\01123}{}$,
$\subspace{10300\\01330}{}$,
$\subspace{10312\\01314}{}$,
$\subspace{10331\\01303}{}$,
$\subspace{10334\\01400}{}$,
$\subspace{10331\\01420}{}$,
$\subspace{10433\\01111}{}$,
$\subspace{10434\\01122}{}$,
$\subspace{10424\\01243}{}$,
$\subspace{10441\\01221}{}$,
$\subspace{10422\\01324}{}$,
$\subspace{10414\\01432}{}$,
$\subspace{11013\\00103}{}$,
$\subspace{11203\\00014}{}$,
$\subspace{12004\\00101}{}$,
$\subspace{12001\\00131}{}$,
$\subspace{12021\\00120}{}$,
$\subspace{12032\\00143}{}$,
$\subspace{13033\\00124}{}$,
$\subspace{13100\\00010}{}$,
$\subspace{14044\\00133}{}$,
$\subspace{14041\\00141}{}$,
$\subspace{01010\\00143}{55}$,
$\subspace{01023\\00134}{}$,
$\subspace{10042\\00131}{}$,
$\subspace{10202\\00014}{}$,
$\subspace{10001\\01024}{}$,
$\subspace{10013\\01041}{}$,
$\subspace{10020\\01140}{}$,
$\subspace{10041\\01110}{}$,
$\subspace{10014\\01303}{}$,
$\subspace{10040\\01314}{}$,
$\subspace{10004\\01434}{}$,
$\subspace{10113\\01014}{}$,
$\subspace{10131\\01044}{}$,
$\subspace{10142\\01130}{}$,
$\subspace{10101\\01321}{}$,
$\subspace{10104\\01343}{}$,
$\subspace{10130\\01341}{}$,
$\subspace{10100\\01432}{}$,
$\subspace{10102\\01440}{}$,
$\subspace{10122\\01424}{}$,
$\subspace{10134\\01414}{}$,
$\subspace{10223\\01002}{}$,
$\subspace{10232\\01013}{}$,
$\subspace{10240\\01040}{}$,
$\subspace{10203\\01112}{}$,
$\subspace{10233\\01220}{}$,
$\subspace{10200\\01322}{}$,
$\subspace{10214\\01301}{}$,
$\subspace{10230\\01344}{}$,
$\subspace{10244\\01401}{}$,
$\subspace{10303\\01011}{}$,
$\subspace{10324\\01000}{}$,
$\subspace{10333\\01004}{}$,
$\subspace{10313\\01104}{}$,
$\subspace{10334\\01211}{}$,
$\subspace{10342\\01221}{}$,
$\subspace{10413\\01042}{}$,
$\subspace{10420\\01021}{}$,
$\subspace{10404\\01131}{}$,
$\subspace{10403\\01204}{}$,
$\subspace{10401\\01222}{}$,
$\subspace{10412\\01223}{}$,
$\subspace{10444\\01242}{}$,
$\subspace{10433\\01340}{}$,
$\subspace{10440\\01311}{}$,
$\subspace{10424\\01423}{}$,
$\subspace{10421\\01431}{}$,
$\subspace{10423\\01443}{}$,
$\subspace{10431\\01420}{}$,
$\subspace{11001\\00104}{}$,
$\subspace{11012\\00142}{}$,
$\subspace{12030\\00122}{}$,
$\subspace{13030\\00144}{}$,
$\subspace{13044\\00100}{}$,
$\subspace{14302\\00012}{}$,
$\subspace{01020\\00141}{55}$,
$\subspace{01030\\00124}{}$,
$\subspace{10023\\00102}{}$,
$\subspace{10031\\00141}{}$,
$\subspace{10010\\01022}{}$,
$\subspace{10023\\01044}{}$,
$\subspace{10004\\01122}{}$,
$\subspace{10030\\01144}{}$,
$\subspace{10030\\01212}{}$,
$\subspace{10042\\01242}{}$,
$\subspace{10032\\01324}{}$,
$\subspace{10031\\01344}{}$,
$\subspace{10020\\01414}{}$,
$\subspace{10112\\01000}{}$,
$\subspace{10112\\01041}{}$,
$\subspace{10132\\01033}{}$,
$\subspace{10140\\01031}{}$,
$\subspace{10100\\01111}{}$,
$\subspace{10132\\01130}{}$,
$\subspace{10113\\01234}{}$,
$\subspace{10124\\01222}{}$,
$\subspace{10133\\01204}{}$,
$\subspace{10202\\01022}{}$,
$\subspace{10222\\01010}{}$,
$\subspace{10232\\01002}{}$,
$\subspace{10200\\01123}{}$,
$\subspace{10242\\01124}{}$,
$\subspace{10220\\01200}{}$,
$\subspace{10224\\01243}{}$,
$\subspace{10240\\01203}{}$,
$\subspace{10203\\01403}{}$,
$\subspace{10201\\01411}{}$,
$\subspace{10301\\01103}{}$,
$\subspace{10301\\01224}{}$,
$\subspace{10302\\01243}{}$,
$\subspace{10342\\01244}{}$,
$\subspace{10311\\01341}{}$,
$\subspace{10331\\01301}{}$,
$\subspace{10331\\01324}{}$,
$\subspace{10310\\01404}{}$,
$\subspace{10332\\01444}{}$,
$\subspace{10404\\01043}{}$,
$\subspace{10420\\01033}{}$,
$\subspace{10432\\01011}{}$,
$\subspace{10411\\01122}{}$,
$\subspace{10431\\01123}{}$,
$\subspace{10430\\01332}{}$,
$\subspace{11042\\00103}{}$,
$\subspace{11202\\00012}{}$,
$\subspace{12001\\00134}{}$,
$\subspace{12021\\00103}{}$,
$\subspace{12024\\00122}{}$,
$\subspace{13011\\00124}{}$,
$\subspace{13010\\00140}{}$,
$\subspace{14044\\00104}{}$,
$\subspace{10004\\01233}{55}$,
$\subspace{10042\\01212}{}$,
$\subspace{10014\\01421}{}$,
$\subspace{10113\\01030}{}$,
$\subspace{10140\\01014}{}$,
$\subspace{10123\\01104}{}$,
$\subspace{10100\\01224}{}$,
$\subspace{10104\\01222}{}$,
$\subspace{10120\\01343}{}$,
$\subspace{10134\\01301}{}$,
$\subspace{10143\\01311}{}$,
$\subspace{10143\\01332}{}$,
$\subspace{10144\\01340}{}$,
$\subspace{10101\\01441}{}$,
$\subspace{10131\\01433}{}$,
$\subspace{10133\\01431}{}$,
$\subspace{10142\\01404}{}$,
$\subspace{10144\\01421}{}$,
$\subspace{10241\\01000}{}$,
$\subspace{10241\\01042}{}$,
$\subspace{10203\\01234}{}$,
$\subspace{10230\\01221}{}$,
$\subspace{10202\\01334}{}$,
$\subspace{10204\\01344}{}$,
$\subspace{10211\\01314}{}$,
$\subspace{10242\\01312}{}$,
$\subspace{10232\\01441}{}$,
$\subspace{10234\\01440}{}$,
$\subspace{10304\\01011}{}$,
$\subspace{10311\\01010}{}$,
$\subspace{10313\\01043}{}$,
$\subspace{10314\\01044}{}$,
$\subspace{10303\\01110}{}$,
$\subspace{10302\\01130}{}$,
$\subspace{10314\\01231}{}$,
$\subspace{10333\\01202}{}$,
$\subspace{10343\\01312}{}$,
$\subspace{10343\\01322}{}$,
$\subspace{10321\\01414}{}$,
$\subspace{10321\\01422}{}$,
$\subspace{10420\\01013}{}$,
$\subspace{10433\\01023}{}$,
$\subspace{10404\\01232}{}$,
$\subspace{10424\\01233}{}$,
$\subspace{10432\\01211}{}$,
$\subspace{10440\\01232}{}$,
$\subspace{10410\\01401}{}$,
$\subspace{10410\\01402}{}$,
$\subspace{10412\\01420}{}$,
$\subspace{10423\\01402}{}$,
$\subspace{10421\\01433}{}$,
$\subspace{12040\\00134}{}$,
$\subspace{12303\\00012}{}$,
$\subspace{14032\\00144}{}$,
$\subspace{14040\\00122}{}$,
$\subspace{10040\\01203}{55}$,
$\subspace{10041\\01200}{}$,
$\subspace{10040\\01212}{}$,
$\subspace{10041\\01244}{}$,
$\subspace{10010\\01424}{}$,
$\subspace{10032\\01420}{}$,
$\subspace{10104\\01020}{}$,
$\subspace{10124\\01004}{}$,
$\subspace{10131\\01102}{}$,
$\subspace{10133\\01140}{}$,
$\subspace{10140\\01102}{}$,
$\subspace{10142\\01140}{}$,
$\subspace{10133\\01220}{}$,
$\subspace{10120\\01332}{}$,
$\subspace{10140\\01332}{}$,
$\subspace{10120\\01404}{}$,
$\subspace{10123\\01404}{}$,
$\subspace{10201\\01042}{}$,
$\subspace{10223\\01014}{}$,
$\subspace{10234\\01040}{}$,
$\subspace{10242\\01024}{}$,
$\subspace{10243\\01030}{}$,
$\subspace{10224\\01202}{}$,
$\subspace{10223\\01231}{}$,
$\subspace{10244\\01202}{}$,
$\subspace{10243\\01212}{}$,
$\subspace{10242\\01241}{}$,
$\subspace{10213\\01311}{}$,
$\subspace{10244\\01342}{}$,
$\subspace{10304\\01024}{}$,
$\subspace{10332\\01030}{}$,
$\subspace{10302\\01144}{}$,
$\subspace{10330\\01124}{}$,
$\subspace{10303\\01224}{}$,
$\subspace{10304\\01234}{}$,
$\subspace{10302\\01241}{}$,
$\subspace{10311\\01220}{}$,
$\subspace{10311\\01224}{}$,
$\subspace{10310\\01314}{}$,
$\subspace{10330\\01403}{}$,
$\subspace{10401\\01043}{}$,
$\subspace{10424\\01031}{}$,
$\subspace{10430\\01004}{}$,
$\subspace{10432\\01040}{}$,
$\subspace{10433\\01043}{}$,
$\subspace{10403\\01234}{}$,
$\subspace{10432\\01221}{}$,
$\subspace{10401\\01424}{}$,
$\subspace{10403\\01422}{}$,
$\subspace{12002\\00142}{}$,
$\subspace{12024\\00100}{}$,
$\subspace{12031\\00100}{}$,
$\subspace{12030\\00142}{}$,
$\subspace{14022\\00140}{}$,
$\subspace{14040\\00102}{}$,
$\subspace{10043\\01242}{55}$,
$\subspace{10020\\01433}{}$,
$\subspace{10020\\01444}{}$,
$\subspace{10040\\01422}{}$,
$\subspace{10041\\01442}{}$,
$\subspace{10104\\01041}{}$,
$\subspace{10143\\01002}{}$,
$\subspace{10131\\01114}{}$,
$\subspace{10144\\01223}{}$,
$\subspace{10111\\01341}{}$,
$\subspace{10122\\01320}{}$,
$\subspace{10123\\01333}{}$,
$\subspace{10103\\01434}{}$,
$\subspace{10122\\01421}{}$,
$\subspace{10142\\01430}{}$,
$\subspace{10200\\01041}{}$,
$\subspace{10231\\01014}{}$,
$\subspace{10220\\01204}{}$,
$\subspace{10234\\01233}{}$,
$\subspace{10241\\01214}{}$,
$\subspace{10240\\01231}{}$,
$\subspace{10200\\01314}{}$,
$\subspace{10220\\01312}{}$,
$\subspace{10244\\01341}{}$,
$\subspace{10223\\01420}{}$,
$\subspace{10240\\01402}{}$,
$\subspace{10243\\01413}{}$,
$\subspace{10322\\01004}{}$,
$\subspace{10323\\01024}{}$,
$\subspace{10324\\01034}{}$,
$\subspace{10314\\01102}{}$,
$\subspace{10324\\01103}{}$,
$\subspace{10343\\01103}{}$,
$\subspace{10342\\01140}{}$,
$\subspace{10303\\01220}{}$,
$\subspace{10321\\01241}{}$,
$\subspace{10341\\01221}{}$,
$\subspace{10342\\01232}{}$,
$\subspace{10330\\01311}{}$,
$\subspace{10344\\01434}{}$,
$\subspace{10403\\01032}{}$,
$\subspace{10402\\01040}{}$,
$\subspace{10411\\01002}{}$,
$\subspace{10411\\01042}{}$,
$\subspace{10410\\01204}{}$,
$\subspace{10433\\01223}{}$,
$\subspace{10431\\01242}{}$,
$\subspace{10401\\01441}{}$,
$\subspace{10424\\01444}{}$,
$\subspace{10431\\01424}{}$,
$\subspace{12043\\00121}{}$,
$\subspace{14012\\00142}{}$,
$\subspace{14023\\00100}{}$,
$\subspace{14030\\00104}{}$,
$\subspace{14040\\00104}{}$.

\smallskip

$n_5(5,2;37)\ge 927$,$\left[\!\begin{smallmatrix}10000\\00100\\04410\\00001\\00044\end{smallmatrix}\!\right]$:
$\subspace{01201\\00011}{3}$,
$\subspace{01303\\00010}{}$,
$\subspace{01400\\00001}{}$,
$\subspace{01201\\00011}{3}$,
$\subspace{01303\\00010}{}$,
$\subspace{01400\\00001}{}$,
$\subspace{01201\\00011}{3}$,
$\subspace{01303\\00010}{}$,
$\subspace{01400\\00001}{}$,
$\subspace{01201\\00011}{3}$,
$\subspace{01303\\00010}{}$,
$\subspace{01400\\00001}{}$,
$\subspace{10000\\00012}{3}$,
$\subspace{10000\\00013}{}$,
$\subspace{10000\\00014}{}$,
$\subspace{10013\\00102}{15}$,
$\subspace{10013\\00114}{}$,
$\subspace{10013\\00121}{}$,
$\subspace{10013\\00133}{}$,
$\subspace{10013\\00140}{}$,
$\subspace{10024\\01002}{}$,
$\subspace{10024\\01011}{}$,
$\subspace{10024\\01020}{}$,
$\subspace{10024\\01034}{}$,
$\subspace{10024\\01043}{}$,
$\subspace{10023\\01104}{}$,
$\subspace{10023\\01112}{}$,
$\subspace{10023\\01120}{}$,
$\subspace{10023\\01133}{}$,
$\subspace{10023\\01141}{}$,
$\subspace{10030\\00104}{15}$,
$\subspace{10030\\00111}{}$,
$\subspace{10030\\00123}{}$,
$\subspace{10030\\00130}{}$,
$\subspace{10030\\00142}{}$,
$\subspace{10022\\01004}{}$,
$\subspace{10022\\01013}{}$,
$\subspace{10022\\01022}{}$,
$\subspace{10022\\01031}{}$,
$\subspace{10022\\01040}{}$,
$\subspace{10003\\01100}{}$,
$\subspace{10003\\01113}{}$,
$\subspace{10003\\01121}{}$,
$\subspace{10003\\01134}{}$,
$\subspace{10003\\01142}{}$,
$\subspace{10101\\00012}{3}$,
$\subspace{11001\\00014}{}$,
$\subspace{14404\\00013}{}$,
$\subspace{10200\\00012}{3}$,
$\subspace{12000\\00014}{}$,
$\subspace{13304\\00013}{}$,
$\subspace{10302\\00012}{3}$,
$\subspace{12200\\00013}{}$,
$\subspace{13002\\00014}{}$,
$\subspace{10402\\00012}{3}$,
$\subspace{11102\\00013}{}$,
$\subspace{14002\\00014}{}$,
$\subspace{10002\\01202}{15}$,
$\subspace{10002\\01213}{}$,
$\subspace{10002\\01224}{}$,
$\subspace{10002\\01230}{}$,
$\subspace{10002\\01241}{}$,
$\subspace{10033\\01301}{}$,
$\subspace{10033\\01311}{}$,
$\subspace{10033\\01321}{}$,
$\subspace{10033\\01331}{}$,
$\subspace{10033\\01341}{}$,
$\subspace{10020\\01410}{}$,
$\subspace{10020\\01411}{}$,
$\subspace{10020\\01412}{}$,
$\subspace{10020\\01413}{}$,
$\subspace{10020\\01414}{}$,
$\subspace{10010\\01202}{15}$,
$\subspace{10010\\01213}{}$,
$\subspace{10010\\01224}{}$,
$\subspace{10010\\01230}{}$,
$\subspace{10010\\01241}{}$,
$\subspace{10001\\01301}{}$,
$\subspace{10001\\01311}{}$,
$\subspace{10001\\01321}{}$,
$\subspace{10001\\01331}{}$,
$\subspace{10001\\01341}{}$,
$\subspace{10044\\01410}{}$,
$\subspace{10044\\01411}{}$,
$\subspace{10044\\01412}{}$,
$\subspace{10044\\01413}{}$,
$\subspace{10044\\01414}{}$,
$\subspace{10014\\01200}{15}$,
$\subspace{10014\\01211}{}$,
$\subspace{10014\\01222}{}$,
$\subspace{10014\\01233}{}$,
$\subspace{10014\\01244}{}$,
$\subspace{10012\\01300}{}$,
$\subspace{10012\\01310}{}$,
$\subspace{10012\\01320}{}$,
$\subspace{10012\\01330}{}$,
$\subspace{10012\\01340}{}$,
$\subspace{10034\\01440}{}$,
$\subspace{10034\\01441}{}$,
$\subspace{10034\\01442}{}$,
$\subspace{10034\\01443}{}$,
$\subspace{10034\\01444}{}$,
$\subspace{10031\\01202}{15}$,
$\subspace{10031\\01213}{}$,
$\subspace{10031\\01224}{}$,
$\subspace{10031\\01230}{}$,
$\subspace{10031\\01241}{}$,
$\subspace{10042\\01301}{}$,
$\subspace{10042\\01311}{}$,
$\subspace{10042\\01321}{}$,
$\subspace{10042\\01331}{}$,
$\subspace{10042\\01341}{}$,
$\subspace{10032\\01410}{}$,
$\subspace{10032\\01411}{}$,
$\subspace{10032\\01412}{}$,
$\subspace{10032\\01413}{}$,
$\subspace{10032\\01414}{}$,
$\subspace{10040\\01200}{15}$,
$\subspace{10040\\01211}{}$,
$\subspace{10040\\01222}{}$,
$\subspace{10040\\01233}{}$,
$\subspace{10040\\01244}{}$,
$\subspace{10004\\01300}{}$,
$\subspace{10004\\01310}{}$,
$\subspace{10004\\01320}{}$,
$\subspace{10004\\01330}{}$,
$\subspace{10004\\01340}{}$,
$\subspace{10011\\01440}{}$,
$\subspace{10011\\01441}{}$,
$\subspace{10011\\01442}{}$,
$\subspace{10011\\01443}{}$,
$\subspace{10011\\01444}{}$,
$\subspace{10041\\01203}{15}$,
$\subspace{10041\\01214}{}$,
$\subspace{10041\\01220}{}$,
$\subspace{10041\\01231}{}$,
$\subspace{10041\\01242}{}$,
$\subspace{10043\\01304}{}$,
$\subspace{10043\\01314}{}$,
$\subspace{10043\\01324}{}$,
$\subspace{10043\\01334}{}$,
$\subspace{10043\\01344}{}$,
$\subspace{10021\\01420}{}$,
$\subspace{10021\\01421}{}$,
$\subspace{10021\\01422}{}$,
$\subspace{10021\\01423}{}$,
$\subspace{10021\\01424}{}$,
$\subspace{10100\\01012}{15}$,
$\subspace{10112\\01021}{}$,
$\subspace{10124\\01030}{}$,
$\subspace{10131\\01044}{}$,
$\subspace{10143\\01003}{}$,
$\subspace{10403\\01102}{}$,
$\subspace{10410\\01131}{}$,
$\subspace{10422\\01110}{}$,
$\subspace{10434\\01144}{}$,
$\subspace{10441\\01123}{}$,
$\subspace{14003\\00122}{}$,
$\subspace{14012\\00110}{}$,
$\subspace{14021\\00103}{}$,
$\subspace{14030\\00141}{}$,
$\subspace{14044\\00134}{}$,
$\subspace{10101\\01014}{15}$,
$\subspace{10113\\01023}{}$,
$\subspace{10120\\01032}{}$,
$\subspace{10132\\01041}{}$,
$\subspace{10144\\01000}{}$,
$\subspace{10402\\01140}{}$,
$\subspace{10414\\01124}{}$,
$\subspace{10421\\01103}{}$,
$\subspace{10433\\01132}{}$,
$\subspace{10440\\01111}{}$,
$\subspace{14002\\00124}{}$,
$\subspace{14011\\00112}{}$,
$\subspace{14020\\00100}{}$,
$\subspace{14034\\00143}{}$,
$\subspace{14043\\00131}{}$,
$\subspace{10100\\01023}{15}$,
$\subspace{10112\\01032}{}$,
$\subspace{10124\\01041}{}$,
$\subspace{10131\\01000}{}$,
$\subspace{10143\\01014}{}$,
$\subspace{10402\\01103}{}$,
$\subspace{10414\\01132}{}$,
$\subspace{10421\\01111}{}$,
$\subspace{10433\\01140}{}$,
$\subspace{10440\\01124}{}$,
$\subspace{14002\\00100}{}$,
$\subspace{14011\\00143}{}$,
$\subspace{14020\\00131}{}$,
$\subspace{14034\\00124}{}$,
$\subspace{14043\\00112}{}$,
$\subspace{10102\\01023}{15}$,
$\subspace{10114\\01032}{}$,
$\subspace{10121\\01041}{}$,
$\subspace{10133\\01000}{}$,
$\subspace{10140\\01014}{}$,
$\subspace{10403\\01140}{}$,
$\subspace{10410\\01124}{}$,
$\subspace{10422\\01103}{}$,
$\subspace{10434\\01132}{}$,
$\subspace{10441\\01111}{}$,
$\subspace{14003\\00131}{}$,
$\subspace{14012\\00124}{}$,
$\subspace{14021\\00112}{}$,
$\subspace{14030\\00100}{}$,
$\subspace{14044\\00143}{}$,
$\subspace{10103\\01020}{15}$,
$\subspace{10110\\01034}{}$,
$\subspace{10122\\01043}{}$,
$\subspace{10134\\01002}{}$,
$\subspace{10141\\01011}{}$,
$\subspace{10402\\01133}{}$,
$\subspace{10414\\01112}{}$,
$\subspace{10421\\01141}{}$,
$\subspace{10433\\01120}{}$,
$\subspace{10440\\01104}{}$,
$\subspace{14002\\00133}{}$,
$\subspace{14011\\00121}{}$,
$\subspace{14020\\00114}{}$,
$\subspace{14034\\00102}{}$,
$\subspace{14043\\00140}{}$,
$\subspace{10101\\01032}{15}$,
$\subspace{10113\\01041}{}$,
$\subspace{10120\\01000}{}$,
$\subspace{10132\\01014}{}$,
$\subspace{10144\\01023}{}$,
$\subspace{10403\\01103}{}$,
$\subspace{10410\\01132}{}$,
$\subspace{10422\\01111}{}$,
$\subspace{10434\\01140}{}$,
$\subspace{10441\\01124}{}$,
$\subspace{14003\\00112}{}$,
$\subspace{14012\\00100}{}$,
$\subspace{14021\\00143}{}$,
$\subspace{14030\\00131}{}$,
$\subspace{14044\\00124}{}$,
$\subspace{10101\\01033}{15}$,
$\subspace{10113\\01042}{}$,
$\subspace{10120\\01001}{}$,
$\subspace{10132\\01010}{}$,
$\subspace{10144\\01024}{}$,
$\subspace{10401\\01114}{}$,
$\subspace{10413\\01143}{}$,
$\subspace{10420\\01122}{}$,
$\subspace{10432\\01101}{}$,
$\subspace{10444\\01130}{}$,
$\subspace{14001\\00144}{}$,
$\subspace{14010\\00132}{}$,
$\subspace{14024\\00120}{}$,
$\subspace{14033\\00113}{}$,
$\subspace{14042\\00101}{}$,
$\subspace{10102\\01104}{15}$,
$\subspace{10114\\01120}{}$,
$\subspace{10121\\01141}{}$,
$\subspace{10133\\01112}{}$,
$\subspace{10140\\01133}{}$,
$\subspace{10403\\01002}{}$,
$\subspace{10410\\01043}{}$,
$\subspace{10422\\01034}{}$,
$\subspace{10434\\01020}{}$,
$\subspace{10441\\01011}{}$,
$\subspace{11002\\00102}{}$,
$\subspace{11011\\00114}{}$,
$\subspace{11020\\00121}{}$,
$\subspace{11034\\00133}{}$,
$\subspace{11043\\00140}{}$,
$\subspace{10104\\01121}{15}$,
$\subspace{10111\\01142}{}$,
$\subspace{10123\\01113}{}$,
$\subspace{10130\\01134}{}$,
$\subspace{10142\\01100}{}$,
$\subspace{10401\\01022}{}$,
$\subspace{10413\\01013}{}$,
$\subspace{10420\\01004}{}$,
$\subspace{10432\\01040}{}$,
$\subspace{10444\\01031}{}$,
$\subspace{11004\\00130}{}$,
$\subspace{11013\\00142}{}$,
$\subspace{11022\\00104}{}$,
$\subspace{11031\\00111}{}$,
$\subspace{11040\\00123}{}$,
$\subspace{10104\\01121}{15}$,
$\subspace{10111\\01142}{}$,
$\subspace{10123\\01113}{}$,
$\subspace{10130\\01134}{}$,
$\subspace{10142\\01100}{}$,
$\subspace{10401\\01022}{}$,
$\subspace{10413\\01013}{}$,
$\subspace{10420\\01004}{}$,
$\subspace{10432\\01040}{}$,
$\subspace{10444\\01031}{}$,
$\subspace{11004\\00130}{}$,
$\subspace{11013\\00142}{}$,
$\subspace{11022\\00104}{}$,
$\subspace{11031\\00111}{}$,
$\subspace{11040\\00123}{}$,
$\subspace{10104\\01121}{15}$,
$\subspace{10111\\01142}{}$,
$\subspace{10123\\01113}{}$,
$\subspace{10130\\01134}{}$,
$\subspace{10142\\01100}{}$,
$\subspace{10401\\01022}{}$,
$\subspace{10413\\01013}{}$,
$\subspace{10420\\01004}{}$,
$\subspace{10432\\01040}{}$,
$\subspace{10444\\01031}{}$,
$\subspace{11004\\00130}{}$,
$\subspace{11013\\00142}{}$,
$\subspace{11022\\00104}{}$,
$\subspace{11031\\00111}{}$,
$\subspace{11040\\00123}{}$,
$\subspace{10102\\01144}{15}$,
$\subspace{10114\\01110}{}$,
$\subspace{10121\\01131}{}$,
$\subspace{10133\\01102}{}$,
$\subspace{10140\\01123}{}$,
$\subspace{10404\\01003}{}$,
$\subspace{10411\\01044}{}$,
$\subspace{10423\\01030}{}$,
$\subspace{10430\\01021}{}$,
$\subspace{10442\\01012}{}$,
$\subspace{11002\\00141}{}$,
$\subspace{11011\\00103}{}$,
$\subspace{11020\\00110}{}$,
$\subspace{11034\\00122}{}$,
$\subspace{11043\\00134}{}$,
$\subspace{10104\\01144}{15}$,
$\subspace{10111\\01110}{}$,
$\subspace{10123\\01131}{}$,
$\subspace{10130\\01102}{}$,
$\subspace{10142\\01123}{}$,
$\subspace{10401\\01030}{}$,
$\subspace{10413\\01021}{}$,
$\subspace{10420\\01012}{}$,
$\subspace{10432\\01003}{}$,
$\subspace{10444\\01044}{}$,
$\subspace{11004\\00122}{}$,
$\subspace{11013\\00134}{}$,
$\subspace{11022\\00141}{}$,
$\subspace{11031\\00103}{}$,
$\subspace{11040\\00110}{}$,
$\subspace{10104\\01144}{15}$,
$\subspace{10111\\01110}{}$,
$\subspace{10123\\01131}{}$,
$\subspace{10130\\01102}{}$,
$\subspace{10142\\01123}{}$,
$\subspace{10401\\01030}{}$,
$\subspace{10413\\01021}{}$,
$\subspace{10420\\01012}{}$,
$\subspace{10432\\01003}{}$,
$\subspace{10444\\01044}{}$,
$\subspace{11004\\00122}{}$,
$\subspace{11013\\00134}{}$,
$\subspace{11022\\00141}{}$,
$\subspace{11031\\00103}{}$,
$\subspace{11040\\00110}{}$,
$\subspace{10100\\01211}{15}$,
$\subspace{10112\\01244}{}$,
$\subspace{10124\\01222}{}$,
$\subspace{10131\\01200}{}$,
$\subspace{10143\\01233}{}$,
$\subspace{10100\\01441}{}$,
$\subspace{10112\\01443}{}$,
$\subspace{10124\\01440}{}$,
$\subspace{10131\\01442}{}$,
$\subspace{10143\\01444}{}$,
$\subspace{10204\\01310}{}$,
$\subspace{10211\\01330}{}$,
$\subspace{10223\\01300}{}$,
$\subspace{10230\\01320}{}$,
$\subspace{10242\\01340}{}$,
$\subspace{10100\\01211}{15}$,
$\subspace{10112\\01244}{}$,
$\subspace{10124\\01222}{}$,
$\subspace{10131\\01200}{}$,
$\subspace{10143\\01233}{}$,
$\subspace{10100\\01441}{}$,
$\subspace{10112\\01443}{}$,
$\subspace{10124\\01440}{}$,
$\subspace{10131\\01442}{}$,
$\subspace{10143\\01444}{}$,
$\subspace{10204\\01310}{}$,
$\subspace{10211\\01330}{}$,
$\subspace{10223\\01300}{}$,
$\subspace{10230\\01320}{}$,
$\subspace{10242\\01340}{}$,
$\subspace{10102\\01222}{15}$,
$\subspace{10114\\01200}{}$,
$\subspace{10121\\01233}{}$,
$\subspace{10133\\01211}{}$,
$\subspace{10140\\01244}{}$,
$\subspace{10102\\01441}{}$,
$\subspace{10114\\01443}{}$,
$\subspace{10121\\01440}{}$,
$\subspace{10133\\01442}{}$,
$\subspace{10140\\01444}{}$,
$\subspace{10200\\01320}{}$,
$\subspace{10212\\01340}{}$,
$\subspace{10224\\01310}{}$,
$\subspace{10231\\01330}{}$,
$\subspace{10243\\01300}{}$,
$\subspace{10102\\01222}{15}$,
$\subspace{10114\\01200}{}$,
$\subspace{10121\\01233}{}$,
$\subspace{10133\\01211}{}$,
$\subspace{10140\\01244}{}$,
$\subspace{10102\\01441}{}$,
$\subspace{10114\\01443}{}$,
$\subspace{10121\\01440}{}$,
$\subspace{10133\\01442}{}$,
$\subspace{10140\\01444}{}$,
$\subspace{10200\\01320}{}$,
$\subspace{10212\\01340}{}$,
$\subspace{10224\\01310}{}$,
$\subspace{10231\\01330}{}$,
$\subspace{10243\\01300}{}$,
$\subspace{10101\\01232}{15}$,
$\subspace{10113\\01210}{}$,
$\subspace{10120\\01243}{}$,
$\subspace{10132\\01221}{}$,
$\subspace{10144\\01204}{}$,
$\subspace{10103\\01430}{}$,
$\subspace{10110\\01432}{}$,
$\subspace{10122\\01434}{}$,
$\subspace{10134\\01431}{}$,
$\subspace{10141\\01433}{}$,
$\subspace{10201\\01302}{}$,
$\subspace{10213\\01322}{}$,
$\subspace{10220\\01342}{}$,
$\subspace{10232\\01312}{}$,
$\subspace{10244\\01332}{}$,
$\subspace{10103\\01234}{15}$,
$\subspace{10110\\01212}{}$,
$\subspace{10122\\01240}{}$,
$\subspace{10134\\01223}{}$,
$\subspace{10141\\01201}{}$,
$\subspace{10103\\01400}{}$,
$\subspace{10110\\01402}{}$,
$\subspace{10122\\01404}{}$,
$\subspace{10134\\01401}{}$,
$\subspace{10141\\01403}{}$,
$\subspace{10202\\01333}{}$,
$\subspace{10214\\01303}{}$,
$\subspace{10221\\01323}{}$,
$\subspace{10233\\01343}{}$,
$\subspace{10240\\01313}{}$,
$\subspace{10101\\01303}{15}$,
$\subspace{10113\\01343}{}$,
$\subspace{10120\\01333}{}$,
$\subspace{10132\\01323}{}$,
$\subspace{10144\\01313}{}$,
$\subspace{10303\\01212}{}$,
$\subspace{10310\\01223}{}$,
$\subspace{10322\\01234}{}$,
$\subspace{10334\\01240}{}$,
$\subspace{10341\\01201}{}$,
$\subspace{10300\\01401}{}$,
$\subspace{10312\\01400}{}$,
$\subspace{10324\\01404}{}$,
$\subspace{10331\\01403}{}$,
$\subspace{10343\\01402}{}$,
$\subspace{10100\\01312}{15}$,
$\subspace{10112\\01302}{}$,
$\subspace{10124\\01342}{}$,
$\subspace{10131\\01332}{}$,
$\subspace{10143\\01322}{}$,
$\subspace{10301\\01204}{}$,
$\subspace{10313\\01210}{}$,
$\subspace{10320\\01221}{}$,
$\subspace{10332\\01232}{}$,
$\subspace{10344\\01243}{}$,
$\subspace{10304\\01431}{}$,
$\subspace{10311\\01430}{}$,
$\subspace{10323\\01434}{}$,
$\subspace{10330\\01433}{}$,
$\subspace{10342\\01432}{}$,
$\subspace{10103\\01311}{15}$,
$\subspace{10110\\01301}{}$,
$\subspace{10122\\01341}{}$,
$\subspace{10134\\01331}{}$,
$\subspace{10141\\01321}{}$,
$\subspace{10304\\01213}{}$,
$\subspace{10311\\01224}{}$,
$\subspace{10323\\01230}{}$,
$\subspace{10330\\01241}{}$,
$\subspace{10342\\01202}{}$,
$\subspace{10304\\01414}{}$,
$\subspace{10311\\01413}{}$,
$\subspace{10323\\01412}{}$,
$\subspace{10330\\01411}{}$,
$\subspace{10342\\01410}{}$,
$\subspace{10104\\01311}{15}$,
$\subspace{10111\\01301}{}$,
$\subspace{10123\\01341}{}$,
$\subspace{10130\\01331}{}$,
$\subspace{10142\\01321}{}$,
$\subspace{10302\\01202}{}$,
$\subspace{10314\\01213}{}$,
$\subspace{10321\\01224}{}$,
$\subspace{10333\\01230}{}$,
$\subspace{10340\\01241}{}$,
$\subspace{10300\\01413}{}$,
$\subspace{10312\\01412}{}$,
$\subspace{10324\\01411}{}$,
$\subspace{10331\\01410}{}$,
$\subspace{10343\\01414}{}$,
$\subspace{10104\\01333}{15}$,
$\subspace{10111\\01323}{}$,
$\subspace{10123\\01313}{}$,
$\subspace{10130\\01303}{}$,
$\subspace{10142\\01343}{}$,
$\subspace{10303\\01201}{}$,
$\subspace{10310\\01212}{}$,
$\subspace{10322\\01223}{}$,
$\subspace{10334\\01234}{}$,
$\subspace{10341\\01240}{}$,
$\subspace{10302\\01402}{}$,
$\subspace{10314\\01401}{}$,
$\subspace{10321\\01400}{}$,
$\subspace{10333\\01404}{}$,
$\subspace{10340\\01403}{}$,
$\subspace{10101\\01340}{15}$,
$\subspace{10113\\01330}{}$,
$\subspace{10120\\01320}{}$,
$\subspace{10132\\01310}{}$,
$\subspace{10144\\01300}{}$,
$\subspace{10303\\01244}{}$,
$\subspace{10310\\01200}{}$,
$\subspace{10322\\01211}{}$,
$\subspace{10334\\01222}{}$,
$\subspace{10341\\01233}{}$,
$\subspace{10303\\01441}{}$,
$\subspace{10310\\01440}{}$,
$\subspace{10322\\01444}{}$,
$\subspace{10334\\01443}{}$,
$\subspace{10341\\01442}{}$,
$\subspace{10204\\01010}{15}$,
$\subspace{10211\\01042}{}$,
$\subspace{10223\\01024}{}$,
$\subspace{10230\\01001}{}$,
$\subspace{10242\\01033}{}$,
$\subspace{10301\\01122}{}$,
$\subspace{10313\\01114}{}$,
$\subspace{10320\\01101}{}$,
$\subspace{10332\\01143}{}$,
$\subspace{10344\\01130}{}$,
$\subspace{13001\\00144}{}$,
$\subspace{13010\\00113}{}$,
$\subspace{13024\\00132}{}$,
$\subspace{13033\\00101}{}$,
$\subspace{13042\\00120}{}$,
$\subspace{10200\\01024}{15}$,
$\subspace{10212\\01001}{}$,
$\subspace{10224\\01033}{}$,
$\subspace{10231\\01010}{}$,
$\subspace{10243\\01042}{}$,
$\subspace{10300\\01114}{}$,
$\subspace{10312\\01101}{}$,
$\subspace{10324\\01143}{}$,
$\subspace{10331\\01130}{}$,
$\subspace{10343\\01122}{}$,
$\subspace{13000\\00132}{}$,
$\subspace{13014\\00101}{}$,
$\subspace{13023\\00120}{}$,
$\subspace{13032\\00144}{}$,
$\subspace{13041\\00113}{}$,
$\subspace{10202\\01022}{15}$,
$\subspace{10214\\01004}{}$,
$\subspace{10221\\01031}{}$,
$\subspace{10233\\01013}{}$,
$\subspace{10240\\01040}{}$,
$\subspace{10304\\01113}{}$,
$\subspace{10311\\01100}{}$,
$\subspace{10323\\01142}{}$,
$\subspace{10330\\01134}{}$,
$\subspace{10342\\01121}{}$,
$\subspace{13004\\00111}{}$,
$\subspace{13013\\00130}{}$,
$\subspace{13022\\00104}{}$,
$\subspace{13031\\00123}{}$,
$\subspace{13040\\00142}{}$,
$\subspace{10204\\01023}{15}$,
$\subspace{10211\\01000}{}$,
$\subspace{10223\\01032}{}$,
$\subspace{10230\\01014}{}$,
$\subspace{10242\\01041}{}$,
$\subspace{10301\\01140}{}$,
$\subspace{10313\\01132}{}$,
$\subspace{10320\\01124}{}$,
$\subspace{10332\\01111}{}$,
$\subspace{10344\\01103}{}$,
$\subspace{13001\\00131}{}$,
$\subspace{13010\\00100}{}$,
$\subspace{13024\\00124}{}$,
$\subspace{13033\\00143}{}$,
$\subspace{13042\\00112}{}$,
$\subspace{10203\\01030}{15}$,
$\subspace{10210\\01012}{}$,
$\subspace{10222\\01044}{}$,
$\subspace{10234\\01021}{}$,
$\subspace{10241\\01003}{}$,
$\subspace{10302\\01144}{}$,
$\subspace{10314\\01131}{}$,
$\subspace{10321\\01123}{}$,
$\subspace{10333\\01110}{}$,
$\subspace{10340\\01102}{}$,
$\subspace{13002\\00122}{}$,
$\subspace{13011\\00141}{}$,
$\subspace{13020\\00110}{}$,
$\subspace{13034\\00134}{}$,
$\subspace{13043\\00103}{}$,
$\subspace{10203\\01043}{15}$,
$\subspace{10210\\01020}{}$,
$\subspace{10222\\01002}{}$,
$\subspace{10234\\01034}{}$,
$\subspace{10241\\01011}{}$,
$\subspace{10302\\01112}{}$,
$\subspace{10314\\01104}{}$,
$\subspace{10321\\01141}{}$,
$\subspace{10333\\01133}{}$,
$\subspace{10340\\01120}{}$,
$\subspace{13002\\00114}{}$,
$\subspace{13011\\00133}{}$,
$\subspace{13020\\00102}{}$,
$\subspace{13034\\00121}{}$,
$\subspace{13043\\00140}{}$,
$\subspace{10203\\01043}{15}$,
$\subspace{10210\\01020}{}$,
$\subspace{10222\\01002}{}$,
$\subspace{10234\\01034}{}$,
$\subspace{10241\\01011}{}$,
$\subspace{10302\\01112}{}$,
$\subspace{10314\\01104}{}$,
$\subspace{10321\\01141}{}$,
$\subspace{10333\\01133}{}$,
$\subspace{10340\\01120}{}$,
$\subspace{13002\\00114}{}$,
$\subspace{13011\\00133}{}$,
$\subspace{13020\\00102}{}$,
$\subspace{13034\\00121}{}$,
$\subspace{13043\\00140}{}$,
$\subspace{10200\\01101}{15}$,
$\subspace{10212\\01114}{}$,
$\subspace{10224\\01122}{}$,
$\subspace{10231\\01130}{}$,
$\subspace{10243\\01143}{}$,
$\subspace{10303\\01010}{}$,
$\subspace{10310\\01033}{}$,
$\subspace{10322\\01001}{}$,
$\subspace{10334\\01024}{}$,
$\subspace{10341\\01042}{}$,
$\subspace{12000\\00101}{}$,
$\subspace{12014\\00132}{}$,
$\subspace{12023\\00113}{}$,
$\subspace{12032\\00144}{}$,
$\subspace{12041\\00120}{}$,
$\subspace{10201\\01114}{15}$,
$\subspace{10213\\01122}{}$,
$\subspace{10220\\01130}{}$,
$\subspace{10232\\01143}{}$,
$\subspace{10244\\01101}{}$,
$\subspace{10301\\01042}{}$,
$\subspace{10313\\01010}{}$,
$\subspace{10320\\01033}{}$,
$\subspace{10332\\01001}{}$,
$\subspace{10344\\01024}{}$,
$\subspace{12001\\00101}{}$,
$\subspace{12010\\00132}{}$,
$\subspace{12024\\00113}{}$,
$\subspace{12033\\00144}{}$,
$\subspace{12042\\00120}{}$,
$\subspace{10202\\01110}{15}$,
$\subspace{10214\\01123}{}$,
$\subspace{10221\\01131}{}$,
$\subspace{10233\\01144}{}$,
$\subspace{10240\\01102}{}$,
$\subspace{10300\\01030}{}$,
$\subspace{10312\\01003}{}$,
$\subspace{10324\\01021}{}$,
$\subspace{10331\\01044}{}$,
$\subspace{10343\\01012}{}$,
$\subspace{12002\\00110}{}$,
$\subspace{12011\\00141}{}$,
$\subspace{12020\\00122}{}$,
$\subspace{12034\\00103}{}$,
$\subspace{12043\\00134}{}$,
$\subspace{10201\\01124}{15}$,
$\subspace{10213\\01132}{}$,
$\subspace{10220\\01140}{}$,
$\subspace{10232\\01103}{}$,
$\subspace{10244\\01111}{}$,
$\subspace{10304\\01041}{}$,
$\subspace{10311\\01014}{}$,
$\subspace{10323\\01032}{}$,
$\subspace{10330\\01000}{}$,
$\subspace{10342\\01023}{}$,
$\subspace{12001\\00112}{}$,
$\subspace{12010\\00143}{}$,
$\subspace{12024\\00124}{}$,
$\subspace{12033\\00100}{}$,
$\subspace{12042\\00131}{}$,
$\subspace{10204\\01122}{15}$,
$\subspace{10211\\01130}{}$,
$\subspace{10223\\01143}{}$,
$\subspace{10230\\01101}{}$,
$\subspace{10242\\01114}{}$,
$\subspace{10301\\01010}{}$,
$\subspace{10313\\01033}{}$,
$\subspace{10320\\01001}{}$,
$\subspace{10332\\01024}{}$,
$\subspace{10344\\01042}{}$,
$\subspace{12004\\00144}{}$,
$\subspace{12013\\00120}{}$,
$\subspace{12022\\00101}{}$,
$\subspace{12031\\00132}{}$,
$\subspace{12040\\00113}{}$,
$\subspace{10200\\01141}{15}$,
$\subspace{10212\\01104}{}$,
$\subspace{10224\\01112}{}$,
$\subspace{10231\\01120}{}$,
$\subspace{10243\\01133}{}$,
$\subspace{10300\\01011}{}$,
$\subspace{10312\\01034}{}$,
$\subspace{10324\\01002}{}$,
$\subspace{10331\\01020}{}$,
$\subspace{10343\\01043}{}$,
$\subspace{12000\\00140}{}$,
$\subspace{12014\\00121}{}$,
$\subspace{12023\\00102}{}$,
$\subspace{12032\\00133}{}$,
$\subspace{12041\\00114}{}$,
$\subspace{10202\\01142}{15}$,
$\subspace{10214\\01100}{}$,
$\subspace{10221\\01113}{}$,
$\subspace{10233\\01121}{}$,
$\subspace{10240\\01134}{}$,
$\subspace{10300\\01022}{}$,
$\subspace{10312\\01040}{}$,
$\subspace{10324\\01013}{}$,
$\subspace{10331\\01031}{}$,
$\subspace{10343\\01004}{}$,
$\subspace{12002\\00123}{}$,
$\subspace{12011\\00104}{}$,
$\subspace{12020\\00130}{}$,
$\subspace{12034\\00111}{}$,
$\subspace{12043\\00142}{}$,
$\subspace{10203\\01204}{15}$,
$\subspace{10210\\01243}{}$,
$\subspace{10222\\01232}{}$,
$\subspace{10234\\01221}{}$,
$\subspace{10241\\01210}{}$,
$\subspace{10203\\01432}{}$,
$\subspace{10210\\01433}{}$,
$\subspace{10222\\01434}{}$,
$\subspace{10234\\01430}{}$,
$\subspace{10241\\01431}{}$,
$\subspace{10404\\01302}{}$,
$\subspace{10411\\01312}{}$,
$\subspace{10423\\01322}{}$,
$\subspace{10430\\01332}{}$,
$\subspace{10442\\01342}{}$,
$\subspace{10203\\01204}{15}$,
$\subspace{10210\\01243}{}$,
$\subspace{10222\\01232}{}$,
$\subspace{10234\\01221}{}$,
$\subspace{10241\\01210}{}$,
$\subspace{10203\\01432}{}$,
$\subspace{10210\\01433}{}$,
$\subspace{10222\\01434}{}$,
$\subspace{10234\\01430}{}$,
$\subspace{10241\\01431}{}$,
$\subspace{10404\\01302}{}$,
$\subspace{10411\\01312}{}$,
$\subspace{10423\\01322}{}$,
$\subspace{10430\\01332}{}$,
$\subspace{10442\\01342}{}$,
$\subspace{10202\\01231}{15}$,
$\subspace{10214\\01220}{}$,
$\subspace{10221\\01214}{}$,
$\subspace{10233\\01203}{}$,
$\subspace{10240\\01242}{}$,
$\subspace{10202\\01423}{}$,
$\subspace{10214\\01424}{}$,
$\subspace{10221\\01420}{}$,
$\subspace{10233\\01421}{}$,
$\subspace{10240\\01422}{}$,
$\subspace{10403\\01334}{}$,
$\subspace{10410\\01344}{}$,
$\subspace{10422\\01304}{}$,
$\subspace{10434\\01314}{}$,
$\subspace{10441\\01324}{}$,
$\subspace{10201\\01242}{15}$,
$\subspace{10213\\01231}{}$,
$\subspace{10220\\01220}{}$,
$\subspace{10232\\01214}{}$,
$\subspace{10244\\01203}{}$,
$\subspace{10201\\01423}{}$,
$\subspace{10213\\01424}{}$,
$\subspace{10220\\01420}{}$,
$\subspace{10232\\01421}{}$,
$\subspace{10244\\01422}{}$,
$\subspace{10400\\01344}{}$,
$\subspace{10412\\01304}{}$,
$\subspace{10424\\01314}{}$,
$\subspace{10431\\01324}{}$,
$\subspace{10443\\01334}{}$,
$\subspace{10201\\01242}{15}$,
$\subspace{10213\\01231}{}$,
$\subspace{10220\\01220}{}$,
$\subspace{10232\\01214}{}$,
$\subspace{10244\\01203}{}$,
$\subspace{10201\\01423}{}$,
$\subspace{10213\\01424}{}$,
$\subspace{10220\\01420}{}$,
$\subspace{10232\\01421}{}$,
$\subspace{10244\\01422}{}$,
$\subspace{10400\\01344}{}$,
$\subspace{10412\\01304}{}$,
$\subspace{10424\\01314}{}$,
$\subspace{10431\\01324}{}$,
$\subspace{10443\\01334}{}$,
$\subspace{10204\\01243}{15}$,
$\subspace{10211\\01232}{}$,
$\subspace{10223\\01221}{}$,
$\subspace{10230\\01210}{}$,
$\subspace{10242\\01204}{}$,
$\subspace{10204\\01432}{}$,
$\subspace{10211\\01433}{}$,
$\subspace{10223\\01434}{}$,
$\subspace{10230\\01430}{}$,
$\subspace{10242\\01431}{}$,
$\subspace{10402\\01342}{}$,
$\subspace{10414\\01302}{}$,
$\subspace{10421\\01312}{}$,
$\subspace{10433\\01322}{}$,
$\subspace{10440\\01332}{}$,
$\subspace{10303\\01301}{15}$,
$\subspace{10310\\01331}{}$,
$\subspace{10322\\01311}{}$,
$\subspace{10334\\01341}{}$,
$\subspace{10341\\01321}{}$,
$\subspace{10403\\01241}{}$,
$\subspace{10410\\01213}{}$,
$\subspace{10422\\01230}{}$,
$\subspace{10434\\01202}{}$,
$\subspace{10441\\01224}{}$,
$\subspace{10402\\01413}{}$,
$\subspace{10414\\01411}{}$,
$\subspace{10421\\01414}{}$,
$\subspace{10433\\01412}{}$,
$\subspace{10440\\01410}{}$,
$\subspace{10304\\01304}{15}$,
$\subspace{10311\\01334}{}$,
$\subspace{10323\\01314}{}$,
$\subspace{10330\\01344}{}$,
$\subspace{10342\\01324}{}$,
$\subspace{10400\\01203}{}$,
$\subspace{10412\\01220}{}$,
$\subspace{10424\\01242}{}$,
$\subspace{10431\\01214}{}$,
$\subspace{10443\\01231}{}$,
$\subspace{10400\\01423}{}$,
$\subspace{10412\\01421}{}$,
$\subspace{10424\\01424}{}$,
$\subspace{10431\\01422}{}$,
$\subspace{10443\\01420}{}$,
$\subspace{10304\\01304}{15}$,
$\subspace{10311\\01334}{}$,
$\subspace{10323\\01314}{}$,
$\subspace{10330\\01344}{}$,
$\subspace{10342\\01324}{}$,
$\subspace{10400\\01203}{}$,
$\subspace{10412\\01220}{}$,
$\subspace{10424\\01242}{}$,
$\subspace{10431\\01214}{}$,
$\subspace{10443\\01231}{}$,
$\subspace{10400\\01423}{}$,
$\subspace{10412\\01421}{}$,
$\subspace{10424\\01424}{}$,
$\subspace{10431\\01422}{}$,
$\subspace{10443\\01420}{}$,
$\subspace{10300\\01312}{15}$,
$\subspace{10312\\01342}{}$,
$\subspace{10324\\01322}{}$,
$\subspace{10331\\01302}{}$,
$\subspace{10343\\01332}{}$,
$\subspace{10403\\01204}{}$,
$\subspace{10410\\01221}{}$,
$\subspace{10422\\01243}{}$,
$\subspace{10434\\01210}{}$,
$\subspace{10441\\01232}{}$,
$\subspace{10402\\01431}{}$,
$\subspace{10414\\01434}{}$,
$\subspace{10421\\01432}{}$,
$\subspace{10433\\01430}{}$,
$\subspace{10440\\01433}{}$,
$\subspace{10301\\01334}{15}$,
$\subspace{10313\\01314}{}$,
$\subspace{10320\\01344}{}$,
$\subspace{10332\\01324}{}$,
$\subspace{10344\\01304}{}$,
$\subspace{10404\\01231}{}$,
$\subspace{10411\\01203}{}$,
$\subspace{10423\\01220}{}$,
$\subspace{10430\\01242}{}$,
$\subspace{10442\\01214}{}$,
$\subspace{10404\\01423}{}$,
$\subspace{10411\\01421}{}$,
$\subspace{10423\\01424}{}$,
$\subspace{10430\\01422}{}$,
$\subspace{10442\\01420}{}$,
$\subspace{10303\\01331}{15}$,
$\subspace{10310\\01311}{}$,
$\subspace{10322\\01341}{}$,
$\subspace{10334\\01321}{}$,
$\subspace{10341\\01301}{}$,
$\subspace{10401\\01213}{}$,
$\subspace{10413\\01230}{}$,
$\subspace{10420\\01202}{}$,
$\subspace{10432\\01224}{}$,
$\subspace{10444\\01241}{}$,
$\subspace{10404\\01412}{}$,
$\subspace{10411\\01410}{}$,
$\subspace{10423\\01413}{}$,
$\subspace{10430\\01411}{}$,
$\subspace{10442\\01414}{}$.

\smallskip

$n_5(5,2;39)\ge 969$,$\left[\!\begin{smallmatrix}10000\\00100\\04410\\00001\\00044\end{smallmatrix}\!\right]$:
$\subspace{01001\\00101}{15}$,
$\subspace{01001\\00134}{}$,
$\subspace{01003\\00132}{}$,
$\subspace{01010\\00113}{}$,
$\subspace{01010\\00141}{}$,
$\subspace{01012\\00144}{}$,
$\subspace{01021\\00101}{}$,
$\subspace{01024\\00103}{}$,
$\subspace{01024\\00120}{}$,
$\subspace{01030\\00113}{}$,
$\subspace{01033\\00110}{}$,
$\subspace{01033\\00132}{}$,
$\subspace{01042\\00122}{}$,
$\subspace{01044\\00120}{}$,
$\subspace{01042\\00144}{}$,
$\subspace{01002\\00102}{5}$,
$\subspace{01011\\00114}{}$,
$\subspace{01020\\00121}{}$,
$\subspace{01034\\00133}{}$,
$\subspace{01043\\00140}{}$,
$\subspace{01002\\00102}{5}$,
$\subspace{01011\\00114}{}$,
$\subspace{01020\\00121}{}$,
$\subspace{01034\\00133}{}$,
$\subspace{01043\\00140}{}$,
$\subspace{01002\\00102}{5}$,
$\subspace{01011\\00114}{}$,
$\subspace{01020\\00121}{}$,
$\subspace{01034\\00133}{}$,
$\subspace{01043\\00140}{}$,
$\subspace{01203\\00011}{3}$,
$\subspace{01304\\00010}{}$,
$\subspace{01420\\00001}{}$,
$\subspace{10000\\00012}{3}$,
$\subspace{10000\\00013}{}$,
$\subspace{10000\\00014}{}$,
$\subspace{10013\\00104}{15}$,
$\subspace{10013\\00111}{}$,
$\subspace{10013\\00123}{}$,
$\subspace{10013\\00130}{}$,
$\subspace{10013\\00142}{}$,
$\subspace{10024\\01004}{}$,
$\subspace{10024\\01013}{}$,
$\subspace{10024\\01022}{}$,
$\subspace{10024\\01031}{}$,
$\subspace{10024\\01040}{}$,
$\subspace{10023\\01100}{}$,
$\subspace{10023\\01113}{}$,
$\subspace{10023\\01121}{}$,
$\subspace{10023\\01134}{}$,
$\subspace{10023\\01142}{}$,
$\subspace{10020\\00100}{15}$,
$\subspace{10020\\00112}{}$,
$\subspace{10020\\00124}{}$,
$\subspace{10020\\00131}{}$,
$\subspace{10020\\00143}{}$,
$\subspace{10033\\01000}{}$,
$\subspace{10033\\01014}{}$,
$\subspace{10033\\01023}{}$,
$\subspace{10033\\01032}{}$,
$\subspace{10033\\01041}{}$,
$\subspace{10002\\01103}{}$,
$\subspace{10002\\01111}{}$,
$\subspace{10002\\01124}{}$,
$\subspace{10002\\01132}{}$,
$\subspace{10002\\01140}{}$,
$\subspace{10021\\00103}{15}$,
$\subspace{10021\\00110}{}$,
$\subspace{10021\\00122}{}$,
$\subspace{10021\\00134}{}$,
$\subspace{10021\\00141}{}$,
$\subspace{10043\\01003}{}$,
$\subspace{10043\\01012}{}$,
$\subspace{10043\\01021}{}$,
$\subspace{10043\\01030}{}$,
$\subspace{10043\\01044}{}$,
$\subspace{10041\\01102}{}$,
$\subspace{10041\\01110}{}$,
$\subspace{10041\\01123}{}$,
$\subspace{10041\\01131}{}$,
$\subspace{10041\\01144}{}$,
$\subspace{10030\\00103}{15}$,
$\subspace{10030\\00110}{}$,
$\subspace{10030\\00122}{}$,
$\subspace{10030\\00134}{}$,
$\subspace{10030\\00141}{}$,
$\subspace{10022\\01003}{}$,
$\subspace{10022\\01012}{}$,
$\subspace{10022\\01021}{}$,
$\subspace{10022\\01030}{}$,
$\subspace{10022\\01044}{}$,
$\subspace{10003\\01102}{}$,
$\subspace{10003\\01110}{}$,
$\subspace{10003\\01123}{}$,
$\subspace{10003\\01131}{}$,
$\subspace{10003\\01144}{}$,
$\subspace{10032\\00103}{15}$,
$\subspace{10032\\00110}{}$,
$\subspace{10032\\00122}{}$,
$\subspace{10032\\00134}{}$,
$\subspace{10032\\00141}{}$,
$\subspace{10042\\01003}{}$,
$\subspace{10042\\01012}{}$,
$\subspace{10042\\01021}{}$,
$\subspace{10042\\01030}{}$,
$\subspace{10042\\01044}{}$,
$\subspace{10031\\01102}{}$,
$\subspace{10031\\01110}{}$,
$\subspace{10031\\01123}{}$,
$\subspace{10031\\01131}{}$,
$\subspace{10031\\01144}{}$,
$\subspace{10034\\00103}{15}$,
$\subspace{10034\\00110}{}$,
$\subspace{10034\\00122}{}$,
$\subspace{10034\\00134}{}$,
$\subspace{10034\\00141}{}$,
$\subspace{10012\\01003}{}$,
$\subspace{10012\\01012}{}$,
$\subspace{10012\\01021}{}$,
$\subspace{10012\\01030}{}$,
$\subspace{10012\\01044}{}$,
$\subspace{10014\\01102}{}$,
$\subspace{10014\\01110}{}$,
$\subspace{10014\\01123}{}$,
$\subspace{10014\\01131}{}$,
$\subspace{10014\\01144}{}$,
$\subspace{10100\\00012}{3}$,
$\subspace{11000\\00014}{}$,
$\subspace{14402\\00013}{}$,
$\subspace{10204\\00012}{3}$,
$\subspace{12004\\00014}{}$,
$\subspace{13302\\00013}{}$,
$\subspace{10303\\00012}{3}$,
$\subspace{12202\\00013}{}$,
$\subspace{13003\\00014}{}$,
$\subspace{10303\\00012}{3}$,
$\subspace{12202\\00013}{}$,
$\subspace{13003\\00014}{}$,
$\subspace{10001\\01204}{15}$,
$\subspace{10001\\01210}{}$,
$\subspace{10001\\01221}{}$,
$\subspace{10001\\01232}{}$,
$\subspace{10001\\01243}{}$,
$\subspace{10044\\01302}{}$,
$\subspace{10044\\01312}{}$,
$\subspace{10044\\01322}{}$,
$\subspace{10044\\01332}{}$,
$\subspace{10044\\01342}{}$,
$\subspace{10010\\01430}{}$,
$\subspace{10010\\01431}{}$,
$\subspace{10010\\01432}{}$,
$\subspace{10010\\01433}{}$,
$\subspace{10010\\01434}{}$,
$\subspace{10040\\01203}{15}$,
$\subspace{10040\\01214}{}$,
$\subspace{10040\\01220}{}$,
$\subspace{10040\\01231}{}$,
$\subspace{10040\\01242}{}$,
$\subspace{10004\\01304}{}$,
$\subspace{10004\\01314}{}$,
$\subspace{10004\\01324}{}$,
$\subspace{10004\\01334}{}$,
$\subspace{10004\\01344}{}$,
$\subspace{10011\\01420}{}$,
$\subspace{10011\\01421}{}$,
$\subspace{10011\\01422}{}$,
$\subspace{10011\\01423}{}$,
$\subspace{10011\\01424}{}$,
$\subspace{10100\\01000}{15}$,
$\subspace{10112\\01014}{}$,
$\subspace{10124\\01023}{}$,
$\subspace{10131\\01032}{}$,
$\subspace{10143\\01041}{}$,
$\subspace{10401\\01140}{}$,
$\subspace{10413\\01124}{}$,
$\subspace{10420\\01103}{}$,
$\subspace{10432\\01132}{}$,
$\subspace{10444\\01111}{}$,
$\subspace{14001\\00112}{}$,
$\subspace{14010\\00100}{}$,
$\subspace{14024\\00143}{}$,
$\subspace{14033\\00131}{}$,
$\subspace{14042\\00124}{}$,
$\subspace{10100\\01004}{15}$,
$\subspace{10112\\01013}{}$,
$\subspace{10124\\01022}{}$,
$\subspace{10131\\01031}{}$,
$\subspace{10143\\01040}{}$,
$\subspace{10403\\01134}{}$,
$\subspace{10410\\01113}{}$,
$\subspace{10422\\01142}{}$,
$\subspace{10434\\01121}{}$,
$\subspace{10441\\01100}{}$,
$\subspace{14003\\00130}{}$,
$\subspace{14012\\00123}{}$,
$\subspace{14021\\00111}{}$,
$\subspace{14030\\00104}{}$,
$\subspace{14044\\00142}{}$,
$\subspace{10101\\01002}{15}$,
$\subspace{10113\\01011}{}$,
$\subspace{10120\\01020}{}$,
$\subspace{10132\\01034}{}$,
$\subspace{10144\\01043}{}$,
$\subspace{10400\\01133}{}$,
$\subspace{10412\\01112}{}$,
$\subspace{10424\\01141}{}$,
$\subspace{10431\\01120}{}$,
$\subspace{10443\\01104}{}$,
$\subspace{14000\\00114}{}$,
$\subspace{14014\\00102}{}$,
$\subspace{14023\\00140}{}$,
$\subspace{14032\\00133}{}$,
$\subspace{14041\\00121}{}$,
$\subspace{10101\\01012}{15}$,
$\subspace{10113\\01021}{}$,
$\subspace{10120\\01030}{}$,
$\subspace{10132\\01044}{}$,
$\subspace{10144\\01003}{}$,
$\subspace{10401\\01123}{}$,
$\subspace{10413\\01102}{}$,
$\subspace{10420\\01131}{}$,
$\subspace{10432\\01110}{}$,
$\subspace{10444\\01144}{}$,
$\subspace{14001\\00110}{}$,
$\subspace{14010\\00103}{}$,
$\subspace{14024\\00141}{}$,
$\subspace{14033\\00134}{}$,
$\subspace{14042\\00122}{}$,
$\subspace{10102\\01031}{15}$,
$\subspace{10114\\01040}{}$,
$\subspace{10121\\01004}{}$,
$\subspace{10133\\01013}{}$,
$\subspace{10140\\01022}{}$,
$\subspace{10403\\01113}{}$,
$\subspace{10410\\01142}{}$,
$\subspace{10422\\01121}{}$,
$\subspace{10434\\01100}{}$,
$\subspace{10441\\01134}{}$,
$\subspace{14003\\00123}{}$,
$\subspace{14012\\00111}{}$,
$\subspace{14021\\00104}{}$,
$\subspace{14030\\00142}{}$,
$\subspace{14044\\00130}{}$,
$\subspace{10101\\01041}{15}$,
$\subspace{10113\\01000}{}$,
$\subspace{10120\\01014}{}$,
$\subspace{10132\\01023}{}$,
$\subspace{10144\\01032}{}$,
$\subspace{10401\\01132}{}$,
$\subspace{10413\\01111}{}$,
$\subspace{10420\\01140}{}$,
$\subspace{10432\\01124}{}$,
$\subspace{10444\\01103}{}$,
$\subspace{14001\\00131}{}$,
$\subspace{14010\\00124}{}$,
$\subspace{14024\\00112}{}$,
$\subspace{14033\\00100}{}$,
$\subspace{14042\\00143}{}$,
$\subspace{10102\\01044}{15}$,
$\subspace{10114\\01003}{}$,
$\subspace{10121\\01012}{}$,
$\subspace{10133\\01021}{}$,
$\subspace{10140\\01030}{}$,
$\subspace{10403\\01131}{}$,
$\subspace{10410\\01110}{}$,
$\subspace{10422\\01144}{}$,
$\subspace{10434\\01123}{}$,
$\subspace{10441\\01102}{}$,
$\subspace{14003\\00110}{}$,
$\subspace{14012\\00103}{}$,
$\subspace{14021\\00141}{}$,
$\subspace{14030\\00134}{}$,
$\subspace{14044\\00122}{}$,
$\subspace{10101\\01102}{15}$,
$\subspace{10113\\01123}{}$,
$\subspace{10120\\01144}{}$,
$\subspace{10132\\01110}{}$,
$\subspace{10144\\01131}{}$,
$\subspace{10404\\01021}{}$,
$\subspace{10411\\01012}{}$,
$\subspace{10423\\01003}{}$,
$\subspace{10430\\01044}{}$,
$\subspace{10442\\01030}{}$,
$\subspace{11001\\00134}{}$,
$\subspace{11010\\00141}{}$,
$\subspace{11024\\00103}{}$,
$\subspace{11033\\00110}{}$,
$\subspace{11042\\00122}{}$,
$\subspace{10100\\01141}{15}$,
$\subspace{10112\\01112}{}$,
$\subspace{10124\\01133}{}$,
$\subspace{10131\\01104}{}$,
$\subspace{10143\\01120}{}$,
$\subspace{10400\\01011}{}$,
$\subspace{10412\\01002}{}$,
$\subspace{10424\\01043}{}$,
$\subspace{10431\\01034}{}$,
$\subspace{10443\\01020}{}$,
$\subspace{11000\\00140}{}$,
$\subspace{11014\\00102}{}$,
$\subspace{11023\\00114}{}$,
$\subspace{11032\\00121}{}$,
$\subspace{11041\\00133}{}$,
$\subspace{10101\\01141}{15}$,
$\subspace{10113\\01112}{}$,
$\subspace{10120\\01133}{}$,
$\subspace{10132\\01104}{}$,
$\subspace{10144\\01120}{}$,
$\subspace{10401\\01002}{}$,
$\subspace{10413\\01043}{}$,
$\subspace{10420\\01034}{}$,
$\subspace{10432\\01020}{}$,
$\subspace{10444\\01011}{}$,
$\subspace{11001\\00133}{}$,
$\subspace{11010\\00140}{}$,
$\subspace{11024\\00102}{}$,
$\subspace{11033\\00114}{}$,
$\subspace{11042\\00121}{}$,
$\subspace{10102\\01142}{15}$,
$\subspace{10114\\01113}{}$,
$\subspace{10121\\01134}{}$,
$\subspace{10133\\01100}{}$,
$\subspace{10140\\01121}{}$,
$\subspace{10401\\01013}{}$,
$\subspace{10413\\01004}{}$,
$\subspace{10420\\01040}{}$,
$\subspace{10432\\01031}{}$,
$\subspace{10444\\01022}{}$,
$\subspace{11002\\00111}{}$,
$\subspace{11011\\00123}{}$,
$\subspace{11020\\00130}{}$,
$\subspace{11034\\00142}{}$,
$\subspace{11043\\00104}{}$,
$\subspace{10102\\01143}{15}$,
$\subspace{10114\\01114}{}$,
$\subspace{10121\\01130}{}$,
$\subspace{10133\\01101}{}$,
$\subspace{10140\\01122}{}$,
$\subspace{10400\\01033}{}$,
$\subspace{10412\\01024}{}$,
$\subspace{10424\\01010}{}$,
$\subspace{10431\\01001}{}$,
$\subspace{10443\\01042}{}$,
$\subspace{11002\\00101}{}$,
$\subspace{11011\\00113}{}$,
$\subspace{11020\\00120}{}$,
$\subspace{11034\\00132}{}$,
$\subspace{11043\\00144}{}$,
$\subspace{10104\\01140}{15}$,
$\subspace{10111\\01111}{}$,
$\subspace{10123\\01132}{}$,
$\subspace{10130\\01103}{}$,
$\subspace{10142\\01124}{}$,
$\subspace{10400\\01000}{}$,
$\subspace{10412\\01041}{}$,
$\subspace{10424\\01032}{}$,
$\subspace{10431\\01023}{}$,
$\subspace{10443\\01014}{}$,
$\subspace{11004\\00112}{}$,
$\subspace{11013\\00124}{}$,
$\subspace{11022\\00131}{}$,
$\subspace{11031\\00143}{}$,
$\subspace{11040\\00100}{}$,
$\subspace{10104\\01140}{15}$,
$\subspace{10111\\01111}{}$,
$\subspace{10123\\01132}{}$,
$\subspace{10130\\01103}{}$,
$\subspace{10142\\01124}{}$,
$\subspace{10400\\01000}{}$,
$\subspace{10412\\01041}{}$,
$\subspace{10424\\01032}{}$,
$\subspace{10431\\01023}{}$,
$\subspace{10443\\01014}{}$,
$\subspace{11004\\00112}{}$,
$\subspace{11013\\00124}{}$,
$\subspace{11022\\00131}{}$,
$\subspace{11031\\00143}{}$,
$\subspace{11040\\00100}{}$,
$\subspace{10100\\01203}{15}$,
$\subspace{10112\\01231}{}$,
$\subspace{10124\\01214}{}$,
$\subspace{10131\\01242}{}$,
$\subspace{10143\\01220}{}$,
$\subspace{10100\\01423}{}$,
$\subspace{10112\\01420}{}$,
$\subspace{10124\\01422}{}$,
$\subspace{10131\\01424}{}$,
$\subspace{10143\\01421}{}$,
$\subspace{10201\\01304}{}$,
$\subspace{10213\\01324}{}$,
$\subspace{10220\\01344}{}$,
$\subspace{10232\\01314}{}$,
$\subspace{10244\\01334}{}$,
$\subspace{10104\\01212}{15}$,
$\subspace{10111\\01240}{}$,
$\subspace{10123\\01223}{}$,
$\subspace{10130\\01201}{}$,
$\subspace{10142\\01234}{}$,
$\subspace{10104\\01400}{}$,
$\subspace{10111\\01402}{}$,
$\subspace{10123\\01404}{}$,
$\subspace{10130\\01401}{}$,
$\subspace{10142\\01403}{}$,
$\subspace{10200\\01313}{}$,
$\subspace{10212\\01333}{}$,
$\subspace{10224\\01303}{}$,
$\subspace{10231\\01323}{}$,
$\subspace{10243\\01343}{}$,
$\subspace{10103\\01221}{15}$,
$\subspace{10110\\01204}{}$,
$\subspace{10122\\01232}{}$,
$\subspace{10134\\01210}{}$,
$\subspace{10141\\01243}{}$,
$\subspace{10103\\01432}{}$,
$\subspace{10110\\01434}{}$,
$\subspace{10122\\01431}{}$,
$\subspace{10134\\01433}{}$,
$\subspace{10141\\01430}{}$,
$\subspace{10204\\01322}{}$,
$\subspace{10211\\01342}{}$,
$\subspace{10223\\01312}{}$,
$\subspace{10230\\01332}{}$,
$\subspace{10242\\01302}{}$,
$\subspace{10103\\01221}{15}$,
$\subspace{10110\\01204}{}$,
$\subspace{10122\\01232}{}$,
$\subspace{10134\\01210}{}$,
$\subspace{10141\\01243}{}$,
$\subspace{10103\\01432}{}$,
$\subspace{10110\\01434}{}$,
$\subspace{10122\\01431}{}$,
$\subspace{10134\\01433}{}$,
$\subspace{10141\\01430}{}$,
$\subspace{10204\\01322}{}$,
$\subspace{10211\\01342}{}$,
$\subspace{10223\\01312}{}$,
$\subspace{10230\\01332}{}$,
$\subspace{10242\\01302}{}$,
$\subspace{10103\\01221}{15}$,
$\subspace{10110\\01204}{}$,
$\subspace{10122\\01232}{}$,
$\subspace{10134\\01210}{}$,
$\subspace{10141\\01243}{}$,
$\subspace{10103\\01432}{}$,
$\subspace{10110\\01434}{}$,
$\subspace{10122\\01431}{}$,
$\subspace{10134\\01433}{}$,
$\subspace{10141\\01430}{}$,
$\subspace{10204\\01322}{}$,
$\subspace{10211\\01342}{}$,
$\subspace{10223\\01312}{}$,
$\subspace{10230\\01332}{}$,
$\subspace{10242\\01302}{}$,
$\subspace{10103\\01221}{15}$,
$\subspace{10110\\01204}{}$,
$\subspace{10122\\01232}{}$,
$\subspace{10134\\01210}{}$,
$\subspace{10141\\01243}{}$,
$\subspace{10103\\01432}{}$,
$\subspace{10110\\01434}{}$,
$\subspace{10122\\01431}{}$,
$\subspace{10134\\01433}{}$,
$\subspace{10141\\01430}{}$,
$\subspace{10204\\01322}{}$,
$\subspace{10211\\01342}{}$,
$\subspace{10223\\01312}{}$,
$\subspace{10230\\01332}{}$,
$\subspace{10242\\01302}{}$,
$\subspace{10104\\01233}{15}$,
$\subspace{10111\\01211}{}$,
$\subspace{10123\\01244}{}$,
$\subspace{10130\\01222}{}$,
$\subspace{10142\\01200}{}$,
$\subspace{10104\\01441}{}$,
$\subspace{10111\\01443}{}$,
$\subspace{10123\\01440}{}$,
$\subspace{10130\\01442}{}$,
$\subspace{10142\\01444}{}$,
$\subspace{10201\\01330}{}$,
$\subspace{10213\\01300}{}$,
$\subspace{10220\\01320}{}$,
$\subspace{10232\\01340}{}$,
$\subspace{10244\\01310}{}$,
$\subspace{10102\\01300}{15}$,
$\subspace{10114\\01340}{}$,
$\subspace{10121\\01330}{}$,
$\subspace{10133\\01320}{}$,
$\subspace{10140\\01310}{}$,
$\subspace{10303\\01222}{}$,
$\subspace{10310\\01233}{}$,
$\subspace{10322\\01244}{}$,
$\subspace{10334\\01200}{}$,
$\subspace{10341\\01211}{}$,
$\subspace{10302\\01443}{}$,
$\subspace{10314\\01442}{}$,
$\subspace{10321\\01441}{}$,
$\subspace{10333\\01440}{}$,
$\subspace{10340\\01444}{}$,
$\subspace{10101\\01324}{15}$,
$\subspace{10113\\01314}{}$,
$\subspace{10120\\01304}{}$,
$\subspace{10132\\01344}{}$,
$\subspace{10144\\01334}{}$,
$\subspace{10303\\01203}{}$,
$\subspace{10310\\01214}{}$,
$\subspace{10322\\01220}{}$,
$\subspace{10334\\01231}{}$,
$\subspace{10341\\01242}{}$,
$\subspace{10304\\01421}{}$,
$\subspace{10311\\01420}{}$,
$\subspace{10323\\01424}{}$,
$\subspace{10330\\01423}{}$,
$\subspace{10342\\01422}{}$,
$\subspace{10102\\01321}{15}$,
$\subspace{10114\\01311}{}$,
$\subspace{10121\\01301}{}$,
$\subspace{10133\\01341}{}$,
$\subspace{10140\\01331}{}$,
$\subspace{10303\\01213}{}$,
$\subspace{10310\\01224}{}$,
$\subspace{10322\\01230}{}$,
$\subspace{10334\\01241}{}$,
$\subspace{10341\\01202}{}$,
$\subspace{10301\\01413}{}$,
$\subspace{10313\\01412}{}$,
$\subspace{10320\\01411}{}$,
$\subspace{10332\\01410}{}$,
$\subspace{10344\\01414}{}$,
$\subspace{10102\\01324}{15}$,
$\subspace{10114\\01314}{}$,
$\subspace{10121\\01304}{}$,
$\subspace{10133\\01344}{}$,
$\subspace{10140\\01334}{}$,
$\subspace{10301\\01242}{}$,
$\subspace{10313\\01203}{}$,
$\subspace{10320\\01214}{}$,
$\subspace{10332\\01220}{}$,
$\subspace{10344\\01231}{}$,
$\subspace{10300\\01420}{}$,
$\subspace{10312\\01424}{}$,
$\subspace{10324\\01423}{}$,
$\subspace{10331\\01422}{}$,
$\subspace{10343\\01421}{}$,
$\subspace{10104\\01331}{15}$,
$\subspace{10111\\01321}{}$,
$\subspace{10123\\01311}{}$,
$\subspace{10130\\01301}{}$,
$\subspace{10142\\01341}{}$,
$\subspace{10301\\01230}{}$,
$\subspace{10313\\01241}{}$,
$\subspace{10320\\01202}{}$,
$\subspace{10332\\01213}{}$,
$\subspace{10344\\01224}{}$,
$\subspace{10301\\01414}{}$,
$\subspace{10313\\01413}{}$,
$\subspace{10320\\01412}{}$,
$\subspace{10332\\01411}{}$,
$\subspace{10344\\01410}{}$,
$\subspace{10203\\01001}{15}$,
$\subspace{10210\\01033}{}$,
$\subspace{10222\\01010}{}$,
$\subspace{10234\\01042}{}$,
$\subspace{10241\\01024}{}$,
$\subspace{10302\\01130}{}$,
$\subspace{10314\\01122}{}$,
$\subspace{10321\\01114}{}$,
$\subspace{10333\\01101}{}$,
$\subspace{10340\\01143}{}$,
$\subspace{13002\\00101}{}$,
$\subspace{13011\\00120}{}$,
$\subspace{13020\\00144}{}$,
$\subspace{13034\\00113}{}$,
$\subspace{13043\\00132}{}$,
$\subspace{10203\\01001}{15}$,
$\subspace{10210\\01033}{}$,
$\subspace{10222\\01010}{}$,
$\subspace{10234\\01042}{}$,
$\subspace{10241\\01024}{}$,
$\subspace{10302\\01130}{}$,
$\subspace{10314\\01122}{}$,
$\subspace{10321\\01114}{}$,
$\subspace{10333\\01101}{}$,
$\subspace{10340\\01143}{}$,
$\subspace{13002\\00101}{}$,
$\subspace{13011\\00120}{}$,
$\subspace{13020\\00144}{}$,
$\subspace{13034\\00113}{}$,
$\subspace{13043\\00132}{}$,
$\subspace{10200\\01024}{15}$,
$\subspace{10212\\01001}{}$,
$\subspace{10224\\01033}{}$,
$\subspace{10231\\01010}{}$,
$\subspace{10243\\01042}{}$,
$\subspace{10300\\01114}{}$,
$\subspace{10312\\01101}{}$,
$\subspace{10324\\01143}{}$,
$\subspace{10331\\01130}{}$,
$\subspace{10343\\01122}{}$,
$\subspace{13000\\00132}{}$,
$\subspace{13014\\00101}{}$,
$\subspace{13023\\00120}{}$,
$\subspace{13032\\00144}{}$,
$\subspace{13041\\00113}{}$,
$\subspace{10202\\01040}{15}$,
$\subspace{10214\\01022}{}$,
$\subspace{10221\\01004}{}$,
$\subspace{10233\\01031}{}$,
$\subspace{10240\\01013}{}$,
$\subspace{10301\\01121}{}$,
$\subspace{10313\\01113}{}$,
$\subspace{10320\\01100}{}$,
$\subspace{10332\\01142}{}$,
$\subspace{10344\\01134}{}$,
$\subspace{13001\\00104}{}$,
$\subspace{13010\\00123}{}$,
$\subspace{13024\\00142}{}$,
$\subspace{13033\\00111}{}$,
$\subspace{13042\\00130}{}$,
$\subspace{10202\\01041}{15}$,
$\subspace{10214\\01023}{}$,
$\subspace{10221\\01000}{}$,
$\subspace{10233\\01032}{}$,
$\subspace{10240\\01014}{}$,
$\subspace{10302\\01132}{}$,
$\subspace{10314\\01124}{}$,
$\subspace{10321\\01111}{}$,
$\subspace{10333\\01103}{}$,
$\subspace{10340\\01140}{}$,
$\subspace{13002\\00131}{}$,
$\subspace{13011\\00100}{}$,
$\subspace{13020\\00124}{}$,
$\subspace{13034\\00143}{}$,
$\subspace{13043\\00112}{}$,
$\subspace{10202\\01043}{15}$,
$\subspace{10214\\01020}{}$,
$\subspace{10221\\01002}{}$,
$\subspace{10233\\01034}{}$,
$\subspace{10240\\01011}{}$,
$\subspace{10304\\01104}{}$,
$\subspace{10311\\01141}{}$,
$\subspace{10323\\01133}{}$,
$\subspace{10330\\01120}{}$,
$\subspace{10342\\01112}{}$,
$\subspace{13004\\00140}{}$,
$\subspace{13013\\00114}{}$,
$\subspace{13022\\00133}{}$,
$\subspace{13031\\00102}{}$,
$\subspace{13040\\00121}{}$,
$\subspace{10203\\01040}{15}$,
$\subspace{10210\\01022}{}$,
$\subspace{10222\\01004}{}$,
$\subspace{10234\\01031}{}$,
$\subspace{10241\\01013}{}$,
$\subspace{10304\\01134}{}$,
$\subspace{10311\\01121}{}$,
$\subspace{10323\\01113}{}$,
$\subspace{10330\\01100}{}$,
$\subspace{10342\\01142}{}$,
$\subspace{13004\\00123}{}$,
$\subspace{13013\\00142}{}$,
$\subspace{13022\\00111}{}$,
$\subspace{13031\\00130}{}$,
$\subspace{13040\\00104}{}$,
$\subspace{10201\\01103}{15}$,
$\subspace{10213\\01111}{}$,
$\subspace{10220\\01124}{}$,
$\subspace{10232\\01132}{}$,
$\subspace{10244\\01140}{}$,
$\subspace{10300\\01023}{}$,
$\subspace{10312\\01041}{}$,
$\subspace{10324\\01014}{}$,
$\subspace{10331\\01032}{}$,
$\subspace{10343\\01000}{}$,
$\subspace{12001\\00100}{}$,
$\subspace{12010\\00131}{}$,
$\subspace{12024\\00112}{}$,
$\subspace{12033\\00143}{}$,
$\subspace{12042\\00124}{}$,
$\subspace{10202\\01112}{15}$,
$\subspace{10214\\01120}{}$,
$\subspace{10221\\01133}{}$,
$\subspace{10233\\01141}{}$,
$\subspace{10240\\01104}{}$,
$\subspace{10301\\01020}{}$,
$\subspace{10313\\01043}{}$,
$\subspace{10320\\01011}{}$,
$\subspace{10332\\01034}{}$,
$\subspace{10344\\01002}{}$,
$\subspace{12002\\00140}{}$,
$\subspace{12011\\00121}{}$,
$\subspace{12020\\00102}{}$,
$\subspace{12034\\00133}{}$,
$\subspace{12043\\00114}{}$,
$\subspace{10203\\01111}{15}$,
$\subspace{10210\\01124}{}$,
$\subspace{10222\\01132}{}$,
$\subspace{10234\\01140}{}$,
$\subspace{10241\\01103}{}$,
$\subspace{10304\\01023}{}$,
$\subspace{10311\\01041}{}$,
$\subspace{10323\\01014}{}$,
$\subspace{10330\\01032}{}$,
$\subspace{10342\\01000}{}$,
$\subspace{12003\\00124}{}$,
$\subspace{12012\\00100}{}$,
$\subspace{12021\\00131}{}$,
$\subspace{12030\\00112}{}$,
$\subspace{12044\\00143}{}$,
$\subspace{10200\\01121}{15}$,
$\subspace{10212\\01134}{}$,
$\subspace{10224\\01142}{}$,
$\subspace{10231\\01100}{}$,
$\subspace{10243\\01113}{}$,
$\subspace{10304\\01013}{}$,
$\subspace{10311\\01031}{}$,
$\subspace{10323\\01004}{}$,
$\subspace{10330\\01022}{}$,
$\subspace{10342\\01040}{}$,
$\subspace{12000\\00123}{}$,
$\subspace{12014\\00104}{}$,
$\subspace{12023\\00130}{}$,
$\subspace{12032\\00111}{}$,
$\subspace{12041\\00142}{}$,
$\subspace{10200\\01121}{15}$,
$\subspace{10212\\01134}{}$,
$\subspace{10224\\01142}{}$,
$\subspace{10231\\01100}{}$,
$\subspace{10243\\01113}{}$,
$\subspace{10304\\01013}{}$,
$\subspace{10311\\01031}{}$,
$\subspace{10323\\01004}{}$,
$\subspace{10330\\01022}{}$,
$\subspace{10342\\01040}{}$,
$\subspace{12000\\00123}{}$,
$\subspace{12014\\00104}{}$,
$\subspace{12023\\00130}{}$,
$\subspace{12032\\00111}{}$,
$\subspace{12041\\00142}{}$,
$\subspace{10203\\01130}{15}$,
$\subspace{10210\\01143}{}$,
$\subspace{10222\\01101}{}$,
$\subspace{10234\\01114}{}$,
$\subspace{10241\\01122}{}$,
$\subspace{10302\\01001}{}$,
$\subspace{10314\\01024}{}$,
$\subspace{10321\\01042}{}$,
$\subspace{10333\\01010}{}$,
$\subspace{10340\\01033}{}$,
$\subspace{12003\\00101}{}$,
$\subspace{12012\\00132}{}$,
$\subspace{12021\\00113}{}$,
$\subspace{12030\\00144}{}$,
$\subspace{12044\\00120}{}$,
$\subspace{10203\\01130}{15}$,
$\subspace{10210\\01143}{}$,
$\subspace{10222\\01101}{}$,
$\subspace{10234\\01114}{}$,
$\subspace{10241\\01122}{}$,
$\subspace{10302\\01001}{}$,
$\subspace{10314\\01024}{}$,
$\subspace{10321\\01042}{}$,
$\subspace{10333\\01010}{}$,
$\subspace{10340\\01033}{}$,
$\subspace{12003\\00101}{}$,
$\subspace{12012\\00132}{}$,
$\subspace{12021\\00113}{}$,
$\subspace{12030\\00144}{}$,
$\subspace{12044\\00120}{}$,
$\subspace{10200\\01202}{15}$,
$\subspace{10212\\01241}{}$,
$\subspace{10224\\01230}{}$,
$\subspace{10231\\01224}{}$,
$\subspace{10243\\01213}{}$,
$\subspace{10201\\01411}{}$,
$\subspace{10213\\01412}{}$,
$\subspace{10220\\01413}{}$,
$\subspace{10232\\01414}{}$,
$\subspace{10244\\01410}{}$,
$\subspace{10403\\01331}{}$,
$\subspace{10410\\01341}{}$,
$\subspace{10422\\01301}{}$,
$\subspace{10434\\01311}{}$,
$\subspace{10441\\01321}{}$,
$\subspace{10201\\01213}{15}$,
$\subspace{10213\\01202}{}$,
$\subspace{10220\\01241}{}$,
$\subspace{10232\\01230}{}$,
$\subspace{10244\\01224}{}$,
$\subspace{10201\\01414}{}$,
$\subspace{10213\\01410}{}$,
$\subspace{10220\\01411}{}$,
$\subspace{10232\\01412}{}$,
$\subspace{10244\\01413}{}$,
$\subspace{10402\\01311}{}$,
$\subspace{10414\\01321}{}$,
$\subspace{10421\\01331}{}$,
$\subspace{10433\\01341}{}$,
$\subspace{10440\\01301}{}$,
$\subspace{10202\\01210}{15}$,
$\subspace{10214\\01204}{}$,
$\subspace{10221\\01243}{}$,
$\subspace{10233\\01232}{}$,
$\subspace{10240\\01221}{}$,
$\subspace{10202\\01432}{}$,
$\subspace{10214\\01433}{}$,
$\subspace{10221\\01434}{}$,
$\subspace{10233\\01430}{}$,
$\subspace{10240\\01431}{}$,
$\subspace{10401\\01312}{}$,
$\subspace{10413\\01322}{}$,
$\subspace{10420\\01332}{}$,
$\subspace{10432\\01342}{}$,
$\subspace{10444\\01302}{}$,
$\subspace{10200\\01243}{15}$,
$\subspace{10212\\01232}{}$,
$\subspace{10224\\01221}{}$,
$\subspace{10231\\01210}{}$,
$\subspace{10243\\01204}{}$,
$\subspace{10202\\01431}{}$,
$\subspace{10214\\01432}{}$,
$\subspace{10221\\01433}{}$,
$\subspace{10233\\01434}{}$,
$\subspace{10240\\01430}{}$,
$\subspace{10403\\01302}{}$,
$\subspace{10410\\01312}{}$,
$\subspace{10422\\01322}{}$,
$\subspace{10434\\01332}{}$,
$\subspace{10441\\01342}{}$,
$\subspace{10201\\01244}{15}$,
$\subspace{10213\\01233}{}$,
$\subspace{10220\\01222}{}$,
$\subspace{10232\\01211}{}$,
$\subspace{10244\\01200}{}$,
$\subspace{10204\\01442}{}$,
$\subspace{10211\\01443}{}$,
$\subspace{10223\\01444}{}$,
$\subspace{10230\\01440}{}$,
$\subspace{10242\\01441}{}$,
$\subspace{10403\\01330}{}$,
$\subspace{10410\\01340}{}$,
$\subspace{10422\\01300}{}$,
$\subspace{10434\\01310}{}$,
$\subspace{10441\\01320}{}$,
$\subspace{10203\\01244}{15}$,
$\subspace{10210\\01233}{}$,
$\subspace{10222\\01222}{}$,
$\subspace{10234\\01211}{}$,
$\subspace{10241\\01200}{}$,
$\subspace{10200\\01440}{}$,
$\subspace{10212\\01441}{}$,
$\subspace{10224\\01442}{}$,
$\subspace{10231\\01443}{}$,
$\subspace{10243\\01444}{}$,
$\subspace{10400\\01300}{}$,
$\subspace{10412\\01310}{}$,
$\subspace{10424\\01320}{}$,
$\subspace{10431\\01330}{}$,
$\subspace{10443\\01340}{}$,
$\subspace{10303\\01303}{15}$,
$\subspace{10310\\01333}{}$,
$\subspace{10322\\01313}{}$,
$\subspace{10334\\01343}{}$,
$\subspace{10341\\01323}{}$,
$\subspace{10404\\01212}{}$,
$\subspace{10411\\01234}{}$,
$\subspace{10423\\01201}{}$,
$\subspace{10430\\01223}{}$,
$\subspace{10442\\01240}{}$,
$\subspace{10400\\01401}{}$,
$\subspace{10412\\01404}{}$,
$\subspace{10424\\01402}{}$,
$\subspace{10431\\01400}{}$,
$\subspace{10443\\01403}{}$,
$\subspace{10300\\01310}{15}$,
$\subspace{10312\\01340}{}$,
$\subspace{10324\\01320}{}$,
$\subspace{10331\\01300}{}$,
$\subspace{10343\\01330}{}$,
$\subspace{10402\\01233}{}$,
$\subspace{10414\\01200}{}$,
$\subspace{10421\\01222}{}$,
$\subspace{10433\\01244}{}$,
$\subspace{10440\\01211}{}$,
$\subspace{10404\\01443}{}$,
$\subspace{10411\\01441}{}$,
$\subspace{10423\\01444}{}$,
$\subspace{10430\\01442}{}$,
$\subspace{10442\\01440}{}$,
$\subspace{10300\\01310}{15}$,
$\subspace{10312\\01340}{}$,
$\subspace{10324\\01320}{}$,
$\subspace{10331\\01300}{}$,
$\subspace{10343\\01330}{}$,
$\subspace{10402\\01233}{}$,
$\subspace{10414\\01200}{}$,
$\subspace{10421\\01222}{}$,
$\subspace{10433\\01244}{}$,
$\subspace{10440\\01211}{}$,
$\subspace{10404\\01443}{}$,
$\subspace{10411\\01441}{}$,
$\subspace{10423\\01444}{}$,
$\subspace{10430\\01442}{}$,
$\subspace{10442\\01440}{}$,
$\subspace{10300\\01310}{15}$,
$\subspace{10312\\01340}{}$,
$\subspace{10324\\01320}{}$,
$\subspace{10331\\01300}{}$,
$\subspace{10343\\01330}{}$,
$\subspace{10402\\01233}{}$,
$\subspace{10414\\01200}{}$,
$\subspace{10421\\01222}{}$,
$\subspace{10433\\01244}{}$,
$\subspace{10440\\01211}{}$,
$\subspace{10404\\01443}{}$,
$\subspace{10411\\01441}{}$,
$\subspace{10423\\01444}{}$,
$\subspace{10430\\01442}{}$,
$\subspace{10442\\01440}{}$,
$\subspace{10304\\01311}{15}$,
$\subspace{10311\\01341}{}$,
$\subspace{10323\\01321}{}$,
$\subspace{10330\\01301}{}$,
$\subspace{10342\\01331}{}$,
$\subspace{10402\\01213}{}$,
$\subspace{10414\\01230}{}$,
$\subspace{10421\\01202}{}$,
$\subspace{10433\\01224}{}$,
$\subspace{10440\\01241}{}$,
$\subspace{10402\\01414}{}$,
$\subspace{10414\\01412}{}$,
$\subspace{10421\\01410}{}$,
$\subspace{10433\\01413}{}$,
$\subspace{10440\\01411}{}$,
$\subspace{10300\\01333}{15}$,
$\subspace{10312\\01313}{}$,
$\subspace{10324\\01343}{}$,
$\subspace{10331\\01323}{}$,
$\subspace{10343\\01303}{}$,
$\subspace{10403\\01240}{}$,
$\subspace{10410\\01212}{}$,
$\subspace{10422\\01234}{}$,
$\subspace{10434\\01201}{}$,
$\subspace{10441\\01223}{}$,
$\subspace{10404\\01401}{}$,
$\subspace{10411\\01404}{}$,
$\subspace{10423\\01402}{}$,
$\subspace{10430\\01400}{}$,
$\subspace{10442\\01403}{}$,
$\subspace{10302\\01341}{15}$,
$\subspace{10314\\01321}{}$,
$\subspace{10321\\01301}{}$,
$\subspace{10333\\01331}{}$,
$\subspace{10340\\01311}{}$,
$\subspace{10404\\01224}{}$,
$\subspace{10411\\01241}{}$,
$\subspace{10423\\01213}{}$,
$\subspace{10430\\01230}{}$,
$\subspace{10442\\01202}{}$,
$\subspace{10402\\01412}{}$,
$\subspace{10414\\01410}{}$,
$\subspace{10421\\01413}{}$,
$\subspace{10433\\01411}{}$,
$\subspace{10440\\01414}{}$,
$\subspace{11202\\00011}{3}$,
$\subspace{11410\\00001}{}$,
$\subspace{13403\\00010}{}$,
$\subspace{12103\\00010}{3}$,
$\subspace{14130\\00001}{}$,
$\subspace{14302\\00011}{}$.

\smallskip

$n_5(5,2;44)\ge 1096$,$\left[\!\begin{smallmatrix}10000\\00100\\04410\\00001\\00044\end{smallmatrix}\!\right]$:
$\subspace{01001\\00104}{15}$,
$\subspace{01004\\00101}{}$,
$\subspace{01004\\00123}{}$,
$\subspace{01010\\00111}{}$,
$\subspace{01013\\00113}{}$,
$\subspace{01013\\00130}{}$,
$\subspace{01022\\00120}{}$,
$\subspace{01024\\00123}{}$,
$\subspace{01022\\00142}{}$,
$\subspace{01031\\00104}{}$,
$\subspace{01031\\00132}{}$,
$\subspace{01033\\00130}{}$,
$\subspace{01040\\00111}{}$,
$\subspace{01040\\00144}{}$,
$\subspace{01042\\00142}{}$,
$\subspace{01000\\00112}{15}$,
$\subspace{01000\\00140}{}$,
$\subspace{01002\\00143}{}$,
$\subspace{01011\\00100}{}$,
$\subspace{01014\\00102}{}$,
$\subspace{01014\\00124}{}$,
$\subspace{01020\\00112}{}$,
$\subspace{01023\\00114}{}$,
$\subspace{01023\\00131}{}$,
$\subspace{01032\\00121}{}$,
$\subspace{01034\\00124}{}$,
$\subspace{01032\\00143}{}$,
$\subspace{01041\\00100}{}$,
$\subspace{01041\\00133}{}$,
$\subspace{01043\\00131}{}$,
$\subspace{01000\\00124}{5}$,
$\subspace{01014\\00131}{}$,
$\subspace{01023\\00143}{}$,
$\subspace{01032\\00100}{}$,
$\subspace{01041\\00112}{}$,
$\subspace{01004\\00130}{5}$,
$\subspace{01013\\00142}{}$,
$\subspace{01022\\00104}{}$,
$\subspace{01031\\00111}{}$,
$\subspace{01040\\00123}{}$,
$\subspace{10000\\00001}{3}$,
$\subspace{10000\\00010}{}$,
$\subspace{10000\\00011}{}$,
$\subspace{10000\\00001}{3}$,
$\subspace{10000\\00010}{}$,
$\subspace{10000\\00011}{}$,
$\subspace{10000\\00012}{3}$,
$\subspace{10000\\00013}{}$,
$\subspace{10000\\00014}{}$,
$\subspace{10001\\00012}{3}$,
$\subspace{10001\\00014}{}$,
$\subspace{10002\\00013}{}$,
$\subspace{10002\\00010}{3}$,
$\subspace{10003\\00011}{}$,
$\subspace{10030\\00001}{}$,
$\subspace{10002\\00011}{3}$,
$\subspace{10003\\00010}{}$,
$\subspace{10020\\00001}{}$,
$\subspace{10002\\00012}{3}$,
$\subspace{10002\\00014}{}$,
$\subspace{10004\\00013}{}$,
$\subspace{10002\\00012}{3}$,
$\subspace{10002\\00014}{}$,
$\subspace{10004\\00013}{}$,
$\subspace{10021\\00104}{15}$,
$\subspace{10021\\00111}{}$,
$\subspace{10021\\00123}{}$,
$\subspace{10021\\00130}{}$,
$\subspace{10021\\00142}{}$,
$\subspace{10043\\01004}{}$,
$\subspace{10043\\01013}{}$,
$\subspace{10043\\01022}{}$,
$\subspace{10043\\01031}{}$,
$\subspace{10043\\01040}{}$,
$\subspace{10041\\01100}{}$,
$\subspace{10041\\01113}{}$,
$\subspace{10041\\01121}{}$,
$\subspace{10041\\01134}{}$,
$\subspace{10041\\01142}{}$,
$\subspace{10031\\00102}{15}$,
$\subspace{10031\\00114}{}$,
$\subspace{10031\\00121}{}$,
$\subspace{10031\\00133}{}$,
$\subspace{10031\\00140}{}$,
$\subspace{10032\\01002}{}$,
$\subspace{10032\\01011}{}$,
$\subspace{10032\\01020}{}$,
$\subspace{10032\\01034}{}$,
$\subspace{10032\\01043}{}$,
$\subspace{10042\\01104}{}$,
$\subspace{10042\\01112}{}$,
$\subspace{10042\\01120}{}$,
$\subspace{10042\\01133}{}$,
$\subspace{10042\\01141}{}$,
$\subspace{10034\\00100}{15}$,
$\subspace{10034\\00112}{}$,
$\subspace{10034\\00124}{}$,
$\subspace{10034\\00131}{}$,
$\subspace{10034\\00143}{}$,
$\subspace{10012\\01000}{}$,
$\subspace{10012\\01014}{}$,
$\subspace{10012\\01023}{}$,
$\subspace{10012\\01032}{}$,
$\subspace{10012\\01041}{}$,
$\subspace{10014\\01103}{}$,
$\subspace{10014\\01111}{}$,
$\subspace{10014\\01124}{}$,
$\subspace{10014\\01132}{}$,
$\subspace{10014\\01140}{}$,
$\subspace{10203\\00012}{3}$,
$\subspace{12003\\00014}{}$,
$\subspace{13300\\00013}{}$,
$\subspace{10203\\00012}{3}$,
$\subspace{12003\\00014}{}$,
$\subspace{13300\\00013}{}$,
$\subspace{10301\\00012}{3}$,
$\subspace{12203\\00013}{}$,
$\subspace{13001\\00014}{}$,
$\subspace{10301\\00012}{3}$,
$\subspace{12203\\00013}{}$,
$\subspace{13001\\00014}{}$,
$\subspace{10022\\01200}{15}$,
$\subspace{10022\\01211}{}$,
$\subspace{10022\\01222}{}$,
$\subspace{10022\\01233}{}$,
$\subspace{10022\\01244}{}$,
$\subspace{10030\\01300}{}$,
$\subspace{10030\\01310}{}$,
$\subspace{10030\\01320}{}$,
$\subspace{10030\\01330}{}$,
$\subspace{10030\\01340}{}$,
$\subspace{10003\\01440}{}$,
$\subspace{10003\\01441}{}$,
$\subspace{10003\\01442}{}$,
$\subspace{10003\\01443}{}$,
$\subspace{10003\\01444}{}$,
$\subspace{10023\\01201}{15}$,
$\subspace{10023\\01212}{}$,
$\subspace{10023\\01223}{}$,
$\subspace{10023\\01234}{}$,
$\subspace{10023\\01240}{}$,
$\subspace{10024\\01303}{}$,
$\subspace{10024\\01313}{}$,
$\subspace{10024\\01323}{}$,
$\subspace{10024\\01333}{}$,
$\subspace{10024\\01343}{}$,
$\subspace{10013\\01400}{}$,
$\subspace{10013\\01401}{}$,
$\subspace{10013\\01402}{}$,
$\subspace{10013\\01403}{}$,
$\subspace{10013\\01404}{}$,
$\subspace{10040\\01204}{15}$,
$\subspace{10040\\01210}{}$,
$\subspace{10040\\01221}{}$,
$\subspace{10040\\01232}{}$,
$\subspace{10040\\01243}{}$,
$\subspace{10004\\01302}{}$,
$\subspace{10004\\01312}{}$,
$\subspace{10004\\01322}{}$,
$\subspace{10004\\01332}{}$,
$\subspace{10004\\01342}{}$,
$\subspace{10011\\01430}{}$,
$\subspace{10011\\01431}{}$,
$\subspace{10011\\01432}{}$,
$\subspace{10011\\01433}{}$,
$\subspace{10011\\01434}{}$,
$\subspace{10044\\01201}{15}$,
$\subspace{10044\\01212}{}$,
$\subspace{10044\\01223}{}$,
$\subspace{10044\\01234}{}$,
$\subspace{10044\\01240}{}$,
$\subspace{10010\\01303}{}$,
$\subspace{10010\\01313}{}$,
$\subspace{10010\\01323}{}$,
$\subspace{10010\\01333}{}$,
$\subspace{10010\\01343}{}$,
$\subspace{10001\\01400}{}$,
$\subspace{10001\\01401}{}$,
$\subspace{10001\\01402}{}$,
$\subspace{10001\\01403}{}$,
$\subspace{10001\\01404}{}$,
$\subspace{10104\\01030}{15}$,
$\subspace{10111\\01044}{}$,
$\subspace{10123\\01003}{}$,
$\subspace{10130\\01012}{}$,
$\subspace{10142\\01021}{}$,
$\subspace{10401\\01144}{}$,
$\subspace{10413\\01123}{}$,
$\subspace{10420\\01102}{}$,
$\subspace{10432\\01131}{}$,
$\subspace{10444\\01110}{}$,
$\subspace{14001\\00122}{}$,
$\subspace{14010\\00110}{}$,
$\subspace{14024\\00103}{}$,
$\subspace{14033\\00141}{}$,
$\subspace{14042\\00134}{}$,
$\subspace{10104\\01030}{15}$,
$\subspace{10111\\01044}{}$,
$\subspace{10123\\01003}{}$,
$\subspace{10130\\01012}{}$,
$\subspace{10142\\01021}{}$,
$\subspace{10401\\01144}{}$,
$\subspace{10413\\01123}{}$,
$\subspace{10420\\01102}{}$,
$\subspace{10432\\01131}{}$,
$\subspace{10444\\01110}{}$,
$\subspace{14001\\00122}{}$,
$\subspace{14010\\00110}{}$,
$\subspace{14024\\00103}{}$,
$\subspace{14033\\00141}{}$,
$\subspace{14042\\00134}{}$,
$\subspace{10104\\01030}{15}$,
$\subspace{10111\\01044}{}$,
$\subspace{10123\\01003}{}$,
$\subspace{10130\\01012}{}$,
$\subspace{10142\\01021}{}$,
$\subspace{10401\\01144}{}$,
$\subspace{10413\\01123}{}$,
$\subspace{10420\\01102}{}$,
$\subspace{10432\\01131}{}$,
$\subspace{10444\\01110}{}$,
$\subspace{14001\\00122}{}$,
$\subspace{14010\\00110}{}$,
$\subspace{14024\\00103}{}$,
$\subspace{14033\\00141}{}$,
$\subspace{14042\\00134}{}$,
$\subspace{10104\\01032}{15}$,
$\subspace{10111\\01041}{}$,
$\subspace{10123\\01000}{}$,
$\subspace{10130\\01014}{}$,
$\subspace{10142\\01023}{}$,
$\subspace{10402\\01111}{}$,
$\subspace{10414\\01140}{}$,
$\subspace{10421\\01124}{}$,
$\subspace{10433\\01103}{}$,
$\subspace{10440\\01132}{}$,
$\subspace{14002\\00131}{}$,
$\subspace{14011\\00124}{}$,
$\subspace{14020\\00112}{}$,
$\subspace{14034\\00100}{}$,
$\subspace{14043\\00143}{}$,
$\subspace{10104\\01043}{15}$,
$\subspace{10111\\01002}{}$,
$\subspace{10123\\01011}{}$,
$\subspace{10130\\01020}{}$,
$\subspace{10142\\01034}{}$,
$\subspace{10401\\01112}{}$,
$\subspace{10413\\01141}{}$,
$\subspace{10420\\01120}{}$,
$\subspace{10432\\01104}{}$,
$\subspace{10444\\01133}{}$,
$\subspace{14001\\00114}{}$,
$\subspace{14010\\00102}{}$,
$\subspace{14024\\00140}{}$,
$\subspace{14033\\00133}{}$,
$\subspace{14042\\00121}{}$,
$\subspace{10100\\01114}{15}$,
$\subspace{10112\\01130}{}$,
$\subspace{10124\\01101}{}$,
$\subspace{10131\\01122}{}$,
$\subspace{10143\\01143}{}$,
$\subspace{10400\\01024}{}$,
$\subspace{10412\\01010}{}$,
$\subspace{10424\\01001}{}$,
$\subspace{10431\\01042}{}$,
$\subspace{10443\\01033}{}$,
$\subspace{11000\\00132}{}$,
$\subspace{11014\\00144}{}$,
$\subspace{11023\\00101}{}$,
$\subspace{11032\\00113}{}$,
$\subspace{11041\\00120}{}$,
$\subspace{10100\\01114}{15}$,
$\subspace{10112\\01130}{}$,
$\subspace{10124\\01101}{}$,
$\subspace{10131\\01122}{}$,
$\subspace{10143\\01143}{}$,
$\subspace{10400\\01024}{}$,
$\subspace{10412\\01010}{}$,
$\subspace{10424\\01001}{}$,
$\subspace{10431\\01042}{}$,
$\subspace{10443\\01033}{}$,
$\subspace{11000\\00132}{}$,
$\subspace{11014\\00144}{}$,
$\subspace{11023\\00101}{}$,
$\subspace{11032\\00113}{}$,
$\subspace{11041\\00120}{}$,
$\subspace{10101\\01120}{15}$,
$\subspace{10113\\01141}{}$,
$\subspace{10120\\01112}{}$,
$\subspace{10132\\01133}{}$,
$\subspace{10144\\01104}{}$,
$\subspace{10404\\01034}{}$,
$\subspace{10411\\01020}{}$,
$\subspace{10423\\01011}{}$,
$\subspace{10430\\01002}{}$,
$\subspace{10442\\01043}{}$,
$\subspace{11001\\00121}{}$,
$\subspace{11010\\00133}{}$,
$\subspace{11024\\00140}{}$,
$\subspace{11033\\00102}{}$,
$\subspace{11042\\00114}{}$,
$\subspace{10101\\01120}{15}$,
$\subspace{10113\\01141}{}$,
$\subspace{10120\\01112}{}$,
$\subspace{10132\\01133}{}$,
$\subspace{10144\\01104}{}$,
$\subspace{10404\\01034}{}$,
$\subspace{10411\\01020}{}$,
$\subspace{10423\\01011}{}$,
$\subspace{10430\\01002}{}$,
$\subspace{10442\\01043}{}$,
$\subspace{11001\\00121}{}$,
$\subspace{11010\\00133}{}$,
$\subspace{11024\\00140}{}$,
$\subspace{11033\\00102}{}$,
$\subspace{11042\\00114}{}$,
$\subspace{10101\\01120}{15}$,
$\subspace{10113\\01141}{}$,
$\subspace{10120\\01112}{}$,
$\subspace{10132\\01133}{}$,
$\subspace{10144\\01104}{}$,
$\subspace{10404\\01034}{}$,
$\subspace{10411\\01020}{}$,
$\subspace{10423\\01011}{}$,
$\subspace{10430\\01002}{}$,
$\subspace{10442\\01043}{}$,
$\subspace{11001\\00121}{}$,
$\subspace{11010\\00133}{}$,
$\subspace{11024\\00140}{}$,
$\subspace{11033\\00102}{}$,
$\subspace{11042\\00114}{}$,
$\subspace{10101\\01120}{15}$,
$\subspace{10113\\01141}{}$,
$\subspace{10120\\01112}{}$,
$\subspace{10132\\01133}{}$,
$\subspace{10144\\01104}{}$,
$\subspace{10404\\01034}{}$,
$\subspace{10411\\01020}{}$,
$\subspace{10423\\01011}{}$,
$\subspace{10430\\01002}{}$,
$\subspace{10442\\01043}{}$,
$\subspace{11001\\00121}{}$,
$\subspace{11010\\00133}{}$,
$\subspace{11024\\00140}{}$,
$\subspace{11033\\00102}{}$,
$\subspace{11042\\00114}{}$,
$\subspace{10102\\01131}{15}$,
$\subspace{10114\\01102}{}$,
$\subspace{10121\\01123}{}$,
$\subspace{10133\\01144}{}$,
$\subspace{10140\\01110}{}$,
$\subspace{10403\\01044}{}$,
$\subspace{10410\\01030}{}$,
$\subspace{10422\\01021}{}$,
$\subspace{10434\\01012}{}$,
$\subspace{10441\\01003}{}$,
$\subspace{11002\\00110}{}$,
$\subspace{11011\\00122}{}$,
$\subspace{11020\\00134}{}$,
$\subspace{11034\\00141}{}$,
$\subspace{11043\\00103}{}$,
$\subspace{10104\\01130}{15}$,
$\subspace{10111\\01101}{}$,
$\subspace{10123\\01122}{}$,
$\subspace{10130\\01143}{}$,
$\subspace{10142\\01114}{}$,
$\subspace{10401\\01001}{}$,
$\subspace{10413\\01042}{}$,
$\subspace{10420\\01033}{}$,
$\subspace{10432\\01024}{}$,
$\subspace{10444\\01010}{}$,
$\subspace{11004\\00101}{}$,
$\subspace{11013\\00113}{}$,
$\subspace{11022\\00120}{}$,
$\subspace{11031\\00132}{}$,
$\subspace{11040\\00144}{}$,
$\subspace{10102\\01211}{15}$,
$\subspace{10114\\01244}{}$,
$\subspace{10121\\01222}{}$,
$\subspace{10133\\01200}{}$,
$\subspace{10140\\01233}{}$,
$\subspace{10101\\01442}{}$,
$\subspace{10113\\01444}{}$,
$\subspace{10120\\01441}{}$,
$\subspace{10132\\01443}{}$,
$\subspace{10144\\01440}{}$,
$\subspace{10201\\01300}{}$,
$\subspace{10213\\01320}{}$,
$\subspace{10220\\01340}{}$,
$\subspace{10232\\01310}{}$,
$\subspace{10244\\01330}{}$,
$\subspace{10104\\01212}{15}$,
$\subspace{10111\\01240}{}$,
$\subspace{10123\\01223}{}$,
$\subspace{10130\\01201}{}$,
$\subspace{10142\\01234}{}$,
$\subspace{10104\\01400}{}$,
$\subspace{10111\\01402}{}$,
$\subspace{10123\\01404}{}$,
$\subspace{10130\\01401}{}$,
$\subspace{10142\\01403}{}$,
$\subspace{10200\\01313}{}$,
$\subspace{10212\\01333}{}$,
$\subspace{10224\\01303}{}$,
$\subspace{10231\\01323}{}$,
$\subspace{10243\\01343}{}$,
$\subspace{10100\\01224}{15}$,
$\subspace{10112\\01202}{}$,
$\subspace{10124\\01230}{}$,
$\subspace{10131\\01213}{}$,
$\subspace{10143\\01241}{}$,
$\subspace{10100\\01414}{}$,
$\subspace{10112\\01411}{}$,
$\subspace{10124\\01413}{}$,
$\subspace{10131\\01410}{}$,
$\subspace{10143\\01412}{}$,
$\subspace{10202\\01321}{}$,
$\subspace{10214\\01341}{}$,
$\subspace{10221\\01311}{}$,
$\subspace{10233\\01331}{}$,
$\subspace{10240\\01301}{}$,
$\subspace{10100\\01224}{15}$,
$\subspace{10112\\01202}{}$,
$\subspace{10124\\01230}{}$,
$\subspace{10131\\01213}{}$,
$\subspace{10143\\01241}{}$,
$\subspace{10100\\01414}{}$,
$\subspace{10112\\01411}{}$,
$\subspace{10124\\01413}{}$,
$\subspace{10131\\01410}{}$,
$\subspace{10143\\01412}{}$,
$\subspace{10202\\01321}{}$,
$\subspace{10214\\01341}{}$,
$\subspace{10221\\01311}{}$,
$\subspace{10233\\01331}{}$,
$\subspace{10240\\01301}{}$,
$\subspace{10100\\01232}{15}$,
$\subspace{10112\\01210}{}$,
$\subspace{10124\\01243}{}$,
$\subspace{10131\\01221}{}$,
$\subspace{10143\\01204}{}$,
$\subspace{10100\\01432}{}$,
$\subspace{10112\\01434}{}$,
$\subspace{10124\\01431}{}$,
$\subspace{10131\\01433}{}$,
$\subspace{10143\\01430}{}$,
$\subspace{10200\\01332}{}$,
$\subspace{10212\\01302}{}$,
$\subspace{10224\\01322}{}$,
$\subspace{10231\\01342}{}$,
$\subspace{10243\\01312}{}$,
$\subspace{10103\\01234}{15}$,
$\subspace{10110\\01212}{}$,
$\subspace{10122\\01240}{}$,
$\subspace{10134\\01223}{}$,
$\subspace{10141\\01201}{}$,
$\subspace{10103\\01400}{}$,
$\subspace{10110\\01402}{}$,
$\subspace{10122\\01404}{}$,
$\subspace{10134\\01401}{}$,
$\subspace{10141\\01403}{}$,
$\subspace{10202\\01333}{}$,
$\subspace{10214\\01303}{}$,
$\subspace{10221\\01323}{}$,
$\subspace{10233\\01343}{}$,
$\subspace{10240\\01313}{}$,
$\subspace{10101\\01244}{15}$,
$\subspace{10113\\01222}{}$,
$\subspace{10120\\01200}{}$,
$\subspace{10132\\01233}{}$,
$\subspace{10144\\01211}{}$,
$\subspace{10101\\01441}{}$,
$\subspace{10113\\01443}{}$,
$\subspace{10120\\01440}{}$,
$\subspace{10132\\01442}{}$,
$\subspace{10144\\01444}{}$,
$\subspace{10202\\01340}{}$,
$\subspace{10214\\01310}{}$,
$\subspace{10221\\01330}{}$,
$\subspace{10233\\01300}{}$,
$\subspace{10240\\01320}{}$,
$\subspace{10103\\01242}{15}$,
$\subspace{10110\\01220}{}$,
$\subspace{10122\\01203}{}$,
$\subspace{10134\\01231}{}$,
$\subspace{10141\\01214}{}$,
$\subspace{10103\\01423}{}$,
$\subspace{10110\\01420}{}$,
$\subspace{10122\\01422}{}$,
$\subspace{10134\\01424}{}$,
$\subspace{10141\\01421}{}$,
$\subspace{10200\\01344}{}$,
$\subspace{10212\\01314}{}$,
$\subspace{10224\\01334}{}$,
$\subspace{10231\\01304}{}$,
$\subspace{10243\\01324}{}$,
$\subspace{10102\\01304}{15}$,
$\subspace{10114\\01344}{}$,
$\subspace{10121\\01334}{}$,
$\subspace{10133\\01324}{}$,
$\subspace{10140\\01314}{}$,
$\subspace{10302\\01214}{}$,
$\subspace{10314\\01220}{}$,
$\subspace{10321\\01231}{}$,
$\subspace{10333\\01242}{}$,
$\subspace{10340\\01203}{}$,
$\subspace{10304\\01424}{}$,
$\subspace{10311\\01423}{}$,
$\subspace{10323\\01422}{}$,
$\subspace{10330\\01421}{}$,
$\subspace{10342\\01420}{}$,
$\subspace{10102\\01304}{15}$,
$\subspace{10114\\01344}{}$,
$\subspace{10121\\01334}{}$,
$\subspace{10133\\01324}{}$,
$\subspace{10140\\01314}{}$,
$\subspace{10302\\01214}{}$,
$\subspace{10314\\01220}{}$,
$\subspace{10321\\01231}{}$,
$\subspace{10333\\01242}{}$,
$\subspace{10340\\01203}{}$,
$\subspace{10304\\01424}{}$,
$\subspace{10311\\01423}{}$,
$\subspace{10323\\01422}{}$,
$\subspace{10330\\01421}{}$,
$\subspace{10342\\01420}{}$,
$\subspace{10102\\01304}{15}$,
$\subspace{10114\\01344}{}$,
$\subspace{10121\\01334}{}$,
$\subspace{10133\\01324}{}$,
$\subspace{10140\\01314}{}$,
$\subspace{10302\\01214}{}$,
$\subspace{10314\\01220}{}$,
$\subspace{10321\\01231}{}$,
$\subspace{10333\\01242}{}$,
$\subspace{10340\\01203}{}$,
$\subspace{10304\\01424}{}$,
$\subspace{10311\\01423}{}$,
$\subspace{10323\\01422}{}$,
$\subspace{10330\\01421}{}$,
$\subspace{10342\\01420}{}$,
$\subspace{10102\\01304}{15}$,
$\subspace{10114\\01344}{}$,
$\subspace{10121\\01334}{}$,
$\subspace{10133\\01324}{}$,
$\subspace{10140\\01314}{}$,
$\subspace{10302\\01214}{}$,
$\subspace{10314\\01220}{}$,
$\subspace{10321\\01231}{}$,
$\subspace{10333\\01242}{}$,
$\subspace{10340\\01203}{}$,
$\subspace{10304\\01424}{}$,
$\subspace{10311\\01423}{}$,
$\subspace{10323\\01422}{}$,
$\subspace{10330\\01421}{}$,
$\subspace{10342\\01420}{}$,
$\subspace{10102\\01304}{15}$,
$\subspace{10114\\01344}{}$,
$\subspace{10121\\01334}{}$,
$\subspace{10133\\01324}{}$,
$\subspace{10140\\01314}{}$,
$\subspace{10302\\01214}{}$,
$\subspace{10314\\01220}{}$,
$\subspace{10321\\01231}{}$,
$\subspace{10333\\01242}{}$,
$\subspace{10340\\01203}{}$,
$\subspace{10304\\01424}{}$,
$\subspace{10311\\01423}{}$,
$\subspace{10323\\01422}{}$,
$\subspace{10330\\01421}{}$,
$\subspace{10342\\01420}{}$,
$\subspace{10103\\01312}{15}$,
$\subspace{10110\\01302}{}$,
$\subspace{10122\\01342}{}$,
$\subspace{10134\\01332}{}$,
$\subspace{10141\\01322}{}$,
$\subspace{10300\\01221}{}$,
$\subspace{10312\\01232}{}$,
$\subspace{10324\\01243}{}$,
$\subspace{10331\\01204}{}$,
$\subspace{10343\\01210}{}$,
$\subspace{10302\\01433}{}$,
$\subspace{10314\\01432}{}$,
$\subspace{10321\\01431}{}$,
$\subspace{10333\\01430}{}$,
$\subspace{10340\\01434}{}$,
$\subspace{10103\\01312}{15}$,
$\subspace{10110\\01302}{}$,
$\subspace{10122\\01342}{}$,
$\subspace{10134\\01332}{}$,
$\subspace{10141\\01322}{}$,
$\subspace{10300\\01221}{}$,
$\subspace{10312\\01232}{}$,
$\subspace{10324\\01243}{}$,
$\subspace{10331\\01204}{}$,
$\subspace{10343\\01210}{}$,
$\subspace{10302\\01433}{}$,
$\subspace{10314\\01432}{}$,
$\subspace{10321\\01431}{}$,
$\subspace{10333\\01430}{}$,
$\subspace{10340\\01434}{}$,
$\subspace{10103\\01312}{15}$,
$\subspace{10110\\01302}{}$,
$\subspace{10122\\01342}{}$,
$\subspace{10134\\01332}{}$,
$\subspace{10141\\01322}{}$,
$\subspace{10300\\01221}{}$,
$\subspace{10312\\01232}{}$,
$\subspace{10324\\01243}{}$,
$\subspace{10331\\01204}{}$,
$\subspace{10343\\01210}{}$,
$\subspace{10302\\01433}{}$,
$\subspace{10314\\01432}{}$,
$\subspace{10321\\01431}{}$,
$\subspace{10333\\01430}{}$,
$\subspace{10340\\01434}{}$,
$\subspace{10204\\01002}{15}$,
$\subspace{10211\\01034}{}$,
$\subspace{10223\\01011}{}$,
$\subspace{10230\\01043}{}$,
$\subspace{10242\\01020}{}$,
$\subspace{10301\\01104}{}$,
$\subspace{10313\\01141}{}$,
$\subspace{10320\\01133}{}$,
$\subspace{10332\\01120}{}$,
$\subspace{10344\\01112}{}$,
$\subspace{13001\\00102}{}$,
$\subspace{13010\\00121}{}$,
$\subspace{13024\\00140}{}$,
$\subspace{13033\\00114}{}$,
$\subspace{13042\\00133}{}$,
$\subspace{10204\\01014}{15}$,
$\subspace{10211\\01041}{}$,
$\subspace{10223\\01023}{}$,
$\subspace{10230\\01000}{}$,
$\subspace{10242\\01032}{}$,
$\subspace{10300\\01111}{}$,
$\subspace{10312\\01103}{}$,
$\subspace{10324\\01140}{}$,
$\subspace{10331\\01132}{}$,
$\subspace{10343\\01124}{}$,
$\subspace{13000\\00112}{}$,
$\subspace{13014\\00131}{}$,
$\subspace{13023\\00100}{}$,
$\subspace{13032\\00124}{}$,
$\subspace{13041\\00143}{}$,
$\subspace{10204\\01014}{15}$,
$\subspace{10211\\01041}{}$,
$\subspace{10223\\01023}{}$,
$\subspace{10230\\01000}{}$,
$\subspace{10242\\01032}{}$,
$\subspace{10300\\01111}{}$,
$\subspace{10312\\01103}{}$,
$\subspace{10324\\01140}{}$,
$\subspace{10331\\01132}{}$,
$\subspace{10343\\01124}{}$,
$\subspace{13000\\00112}{}$,
$\subspace{13014\\00131}{}$,
$\subspace{13023\\00100}{}$,
$\subspace{13032\\00124}{}$,
$\subspace{13041\\00143}{}$,
$\subspace{10204\\01014}{15}$,
$\subspace{10211\\01041}{}$,
$\subspace{10223\\01023}{}$,
$\subspace{10230\\01000}{}$,
$\subspace{10242\\01032}{}$,
$\subspace{10300\\01111}{}$,
$\subspace{10312\\01103}{}$,
$\subspace{10324\\01140}{}$,
$\subspace{10331\\01132}{}$,
$\subspace{10343\\01124}{}$,
$\subspace{13000\\00112}{}$,
$\subspace{13014\\00131}{}$,
$\subspace{13023\\00100}{}$,
$\subspace{13032\\00124}{}$,
$\subspace{13041\\00143}{}$,
$\subspace{10200\\01022}{15}$,
$\subspace{10212\\01004}{}$,
$\subspace{10224\\01031}{}$,
$\subspace{10231\\01013}{}$,
$\subspace{10243\\01040}{}$,
$\subspace{10303\\01142}{}$,
$\subspace{10310\\01134}{}$,
$\subspace{10322\\01121}{}$,
$\subspace{10334\\01113}{}$,
$\subspace{10341\\01100}{}$,
$\subspace{13003\\00123}{}$,
$\subspace{13012\\00142}{}$,
$\subspace{13021\\00111}{}$,
$\subspace{13030\\00130}{}$,
$\subspace{13044\\00104}{}$,
$\subspace{10201\\01024}{15}$,
$\subspace{10213\\01001}{}$,
$\subspace{10220\\01033}{}$,
$\subspace{10232\\01010}{}$,
$\subspace{10244\\01042}{}$,
$\subspace{10303\\01122}{}$,
$\subspace{10310\\01114}{}$,
$\subspace{10322\\01101}{}$,
$\subspace{10334\\01143}{}$,
$\subspace{10341\\01130}{}$,
$\subspace{13003\\00101}{}$,
$\subspace{13012\\00120}{}$,
$\subspace{13021\\00144}{}$,
$\subspace{13030\\00113}{}$,
$\subspace{13044\\00132}{}$,
$\subspace{10201\\01024}{15}$,
$\subspace{10213\\01001}{}$,
$\subspace{10220\\01033}{}$,
$\subspace{10232\\01010}{}$,
$\subspace{10244\\01042}{}$,
$\subspace{10303\\01122}{}$,
$\subspace{10310\\01114}{}$,
$\subspace{10322\\01101}{}$,
$\subspace{10334\\01143}{}$,
$\subspace{10341\\01130}{}$,
$\subspace{13003\\00101}{}$,
$\subspace{13012\\00120}{}$,
$\subspace{13021\\00144}{}$,
$\subspace{13030\\00113}{}$,
$\subspace{13044\\00132}{}$,
$\subspace{10204\\01022}{15}$,
$\subspace{10211\\01004}{}$,
$\subspace{10223\\01031}{}$,
$\subspace{10230\\01013}{}$,
$\subspace{10242\\01040}{}$,
$\subspace{10300\\01134}{}$,
$\subspace{10312\\01121}{}$,
$\subspace{10324\\01113}{}$,
$\subspace{10331\\01100}{}$,
$\subspace{10343\\01142}{}$,
$\subspace{13000\\00104}{}$,
$\subspace{13014\\00123}{}$,
$\subspace{13023\\00142}{}$,
$\subspace{13032\\00111}{}$,
$\subspace{13041\\00130}{}$,
$\subspace{10201\\01101}{15}$,
$\subspace{10213\\01114}{}$,
$\subspace{10220\\01122}{}$,
$\subspace{10232\\01130}{}$,
$\subspace{10244\\01143}{}$,
$\subspace{10304\\01033}{}$,
$\subspace{10311\\01001}{}$,
$\subspace{10323\\01024}{}$,
$\subspace{10330\\01042}{}$,
$\subspace{10342\\01010}{}$,
$\subspace{12001\\00120}{}$,
$\subspace{12010\\00101}{}$,
$\subspace{12024\\00132}{}$,
$\subspace{12033\\00113}{}$,
$\subspace{12042\\00144}{}$,
$\subspace{10203\\01102}{15}$,
$\subspace{10210\\01110}{}$,
$\subspace{10222\\01123}{}$,
$\subspace{10234\\01131}{}$,
$\subspace{10241\\01144}{}$,
$\subspace{10304\\01044}{}$,
$\subspace{10311\\01012}{}$,
$\subspace{10323\\01030}{}$,
$\subspace{10330\\01003}{}$,
$\subspace{10342\\01021}{}$,
$\subspace{12003\\00103}{}$,
$\subspace{12012\\00134}{}$,
$\subspace{12021\\00110}{}$,
$\subspace{12030\\00141}{}$,
$\subspace{12044\\00122}{}$,
$\subspace{10200\\01114}{15}$,
$\subspace{10212\\01122}{}$,
$\subspace{10224\\01130}{}$,
$\subspace{10231\\01143}{}$,
$\subspace{10243\\01101}{}$,
$\subspace{10300\\01024}{}$,
$\subspace{10312\\01042}{}$,
$\subspace{10324\\01010}{}$,
$\subspace{10331\\01033}{}$,
$\subspace{10343\\01001}{}$,
$\subspace{12000\\00132}{}$,
$\subspace{12014\\00113}{}$,
$\subspace{12023\\00144}{}$,
$\subspace{12032\\00120}{}$,
$\subspace{12041\\00101}{}$,
$\subspace{10202\\01113}{15}$,
$\subspace{10214\\01121}{}$,
$\subspace{10221\\01134}{}$,
$\subspace{10233\\01142}{}$,
$\subspace{10240\\01100}{}$,
$\subspace{10304\\01040}{}$,
$\subspace{10311\\01013}{}$,
$\subspace{10323\\01031}{}$,
$\subspace{10330\\01004}{}$,
$\subspace{10342\\01022}{}$,
$\subspace{12002\\00130}{}$,
$\subspace{12011\\00111}{}$,
$\subspace{12020\\00142}{}$,
$\subspace{12034\\00123}{}$,
$\subspace{12043\\00104}{}$,
$\subspace{10203\\01123}{15}$,
$\subspace{10210\\01131}{}$,
$\subspace{10222\\01144}{}$,
$\subspace{10234\\01102}{}$,
$\subspace{10241\\01110}{}$,
$\subspace{10303\\01012}{}$,
$\subspace{10310\\01030}{}$,
$\subspace{10322\\01003}{}$,
$\subspace{10334\\01021}{}$,
$\subspace{10341\\01044}{}$,
$\subspace{12003\\00110}{}$,
$\subspace{12012\\00141}{}$,
$\subspace{12021\\00122}{}$,
$\subspace{12030\\00103}{}$,
$\subspace{12044\\00134}{}$,
$\subspace{10202\\01134}{15}$,
$\subspace{10214\\01142}{}$,
$\subspace{10221\\01100}{}$,
$\subspace{10233\\01113}{}$,
$\subspace{10240\\01121}{}$,
$\subspace{10303\\01013}{}$,
$\subspace{10310\\01031}{}$,
$\subspace{10322\\01004}{}$,
$\subspace{10334\\01022}{}$,
$\subspace{10341\\01040}{}$,
$\subspace{12002\\00142}{}$,
$\subspace{12011\\00123}{}$,
$\subspace{12020\\00104}{}$,
$\subspace{12034\\00130}{}$,
$\subspace{12043\\00111}{}$,
$\subspace{10203\\01144}{15}$,
$\subspace{10210\\01102}{}$,
$\subspace{10222\\01110}{}$,
$\subspace{10234\\01123}{}$,
$\subspace{10241\\01131}{}$,
$\subspace{10302\\01030}{}$,
$\subspace{10314\\01003}{}$,
$\subspace{10321\\01021}{}$,
$\subspace{10333\\01044}{}$,
$\subspace{10340\\01012}{}$,
$\subspace{12003\\00122}{}$,
$\subspace{12012\\00103}{}$,
$\subspace{12021\\00134}{}$,
$\subspace{12030\\00110}{}$,
$\subspace{12044\\00141}{}$,
$\subspace{10204\\01144}{15}$,
$\subspace{10211\\01102}{}$,
$\subspace{10223\\01110}{}$,
$\subspace{10230\\01123}{}$,
$\subspace{10242\\01131}{}$,
$\subspace{10303\\01003}{}$,
$\subspace{10310\\01021}{}$,
$\subspace{10322\\01044}{}$,
$\subspace{10334\\01012}{}$,
$\subspace{10341\\01030}{}$,
$\subspace{12004\\00141}{}$,
$\subspace{12013\\00122}{}$,
$\subspace{12022\\00103}{}$,
$\subspace{12031\\00134}{}$,
$\subspace{12040\\00110}{}$,
$\subspace{10201\\01212}{15}$,
$\subspace{10213\\01201}{}$,
$\subspace{10220\\01240}{}$,
$\subspace{10232\\01234}{}$,
$\subspace{10244\\01223}{}$,
$\subspace{10202\\01402}{}$,
$\subspace{10214\\01403}{}$,
$\subspace{10221\\01404}{}$,
$\subspace{10233\\01400}{}$,
$\subspace{10240\\01401}{}$,
$\subspace{10403\\01343}{}$,
$\subspace{10410\\01303}{}$,
$\subspace{10422\\01313}{}$,
$\subspace{10434\\01323}{}$,
$\subspace{10441\\01333}{}$,
$\subspace{10202\\01210}{15}$,
$\subspace{10214\\01204}{}$,
$\subspace{10221\\01243}{}$,
$\subspace{10233\\01232}{}$,
$\subspace{10240\\01221}{}$,
$\subspace{10202\\01432}{}$,
$\subspace{10214\\01433}{}$,
$\subspace{10221\\01434}{}$,
$\subspace{10233\\01430}{}$,
$\subspace{10240\\01431}{}$,
$\subspace{10401\\01312}{}$,
$\subspace{10413\\01322}{}$,
$\subspace{10420\\01332}{}$,
$\subspace{10432\\01342}{}$,
$\subspace{10444\\01302}{}$,
$\subspace{10200\\01224}{15}$,
$\subspace{10212\\01213}{}$,
$\subspace{10224\\01202}{}$,
$\subspace{10231\\01241}{}$,
$\subspace{10243\\01230}{}$,
$\subspace{10200\\01414}{}$,
$\subspace{10212\\01410}{}$,
$\subspace{10224\\01411}{}$,
$\subspace{10231\\01412}{}$,
$\subspace{10243\\01413}{}$,
$\subspace{10404\\01321}{}$,
$\subspace{10411\\01331}{}$,
$\subspace{10423\\01341}{}$,
$\subspace{10430\\01301}{}$,
$\subspace{10442\\01311}{}$,
$\subspace{10203\\01224}{15}$,
$\subspace{10210\\01213}{}$,
$\subspace{10222\\01202}{}$,
$\subspace{10234\\01241}{}$,
$\subspace{10241\\01230}{}$,
$\subspace{10204\\01411}{}$,
$\subspace{10211\\01412}{}$,
$\subspace{10223\\01413}{}$,
$\subspace{10230\\01414}{}$,
$\subspace{10242\\01410}{}$,
$\subspace{10402\\01301}{}$,
$\subspace{10414\\01311}{}$,
$\subspace{10421\\01321}{}$,
$\subspace{10433\\01331}{}$,
$\subspace{10440\\01341}{}$,
$\subspace{10201\\01240}{15}$,
$\subspace{10213\\01234}{}$,
$\subspace{10220\\01223}{}$,
$\subspace{10232\\01212}{}$,
$\subspace{10244\\01201}{}$,
$\subspace{10203\\01404}{}$,
$\subspace{10210\\01400}{}$,
$\subspace{10222\\01401}{}$,
$\subspace{10234\\01402}{}$,
$\subspace{10241\\01403}{}$,
$\subspace{10402\\01303}{}$,
$\subspace{10414\\01313}{}$,
$\subspace{10421\\01323}{}$,
$\subspace{10433\\01333}{}$,
$\subspace{10440\\01343}{}$,
$\subspace{10201\\01242}{15}$,
$\subspace{10213\\01231}{}$,
$\subspace{10220\\01220}{}$,
$\subspace{10232\\01214}{}$,
$\subspace{10244\\01203}{}$,
$\subspace{10201\\01423}{}$,
$\subspace{10213\\01424}{}$,
$\subspace{10220\\01420}{}$,
$\subspace{10232\\01421}{}$,
$\subspace{10244\\01422}{}$,
$\subspace{10400\\01344}{}$,
$\subspace{10412\\01304}{}$,
$\subspace{10424\\01314}{}$,
$\subspace{10431\\01324}{}$,
$\subspace{10443\\01334}{}$,
$\subspace{10203\\01244}{15}$,
$\subspace{10210\\01233}{}$,
$\subspace{10222\\01222}{}$,
$\subspace{10234\\01211}{}$,
$\subspace{10241\\01200}{}$,
$\subspace{10200\\01440}{}$,
$\subspace{10212\\01441}{}$,
$\subspace{10224\\01442}{}$,
$\subspace{10231\\01443}{}$,
$\subspace{10243\\01444}{}$,
$\subspace{10400\\01300}{}$,
$\subspace{10412\\01310}{}$,
$\subspace{10424\\01320}{}$,
$\subspace{10431\\01330}{}$,
$\subspace{10443\\01340}{}$,
$\subspace{10204\\01241}{15}$,
$\subspace{10211\\01230}{}$,
$\subspace{10223\\01224}{}$,
$\subspace{10230\\01213}{}$,
$\subspace{10242\\01202}{}$,
$\subspace{10201\\01413}{}$,
$\subspace{10213\\01414}{}$,
$\subspace{10220\\01410}{}$,
$\subspace{10232\\01411}{}$,
$\subspace{10244\\01412}{}$,
$\subspace{10404\\01301}{}$,
$\subspace{10411\\01311}{}$,
$\subspace{10423\\01321}{}$,
$\subspace{10430\\01331}{}$,
$\subspace{10442\\01341}{}$,
$\subspace{10303\\01303}{15}$,
$\subspace{10310\\01333}{}$,
$\subspace{10322\\01313}{}$,
$\subspace{10334\\01343}{}$,
$\subspace{10341\\01323}{}$,
$\subspace{10404\\01212}{}$,
$\subspace{10411\\01234}{}$,
$\subspace{10423\\01201}{}$,
$\subspace{10430\\01223}{}$,
$\subspace{10442\\01240}{}$,
$\subspace{10400\\01401}{}$,
$\subspace{10412\\01404}{}$,
$\subspace{10424\\01402}{}$,
$\subspace{10431\\01400}{}$,
$\subspace{10443\\01403}{}$,
$\subspace{10301\\01314}{15}$,
$\subspace{10313\\01344}{}$,
$\subspace{10320\\01324}{}$,
$\subspace{10332\\01304}{}$,
$\subspace{10344\\01334}{}$,
$\subspace{10402\\01203}{}$,
$\subspace{10414\\01220}{}$,
$\subspace{10421\\01242}{}$,
$\subspace{10433\\01214}{}$,
$\subspace{10440\\01231}{}$,
$\subspace{10401\\01422}{}$,
$\subspace{10413\\01420}{}$,
$\subspace{10420\\01423}{}$,
$\subspace{10432\\01421}{}$,
$\subspace{10444\\01424}{}$,
$\subspace{10300\\01322}{15}$,
$\subspace{10312\\01302}{}$,
$\subspace{10324\\01332}{}$,
$\subspace{10331\\01312}{}$,
$\subspace{10343\\01342}{}$,
$\subspace{10404\\01243}{}$,
$\subspace{10411\\01210}{}$,
$\subspace{10423\\01232}{}$,
$\subspace{10430\\01204}{}$,
$\subspace{10442\\01221}{}$,
$\subspace{10401\\01434}{}$,
$\subspace{10413\\01432}{}$,
$\subspace{10420\\01430}{}$,
$\subspace{10432\\01433}{}$,
$\subspace{10444\\01431}{}$,
$\subspace{10301\\01330}{15}$,
$\subspace{10313\\01310}{}$,
$\subspace{10320\\01340}{}$,
$\subspace{10332\\01320}{}$,
$\subspace{10344\\01300}{}$,
$\subspace{10402\\01244}{}$,
$\subspace{10414\\01211}{}$,
$\subspace{10421\\01233}{}$,
$\subspace{10433\\01200}{}$,
$\subspace{10440\\01222}{}$,
$\subspace{10403\\01442}{}$,
$\subspace{10410\\01440}{}$,
$\subspace{10422\\01443}{}$,
$\subspace{10434\\01441}{}$,
$\subspace{10441\\01444}{}$,
$\subspace{10301\\01330}{15}$,
$\subspace{10313\\01310}{}$,
$\subspace{10320\\01340}{}$,
$\subspace{10332\\01320}{}$,
$\subspace{10344\\01300}{}$,
$\subspace{10402\\01244}{}$,
$\subspace{10414\\01211}{}$,
$\subspace{10421\\01233}{}$,
$\subspace{10433\\01200}{}$,
$\subspace{10440\\01222}{}$,
$\subspace{10403\\01442}{}$,
$\subspace{10410\\01440}{}$,
$\subspace{10422\\01443}{}$,
$\subspace{10434\\01441}{}$,
$\subspace{10441\\01444}{}$,
$\subspace{10301\\01331}{15}$,
$\subspace{10313\\01311}{}$,
$\subspace{10320\\01341}{}$,
$\subspace{10332\\01321}{}$,
$\subspace{10344\\01301}{}$,
$\subspace{10400\\01202}{}$,
$\subspace{10412\\01224}{}$,
$\subspace{10424\\01241}{}$,
$\subspace{10431\\01213}{}$,
$\subspace{10443\\01230}{}$,
$\subspace{10402\\01411}{}$,
$\subspace{10414\\01414}{}$,
$\subspace{10421\\01412}{}$,
$\subspace{10433\\01410}{}$,
$\subspace{10440\\01413}{}$,
$\subspace{10301\\01331}{15}$,
$\subspace{10313\\01311}{}$,
$\subspace{10320\\01341}{}$,
$\subspace{10332\\01321}{}$,
$\subspace{10344\\01301}{}$,
$\subspace{10400\\01202}{}$,
$\subspace{10412\\01224}{}$,
$\subspace{10424\\01241}{}$,
$\subspace{10431\\01213}{}$,
$\subspace{10443\\01230}{}$,
$\subspace{10402\\01411}{}$,
$\subspace{10414\\01414}{}$,
$\subspace{10421\\01412}{}$,
$\subspace{10433\\01410}{}$,
$\subspace{10440\\01413}{}$,
$\subspace{10303\\01330}{15}$,
$\subspace{10310\\01310}{}$,
$\subspace{10322\\01340}{}$,
$\subspace{10334\\01320}{}$,
$\subspace{10341\\01300}{}$,
$\subspace{10403\\01200}{}$,
$\subspace{10410\\01222}{}$,
$\subspace{10422\\01244}{}$,
$\subspace{10434\\01211}{}$,
$\subspace{10441\\01233}{}$,
$\subspace{10400\\01443}{}$,
$\subspace{10412\\01441}{}$,
$\subspace{10424\\01444}{}$,
$\subspace{10431\\01442}{}$,
$\subspace{10443\\01440}{}$.

\smallskip

$n_5(5,2;48)\ge 1203$,$\left[\!\begin{smallmatrix}10000\\00100\\04410\\00001\\00044\end{smallmatrix}\!\right]$:
$\subspace{00101\\00012}{3}$,
$\subspace{01001\\00014}{}$,
$\subspace{01101\\00013}{}$,
$\subspace{01000\\00100}{15}$,
$\subspace{01000\\00111}{}$,
$\subspace{01004\\00112}{}$,
$\subspace{01014\\00112}{}$,
$\subspace{01013\\00124}{}$,
$\subspace{01014\\00123}{}$,
$\subspace{01023\\00124}{}$,
$\subspace{01022\\00131}{}$,
$\subspace{01023\\00130}{}$,
$\subspace{01032\\00131}{}$,
$\subspace{01031\\00143}{}$,
$\subspace{01032\\00142}{}$,
$\subspace{01040\\00100}{}$,
$\subspace{01041\\00104}{}$,
$\subspace{01041\\00143}{}$,
$\subspace{01001\\00101}{15}$,
$\subspace{01001\\00134}{}$,
$\subspace{01003\\00132}{}$,
$\subspace{01010\\00113}{}$,
$\subspace{01010\\00141}{}$,
$\subspace{01012\\00144}{}$,
$\subspace{01021\\00101}{}$,
$\subspace{01024\\00103}{}$,
$\subspace{01024\\00120}{}$,
$\subspace{01030\\00113}{}$,
$\subspace{01033\\00110}{}$,
$\subspace{01033\\00132}{}$,
$\subspace{01042\\00122}{}$,
$\subspace{01044\\00120}{}$,
$\subspace{01042\\00144}{}$,
$\subspace{01003\\00104}{15}$,
$\subspace{01004\\00103}{}$,
$\subspace{01003\\00110}{}$,
$\subspace{01012\\00111}{}$,
$\subspace{01013\\00110}{}$,
$\subspace{01012\\00122}{}$,
$\subspace{01021\\00123}{}$,
$\subspace{01022\\00122}{}$,
$\subspace{01021\\00134}{}$,
$\subspace{01030\\00130}{}$,
$\subspace{01031\\00134}{}$,
$\subspace{01030\\00141}{}$,
$\subspace{01044\\00103}{}$,
$\subspace{01040\\00141}{}$,
$\subspace{01044\\00142}{}$,
$\subspace{10000\\00001}{3}$,
$\subspace{10000\\00010}{}$,
$\subspace{10000\\00011}{}$,
$\subspace{10000\\00012}{3}$,
$\subspace{10000\\00013}{}$,
$\subspace{10000\\00014}{}$,
$\subspace{10000\\00012}{3}$,
$\subspace{10000\\00013}{}$,
$\subspace{10000\\00014}{}$,
$\subspace{10001\\00011}{3}$,
$\subspace{10004\\00010}{}$,
$\subspace{10010\\00001}{}$,
$\subspace{10001\\00013}{3}$,
$\subspace{10003\\00012}{}$,
$\subspace{10003\\00014}{}$,
$\subspace{10002\\00012}{3}$,
$\subspace{10002\\00014}{}$,
$\subspace{10004\\00013}{}$,
$\subspace{10023\\00101}{15}$,
$\subspace{10023\\00113}{}$,
$\subspace{10023\\00120}{}$,
$\subspace{10023\\00132}{}$,
$\subspace{10023\\00144}{}$,
$\subspace{10013\\01001}{}$,
$\subspace{10013\\01010}{}$,
$\subspace{10013\\01024}{}$,
$\subspace{10013\\01033}{}$,
$\subspace{10013\\01042}{}$,
$\subspace{10024\\01101}{}$,
$\subspace{10024\\01114}{}$,
$\subspace{10024\\01122}{}$,
$\subspace{10024\\01130}{}$,
$\subspace{10024\\01143}{}$,
$\subspace{10043\\00102}{15}$,
$\subspace{10043\\00114}{}$,
$\subspace{10043\\00121}{}$,
$\subspace{10043\\00133}{}$,
$\subspace{10043\\00140}{}$,
$\subspace{10041\\01002}{}$,
$\subspace{10041\\01011}{}$,
$\subspace{10041\\01020}{}$,
$\subspace{10041\\01034}{}$,
$\subspace{10041\\01043}{}$,
$\subspace{10021\\01104}{}$,
$\subspace{10021\\01112}{}$,
$\subspace{10021\\01120}{}$,
$\subspace{10021\\01133}{}$,
$\subspace{10021\\01141}{}$,
$\subspace{10104\\00012}{3}$,
$\subspace{11004\\00014}{}$,
$\subspace{14400\\00013}{}$,
$\subspace{10203\\00012}{3}$,
$\subspace{12003\\00014}{}$,
$\subspace{13300\\00013}{}$,
$\subspace{10303\\00012}{3}$,
$\subspace{12202\\00013}{}$,
$\subspace{13003\\00014}{}$,
$\subspace{10401\\00012}{3}$,
$\subspace{11100\\00013}{}$,
$\subspace{14001\\00014}{}$,
$\subspace{10002\\01204}{15}$,
$\subspace{10002\\01210}{}$,
$\subspace{10002\\01221}{}$,
$\subspace{10002\\01232}{}$,
$\subspace{10002\\01243}{}$,
$\subspace{10033\\01302}{}$,
$\subspace{10033\\01312}{}$,
$\subspace{10033\\01322}{}$,
$\subspace{10033\\01332}{}$,
$\subspace{10033\\01342}{}$,
$\subspace{10020\\01430}{}$,
$\subspace{10020\\01431}{}$,
$\subspace{10020\\01432}{}$,
$\subspace{10020\\01433}{}$,
$\subspace{10020\\01434}{}$,
$\subspace{10003\\01204}{15}$,
$\subspace{10003\\01210}{}$,
$\subspace{10003\\01221}{}$,
$\subspace{10003\\01232}{}$,
$\subspace{10003\\01243}{}$,
$\subspace{10022\\01302}{}$,
$\subspace{10022\\01312}{}$,
$\subspace{10022\\01322}{}$,
$\subspace{10022\\01332}{}$,
$\subspace{10022\\01342}{}$,
$\subspace{10030\\01430}{}$,
$\subspace{10030\\01431}{}$,
$\subspace{10030\\01432}{}$,
$\subspace{10030\\01433}{}$,
$\subspace{10030\\01434}{}$,
$\subspace{10011\\01203}{15}$,
$\subspace{10011\\01214}{}$,
$\subspace{10011\\01220}{}$,
$\subspace{10011\\01231}{}$,
$\subspace{10011\\01242}{}$,
$\subspace{10040\\01304}{}$,
$\subspace{10040\\01314}{}$,
$\subspace{10040\\01324}{}$,
$\subspace{10040\\01334}{}$,
$\subspace{10040\\01344}{}$,
$\subspace{10004\\01420}{}$,
$\subspace{10004\\01421}{}$,
$\subspace{10004\\01422}{}$,
$\subspace{10004\\01423}{}$,
$\subspace{10004\\01424}{}$,
$\subspace{10042\\01203}{15}$,
$\subspace{10042\\01214}{}$,
$\subspace{10042\\01220}{}$,
$\subspace{10042\\01231}{}$,
$\subspace{10042\\01242}{}$,
$\subspace{10032\\01304}{}$,
$\subspace{10032\\01314}{}$,
$\subspace{10032\\01324}{}$,
$\subspace{10032\\01334}{}$,
$\subspace{10032\\01344}{}$,
$\subspace{10031\\01420}{}$,
$\subspace{10031\\01421}{}$,
$\subspace{10031\\01422}{}$,
$\subspace{10031\\01423}{}$,
$\subspace{10031\\01424}{}$,
$\subspace{10044\\01201}{15}$,
$\subspace{10044\\01212}{}$,
$\subspace{10044\\01223}{}$,
$\subspace{10044\\01234}{}$,
$\subspace{10044\\01240}{}$,
$\subspace{10010\\01303}{}$,
$\subspace{10010\\01313}{}$,
$\subspace{10010\\01323}{}$,
$\subspace{10010\\01333}{}$,
$\subspace{10010\\01343}{}$,
$\subspace{10001\\01400}{}$,
$\subspace{10001\\01401}{}$,
$\subspace{10001\\01402}{}$,
$\subspace{10001\\01403}{}$,
$\subspace{10001\\01404}{}$,
$\subspace{10101\\01003}{15}$,
$\subspace{10113\\01012}{}$,
$\subspace{10120\\01021}{}$,
$\subspace{10132\\01030}{}$,
$\subspace{10144\\01044}{}$,
$\subspace{10403\\01144}{}$,
$\subspace{10410\\01123}{}$,
$\subspace{10422\\01102}{}$,
$\subspace{10434\\01131}{}$,
$\subspace{10441\\01110}{}$,
$\subspace{14003\\00141}{}$,
$\subspace{14012\\00134}{}$,
$\subspace{14021\\00122}{}$,
$\subspace{14030\\00110}{}$,
$\subspace{14044\\00103}{}$,
$\subspace{10102\\01011}{15}$,
$\subspace{10114\\01020}{}$,
$\subspace{10121\\01034}{}$,
$\subspace{10133\\01043}{}$,
$\subspace{10140\\01002}{}$,
$\subspace{10401\\01133}{}$,
$\subspace{10413\\01112}{}$,
$\subspace{10420\\01141}{}$,
$\subspace{10432\\01120}{}$,
$\subspace{10444\\01104}{}$,
$\subspace{14001\\00121}{}$,
$\subspace{14010\\00114}{}$,
$\subspace{14024\\00102}{}$,
$\subspace{14033\\00140}{}$,
$\subspace{14042\\00133}{}$,
$\subspace{10100\\01024}{15}$,
$\subspace{10112\\01033}{}$,
$\subspace{10124\\01042}{}$,
$\subspace{10131\\01001}{}$,
$\subspace{10143\\01010}{}$,
$\subspace{10400\\01114}{}$,
$\subspace{10412\\01143}{}$,
$\subspace{10424\\01122}{}$,
$\subspace{10431\\01101}{}$,
$\subspace{10443\\01130}{}$,
$\subspace{14000\\00132}{}$,
$\subspace{14014\\00120}{}$,
$\subspace{14023\\00113}{}$,
$\subspace{14032\\00101}{}$,
$\subspace{14041\\00144}{}$,
$\subspace{10100\\01024}{15}$,
$\subspace{10112\\01033}{}$,
$\subspace{10124\\01042}{}$,
$\subspace{10131\\01001}{}$,
$\subspace{10143\\01010}{}$,
$\subspace{10400\\01114}{}$,
$\subspace{10412\\01143}{}$,
$\subspace{10424\\01122}{}$,
$\subspace{10431\\01101}{}$,
$\subspace{10443\\01130}{}$,
$\subspace{14000\\00132}{}$,
$\subspace{14014\\00120}{}$,
$\subspace{14023\\00113}{}$,
$\subspace{14032\\00101}{}$,
$\subspace{14041\\00144}{}$,
$\subspace{10103\\01023}{15}$,
$\subspace{10110\\01032}{}$,
$\subspace{10122\\01041}{}$,
$\subspace{10134\\01000}{}$,
$\subspace{10141\\01014}{}$,
$\subspace{10401\\01111}{}$,
$\subspace{10413\\01140}{}$,
$\subspace{10420\\01124}{}$,
$\subspace{10432\\01103}{}$,
$\subspace{10444\\01132}{}$,
$\subspace{14001\\00124}{}$,
$\subspace{14010\\00112}{}$,
$\subspace{14024\\00100}{}$,
$\subspace{14033\\00143}{}$,
$\subspace{14042\\00131}{}$,
$\subspace{10100\\01032}{15}$,
$\subspace{10112\\01041}{}$,
$\subspace{10124\\01000}{}$,
$\subspace{10131\\01014}{}$,
$\subspace{10143\\01023}{}$,
$\subspace{10400\\01132}{}$,
$\subspace{10412\\01111}{}$,
$\subspace{10424\\01140}{}$,
$\subspace{10431\\01124}{}$,
$\subspace{10443\\01103}{}$,
$\subspace{14000\\00124}{}$,
$\subspace{14014\\00112}{}$,
$\subspace{14023\\00100}{}$,
$\subspace{14032\\00143}{}$,
$\subspace{14041\\00131}{}$,
$\subspace{10104\\01030}{15}$,
$\subspace{10111\\01044}{}$,
$\subspace{10123\\01003}{}$,
$\subspace{10130\\01012}{}$,
$\subspace{10142\\01021}{}$,
$\subspace{10401\\01144}{}$,
$\subspace{10413\\01123}{}$,
$\subspace{10420\\01102}{}$,
$\subspace{10432\\01131}{}$,
$\subspace{10444\\01110}{}$,
$\subspace{14001\\00122}{}$,
$\subspace{14010\\00110}{}$,
$\subspace{14024\\00103}{}$,
$\subspace{14033\\00141}{}$,
$\subspace{14042\\00134}{}$,
$\subspace{10102\\01044}{15}$,
$\subspace{10114\\01003}{}$,
$\subspace{10121\\01012}{}$,
$\subspace{10133\\01021}{}$,
$\subspace{10140\\01030}{}$,
$\subspace{10403\\01131}{}$,
$\subspace{10410\\01110}{}$,
$\subspace{10422\\01144}{}$,
$\subspace{10434\\01123}{}$,
$\subspace{10441\\01102}{}$,
$\subspace{14003\\00110}{}$,
$\subspace{14012\\00103}{}$,
$\subspace{14021\\00141}{}$,
$\subspace{14030\\00134}{}$,
$\subspace{14044\\00122}{}$,
$\subspace{10102\\01044}{15}$,
$\subspace{10114\\01003}{}$,
$\subspace{10121\\01012}{}$,
$\subspace{10133\\01021}{}$,
$\subspace{10140\\01030}{}$,
$\subspace{10403\\01131}{}$,
$\subspace{10410\\01110}{}$,
$\subspace{10422\\01144}{}$,
$\subspace{10434\\01123}{}$,
$\subspace{10441\\01102}{}$,
$\subspace{14003\\00110}{}$,
$\subspace{14012\\00103}{}$,
$\subspace{14021\\00141}{}$,
$\subspace{14030\\00134}{}$,
$\subspace{14044\\00122}{}$,
$\subspace{10104\\01112}{15}$,
$\subspace{10111\\01133}{}$,
$\subspace{10123\\01104}{}$,
$\subspace{10130\\01120}{}$,
$\subspace{10142\\01141}{}$,
$\subspace{10401\\01043}{}$,
$\subspace{10413\\01034}{}$,
$\subspace{10420\\01020}{}$,
$\subspace{10432\\01011}{}$,
$\subspace{10444\\01002}{}$,
$\subspace{11004\\00114}{}$,
$\subspace{11013\\00121}{}$,
$\subspace{11022\\00133}{}$,
$\subspace{11031\\00140}{}$,
$\subspace{11040\\00102}{}$,
$\subspace{10104\\01113}{15}$,
$\subspace{10111\\01134}{}$,
$\subspace{10123\\01100}{}$,
$\subspace{10130\\01121}{}$,
$\subspace{10142\\01142}{}$,
$\subspace{10400\\01013}{}$,
$\subspace{10412\\01004}{}$,
$\subspace{10424\\01040}{}$,
$\subspace{10431\\01031}{}$,
$\subspace{10443\\01022}{}$,
$\subspace{11004\\00104}{}$,
$\subspace{11013\\00111}{}$,
$\subspace{11022\\00123}{}$,
$\subspace{11031\\00130}{}$,
$\subspace{11040\\00142}{}$,
$\subspace{10104\\01121}{15}$,
$\subspace{10111\\01142}{}$,
$\subspace{10123\\01113}{}$,
$\subspace{10130\\01134}{}$,
$\subspace{10142\\01100}{}$,
$\subspace{10401\\01022}{}$,
$\subspace{10413\\01013}{}$,
$\subspace{10420\\01004}{}$,
$\subspace{10432\\01040}{}$,
$\subspace{10444\\01031}{}$,
$\subspace{11004\\00130}{}$,
$\subspace{11013\\00142}{}$,
$\subspace{11022\\00104}{}$,
$\subspace{11031\\00111}{}$,
$\subspace{11040\\00123}{}$,
$\subspace{10102\\01131}{15}$,
$\subspace{10114\\01102}{}$,
$\subspace{10121\\01123}{}$,
$\subspace{10133\\01144}{}$,
$\subspace{10140\\01110}{}$,
$\subspace{10403\\01044}{}$,
$\subspace{10410\\01030}{}$,
$\subspace{10422\\01021}{}$,
$\subspace{10434\\01012}{}$,
$\subspace{10441\\01003}{}$,
$\subspace{11002\\00110}{}$,
$\subspace{11011\\00122}{}$,
$\subspace{11020\\00134}{}$,
$\subspace{11034\\00141}{}$,
$\subspace{11043\\00103}{}$,
$\subspace{10102\\01131}{15}$,
$\subspace{10114\\01102}{}$,
$\subspace{10121\\01123}{}$,
$\subspace{10133\\01144}{}$,
$\subspace{10140\\01110}{}$,
$\subspace{10403\\01044}{}$,
$\subspace{10410\\01030}{}$,
$\subspace{10422\\01021}{}$,
$\subspace{10434\\01012}{}$,
$\subspace{10441\\01003}{}$,
$\subspace{11002\\00110}{}$,
$\subspace{11011\\00122}{}$,
$\subspace{11020\\00134}{}$,
$\subspace{11034\\00141}{}$,
$\subspace{11043\\00103}{}$,
$\subspace{10102\\01132}{15}$,
$\subspace{10114\\01103}{}$,
$\subspace{10121\\01124}{}$,
$\subspace{10133\\01140}{}$,
$\subspace{10140\\01111}{}$,
$\subspace{10402\\01014}{}$,
$\subspace{10414\\01000}{}$,
$\subspace{10421\\01041}{}$,
$\subspace{10433\\01032}{}$,
$\subspace{10440\\01023}{}$,
$\subspace{11002\\00100}{}$,
$\subspace{11011\\00112}{}$,
$\subspace{11020\\00124}{}$,
$\subspace{11034\\00131}{}$,
$\subspace{11043\\00143}{}$,
$\subspace{10102\\01142}{15}$,
$\subspace{10114\\01113}{}$,
$\subspace{10121\\01134}{}$,
$\subspace{10133\\01100}{}$,
$\subspace{10140\\01121}{}$,
$\subspace{10401\\01013}{}$,
$\subspace{10413\\01004}{}$,
$\subspace{10420\\01040}{}$,
$\subspace{10432\\01031}{}$,
$\subspace{10444\\01022}{}$,
$\subspace{11002\\00111}{}$,
$\subspace{11011\\00123}{}$,
$\subspace{11020\\00130}{}$,
$\subspace{11034\\00142}{}$,
$\subspace{11043\\00104}{}$,
$\subspace{10104\\01213}{15}$,
$\subspace{10111\\01241}{}$,
$\subspace{10123\\01224}{}$,
$\subspace{10130\\01202}{}$,
$\subspace{10142\\01230}{}$,
$\subspace{10101\\01412}{}$,
$\subspace{10113\\01414}{}$,
$\subspace{10120\\01411}{}$,
$\subspace{10132\\01413}{}$,
$\subspace{10144\\01410}{}$,
$\subspace{10202\\01331}{}$,
$\subspace{10214\\01301}{}$,
$\subspace{10221\\01321}{}$,
$\subspace{10233\\01341}{}$,
$\subspace{10240\\01311}{}$,
$\subspace{10104\\01213}{15}$,
$\subspace{10111\\01241}{}$,
$\subspace{10123\\01224}{}$,
$\subspace{10130\\01202}{}$,
$\subspace{10142\\01230}{}$,
$\subspace{10101\\01412}{}$,
$\subspace{10113\\01414}{}$,
$\subspace{10120\\01411}{}$,
$\subspace{10132\\01413}{}$,
$\subspace{10144\\01410}{}$,
$\subspace{10202\\01331}{}$,
$\subspace{10214\\01301}{}$,
$\subspace{10221\\01321}{}$,
$\subspace{10233\\01341}{}$,
$\subspace{10240\\01311}{}$,
$\subspace{10102\\01223}{15}$,
$\subspace{10114\\01201}{}$,
$\subspace{10121\\01234}{}$,
$\subspace{10133\\01212}{}$,
$\subspace{10140\\01240}{}$,
$\subspace{10104\\01403}{}$,
$\subspace{10111\\01400}{}$,
$\subspace{10123\\01402}{}$,
$\subspace{10130\\01404}{}$,
$\subspace{10142\\01401}{}$,
$\subspace{10202\\01343}{}$,
$\subspace{10214\\01313}{}$,
$\subspace{10221\\01333}{}$,
$\subspace{10233\\01303}{}$,
$\subspace{10240\\01323}{}$,
$\subspace{10103\\01221}{15}$,
$\subspace{10110\\01204}{}$,
$\subspace{10122\\01232}{}$,
$\subspace{10134\\01210}{}$,
$\subspace{10141\\01243}{}$,
$\subspace{10103\\01432}{}$,
$\subspace{10110\\01434}{}$,
$\subspace{10122\\01431}{}$,
$\subspace{10134\\01433}{}$,
$\subspace{10141\\01430}{}$,
$\subspace{10204\\01322}{}$,
$\subspace{10211\\01342}{}$,
$\subspace{10223\\01312}{}$,
$\subspace{10230\\01332}{}$,
$\subspace{10242\\01302}{}$,
$\subspace{10104\\01221}{15}$,
$\subspace{10111\\01204}{}$,
$\subspace{10123\\01232}{}$,
$\subspace{10130\\01210}{}$,
$\subspace{10142\\01243}{}$,
$\subspace{10101\\01430}{}$,
$\subspace{10113\\01432}{}$,
$\subspace{10120\\01434}{}$,
$\subspace{10132\\01431}{}$,
$\subspace{10144\\01433}{}$,
$\subspace{10200\\01342}{}$,
$\subspace{10212\\01312}{}$,
$\subspace{10224\\01332}{}$,
$\subspace{10231\\01302}{}$,
$\subspace{10243\\01322}{}$,
$\subspace{10101\\01231}{15}$,
$\subspace{10113\\01214}{}$,
$\subspace{10120\\01242}{}$,
$\subspace{10132\\01220}{}$,
$\subspace{10144\\01203}{}$,
$\subspace{10101\\01423}{}$,
$\subspace{10113\\01420}{}$,
$\subspace{10120\\01422}{}$,
$\subspace{10132\\01424}{}$,
$\subspace{10144\\01421}{}$,
$\subspace{10204\\01334}{}$,
$\subspace{10211\\01304}{}$,
$\subspace{10223\\01324}{}$,
$\subspace{10230\\01344}{}$,
$\subspace{10242\\01314}{}$,
$\subspace{10103\\01230}{15}$,
$\subspace{10110\\01213}{}$,
$\subspace{10122\\01241}{}$,
$\subspace{10134\\01224}{}$,
$\subspace{10141\\01202}{}$,
$\subspace{10100\\01412}{}$,
$\subspace{10112\\01414}{}$,
$\subspace{10124\\01411}{}$,
$\subspace{10131\\01413}{}$,
$\subspace{10143\\01410}{}$,
$\subspace{10204\\01301}{}$,
$\subspace{10211\\01321}{}$,
$\subspace{10223\\01341}{}$,
$\subspace{10230\\01311}{}$,
$\subspace{10242\\01331}{}$,
$\subspace{10103\\01230}{15}$,
$\subspace{10110\\01213}{}$,
$\subspace{10122\\01241}{}$,
$\subspace{10134\\01224}{}$,
$\subspace{10141\\01202}{}$,
$\subspace{10100\\01412}{}$,
$\subspace{10112\\01414}{}$,
$\subspace{10124\\01411}{}$,
$\subspace{10131\\01413}{}$,
$\subspace{10143\\01410}{}$,
$\subspace{10204\\01301}{}$,
$\subspace{10211\\01321}{}$,
$\subspace{10223\\01341}{}$,
$\subspace{10230\\01311}{}$,
$\subspace{10242\\01331}{}$,
$\subspace{10103\\01230}{15}$,
$\subspace{10110\\01213}{}$,
$\subspace{10122\\01241}{}$,
$\subspace{10134\\01224}{}$,
$\subspace{10141\\01202}{}$,
$\subspace{10100\\01412}{}$,
$\subspace{10112\\01414}{}$,
$\subspace{10124\\01411}{}$,
$\subspace{10131\\01413}{}$,
$\subspace{10143\\01410}{}$,
$\subspace{10204\\01301}{}$,
$\subspace{10211\\01321}{}$,
$\subspace{10223\\01341}{}$,
$\subspace{10230\\01311}{}$,
$\subspace{10242\\01331}{}$,
$\subspace{10101\\01301}{15}$,
$\subspace{10113\\01341}{}$,
$\subspace{10120\\01331}{}$,
$\subspace{10132\\01321}{}$,
$\subspace{10144\\01311}{}$,
$\subspace{10301\\01241}{}$,
$\subspace{10313\\01202}{}$,
$\subspace{10320\\01213}{}$,
$\subspace{10332\\01224}{}$,
$\subspace{10344\\01230}{}$,
$\subspace{10304\\01413}{}$,
$\subspace{10311\\01412}{}$,
$\subspace{10323\\01411}{}$,
$\subspace{10330\\01410}{}$,
$\subspace{10342\\01414}{}$,
$\subspace{10101\\01301}{15}$,
$\subspace{10113\\01341}{}$,
$\subspace{10120\\01331}{}$,
$\subspace{10132\\01321}{}$,
$\subspace{10144\\01311}{}$,
$\subspace{10301\\01241}{}$,
$\subspace{10313\\01202}{}$,
$\subspace{10320\\01213}{}$,
$\subspace{10332\\01224}{}$,
$\subspace{10344\\01230}{}$,
$\subspace{10304\\01413}{}$,
$\subspace{10311\\01412}{}$,
$\subspace{10323\\01411}{}$,
$\subspace{10330\\01410}{}$,
$\subspace{10342\\01414}{}$,
$\subspace{10103\\01311}{15}$,
$\subspace{10110\\01301}{}$,
$\subspace{10122\\01341}{}$,
$\subspace{10134\\01331}{}$,
$\subspace{10141\\01321}{}$,
$\subspace{10304\\01213}{}$,
$\subspace{10311\\01224}{}$,
$\subspace{10323\\01230}{}$,
$\subspace{10330\\01241}{}$,
$\subspace{10342\\01202}{}$,
$\subspace{10304\\01414}{}$,
$\subspace{10311\\01413}{}$,
$\subspace{10323\\01412}{}$,
$\subspace{10330\\01411}{}$,
$\subspace{10342\\01410}{}$,
$\subspace{10100\\01320}{15}$,
$\subspace{10112\\01310}{}$,
$\subspace{10124\\01300}{}$,
$\subspace{10131\\01340}{}$,
$\subspace{10143\\01330}{}$,
$\subspace{10301\\01222}{}$,
$\subspace{10313\\01233}{}$,
$\subspace{10320\\01244}{}$,
$\subspace{10332\\01200}{}$,
$\subspace{10344\\01211}{}$,
$\subspace{10301\\01441}{}$,
$\subspace{10313\\01440}{}$,
$\subspace{10320\\01444}{}$,
$\subspace{10332\\01443}{}$,
$\subspace{10344\\01442}{}$,
$\subspace{10100\\01320}{15}$,
$\subspace{10112\\01310}{}$,
$\subspace{10124\\01300}{}$,
$\subspace{10131\\01340}{}$,
$\subspace{10143\\01330}{}$,
$\subspace{10301\\01222}{}$,
$\subspace{10313\\01233}{}$,
$\subspace{10320\\01244}{}$,
$\subspace{10332\\01200}{}$,
$\subspace{10344\\01211}{}$,
$\subspace{10301\\01441}{}$,
$\subspace{10313\\01440}{}$,
$\subspace{10320\\01444}{}$,
$\subspace{10332\\01443}{}$,
$\subspace{10344\\01442}{}$,
$\subspace{10103\\01322}{15}$,
$\subspace{10110\\01312}{}$,
$\subspace{10122\\01302}{}$,
$\subspace{10134\\01342}{}$,
$\subspace{10141\\01332}{}$,
$\subspace{10302\\01210}{}$,
$\subspace{10314\\01221}{}$,
$\subspace{10321\\01232}{}$,
$\subspace{10333\\01243}{}$,
$\subspace{10340\\01204}{}$,
$\subspace{10300\\01431}{}$,
$\subspace{10312\\01430}{}$,
$\subspace{10324\\01434}{}$,
$\subspace{10331\\01433}{}$,
$\subspace{10343\\01432}{}$,
$\subspace{10103\\01322}{15}$,
$\subspace{10110\\01312}{}$,
$\subspace{10122\\01302}{}$,
$\subspace{10134\\01342}{}$,
$\subspace{10141\\01332}{}$,
$\subspace{10302\\01210}{}$,
$\subspace{10314\\01221}{}$,
$\subspace{10321\\01232}{}$,
$\subspace{10333\\01243}{}$,
$\subspace{10340\\01204}{}$,
$\subspace{10300\\01431}{}$,
$\subspace{10312\\01430}{}$,
$\subspace{10324\\01434}{}$,
$\subspace{10331\\01433}{}$,
$\subspace{10343\\01432}{}$,
$\subspace{10101\\01333}{15}$,
$\subspace{10113\\01323}{}$,
$\subspace{10120\\01313}{}$,
$\subspace{10132\\01303}{}$,
$\subspace{10144\\01343}{}$,
$\subspace{10304\\01234}{}$,
$\subspace{10311\\01240}{}$,
$\subspace{10323\\01201}{}$,
$\subspace{10330\\01212}{}$,
$\subspace{10342\\01223}{}$,
$\subspace{10304\\01400}{}$,
$\subspace{10311\\01404}{}$,
$\subspace{10323\\01403}{}$,
$\subspace{10330\\01402}{}$,
$\subspace{10342\\01401}{}$,
$\subspace{10202\\01000}{15}$,
$\subspace{10214\\01032}{}$,
$\subspace{10221\\01014}{}$,
$\subspace{10233\\01041}{}$,
$\subspace{10240\\01023}{}$,
$\subspace{10303\\01111}{}$,
$\subspace{10310\\01103}{}$,
$\subspace{10322\\01140}{}$,
$\subspace{10334\\01132}{}$,
$\subspace{10341\\01124}{}$,
$\subspace{13003\\00100}{}$,
$\subspace{13012\\00124}{}$,
$\subspace{13021\\00143}{}$,
$\subspace{13030\\00112}{}$,
$\subspace{13044\\00131}{}$,
$\subspace{10203\\01001}{15}$,
$\subspace{10210\\01033}{}$,
$\subspace{10222\\01010}{}$,
$\subspace{10234\\01042}{}$,
$\subspace{10241\\01024}{}$,
$\subspace{10302\\01130}{}$,
$\subspace{10314\\01122}{}$,
$\subspace{10321\\01114}{}$,
$\subspace{10333\\01101}{}$,
$\subspace{10340\\01143}{}$,
$\subspace{13002\\00101}{}$,
$\subspace{13011\\00120}{}$,
$\subspace{13020\\00144}{}$,
$\subspace{13034\\00113}{}$,
$\subspace{13043\\00132}{}$,
$\subspace{10203\\01001}{15}$,
$\subspace{10210\\01033}{}$,
$\subspace{10222\\01010}{}$,
$\subspace{10234\\01042}{}$,
$\subspace{10241\\01024}{}$,
$\subspace{10302\\01130}{}$,
$\subspace{10314\\01122}{}$,
$\subspace{10321\\01114}{}$,
$\subspace{10333\\01101}{}$,
$\subspace{10340\\01143}{}$,
$\subspace{13002\\00101}{}$,
$\subspace{13011\\00120}{}$,
$\subspace{13020\\00144}{}$,
$\subspace{13034\\00113}{}$,
$\subspace{13043\\00132}{}$,
$\subspace{10204\\01004}{15}$,
$\subspace{10211\\01031}{}$,
$\subspace{10223\\01013}{}$,
$\subspace{10230\\01040}{}$,
$\subspace{10242\\01022}{}$,
$\subspace{10303\\01121}{}$,
$\subspace{10310\\01113}{}$,
$\subspace{10322\\01100}{}$,
$\subspace{10334\\01142}{}$,
$\subspace{10341\\01134}{}$,
$\subspace{13003\\00111}{}$,
$\subspace{13012\\00130}{}$,
$\subspace{13021\\00104}{}$,
$\subspace{13030\\00123}{}$,
$\subspace{13044\\00142}{}$,
$\subspace{10201\\01020}{15}$,
$\subspace{10213\\01002}{}$,
$\subspace{10220\\01034}{}$,
$\subspace{10232\\01011}{}$,
$\subspace{10244\\01043}{}$,
$\subspace{10304\\01133}{}$,
$\subspace{10311\\01120}{}$,
$\subspace{10323\\01112}{}$,
$\subspace{10330\\01104}{}$,
$\subspace{10342\\01141}{}$,
$\subspace{13004\\00133}{}$,
$\subspace{13013\\00102}{}$,
$\subspace{13022\\00121}{}$,
$\subspace{13031\\00140}{}$,
$\subspace{13040\\00114}{}$,
$\subspace{10201\\01020}{15}$,
$\subspace{10213\\01002}{}$,
$\subspace{10220\\01034}{}$,
$\subspace{10232\\01011}{}$,
$\subspace{10244\\01043}{}$,
$\subspace{10304\\01133}{}$,
$\subspace{10311\\01120}{}$,
$\subspace{10323\\01112}{}$,
$\subspace{10330\\01104}{}$,
$\subspace{10342\\01141}{}$,
$\subspace{13004\\00133}{}$,
$\subspace{13013\\00102}{}$,
$\subspace{13022\\00121}{}$,
$\subspace{13031\\00140}{}$,
$\subspace{13040\\00114}{}$,
$\subspace{10203\\01023}{15}$,
$\subspace{10210\\01000}{}$,
$\subspace{10222\\01032}{}$,
$\subspace{10234\\01014}{}$,
$\subspace{10241\\01041}{}$,
$\subspace{10303\\01132}{}$,
$\subspace{10310\\01124}{}$,
$\subspace{10322\\01111}{}$,
$\subspace{10334\\01103}{}$,
$\subspace{10341\\01140}{}$,
$\subspace{13003\\00112}{}$,
$\subspace{13012\\00131}{}$,
$\subspace{13021\\00100}{}$,
$\subspace{13030\\00124}{}$,
$\subspace{13044\\00143}{}$,
$\subspace{10202\\01042}{15}$,
$\subspace{10214\\01024}{}$,
$\subspace{10221\\01001}{}$,
$\subspace{10233\\01033}{}$,
$\subspace{10240\\01010}{}$,
$\subspace{10303\\01143}{}$,
$\subspace{10310\\01130}{}$,
$\subspace{10322\\01122}{}$,
$\subspace{10334\\01114}{}$,
$\subspace{10341\\01101}{}$,
$\subspace{13003\\00113}{}$,
$\subspace{13012\\00132}{}$,
$\subspace{13021\\00101}{}$,
$\subspace{13030\\00120}{}$,
$\subspace{13044\\00144}{}$,
$\subspace{10201\\01103}{15}$,
$\subspace{10213\\01111}{}$,
$\subspace{10220\\01124}{}$,
$\subspace{10232\\01132}{}$,
$\subspace{10244\\01140}{}$,
$\subspace{10300\\01023}{}$,
$\subspace{10312\\01041}{}$,
$\subspace{10324\\01014}{}$,
$\subspace{10331\\01032}{}$,
$\subspace{10343\\01000}{}$,
$\subspace{12001\\00100}{}$,
$\subspace{12010\\00131}{}$,
$\subspace{12024\\00112}{}$,
$\subspace{12033\\00143}{}$,
$\subspace{12042\\00124}{}$,
$\subspace{10201\\01103}{15}$,
$\subspace{10213\\01111}{}$,
$\subspace{10220\\01124}{}$,
$\subspace{10232\\01132}{}$,
$\subspace{10244\\01140}{}$,
$\subspace{10300\\01023}{}$,
$\subspace{10312\\01041}{}$,
$\subspace{10324\\01014}{}$,
$\subspace{10331\\01032}{}$,
$\subspace{10343\\01000}{}$,
$\subspace{12001\\00100}{}$,
$\subspace{12010\\00131}{}$,
$\subspace{12024\\00112}{}$,
$\subspace{12033\\00143}{}$,
$\subspace{12042\\00124}{}$,
$\subspace{10203\\01100}{15}$,
$\subspace{10210\\01113}{}$,
$\subspace{10222\\01121}{}$,
$\subspace{10234\\01134}{}$,
$\subspace{10241\\01142}{}$,
$\subspace{10303\\01004}{}$,
$\subspace{10310\\01022}{}$,
$\subspace{10322\\01040}{}$,
$\subspace{10334\\01013}{}$,
$\subspace{10341\\01031}{}$,
$\subspace{12003\\00123}{}$,
$\subspace{12012\\00104}{}$,
$\subspace{12021\\00130}{}$,
$\subspace{12030\\00111}{}$,
$\subspace{12044\\00142}{}$,
$\subspace{10200\\01112}{15}$,
$\subspace{10212\\01120}{}$,
$\subspace{10224\\01133}{}$,
$\subspace{10231\\01141}{}$,
$\subspace{10243\\01104}{}$,
$\subspace{10304\\01034}{}$,
$\subspace{10311\\01002}{}$,
$\subspace{10323\\01020}{}$,
$\subspace{10330\\01043}{}$,
$\subspace{10342\\01011}{}$,
$\subspace{12000\\00102}{}$,
$\subspace{12014\\00133}{}$,
$\subspace{12023\\00114}{}$,
$\subspace{12032\\00140}{}$,
$\subspace{12041\\00121}{}$,
$\subspace{10204\\01120}{15}$,
$\subspace{10211\\01133}{}$,
$\subspace{10223\\01141}{}$,
$\subspace{10230\\01104}{}$,
$\subspace{10242\\01112}{}$,
$\subspace{10300\\01020}{}$,
$\subspace{10312\\01043}{}$,
$\subspace{10324\\01011}{}$,
$\subspace{10331\\01034}{}$,
$\subspace{10343\\01002}{}$,
$\subspace{12004\\00114}{}$,
$\subspace{12013\\00140}{}$,
$\subspace{12022\\00121}{}$,
$\subspace{12031\\00102}{}$,
$\subspace{12040\\00133}{}$,
$\subspace{10204\\01120}{15}$,
$\subspace{10211\\01133}{}$,
$\subspace{10223\\01141}{}$,
$\subspace{10230\\01104}{}$,
$\subspace{10242\\01112}{}$,
$\subspace{10300\\01020}{}$,
$\subspace{10312\\01043}{}$,
$\subspace{10324\\01011}{}$,
$\subspace{10331\\01034}{}$,
$\subspace{10343\\01002}{}$,
$\subspace{12004\\00114}{}$,
$\subspace{12013\\00140}{}$,
$\subspace{12022\\00121}{}$,
$\subspace{12031\\00102}{}$,
$\subspace{12040\\00133}{}$,
$\subspace{10202\\01134}{15}$,
$\subspace{10214\\01142}{}$,
$\subspace{10221\\01100}{}$,
$\subspace{10233\\01113}{}$,
$\subspace{10240\\01121}{}$,
$\subspace{10303\\01013}{}$,
$\subspace{10310\\01031}{}$,
$\subspace{10322\\01004}{}$,
$\subspace{10334\\01022}{}$,
$\subspace{10341\\01040}{}$,
$\subspace{12002\\00142}{}$,
$\subspace{12011\\00123}{}$,
$\subspace{12020\\00104}{}$,
$\subspace{12034\\00130}{}$,
$\subspace{12043\\00111}{}$,
$\subspace{10200\\01142}{15}$,
$\subspace{10212\\01100}{}$,
$\subspace{10224\\01113}{}$,
$\subspace{10231\\01121}{}$,
$\subspace{10243\\01134}{}$,
$\subspace{10303\\01031}{}$,
$\subspace{10310\\01004}{}$,
$\subspace{10322\\01022}{}$,
$\subspace{10334\\01040}{}$,
$\subspace{10341\\01013}{}$,
$\subspace{12000\\00130}{}$,
$\subspace{12014\\00111}{}$,
$\subspace{12023\\00142}{}$,
$\subspace{12032\\00123}{}$,
$\subspace{12041\\00104}{}$,
$\subspace{10203\\01141}{15}$,
$\subspace{10210\\01104}{}$,
$\subspace{10222\\01112}{}$,
$\subspace{10234\\01120}{}$,
$\subspace{10241\\01133}{}$,
$\subspace{10303\\01020}{}$,
$\subspace{10310\\01043}{}$,
$\subspace{10322\\01011}{}$,
$\subspace{10334\\01034}{}$,
$\subspace{10341\\01002}{}$,
$\subspace{12003\\00102}{}$,
$\subspace{12012\\00133}{}$,
$\subspace{12021\\00114}{}$,
$\subspace{12030\\00140}{}$,
$\subspace{12044\\00121}{}$,
$\subspace{10200\\01203}{15}$,
$\subspace{10212\\01242}{}$,
$\subspace{10224\\01231}{}$,
$\subspace{10231\\01220}{}$,
$\subspace{10243\\01214}{}$,
$\subspace{10200\\01423}{}$,
$\subspace{10212\\01424}{}$,
$\subspace{10224\\01420}{}$,
$\subspace{10231\\01421}{}$,
$\subspace{10243\\01422}{}$,
$\subspace{10402\\01304}{}$,
$\subspace{10414\\01314}{}$,
$\subspace{10421\\01324}{}$,
$\subspace{10433\\01334}{}$,
$\subspace{10440\\01344}{}$,
$\subspace{10200\\01203}{15}$,
$\subspace{10212\\01242}{}$,
$\subspace{10224\\01231}{}$,
$\subspace{10231\\01220}{}$,
$\subspace{10243\\01214}{}$,
$\subspace{10200\\01423}{}$,
$\subspace{10212\\01424}{}$,
$\subspace{10224\\01420}{}$,
$\subspace{10231\\01421}{}$,
$\subspace{10243\\01422}{}$,
$\subspace{10402\\01304}{}$,
$\subspace{10414\\01314}{}$,
$\subspace{10421\\01324}{}$,
$\subspace{10433\\01334}{}$,
$\subspace{10440\\01344}{}$,
$\subspace{10201\\01200}{15}$,
$\subspace{10213\\01244}{}$,
$\subspace{10220\\01233}{}$,
$\subspace{10232\\01222}{}$,
$\subspace{10244\\01211}{}$,
$\subspace{10201\\01441}{}$,
$\subspace{10213\\01442}{}$,
$\subspace{10220\\01443}{}$,
$\subspace{10232\\01444}{}$,
$\subspace{10244\\01440}{}$,
$\subspace{10401\\01300}{}$,
$\subspace{10413\\01310}{}$,
$\subspace{10420\\01320}{}$,
$\subspace{10432\\01330}{}$,
$\subspace{10444\\01340}{}$,
$\subspace{10202\\01211}{15}$,
$\subspace{10214\\01200}{}$,
$\subspace{10221\\01244}{}$,
$\subspace{10233\\01233}{}$,
$\subspace{10240\\01222}{}$,
$\subspace{10201\\01444}{}$,
$\subspace{10213\\01440}{}$,
$\subspace{10220\\01441}{}$,
$\subspace{10232\\01442}{}$,
$\subspace{10244\\01443}{}$,
$\subspace{10400\\01330}{}$,
$\subspace{10412\\01340}{}$,
$\subspace{10424\\01300}{}$,
$\subspace{10431\\01310}{}$,
$\subspace{10443\\01320}{}$,
$\subspace{10203\\01212}{15}$,
$\subspace{10210\\01201}{}$,
$\subspace{10222\\01240}{}$,
$\subspace{10234\\01234}{}$,
$\subspace{10241\\01223}{}$,
$\subspace{10203\\01400}{}$,
$\subspace{10210\\01401}{}$,
$\subspace{10222\\01402}{}$,
$\subspace{10234\\01403}{}$,
$\subspace{10241\\01404}{}$,
$\subspace{10400\\01313}{}$,
$\subspace{10412\\01323}{}$,
$\subspace{10424\\01333}{}$,
$\subspace{10431\\01343}{}$,
$\subspace{10443\\01303}{}$,
$\subspace{10203\\01213}{15}$,
$\subspace{10210\\01202}{}$,
$\subspace{10222\\01241}{}$,
$\subspace{10234\\01230}{}$,
$\subspace{10241\\01224}{}$,
$\subspace{10202\\01412}{}$,
$\subspace{10214\\01413}{}$,
$\subspace{10221\\01414}{}$,
$\subspace{10233\\01410}{}$,
$\subspace{10240\\01411}{}$,
$\subspace{10404\\01331}{}$,
$\subspace{10411\\01341}{}$,
$\subspace{10423\\01301}{}$,
$\subspace{10430\\01311}{}$,
$\subspace{10442\\01321}{}$,
$\subspace{10201\\01233}{15}$,
$\subspace{10213\\01222}{}$,
$\subspace{10220\\01211}{}$,
$\subspace{10232\\01200}{}$,
$\subspace{10244\\01244}{}$,
$\subspace{10202\\01443}{}$,
$\subspace{10214\\01444}{}$,
$\subspace{10221\\01440}{}$,
$\subspace{10233\\01441}{}$,
$\subspace{10240\\01442}{}$,
$\subspace{10400\\01310}{}$,
$\subspace{10412\\01320}{}$,
$\subspace{10424\\01330}{}$,
$\subspace{10431\\01340}{}$,
$\subspace{10443\\01300}{}$,
$\subspace{10200\\01240}{15}$,
$\subspace{10212\\01234}{}$,
$\subspace{10224\\01223}{}$,
$\subspace{10231\\01212}{}$,
$\subspace{10243\\01201}{}$,
$\subspace{10200\\01400}{}$,
$\subspace{10212\\01401}{}$,
$\subspace{10224\\01402}{}$,
$\subspace{10231\\01403}{}$,
$\subspace{10243\\01404}{}$,
$\subspace{10401\\01343}{}$,
$\subspace{10413\\01303}{}$,
$\subspace{10420\\01313}{}$,
$\subspace{10432\\01323}{}$,
$\subspace{10444\\01333}{}$,
$\subspace{10301\\01303}{15}$,
$\subspace{10313\\01333}{}$,
$\subspace{10320\\01313}{}$,
$\subspace{10332\\01343}{}$,
$\subspace{10344\\01323}{}$,
$\subspace{10403\\01201}{}$,
$\subspace{10410\\01223}{}$,
$\subspace{10422\\01240}{}$,
$\subspace{10434\\01212}{}$,
$\subspace{10441\\01234}{}$,
$\subspace{10403\\01400}{}$,
$\subspace{10410\\01403}{}$,
$\subspace{10422\\01401}{}$,
$\subspace{10434\\01404}{}$,
$\subspace{10441\\01402}{}$,
$\subspace{10301\\01310}{15}$,
$\subspace{10313\\01340}{}$,
$\subspace{10320\\01320}{}$,
$\subspace{10332\\01300}{}$,
$\subspace{10344\\01330}{}$,
$\subspace{10400\\01211}{}$,
$\subspace{10412\\01233}{}$,
$\subspace{10424\\01200}{}$,
$\subspace{10431\\01222}{}$,
$\subspace{10443\\01244}{}$,
$\subspace{10400\\01441}{}$,
$\subspace{10412\\01444}{}$,
$\subspace{10424\\01442}{}$,
$\subspace{10431\\01440}{}$,
$\subspace{10443\\01443}{}$,
$\subspace{10302\\01312}{15}$,
$\subspace{10314\\01342}{}$,
$\subspace{10321\\01322}{}$,
$\subspace{10333\\01302}{}$,
$\subspace{10340\\01332}{}$,
$\subspace{10404\\01210}{}$,
$\subspace{10411\\01232}{}$,
$\subspace{10423\\01204}{}$,
$\subspace{10430\\01221}{}$,
$\subspace{10442\\01243}{}$,
$\subspace{10404\\01432}{}$,
$\subspace{10411\\01430}{}$,
$\subspace{10423\\01433}{}$,
$\subspace{10430\\01431}{}$,
$\subspace{10442\\01434}{}$,
$\subspace{10302\\01312}{15}$,
$\subspace{10314\\01342}{}$,
$\subspace{10321\\01322}{}$,
$\subspace{10333\\01302}{}$,
$\subspace{10340\\01332}{}$,
$\subspace{10404\\01210}{}$,
$\subspace{10411\\01232}{}$,
$\subspace{10423\\01204}{}$,
$\subspace{10430\\01221}{}$,
$\subspace{10442\\01243}{}$,
$\subspace{10404\\01432}{}$,
$\subspace{10411\\01430}{}$,
$\subspace{10423\\01433}{}$,
$\subspace{10430\\01431}{}$,
$\subspace{10442\\01434}{}$,
$\subspace{10302\\01312}{15}$,
$\subspace{10314\\01342}{}$,
$\subspace{10321\\01322}{}$,
$\subspace{10333\\01302}{}$,
$\subspace{10340\\01332}{}$,
$\subspace{10404\\01210}{}$,
$\subspace{10411\\01232}{}$,
$\subspace{10423\\01204}{}$,
$\subspace{10430\\01221}{}$,
$\subspace{10442\\01243}{}$,
$\subspace{10404\\01432}{}$,
$\subspace{10411\\01430}{}$,
$\subspace{10423\\01433}{}$,
$\subspace{10430\\01431}{}$,
$\subspace{10442\\01434}{}$,
$\subspace{10302\\01312}{15}$,
$\subspace{10314\\01342}{}$,
$\subspace{10321\\01322}{}$,
$\subspace{10333\\01302}{}$,
$\subspace{10340\\01332}{}$,
$\subspace{10404\\01210}{}$,
$\subspace{10411\\01232}{}$,
$\subspace{10423\\01204}{}$,
$\subspace{10430\\01221}{}$,
$\subspace{10442\\01243}{}$,
$\subspace{10404\\01432}{}$,
$\subspace{10411\\01430}{}$,
$\subspace{10423\\01433}{}$,
$\subspace{10430\\01431}{}$,
$\subspace{10442\\01434}{}$,
$\subspace{10300\\01344}{15}$,
$\subspace{10312\\01324}{}$,
$\subspace{10324\\01304}{}$,
$\subspace{10331\\01334}{}$,
$\subspace{10343\\01314}{}$,
$\subspace{10402\\01242}{}$,
$\subspace{10414\\01214}{}$,
$\subspace{10421\\01231}{}$,
$\subspace{10433\\01203}{}$,
$\subspace{10440\\01220}{}$,
$\subspace{10402\\01423}{}$,
$\subspace{10414\\01421}{}$,
$\subspace{10421\\01424}{}$,
$\subspace{10433\\01422}{}$,
$\subspace{10440\\01420}{}$,
$\subspace{10300\\01344}{15}$,
$\subspace{10312\\01324}{}$,
$\subspace{10324\\01304}{}$,
$\subspace{10331\\01334}{}$,
$\subspace{10343\\01314}{}$,
$\subspace{10402\\01242}{}$,
$\subspace{10414\\01214}{}$,
$\subspace{10421\\01231}{}$,
$\subspace{10433\\01203}{}$,
$\subspace{10440\\01220}{}$,
$\subspace{10402\\01423}{}$,
$\subspace{10414\\01421}{}$,
$\subspace{10421\\01424}{}$,
$\subspace{10433\\01422}{}$,
$\subspace{10440\\01420}{}$,
$\subspace{10300\\01344}{15}$,
$\subspace{10312\\01324}{}$,
$\subspace{10324\\01304}{}$,
$\subspace{10331\\01334}{}$,
$\subspace{10343\\01314}{}$,
$\subspace{10402\\01242}{}$,
$\subspace{10414\\01214}{}$,
$\subspace{10421\\01231}{}$,
$\subspace{10433\\01203}{}$,
$\subspace{10440\\01220}{}$,
$\subspace{10402\\01423}{}$,
$\subspace{10414\\01421}{}$,
$\subspace{10421\\01424}{}$,
$\subspace{10433\\01422}{}$,
$\subspace{10440\\01420}{}$,
$\subspace{11200\\00011}{3}$,
$\subspace{11440\\00001}{}$,
$\subspace{13400\\00010}{}$,
$\subspace{11202\\00011}{3}$,
$\subspace{11410\\00001}{}$,
$\subspace{13403\\00010}{}$,
$\subspace{11204\\00011}{3}$,
$\subspace{11430\\00001}{}$,
$\subspace{13401\\00010}{}$,
$\subspace{11204\\00011}{3}$,
$\subspace{11430\\00001}{}$,
$\subspace{13401\\00010}{}$,
$\subspace{11204\\00011}{3}$,
$\subspace{11430\\00001}{}$,
$\subspace{13401\\00010}{}$.

\smallskip

$n_5(5,2;54)\ge 1354$,$\left[\!\begin{smallmatrix}10000\\00100\\04410\\00001\\00044\end{smallmatrix}\!\right]$:
$\subspace{01000\\00102}{15}$,
$\subspace{01002\\00122}{}$,
$\subspace{01003\\00143}{}$,
$\subspace{01012\\00100}{}$,
$\subspace{01014\\00114}{}$,
$\subspace{01011\\00134}{}$,
$\subspace{01021\\00112}{}$,
$\subspace{01023\\00121}{}$,
$\subspace{01020\\00141}{}$,
$\subspace{01034\\00103}{}$,
$\subspace{01030\\00124}{}$,
$\subspace{01032\\00133}{}$,
$\subspace{01043\\00110}{}$,
$\subspace{01044\\00131}{}$,
$\subspace{01041\\00140}{}$,
$\subspace{01002\\00102}{5}$,
$\subspace{01011\\00114}{}$,
$\subspace{01020\\00121}{}$,
$\subspace{01034\\00133}{}$,
$\subspace{01043\\00140}{}$,
$\subspace{01003\\00141}{5}$,
$\subspace{01012\\00103}{}$,
$\subspace{01021\\00110}{}$,
$\subspace{01030\\00122}{}$,
$\subspace{01044\\00134}{}$,
$\subspace{01201\\00011}{3}$,
$\subspace{01303\\00010}{}$,
$\subspace{01400\\00001}{}$,
$\subspace{01202\\00011}{3}$,
$\subspace{01301\\00010}{}$,
$\subspace{01410\\00001}{}$,
$\subspace{01202\\00011}{3}$,
$\subspace{01301\\00010}{}$,
$\subspace{01410\\00001}{}$,
$\subspace{10000\\00001}{3}$,
$\subspace{10000\\00010}{}$,
$\subspace{10000\\00011}{}$,
$\subspace{10000\\00001}{3}$,
$\subspace{10000\\00010}{}$,
$\subspace{10000\\00011}{}$,
$\subspace{10000\\00012}{3}$,
$\subspace{10000\\00013}{}$,
$\subspace{10000\\00014}{}$,
$\subspace{10012\\00103}{15}$,
$\subspace{10012\\00110}{}$,
$\subspace{10012\\00122}{}$,
$\subspace{10012\\00134}{}$,
$\subspace{10012\\00141}{}$,
$\subspace{10014\\01003}{}$,
$\subspace{10014\\01012}{}$,
$\subspace{10014\\01021}{}$,
$\subspace{10014\\01030}{}$,
$\subspace{10014\\01044}{}$,
$\subspace{10034\\01102}{}$,
$\subspace{10034\\01110}{}$,
$\subspace{10034\\01123}{}$,
$\subspace{10034\\01131}{}$,
$\subspace{10034\\01144}{}$,
$\subspace{10021\\00101}{15}$,
$\subspace{10021\\00113}{}$,
$\subspace{10021\\00120}{}$,
$\subspace{10021\\00132}{}$,
$\subspace{10021\\00144}{}$,
$\subspace{10043\\01001}{}$,
$\subspace{10043\\01010}{}$,
$\subspace{10043\\01024}{}$,
$\subspace{10043\\01033}{}$,
$\subspace{10043\\01042}{}$,
$\subspace{10041\\01101}{}$,
$\subspace{10041\\01114}{}$,
$\subspace{10041\\01122}{}$,
$\subspace{10041\\01130}{}$,
$\subspace{10041\\01143}{}$,
$\subspace{10021\\00103}{15}$,
$\subspace{10021\\00110}{}$,
$\subspace{10021\\00122}{}$,
$\subspace{10021\\00134}{}$,
$\subspace{10021\\00141}{}$,
$\subspace{10043\\01003}{}$,
$\subspace{10043\\01012}{}$,
$\subspace{10043\\01021}{}$,
$\subspace{10043\\01030}{}$,
$\subspace{10043\\01044}{}$,
$\subspace{10041\\01102}{}$,
$\subspace{10041\\01110}{}$,
$\subspace{10041\\01123}{}$,
$\subspace{10041\\01131}{}$,
$\subspace{10041\\01144}{}$,
$\subspace{10024\\00104}{15}$,
$\subspace{10024\\00111}{}$,
$\subspace{10024\\00123}{}$,
$\subspace{10024\\00130}{}$,
$\subspace{10024\\00142}{}$,
$\subspace{10023\\01004}{}$,
$\subspace{10023\\01013}{}$,
$\subspace{10023\\01022}{}$,
$\subspace{10023\\01031}{}$,
$\subspace{10023\\01040}{}$,
$\subspace{10013\\01100}{}$,
$\subspace{10013\\01113}{}$,
$\subspace{10013\\01121}{}$,
$\subspace{10013\\01134}{}$,
$\subspace{10013\\01142}{}$,
$\subspace{10031\\00103}{15}$,
$\subspace{10031\\00110}{}$,
$\subspace{10031\\00122}{}$,
$\subspace{10031\\00134}{}$,
$\subspace{10031\\00141}{}$,
$\subspace{10032\\01003}{}$,
$\subspace{10032\\01012}{}$,
$\subspace{10032\\01021}{}$,
$\subspace{10032\\01030}{}$,
$\subspace{10032\\01044}{}$,
$\subspace{10042\\01102}{}$,
$\subspace{10042\\01110}{}$,
$\subspace{10042\\01123}{}$,
$\subspace{10042\\01131}{}$,
$\subspace{10042\\01144}{}$,
$\subspace{10033\\00103}{15}$,
$\subspace{10033\\00110}{}$,
$\subspace{10033\\00122}{}$,
$\subspace{10033\\00134}{}$,
$\subspace{10033\\00141}{}$,
$\subspace{10002\\01003}{}$,
$\subspace{10002\\01012}{}$,
$\subspace{10002\\01021}{}$,
$\subspace{10002\\01030}{}$,
$\subspace{10002\\01044}{}$,
$\subspace{10020\\01102}{}$,
$\subspace{10020\\01110}{}$,
$\subspace{10020\\01123}{}$,
$\subspace{10020\\01131}{}$,
$\subspace{10020\\01144}{}$,
$\subspace{10033\\00103}{15}$,
$\subspace{10033\\00110}{}$,
$\subspace{10033\\00122}{}$,
$\subspace{10033\\00134}{}$,
$\subspace{10033\\00141}{}$,
$\subspace{10002\\01003}{}$,
$\subspace{10002\\01012}{}$,
$\subspace{10002\\01021}{}$,
$\subspace{10002\\01030}{}$,
$\subspace{10002\\01044}{}$,
$\subspace{10020\\01102}{}$,
$\subspace{10020\\01110}{}$,
$\subspace{10020\\01123}{}$,
$\subspace{10020\\01131}{}$,
$\subspace{10020\\01144}{}$,
$\subspace{10042\\00101}{15}$,
$\subspace{10042\\00113}{}$,
$\subspace{10042\\00120}{}$,
$\subspace{10042\\00132}{}$,
$\subspace{10042\\00144}{}$,
$\subspace{10031\\01001}{}$,
$\subspace{10031\\01010}{}$,
$\subspace{10031\\01024}{}$,
$\subspace{10031\\01033}{}$,
$\subspace{10031\\01042}{}$,
$\subspace{10032\\01101}{}$,
$\subspace{10032\\01114}{}$,
$\subspace{10032\\01122}{}$,
$\subspace{10032\\01130}{}$,
$\subspace{10032\\01143}{}$,
$\subspace{10103\\00012}{3}$,
$\subspace{11003\\00014}{}$,
$\subspace{14403\\00013}{}$,
$\subspace{10202\\00012}{3}$,
$\subspace{12002\\00014}{}$,
$\subspace{13303\\00013}{}$,
$\subspace{10202\\00012}{3}$,
$\subspace{12002\\00014}{}$,
$\subspace{13303\\00013}{}$,
$\subspace{10400\\00012}{3}$,
$\subspace{11103\\00013}{}$,
$\subspace{14000\\00014}{}$,
$\subspace{10400\\00012}{3}$,
$\subspace{11103\\00013}{}$,
$\subspace{14000\\00014}{}$,
$\subspace{10001\\01203}{15}$,
$\subspace{10001\\01214}{}$,
$\subspace{10001\\01220}{}$,
$\subspace{10001\\01231}{}$,
$\subspace{10001\\01242}{}$,
$\subspace{10044\\01304}{}$,
$\subspace{10044\\01314}{}$,
$\subspace{10044\\01324}{}$,
$\subspace{10044\\01334}{}$,
$\subspace{10044\\01344}{}$,
$\subspace{10010\\01420}{}$,
$\subspace{10010\\01421}{}$,
$\subspace{10010\\01422}{}$,
$\subspace{10010\\01423}{}$,
$\subspace{10010\\01424}{}$,
$\subspace{10014\\01201}{15}$,
$\subspace{10014\\01212}{}$,
$\subspace{10014\\01223}{}$,
$\subspace{10014\\01234}{}$,
$\subspace{10014\\01240}{}$,
$\subspace{10012\\01303}{}$,
$\subspace{10012\\01313}{}$,
$\subspace{10012\\01323}{}$,
$\subspace{10012\\01333}{}$,
$\subspace{10012\\01343}{}$,
$\subspace{10034\\01400}{}$,
$\subspace{10034\\01401}{}$,
$\subspace{10034\\01402}{}$,
$\subspace{10034\\01403}{}$,
$\subspace{10034\\01404}{}$,
$\subspace{10022\\01200}{15}$,
$\subspace{10022\\01211}{}$,
$\subspace{10022\\01222}{}$,
$\subspace{10022\\01233}{}$,
$\subspace{10022\\01244}{}$,
$\subspace{10030\\01300}{}$,
$\subspace{10030\\01310}{}$,
$\subspace{10030\\01320}{}$,
$\subspace{10030\\01330}{}$,
$\subspace{10030\\01340}{}$,
$\subspace{10003\\01440}{}$,
$\subspace{10003\\01441}{}$,
$\subspace{10003\\01442}{}$,
$\subspace{10003\\01443}{}$,
$\subspace{10003\\01444}{}$,
$\subspace{10022\\01200}{15}$,
$\subspace{10022\\01211}{}$,
$\subspace{10022\\01222}{}$,
$\subspace{10022\\01233}{}$,
$\subspace{10022\\01244}{}$,
$\subspace{10030\\01300}{}$,
$\subspace{10030\\01310}{}$,
$\subspace{10030\\01320}{}$,
$\subspace{10030\\01330}{}$,
$\subspace{10030\\01340}{}$,
$\subspace{10003\\01440}{}$,
$\subspace{10003\\01441}{}$,
$\subspace{10003\\01442}{}$,
$\subspace{10003\\01443}{}$,
$\subspace{10003\\01444}{}$,
$\subspace{10024\\01202}{15}$,
$\subspace{10024\\01213}{}$,
$\subspace{10024\\01224}{}$,
$\subspace{10024\\01230}{}$,
$\subspace{10024\\01241}{}$,
$\subspace{10013\\01301}{}$,
$\subspace{10013\\01311}{}$,
$\subspace{10013\\01321}{}$,
$\subspace{10013\\01331}{}$,
$\subspace{10013\\01341}{}$,
$\subspace{10023\\01410}{}$,
$\subspace{10023\\01411}{}$,
$\subspace{10023\\01412}{}$,
$\subspace{10023\\01413}{}$,
$\subspace{10023\\01414}{}$,
$\subspace{10040\\01201}{15}$,
$\subspace{10040\\01212}{}$,
$\subspace{10040\\01223}{}$,
$\subspace{10040\\01234}{}$,
$\subspace{10040\\01240}{}$,
$\subspace{10004\\01303}{}$,
$\subspace{10004\\01313}{}$,
$\subspace{10004\\01323}{}$,
$\subspace{10004\\01333}{}$,
$\subspace{10004\\01343}{}$,
$\subspace{10011\\01400}{}$,
$\subspace{10011\\01401}{}$,
$\subspace{10011\\01402}{}$,
$\subspace{10011\\01403}{}$,
$\subspace{10011\\01404}{}$,
$\subspace{10040\\01203}{15}$,
$\subspace{10040\\01214}{}$,
$\subspace{10040\\01220}{}$,
$\subspace{10040\\01231}{}$,
$\subspace{10040\\01242}{}$,
$\subspace{10004\\01304}{}$,
$\subspace{10004\\01314}{}$,
$\subspace{10004\\01324}{}$,
$\subspace{10004\\01334}{}$,
$\subspace{10004\\01344}{}$,
$\subspace{10011\\01420}{}$,
$\subspace{10011\\01421}{}$,
$\subspace{10011\\01422}{}$,
$\subspace{10011\\01423}{}$,
$\subspace{10011\\01424}{}$,
$\subspace{10044\\01201}{15}$,
$\subspace{10044\\01212}{}$,
$\subspace{10044\\01223}{}$,
$\subspace{10044\\01234}{}$,
$\subspace{10044\\01240}{}$,
$\subspace{10010\\01303}{}$,
$\subspace{10010\\01313}{}$,
$\subspace{10010\\01323}{}$,
$\subspace{10010\\01333}{}$,
$\subspace{10010\\01343}{}$,
$\subspace{10001\\01400}{}$,
$\subspace{10001\\01401}{}$,
$\subspace{10001\\01402}{}$,
$\subspace{10001\\01403}{}$,
$\subspace{10001\\01404}{}$,
$\subspace{10102\\01010}{15}$,
$\subspace{10114\\01024}{}$,
$\subspace{10121\\01033}{}$,
$\subspace{10133\\01042}{}$,
$\subspace{10140\\01001}{}$,
$\subspace{10403\\01122}{}$,
$\subspace{10410\\01101}{}$,
$\subspace{10422\\01130}{}$,
$\subspace{10434\\01114}{}$,
$\subspace{10441\\01143}{}$,
$\subspace{14003\\00144}{}$,
$\subspace{14012\\00132}{}$,
$\subspace{14021\\00120}{}$,
$\subspace{14030\\00113}{}$,
$\subspace{14044\\00101}{}$,
$\subspace{10102\\01010}{15}$,
$\subspace{10114\\01024}{}$,
$\subspace{10121\\01033}{}$,
$\subspace{10133\\01042}{}$,
$\subspace{10140\\01001}{}$,
$\subspace{10403\\01122}{}$,
$\subspace{10410\\01101}{}$,
$\subspace{10422\\01130}{}$,
$\subspace{10434\\01114}{}$,
$\subspace{10441\\01143}{}$,
$\subspace{14003\\00144}{}$,
$\subspace{14012\\00132}{}$,
$\subspace{14021\\00120}{}$,
$\subspace{14030\\00113}{}$,
$\subspace{14044\\00101}{}$,
$\subspace{10104\\01022}{15}$,
$\subspace{10111\\01031}{}$,
$\subspace{10123\\01040}{}$,
$\subspace{10130\\01004}{}$,
$\subspace{10142\\01013}{}$,
$\subspace{10401\\01121}{}$,
$\subspace{10413\\01100}{}$,
$\subspace{10420\\01134}{}$,
$\subspace{10432\\01113}{}$,
$\subspace{10444\\01142}{}$,
$\subspace{14001\\00130}{}$,
$\subspace{14010\\00123}{}$,
$\subspace{14024\\00111}{}$,
$\subspace{14033\\00104}{}$,
$\subspace{14042\\00142}{}$,
$\subspace{10101\\01030}{15}$,
$\subspace{10113\\01044}{}$,
$\subspace{10120\\01003}{}$,
$\subspace{10132\\01012}{}$,
$\subspace{10144\\01021}{}$,
$\subspace{10402\\01131}{}$,
$\subspace{10414\\01110}{}$,
$\subspace{10421\\01144}{}$,
$\subspace{10433\\01123}{}$,
$\subspace{10440\\01102}{}$,
$\subspace{14002\\00103}{}$,
$\subspace{14011\\00141}{}$,
$\subspace{14020\\00134}{}$,
$\subspace{14034\\00122}{}$,
$\subspace{14043\\00110}{}$,
$\subspace{10101\\01030}{15}$,
$\subspace{10113\\01044}{}$,
$\subspace{10120\\01003}{}$,
$\subspace{10132\\01012}{}$,
$\subspace{10144\\01021}{}$,
$\subspace{10402\\01131}{}$,
$\subspace{10414\\01110}{}$,
$\subspace{10421\\01144}{}$,
$\subspace{10433\\01123}{}$,
$\subspace{10440\\01102}{}$,
$\subspace{14002\\00103}{}$,
$\subspace{14011\\00141}{}$,
$\subspace{14020\\00134}{}$,
$\subspace{14034\\00122}{}$,
$\subspace{14043\\00110}{}$,
$\subspace{10104\\01032}{15}$,
$\subspace{10111\\01041}{}$,
$\subspace{10123\\01000}{}$,
$\subspace{10130\\01014}{}$,
$\subspace{10142\\01023}{}$,
$\subspace{10402\\01111}{}$,
$\subspace{10414\\01140}{}$,
$\subspace{10421\\01124}{}$,
$\subspace{10433\\01103}{}$,
$\subspace{10440\\01132}{}$,
$\subspace{14002\\00131}{}$,
$\subspace{14011\\00124}{}$,
$\subspace{14020\\00112}{}$,
$\subspace{14034\\00100}{}$,
$\subspace{14043\\00143}{}$,
$\subspace{10104\\01040}{15}$,
$\subspace{10111\\01004}{}$,
$\subspace{10123\\01013}{}$,
$\subspace{10130\\01022}{}$,
$\subspace{10142\\01031}{}$,
$\subspace{10402\\01134}{}$,
$\subspace{10414\\01113}{}$,
$\subspace{10421\\01142}{}$,
$\subspace{10433\\01121}{}$,
$\subspace{10440\\01100}{}$,
$\subspace{14002\\00123}{}$,
$\subspace{14011\\00111}{}$,
$\subspace{14020\\00104}{}$,
$\subspace{14034\\00142}{}$,
$\subspace{14043\\00130}{}$,
$\subspace{10104\\01040}{15}$,
$\subspace{10111\\01004}{}$,
$\subspace{10123\\01013}{}$,
$\subspace{10130\\01022}{}$,
$\subspace{10142\\01031}{}$,
$\subspace{10402\\01134}{}$,
$\subspace{10414\\01113}{}$,
$\subspace{10421\\01142}{}$,
$\subspace{10433\\01121}{}$,
$\subspace{10440\\01100}{}$,
$\subspace{14002\\00123}{}$,
$\subspace{14011\\00111}{}$,
$\subspace{14020\\00104}{}$,
$\subspace{14034\\00142}{}$,
$\subspace{14043\\00130}{}$,
$\subspace{10104\\01041}{15}$,
$\subspace{10111\\01000}{}$,
$\subspace{10123\\01014}{}$,
$\subspace{10130\\01023}{}$,
$\subspace{10142\\01032}{}$,
$\subspace{10400\\01140}{}$,
$\subspace{10412\\01124}{}$,
$\subspace{10424\\01103}{}$,
$\subspace{10431\\01132}{}$,
$\subspace{10443\\01111}{}$,
$\subspace{14000\\00100}{}$,
$\subspace{14014\\00143}{}$,
$\subspace{14023\\00131}{}$,
$\subspace{14032\\00124}{}$,
$\subspace{14041\\00112}{}$,
$\subspace{10104\\01041}{15}$,
$\subspace{10111\\01000}{}$,
$\subspace{10123\\01014}{}$,
$\subspace{10130\\01023}{}$,
$\subspace{10142\\01032}{}$,
$\subspace{10400\\01140}{}$,
$\subspace{10412\\01124}{}$,
$\subspace{10424\\01103}{}$,
$\subspace{10431\\01132}{}$,
$\subspace{10443\\01111}{}$,
$\subspace{14000\\00100}{}$,
$\subspace{14014\\00143}{}$,
$\subspace{14023\\00131}{}$,
$\subspace{14032\\00124}{}$,
$\subspace{14041\\00112}{}$,
$\subspace{10102\\01104}{15}$,
$\subspace{10114\\01120}{}$,
$\subspace{10121\\01141}{}$,
$\subspace{10133\\01112}{}$,
$\subspace{10140\\01133}{}$,
$\subspace{10403\\01002}{}$,
$\subspace{10410\\01043}{}$,
$\subspace{10422\\01034}{}$,
$\subspace{10434\\01020}{}$,
$\subspace{10441\\01011}{}$,
$\subspace{11002\\00102}{}$,
$\subspace{11011\\00114}{}$,
$\subspace{11020\\00121}{}$,
$\subspace{11034\\00133}{}$,
$\subspace{11043\\00140}{}$,
$\subspace{10100\\01114}{15}$,
$\subspace{10112\\01130}{}$,
$\subspace{10124\\01101}{}$,
$\subspace{10131\\01122}{}$,
$\subspace{10143\\01143}{}$,
$\subspace{10400\\01024}{}$,
$\subspace{10412\\01010}{}$,
$\subspace{10424\\01001}{}$,
$\subspace{10431\\01042}{}$,
$\subspace{10443\\01033}{}$,
$\subspace{11000\\00132}{}$,
$\subspace{11014\\00144}{}$,
$\subspace{11023\\00101}{}$,
$\subspace{11032\\00113}{}$,
$\subspace{11041\\00120}{}$,
$\subspace{10100\\01114}{15}$,
$\subspace{10112\\01130}{}$,
$\subspace{10124\\01101}{}$,
$\subspace{10131\\01122}{}$,
$\subspace{10143\\01143}{}$,
$\subspace{10400\\01024}{}$,
$\subspace{10412\\01010}{}$,
$\subspace{10424\\01001}{}$,
$\subspace{10431\\01042}{}$,
$\subspace{10443\\01033}{}$,
$\subspace{11000\\00132}{}$,
$\subspace{11014\\00144}{}$,
$\subspace{11023\\00101}{}$,
$\subspace{11032\\00113}{}$,
$\subspace{11041\\00120}{}$,
$\subspace{10100\\01114}{15}$,
$\subspace{10112\\01130}{}$,
$\subspace{10124\\01101}{}$,
$\subspace{10131\\01122}{}$,
$\subspace{10143\\01143}{}$,
$\subspace{10400\\01024}{}$,
$\subspace{10412\\01010}{}$,
$\subspace{10424\\01001}{}$,
$\subspace{10431\\01042}{}$,
$\subspace{10443\\01033}{}$,
$\subspace{11000\\00132}{}$,
$\subspace{11014\\00144}{}$,
$\subspace{11023\\00101}{}$,
$\subspace{11032\\00113}{}$,
$\subspace{11041\\00120}{}$,
$\subspace{10100\\01120}{15}$,
$\subspace{10112\\01141}{}$,
$\subspace{10124\\01112}{}$,
$\subspace{10131\\01133}{}$,
$\subspace{10143\\01104}{}$,
$\subspace{10403\\01043}{}$,
$\subspace{10410\\01034}{}$,
$\subspace{10422\\01020}{}$,
$\subspace{10434\\01011}{}$,
$\subspace{10441\\01002}{}$,
$\subspace{11000\\00133}{}$,
$\subspace{11014\\00140}{}$,
$\subspace{11023\\00102}{}$,
$\subspace{11032\\00114}{}$,
$\subspace{11041\\00121}{}$,
$\subspace{10100\\01132}{15}$,
$\subspace{10112\\01103}{}$,
$\subspace{10124\\01124}{}$,
$\subspace{10131\\01140}{}$,
$\subspace{10143\\01111}{}$,
$\subspace{10400\\01032}{}$,
$\subspace{10412\\01023}{}$,
$\subspace{10424\\01014}{}$,
$\subspace{10431\\01000}{}$,
$\subspace{10443\\01041}{}$,
$\subspace{11000\\00124}{}$,
$\subspace{11014\\00131}{}$,
$\subspace{11023\\00143}{}$,
$\subspace{11032\\00100}{}$,
$\subspace{11041\\00112}{}$,
$\subspace{10104\\01144}{15}$,
$\subspace{10111\\01110}{}$,
$\subspace{10123\\01131}{}$,
$\subspace{10130\\01102}{}$,
$\subspace{10142\\01123}{}$,
$\subspace{10401\\01030}{}$,
$\subspace{10413\\01021}{}$,
$\subspace{10420\\01012}{}$,
$\subspace{10432\\01003}{}$,
$\subspace{10444\\01044}{}$,
$\subspace{11004\\00122}{}$,
$\subspace{11013\\00134}{}$,
$\subspace{11022\\00141}{}$,
$\subspace{11031\\00103}{}$,
$\subspace{11040\\00110}{}$,
$\subspace{10103\\01210}{15}$,
$\subspace{10110\\01243}{}$,
$\subspace{10122\\01221}{}$,
$\subspace{10134\\01204}{}$,
$\subspace{10141\\01232}{}$,
$\subspace{10102\\01433}{}$,
$\subspace{10114\\01430}{}$,
$\subspace{10121\\01432}{}$,
$\subspace{10133\\01434}{}$,
$\subspace{10140\\01431}{}$,
$\subspace{10200\\01302}{}$,
$\subspace{10212\\01322}{}$,
$\subspace{10224\\01342}{}$,
$\subspace{10231\\01312}{}$,
$\subspace{10243\\01332}{}$,
$\subspace{10103\\01210}{15}$,
$\subspace{10110\\01243}{}$,
$\subspace{10122\\01221}{}$,
$\subspace{10134\\01204}{}$,
$\subspace{10141\\01232}{}$,
$\subspace{10102\\01433}{}$,
$\subspace{10114\\01430}{}$,
$\subspace{10121\\01432}{}$,
$\subspace{10133\\01434}{}$,
$\subspace{10140\\01431}{}$,
$\subspace{10200\\01302}{}$,
$\subspace{10212\\01322}{}$,
$\subspace{10224\\01342}{}$,
$\subspace{10231\\01312}{}$,
$\subspace{10243\\01332}{}$,
$\subspace{10101\\01223}{15}$,
$\subspace{10113\\01201}{}$,
$\subspace{10120\\01234}{}$,
$\subspace{10132\\01212}{}$,
$\subspace{10144\\01240}{}$,
$\subspace{10101\\01400}{}$,
$\subspace{10113\\01402}{}$,
$\subspace{10120\\01404}{}$,
$\subspace{10132\\01401}{}$,
$\subspace{10144\\01403}{}$,
$\subspace{10201\\01323}{}$,
$\subspace{10213\\01343}{}$,
$\subspace{10220\\01313}{}$,
$\subspace{10232\\01333}{}$,
$\subspace{10244\\01303}{}$,
$\subspace{10104\\01234}{15}$,
$\subspace{10111\\01212}{}$,
$\subspace{10123\\01240}{}$,
$\subspace{10130\\01223}{}$,
$\subspace{10142\\01201}{}$,
$\subspace{10101\\01403}{}$,
$\subspace{10113\\01400}{}$,
$\subspace{10120\\01402}{}$,
$\subspace{10132\\01404}{}$,
$\subspace{10144\\01401}{}$,
$\subspace{10203\\01303}{}$,
$\subspace{10210\\01323}{}$,
$\subspace{10222\\01343}{}$,
$\subspace{10234\\01313}{}$,
$\subspace{10241\\01333}{}$,
$\subspace{10103\\01244}{15}$,
$\subspace{10110\\01222}{}$,
$\subspace{10122\\01200}{}$,
$\subspace{10134\\01233}{}$,
$\subspace{10141\\01211}{}$,
$\subspace{10102\\01442}{}$,
$\subspace{10114\\01444}{}$,
$\subspace{10121\\01441}{}$,
$\subspace{10133\\01443}{}$,
$\subspace{10140\\01440}{}$,
$\subspace{10204\\01330}{}$,
$\subspace{10211\\01300}{}$,
$\subspace{10223\\01320}{}$,
$\subspace{10230\\01340}{}$,
$\subspace{10242\\01310}{}$,
$\subspace{10103\\01244}{15}$,
$\subspace{10110\\01222}{}$,
$\subspace{10122\\01200}{}$,
$\subspace{10134\\01233}{}$,
$\subspace{10141\\01211}{}$,
$\subspace{10102\\01442}{}$,
$\subspace{10114\\01444}{}$,
$\subspace{10121\\01441}{}$,
$\subspace{10133\\01443}{}$,
$\subspace{10140\\01440}{}$,
$\subspace{10204\\01330}{}$,
$\subspace{10211\\01300}{}$,
$\subspace{10223\\01320}{}$,
$\subspace{10230\\01340}{}$,
$\subspace{10242\\01310}{}$,
$\subspace{10103\\01244}{15}$,
$\subspace{10110\\01222}{}$,
$\subspace{10122\\01200}{}$,
$\subspace{10134\\01233}{}$,
$\subspace{10141\\01211}{}$,
$\subspace{10102\\01442}{}$,
$\subspace{10114\\01444}{}$,
$\subspace{10121\\01441}{}$,
$\subspace{10133\\01443}{}$,
$\subspace{10140\\01440}{}$,
$\subspace{10204\\01330}{}$,
$\subspace{10211\\01300}{}$,
$\subspace{10223\\01320}{}$,
$\subspace{10230\\01340}{}$,
$\subspace{10242\\01310}{}$,
$\subspace{10103\\01244}{15}$,
$\subspace{10110\\01222}{}$,
$\subspace{10122\\01200}{}$,
$\subspace{10134\\01233}{}$,
$\subspace{10141\\01211}{}$,
$\subspace{10102\\01442}{}$,
$\subspace{10114\\01444}{}$,
$\subspace{10121\\01441}{}$,
$\subspace{10133\\01443}{}$,
$\subspace{10140\\01440}{}$,
$\subspace{10204\\01330}{}$,
$\subspace{10211\\01300}{}$,
$\subspace{10223\\01320}{}$,
$\subspace{10230\\01340}{}$,
$\subspace{10242\\01310}{}$,
$\subspace{10103\\01244}{15}$,
$\subspace{10110\\01222}{}$,
$\subspace{10122\\01200}{}$,
$\subspace{10134\\01233}{}$,
$\subspace{10141\\01211}{}$,
$\subspace{10102\\01442}{}$,
$\subspace{10114\\01444}{}$,
$\subspace{10121\\01441}{}$,
$\subspace{10133\\01443}{}$,
$\subspace{10140\\01440}{}$,
$\subspace{10204\\01330}{}$,
$\subspace{10211\\01300}{}$,
$\subspace{10223\\01320}{}$,
$\subspace{10230\\01340}{}$,
$\subspace{10242\\01310}{}$,
$\subspace{10100\\01301}{15}$,
$\subspace{10112\\01341}{}$,
$\subspace{10124\\01331}{}$,
$\subspace{10131\\01321}{}$,
$\subspace{10143\\01311}{}$,
$\subspace{10303\\01202}{}$,
$\subspace{10310\\01213}{}$,
$\subspace{10322\\01224}{}$,
$\subspace{10334\\01230}{}$,
$\subspace{10341\\01241}{}$,
$\subspace{10303\\01414}{}$,
$\subspace{10310\\01413}{}$,
$\subspace{10322\\01412}{}$,
$\subspace{10334\\01411}{}$,
$\subspace{10341\\01410}{}$,
$\subspace{10101\\01304}{15}$,
$\subspace{10113\\01344}{}$,
$\subspace{10120\\01334}{}$,
$\subspace{10132\\01324}{}$,
$\subspace{10144\\01314}{}$,
$\subspace{10304\\01220}{}$,
$\subspace{10311\\01231}{}$,
$\subspace{10323\\01242}{}$,
$\subspace{10330\\01203}{}$,
$\subspace{10342\\01214}{}$,
$\subspace{10303\\01420}{}$,
$\subspace{10310\\01424}{}$,
$\subspace{10322\\01423}{}$,
$\subspace{10334\\01422}{}$,
$\subspace{10341\\01421}{}$,
$\subspace{10101\\01304}{15}$,
$\subspace{10113\\01344}{}$,
$\subspace{10120\\01334}{}$,
$\subspace{10132\\01324}{}$,
$\subspace{10144\\01314}{}$,
$\subspace{10304\\01220}{}$,
$\subspace{10311\\01231}{}$,
$\subspace{10323\\01242}{}$,
$\subspace{10330\\01203}{}$,
$\subspace{10342\\01214}{}$,
$\subspace{10303\\01420}{}$,
$\subspace{10310\\01424}{}$,
$\subspace{10322\\01423}{}$,
$\subspace{10334\\01422}{}$,
$\subspace{10341\\01421}{}$,
$\subspace{10104\\01303}{15}$,
$\subspace{10111\\01343}{}$,
$\subspace{10123\\01333}{}$,
$\subspace{10130\\01323}{}$,
$\subspace{10142\\01313}{}$,
$\subspace{10302\\01234}{}$,
$\subspace{10314\\01240}{}$,
$\subspace{10321\\01201}{}$,
$\subspace{10333\\01212}{}$,
$\subspace{10340\\01223}{}$,
$\subspace{10303\\01403}{}$,
$\subspace{10310\\01402}{}$,
$\subspace{10322\\01401}{}$,
$\subspace{10334\\01400}{}$,
$\subspace{10341\\01404}{}$,
$\subspace{10104\\01303}{15}$,
$\subspace{10111\\01343}{}$,
$\subspace{10123\\01333}{}$,
$\subspace{10130\\01323}{}$,
$\subspace{10142\\01313}{}$,
$\subspace{10302\\01234}{}$,
$\subspace{10314\\01240}{}$,
$\subspace{10321\\01201}{}$,
$\subspace{10333\\01212}{}$,
$\subspace{10340\\01223}{}$,
$\subspace{10303\\01403}{}$,
$\subspace{10310\\01402}{}$,
$\subspace{10322\\01401}{}$,
$\subspace{10334\\01400}{}$,
$\subspace{10341\\01404}{}$,
$\subspace{10100\\01314}{15}$,
$\subspace{10112\\01304}{}$,
$\subspace{10124\\01344}{}$,
$\subspace{10131\\01334}{}$,
$\subspace{10143\\01324}{}$,
$\subspace{10303\\01220}{}$,
$\subspace{10310\\01231}{}$,
$\subspace{10322\\01242}{}$,
$\subspace{10334\\01203}{}$,
$\subspace{10341\\01214}{}$,
$\subspace{10300\\01424}{}$,
$\subspace{10312\\01423}{}$,
$\subspace{10324\\01422}{}$,
$\subspace{10331\\01421}{}$,
$\subspace{10343\\01420}{}$,
$\subspace{10101\\01311}{15}$,
$\subspace{10113\\01301}{}$,
$\subspace{10120\\01341}{}$,
$\subspace{10132\\01331}{}$,
$\subspace{10144\\01321}{}$,
$\subspace{10303\\01230}{}$,
$\subspace{10310\\01241}{}$,
$\subspace{10322\\01202}{}$,
$\subspace{10334\\01213}{}$,
$\subspace{10341\\01224}{}$,
$\subspace{10302\\01411}{}$,
$\subspace{10314\\01410}{}$,
$\subspace{10321\\01414}{}$,
$\subspace{10333\\01413}{}$,
$\subspace{10340\\01412}{}$,
$\subspace{10101\\01311}{15}$,
$\subspace{10113\\01301}{}$,
$\subspace{10120\\01341}{}$,
$\subspace{10132\\01331}{}$,
$\subspace{10144\\01321}{}$,
$\subspace{10303\\01230}{}$,
$\subspace{10310\\01241}{}$,
$\subspace{10322\\01202}{}$,
$\subspace{10334\\01213}{}$,
$\subspace{10341\\01224}{}$,
$\subspace{10302\\01411}{}$,
$\subspace{10314\\01410}{}$,
$\subspace{10321\\01414}{}$,
$\subspace{10333\\01413}{}$,
$\subspace{10340\\01412}{}$,
$\subspace{10101\\01314}{15}$,
$\subspace{10113\\01304}{}$,
$\subspace{10120\\01344}{}$,
$\subspace{10132\\01334}{}$,
$\subspace{10144\\01324}{}$,
$\subspace{10301\\01214}{}$,
$\subspace{10313\\01220}{}$,
$\subspace{10320\\01231}{}$,
$\subspace{10332\\01242}{}$,
$\subspace{10344\\01203}{}$,
$\subspace{10301\\01423}{}$,
$\subspace{10313\\01422}{}$,
$\subspace{10320\\01421}{}$,
$\subspace{10332\\01420}{}$,
$\subspace{10344\\01424}{}$,
$\subspace{10101\\01314}{15}$,
$\subspace{10113\\01304}{}$,
$\subspace{10120\\01344}{}$,
$\subspace{10132\\01334}{}$,
$\subspace{10144\\01324}{}$,
$\subspace{10301\\01214}{}$,
$\subspace{10313\\01220}{}$,
$\subspace{10320\\01231}{}$,
$\subspace{10332\\01242}{}$,
$\subspace{10344\\01203}{}$,
$\subspace{10301\\01423}{}$,
$\subspace{10313\\01422}{}$,
$\subspace{10320\\01421}{}$,
$\subspace{10332\\01420}{}$,
$\subspace{10344\\01424}{}$,
$\subspace{10201\\01004}{15}$,
$\subspace{10213\\01031}{}$,
$\subspace{10220\\01013}{}$,
$\subspace{10232\\01040}{}$,
$\subspace{10244\\01022}{}$,
$\subspace{10304\\01142}{}$,
$\subspace{10311\\01134}{}$,
$\subspace{10323\\01121}{}$,
$\subspace{10330\\01113}{}$,
$\subspace{10342\\01100}{}$,
$\subspace{13004\\00104}{}$,
$\subspace{13013\\00123}{}$,
$\subspace{13022\\00142}{}$,
$\subspace{13031\\00111}{}$,
$\subspace{13040\\00130}{}$,
$\subspace{10201\\01004}{15}$,
$\subspace{10213\\01031}{}$,
$\subspace{10220\\01013}{}$,
$\subspace{10232\\01040}{}$,
$\subspace{10244\\01022}{}$,
$\subspace{10304\\01142}{}$,
$\subspace{10311\\01134}{}$,
$\subspace{10323\\01121}{}$,
$\subspace{10330\\01113}{}$,
$\subspace{10342\\01100}{}$,
$\subspace{13004\\00104}{}$,
$\subspace{13013\\00123}{}$,
$\subspace{13022\\00142}{}$,
$\subspace{13031\\00111}{}$,
$\subspace{13040\\00130}{}$,
$\subspace{10200\\01014}{15}$,
$\subspace{10212\\01041}{}$,
$\subspace{10224\\01023}{}$,
$\subspace{10231\\01000}{}$,
$\subspace{10243\\01032}{}$,
$\subspace{10303\\01124}{}$,
$\subspace{10310\\01111}{}$,
$\subspace{10322\\01103}{}$,
$\subspace{10334\\01140}{}$,
$\subspace{10341\\01132}{}$,
$\subspace{13003\\00131}{}$,
$\subspace{13012\\00100}{}$,
$\subspace{13021\\00124}{}$,
$\subspace{13030\\00143}{}$,
$\subspace{13044\\00112}{}$,
$\subspace{10203\\01011}{15}$,
$\subspace{10210\\01043}{}$,
$\subspace{10222\\01020}{}$,
$\subspace{10234\\01002}{}$,
$\subspace{10241\\01034}{}$,
$\subspace{10304\\01120}{}$,
$\subspace{10311\\01112}{}$,
$\subspace{10323\\01104}{}$,
$\subspace{10330\\01141}{}$,
$\subspace{10342\\01133}{}$,
$\subspace{13004\\00102}{}$,
$\subspace{13013\\00121}{}$,
$\subspace{13022\\00140}{}$,
$\subspace{13031\\00114}{}$,
$\subspace{13040\\00133}{}$,
$\subspace{10203\\01011}{15}$,
$\subspace{10210\\01043}{}$,
$\subspace{10222\\01020}{}$,
$\subspace{10234\\01002}{}$,
$\subspace{10241\\01034}{}$,
$\subspace{10304\\01120}{}$,
$\subspace{10311\\01112}{}$,
$\subspace{10323\\01104}{}$,
$\subspace{10330\\01141}{}$,
$\subspace{10342\\01133}{}$,
$\subspace{13004\\00102}{}$,
$\subspace{13013\\00121}{}$,
$\subspace{13022\\00140}{}$,
$\subspace{13031\\00114}{}$,
$\subspace{13040\\00133}{}$,
$\subspace{10200\\01032}{15}$,
$\subspace{10212\\01014}{}$,
$\subspace{10224\\01041}{}$,
$\subspace{10231\\01023}{}$,
$\subspace{10243\\01000}{}$,
$\subspace{10300\\01132}{}$,
$\subspace{10312\\01124}{}$,
$\subspace{10324\\01111}{}$,
$\subspace{10331\\01103}{}$,
$\subspace{10343\\01140}{}$,
$\subspace{13000\\00124}{}$,
$\subspace{13014\\00143}{}$,
$\subspace{13023\\00112}{}$,
$\subspace{13032\\00131}{}$,
$\subspace{13041\\00100}{}$,
$\subspace{10200\\01033}{15}$,
$\subspace{10212\\01010}{}$,
$\subspace{10224\\01042}{}$,
$\subspace{10231\\01024}{}$,
$\subspace{10243\\01001}{}$,
$\subspace{10301\\01143}{}$,
$\subspace{10313\\01130}{}$,
$\subspace{10320\\01122}{}$,
$\subspace{10332\\01114}{}$,
$\subspace{10344\\01101}{}$,
$\subspace{13001\\00101}{}$,
$\subspace{13010\\00120}{}$,
$\subspace{13024\\00144}{}$,
$\subspace{13033\\00113}{}$,
$\subspace{13042\\00132}{}$,
$\subspace{10200\\01033}{15}$,
$\subspace{10212\\01010}{}$,
$\subspace{10224\\01042}{}$,
$\subspace{10231\\01024}{}$,
$\subspace{10243\\01001}{}$,
$\subspace{10301\\01143}{}$,
$\subspace{10313\\01130}{}$,
$\subspace{10320\\01122}{}$,
$\subspace{10332\\01114}{}$,
$\subspace{10344\\01101}{}$,
$\subspace{13001\\00101}{}$,
$\subspace{13010\\00120}{}$,
$\subspace{13024\\00144}{}$,
$\subspace{13033\\00113}{}$,
$\subspace{13042\\00132}{}$,
$\subspace{10202\\01031}{15}$,
$\subspace{10214\\01013}{}$,
$\subspace{10221\\01040}{}$,
$\subspace{10233\\01022}{}$,
$\subspace{10240\\01004}{}$,
$\subspace{10300\\01142}{}$,
$\subspace{10312\\01134}{}$,
$\subspace{10324\\01121}{}$,
$\subspace{10331\\01113}{}$,
$\subspace{10343\\01100}{}$,
$\subspace{13000\\00130}{}$,
$\subspace{13014\\00104}{}$,
$\subspace{13023\\00123}{}$,
$\subspace{13032\\00142}{}$,
$\subspace{13041\\00111}{}$,
$\subspace{10202\\01031}{15}$,
$\subspace{10214\\01013}{}$,
$\subspace{10221\\01040}{}$,
$\subspace{10233\\01022}{}$,
$\subspace{10240\\01004}{}$,
$\subspace{10300\\01142}{}$,
$\subspace{10312\\01134}{}$,
$\subspace{10324\\01121}{}$,
$\subspace{10331\\01113}{}$,
$\subspace{10343\\01100}{}$,
$\subspace{13000\\00130}{}$,
$\subspace{13014\\00104}{}$,
$\subspace{13023\\00123}{}$,
$\subspace{13032\\00142}{}$,
$\subspace{13041\\00111}{}$,
$\subspace{10200\\01100}{15}$,
$\subspace{10212\\01113}{}$,
$\subspace{10224\\01121}{}$,
$\subspace{10231\\01134}{}$,
$\subspace{10243\\01142}{}$,
$\subspace{10300\\01040}{}$,
$\subspace{10312\\01013}{}$,
$\subspace{10324\\01031}{}$,
$\subspace{10331\\01004}{}$,
$\subspace{10343\\01022}{}$,
$\subspace{12000\\00111}{}$,
$\subspace{12014\\00142}{}$,
$\subspace{12023\\00123}{}$,
$\subspace{12032\\00104}{}$,
$\subspace{12041\\00130}{}$,
$\subspace{10200\\01100}{15}$,
$\subspace{10212\\01113}{}$,
$\subspace{10224\\01121}{}$,
$\subspace{10231\\01134}{}$,
$\subspace{10243\\01142}{}$,
$\subspace{10300\\01040}{}$,
$\subspace{10312\\01013}{}$,
$\subspace{10324\\01031}{}$,
$\subspace{10331\\01004}{}$,
$\subspace{10343\\01022}{}$,
$\subspace{12000\\00111}{}$,
$\subspace{12014\\00142}{}$,
$\subspace{12023\\00123}{}$,
$\subspace{12032\\00104}{}$,
$\subspace{12041\\00130}{}$,
$\subspace{10204\\01104}{15}$,
$\subspace{10211\\01112}{}$,
$\subspace{10223\\01120}{}$,
$\subspace{10230\\01133}{}$,
$\subspace{10242\\01141}{}$,
$\subspace{10301\\01002}{}$,
$\subspace{10313\\01020}{}$,
$\subspace{10320\\01043}{}$,
$\subspace{10332\\01011}{}$,
$\subspace{10344\\01034}{}$,
$\subspace{12004\\00102}{}$,
$\subspace{12013\\00133}{}$,
$\subspace{12022\\00114}{}$,
$\subspace{12031\\00140}{}$,
$\subspace{12040\\00121}{}$,
$\subspace{10201\\01111}{15}$,
$\subspace{10213\\01124}{}$,
$\subspace{10220\\01132}{}$,
$\subspace{10232\\01140}{}$,
$\subspace{10244\\01103}{}$,
$\subspace{10302\\01032}{}$,
$\subspace{10314\\01000}{}$,
$\subspace{10321\\01023}{}$,
$\subspace{10333\\01041}{}$,
$\subspace{10340\\01014}{}$,
$\subspace{12001\\00131}{}$,
$\subspace{12010\\00112}{}$,
$\subspace{12024\\00143}{}$,
$\subspace{12033\\00124}{}$,
$\subspace{12042\\00100}{}$,
$\subspace{10201\\01111}{15}$,
$\subspace{10213\\01124}{}$,
$\subspace{10220\\01132}{}$,
$\subspace{10232\\01140}{}$,
$\subspace{10244\\01103}{}$,
$\subspace{10302\\01032}{}$,
$\subspace{10314\\01000}{}$,
$\subspace{10321\\01023}{}$,
$\subspace{10333\\01041}{}$,
$\subspace{10340\\01014}{}$,
$\subspace{12001\\00131}{}$,
$\subspace{12010\\00112}{}$,
$\subspace{12024\\00143}{}$,
$\subspace{12033\\00124}{}$,
$\subspace{12042\\00100}{}$,
$\subspace{10201\\01120}{15}$,
$\subspace{10213\\01133}{}$,
$\subspace{10220\\01141}{}$,
$\subspace{10232\\01104}{}$,
$\subspace{10244\\01112}{}$,
$\subspace{10302\\01011}{}$,
$\subspace{10314\\01034}{}$,
$\subspace{10321\\01002}{}$,
$\subspace{10333\\01020}{}$,
$\subspace{10340\\01043}{}$,
$\subspace{12001\\00102}{}$,
$\subspace{12010\\00133}{}$,
$\subspace{12024\\00114}{}$,
$\subspace{12033\\00140}{}$,
$\subspace{12042\\00121}{}$,
$\subspace{10202\\01124}{15}$,
$\subspace{10214\\01132}{}$,
$\subspace{10221\\01140}{}$,
$\subspace{10233\\01103}{}$,
$\subspace{10240\\01111}{}$,
$\subspace{10300\\01014}{}$,
$\subspace{10312\\01032}{}$,
$\subspace{10324\\01000}{}$,
$\subspace{10331\\01023}{}$,
$\subspace{10343\\01041}{}$,
$\subspace{12002\\00131}{}$,
$\subspace{12011\\00112}{}$,
$\subspace{12020\\00143}{}$,
$\subspace{12034\\00124}{}$,
$\subspace{12043\\00100}{}$,
$\subspace{10200\\01141}{15}$,
$\subspace{10212\\01104}{}$,
$\subspace{10224\\01112}{}$,
$\subspace{10231\\01120}{}$,
$\subspace{10243\\01133}{}$,
$\subspace{10300\\01011}{}$,
$\subspace{10312\\01034}{}$,
$\subspace{10324\\01002}{}$,
$\subspace{10331\\01020}{}$,
$\subspace{10343\\01043}{}$,
$\subspace{12000\\00140}{}$,
$\subspace{12014\\00121}{}$,
$\subspace{12023\\00102}{}$,
$\subspace{12032\\00133}{}$,
$\subspace{12041\\00114}{}$,
$\subspace{10203\\01143}{15}$,
$\subspace{10210\\01101}{}$,
$\subspace{10222\\01114}{}$,
$\subspace{10234\\01122}{}$,
$\subspace{10241\\01130}{}$,
$\subspace{10304\\01010}{}$,
$\subspace{10311\\01033}{}$,
$\subspace{10323\\01001}{}$,
$\subspace{10330\\01024}{}$,
$\subspace{10342\\01042}{}$,
$\subspace{12003\\00132}{}$,
$\subspace{12012\\00113}{}$,
$\subspace{12021\\00144}{}$,
$\subspace{12030\\00120}{}$,
$\subspace{12044\\00101}{}$,
$\subspace{10200\\01204}{15}$,
$\subspace{10212\\01243}{}$,
$\subspace{10224\\01232}{}$,
$\subspace{10231\\01221}{}$,
$\subspace{10243\\01210}{}$,
$\subspace{10204\\01430}{}$,
$\subspace{10211\\01431}{}$,
$\subspace{10223\\01432}{}$,
$\subspace{10230\\01433}{}$,
$\subspace{10242\\01434}{}$,
$\subspace{10401\\01322}{}$,
$\subspace{10413\\01332}{}$,
$\subspace{10420\\01342}{}$,
$\subspace{10432\\01302}{}$,
$\subspace{10444\\01312}{}$,
$\subspace{10200\\01204}{15}$,
$\subspace{10212\\01243}{}$,
$\subspace{10224\\01232}{}$,
$\subspace{10231\\01221}{}$,
$\subspace{10243\\01210}{}$,
$\subspace{10204\\01430}{}$,
$\subspace{10211\\01431}{}$,
$\subspace{10223\\01432}{}$,
$\subspace{10230\\01433}{}$,
$\subspace{10242\\01434}{}$,
$\subspace{10401\\01322}{}$,
$\subspace{10413\\01332}{}$,
$\subspace{10420\\01342}{}$,
$\subspace{10432\\01302}{}$,
$\subspace{10444\\01312}{}$,
$\subspace{10202\\01210}{15}$,
$\subspace{10214\\01204}{}$,
$\subspace{10221\\01243}{}$,
$\subspace{10233\\01232}{}$,
$\subspace{10240\\01221}{}$,
$\subspace{10202\\01432}{}$,
$\subspace{10214\\01433}{}$,
$\subspace{10221\\01434}{}$,
$\subspace{10233\\01430}{}$,
$\subspace{10240\\01431}{}$,
$\subspace{10401\\01312}{}$,
$\subspace{10413\\01322}{}$,
$\subspace{10420\\01332}{}$,
$\subspace{10432\\01342}{}$,
$\subspace{10444\\01302}{}$,
$\subspace{10202\\01214}{15}$,
$\subspace{10214\\01203}{}$,
$\subspace{10221\\01242}{}$,
$\subspace{10233\\01231}{}$,
$\subspace{10240\\01220}{}$,
$\subspace{10203\\01420}{}$,
$\subspace{10210\\01421}{}$,
$\subspace{10222\\01422}{}$,
$\subspace{10234\\01423}{}$,
$\subspace{10241\\01424}{}$,
$\subspace{10402\\01344}{}$,
$\subspace{10414\\01304}{}$,
$\subspace{10421\\01314}{}$,
$\subspace{10433\\01324}{}$,
$\subspace{10440\\01334}{}$,
$\subspace{10203\\01213}{15}$,
$\subspace{10210\\01202}{}$,
$\subspace{10222\\01241}{}$,
$\subspace{10234\\01230}{}$,
$\subspace{10241\\01224}{}$,
$\subspace{10202\\01412}{}$,
$\subspace{10214\\01413}{}$,
$\subspace{10221\\01414}{}$,
$\subspace{10233\\01410}{}$,
$\subspace{10240\\01411}{}$,
$\subspace{10404\\01331}{}$,
$\subspace{10411\\01341}{}$,
$\subspace{10423\\01301}{}$,
$\subspace{10430\\01311}{}$,
$\subspace{10442\\01321}{}$,
$\subspace{10204\\01211}{15}$,
$\subspace{10211\\01200}{}$,
$\subspace{10223\\01244}{}$,
$\subspace{10230\\01233}{}$,
$\subspace{10242\\01222}{}$,
$\subspace{10202\\01442}{}$,
$\subspace{10214\\01443}{}$,
$\subspace{10221\\01444}{}$,
$\subspace{10233\\01440}{}$,
$\subspace{10240\\01441}{}$,
$\subspace{10402\\01300}{}$,
$\subspace{10414\\01310}{}$,
$\subspace{10421\\01320}{}$,
$\subspace{10433\\01330}{}$,
$\subspace{10440\\01340}{}$,
$\subspace{10202\\01232}{15}$,
$\subspace{10214\\01221}{}$,
$\subspace{10221\\01210}{}$,
$\subspace{10233\\01204}{}$,
$\subspace{10240\\01243}{}$,
$\subspace{10201\\01430}{}$,
$\subspace{10213\\01431}{}$,
$\subspace{10220\\01432}{}$,
$\subspace{10232\\01433}{}$,
$\subspace{10244\\01434}{}$,
$\subspace{10402\\01302}{}$,
$\subspace{10414\\01312}{}$,
$\subspace{10421\\01322}{}$,
$\subspace{10433\\01332}{}$,
$\subspace{10440\\01342}{}$,
$\subspace{10203\\01243}{15}$,
$\subspace{10210\\01232}{}$,
$\subspace{10222\\01221}{}$,
$\subspace{10234\\01210}{}$,
$\subspace{10241\\01204}{}$,
$\subspace{10201\\01433}{}$,
$\subspace{10213\\01434}{}$,
$\subspace{10220\\01430}{}$,
$\subspace{10232\\01431}{}$,
$\subspace{10244\\01432}{}$,
$\subspace{10401\\01332}{}$,
$\subspace{10413\\01342}{}$,
$\subspace{10420\\01302}{}$,
$\subspace{10432\\01312}{}$,
$\subspace{10444\\01322}{}$,
$\subspace{10302\\01312}{15}$,
$\subspace{10314\\01342}{}$,
$\subspace{10321\\01322}{}$,
$\subspace{10333\\01302}{}$,
$\subspace{10340\\01332}{}$,
$\subspace{10404\\01210}{}$,
$\subspace{10411\\01232}{}$,
$\subspace{10423\\01204}{}$,
$\subspace{10430\\01221}{}$,
$\subspace{10442\\01243}{}$,
$\subspace{10404\\01432}{}$,
$\subspace{10411\\01430}{}$,
$\subspace{10423\\01433}{}$,
$\subspace{10430\\01431}{}$,
$\subspace{10442\\01434}{}$,
$\subspace{10302\\01312}{15}$,
$\subspace{10314\\01342}{}$,
$\subspace{10321\\01322}{}$,
$\subspace{10333\\01302}{}$,
$\subspace{10340\\01332}{}$,
$\subspace{10404\\01210}{}$,
$\subspace{10411\\01232}{}$,
$\subspace{10423\\01204}{}$,
$\subspace{10430\\01221}{}$,
$\subspace{10442\\01243}{}$,
$\subspace{10404\\01432}{}$,
$\subspace{10411\\01430}{}$,
$\subspace{10423\\01433}{}$,
$\subspace{10430\\01431}{}$,
$\subspace{10442\\01434}{}$,
$\subspace{10304\\01314}{15}$,
$\subspace{10311\\01344}{}$,
$\subspace{10323\\01324}{}$,
$\subspace{10330\\01304}{}$,
$\subspace{10342\\01334}{}$,
$\subspace{10401\\01242}{}$,
$\subspace{10413\\01214}{}$,
$\subspace{10420\\01231}{}$,
$\subspace{10432\\01203}{}$,
$\subspace{10444\\01220}{}$,
$\subspace{10404\\01421}{}$,
$\subspace{10411\\01424}{}$,
$\subspace{10423\\01422}{}$,
$\subspace{10430\\01420}{}$,
$\subspace{10442\\01423}{}$,
$\subspace{10304\\01314}{15}$,
$\subspace{10311\\01344}{}$,
$\subspace{10323\\01324}{}$,
$\subspace{10330\\01304}{}$,
$\subspace{10342\\01334}{}$,
$\subspace{10401\\01242}{}$,
$\subspace{10413\\01214}{}$,
$\subspace{10420\\01231}{}$,
$\subspace{10432\\01203}{}$,
$\subspace{10444\\01220}{}$,
$\subspace{10404\\01421}{}$,
$\subspace{10411\\01424}{}$,
$\subspace{10423\\01422}{}$,
$\subspace{10430\\01420}{}$,
$\subspace{10442\\01423}{}$,
$\subspace{10300\\01320}{15}$,
$\subspace{10312\\01300}{}$,
$\subspace{10324\\01330}{}$,
$\subspace{10331\\01310}{}$,
$\subspace{10343\\01340}{}$,
$\subspace{10403\\01222}{}$,
$\subspace{10410\\01244}{}$,
$\subspace{10422\\01211}{}$,
$\subspace{10434\\01233}{}$,
$\subspace{10441\\01200}{}$,
$\subspace{10403\\01441}{}$,
$\subspace{10410\\01444}{}$,
$\subspace{10422\\01442}{}$,
$\subspace{10434\\01440}{}$,
$\subspace{10441\\01443}{}$,
$\subspace{10300\\01320}{15}$,
$\subspace{10312\\01300}{}$,
$\subspace{10324\\01330}{}$,
$\subspace{10331\\01310}{}$,
$\subspace{10343\\01340}{}$,
$\subspace{10403\\01222}{}$,
$\subspace{10410\\01244}{}$,
$\subspace{10422\\01211}{}$,
$\subspace{10434\\01233}{}$,
$\subspace{10441\\01200}{}$,
$\subspace{10403\\01441}{}$,
$\subspace{10410\\01444}{}$,
$\subspace{10422\\01442}{}$,
$\subspace{10434\\01440}{}$,
$\subspace{10441\\01443}{}$,
$\subspace{10303\\01321}{15}$,
$\subspace{10310\\01301}{}$,
$\subspace{10322\\01331}{}$,
$\subspace{10334\\01311}{}$,
$\subspace{10341\\01341}{}$,
$\subspace{10400\\01224}{}$,
$\subspace{10412\\01241}{}$,
$\subspace{10424\\01213}{}$,
$\subspace{10431\\01230}{}$,
$\subspace{10443\\01202}{}$,
$\subspace{10400\\01414}{}$,
$\subspace{10412\\01412}{}$,
$\subspace{10424\\01410}{}$,
$\subspace{10431\\01413}{}$,
$\subspace{10443\\01411}{}$,
$\subspace{11302\\00010}{3}$,
$\subspace{12310\\00001}{}$,
$\subspace{12403\\00011}{}$,
$\subspace{13103\\00011}{3}$,
$\subspace{13200\\00001}{}$,
$\subspace{14202\\00010}{}$.

\smallskip

$n_5(5,2;58)\ge 1458$,$\left[\!\begin{smallmatrix}10000\\00100\\04410\\00001\\00044\end{smallmatrix}\!\right]$:
$\subspace{00103\\00012}{3}$,
$\subspace{01003\\00014}{}$,
$\subspace{01102\\00013}{}$,
$\subspace{01003\\00113}{15}$,
$\subspace{01000\\00122}{}$,
$\subspace{01001\\00143}{}$,
$\subspace{01010\\00100}{}$,
$\subspace{01012\\00120}{}$,
$\subspace{01014\\00134}{}$,
$\subspace{01024\\00112}{}$,
$\subspace{01021\\00132}{}$,
$\subspace{01023\\00141}{}$,
$\subspace{01032\\00103}{}$,
$\subspace{01033\\00124}{}$,
$\subspace{01030\\00144}{}$,
$\subspace{01044\\00101}{}$,
$\subspace{01041\\00110}{}$,
$\subspace{01042\\00131}{}$,
$\subspace{01002\\00123}{15}$,
$\subspace{01004\\00121}{}$,
$\subspace{01002\\00140}{}$,
$\subspace{01011\\00102}{}$,
$\subspace{01011\\00130}{}$,
$\subspace{01013\\00133}{}$,
$\subspace{01020\\00114}{}$,
$\subspace{01020\\00142}{}$,
$\subspace{01022\\00140}{}$,
$\subspace{01031\\00102}{}$,
$\subspace{01034\\00104}{}$,
$\subspace{01034\\00121}{}$,
$\subspace{01040\\00114}{}$,
$\subspace{01043\\00111}{}$,
$\subspace{01043\\00133}{}$,
$\subspace{10000\\00001}{3}$,
$\subspace{10000\\00010}{}$,
$\subspace{10000\\00011}{}$,
$\subspace{10000\\00001}{3}$,
$\subspace{10000\\00010}{}$,
$\subspace{10000\\00011}{}$,
$\subspace{10000\\00012}{3}$,
$\subspace{10000\\00013}{}$,
$\subspace{10000\\00014}{}$,
$\subspace{10000\\00012}{3}$,
$\subspace{10000\\00013}{}$,
$\subspace{10000\\00014}{}$,
$\subspace{10012\\00102}{15}$,
$\subspace{10012\\00114}{}$,
$\subspace{10012\\00121}{}$,
$\subspace{10012\\00133}{}$,
$\subspace{10012\\00140}{}$,
$\subspace{10014\\01002}{}$,
$\subspace{10014\\01011}{}$,
$\subspace{10014\\01020}{}$,
$\subspace{10014\\01034}{}$,
$\subspace{10014\\01043}{}$,
$\subspace{10034\\01104}{}$,
$\subspace{10034\\01112}{}$,
$\subspace{10034\\01120}{}$,
$\subspace{10034\\01133}{}$,
$\subspace{10034\\01141}{}$,
$\subspace{10020\\00101}{15}$,
$\subspace{10020\\00113}{}$,
$\subspace{10020\\00120}{}$,
$\subspace{10020\\00132}{}$,
$\subspace{10020\\00144}{}$,
$\subspace{10033\\01001}{}$,
$\subspace{10033\\01010}{}$,
$\subspace{10033\\01024}{}$,
$\subspace{10033\\01033}{}$,
$\subspace{10033\\01042}{}$,
$\subspace{10002\\01101}{}$,
$\subspace{10002\\01114}{}$,
$\subspace{10002\\01122}{}$,
$\subspace{10002\\01130}{}$,
$\subspace{10002\\01143}{}$,
$\subspace{10023\\00102}{15}$,
$\subspace{10023\\00114}{}$,
$\subspace{10023\\00121}{}$,
$\subspace{10023\\00133}{}$,
$\subspace{10023\\00140}{}$,
$\subspace{10013\\01002}{}$,
$\subspace{10013\\01011}{}$,
$\subspace{10013\\01020}{}$,
$\subspace{10013\\01034}{}$,
$\subspace{10013\\01043}{}$,
$\subspace{10024\\01104}{}$,
$\subspace{10024\\01112}{}$,
$\subspace{10024\\01120}{}$,
$\subspace{10024\\01133}{}$,
$\subspace{10024\\01141}{}$,
$\subspace{10031\\00100}{15}$,
$\subspace{10031\\00112}{}$,
$\subspace{10031\\00124}{}$,
$\subspace{10031\\00131}{}$,
$\subspace{10031\\00143}{}$,
$\subspace{10032\\01000}{}$,
$\subspace{10032\\01014}{}$,
$\subspace{10032\\01023}{}$,
$\subspace{10032\\01032}{}$,
$\subspace{10032\\01041}{}$,
$\subspace{10042\\01103}{}$,
$\subspace{10042\\01111}{}$,
$\subspace{10042\\01124}{}$,
$\subspace{10042\\01132}{}$,
$\subspace{10042\\01140}{}$,
$\subspace{10032\\00103}{15}$,
$\subspace{10032\\00110}{}$,
$\subspace{10032\\00122}{}$,
$\subspace{10032\\00134}{}$,
$\subspace{10032\\00141}{}$,
$\subspace{10042\\01003}{}$,
$\subspace{10042\\01012}{}$,
$\subspace{10042\\01021}{}$,
$\subspace{10042\\01030}{}$,
$\subspace{10042\\01044}{}$,
$\subspace{10031\\01102}{}$,
$\subspace{10031\\01110}{}$,
$\subspace{10031\\01123}{}$,
$\subspace{10031\\01131}{}$,
$\subspace{10031\\01144}{}$,
$\subspace{10034\\00100}{15}$,
$\subspace{10034\\00112}{}$,
$\subspace{10034\\00124}{}$,
$\subspace{10034\\00131}{}$,
$\subspace{10034\\00143}{}$,
$\subspace{10012\\01000}{}$,
$\subspace{10012\\01014}{}$,
$\subspace{10012\\01023}{}$,
$\subspace{10012\\01032}{}$,
$\subspace{10012\\01041}{}$,
$\subspace{10014\\01103}{}$,
$\subspace{10014\\01111}{}$,
$\subspace{10014\\01124}{}$,
$\subspace{10014\\01132}{}$,
$\subspace{10014\\01140}{}$,
$\subspace{10041\\00102}{15}$,
$\subspace{10041\\00114}{}$,
$\subspace{10041\\00121}{}$,
$\subspace{10041\\00133}{}$,
$\subspace{10041\\00140}{}$,
$\subspace{10021\\01002}{}$,
$\subspace{10021\\01011}{}$,
$\subspace{10021\\01020}{}$,
$\subspace{10021\\01034}{}$,
$\subspace{10021\\01043}{}$,
$\subspace{10043\\01104}{}$,
$\subspace{10043\\01112}{}$,
$\subspace{10043\\01120}{}$,
$\subspace{10043\\01133}{}$,
$\subspace{10043\\01141}{}$,
$\subspace{10043\\00103}{15}$,
$\subspace{10043\\00110}{}$,
$\subspace{10043\\00122}{}$,
$\subspace{10043\\00134}{}$,
$\subspace{10043\\00141}{}$,
$\subspace{10041\\01003}{}$,
$\subspace{10041\\01012}{}$,
$\subspace{10041\\01021}{}$,
$\subspace{10041\\01030}{}$,
$\subspace{10041\\01044}{}$,
$\subspace{10021\\01102}{}$,
$\subspace{10021\\01110}{}$,
$\subspace{10021\\01123}{}$,
$\subspace{10021\\01131}{}$,
$\subspace{10021\\01144}{}$,
$\subspace{10103\\00012}{3}$,
$\subspace{11003\\00014}{}$,
$\subspace{14403\\00013}{}$,
$\subspace{10200\\00012}{3}$,
$\subspace{12000\\00014}{}$,
$\subspace{13304\\00013}{}$,
$\subspace{10300\\00012}{3}$,
$\subspace{12201\\00013}{}$,
$\subspace{13000\\00014}{}$,
$\subspace{10300\\00012}{3}$,
$\subspace{12201\\00013}{}$,
$\subspace{13000\\00014}{}$,
$\subspace{10301\\00012}{3}$,
$\subspace{12203\\00013}{}$,
$\subspace{13001\\00014}{}$,
$\subspace{10001\\01202}{15}$,
$\subspace{10001\\01213}{}$,
$\subspace{10001\\01224}{}$,
$\subspace{10001\\01230}{}$,
$\subspace{10001\\01241}{}$,
$\subspace{10044\\01301}{}$,
$\subspace{10044\\01311}{}$,
$\subspace{10044\\01321}{}$,
$\subspace{10044\\01331}{}$,
$\subspace{10044\\01341}{}$,
$\subspace{10010\\01410}{}$,
$\subspace{10010\\01411}{}$,
$\subspace{10010\\01412}{}$,
$\subspace{10010\\01413}{}$,
$\subspace{10010\\01414}{}$,
$\subspace{10001\\01202}{15}$,
$\subspace{10001\\01213}{}$,
$\subspace{10001\\01224}{}$,
$\subspace{10001\\01230}{}$,
$\subspace{10001\\01241}{}$,
$\subspace{10044\\01301}{}$,
$\subspace{10044\\01311}{}$,
$\subspace{10044\\01321}{}$,
$\subspace{10044\\01331}{}$,
$\subspace{10044\\01341}{}$,
$\subspace{10010\\01410}{}$,
$\subspace{10010\\01411}{}$,
$\subspace{10010\\01412}{}$,
$\subspace{10010\\01413}{}$,
$\subspace{10010\\01414}{}$,
$\subspace{10003\\01201}{15}$,
$\subspace{10003\\01212}{}$,
$\subspace{10003\\01223}{}$,
$\subspace{10003\\01234}{}$,
$\subspace{10003\\01240}{}$,
$\subspace{10022\\01303}{}$,
$\subspace{10022\\01313}{}$,
$\subspace{10022\\01323}{}$,
$\subspace{10022\\01333}{}$,
$\subspace{10022\\01343}{}$,
$\subspace{10030\\01400}{}$,
$\subspace{10030\\01401}{}$,
$\subspace{10030\\01402}{}$,
$\subspace{10030\\01403}{}$,
$\subspace{10030\\01404}{}$,
$\subspace{10003\\01201}{15}$,
$\subspace{10003\\01212}{}$,
$\subspace{10003\\01223}{}$,
$\subspace{10003\\01234}{}$,
$\subspace{10003\\01240}{}$,
$\subspace{10022\\01303}{}$,
$\subspace{10022\\01313}{}$,
$\subspace{10022\\01323}{}$,
$\subspace{10022\\01333}{}$,
$\subspace{10022\\01343}{}$,
$\subspace{10030\\01400}{}$,
$\subspace{10030\\01401}{}$,
$\subspace{10030\\01402}{}$,
$\subspace{10030\\01403}{}$,
$\subspace{10030\\01404}{}$,
$\subspace{10004\\01203}{15}$,
$\subspace{10004\\01214}{}$,
$\subspace{10004\\01220}{}$,
$\subspace{10004\\01231}{}$,
$\subspace{10004\\01242}{}$,
$\subspace{10011\\01304}{}$,
$\subspace{10011\\01314}{}$,
$\subspace{10011\\01324}{}$,
$\subspace{10011\\01334}{}$,
$\subspace{10011\\01344}{}$,
$\subspace{10040\\01420}{}$,
$\subspace{10040\\01421}{}$,
$\subspace{10040\\01422}{}$,
$\subspace{10040\\01423}{}$,
$\subspace{10040\\01424}{}$,
$\subspace{10004\\01203}{15}$,
$\subspace{10004\\01214}{}$,
$\subspace{10004\\01220}{}$,
$\subspace{10004\\01231}{}$,
$\subspace{10004\\01242}{}$,
$\subspace{10011\\01304}{}$,
$\subspace{10011\\01314}{}$,
$\subspace{10011\\01324}{}$,
$\subspace{10011\\01334}{}$,
$\subspace{10011\\01344}{}$,
$\subspace{10040\\01420}{}$,
$\subspace{10040\\01421}{}$,
$\subspace{10040\\01422}{}$,
$\subspace{10040\\01423}{}$,
$\subspace{10040\\01424}{}$,
$\subspace{10023\\01200}{15}$,
$\subspace{10023\\01211}{}$,
$\subspace{10023\\01222}{}$,
$\subspace{10023\\01233}{}$,
$\subspace{10023\\01244}{}$,
$\subspace{10024\\01300}{}$,
$\subspace{10024\\01310}{}$,
$\subspace{10024\\01320}{}$,
$\subspace{10024\\01330}{}$,
$\subspace{10024\\01340}{}$,
$\subspace{10013\\01440}{}$,
$\subspace{10013\\01441}{}$,
$\subspace{10013\\01442}{}$,
$\subspace{10013\\01443}{}$,
$\subspace{10013\\01444}{}$,
$\subspace{10033\\01201}{15}$,
$\subspace{10033\\01212}{}$,
$\subspace{10033\\01223}{}$,
$\subspace{10033\\01234}{}$,
$\subspace{10033\\01240}{}$,
$\subspace{10020\\01303}{}$,
$\subspace{10020\\01313}{}$,
$\subspace{10020\\01323}{}$,
$\subspace{10020\\01333}{}$,
$\subspace{10020\\01343}{}$,
$\subspace{10002\\01400}{}$,
$\subspace{10002\\01401}{}$,
$\subspace{10002\\01402}{}$,
$\subspace{10002\\01403}{}$,
$\subspace{10002\\01404}{}$,
$\subspace{10104\\01021}{15}$,
$\subspace{10111\\01030}{}$,
$\subspace{10123\\01044}{}$,
$\subspace{10130\\01003}{}$,
$\subspace{10142\\01012}{}$,
$\subspace{10403\\01110}{}$,
$\subspace{10410\\01144}{}$,
$\subspace{10422\\01123}{}$,
$\subspace{10434\\01102}{}$,
$\subspace{10441\\01131}{}$,
$\subspace{14003\\00103}{}$,
$\subspace{14012\\00141}{}$,
$\subspace{14021\\00134}{}$,
$\subspace{14030\\00122}{}$,
$\subspace{14044\\00110}{}$,
$\subspace{10104\\01021}{15}$,
$\subspace{10111\\01030}{}$,
$\subspace{10123\\01044}{}$,
$\subspace{10130\\01003}{}$,
$\subspace{10142\\01012}{}$,
$\subspace{10403\\01110}{}$,
$\subspace{10410\\01144}{}$,
$\subspace{10422\\01123}{}$,
$\subspace{10434\\01102}{}$,
$\subspace{10441\\01131}{}$,
$\subspace{14003\\00103}{}$,
$\subspace{14012\\00141}{}$,
$\subspace{14021\\00134}{}$,
$\subspace{14030\\00122}{}$,
$\subspace{14044\\00110}{}$,
$\subspace{10104\\01021}{15}$,
$\subspace{10111\\01030}{}$,
$\subspace{10123\\01044}{}$,
$\subspace{10130\\01003}{}$,
$\subspace{10142\\01012}{}$,
$\subspace{10403\\01110}{}$,
$\subspace{10410\\01144}{}$,
$\subspace{10422\\01123}{}$,
$\subspace{10434\\01102}{}$,
$\subspace{10441\\01131}{}$,
$\subspace{14003\\00103}{}$,
$\subspace{14012\\00141}{}$,
$\subspace{14021\\00134}{}$,
$\subspace{14030\\00122}{}$,
$\subspace{14044\\00110}{}$,
$\subspace{10104\\01021}{15}$,
$\subspace{10111\\01030}{}$,
$\subspace{10123\\01044}{}$,
$\subspace{10130\\01003}{}$,
$\subspace{10142\\01012}{}$,
$\subspace{10403\\01110}{}$,
$\subspace{10410\\01144}{}$,
$\subspace{10422\\01123}{}$,
$\subspace{10434\\01102}{}$,
$\subspace{10441\\01131}{}$,
$\subspace{14003\\00103}{}$,
$\subspace{14012\\00141}{}$,
$\subspace{14021\\00134}{}$,
$\subspace{14030\\00122}{}$,
$\subspace{14044\\00110}{}$,
$\subspace{10103\\01034}{15}$,
$\subspace{10110\\01043}{}$,
$\subspace{10122\\01002}{}$,
$\subspace{10134\\01011}{}$,
$\subspace{10141\\01020}{}$,
$\subspace{10400\\01112}{}$,
$\subspace{10412\\01141}{}$,
$\subspace{10424\\01120}{}$,
$\subspace{10431\\01104}{}$,
$\subspace{10443\\01133}{}$,
$\subspace{14000\\00102}{}$,
$\subspace{14014\\00140}{}$,
$\subspace{14023\\00133}{}$,
$\subspace{14032\\00121}{}$,
$\subspace{14041\\00114}{}$,
$\subspace{10104\\01030}{15}$,
$\subspace{10111\\01044}{}$,
$\subspace{10123\\01003}{}$,
$\subspace{10130\\01012}{}$,
$\subspace{10142\\01021}{}$,
$\subspace{10401\\01144}{}$,
$\subspace{10413\\01123}{}$,
$\subspace{10420\\01102}{}$,
$\subspace{10432\\01131}{}$,
$\subspace{10444\\01110}{}$,
$\subspace{14001\\00122}{}$,
$\subspace{14010\\00110}{}$,
$\subspace{14024\\00103}{}$,
$\subspace{14033\\00141}{}$,
$\subspace{14042\\00134}{}$,
$\subspace{10100\\01040}{15}$,
$\subspace{10112\\01004}{}$,
$\subspace{10124\\01013}{}$,
$\subspace{10131\\01022}{}$,
$\subspace{10143\\01031}{}$,
$\subspace{10400\\01100}{}$,
$\subspace{10412\\01134}{}$,
$\subspace{10424\\01113}{}$,
$\subspace{10431\\01142}{}$,
$\subspace{10443\\01121}{}$,
$\subspace{14000\\00111}{}$,
$\subspace{14014\\00104}{}$,
$\subspace{14023\\00142}{}$,
$\subspace{14032\\00130}{}$,
$\subspace{14041\\00123}{}$,
$\subspace{10100\\01040}{15}$,
$\subspace{10112\\01004}{}$,
$\subspace{10124\\01013}{}$,
$\subspace{10131\\01022}{}$,
$\subspace{10143\\01031}{}$,
$\subspace{10400\\01100}{}$,
$\subspace{10412\\01134}{}$,
$\subspace{10424\\01113}{}$,
$\subspace{10431\\01142}{}$,
$\subspace{10443\\01121}{}$,
$\subspace{14000\\00111}{}$,
$\subspace{14014\\00104}{}$,
$\subspace{14023\\00142}{}$,
$\subspace{14032\\00130}{}$,
$\subspace{14041\\00123}{}$,
$\subspace{10100\\01040}{15}$,
$\subspace{10112\\01004}{}$,
$\subspace{10124\\01013}{}$,
$\subspace{10131\\01022}{}$,
$\subspace{10143\\01031}{}$,
$\subspace{10400\\01100}{}$,
$\subspace{10412\\01134}{}$,
$\subspace{10424\\01113}{}$,
$\subspace{10431\\01142}{}$,
$\subspace{10443\\01121}{}$,
$\subspace{14000\\00111}{}$,
$\subspace{14014\\00104}{}$,
$\subspace{14023\\00142}{}$,
$\subspace{14032\\00130}{}$,
$\subspace{14041\\00123}{}$,
$\subspace{10100\\01040}{15}$,
$\subspace{10112\\01004}{}$,
$\subspace{10124\\01013}{}$,
$\subspace{10131\\01022}{}$,
$\subspace{10143\\01031}{}$,
$\subspace{10400\\01100}{}$,
$\subspace{10412\\01134}{}$,
$\subspace{10424\\01113}{}$,
$\subspace{10431\\01142}{}$,
$\subspace{10443\\01121}{}$,
$\subspace{14000\\00111}{}$,
$\subspace{14014\\00104}{}$,
$\subspace{14023\\00142}{}$,
$\subspace{14032\\00130}{}$,
$\subspace{14041\\00123}{}$,
$\subspace{10100\\01104}{15}$,
$\subspace{10112\\01120}{}$,
$\subspace{10124\\01141}{}$,
$\subspace{10131\\01112}{}$,
$\subspace{10143\\01133}{}$,
$\subspace{10401\\01020}{}$,
$\subspace{10413\\01011}{}$,
$\subspace{10420\\01002}{}$,
$\subspace{10432\\01043}{}$,
$\subspace{10444\\01034}{}$,
$\subspace{11000\\00121}{}$,
$\subspace{11014\\00133}{}$,
$\subspace{11023\\00140}{}$,
$\subspace{11032\\00102}{}$,
$\subspace{11041\\00114}{}$,
$\subspace{10104\\01103}{15}$,
$\subspace{10111\\01124}{}$,
$\subspace{10123\\01140}{}$,
$\subspace{10130\\01111}{}$,
$\subspace{10142\\01132}{}$,
$\subspace{10401\\01014}{}$,
$\subspace{10413\\01000}{}$,
$\subspace{10420\\01041}{}$,
$\subspace{10432\\01032}{}$,
$\subspace{10444\\01023}{}$,
$\subspace{11004\\00143}{}$,
$\subspace{11013\\00100}{}$,
$\subspace{11022\\00112}{}$,
$\subspace{11031\\00124}{}$,
$\subspace{11040\\00131}{}$,
$\subspace{10100\\01124}{15}$,
$\subspace{10112\\01140}{}$,
$\subspace{10124\\01111}{}$,
$\subspace{10131\\01132}{}$,
$\subspace{10143\\01103}{}$,
$\subspace{10404\\01023}{}$,
$\subspace{10411\\01014}{}$,
$\subspace{10423\\01000}{}$,
$\subspace{10430\\01041}{}$,
$\subspace{10442\\01032}{}$,
$\subspace{11000\\00143}{}$,
$\subspace{11014\\00100}{}$,
$\subspace{11023\\00112}{}$,
$\subspace{11032\\00124}{}$,
$\subspace{11041\\00131}{}$,
$\subspace{10100\\01124}{15}$,
$\subspace{10112\\01140}{}$,
$\subspace{10124\\01111}{}$,
$\subspace{10131\\01132}{}$,
$\subspace{10143\\01103}{}$,
$\subspace{10404\\01023}{}$,
$\subspace{10411\\01014}{}$,
$\subspace{10423\\01000}{}$,
$\subspace{10430\\01041}{}$,
$\subspace{10442\\01032}{}$,
$\subspace{11000\\00143}{}$,
$\subspace{11014\\00100}{}$,
$\subspace{11023\\00112}{}$,
$\subspace{11032\\00124}{}$,
$\subspace{11041\\00131}{}$,
$\subspace{10103\\01120}{15}$,
$\subspace{10110\\01141}{}$,
$\subspace{10122\\01112}{}$,
$\subspace{10134\\01133}{}$,
$\subspace{10141\\01104}{}$,
$\subspace{10401\\01011}{}$,
$\subspace{10413\\01002}{}$,
$\subspace{10420\\01043}{}$,
$\subspace{10432\\01034}{}$,
$\subspace{10444\\01020}{}$,
$\subspace{11003\\00102}{}$,
$\subspace{11012\\00114}{}$,
$\subspace{11021\\00121}{}$,
$\subspace{11030\\00133}{}$,
$\subspace{11044\\00140}{}$,
$\subspace{10104\\01121}{15}$,
$\subspace{10111\\01142}{}$,
$\subspace{10123\\01113}{}$,
$\subspace{10130\\01134}{}$,
$\subspace{10142\\01100}{}$,
$\subspace{10401\\01022}{}$,
$\subspace{10413\\01013}{}$,
$\subspace{10420\\01004}{}$,
$\subspace{10432\\01040}{}$,
$\subspace{10444\\01031}{}$,
$\subspace{11004\\00130}{}$,
$\subspace{11013\\00142}{}$,
$\subspace{11022\\00104}{}$,
$\subspace{11031\\00111}{}$,
$\subspace{11040\\00123}{}$,
$\subspace{10103\\01134}{15}$,
$\subspace{10110\\01100}{}$,
$\subspace{10122\\01121}{}$,
$\subspace{10134\\01142}{}$,
$\subspace{10141\\01113}{}$,
$\subspace{10401\\01040}{}$,
$\subspace{10413\\01031}{}$,
$\subspace{10420\\01022}{}$,
$\subspace{10432\\01013}{}$,
$\subspace{10444\\01004}{}$,
$\subspace{11003\\00123}{}$,
$\subspace{11012\\00130}{}$,
$\subspace{11021\\00142}{}$,
$\subspace{11030\\00104}{}$,
$\subspace{11044\\00111}{}$,
$\subspace{10103\\01144}{15}$,
$\subspace{10110\\01110}{}$,
$\subspace{10122\\01131}{}$,
$\subspace{10134\\01102}{}$,
$\subspace{10141\\01123}{}$,
$\subspace{10400\\01044}{}$,
$\subspace{10412\\01030}{}$,
$\subspace{10424\\01021}{}$,
$\subspace{10431\\01012}{}$,
$\subspace{10443\\01003}{}$,
$\subspace{11003\\00134}{}$,
$\subspace{11012\\00141}{}$,
$\subspace{11021\\00103}{}$,
$\subspace{11030\\00110}{}$,
$\subspace{11044\\00122}{}$,
$\subspace{10104\\01141}{15}$,
$\subspace{10111\\01112}{}$,
$\subspace{10123\\01133}{}$,
$\subspace{10130\\01104}{}$,
$\subspace{10142\\01120}{}$,
$\subspace{10404\\01020}{}$,
$\subspace{10411\\01011}{}$,
$\subspace{10423\\01002}{}$,
$\subspace{10430\\01043}{}$,
$\subspace{10442\\01034}{}$,
$\subspace{11004\\00102}{}$,
$\subspace{11013\\00114}{}$,
$\subspace{11022\\00121}{}$,
$\subspace{11031\\00133}{}$,
$\subspace{11040\\00140}{}$,
$\subspace{10104\\01141}{15}$,
$\subspace{10111\\01112}{}$,
$\subspace{10123\\01133}{}$,
$\subspace{10130\\01104}{}$,
$\subspace{10142\\01120}{}$,
$\subspace{10404\\01020}{}$,
$\subspace{10411\\01011}{}$,
$\subspace{10423\\01002}{}$,
$\subspace{10430\\01043}{}$,
$\subspace{10442\\01034}{}$,
$\subspace{11004\\00102}{}$,
$\subspace{11013\\00114}{}$,
$\subspace{11022\\00121}{}$,
$\subspace{11031\\00133}{}$,
$\subspace{11040\\00140}{}$,
$\subspace{10104\\01223}{15}$,
$\subspace{10111\\01201}{}$,
$\subspace{10123\\01234}{}$,
$\subspace{10130\\01212}{}$,
$\subspace{10142\\01240}{}$,
$\subspace{10100\\01404}{}$,
$\subspace{10112\\01401}{}$,
$\subspace{10124\\01403}{}$,
$\subspace{10131\\01400}{}$,
$\subspace{10143\\01402}{}$,
$\subspace{10204\\01333}{}$,
$\subspace{10211\\01303}{}$,
$\subspace{10223\\01323}{}$,
$\subspace{10230\\01343}{}$,
$\subspace{10242\\01313}{}$,
$\subspace{10104\\01224}{15}$,
$\subspace{10111\\01202}{}$,
$\subspace{10123\\01230}{}$,
$\subspace{10130\\01213}{}$,
$\subspace{10142\\01241}{}$,
$\subspace{10102\\01411}{}$,
$\subspace{10114\\01413}{}$,
$\subspace{10121\\01410}{}$,
$\subspace{10133\\01412}{}$,
$\subspace{10140\\01414}{}$,
$\subspace{10201\\01301}{}$,
$\subspace{10213\\01321}{}$,
$\subspace{10220\\01341}{}$,
$\subspace{10232\\01311}{}$,
$\subspace{10244\\01331}{}$,
$\subspace{10102\\01234}{15}$,
$\subspace{10114\\01212}{}$,
$\subspace{10121\\01240}{}$,
$\subspace{10133\\01223}{}$,
$\subspace{10140\\01201}{}$,
$\subspace{10100\\01402}{}$,
$\subspace{10112\\01404}{}$,
$\subspace{10124\\01401}{}$,
$\subspace{10131\\01403}{}$,
$\subspace{10143\\01400}{}$,
$\subspace{10201\\01313}{}$,
$\subspace{10213\\01333}{}$,
$\subspace{10220\\01303}{}$,
$\subspace{10232\\01323}{}$,
$\subspace{10244\\01343}{}$,
$\subspace{10102\\01234}{15}$,
$\subspace{10114\\01212}{}$,
$\subspace{10121\\01240}{}$,
$\subspace{10133\\01223}{}$,
$\subspace{10140\\01201}{}$,
$\subspace{10100\\01402}{}$,
$\subspace{10112\\01404}{}$,
$\subspace{10124\\01401}{}$,
$\subspace{10131\\01403}{}$,
$\subspace{10143\\01400}{}$,
$\subspace{10201\\01313}{}$,
$\subspace{10213\\01333}{}$,
$\subspace{10220\\01303}{}$,
$\subspace{10232\\01323}{}$,
$\subspace{10244\\01343}{}$,
$\subspace{10101\\01242}{15}$,
$\subspace{10113\\01220}{}$,
$\subspace{10120\\01203}{}$,
$\subspace{10132\\01231}{}$,
$\subspace{10144\\01214}{}$,
$\subspace{10102\\01422}{}$,
$\subspace{10114\\01424}{}$,
$\subspace{10121\\01421}{}$,
$\subspace{10133\\01423}{}$,
$\subspace{10140\\01420}{}$,
$\subspace{10203\\01304}{}$,
$\subspace{10210\\01324}{}$,
$\subspace{10222\\01344}{}$,
$\subspace{10234\\01314}{}$,
$\subspace{10241\\01334}{}$,
$\subspace{10101\\01242}{15}$,
$\subspace{10113\\01220}{}$,
$\subspace{10120\\01203}{}$,
$\subspace{10132\\01231}{}$,
$\subspace{10144\\01214}{}$,
$\subspace{10102\\01422}{}$,
$\subspace{10114\\01424}{}$,
$\subspace{10121\\01421}{}$,
$\subspace{10133\\01423}{}$,
$\subspace{10140\\01420}{}$,
$\subspace{10203\\01304}{}$,
$\subspace{10210\\01324}{}$,
$\subspace{10222\\01344}{}$,
$\subspace{10234\\01314}{}$,
$\subspace{10241\\01334}{}$,
$\subspace{10101\\01242}{15}$,
$\subspace{10113\\01220}{}$,
$\subspace{10120\\01203}{}$,
$\subspace{10132\\01231}{}$,
$\subspace{10144\\01214}{}$,
$\subspace{10102\\01422}{}$,
$\subspace{10114\\01424}{}$,
$\subspace{10121\\01421}{}$,
$\subspace{10133\\01423}{}$,
$\subspace{10140\\01420}{}$,
$\subspace{10203\\01304}{}$,
$\subspace{10210\\01324}{}$,
$\subspace{10222\\01344}{}$,
$\subspace{10234\\01314}{}$,
$\subspace{10241\\01334}{}$,
$\subspace{10101\\01242}{15}$,
$\subspace{10113\\01220}{}$,
$\subspace{10120\\01203}{}$,
$\subspace{10132\\01231}{}$,
$\subspace{10144\\01214}{}$,
$\subspace{10102\\01422}{}$,
$\subspace{10114\\01424}{}$,
$\subspace{10121\\01421}{}$,
$\subspace{10133\\01423}{}$,
$\subspace{10140\\01420}{}$,
$\subspace{10203\\01304}{}$,
$\subspace{10210\\01324}{}$,
$\subspace{10222\\01344}{}$,
$\subspace{10234\\01314}{}$,
$\subspace{10241\\01334}{}$,
$\subspace{10101\\01242}{15}$,
$\subspace{10113\\01220}{}$,
$\subspace{10120\\01203}{}$,
$\subspace{10132\\01231}{}$,
$\subspace{10144\\01214}{}$,
$\subspace{10102\\01422}{}$,
$\subspace{10114\\01424}{}$,
$\subspace{10121\\01421}{}$,
$\subspace{10133\\01423}{}$,
$\subspace{10140\\01420}{}$,
$\subspace{10203\\01304}{}$,
$\subspace{10210\\01324}{}$,
$\subspace{10222\\01344}{}$,
$\subspace{10234\\01314}{}$,
$\subspace{10241\\01334}{}$,
$\subspace{10101\\01242}{15}$,
$\subspace{10113\\01220}{}$,
$\subspace{10120\\01203}{}$,
$\subspace{10132\\01231}{}$,
$\subspace{10144\\01214}{}$,
$\subspace{10102\\01422}{}$,
$\subspace{10114\\01424}{}$,
$\subspace{10121\\01421}{}$,
$\subspace{10133\\01423}{}$,
$\subspace{10140\\01420}{}$,
$\subspace{10203\\01304}{}$,
$\subspace{10210\\01324}{}$,
$\subspace{10222\\01344}{}$,
$\subspace{10234\\01314}{}$,
$\subspace{10241\\01334}{}$,
$\subspace{10101\\01242}{15}$,
$\subspace{10113\\01220}{}$,
$\subspace{10120\\01203}{}$,
$\subspace{10132\\01231}{}$,
$\subspace{10144\\01214}{}$,
$\subspace{10102\\01422}{}$,
$\subspace{10114\\01424}{}$,
$\subspace{10121\\01421}{}$,
$\subspace{10133\\01423}{}$,
$\subspace{10140\\01420}{}$,
$\subspace{10203\\01304}{}$,
$\subspace{10210\\01324}{}$,
$\subspace{10222\\01344}{}$,
$\subspace{10234\\01314}{}$,
$\subspace{10241\\01334}{}$,
$\subspace{10103\\01301}{15}$,
$\subspace{10110\\01341}{}$,
$\subspace{10122\\01331}{}$,
$\subspace{10134\\01321}{}$,
$\subspace{10141\\01311}{}$,
$\subspace{10302\\01224}{}$,
$\subspace{10314\\01230}{}$,
$\subspace{10321\\01241}{}$,
$\subspace{10333\\01202}{}$,
$\subspace{10340\\01213}{}$,
$\subspace{10301\\01411}{}$,
$\subspace{10313\\01410}{}$,
$\subspace{10320\\01414}{}$,
$\subspace{10332\\01413}{}$,
$\subspace{10344\\01412}{}$,
$\subspace{10103\\01301}{15}$,
$\subspace{10110\\01341}{}$,
$\subspace{10122\\01331}{}$,
$\subspace{10134\\01321}{}$,
$\subspace{10141\\01311}{}$,
$\subspace{10302\\01224}{}$,
$\subspace{10314\\01230}{}$,
$\subspace{10321\\01241}{}$,
$\subspace{10333\\01202}{}$,
$\subspace{10340\\01213}{}$,
$\subspace{10301\\01411}{}$,
$\subspace{10313\\01410}{}$,
$\subspace{10320\\01414}{}$,
$\subspace{10332\\01413}{}$,
$\subspace{10344\\01412}{}$,
$\subspace{10103\\01311}{15}$,
$\subspace{10110\\01301}{}$,
$\subspace{10122\\01341}{}$,
$\subspace{10134\\01331}{}$,
$\subspace{10141\\01321}{}$,
$\subspace{10304\\01213}{}$,
$\subspace{10311\\01224}{}$,
$\subspace{10323\\01230}{}$,
$\subspace{10330\\01241}{}$,
$\subspace{10342\\01202}{}$,
$\subspace{10304\\01414}{}$,
$\subspace{10311\\01413}{}$,
$\subspace{10323\\01412}{}$,
$\subspace{10330\\01411}{}$,
$\subspace{10342\\01410}{}$,
$\subspace{10103\\01311}{15}$,
$\subspace{10110\\01301}{}$,
$\subspace{10122\\01341}{}$,
$\subspace{10134\\01331}{}$,
$\subspace{10141\\01321}{}$,
$\subspace{10304\\01213}{}$,
$\subspace{10311\\01224}{}$,
$\subspace{10323\\01230}{}$,
$\subspace{10330\\01241}{}$,
$\subspace{10342\\01202}{}$,
$\subspace{10304\\01414}{}$,
$\subspace{10311\\01413}{}$,
$\subspace{10323\\01412}{}$,
$\subspace{10330\\01411}{}$,
$\subspace{10342\\01410}{}$,
$\subspace{10103\\01332}{15}$,
$\subspace{10110\\01322}{}$,
$\subspace{10122\\01312}{}$,
$\subspace{10134\\01302}{}$,
$\subspace{10141\\01342}{}$,
$\subspace{10304\\01204}{}$,
$\subspace{10311\\01210}{}$,
$\subspace{10323\\01221}{}$,
$\subspace{10330\\01232}{}$,
$\subspace{10342\\01243}{}$,
$\subspace{10303\\01434}{}$,
$\subspace{10310\\01433}{}$,
$\subspace{10322\\01432}{}$,
$\subspace{10334\\01431}{}$,
$\subspace{10341\\01430}{}$,
$\subspace{10101\\01340}{15}$,
$\subspace{10113\\01330}{}$,
$\subspace{10120\\01320}{}$,
$\subspace{10132\\01310}{}$,
$\subspace{10144\\01300}{}$,
$\subspace{10303\\01244}{}$,
$\subspace{10310\\01200}{}$,
$\subspace{10322\\01211}{}$,
$\subspace{10334\\01222}{}$,
$\subspace{10341\\01233}{}$,
$\subspace{10303\\01441}{}$,
$\subspace{10310\\01440}{}$,
$\subspace{10322\\01444}{}$,
$\subspace{10334\\01443}{}$,
$\subspace{10341\\01442}{}$,
$\subspace{10101\\01340}{15}$,
$\subspace{10113\\01330}{}$,
$\subspace{10120\\01320}{}$,
$\subspace{10132\\01310}{}$,
$\subspace{10144\\01300}{}$,
$\subspace{10303\\01244}{}$,
$\subspace{10310\\01200}{}$,
$\subspace{10322\\01211}{}$,
$\subspace{10334\\01222}{}$,
$\subspace{10341\\01233}{}$,
$\subspace{10303\\01441}{}$,
$\subspace{10310\\01440}{}$,
$\subspace{10322\\01444}{}$,
$\subspace{10334\\01443}{}$,
$\subspace{10341\\01442}{}$,
$\subspace{10101\\01340}{15}$,
$\subspace{10113\\01330}{}$,
$\subspace{10120\\01320}{}$,
$\subspace{10132\\01310}{}$,
$\subspace{10144\\01300}{}$,
$\subspace{10303\\01244}{}$,
$\subspace{10310\\01200}{}$,
$\subspace{10322\\01211}{}$,
$\subspace{10334\\01222}{}$,
$\subspace{10341\\01233}{}$,
$\subspace{10303\\01441}{}$,
$\subspace{10310\\01440}{}$,
$\subspace{10322\\01444}{}$,
$\subspace{10334\\01443}{}$,
$\subspace{10341\\01442}{}$,
$\subspace{10101\\01340}{15}$,
$\subspace{10113\\01330}{}$,
$\subspace{10120\\01320}{}$,
$\subspace{10132\\01310}{}$,
$\subspace{10144\\01300}{}$,
$\subspace{10303\\01244}{}$,
$\subspace{10310\\01200}{}$,
$\subspace{10322\\01211}{}$,
$\subspace{10334\\01222}{}$,
$\subspace{10341\\01233}{}$,
$\subspace{10303\\01441}{}$,
$\subspace{10310\\01440}{}$,
$\subspace{10322\\01444}{}$,
$\subspace{10334\\01443}{}$,
$\subspace{10341\\01442}{}$,
$\subspace{10203\\01004}{15}$,
$\subspace{10210\\01031}{}$,
$\subspace{10222\\01013}{}$,
$\subspace{10234\\01040}{}$,
$\subspace{10241\\01022}{}$,
$\subspace{10300\\01113}{}$,
$\subspace{10312\\01100}{}$,
$\subspace{10324\\01142}{}$,
$\subspace{10331\\01134}{}$,
$\subspace{10343\\01121}{}$,
$\subspace{13000\\00142}{}$,
$\subspace{13014\\00111}{}$,
$\subspace{13023\\00130}{}$,
$\subspace{13032\\00104}{}$,
$\subspace{13041\\00123}{}$,
$\subspace{10204\\01000}{15}$,
$\subspace{10211\\01032}{}$,
$\subspace{10223\\01014}{}$,
$\subspace{10230\\01041}{}$,
$\subspace{10242\\01023}{}$,
$\subspace{10304\\01132}{}$,
$\subspace{10311\\01124}{}$,
$\subspace{10323\\01111}{}$,
$\subspace{10330\\01103}{}$,
$\subspace{10342\\01140}{}$,
$\subspace{13004\\00143}{}$,
$\subspace{13013\\00112}{}$,
$\subspace{13022\\00131}{}$,
$\subspace{13031\\00100}{}$,
$\subspace{13040\\00124}{}$,
$\subspace{10204\\01002}{15}$,
$\subspace{10211\\01034}{}$,
$\subspace{10223\\01011}{}$,
$\subspace{10230\\01043}{}$,
$\subspace{10242\\01020}{}$,
$\subspace{10301\\01104}{}$,
$\subspace{10313\\01141}{}$,
$\subspace{10320\\01133}{}$,
$\subspace{10332\\01120}{}$,
$\subspace{10344\\01112}{}$,
$\subspace{13001\\00102}{}$,
$\subspace{13010\\00121}{}$,
$\subspace{13024\\00140}{}$,
$\subspace{13033\\00114}{}$,
$\subspace{13042\\00133}{}$,
$\subspace{10204\\01013}{15}$,
$\subspace{10211\\01040}{}$,
$\subspace{10223\\01022}{}$,
$\subspace{10230\\01004}{}$,
$\subspace{10242\\01031}{}$,
$\subspace{10304\\01100}{}$,
$\subspace{10311\\01142}{}$,
$\subspace{10323\\01134}{}$,
$\subspace{10330\\01121}{}$,
$\subspace{10342\\01113}{}$,
$\subspace{13004\\00130}{}$,
$\subspace{13013\\00104}{}$,
$\subspace{13022\\00123}{}$,
$\subspace{13031\\00142}{}$,
$\subspace{13040\\00111}{}$,
$\subspace{10200\\01022}{15}$,
$\subspace{10212\\01004}{}$,
$\subspace{10224\\01031}{}$,
$\subspace{10231\\01013}{}$,
$\subspace{10243\\01040}{}$,
$\subspace{10303\\01142}{}$,
$\subspace{10310\\01134}{}$,
$\subspace{10322\\01121}{}$,
$\subspace{10334\\01113}{}$,
$\subspace{10341\\01100}{}$,
$\subspace{13003\\00123}{}$,
$\subspace{13012\\00142}{}$,
$\subspace{13021\\00111}{}$,
$\subspace{13030\\00130}{}$,
$\subspace{13044\\00104}{}$,
$\subspace{10204\\01023}{15}$,
$\subspace{10211\\01000}{}$,
$\subspace{10223\\01032}{}$,
$\subspace{10230\\01014}{}$,
$\subspace{10242\\01041}{}$,
$\subspace{10301\\01140}{}$,
$\subspace{10313\\01132}{}$,
$\subspace{10320\\01124}{}$,
$\subspace{10332\\01111}{}$,
$\subspace{10344\\01103}{}$,
$\subspace{13001\\00131}{}$,
$\subspace{13010\\00100}{}$,
$\subspace{13024\\00124}{}$,
$\subspace{13033\\00143}{}$,
$\subspace{13042\\00112}{}$,
$\subspace{10200\\01032}{15}$,
$\subspace{10212\\01014}{}$,
$\subspace{10224\\01041}{}$,
$\subspace{10231\\01023}{}$,
$\subspace{10243\\01000}{}$,
$\subspace{10300\\01132}{}$,
$\subspace{10312\\01124}{}$,
$\subspace{10324\\01111}{}$,
$\subspace{10331\\01103}{}$,
$\subspace{10343\\01140}{}$,
$\subspace{13000\\00124}{}$,
$\subspace{13014\\00143}{}$,
$\subspace{13023\\00112}{}$,
$\subspace{13032\\00131}{}$,
$\subspace{13041\\00100}{}$,
$\subspace{10200\\01032}{15}$,
$\subspace{10212\\01014}{}$,
$\subspace{10224\\01041}{}$,
$\subspace{10231\\01023}{}$,
$\subspace{10243\\01000}{}$,
$\subspace{10300\\01132}{}$,
$\subspace{10312\\01124}{}$,
$\subspace{10324\\01111}{}$,
$\subspace{10331\\01103}{}$,
$\subspace{10343\\01140}{}$,
$\subspace{13000\\00124}{}$,
$\subspace{13014\\00143}{}$,
$\subspace{13023\\00112}{}$,
$\subspace{13032\\00131}{}$,
$\subspace{13041\\00100}{}$,
$\subspace{10200\\01032}{15}$,
$\subspace{10212\\01014}{}$,
$\subspace{10224\\01041}{}$,
$\subspace{10231\\01023}{}$,
$\subspace{10243\\01000}{}$,
$\subspace{10300\\01132}{}$,
$\subspace{10312\\01124}{}$,
$\subspace{10324\\01111}{}$,
$\subspace{10331\\01103}{}$,
$\subspace{10343\\01140}{}$,
$\subspace{13000\\00124}{}$,
$\subspace{13014\\00143}{}$,
$\subspace{13023\\00112}{}$,
$\subspace{13032\\00131}{}$,
$\subspace{13041\\00100}{}$,
$\subspace{10200\\01040}{15}$,
$\subspace{10212\\01022}{}$,
$\subspace{10224\\01004}{}$,
$\subspace{10231\\01031}{}$,
$\subspace{10243\\01013}{}$,
$\subspace{10300\\01100}{}$,
$\subspace{10312\\01142}{}$,
$\subspace{10324\\01134}{}$,
$\subspace{10331\\01121}{}$,
$\subspace{10343\\01113}{}$,
$\subspace{13000\\00111}{}$,
$\subspace{13014\\00130}{}$,
$\subspace{13023\\00104}{}$,
$\subspace{13032\\00123}{}$,
$\subspace{13041\\00142}{}$,
$\subspace{10202\\01114}{15}$,
$\subspace{10214\\01122}{}$,
$\subspace{10221\\01130}{}$,
$\subspace{10233\\01143}{}$,
$\subspace{10240\\01101}{}$,
$\subspace{10302\\01010}{}$,
$\subspace{10314\\01033}{}$,
$\subspace{10321\\01001}{}$,
$\subspace{10333\\01024}{}$,
$\subspace{10340\\01042}{}$,
$\subspace{12002\\00120}{}$,
$\subspace{12011\\00101}{}$,
$\subspace{12020\\00132}{}$,
$\subspace{12034\\00113}{}$,
$\subspace{12043\\00144}{}$,
$\subspace{10202\\01114}{15}$,
$\subspace{10214\\01122}{}$,
$\subspace{10221\\01130}{}$,
$\subspace{10233\\01143}{}$,
$\subspace{10240\\01101}{}$,
$\subspace{10302\\01010}{}$,
$\subspace{10314\\01033}{}$,
$\subspace{10321\\01001}{}$,
$\subspace{10333\\01024}{}$,
$\subspace{10340\\01042}{}$,
$\subspace{12002\\00120}{}$,
$\subspace{12011\\00101}{}$,
$\subspace{12020\\00132}{}$,
$\subspace{12034\\00113}{}$,
$\subspace{12043\\00144}{}$,
$\subspace{10202\\01114}{15}$,
$\subspace{10214\\01122}{}$,
$\subspace{10221\\01130}{}$,
$\subspace{10233\\01143}{}$,
$\subspace{10240\\01101}{}$,
$\subspace{10302\\01010}{}$,
$\subspace{10314\\01033}{}$,
$\subspace{10321\\01001}{}$,
$\subspace{10333\\01024}{}$,
$\subspace{10340\\01042}{}$,
$\subspace{12002\\00120}{}$,
$\subspace{12011\\00101}{}$,
$\subspace{12020\\00132}{}$,
$\subspace{12034\\00113}{}$,
$\subspace{12043\\00144}{}$,
$\subspace{10202\\01114}{15}$,
$\subspace{10214\\01122}{}$,
$\subspace{10221\\01130}{}$,
$\subspace{10233\\01143}{}$,
$\subspace{10240\\01101}{}$,
$\subspace{10302\\01010}{}$,
$\subspace{10314\\01033}{}$,
$\subspace{10321\\01001}{}$,
$\subspace{10333\\01024}{}$,
$\subspace{10340\\01042}{}$,
$\subspace{12002\\00120}{}$,
$\subspace{12011\\00101}{}$,
$\subspace{12020\\00132}{}$,
$\subspace{12034\\00113}{}$,
$\subspace{12043\\00144}{}$,
$\subspace{10202\\01114}{15}$,
$\subspace{10214\\01122}{}$,
$\subspace{10221\\01130}{}$,
$\subspace{10233\\01143}{}$,
$\subspace{10240\\01101}{}$,
$\subspace{10302\\01010}{}$,
$\subspace{10314\\01033}{}$,
$\subspace{10321\\01001}{}$,
$\subspace{10333\\01024}{}$,
$\subspace{10340\\01042}{}$,
$\subspace{12002\\00120}{}$,
$\subspace{12011\\00101}{}$,
$\subspace{12020\\00132}{}$,
$\subspace{12034\\00113}{}$,
$\subspace{12043\\00144}{}$,
$\subspace{10202\\01114}{15}$,
$\subspace{10214\\01122}{}$,
$\subspace{10221\\01130}{}$,
$\subspace{10233\\01143}{}$,
$\subspace{10240\\01101}{}$,
$\subspace{10302\\01010}{}$,
$\subspace{10314\\01033}{}$,
$\subspace{10321\\01001}{}$,
$\subspace{10333\\01024}{}$,
$\subspace{10340\\01042}{}$,
$\subspace{12002\\00120}{}$,
$\subspace{12011\\00101}{}$,
$\subspace{12020\\00132}{}$,
$\subspace{12034\\00113}{}$,
$\subspace{12043\\00144}{}$,
$\subspace{10202\\01114}{15}$,
$\subspace{10214\\01122}{}$,
$\subspace{10221\\01130}{}$,
$\subspace{10233\\01143}{}$,
$\subspace{10240\\01101}{}$,
$\subspace{10302\\01010}{}$,
$\subspace{10314\\01033}{}$,
$\subspace{10321\\01001}{}$,
$\subspace{10333\\01024}{}$,
$\subspace{10340\\01042}{}$,
$\subspace{12002\\00120}{}$,
$\subspace{12011\\00101}{}$,
$\subspace{12020\\00132}{}$,
$\subspace{12034\\00113}{}$,
$\subspace{12043\\00144}{}$,
$\subspace{10202\\01114}{15}$,
$\subspace{10214\\01122}{}$,
$\subspace{10221\\01130}{}$,
$\subspace{10233\\01143}{}$,
$\subspace{10240\\01101}{}$,
$\subspace{10302\\01010}{}$,
$\subspace{10314\\01033}{}$,
$\subspace{10321\\01001}{}$,
$\subspace{10333\\01024}{}$,
$\subspace{10340\\01042}{}$,
$\subspace{12002\\00120}{}$,
$\subspace{12011\\00101}{}$,
$\subspace{12020\\00132}{}$,
$\subspace{12034\\00113}{}$,
$\subspace{12043\\00144}{}$,
$\subspace{10202\\01114}{15}$,
$\subspace{10214\\01122}{}$,
$\subspace{10221\\01130}{}$,
$\subspace{10233\\01143}{}$,
$\subspace{10240\\01101}{}$,
$\subspace{10302\\01010}{}$,
$\subspace{10314\\01033}{}$,
$\subspace{10321\\01001}{}$,
$\subspace{10333\\01024}{}$,
$\subspace{10340\\01042}{}$,
$\subspace{12002\\00120}{}$,
$\subspace{12011\\00101}{}$,
$\subspace{12020\\00132}{}$,
$\subspace{12034\\00113}{}$,
$\subspace{12043\\00144}{}$,
$\subspace{10204\\01131}{15}$,
$\subspace{10211\\01144}{}$,
$\subspace{10223\\01102}{}$,
$\subspace{10230\\01110}{}$,
$\subspace{10242\\01123}{}$,
$\subspace{10301\\01044}{}$,
$\subspace{10313\\01012}{}$,
$\subspace{10320\\01030}{}$,
$\subspace{10332\\01003}{}$,
$\subspace{10344\\01021}{}$,
$\subspace{12004\\00110}{}$,
$\subspace{12013\\00141}{}$,
$\subspace{12022\\00122}{}$,
$\subspace{12031\\00103}{}$,
$\subspace{12040\\00134}{}$,
$\subspace{10201\\01200}{15}$,
$\subspace{10213\\01244}{}$,
$\subspace{10220\\01233}{}$,
$\subspace{10232\\01222}{}$,
$\subspace{10244\\01211}{}$,
$\subspace{10201\\01441}{}$,
$\subspace{10213\\01442}{}$,
$\subspace{10220\\01443}{}$,
$\subspace{10232\\01444}{}$,
$\subspace{10244\\01440}{}$,
$\subspace{10401\\01300}{}$,
$\subspace{10413\\01310}{}$,
$\subspace{10420\\01320}{}$,
$\subspace{10432\\01330}{}$,
$\subspace{10444\\01340}{}$,
$\subspace{10201\\01201}{15}$,
$\subspace{10213\\01240}{}$,
$\subspace{10220\\01234}{}$,
$\subspace{10232\\01223}{}$,
$\subspace{10244\\01212}{}$,
$\subspace{10200\\01403}{}$,
$\subspace{10212\\01404}{}$,
$\subspace{10224\\01400}{}$,
$\subspace{10231\\01401}{}$,
$\subspace{10243\\01402}{}$,
$\subspace{10400\\01323}{}$,
$\subspace{10412\\01333}{}$,
$\subspace{10424\\01343}{}$,
$\subspace{10431\\01303}{}$,
$\subspace{10443\\01313}{}$,
$\subspace{10204\\01214}{15}$,
$\subspace{10211\\01203}{}$,
$\subspace{10223\\01242}{}$,
$\subspace{10230\\01231}{}$,
$\subspace{10242\\01220}{}$,
$\subspace{10204\\01423}{}$,
$\subspace{10211\\01424}{}$,
$\subspace{10223\\01420}{}$,
$\subspace{10230\\01421}{}$,
$\subspace{10242\\01422}{}$,
$\subspace{10404\\01314}{}$,
$\subspace{10411\\01324}{}$,
$\subspace{10423\\01334}{}$,
$\subspace{10430\\01344}{}$,
$\subspace{10442\\01304}{}$,
$\subspace{10200\\01224}{15}$,
$\subspace{10212\\01213}{}$,
$\subspace{10224\\01202}{}$,
$\subspace{10231\\01241}{}$,
$\subspace{10243\\01230}{}$,
$\subspace{10200\\01414}{}$,
$\subspace{10212\\01410}{}$,
$\subspace{10224\\01411}{}$,
$\subspace{10231\\01412}{}$,
$\subspace{10243\\01413}{}$,
$\subspace{10404\\01321}{}$,
$\subspace{10411\\01331}{}$,
$\subspace{10423\\01341}{}$,
$\subspace{10430\\01301}{}$,
$\subspace{10442\\01311}{}$,
$\subspace{10200\\01232}{15}$,
$\subspace{10212\\01221}{}$,
$\subspace{10224\\01210}{}$,
$\subspace{10231\\01204}{}$,
$\subspace{10243\\01243}{}$,
$\subspace{10200\\01432}{}$,
$\subspace{10212\\01433}{}$,
$\subspace{10224\\01434}{}$,
$\subspace{10231\\01430}{}$,
$\subspace{10243\\01431}{}$,
$\subspace{10400\\01332}{}$,
$\subspace{10412\\01342}{}$,
$\subspace{10424\\01302}{}$,
$\subspace{10431\\01312}{}$,
$\subspace{10443\\01322}{}$,
$\subspace{10201\\01234}{15}$,
$\subspace{10213\\01223}{}$,
$\subspace{10220\\01212}{}$,
$\subspace{10232\\01201}{}$,
$\subspace{10244\\01240}{}$,
$\subspace{10201\\01400}{}$,
$\subspace{10213\\01401}{}$,
$\subspace{10220\\01402}{}$,
$\subspace{10232\\01403}{}$,
$\subspace{10244\\01404}{}$,
$\subspace{10404\\01333}{}$,
$\subspace{10411\\01343}{}$,
$\subspace{10423\\01303}{}$,
$\subspace{10430\\01313}{}$,
$\subspace{10442\\01323}{}$,
$\subspace{10203\\01234}{15}$,
$\subspace{10210\\01223}{}$,
$\subspace{10222\\01212}{}$,
$\subspace{10234\\01201}{}$,
$\subspace{10241\\01240}{}$,
$\subspace{10202\\01403}{}$,
$\subspace{10214\\01404}{}$,
$\subspace{10221\\01400}{}$,
$\subspace{10233\\01401}{}$,
$\subspace{10240\\01402}{}$,
$\subspace{10401\\01303}{}$,
$\subspace{10413\\01313}{}$,
$\subspace{10420\\01323}{}$,
$\subspace{10432\\01333}{}$,
$\subspace{10444\\01343}{}$,
$\subspace{10201\\01244}{15}$,
$\subspace{10213\\01233}{}$,
$\subspace{10220\\01222}{}$,
$\subspace{10232\\01211}{}$,
$\subspace{10244\\01200}{}$,
$\subspace{10204\\01442}{}$,
$\subspace{10211\\01443}{}$,
$\subspace{10223\\01444}{}$,
$\subspace{10230\\01440}{}$,
$\subspace{10242\\01441}{}$,
$\subspace{10403\\01330}{}$,
$\subspace{10410\\01340}{}$,
$\subspace{10422\\01300}{}$,
$\subspace{10434\\01310}{}$,
$\subspace{10441\\01320}{}$,
$\subspace{10201\\01244}{15}$,
$\subspace{10213\\01233}{}$,
$\subspace{10220\\01222}{}$,
$\subspace{10232\\01211}{}$,
$\subspace{10244\\01200}{}$,
$\subspace{10204\\01442}{}$,
$\subspace{10211\\01443}{}$,
$\subspace{10223\\01444}{}$,
$\subspace{10230\\01440}{}$,
$\subspace{10242\\01441}{}$,
$\subspace{10403\\01330}{}$,
$\subspace{10410\\01340}{}$,
$\subspace{10422\\01300}{}$,
$\subspace{10434\\01310}{}$,
$\subspace{10441\\01320}{}$,
$\subspace{10201\\01244}{15}$,
$\subspace{10213\\01233}{}$,
$\subspace{10220\\01222}{}$,
$\subspace{10232\\01211}{}$,
$\subspace{10244\\01200}{}$,
$\subspace{10204\\01442}{}$,
$\subspace{10211\\01443}{}$,
$\subspace{10223\\01444}{}$,
$\subspace{10230\\01440}{}$,
$\subspace{10242\\01441}{}$,
$\subspace{10403\\01330}{}$,
$\subspace{10410\\01340}{}$,
$\subspace{10422\\01300}{}$,
$\subspace{10434\\01310}{}$,
$\subspace{10441\\01320}{}$,
$\subspace{10301\\01322}{15}$,
$\subspace{10313\\01302}{}$,
$\subspace{10320\\01332}{}$,
$\subspace{10332\\01312}{}$,
$\subspace{10344\\01342}{}$,
$\subspace{10402\\01221}{}$,
$\subspace{10414\\01243}{}$,
$\subspace{10421\\01210}{}$,
$\subspace{10433\\01232}{}$,
$\subspace{10440\\01204}{}$,
$\subspace{10402\\01432}{}$,
$\subspace{10414\\01430}{}$,
$\subspace{10421\\01433}{}$,
$\subspace{10433\\01431}{}$,
$\subspace{10440\\01434}{}$,
$\subspace{10301\\01322}{15}$,
$\subspace{10313\\01302}{}$,
$\subspace{10320\\01332}{}$,
$\subspace{10332\\01312}{}$,
$\subspace{10344\\01342}{}$,
$\subspace{10402\\01221}{}$,
$\subspace{10414\\01243}{}$,
$\subspace{10421\\01210}{}$,
$\subspace{10433\\01232}{}$,
$\subspace{10440\\01204}{}$,
$\subspace{10402\\01432}{}$,
$\subspace{10414\\01430}{}$,
$\subspace{10421\\01433}{}$,
$\subspace{10433\\01431}{}$,
$\subspace{10440\\01434}{}$,
$\subspace{10301\\01322}{15}$,
$\subspace{10313\\01302}{}$,
$\subspace{10320\\01332}{}$,
$\subspace{10332\\01312}{}$,
$\subspace{10344\\01342}{}$,
$\subspace{10402\\01221}{}$,
$\subspace{10414\\01243}{}$,
$\subspace{10421\\01210}{}$,
$\subspace{10433\\01232}{}$,
$\subspace{10440\\01204}{}$,
$\subspace{10402\\01432}{}$,
$\subspace{10414\\01430}{}$,
$\subspace{10421\\01433}{}$,
$\subspace{10433\\01431}{}$,
$\subspace{10440\\01434}{}$,
$\subspace{10301\\01322}{15}$,
$\subspace{10313\\01302}{}$,
$\subspace{10320\\01332}{}$,
$\subspace{10332\\01312}{}$,
$\subspace{10344\\01342}{}$,
$\subspace{10402\\01221}{}$,
$\subspace{10414\\01243}{}$,
$\subspace{10421\\01210}{}$,
$\subspace{10433\\01232}{}$,
$\subspace{10440\\01204}{}$,
$\subspace{10402\\01432}{}$,
$\subspace{10414\\01430}{}$,
$\subspace{10421\\01433}{}$,
$\subspace{10433\\01431}{}$,
$\subspace{10440\\01434}{}$,
$\subspace{10304\\01323}{15}$,
$\subspace{10311\\01303}{}$,
$\subspace{10323\\01333}{}$,
$\subspace{10330\\01313}{}$,
$\subspace{10342\\01343}{}$,
$\subspace{10404\\01223}{}$,
$\subspace{10411\\01240}{}$,
$\subspace{10423\\01212}{}$,
$\subspace{10430\\01234}{}$,
$\subspace{10442\\01201}{}$,
$\subspace{10404\\01400}{}$,
$\subspace{10411\\01403}{}$,
$\subspace{10423\\01401}{}$,
$\subspace{10430\\01404}{}$,
$\subspace{10442\\01402}{}$,
$\subspace{10300\\01342}{15}$,
$\subspace{10312\\01322}{}$,
$\subspace{10324\\01302}{}$,
$\subspace{10331\\01332}{}$,
$\subspace{10343\\01312}{}$,
$\subspace{10401\\01221}{}$,
$\subspace{10413\\01243}{}$,
$\subspace{10420\\01210}{}$,
$\subspace{10432\\01232}{}$,
$\subspace{10444\\01204}{}$,
$\subspace{10404\\01430}{}$,
$\subspace{10411\\01433}{}$,
$\subspace{10423\\01431}{}$,
$\subspace{10430\\01434}{}$,
$\subspace{10442\\01432}{}$,
$\subspace{11200\\00011}{3}$,
$\subspace{11440\\00001}{}$,
$\subspace{13400\\00010}{}$,
$\subspace{11202\\00011}{3}$,
$\subspace{11410\\00001}{}$,
$\subspace{13403\\00010}{}$,
$\subspace{11203\\00011}{3}$,
$\subspace{11420\\00001}{}$,
$\subspace{13402\\00010}{}$,
$\subspace{12101\\00010}{3}$,
$\subspace{14100\\00001}{}$,
$\subspace{14304\\00011}{}$,
$\subspace{12103\\00010}{3}$,
$\subspace{14130\\00001}{}$,
$\subspace{14302\\00011}{}$,
$\subspace{12103\\00010}{3}$,
$\subspace{14130\\00001}{}$,
$\subspace{14302\\00011}{}$.

\smallskip

$n_5(5,2;59)\ge 1484$,$\left[\!\begin{smallmatrix}10000\\00100\\04410\\00001\\00044\end{smallmatrix}\!\right]$:
$\subspace{01000\\00101}{15}$,
$\subspace{01001\\00100}{}$,
$\subspace{01001\\00144}{}$,
$\subspace{01010\\00101}{}$,
$\subspace{01010\\00112}{}$,
$\subspace{01014\\00113}{}$,
$\subspace{01024\\00113}{}$,
$\subspace{01023\\00120}{}$,
$\subspace{01024\\00124}{}$,
$\subspace{01033\\00120}{}$,
$\subspace{01032\\00132}{}$,
$\subspace{01033\\00131}{}$,
$\subspace{01042\\00132}{}$,
$\subspace{01041\\00144}{}$,
$\subspace{01042\\00143}{}$,
$\subspace{01002\\00102}{5}$,
$\subspace{01011\\00114}{}$,
$\subspace{01020\\00121}{}$,
$\subspace{01034\\00133}{}$,
$\subspace{01043\\00140}{}$,
$\subspace{01001\\00113}{5}$,
$\subspace{01010\\00120}{}$,
$\subspace{01024\\00132}{}$,
$\subspace{01033\\00144}{}$,
$\subspace{01042\\00101}{}$,
$\subspace{01001\\00113}{5}$,
$\subspace{01010\\00120}{}$,
$\subspace{01024\\00132}{}$,
$\subspace{01033\\00144}{}$,
$\subspace{01042\\00101}{}$,
$\subspace{01003\\00141}{5}$,
$\subspace{01012\\00103}{}$,
$\subspace{01021\\00110}{}$,
$\subspace{01030\\00122}{}$,
$\subspace{01044\\00134}{}$,
$\subspace{01003\\00141}{5}$,
$\subspace{01012\\00103}{}$,
$\subspace{01021\\00110}{}$,
$\subspace{01030\\00122}{}$,
$\subspace{01044\\00134}{}$,
$\subspace{01003\\00141}{5}$,
$\subspace{01012\\00103}{}$,
$\subspace{01021\\00110}{}$,
$\subspace{01030\\00122}{}$,
$\subspace{01044\\00134}{}$,
$\subspace{01003\\00141}{5}$,
$\subspace{01012\\00103}{}$,
$\subspace{01021\\00110}{}$,
$\subspace{01030\\00122}{}$,
$\subspace{01044\\00134}{}$,
$\subspace{01200\\00011}{3}$,
$\subspace{01300\\00010}{}$,
$\subspace{01440\\00001}{}$,
$\subspace{01202\\00011}{3}$,
$\subspace{01301\\00010}{}$,
$\subspace{01410\\00001}{}$,
$\subspace{01204\\00011}{3}$,
$\subspace{01302\\00010}{}$,
$\subspace{01430\\00001}{}$,
$\subspace{10000\\00001}{3}$,
$\subspace{10000\\00010}{}$,
$\subspace{10000\\00011}{}$,
$\subspace{10000\\00001}{3}$,
$\subspace{10000\\00010}{}$,
$\subspace{10000\\00011}{}$,
$\subspace{10000\\00012}{3}$,
$\subspace{10000\\00013}{}$,
$\subspace{10000\\00014}{}$,
$\subspace{10000\\00012}{3}$,
$\subspace{10000\\00013}{}$,
$\subspace{10000\\00014}{}$,
$\subspace{10002\\00101}{15}$,
$\subspace{10002\\00113}{}$,
$\subspace{10002\\00120}{}$,
$\subspace{10002\\00132}{}$,
$\subspace{10002\\00144}{}$,
$\subspace{10020\\01001}{}$,
$\subspace{10020\\01010}{}$,
$\subspace{10020\\01024}{}$,
$\subspace{10020\\01033}{}$,
$\subspace{10020\\01042}{}$,
$\subspace{10033\\01101}{}$,
$\subspace{10033\\01114}{}$,
$\subspace{10033\\01122}{}$,
$\subspace{10033\\01130}{}$,
$\subspace{10033\\01143}{}$,
$\subspace{10020\\00102}{15}$,
$\subspace{10020\\00114}{}$,
$\subspace{10020\\00121}{}$,
$\subspace{10020\\00133}{}$,
$\subspace{10020\\00140}{}$,
$\subspace{10033\\01002}{}$,
$\subspace{10033\\01011}{}$,
$\subspace{10033\\01020}{}$,
$\subspace{10033\\01034}{}$,
$\subspace{10033\\01043}{}$,
$\subspace{10002\\01104}{}$,
$\subspace{10002\\01112}{}$,
$\subspace{10002\\01120}{}$,
$\subspace{10002\\01133}{}$,
$\subspace{10002\\01141}{}$,
$\subspace{10031\\00103}{15}$,
$\subspace{10031\\00110}{}$,
$\subspace{10031\\00122}{}$,
$\subspace{10031\\00134}{}$,
$\subspace{10031\\00141}{}$,
$\subspace{10032\\01003}{}$,
$\subspace{10032\\01012}{}$,
$\subspace{10032\\01021}{}$,
$\subspace{10032\\01030}{}$,
$\subspace{10032\\01044}{}$,
$\subspace{10042\\01102}{}$,
$\subspace{10042\\01110}{}$,
$\subspace{10042\\01123}{}$,
$\subspace{10042\\01131}{}$,
$\subspace{10042\\01144}{}$,
$\subspace{10031\\00103}{15}$,
$\subspace{10031\\00110}{}$,
$\subspace{10031\\00122}{}$,
$\subspace{10031\\00134}{}$,
$\subspace{10031\\00141}{}$,
$\subspace{10032\\01003}{}$,
$\subspace{10032\\01012}{}$,
$\subspace{10032\\01021}{}$,
$\subspace{10032\\01030}{}$,
$\subspace{10032\\01044}{}$,
$\subspace{10042\\01102}{}$,
$\subspace{10042\\01110}{}$,
$\subspace{10042\\01123}{}$,
$\subspace{10042\\01131}{}$,
$\subspace{10042\\01144}{}$,
$\subspace{10034\\00102}{15}$,
$\subspace{10034\\00114}{}$,
$\subspace{10034\\00121}{}$,
$\subspace{10034\\00133}{}$,
$\subspace{10034\\00140}{}$,
$\subspace{10012\\01002}{}$,
$\subspace{10012\\01011}{}$,
$\subspace{10012\\01020}{}$,
$\subspace{10012\\01034}{}$,
$\subspace{10012\\01043}{}$,
$\subspace{10014\\01104}{}$,
$\subspace{10014\\01112}{}$,
$\subspace{10014\\01120}{}$,
$\subspace{10014\\01133}{}$,
$\subspace{10014\\01141}{}$,
$\subspace{10034\\00102}{15}$,
$\subspace{10034\\00114}{}$,
$\subspace{10034\\00121}{}$,
$\subspace{10034\\00133}{}$,
$\subspace{10034\\00140}{}$,
$\subspace{10012\\01002}{}$,
$\subspace{10012\\01011}{}$,
$\subspace{10012\\01020}{}$,
$\subspace{10012\\01034}{}$,
$\subspace{10012\\01043}{}$,
$\subspace{10014\\01104}{}$,
$\subspace{10014\\01112}{}$,
$\subspace{10014\\01120}{}$,
$\subspace{10014\\01133}{}$,
$\subspace{10014\\01141}{}$,
$\subspace{10043\\00104}{15}$,
$\subspace{10043\\00111}{}$,
$\subspace{10043\\00123}{}$,
$\subspace{10043\\00130}{}$,
$\subspace{10043\\00142}{}$,
$\subspace{10041\\01004}{}$,
$\subspace{10041\\01013}{}$,
$\subspace{10041\\01022}{}$,
$\subspace{10041\\01031}{}$,
$\subspace{10041\\01040}{}$,
$\subspace{10021\\01100}{}$,
$\subspace{10021\\01113}{}$,
$\subspace{10021\\01121}{}$,
$\subspace{10021\\01134}{}$,
$\subspace{10021\\01142}{}$,
$\subspace{10043\\00104}{15}$,
$\subspace{10043\\00111}{}$,
$\subspace{10043\\00123}{}$,
$\subspace{10043\\00130}{}$,
$\subspace{10043\\00142}{}$,
$\subspace{10041\\01004}{}$,
$\subspace{10041\\01013}{}$,
$\subspace{10041\\01022}{}$,
$\subspace{10041\\01031}{}$,
$\subspace{10041\\01040}{}$,
$\subspace{10021\\01100}{}$,
$\subspace{10021\\01113}{}$,
$\subspace{10021\\01121}{}$,
$\subspace{10021\\01134}{}$,
$\subspace{10021\\01142}{}$,
$\subspace{10100\\00012}{3}$,
$\subspace{11000\\00014}{}$,
$\subspace{14402\\00013}{}$,
$\subspace{10104\\00012}{3}$,
$\subspace{11004\\00014}{}$,
$\subspace{14400\\00013}{}$,
$\subspace{10301\\00012}{3}$,
$\subspace{12203\\00013}{}$,
$\subspace{13001\\00014}{}$,
$\subspace{10401\\00012}{3}$,
$\subspace{11100\\00013}{}$,
$\subspace{14001\\00014}{}$,
$\subspace{10401\\00012}{3}$,
$\subspace{11100\\00013}{}$,
$\subspace{14001\\00014}{}$,
$\subspace{10401\\00012}{3}$,
$\subspace{11100\\00013}{}$,
$\subspace{14001\\00014}{}$,
$\subspace{10401\\00012}{3}$,
$\subspace{11100\\00013}{}$,
$\subspace{14001\\00014}{}$,
$\subspace{10010\\01201}{15}$,
$\subspace{10010\\01212}{}$,
$\subspace{10010\\01223}{}$,
$\subspace{10010\\01234}{}$,
$\subspace{10010\\01240}{}$,
$\subspace{10001\\01303}{}$,
$\subspace{10001\\01313}{}$,
$\subspace{10001\\01323}{}$,
$\subspace{10001\\01333}{}$,
$\subspace{10001\\01343}{}$,
$\subspace{10044\\01400}{}$,
$\subspace{10044\\01401}{}$,
$\subspace{10044\\01402}{}$,
$\subspace{10044\\01403}{}$,
$\subspace{10044\\01404}{}$,
$\subspace{10011\\01201}{15}$,
$\subspace{10011\\01212}{}$,
$\subspace{10011\\01223}{}$,
$\subspace{10011\\01234}{}$,
$\subspace{10011\\01240}{}$,
$\subspace{10040\\01303}{}$,
$\subspace{10040\\01313}{}$,
$\subspace{10040\\01323}{}$,
$\subspace{10040\\01333}{}$,
$\subspace{10040\\01343}{}$,
$\subspace{10004\\01400}{}$,
$\subspace{10004\\01401}{}$,
$\subspace{10004\\01402}{}$,
$\subspace{10004\\01403}{}$,
$\subspace{10004\\01404}{}$,
$\subspace{10011\\01201}{15}$,
$\subspace{10011\\01212}{}$,
$\subspace{10011\\01223}{}$,
$\subspace{10011\\01234}{}$,
$\subspace{10011\\01240}{}$,
$\subspace{10040\\01303}{}$,
$\subspace{10040\\01313}{}$,
$\subspace{10040\\01323}{}$,
$\subspace{10040\\01333}{}$,
$\subspace{10040\\01343}{}$,
$\subspace{10004\\01400}{}$,
$\subspace{10004\\01401}{}$,
$\subspace{10004\\01402}{}$,
$\subspace{10004\\01403}{}$,
$\subspace{10004\\01404}{}$,
$\subspace{10022\\01204}{15}$,
$\subspace{10022\\01210}{}$,
$\subspace{10022\\01221}{}$,
$\subspace{10022\\01232}{}$,
$\subspace{10022\\01243}{}$,
$\subspace{10030\\01302}{}$,
$\subspace{10030\\01312}{}$,
$\subspace{10030\\01322}{}$,
$\subspace{10030\\01332}{}$,
$\subspace{10030\\01342}{}$,
$\subspace{10003\\01430}{}$,
$\subspace{10003\\01431}{}$,
$\subspace{10003\\01432}{}$,
$\subspace{10003\\01433}{}$,
$\subspace{10003\\01434}{}$,
$\subspace{10023\\01203}{15}$,
$\subspace{10023\\01214}{}$,
$\subspace{10023\\01220}{}$,
$\subspace{10023\\01231}{}$,
$\subspace{10023\\01242}{}$,
$\subspace{10024\\01304}{}$,
$\subspace{10024\\01314}{}$,
$\subspace{10024\\01324}{}$,
$\subspace{10024\\01334}{}$,
$\subspace{10024\\01344}{}$,
$\subspace{10013\\01420}{}$,
$\subspace{10013\\01421}{}$,
$\subspace{10013\\01422}{}$,
$\subspace{10013\\01423}{}$,
$\subspace{10013\\01424}{}$,
$\subspace{10023\\01203}{15}$,
$\subspace{10023\\01214}{}$,
$\subspace{10023\\01220}{}$,
$\subspace{10023\\01231}{}$,
$\subspace{10023\\01242}{}$,
$\subspace{10024\\01304}{}$,
$\subspace{10024\\01314}{}$,
$\subspace{10024\\01324}{}$,
$\subspace{10024\\01334}{}$,
$\subspace{10024\\01344}{}$,
$\subspace{10013\\01420}{}$,
$\subspace{10013\\01421}{}$,
$\subspace{10013\\01422}{}$,
$\subspace{10013\\01423}{}$,
$\subspace{10013\\01424}{}$,
$\subspace{10030\\01204}{15}$,
$\subspace{10030\\01210}{}$,
$\subspace{10030\\01221}{}$,
$\subspace{10030\\01232}{}$,
$\subspace{10030\\01243}{}$,
$\subspace{10003\\01302}{}$,
$\subspace{10003\\01312}{}$,
$\subspace{10003\\01322}{}$,
$\subspace{10003\\01332}{}$,
$\subspace{10003\\01342}{}$,
$\subspace{10022\\01430}{}$,
$\subspace{10022\\01431}{}$,
$\subspace{10022\\01432}{}$,
$\subspace{10022\\01433}{}$,
$\subspace{10022\\01434}{}$,
$\subspace{10044\\01204}{15}$,
$\subspace{10044\\01210}{}$,
$\subspace{10044\\01221}{}$,
$\subspace{10044\\01232}{}$,
$\subspace{10044\\01243}{}$,
$\subspace{10010\\01302}{}$,
$\subspace{10010\\01312}{}$,
$\subspace{10010\\01322}{}$,
$\subspace{10010\\01332}{}$,
$\subspace{10010\\01342}{}$,
$\subspace{10001\\01430}{}$,
$\subspace{10001\\01431}{}$,
$\subspace{10001\\01432}{}$,
$\subspace{10001\\01433}{}$,
$\subspace{10001\\01434}{}$,
$\subspace{10103\\01012}{15}$,
$\subspace{10110\\01021}{}$,
$\subspace{10122\\01030}{}$,
$\subspace{10134\\01044}{}$,
$\subspace{10141\\01003}{}$,
$\subspace{10402\\01110}{}$,
$\subspace{10414\\01144}{}$,
$\subspace{10421\\01123}{}$,
$\subspace{10433\\01102}{}$,
$\subspace{10440\\01131}{}$,
$\subspace{14002\\00141}{}$,
$\subspace{14011\\00134}{}$,
$\subspace{14020\\00122}{}$,
$\subspace{14034\\00110}{}$,
$\subspace{14043\\00103}{}$,
$\subspace{10102\\01023}{15}$,
$\subspace{10114\\01032}{}$,
$\subspace{10121\\01041}{}$,
$\subspace{10133\\01000}{}$,
$\subspace{10140\\01014}{}$,
$\subspace{10403\\01140}{}$,
$\subspace{10410\\01124}{}$,
$\subspace{10422\\01103}{}$,
$\subspace{10434\\01132}{}$,
$\subspace{10441\\01111}{}$,
$\subspace{14003\\00131}{}$,
$\subspace{14012\\00124}{}$,
$\subspace{14021\\00112}{}$,
$\subspace{14030\\00100}{}$,
$\subspace{14044\\00143}{}$,
$\subspace{10102\\01023}{15}$,
$\subspace{10114\\01032}{}$,
$\subspace{10121\\01041}{}$,
$\subspace{10133\\01000}{}$,
$\subspace{10140\\01014}{}$,
$\subspace{10403\\01140}{}$,
$\subspace{10410\\01124}{}$,
$\subspace{10422\\01103}{}$,
$\subspace{10434\\01132}{}$,
$\subspace{10441\\01111}{}$,
$\subspace{14003\\00131}{}$,
$\subspace{14012\\00124}{}$,
$\subspace{14021\\00112}{}$,
$\subspace{14030\\00100}{}$,
$\subspace{14044\\00143}{}$,
$\subspace{10102\\01023}{15}$,
$\subspace{10114\\01032}{}$,
$\subspace{10121\\01041}{}$,
$\subspace{10133\\01000}{}$,
$\subspace{10140\\01014}{}$,
$\subspace{10403\\01140}{}$,
$\subspace{10410\\01124}{}$,
$\subspace{10422\\01103}{}$,
$\subspace{10434\\01132}{}$,
$\subspace{10441\\01111}{}$,
$\subspace{14003\\00131}{}$,
$\subspace{14012\\00124}{}$,
$\subspace{14021\\00112}{}$,
$\subspace{14030\\00100}{}$,
$\subspace{14044\\00143}{}$,
$\subspace{10102\\01023}{15}$,
$\subspace{10114\\01032}{}$,
$\subspace{10121\\01041}{}$,
$\subspace{10133\\01000}{}$,
$\subspace{10140\\01014}{}$,
$\subspace{10403\\01140}{}$,
$\subspace{10410\\01124}{}$,
$\subspace{10422\\01103}{}$,
$\subspace{10434\\01132}{}$,
$\subspace{10441\\01111}{}$,
$\subspace{14003\\00131}{}$,
$\subspace{14012\\00124}{}$,
$\subspace{14021\\00112}{}$,
$\subspace{14030\\00100}{}$,
$\subspace{14044\\00143}{}$,
$\subspace{10102\\01023}{15}$,
$\subspace{10114\\01032}{}$,
$\subspace{10121\\01041}{}$,
$\subspace{10133\\01000}{}$,
$\subspace{10140\\01014}{}$,
$\subspace{10403\\01140}{}$,
$\subspace{10410\\01124}{}$,
$\subspace{10422\\01103}{}$,
$\subspace{10434\\01132}{}$,
$\subspace{10441\\01111}{}$,
$\subspace{14003\\00131}{}$,
$\subspace{14012\\00124}{}$,
$\subspace{14021\\00112}{}$,
$\subspace{14030\\00100}{}$,
$\subspace{14044\\00143}{}$,
$\subspace{10102\\01023}{15}$,
$\subspace{10114\\01032}{}$,
$\subspace{10121\\01041}{}$,
$\subspace{10133\\01000}{}$,
$\subspace{10140\\01014}{}$,
$\subspace{10403\\01140}{}$,
$\subspace{10410\\01124}{}$,
$\subspace{10422\\01103}{}$,
$\subspace{10434\\01132}{}$,
$\subspace{10441\\01111}{}$,
$\subspace{14003\\00131}{}$,
$\subspace{14012\\00124}{}$,
$\subspace{14021\\00112}{}$,
$\subspace{14030\\00100}{}$,
$\subspace{14044\\00143}{}$,
$\subspace{10102\\01042}{15}$,
$\subspace{10114\\01001}{}$,
$\subspace{10121\\01010}{}$,
$\subspace{10133\\01024}{}$,
$\subspace{10140\\01033}{}$,
$\subspace{10402\\01114}{}$,
$\subspace{10414\\01143}{}$,
$\subspace{10421\\01122}{}$,
$\subspace{10433\\01101}{}$,
$\subspace{10440\\01130}{}$,
$\subspace{14002\\00101}{}$,
$\subspace{14011\\00144}{}$,
$\subspace{14020\\00132}{}$,
$\subspace{14034\\00120}{}$,
$\subspace{14043\\00113}{}$,
$\subspace{10102\\01042}{15}$,
$\subspace{10114\\01001}{}$,
$\subspace{10121\\01010}{}$,
$\subspace{10133\\01024}{}$,
$\subspace{10140\\01033}{}$,
$\subspace{10402\\01114}{}$,
$\subspace{10414\\01143}{}$,
$\subspace{10421\\01122}{}$,
$\subspace{10433\\01101}{}$,
$\subspace{10440\\01130}{}$,
$\subspace{14002\\00101}{}$,
$\subspace{14011\\00144}{}$,
$\subspace{14020\\00132}{}$,
$\subspace{14034\\00120}{}$,
$\subspace{14043\\00113}{}$,
$\subspace{10102\\01042}{15}$,
$\subspace{10114\\01001}{}$,
$\subspace{10121\\01010}{}$,
$\subspace{10133\\01024}{}$,
$\subspace{10140\\01033}{}$,
$\subspace{10402\\01114}{}$,
$\subspace{10414\\01143}{}$,
$\subspace{10421\\01122}{}$,
$\subspace{10433\\01101}{}$,
$\subspace{10440\\01130}{}$,
$\subspace{14002\\00101}{}$,
$\subspace{14011\\00144}{}$,
$\subspace{14020\\00132}{}$,
$\subspace{14034\\00120}{}$,
$\subspace{14043\\00113}{}$,
$\subspace{10100\\01100}{15}$,
$\subspace{10112\\01121}{}$,
$\subspace{10124\\01142}{}$,
$\subspace{10131\\01113}{}$,
$\subspace{10143\\01134}{}$,
$\subspace{10400\\01040}{}$,
$\subspace{10412\\01031}{}$,
$\subspace{10424\\01022}{}$,
$\subspace{10431\\01013}{}$,
$\subspace{10443\\01004}{}$,
$\subspace{11000\\00111}{}$,
$\subspace{11014\\00123}{}$,
$\subspace{11023\\00130}{}$,
$\subspace{11032\\00142}{}$,
$\subspace{11041\\00104}{}$,
$\subspace{10100\\01100}{15}$,
$\subspace{10112\\01121}{}$,
$\subspace{10124\\01142}{}$,
$\subspace{10131\\01113}{}$,
$\subspace{10143\\01134}{}$,
$\subspace{10400\\01040}{}$,
$\subspace{10412\\01031}{}$,
$\subspace{10424\\01022}{}$,
$\subspace{10431\\01013}{}$,
$\subspace{10443\\01004}{}$,
$\subspace{11000\\00111}{}$,
$\subspace{11014\\00123}{}$,
$\subspace{11023\\00130}{}$,
$\subspace{11032\\00142}{}$,
$\subspace{11041\\00104}{}$,
$\subspace{10100\\01100}{15}$,
$\subspace{10112\\01121}{}$,
$\subspace{10124\\01142}{}$,
$\subspace{10131\\01113}{}$,
$\subspace{10143\\01134}{}$,
$\subspace{10400\\01040}{}$,
$\subspace{10412\\01031}{}$,
$\subspace{10424\\01022}{}$,
$\subspace{10431\\01013}{}$,
$\subspace{10443\\01004}{}$,
$\subspace{11000\\00111}{}$,
$\subspace{11014\\00123}{}$,
$\subspace{11023\\00130}{}$,
$\subspace{11032\\00142}{}$,
$\subspace{11041\\00104}{}$,
$\subspace{10100\\01100}{15}$,
$\subspace{10112\\01121}{}$,
$\subspace{10124\\01142}{}$,
$\subspace{10131\\01113}{}$,
$\subspace{10143\\01134}{}$,
$\subspace{10400\\01040}{}$,
$\subspace{10412\\01031}{}$,
$\subspace{10424\\01022}{}$,
$\subspace{10431\\01013}{}$,
$\subspace{10443\\01004}{}$,
$\subspace{11000\\00111}{}$,
$\subspace{11014\\00123}{}$,
$\subspace{11023\\00130}{}$,
$\subspace{11032\\00142}{}$,
$\subspace{11041\\00104}{}$,
$\subspace{10100\\01100}{15}$,
$\subspace{10112\\01121}{}$,
$\subspace{10124\\01142}{}$,
$\subspace{10131\\01113}{}$,
$\subspace{10143\\01134}{}$,
$\subspace{10400\\01040}{}$,
$\subspace{10412\\01031}{}$,
$\subspace{10424\\01022}{}$,
$\subspace{10431\\01013}{}$,
$\subspace{10443\\01004}{}$,
$\subspace{11000\\00111}{}$,
$\subspace{11014\\00123}{}$,
$\subspace{11023\\00130}{}$,
$\subspace{11032\\00142}{}$,
$\subspace{11041\\00104}{}$,
$\subspace{10104\\01110}{15}$,
$\subspace{10111\\01131}{}$,
$\subspace{10123\\01102}{}$,
$\subspace{10130\\01123}{}$,
$\subspace{10142\\01144}{}$,
$\subspace{10403\\01003}{}$,
$\subspace{10410\\01044}{}$,
$\subspace{10422\\01030}{}$,
$\subspace{10434\\01021}{}$,
$\subspace{10441\\01012}{}$,
$\subspace{11004\\00134}{}$,
$\subspace{11013\\00141}{}$,
$\subspace{11022\\00103}{}$,
$\subspace{11031\\00110}{}$,
$\subspace{11040\\00122}{}$,
$\subspace{10103\\01140}{15}$,
$\subspace{10110\\01111}{}$,
$\subspace{10122\\01132}{}$,
$\subspace{10134\\01103}{}$,
$\subspace{10141\\01124}{}$,
$\subspace{10404\\01014}{}$,
$\subspace{10411\\01000}{}$,
$\subspace{10423\\01041}{}$,
$\subspace{10430\\01032}{}$,
$\subspace{10442\\01023}{}$,
$\subspace{11003\\00124}{}$,
$\subspace{11012\\00131}{}$,
$\subspace{11021\\00143}{}$,
$\subspace{11030\\00100}{}$,
$\subspace{11044\\00112}{}$,
$\subspace{10103\\01140}{15}$,
$\subspace{10110\\01111}{}$,
$\subspace{10122\\01132}{}$,
$\subspace{10134\\01103}{}$,
$\subspace{10141\\01124}{}$,
$\subspace{10404\\01014}{}$,
$\subspace{10411\\01000}{}$,
$\subspace{10423\\01041}{}$,
$\subspace{10430\\01032}{}$,
$\subspace{10442\\01023}{}$,
$\subspace{11003\\00124}{}$,
$\subspace{11012\\00131}{}$,
$\subspace{11021\\00143}{}$,
$\subspace{11030\\00100}{}$,
$\subspace{11044\\00112}{}$,
$\subspace{10103\\01140}{15}$,
$\subspace{10110\\01111}{}$,
$\subspace{10122\\01132}{}$,
$\subspace{10134\\01103}{}$,
$\subspace{10141\\01124}{}$,
$\subspace{10404\\01014}{}$,
$\subspace{10411\\01000}{}$,
$\subspace{10423\\01041}{}$,
$\subspace{10430\\01032}{}$,
$\subspace{10442\\01023}{}$,
$\subspace{11003\\00124}{}$,
$\subspace{11012\\00131}{}$,
$\subspace{11021\\00143}{}$,
$\subspace{11030\\00100}{}$,
$\subspace{11044\\00112}{}$,
$\subspace{10103\\01140}{15}$,
$\subspace{10110\\01111}{}$,
$\subspace{10122\\01132}{}$,
$\subspace{10134\\01103}{}$,
$\subspace{10141\\01124}{}$,
$\subspace{10404\\01014}{}$,
$\subspace{10411\\01000}{}$,
$\subspace{10423\\01041}{}$,
$\subspace{10430\\01032}{}$,
$\subspace{10442\\01023}{}$,
$\subspace{11003\\00124}{}$,
$\subspace{11012\\00131}{}$,
$\subspace{11021\\00143}{}$,
$\subspace{11030\\00100}{}$,
$\subspace{11044\\00112}{}$,
$\subspace{10103\\01203}{15}$,
$\subspace{10110\\01231}{}$,
$\subspace{10122\\01214}{}$,
$\subspace{10134\\01242}{}$,
$\subspace{10141\\01220}{}$,
$\subspace{10104\\01422}{}$,
$\subspace{10111\\01424}{}$,
$\subspace{10123\\01421}{}$,
$\subspace{10130\\01423}{}$,
$\subspace{10142\\01420}{}$,
$\subspace{10204\\01314}{}$,
$\subspace{10211\\01334}{}$,
$\subspace{10223\\01304}{}$,
$\subspace{10230\\01324}{}$,
$\subspace{10242\\01344}{}$,
$\subspace{10100\\01211}{15}$,
$\subspace{10112\\01244}{}$,
$\subspace{10124\\01222}{}$,
$\subspace{10131\\01200}{}$,
$\subspace{10143\\01233}{}$,
$\subspace{10100\\01441}{}$,
$\subspace{10112\\01443}{}$,
$\subspace{10124\\01440}{}$,
$\subspace{10131\\01442}{}$,
$\subspace{10143\\01444}{}$,
$\subspace{10204\\01310}{}$,
$\subspace{10211\\01330}{}$,
$\subspace{10223\\01300}{}$,
$\subspace{10230\\01320}{}$,
$\subspace{10242\\01340}{}$,
$\subspace{10103\\01213}{15}$,
$\subspace{10110\\01241}{}$,
$\subspace{10122\\01224}{}$,
$\subspace{10134\\01202}{}$,
$\subspace{10141\\01230}{}$,
$\subspace{10103\\01414}{}$,
$\subspace{10110\\01411}{}$,
$\subspace{10122\\01413}{}$,
$\subspace{10134\\01410}{}$,
$\subspace{10141\\01412}{}$,
$\subspace{10201\\01311}{}$,
$\subspace{10213\\01331}{}$,
$\subspace{10220\\01301}{}$,
$\subspace{10232\\01321}{}$,
$\subspace{10244\\01341}{}$,
$\subspace{10101\\01231}{15}$,
$\subspace{10113\\01214}{}$,
$\subspace{10120\\01242}{}$,
$\subspace{10132\\01220}{}$,
$\subspace{10144\\01203}{}$,
$\subspace{10101\\01423}{}$,
$\subspace{10113\\01420}{}$,
$\subspace{10120\\01422}{}$,
$\subspace{10132\\01424}{}$,
$\subspace{10144\\01421}{}$,
$\subspace{10204\\01334}{}$,
$\subspace{10211\\01304}{}$,
$\subspace{10223\\01324}{}$,
$\subspace{10230\\01344}{}$,
$\subspace{10242\\01314}{}$,
$\subspace{10101\\01231}{15}$,
$\subspace{10113\\01214}{}$,
$\subspace{10120\\01242}{}$,
$\subspace{10132\\01220}{}$,
$\subspace{10144\\01203}{}$,
$\subspace{10101\\01423}{}$,
$\subspace{10113\\01420}{}$,
$\subspace{10120\\01422}{}$,
$\subspace{10132\\01424}{}$,
$\subspace{10144\\01421}{}$,
$\subspace{10204\\01334}{}$,
$\subspace{10211\\01304}{}$,
$\subspace{10223\\01324}{}$,
$\subspace{10230\\01344}{}$,
$\subspace{10242\\01314}{}$,
$\subspace{10101\\01231}{15}$,
$\subspace{10113\\01214}{}$,
$\subspace{10120\\01242}{}$,
$\subspace{10132\\01220}{}$,
$\subspace{10144\\01203}{}$,
$\subspace{10101\\01423}{}$,
$\subspace{10113\\01420}{}$,
$\subspace{10120\\01422}{}$,
$\subspace{10132\\01424}{}$,
$\subspace{10144\\01421}{}$,
$\subspace{10204\\01334}{}$,
$\subspace{10211\\01304}{}$,
$\subspace{10223\\01324}{}$,
$\subspace{10230\\01344}{}$,
$\subspace{10242\\01314}{}$,
$\subspace{10102\\01234}{15}$,
$\subspace{10114\\01212}{}$,
$\subspace{10121\\01240}{}$,
$\subspace{10133\\01223}{}$,
$\subspace{10140\\01201}{}$,
$\subspace{10100\\01402}{}$,
$\subspace{10112\\01404}{}$,
$\subspace{10124\\01401}{}$,
$\subspace{10131\\01403}{}$,
$\subspace{10143\\01400}{}$,
$\subspace{10201\\01313}{}$,
$\subspace{10213\\01333}{}$,
$\subspace{10220\\01303}{}$,
$\subspace{10232\\01323}{}$,
$\subspace{10244\\01343}{}$,
$\subspace{10102\\01234}{15}$,
$\subspace{10114\\01212}{}$,
$\subspace{10121\\01240}{}$,
$\subspace{10133\\01223}{}$,
$\subspace{10140\\01201}{}$,
$\subspace{10100\\01402}{}$,
$\subspace{10112\\01404}{}$,
$\subspace{10124\\01401}{}$,
$\subspace{10131\\01403}{}$,
$\subspace{10143\\01400}{}$,
$\subspace{10201\\01313}{}$,
$\subspace{10213\\01333}{}$,
$\subspace{10220\\01303}{}$,
$\subspace{10232\\01323}{}$,
$\subspace{10244\\01343}{}$,
$\subspace{10104\\01241}{15}$,
$\subspace{10111\\01224}{}$,
$\subspace{10123\\01202}{}$,
$\subspace{10130\\01230}{}$,
$\subspace{10142\\01213}{}$,
$\subspace{10104\\01414}{}$,
$\subspace{10111\\01411}{}$,
$\subspace{10123\\01413}{}$,
$\subspace{10130\\01410}{}$,
$\subspace{10142\\01412}{}$,
$\subspace{10204\\01341}{}$,
$\subspace{10211\\01311}{}$,
$\subspace{10223\\01331}{}$,
$\subspace{10230\\01301}{}$,
$\subspace{10242\\01321}{}$,
$\subspace{10104\\01241}{15}$,
$\subspace{10111\\01224}{}$,
$\subspace{10123\\01202}{}$,
$\subspace{10130\\01230}{}$,
$\subspace{10142\\01213}{}$,
$\subspace{10104\\01414}{}$,
$\subspace{10111\\01411}{}$,
$\subspace{10123\\01413}{}$,
$\subspace{10130\\01410}{}$,
$\subspace{10142\\01412}{}$,
$\subspace{10204\\01341}{}$,
$\subspace{10211\\01311}{}$,
$\subspace{10223\\01331}{}$,
$\subspace{10230\\01301}{}$,
$\subspace{10242\\01321}{}$,
$\subspace{10102\\01301}{15}$,
$\subspace{10114\\01341}{}$,
$\subspace{10121\\01331}{}$,
$\subspace{10133\\01321}{}$,
$\subspace{10140\\01311}{}$,
$\subspace{10304\\01230}{}$,
$\subspace{10311\\01241}{}$,
$\subspace{10323\\01202}{}$,
$\subspace{10330\\01213}{}$,
$\subspace{10342\\01224}{}$,
$\subspace{10300\\01412}{}$,
$\subspace{10312\\01411}{}$,
$\subspace{10324\\01410}{}$,
$\subspace{10331\\01414}{}$,
$\subspace{10343\\01413}{}$,
$\subspace{10101\\01322}{15}$,
$\subspace{10113\\01312}{}$,
$\subspace{10120\\01302}{}$,
$\subspace{10132\\01342}{}$,
$\subspace{10144\\01332}{}$,
$\subspace{10301\\01232}{}$,
$\subspace{10313\\01243}{}$,
$\subspace{10320\\01204}{}$,
$\subspace{10332\\01210}{}$,
$\subspace{10344\\01221}{}$,
$\subspace{10303\\01433}{}$,
$\subspace{10310\\01432}{}$,
$\subspace{10322\\01431}{}$,
$\subspace{10334\\01430}{}$,
$\subspace{10341\\01434}{}$,
$\subspace{10101\\01322}{15}$,
$\subspace{10113\\01312}{}$,
$\subspace{10120\\01302}{}$,
$\subspace{10132\\01342}{}$,
$\subspace{10144\\01332}{}$,
$\subspace{10301\\01232}{}$,
$\subspace{10313\\01243}{}$,
$\subspace{10320\\01204}{}$,
$\subspace{10332\\01210}{}$,
$\subspace{10344\\01221}{}$,
$\subspace{10303\\01433}{}$,
$\subspace{10310\\01432}{}$,
$\subspace{10322\\01431}{}$,
$\subspace{10334\\01430}{}$,
$\subspace{10341\\01434}{}$,
$\subspace{10101\\01323}{15}$,
$\subspace{10113\\01313}{}$,
$\subspace{10120\\01303}{}$,
$\subspace{10132\\01343}{}$,
$\subspace{10144\\01333}{}$,
$\subspace{10302\\01240}{}$,
$\subspace{10314\\01201}{}$,
$\subspace{10321\\01212}{}$,
$\subspace{10333\\01223}{}$,
$\subspace{10340\\01234}{}$,
$\subspace{10301\\01402}{}$,
$\subspace{10313\\01401}{}$,
$\subspace{10320\\01400}{}$,
$\subspace{10332\\01404}{}$,
$\subspace{10344\\01403}{}$,
$\subspace{10101\\01323}{15}$,
$\subspace{10113\\01313}{}$,
$\subspace{10120\\01303}{}$,
$\subspace{10132\\01343}{}$,
$\subspace{10144\\01333}{}$,
$\subspace{10302\\01240}{}$,
$\subspace{10314\\01201}{}$,
$\subspace{10321\\01212}{}$,
$\subspace{10333\\01223}{}$,
$\subspace{10340\\01234}{}$,
$\subspace{10301\\01402}{}$,
$\subspace{10313\\01401}{}$,
$\subspace{10320\\01400}{}$,
$\subspace{10332\\01404}{}$,
$\subspace{10344\\01403}{}$,
$\subspace{10101\\01323}{15}$,
$\subspace{10113\\01313}{}$,
$\subspace{10120\\01303}{}$,
$\subspace{10132\\01343}{}$,
$\subspace{10144\\01333}{}$,
$\subspace{10302\\01240}{}$,
$\subspace{10314\\01201}{}$,
$\subspace{10321\\01212}{}$,
$\subspace{10333\\01223}{}$,
$\subspace{10340\\01234}{}$,
$\subspace{10301\\01402}{}$,
$\subspace{10313\\01401}{}$,
$\subspace{10320\\01400}{}$,
$\subspace{10332\\01404}{}$,
$\subspace{10344\\01403}{}$,
$\subspace{10101\\01323}{15}$,
$\subspace{10113\\01313}{}$,
$\subspace{10120\\01303}{}$,
$\subspace{10132\\01343}{}$,
$\subspace{10144\\01333}{}$,
$\subspace{10302\\01240}{}$,
$\subspace{10314\\01201}{}$,
$\subspace{10321\\01212}{}$,
$\subspace{10333\\01223}{}$,
$\subspace{10340\\01234}{}$,
$\subspace{10301\\01402}{}$,
$\subspace{10313\\01401}{}$,
$\subspace{10320\\01400}{}$,
$\subspace{10332\\01404}{}$,
$\subspace{10344\\01403}{}$,
$\subspace{10104\\01320}{15}$,
$\subspace{10111\\01310}{}$,
$\subspace{10123\\01300}{}$,
$\subspace{10130\\01340}{}$,
$\subspace{10142\\01330}{}$,
$\subspace{10303\\01233}{}$,
$\subspace{10310\\01244}{}$,
$\subspace{10322\\01200}{}$,
$\subspace{10334\\01211}{}$,
$\subspace{10341\\01222}{}$,
$\subspace{10300\\01442}{}$,
$\subspace{10312\\01441}{}$,
$\subspace{10324\\01440}{}$,
$\subspace{10331\\01444}{}$,
$\subspace{10343\\01443}{}$,
$\subspace{10104\\01321}{15}$,
$\subspace{10111\\01311}{}$,
$\subspace{10123\\01301}{}$,
$\subspace{10130\\01341}{}$,
$\subspace{10142\\01331}{}$,
$\subspace{10304\\01241}{}$,
$\subspace{10311\\01202}{}$,
$\subspace{10323\\01213}{}$,
$\subspace{10330\\01224}{}$,
$\subspace{10342\\01230}{}$,
$\subspace{10303\\01411}{}$,
$\subspace{10310\\01410}{}$,
$\subspace{10322\\01414}{}$,
$\subspace{10334\\01413}{}$,
$\subspace{10341\\01412}{}$,
$\subspace{10104\\01321}{15}$,
$\subspace{10111\\01311}{}$,
$\subspace{10123\\01301}{}$,
$\subspace{10130\\01341}{}$,
$\subspace{10142\\01331}{}$,
$\subspace{10304\\01241}{}$,
$\subspace{10311\\01202}{}$,
$\subspace{10323\\01213}{}$,
$\subspace{10330\\01224}{}$,
$\subspace{10342\\01230}{}$,
$\subspace{10303\\01411}{}$,
$\subspace{10310\\01410}{}$,
$\subspace{10322\\01414}{}$,
$\subspace{10334\\01413}{}$,
$\subspace{10341\\01412}{}$,
$\subspace{10203\\01001}{15}$,
$\subspace{10210\\01033}{}$,
$\subspace{10222\\01010}{}$,
$\subspace{10234\\01042}{}$,
$\subspace{10241\\01024}{}$,
$\subspace{10302\\01130}{}$,
$\subspace{10314\\01122}{}$,
$\subspace{10321\\01114}{}$,
$\subspace{10333\\01101}{}$,
$\subspace{10340\\01143}{}$,
$\subspace{13002\\00101}{}$,
$\subspace{13011\\00120}{}$,
$\subspace{13020\\00144}{}$,
$\subspace{13034\\00113}{}$,
$\subspace{13043\\00132}{}$,
$\subspace{10202\\01010}{15}$,
$\subspace{10214\\01042}{}$,
$\subspace{10221\\01024}{}$,
$\subspace{10233\\01001}{}$,
$\subspace{10240\\01033}{}$,
$\subspace{10300\\01101}{}$,
$\subspace{10312\\01143}{}$,
$\subspace{10324\\01130}{}$,
$\subspace{10331\\01122}{}$,
$\subspace{10343\\01114}{}$,
$\subspace{13000\\00101}{}$,
$\subspace{13014\\00120}{}$,
$\subspace{13023\\00144}{}$,
$\subspace{13032\\00113}{}$,
$\subspace{13041\\00132}{}$,
$\subspace{10202\\01010}{15}$,
$\subspace{10214\\01042}{}$,
$\subspace{10221\\01024}{}$,
$\subspace{10233\\01001}{}$,
$\subspace{10240\\01033}{}$,
$\subspace{10300\\01101}{}$,
$\subspace{10312\\01143}{}$,
$\subspace{10324\\01130}{}$,
$\subspace{10331\\01122}{}$,
$\subspace{10343\\01114}{}$,
$\subspace{13000\\00101}{}$,
$\subspace{13014\\00120}{}$,
$\subspace{13023\\00144}{}$,
$\subspace{13032\\00113}{}$,
$\subspace{13041\\00132}{}$,
$\subspace{10202\\01021}{15}$,
$\subspace{10214\\01003}{}$,
$\subspace{10221\\01030}{}$,
$\subspace{10233\\01012}{}$,
$\subspace{10240\\01044}{}$,
$\subspace{10303\\01102}{}$,
$\subspace{10310\\01144}{}$,
$\subspace{10322\\01131}{}$,
$\subspace{10334\\01123}{}$,
$\subspace{10341\\01110}{}$,
$\subspace{13003\\00134}{}$,
$\subspace{13012\\00103}{}$,
$\subspace{13021\\00122}{}$,
$\subspace{13030\\00141}{}$,
$\subspace{13044\\00110}{}$,
$\subspace{10202\\01021}{15}$,
$\subspace{10214\\01003}{}$,
$\subspace{10221\\01030}{}$,
$\subspace{10233\\01012}{}$,
$\subspace{10240\\01044}{}$,
$\subspace{10303\\01102}{}$,
$\subspace{10310\\01144}{}$,
$\subspace{10322\\01131}{}$,
$\subspace{10334\\01123}{}$,
$\subspace{10341\\01110}{}$,
$\subspace{13003\\00134}{}$,
$\subspace{13012\\00103}{}$,
$\subspace{13021\\00122}{}$,
$\subspace{13030\\00141}{}$,
$\subspace{13044\\00110}{}$,
$\subspace{10202\\01021}{15}$,
$\subspace{10214\\01003}{}$,
$\subspace{10221\\01030}{}$,
$\subspace{10233\\01012}{}$,
$\subspace{10240\\01044}{}$,
$\subspace{10303\\01102}{}$,
$\subspace{10310\\01144}{}$,
$\subspace{10322\\01131}{}$,
$\subspace{10334\\01123}{}$,
$\subspace{10341\\01110}{}$,
$\subspace{13003\\00134}{}$,
$\subspace{13012\\00103}{}$,
$\subspace{13021\\00122}{}$,
$\subspace{13030\\00141}{}$,
$\subspace{13044\\00110}{}$,
$\subspace{10203\\01104}{15}$,
$\subspace{10210\\01112}{}$,
$\subspace{10222\\01120}{}$,
$\subspace{10234\\01133}{}$,
$\subspace{10241\\01141}{}$,
$\subspace{10300\\01034}{}$,
$\subspace{10312\\01002}{}$,
$\subspace{10324\\01020}{}$,
$\subspace{10331\\01043}{}$,
$\subspace{10343\\01011}{}$,
$\subspace{12003\\00133}{}$,
$\subspace{12012\\00114}{}$,
$\subspace{12021\\00140}{}$,
$\subspace{12030\\00121}{}$,
$\subspace{12044\\00102}{}$,
$\subspace{10203\\01104}{15}$,
$\subspace{10210\\01112}{}$,
$\subspace{10222\\01120}{}$,
$\subspace{10234\\01133}{}$,
$\subspace{10241\\01141}{}$,
$\subspace{10300\\01034}{}$,
$\subspace{10312\\01002}{}$,
$\subspace{10324\\01020}{}$,
$\subspace{10331\\01043}{}$,
$\subspace{10343\\01011}{}$,
$\subspace{12003\\00133}{}$,
$\subspace{12012\\00114}{}$,
$\subspace{12021\\00140}{}$,
$\subspace{12030\\00121}{}$,
$\subspace{12044\\00102}{}$,
$\subspace{10203\\01104}{15}$,
$\subspace{10210\\01112}{}$,
$\subspace{10222\\01120}{}$,
$\subspace{10234\\01133}{}$,
$\subspace{10241\\01141}{}$,
$\subspace{10300\\01034}{}$,
$\subspace{10312\\01002}{}$,
$\subspace{10324\\01020}{}$,
$\subspace{10331\\01043}{}$,
$\subspace{10343\\01011}{}$,
$\subspace{12003\\00133}{}$,
$\subspace{12012\\00114}{}$,
$\subspace{12021\\00140}{}$,
$\subspace{12030\\00121}{}$,
$\subspace{12044\\00102}{}$,
$\subspace{10201\\01133}{15}$,
$\subspace{10213\\01141}{}$,
$\subspace{10220\\01104}{}$,
$\subspace{10232\\01112}{}$,
$\subspace{10244\\01120}{}$,
$\subspace{10304\\01020}{}$,
$\subspace{10311\\01043}{}$,
$\subspace{10323\\01011}{}$,
$\subspace{10330\\01034}{}$,
$\subspace{10342\\01002}{}$,
$\subspace{12001\\00133}{}$,
$\subspace{12010\\00114}{}$,
$\subspace{12024\\00140}{}$,
$\subspace{12033\\00121}{}$,
$\subspace{12042\\00102}{}$,
$\subspace{10201\\01133}{15}$,
$\subspace{10213\\01141}{}$,
$\subspace{10220\\01104}{}$,
$\subspace{10232\\01112}{}$,
$\subspace{10244\\01120}{}$,
$\subspace{10304\\01020}{}$,
$\subspace{10311\\01043}{}$,
$\subspace{10323\\01011}{}$,
$\subspace{10330\\01034}{}$,
$\subspace{10342\\01002}{}$,
$\subspace{12001\\00133}{}$,
$\subspace{12010\\00114}{}$,
$\subspace{12024\\00140}{}$,
$\subspace{12033\\00121}{}$,
$\subspace{12042\\00102}{}$,
$\subspace{10201\\01133}{15}$,
$\subspace{10213\\01141}{}$,
$\subspace{10220\\01104}{}$,
$\subspace{10232\\01112}{}$,
$\subspace{10244\\01120}{}$,
$\subspace{10304\\01020}{}$,
$\subspace{10311\\01043}{}$,
$\subspace{10323\\01011}{}$,
$\subspace{10330\\01034}{}$,
$\subspace{10342\\01002}{}$,
$\subspace{12001\\00133}{}$,
$\subspace{12010\\00114}{}$,
$\subspace{12024\\00140}{}$,
$\subspace{12033\\00121}{}$,
$\subspace{12042\\00102}{}$,
$\subspace{10201\\01133}{15}$,
$\subspace{10213\\01141}{}$,
$\subspace{10220\\01104}{}$,
$\subspace{10232\\01112}{}$,
$\subspace{10244\\01120}{}$,
$\subspace{10304\\01020}{}$,
$\subspace{10311\\01043}{}$,
$\subspace{10323\\01011}{}$,
$\subspace{10330\\01034}{}$,
$\subspace{10342\\01002}{}$,
$\subspace{12001\\00133}{}$,
$\subspace{12010\\00114}{}$,
$\subspace{12024\\00140}{}$,
$\subspace{12033\\00121}{}$,
$\subspace{12042\\00102}{}$,
$\subspace{10201\\01134}{15}$,
$\subspace{10213\\01142}{}$,
$\subspace{10220\\01100}{}$,
$\subspace{10232\\01113}{}$,
$\subspace{10244\\01121}{}$,
$\subspace{10302\\01040}{}$,
$\subspace{10314\\01013}{}$,
$\subspace{10321\\01031}{}$,
$\subspace{10333\\01004}{}$,
$\subspace{10340\\01022}{}$,
$\subspace{12001\\00123}{}$,
$\subspace{12010\\00104}{}$,
$\subspace{12024\\00130}{}$,
$\subspace{12033\\00111}{}$,
$\subspace{12042\\00142}{}$,
$\subspace{10203\\01142}{15}$,
$\subspace{10210\\01100}{}$,
$\subspace{10222\\01113}{}$,
$\subspace{10234\\01121}{}$,
$\subspace{10241\\01134}{}$,
$\subspace{10301\\01040}{}$,
$\subspace{10313\\01013}{}$,
$\subspace{10320\\01031}{}$,
$\subspace{10332\\01004}{}$,
$\subspace{10344\\01022}{}$,
$\subspace{12003\\00142}{}$,
$\subspace{12012\\00123}{}$,
$\subspace{12021\\00104}{}$,
$\subspace{12030\\00130}{}$,
$\subspace{12044\\00111}{}$,
$\subspace{10203\\01142}{15}$,
$\subspace{10210\\01100}{}$,
$\subspace{10222\\01113}{}$,
$\subspace{10234\\01121}{}$,
$\subspace{10241\\01134}{}$,
$\subspace{10301\\01040}{}$,
$\subspace{10313\\01013}{}$,
$\subspace{10320\\01031}{}$,
$\subspace{10332\\01004}{}$,
$\subspace{10344\\01022}{}$,
$\subspace{12003\\00142}{}$,
$\subspace{12012\\00123}{}$,
$\subspace{12021\\00104}{}$,
$\subspace{12030\\00130}{}$,
$\subspace{12044\\00111}{}$,
$\subspace{10203\\01142}{15}$,
$\subspace{10210\\01100}{}$,
$\subspace{10222\\01113}{}$,
$\subspace{10234\\01121}{}$,
$\subspace{10241\\01134}{}$,
$\subspace{10301\\01040}{}$,
$\subspace{10313\\01013}{}$,
$\subspace{10320\\01031}{}$,
$\subspace{10332\\01004}{}$,
$\subspace{10344\\01022}{}$,
$\subspace{12003\\00142}{}$,
$\subspace{12012\\00123}{}$,
$\subspace{12021\\00104}{}$,
$\subspace{12030\\00130}{}$,
$\subspace{12044\\00111}{}$,
$\subspace{10201\\01204}{15}$,
$\subspace{10213\\01243}{}$,
$\subspace{10220\\01232}{}$,
$\subspace{10232\\01221}{}$,
$\subspace{10244\\01210}{}$,
$\subspace{10202\\01434}{}$,
$\subspace{10214\\01430}{}$,
$\subspace{10221\\01431}{}$,
$\subspace{10233\\01432}{}$,
$\subspace{10240\\01433}{}$,
$\subspace{10402\\01332}{}$,
$\subspace{10414\\01342}{}$,
$\subspace{10421\\01302}{}$,
$\subspace{10433\\01312}{}$,
$\subspace{10440\\01322}{}$,
$\subspace{10200\\01211}{15}$,
$\subspace{10212\\01200}{}$,
$\subspace{10224\\01244}{}$,
$\subspace{10231\\01233}{}$,
$\subspace{10243\\01222}{}$,
$\subspace{10200\\01441}{}$,
$\subspace{10212\\01442}{}$,
$\subspace{10224\\01443}{}$,
$\subspace{10231\\01444}{}$,
$\subspace{10243\\01440}{}$,
$\subspace{10403\\01310}{}$,
$\subspace{10410\\01320}{}$,
$\subspace{10422\\01330}{}$,
$\subspace{10434\\01340}{}$,
$\subspace{10441\\01300}{}$,
$\subspace{10200\\01211}{15}$,
$\subspace{10212\\01200}{}$,
$\subspace{10224\\01244}{}$,
$\subspace{10231\\01233}{}$,
$\subspace{10243\\01222}{}$,
$\subspace{10200\\01441}{}$,
$\subspace{10212\\01442}{}$,
$\subspace{10224\\01443}{}$,
$\subspace{10231\\01444}{}$,
$\subspace{10243\\01440}{}$,
$\subspace{10403\\01310}{}$,
$\subspace{10410\\01320}{}$,
$\subspace{10422\\01330}{}$,
$\subspace{10434\\01340}{}$,
$\subspace{10441\\01300}{}$,
$\subspace{10200\\01211}{15}$,
$\subspace{10212\\01200}{}$,
$\subspace{10224\\01244}{}$,
$\subspace{10231\\01233}{}$,
$\subspace{10243\\01222}{}$,
$\subspace{10200\\01441}{}$,
$\subspace{10212\\01442}{}$,
$\subspace{10224\\01443}{}$,
$\subspace{10231\\01444}{}$,
$\subspace{10243\\01440}{}$,
$\subspace{10403\\01310}{}$,
$\subspace{10410\\01320}{}$,
$\subspace{10422\\01330}{}$,
$\subspace{10434\\01340}{}$,
$\subspace{10441\\01300}{}$,
$\subspace{10200\\01211}{15}$,
$\subspace{10212\\01200}{}$,
$\subspace{10224\\01244}{}$,
$\subspace{10231\\01233}{}$,
$\subspace{10243\\01222}{}$,
$\subspace{10200\\01441}{}$,
$\subspace{10212\\01442}{}$,
$\subspace{10224\\01443}{}$,
$\subspace{10231\\01444}{}$,
$\subspace{10243\\01440}{}$,
$\subspace{10403\\01310}{}$,
$\subspace{10410\\01320}{}$,
$\subspace{10422\\01330}{}$,
$\subspace{10434\\01340}{}$,
$\subspace{10441\\01300}{}$,
$\subspace{10200\\01211}{15}$,
$\subspace{10212\\01200}{}$,
$\subspace{10224\\01244}{}$,
$\subspace{10231\\01233}{}$,
$\subspace{10243\\01222}{}$,
$\subspace{10200\\01441}{}$,
$\subspace{10212\\01442}{}$,
$\subspace{10224\\01443}{}$,
$\subspace{10231\\01444}{}$,
$\subspace{10243\\01440}{}$,
$\subspace{10403\\01310}{}$,
$\subspace{10410\\01320}{}$,
$\subspace{10422\\01330}{}$,
$\subspace{10434\\01340}{}$,
$\subspace{10441\\01300}{}$,
$\subspace{10202\\01210}{15}$,
$\subspace{10214\\01204}{}$,
$\subspace{10221\\01243}{}$,
$\subspace{10233\\01232}{}$,
$\subspace{10240\\01221}{}$,
$\subspace{10202\\01432}{}$,
$\subspace{10214\\01433}{}$,
$\subspace{10221\\01434}{}$,
$\subspace{10233\\01430}{}$,
$\subspace{10240\\01431}{}$,
$\subspace{10401\\01312}{}$,
$\subspace{10413\\01322}{}$,
$\subspace{10420\\01332}{}$,
$\subspace{10432\\01342}{}$,
$\subspace{10444\\01302}{}$,
$\subspace{10204\\01211}{15}$,
$\subspace{10211\\01200}{}$,
$\subspace{10223\\01244}{}$,
$\subspace{10230\\01233}{}$,
$\subspace{10242\\01222}{}$,
$\subspace{10202\\01442}{}$,
$\subspace{10214\\01443}{}$,
$\subspace{10221\\01444}{}$,
$\subspace{10233\\01440}{}$,
$\subspace{10240\\01441}{}$,
$\subspace{10402\\01300}{}$,
$\subspace{10414\\01310}{}$,
$\subspace{10421\\01320}{}$,
$\subspace{10433\\01330}{}$,
$\subspace{10440\\01340}{}$,
$\subspace{10204\\01231}{15}$,
$\subspace{10211\\01220}{}$,
$\subspace{10223\\01214}{}$,
$\subspace{10230\\01203}{}$,
$\subspace{10242\\01242}{}$,
$\subspace{10203\\01421}{}$,
$\subspace{10210\\01422}{}$,
$\subspace{10222\\01423}{}$,
$\subspace{10234\\01424}{}$,
$\subspace{10241\\01420}{}$,
$\subspace{10400\\01304}{}$,
$\subspace{10412\\01314}{}$,
$\subspace{10424\\01324}{}$,
$\subspace{10431\\01334}{}$,
$\subspace{10443\\01344}{}$,
$\subspace{10204\\01231}{15}$,
$\subspace{10211\\01220}{}$,
$\subspace{10223\\01214}{}$,
$\subspace{10230\\01203}{}$,
$\subspace{10242\\01242}{}$,
$\subspace{10203\\01421}{}$,
$\subspace{10210\\01422}{}$,
$\subspace{10222\\01423}{}$,
$\subspace{10234\\01424}{}$,
$\subspace{10241\\01420}{}$,
$\subspace{10400\\01304}{}$,
$\subspace{10412\\01314}{}$,
$\subspace{10424\\01324}{}$,
$\subspace{10431\\01334}{}$,
$\subspace{10443\\01344}{}$,
$\subspace{10300\\01300}{15}$,
$\subspace{10312\\01330}{}$,
$\subspace{10324\\01310}{}$,
$\subspace{10331\\01340}{}$,
$\subspace{10343\\01320}{}$,
$\subspace{10401\\01244}{}$,
$\subspace{10413\\01211}{}$,
$\subspace{10420\\01233}{}$,
$\subspace{10432\\01200}{}$,
$\subspace{10444\\01222}{}$,
$\subspace{10400\\01440}{}$,
$\subspace{10412\\01443}{}$,
$\subspace{10424\\01441}{}$,
$\subspace{10431\\01444}{}$,
$\subspace{10443\\01442}{}$,
$\subspace{10301\\01304}{15}$,
$\subspace{10313\\01334}{}$,
$\subspace{10320\\01314}{}$,
$\subspace{10332\\01344}{}$,
$\subspace{10344\\01324}{}$,
$\subspace{10401\\01214}{}$,
$\subspace{10413\\01231}{}$,
$\subspace{10420\\01203}{}$,
$\subspace{10432\\01220}{}$,
$\subspace{10444\\01242}{}$,
$\subspace{10402\\01424}{}$,
$\subspace{10414\\01422}{}$,
$\subspace{10421\\01420}{}$,
$\subspace{10433\\01423}{}$,
$\subspace{10440\\01421}{}$,
$\subspace{10302\\01312}{15}$,
$\subspace{10314\\01342}{}$,
$\subspace{10321\\01322}{}$,
$\subspace{10333\\01302}{}$,
$\subspace{10340\\01332}{}$,
$\subspace{10404\\01210}{}$,
$\subspace{10411\\01232}{}$,
$\subspace{10423\\01204}{}$,
$\subspace{10430\\01221}{}$,
$\subspace{10442\\01243}{}$,
$\subspace{10404\\01432}{}$,
$\subspace{10411\\01430}{}$,
$\subspace{10423\\01433}{}$,
$\subspace{10430\\01431}{}$,
$\subspace{10442\\01434}{}$,
$\subspace{10302\\01312}{15}$,
$\subspace{10314\\01342}{}$,
$\subspace{10321\\01322}{}$,
$\subspace{10333\\01302}{}$,
$\subspace{10340\\01332}{}$,
$\subspace{10404\\01210}{}$,
$\subspace{10411\\01232}{}$,
$\subspace{10423\\01204}{}$,
$\subspace{10430\\01221}{}$,
$\subspace{10442\\01243}{}$,
$\subspace{10404\\01432}{}$,
$\subspace{10411\\01430}{}$,
$\subspace{10423\\01433}{}$,
$\subspace{10430\\01431}{}$,
$\subspace{10442\\01434}{}$,
$\subspace{10303\\01324}{15}$,
$\subspace{10310\\01304}{}$,
$\subspace{10322\\01334}{}$,
$\subspace{10334\\01314}{}$,
$\subspace{10341\\01344}{}$,
$\subspace{10404\\01203}{}$,
$\subspace{10411\\01220}{}$,
$\subspace{10423\\01242}{}$,
$\subspace{10430\\01214}{}$,
$\subspace{10442\\01231}{}$,
$\subspace{10402\\01421}{}$,
$\subspace{10414\\01424}{}$,
$\subspace{10421\\01422}{}$,
$\subspace{10433\\01420}{}$,
$\subspace{10440\\01423}{}$,
$\subspace{10303\\01324}{15}$,
$\subspace{10310\\01304}{}$,
$\subspace{10322\\01334}{}$,
$\subspace{10334\\01314}{}$,
$\subspace{10341\\01344}{}$,
$\subspace{10404\\01203}{}$,
$\subspace{10411\\01220}{}$,
$\subspace{10423\\01242}{}$,
$\subspace{10430\\01214}{}$,
$\subspace{10442\\01231}{}$,
$\subspace{10402\\01421}{}$,
$\subspace{10414\\01424}{}$,
$\subspace{10421\\01422}{}$,
$\subspace{10433\\01420}{}$,
$\subspace{10440\\01423}{}$,
$\subspace{10303\\01324}{15}$,
$\subspace{10310\\01304}{}$,
$\subspace{10322\\01334}{}$,
$\subspace{10334\\01314}{}$,
$\subspace{10341\\01344}{}$,
$\subspace{10404\\01203}{}$,
$\subspace{10411\\01220}{}$,
$\subspace{10423\\01242}{}$,
$\subspace{10430\\01214}{}$,
$\subspace{10442\\01231}{}$,
$\subspace{10402\\01421}{}$,
$\subspace{10414\\01424}{}$,
$\subspace{10421\\01422}{}$,
$\subspace{10433\\01420}{}$,
$\subspace{10440\\01423}{}$,
$\subspace{10300\\01332}{15}$,
$\subspace{10312\\01312}{}$,
$\subspace{10324\\01342}{}$,
$\subspace{10331\\01322}{}$,
$\subspace{10343\\01302}{}$,
$\subspace{10400\\01232}{}$,
$\subspace{10412\\01204}{}$,
$\subspace{10424\\01221}{}$,
$\subspace{10431\\01243}{}$,
$\subspace{10443\\01210}{}$,
$\subspace{10400\\01432}{}$,
$\subspace{10412\\01430}{}$,
$\subspace{10424\\01433}{}$,
$\subspace{10431\\01431}{}$,
$\subspace{10443\\01434}{}$,
$\subspace{10300\\01332}{15}$,
$\subspace{10312\\01312}{}$,
$\subspace{10324\\01342}{}$,
$\subspace{10331\\01322}{}$,
$\subspace{10343\\01302}{}$,
$\subspace{10400\\01232}{}$,
$\subspace{10412\\01204}{}$,
$\subspace{10424\\01221}{}$,
$\subspace{10431\\01243}{}$,
$\subspace{10443\\01210}{}$,
$\subspace{10400\\01432}{}$,
$\subspace{10412\\01430}{}$,
$\subspace{10424\\01433}{}$,
$\subspace{10431\\01431}{}$,
$\subspace{10443\\01434}{}$,
$\subspace{11203\\00011}{3}$,
$\subspace{11420\\00001}{}$,
$\subspace{13402\\00010}{}$,
$\subspace{11302\\00010}{3}$,
$\subspace{12310\\00001}{}$,
$\subspace{12403\\00011}{}$,
$\subspace{12101\\00010}{3}$,
$\subspace{14100\\00001}{}$,
$\subspace{14304\\00011}{}$,
$\subspace{13103\\00011}{3}$,
$\subspace{13200\\00001}{}$,
$\subspace{14202\\00010}{}$.


\end{document}